\documentclass[english,graybox,envcountchap,sectrefs]{svmono}

\usepackage[T1]{fontenc}
\usepackage[latin9]{inputenc}
\usepackage{babel}
\usepackage{array}
\usepackage{refstyle}
\usepackage{float}
\usepackage{textcomp}
\usepackage{mathrsfs}
\usepackage{amsmath}
\usepackage{amssymb}
\usepackage{makeidx}
\makeindex
\usepackage{graphicx}
\usepackage{esint}
\usepackage[all]{xy}
\PassOptionsToPackage{normalem}{ulem}
\usepackage{ulem}
\usepackage{nomencl}
\providecommand{\printnomenclature}{\printglossary}
\providecommand{\makenomenclature}{\makeglossary}
\makenomenclature
\usepackage[unicode=true,pdfusetitle,
 bookmarks=true,bookmarksnumbered=false,bookmarksopen=false,
 breaklinks=false,pdfborder={0 0 1},backref=false,colorlinks=false]
 {hyperref}
\hypersetup{
 colorlinks=true,citecolor=blue,linkcolor=blue,linktocpage=true,urlcolor=blue}

\makeatletter

\AtBeginDocument{\providecommand\figref[1]{\ref{fig:#1}}}
\AtBeginDocument{\providecommand\lemref[1]{\ref{lem:#1}}}
\AtBeginDocument{\providecommand\eqref[1]{\ref{eq:#1}}}
\AtBeginDocument{\providecommand\chapref[1]{\ref{chap:#1}}}
\newcommand{\lyxmathsym}[1]{\ifmmode\begingroup\def\b@ld{bold}
  \text{\ifx\math@version\b@ld\bfseries\fi#1}\endgroup\else#1\fi}

\providecommand{\tabularnewline}{\\}
\RS@ifundefined{subref}
  {\def\RSsubtxt{section~}\newref{sub}{name = \RSsubtxt}}
  {}
\RS@ifundefined{thmref}
  {\def\RSthmtxt{theorem~}\newref{thm}{name = \RSthmtxt}}
  {}
\RS@ifundefined{lemref}
  {\def\RSlemtxt{lemma~}\newref{lem}{name = \RSlemtxt}}
  {}

\usepackage{enumitem}		
  \newenvironment{svmultproof2}{\begin{proof}}{\smartqed\qed\end{proof}}

\usepackage{mathptmx}
\usepackage{helvet}
\usepackage{courier}
\usepackage{type1cm}         
\usepackage{multicol}
\usepackage{makeidx}
\usepackage{graphicx}

\providecommand{\MR}[1]{}

\usepackage{blkarray}

\newcommand\restr[2]{{
  \left.\kern-\nulldelimiterspace 
  #1 
  \vphantom{\big|} 
  \right|_{#2} 
  }}

{\bfseries}{\rmfamily}
{\bfseries}{\rmfamily}

\DeclareMathAlphabet{\mathcal}{OMS}{cmsy}{m}{n}

\usepackage{ifthen}
\renewenvironment{figure}[1][]{%
 \ifthenelse{\equal{#1}{}}{%
   \@float{figure}
 }{%
   \@float{figure}[#1]%
 }%
 \centering
}{%
 \end@float
}

\renewenvironment{table}[1][]{%
 \ifthenelse{\equal{#1}{}}{%
   \@float{table}
 }{%
   \@float{table}[#1]%
 }%
 \centering
}{%
 \end@float
}

\usepackage[justification=justified,format=plain]{caption}

\renewcommand\listoffigures{
\chapter*{\listfigurename}
\addcontentsline{toc}{chapter}{\listfigurename}
\@starttoc{lof}
}

\renewcommand\listoftables{
\chapter*{\listtablename}
\addcontentsline{toc}{chapter}{\listtablename}
\@starttoc{lot}
}

\setenumerate[1]{label={(\arabic*)}, ref={\arabic*}}

\renewcommand{\qedsymbol}{\ensuremath{\square}}
\renewcommand\smartqed{\renewcommand\qed{\relax\ifmmode\qedsymbol\else
  {\nobreak\hfil\penalty50\hskip1em\null\nobreak\hfil\qedsymbol
  \parfillskip=\z@\finalhyphendemerits=0\endgraf}\fi}}

\usepackage{needspace}

\@ifundefined{showcaptionsetup}{}{%
 \PassOptionsToPackage{caption=false}{subfig}}
\usepackage{subfig}
\makeatother

\begin{document}

\title{Extensions of Positive Definite Functions: Applications and Their
Harmonic Analysis}

\author{Palle Jorgensen, Steen Pedersen, Feng Tian}

\maketitle

\paragraph{Palle E.T. Jorgensen\protect \linebreak{}
Department of Mathematics\protect \linebreak{}
The University of Iowa\protect \linebreak{}
Iowa City, IA 52242-1419, U.S.A.\protect \linebreak{}
\emph{Email address}: palle-jorgensen@uiowa.edu\protect \linebreak{}
\emph{URL}: http://www.math.uiowa.edu/\textasciitilde{}jorgen/}

\paragraph{Steen Pedersen\protect \linebreak{}
Department of Mathematics\protect \linebreak{}
Wright State University\protect \linebreak{}
Dayton, OH 45435, U.S.A. \protect \linebreak{}
\emph{E-mail address}: steen.math@gmail.com, steen@math.wright.edu
\protect \linebreak{}
\emph{URL}: http://www.wright.edu/\textasciitilde{}steen.pedersen/}

\paragraph{Feng Tian\protect \linebreak{}
Department of Mathematics\protect \linebreak{}
Trine University\protect \linebreak{}
One University Avenue, Angola, IN 46703, U.S.A.\protect \linebreak{}
\emph{E-mail address}: james.ftian@gmail.com, tianf@trine.edu\protect \linebreak{}
}

\clearpage{}

\renewcommand{\footnotesize}{\small}

\begin{flushleft}
\emph{Dedicated to the memory of}\vspace{5pt}\\
\emph{William B. Arveson}\footnote{William Arveson (1934 \textendash{} 2011) worked in operator algebras
and harmonic analysis, and his results have been influential in our
thinking, and in our approach to the particular extension questions
we consider here. In fact, Arveson's deep and pioneering work on completely
positive maps may be thought of as a non-commutative variant of our
present extension questions. We have chosen to give our results in
the commutative setting, but readers with interests in non-commutative
analysis will be able to make the connections. While the non-commutative
theory was initially motivated by the more classical commutative theory,
the tools involved are different, and there are not always direct
links between theorems in one area and the other. One of Arveson's
earlier results in operator algebras is an extension theorem for completely
positive maps taking values in the algebra of all bounded operators
on a Hilbert space. His theorem led naturally to the question of injectivity
of von-Neumann algebras in general, which culminated in profound work
by Alain Connes relating injectivity to hyperfiniteness. One feature
of Arveson's work dating back to a series of papers in the 60's and
70's, is the study of noncommutative analogues of notions and results
from classical harmonic analysis, including the Shilov and Choquet
boundaries. The commutative analogues are visible in our present presentation.} \\
(22 November 1934 -- 15 November 2011)\vspace{5pt}\\
\emph{Edward Nelson}\footnote{Edward Nelson (1932 \textendash{} 2014) was known for his work on
mathematical physics, stochastic processes, in representation theory,
and in mathematical logic. Especially his work in the first three
areas has influenced our thinking. In more detail: infinite-dimensional
group representations, the mathematical treatment of quantum field
theory, the use of stochastic processes in quantum mechanics, and
his reformulation of probability theory. Readers looking for beautiful
expositions of the foundations in these areas are referred to the
following two set of very accessible lecture notes by Nelson, \emph{Dynamical
theory of Brownian motion; and Topics in Dynamics 1: Flows.}, both
in Princeton University Press, the first 1967, and the second 1969.\\
\\
Our formulation of the present extension problems in the form of Type
I and Type II Extensions (see Chapter \ref{chap:types} below) was
especially influenced by independent ideas and results of both Bill
Arveson and Ed Nelson. We use the two Nelson papers \cite{Nel59,NS59}
in our analysis of extensions of locally defined positive definite
functions on Lie groups.}\\
(May 4, 1932 -- September 10, 2014) 
\par\end{flushleft}

\foreword{}

We study two classes of extension problems, and their interconnections: 
\begin{enumerate}[label=(\roman{enumi}),ref=\roman{enumi}]
\item \label{enu:ext1}Extension of positive definite (p.d.) continuous
functions defined on subsets in locally compact groups $G$; 
\item \label{enu:ext2}In case of Lie groups, representations of the associated
Lie algebras $La\left(G\right)$ by unbounded skew-Hermitian operators
acting in a reproducing kernel Hilbert space (RKHS) $\mathscr{H}_{F}$. 
\end{enumerate}
Our analysis is non-trivial even if $G=\mathbb{R}^{n}$, and even
if $n=1$. If $G=\mathbb{R}^{n}$, we are concerned in (\ref{enu:ext2})
with finding systems of strongly commuting selfadjoint operators $\left\{ T_{i}\right\} $
extending a system of commuting Hermitian operators with common dense
domain in $\mathscr{H}_{F}$. \index{operator!strongly commuting-}
\index{Hermitian}

\emph{Why extensions}? In science, experimentalists frequently gather
spectral data in cases when the observed data is limited, for example
limited by the precision of instruments; or on account of a variety
of other limiting external factors. (For instance, the human eye can
only see a small portion of the electromagnetic spectrum.) Given this
fact of life, it is both an art and a science to still produce solid
conclusions from restricted or limited data. In a general sense, our
monograph deals with the mathematics of extending some such given
partial data-sets obtained from experiments. More specifically, we
are concerned with the problems of extending available partial information,
obtained, for example, from sampling. In our case, the limited information
is a restriction, and the extension in turn is the full positive definite
function (in a dual variable); so an extension if available will be
an everywhere defined generating function for the exact probability
distribution which reflects the data; if it were fully available.
Such extensions of local information (in the form of positive definite
functions) will in turn furnish us with spectral information. In this
form, the problem becomes an operator extension problem, referring
to operators in a suitable reproducing kernel Hilbert spaces (RKHS).
In our presentation we have stressed hands-on-examples. Extensions
are almost never unique, and so we deal with both the question of
existence, and if there are extensions, how they relate back to the
initial completion problem. \index{Hilbert space}

By a theorem of S. Bochner, the continuous p.d. functions are precisely
the Fourier transforms of finite positive measures. In the case of
locally compact Abelian groups $G$, the two sides in the Fourier
duality is that of the group $G$ itself vs the dual character group
$\widehat{G}$ to $G$. Of course if $G=\mathbb{R}^{n}$, we may identify
the two. \index{duality!Fourier-}

But in practical applications a p.d. function will typically be given
only locally, or on some open subset, typically bounded; say an interval
if $G=\mathbb{R}^{1}$; or a square or a disk in case $G=\mathbb{R}^{2}$.
Hence four questions naturally arise: 
\begin{enumerate}[label=(\alph{enumi}),ref=\alph{enumi}]
\item \label{enu:whya}Existence of extensions.
\item \label{enu:whyb}If there are extensions, find procedures for constructing
them. 
\item \label{enu:whyc}Moreover, what is the significance of choice of different
extensions from available sets of p.d. extensions? 
\item \label{enu:whyd}Finally, what are the generalizations and applications
of the results in (\ref{enu:whya})-(\ref{enu:whyc}) to the case
of an infinite number of dimensions? 
\end{enumerate}
All four questions will be addressed, and the connections between
(\ref{enu:whyd}) and probability theory will be stressed.\index{Theorem!Bochner's-}\index{Bochner's Theorem}

While the theory of p.d. functions is important in many areas of pure
and applied mathematics, ranging from harmonic analysis, functional
analysis, spectral theory, representations of Lie groups, and operator
theory on the pure side, to such applications as mathematical statistics,
to approximation theory, to optimization (and more, see details below),
and to quantum physics, it is difficult for students and for the novice
to the field, to find accessible presentations in the literature which
cover all these disparate points of view, as well as stressing common
ideas and interconnections. 

We have aimed at filling this gap, and with a minimum number of prerequisites.
We do expect that readers have some familiarity with measures and
their Fourier transform, as well as with operators in Hilbert space,
especially the theory of unbounded symmetric operators with dense
domain. When needed, we have included brief tutorials. Further, in
our cited references we have included both research papers, and books.
To help with a historical perspective, we have included discussions
of some landmark presentations, especially papers of S. Bochner, M.G.
Krein, J. von Neumann, W. Rudin, and I.J. Sch\"{o}enberg.\index{von Neumann, John}

The significance to \emph{stochastic processes} of the two questions,
(\ref{enu:ext1}) and (\ref{enu:ext2}) above, is as follows. To simplify,
consider first stochastic processes $X_{t}$ indexed by time $t$,
but known only for \textquotedblleft small\textquotedblright{} $t$.
Then the corresponding covariance function $c_{X}\left(s,t\right)$
will also only be known for small values of $s$ and $t$; or in the
case of stationary processes, there is then a locally defined p.d.
function $F$, known only in a bounded interval $J=\left(-a,a\right)$,
and such that $F\left(s-t\right)=c_{X}\left(s,t\right)$. Hence, it
is natural to ask how much can be said about extensions to values
of $t$ in the complement of $J$? For example; what are the possible
global extensions $X_{t}$, i.e., extension to all $t\in\mathbb{R}$?
If there are extensions, then how does information about the locally
defined covariance function influence the extended global process?
What is the structure of the set of all extensions? The analogues
questions are equally interesting for processes indexed by groups
more general than $\mathbb{R}$.

Specifically, we consider partially defined p.d. continuous functions
$F$ on a fixed group. From $F$ we then build a RKHS $\mathscr{H}_{F}$,
and the operator extension problem is concerned with operators acting
in $\mathscr{H}_{F}$, and with unitary representations of $G$ acting
on $\mathscr{H}_{F}$. Our emphasis is on the interplay between the
two problems, and on the harmonic analysis of our RKHS $\mathscr{H}_{F}$.

In the cases of $G=\mathbb{R}^{n}$, and $G=\mathbb{R}^{n}/\mathbb{Z}^{n}$,
and generally for locally compact Abelian groups, we establish a new
Fourier duality theory; including for $G=\mathbb{R}^{n}$ a time/frequency
duality, where the extension questions (\ref{enu:ext1}) are in time
domain, and extensions from (\ref{enu:ext2}) in frequency domain.
Specializing to $n=1$, we arrive of a spectral theoretic characterization
of all skew-Hermitian operators with dense domain in a separable Hilbert
space, having deficiency-indices $\left(1,1\right)$. \index{time/frequency duality}\index{non-commutative harmonic analysis}

Our general results include non-compact and non-Abelian Lie groups,
where the study of unitary representations in $\mathscr{H}_{F}$ is
subtle.\index{skew-Hermitian operator; also called skew-symmetric}

While, in the most general case, the obstructions to extendibility
(of locally defined positive definite (p.d.) functions) is subtle,
we point out that it has several explicit features: algebraic, analytic,
and geometric. In Section \ref{sec:logz}, we give a continuous p.d.
function $F$ in a neighborhood of $0$ in $\mathbb{R}^{2}$ for which
$Ext\left(F\right)$ is empty. In this case, the obstruction for $F$,
is geometric in nature; and involves properties of a certain Riemann
surface. (See Figure \ref{fig:logz0}, and Section \ref{sec:logz},
Figures \ref{fig:logz}-\ref{fig:trans}, for details.)

\begin{figure}
\includegraphics[width=0.4\textwidth]{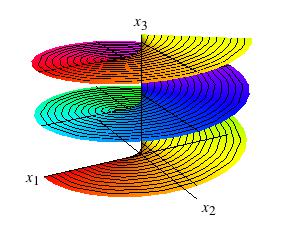}

\protect\caption{\label{fig:logz0}Riemann Surface of $\log z$.}

\end{figure}

\textbf{Note on Presentation.} In presenting our results, we have
aimed for a reader-friendly account. We have found it helpful to illustrate
the ideas with worked examples. Each of our theorems holds in various
degrees of generality, but when appropriate, we have not chosen to
present details in their highest level of generality. Rather, we typically
give the result in a setting where the idea is more transparent, and
easier to grasp. We then work the details systematically at this lower
level of generality. But we also make comments about the more general
versions; sketching these in rough outline. The more general versions
of the respective theorems will typically be easy for readers to follow,
and to appreciate, after the idea has already been fleshed out in
a simpler context.

We have made a second choice in order to make it easier for students
to grasp the ideas as well as the technical details: We have included
a lot of worked examples. And at the end of each of these examples,
we then outline how the specific details (from the example in question)
serve to illustrate one or more features in the general theorems elsewhere
in the monograph. Finally, we have made generous use of both Tables
and Figures. These are listed with page-references at the end of the
book; see the last few items in the Table of Contents. And finally,
we included a list of Symbols on page \pageref{tag:nome}, after Table
of Contents.

\motto{The art of doing mathematics consists in finding that special case
which contains all the germs of generality. --- David Hilbert}

\preface{}

On the one hand, the subject of positive definite (p.d.) functions
has played an important role in standard graduate courses, and in
research papers, over decades; and yet when presenting the material
for a particular purpose, the authors have found that there is not
a single source which will help students and researchers quickly form
an overview of the essential ideas involved. Over the decades new
ideas have been incorporated into the study of p.d. functions, and
their more general cousin, p.d. kernels, from a host of diverse areas.
An influence of more recent vintage is the theory of \emph{operators
in Hilbert space}, and their \emph{spectral theory}. \index{Hilbert space}

A novelty in our present approach is the use of diverse Hilbert spaces.
In summary; starting with a locally defined p.d. $F$, there is a
natural associated Hilbert space, arising as a reproducing kernel
Hilbert space (RKHS), $\mathscr{H}_{F}$. Then the question is: When
is it possible to realize globally defined p.d. extensions of $F$
with the use of spectral theory for operators in the initial RKHS,
$\mathscr{H}_{F}$? And when will it be necessary to enlarge the Hilbert
space, i.e., to pass to a \emph{dilation Hilbert space}; -- a second
Hilbert space $\mathscr{K}$ containing an isometric copy of $\mathscr{H}_{F}$
itself? \index{dilation Hilbert space}

The theory of p.d. functions has a large number of applications in
a host of areas; for example, in harmonic analysis, in representation
theory (of both algebras and groups), in physics, and in the study
of probability models, such as stochastic processes. One reason for
this is the theorem of Bochner  which links continuous p.d. functions
on locally compact Abelian groups $G$ to measures on the corresponding
dual group. Analogous uses of p.d. functions exist for classes for
non-Abelian groups. Even the seemingly modest case of $G=\mathbb{R}$
is of importance in the study of spectral theory for Schrödinger equations.
And further, counting the study of Gaussian stochastic processes,
there are even generalizations (Gelfand-Minlos) \label{nt:gelfand}to
the case of continuous p.d. functions on Fréchet spaces of test functions
which make up part of a Gelfand triple. \index{Gelfand triple}

These cases will be explored below, but with the following important
change in the starting point of the analysis; -- we focus on the case
when the given p.d. function is only \emph{partially defined}, i.e.,
is only known on a proper subset of the ambient group, or space. How
much of duality theory carries over when only partial information
is available?

\textbf{Applications.} In \emph{machine learning} extension problems
for p.d. functions and RHKSs seem to be used in abundance. Of course
the functions are initially defined over finite sets which is different
than present set-up, but the ideas from our continuous setting do
carry over \emph{mutatis mutandis}. Machine learning \cite{PS03}
is a field that has evolved from the study of pattern recognition
and computational learning theory in artificial intelligence. It explores
the construction and study of algorithms that can learn from and make
predictions on limited data. Such algorithms operate by building models
from \textquotedblleft training data\textquotedblright{} inputs in
order to make data-driven predictions or decisions. Contrast this
with strictly static program instructions. Machine learning and pattern
recognition can be viewed as two facets of the same field.

Another connection between extensions of p.d. functions and neighboring
areas is \emph{number theory}: The pair correlation of the zeros of
the Riemann zeta function \cite{MR1018376,MR2609298}. In the 1970's
Hugh Montgomery (assuming the Riemann hypothesis) determined the Fourier
transform of the pair correlation function in number theory (-- it
is a p.d. function). But the pair correlation function is specified
only in a bounded interval centered at zero; again consistent with
pair correlations of eigenvalues of large random Hermitian matrices
(via Freeman Dyson). It is still not known what is the Fourier transform
outside of this interval. Montgomery has conjectured that, on all
of $\mathbb{R}$, it is equal to the Fourier transform of the pair
correlation of the eigenvalues of large random Hermitian matrices;
this is the ``Pair correlation conjecture''. And it is an important
unsolved problem.

Yet another application of the tools for extending locally defined
p.d. functions is that of the pioneering work of M.G. Krein, now called
the \emph{inverse spectral problem of the strings of Krein}, see \cite{MR3160601,MR1699971,MR661628}.
This in turn is directly related to a host of the symmetric moment
problems \cite{MR649179}. In both cases we arrive at the problem
of extending a real p.d. function that is initially only known on
an interval.\index{Krein, M.}\index{moment problem}

\uline{In summary}; the purpose of the present monograph is to
explore what can be said when a continuous p.d. function is only given
on a subset of the ambient group (which is part of the application
setting sketched above.) For this problem of partial information,
even the case of p.d. functions defined only on bounded subsets of
$G=\mathbb{R}$ (say an interval), or on bounded subsets of $G=\mathbb{R}^{n}$,
is of substantial interest.

\index{representation!-of group}

\extrachap{Acknowledgements}

The co-authors thank the following for enlightening discussions: Professors
Daniel Alpay, Sergii Bezuglyi, Dorin Dutkay, Paul Muhly, Rob Martin,
Robert Niedzialomski, Gestur Olafsson, Judy Packer, Wayne Polyzou,
Myung-Sin Song, and members in the Math Physics seminar at the University
of Iowa. 

We also are pleased to thank anonymous referees for careful reading,
for lists of corrections, for constructive criticism, and for many
extremely helpful suggestions; -- for example, pointing out to us
more ways that the question of extensions of fixed locally defined
positive definite functions, impact yet more areas of mathematics,
and are also part of important applications to neighboring areas.
Remaining flaws are the responsibility of the co-authors.

\tableofcontents{}

\nomenclature{RKHS}{reproducing kernel Hilbert space (pg. \pageref{sec:Prelim}, \pageref{sec:embedding})}

\nomenclature{p.d.}{positive definite (pg. \pageref{chap:intro}, \pageref{sec:Organization})}

\nomenclature{$\mathscr{H}_{F}$}{The reproducing kernel Hilbert space (RKHS) of $F$. (pg. \pageref{sec:introRKHS}, \pageref{sec:Prelim})}

\nomenclature{$Ext(F)$}{Set of unitary representations of $G$ on a Hilbert space $\mathscr{K}$, i.e., the triples $\left(U,\mathscr{K},k_{0}\right)$, that extend $F$. (pg. \pageref{sub:ExtSpace}, \pageref{chap:types})}

\nomenclature{$Ext_{1}(F)$}{The triples in $Ext(F)$, where the representation space is $\mathscr{H}_{F}$, and so the extension of $F$ is realized on $\mathscr{H}_{F}$. (pg. \pageref{sub:ExtSpace}, \pageref{chap:types})}

\nomenclature{$Ext_{2}(F)$}{$Ext(F) \setminus Ext_{1}(F)$, i.e., the extension is realized on an  enlargement Hilbert space. (pg. \pageref{sub:ExtSpace}, \pageref{chap:types})}

\nomenclature{$T_F$}{Mercer operator associated to $F$. (pg. \pageref{sec:mercer}, \pageref{sub:green})}

\nomenclature{$F_{\varphi}$}{The convolution of $F$ and $\varphi$, where $\varphi \in C_{c}(\Omega)$. (pg. \pageref{lem:RKHS-def-by-integral}, \pageref{lem:lcg-F_varphi})}

\nomenclature{$D^{(F)}$}{The derivative operator $F_{\varphi} \mapsto F_{\frac{d\varphi}{dx}}$. (pg. \pageref{sub:DF}, \pageref{sub:GR})\\}

\nomenclature{$DEF$}{The deficiency space of $D^{(F)}$ acting in the Hilbert space $\mathscr{H}_{F}$. (pg. \pageref{sub:DF}, \pageref{cor:mer3})}

\nomenclature{$\mathscr{M}(G)$}{All Borel measures on $G$. (pg. \pageref{rem:measRn}, \pageref{def:lcadual})}

\nomenclature{$\mathfrak{M}_{2}\left(\Omega,F\right)$ }{Hilbert space of measures on $\Omega$ associated to a fixed kernel, or a p.d. function $F$. (pg. \pageref{cor:muHF})}

\nomenclature{$\left<\cdot,\cdot\right>$ or $\left<\cdot,\cdot\right>_H$}{Inner product; we add a subscript (when necessary) in order to indicate which Hilbert space is responsible for the inner product in question. Caution, because of a physics tradition, all of our inner products are linear in the second variable. (This convention further has the advantage of giving simpler formulas in case of reproducing kernel Hilbert spaces (RKHSs).) (pg. \pageref{lem:Fdef}, \pageref{sec:R^1})}

\nomenclature{$\mathscr{B}(G)$}{The sigma--algebra of all Borel subsets of $G$. (pg. \pageref{sub:euclid}, \pageref{sec:index (d,d)})}

\nomenclature{ONB}{orthonormal basis (in a Hilbert space) (pg.\pageref{prop:exp}, \pageref{chap:Ext1})}

\nomenclature{$(\Omega, \mathscr{F}, \mathbb{P})$}{probability space: sample space $\Omega$, sigma-algebra $\mathscr{F}$, probability measure $\mathbb{P}$ (pg. \pageref{sec:stach}, \pageref{lem:bbridge})}

\nomenclature{$\{X_{g}\}_{g \in G}$}{Stochastic process indexed by $G$. (pg. \pageref{sec:stach}, \pageref{thm:kol})}

\nomenclature{$\{B_{t}\}_{t \in \mathbb{R}}$}{Standard Brownian motion. (pg. \pageref{rem:OUprocess}, \pageref{rem:bm})}

\nomenclature{$D^{*}$}{Adjoint of a linear operator $D$. (pg. \pageref{sub:DF}, \pageref{thm:Eigenspaces-for-the-adjoint})}

\nomenclature{$dom(D)$}{Domain of a linear operator $D$. (pg. \pageref{sub:DF}, \pageref{prop:logz})}

\nomenclature{$La(G)$}{The Lie algebra of a given Lie group $G$. (pg. \pageref{enu:la1}, \pageref{thm:lgns})}

\nomenclature{$Rep(G, \mathscr{H})$}{Set of unitary representations of a group $G$ acting on some Hilbert space $\mathscr{H}$. (pg. \pageref{thm:log1})}

\nomenclature{$\perp$}{Perpendicular w.r.t. a fixed Hilbert inner product. (pg. \pageref{thm:R^n-spect}, \pageref{thm:mer4})}

\nomenclature{$\delta_{x}$}{Dirac measure, also called Dirac mass. (pg. \pageref{lem:dense}, \pageref{lem:distd})}

\nomenclature{$\mathbb{R}$}{the real line}

\nomenclature{$\mathbb{R}^{n}$}{the $n$-dimensional real Euclidean space}

\nomenclature{$\mathbb{Z}$}{the integers}

\nomenclature{$\mathbb{T}^{n}=\mathbb{R}^{n} / \mathbb{Z}^{n}$}{tori (We identify $\mathbb{T}$ with the circle group.)}

\nomenclature{$G$}{group, with the group operation written $x\cdot y$, or $x+y$, depending on the context.}

\nomenclature{$S (= S_{2})$}{the isometry $S: \mathfrak{M}_{2}(\Omega,F)\rightarrow \mathscr{H}_{F}$ (onto) (pg. \pageref{cor:muHF}, \pageref{prop:exp})}

\nomenclature{$\int_{J}f(t)\,dX_{t}$}{Ito-integral, defined for $f\in L^{2}(J)$. (pg. \pageref{lem:ito}, \pageref{lem:bbridge})}

\nomenclature{$X$}{random variable, $X:\Omega \rightarrow \mathbb{R}$. (pg. \pageref{lem:ito}, \pageref{lem:bbridge})}

\nomenclature{$\mathbb{E}$}{expectation, $\mathbb{E}(\cdots) = \int_{\Omega}\cdots d\mathbb{P}$. (pg. \pageref{lem:ito}, \pageref{lem:bbridge})}

\nomenclature{$J$}{conjugation operator, acting in the Hilbert space $\mathscr{H}_{F}$ for a fixed local p.d. function $F$. To say that $J$ is a conjugation means that $J$ is a conjugate linear operator which is also of period 2. (pg. \pageref{lem:Conjugation-Operator}, \pageref{cor:DF})\\}

\nomenclature{$\widehat{G}$}{the dual character group, where $G$ is a fixed locally compact Abelian\\group, i.e., $\lambda:G\rightarrow\mathbb{T}$, continuous, $\lambda\left(x+y\right)=\lambda\left(x\right)\lambda\left(y\right)$,\\$\forall x,y\in G$, $\lambda\left(-x\right)=\overline{\lambda\left(x\right)}$. (pg. \pageref{thm:lcg-dualG}, \pageref{thm:TT*})}

\nomenclature{$\left\langle \lambda,x\right\rangle$}{duality pairing $\widehat{G}\leftrightarrow G$ of locally compact Abelian groups. (pg. \pageref{sub:lcg}, \pageref{thm:TT*})}

\printnomenclature[2cm]{}\label{tag:nome}

\motto{Mathematics is an experimental science, and definitions do not come
first, but later on. \textemdash{} Oliver Heaviside}

\chapter{\label{chap:intro}Introduction}

Positive-definiteness arises naturally in the theory of the Fourier
transform. There are two directions in transform theory. In the present
setting, one is straightforward, and the other (Bochner) is deep.
First, it is easy to see directly that the Fourier transform of a
positive finite measure is a positive definite function; and that
it is continuous. The converse result is \emph{Bochner's theorem}.
It states that any continuous positive definite function on the real
line is the Fourier transform of a unique positive and finite measure.
However, if some given positive definite function is only \emph{partially
defined}, for example in an interval, or in the planar case, in a
disk or a square, then Bochner's theorem does not apply. One is faced
with first seeking a positive definite extension; hence the theme
of our monograph. 
\begin{definition}
\label{def:pdf}Fix $0<a$, let $\Omega$ be an open interval of length
$a$, then $\Omega-\Omega=(-a,a)$. Let a function 
\begin{equation}
F:\Omega-\Omega\rightarrow\mathbb{C}\label{eq:F}
\end{equation}
be continuous, and defined on $\Omega-\Omega$. $F$ is \emph{positive
definite (p.d.)} if \index{positive definite} 
\begin{equation}
\sum\nolimits _{i}\sum\nolimits _{j}\overline{c_{i}}c_{j}F(x_{i}-x_{j})\geq0,\label{eq:def-pd}
\end{equation}
for all finite sums with $c_{i}\in\mathbb{C}$, and all $x_{i}\in\Omega$.
Hence, $F$ is p.d. iff the $N\times N$ matrix $\left(F(x_{i}-x_{j})\right)_{i,j=1}^{N}$
is p.d. for all $x_{1},\ldots,x_{N}$ in $\Omega$, and all $N\in\mathbb{N}$. 
\end{definition}

Applications of positive definite functions include statistics, especially
Bayesian statistics, and the setting is often the case of real valued
functions; while complex valued and Hilbert space valued functions
are important in mathematical physics. 

In some statistical applications, one often takes $n$ scalar measurements
(sampling of a random variable) of points in $\mathbb{R}^{n}$, and
one requires that points that are closely separated have measurements
that are highly correlated. But in practice, care must be exercised
to ensure that the resulting covariance matrix (an $n$-by-$n$ matrix)
is always positive definite. One proceeds to define such a correlation
matrix $A$ which is then multiplied by a scalar to give us a covariance
matrix: It will be positive definite, and Bochner's theorem applies:
If the correlation between a pair of points depends only on the distance
between them (via a function $F$), then this function $F$ must be
positive definite since the covariance matrix $A$ is positive definite.
In lingo from statistics, Fourier transform becomes instead characteristic
function. It is computed from a distribution (``measure'' in harmonic
analysis).\index{Theorem!Bochner's-}\index{Bochner's Theorem}

\renewcommand{\footnotesize}{\small}

In our monograph we shall consider a host of diverse settings and
generalizations, e.g., to positive definite functions on groups. Indeed,
in the more general setting of locally compact Abelian topological
groups, Bochner's theorem still applies. This is the setting naturally
occurring in the study of representation of groups; representations
typically acting on infinite-dimensional Hilbert spaces (so we study
the theory of unitary representations). The case of locally compact
Abelian groups was pioneered by W. Rudin, and by M. Stone, M.A. Naimark,
W. Ambrose, and R. Godement\footnote{M.H. Stone, Ann. Math. 33 (1932) 643-648\\
M.A. Naimark, Izv. Akad. Nauk SSSR. Ser. Mat. 7 (1943) 237-244 \\
W. Ambrose, Duke Math. J. 11 (1944) 589-595 \\
R. Godement, C.R. Acad. Sci. Paris 218 (1944) 901-903}. Now, in the case of non-Abelian Lie groups, the motivation is from
the study of symmetry in quantum theory, and there are numerous applications
of non-commutative harmonic analysis to physics. In our presentation
we will illustrate the theory in both Abelian and the non-Abelian
cases. This discussion will be supplemented by citations to the References
in the back of the book.\index{non-commutative harmonic analysis}
\index{group!topological-}

We study two classes of extension problems, and their interconnections.
The first class of extension problems concerns (i) continuous positive
definite (p.d.) functions on \emph{Lie groups} $G$; and the second
deals with (ii) \emph{Lie algebras} of unbounded skew-Hermitian operators
in a certain family of \emph{reproducing kernel Hilbert space} (RKHS). 

Our analysis is non-trivial even if $G=\mathbb{R}^{n}$, and even
if $n=1$. 

If $G=\mathbb{R}^{n}$, we are concerned in (ii) with the study of
systems of $n$ skew-Hermitian operators $\left\{ S_{i}\right\} $
on a common dense domain in Hilbert space, and in deciding whether
it is possible to find a corresponding system of strongly commuting
selfadjoint operators $\left\{ T_{i}\right\} $ such that, for each
value of $i$, the operator $T_{i}$ extends $S_{i}$.

\index{extension problem}

\index{operator!selfadjoint}

\index{group!Lie}

\index{operator!skew-Hermitian}

\index{RKHS}

\index{positive definite}

\index{representation!-of group}

\index{operator!strongly commuting-}

The version of this for non-commutative Lie groups $G$ will be stated
in the language of unitary representations of $G$, and corresponding
representations of the Lie algebra $La\left(G\right)$ by skew-Hermitian
unbounded operators.

In summary, for (i) we are concerned with partially defined continuous
p.d. functions $F$ on a Lie group; i.e., at the outset, such a function
$F$ will only be defined on a connected proper subset in $G$. From
this partially defined p.d. function $F$ we then build a RKHS $\mathscr{H}_{F}$,
and the operator extension problem (ii) is concerned with operators
acting on $\mathscr{H}_{F}$, as well as with unitary representations
of $G$ acting on $\mathscr{H}_{F}$. If the Lie group $G$ is not
simply connected, this adds a complication, and we are then making
use of the associated simply connected covering group. For an overview
of high-points, see Sections \ref{sec:introRKHS} and \ref{sec:Organization}
below.

\index{Lie group}\index{Lie algebra}\index{universal covering group}\index{group!universal covering-}

Readers not familiar with some of the terms discussed above may find
the following references helpful \cite{Zie14,KL14,Gne13,MS12,Ber12,GZM11,Kol11,HV11,BT11,Hi80,App09,App08,Ito04,EF11,Lai08,Alp92,Aro50,BCR84,DS88b,MR1069255,Jor91,Nel57,Ru70}.
The list includes both basic papers and texts, as well as some recent
research papers.

\section{\label{sec:2ext}Two Extension Problems}

Our main theme is the interconnection between (i) the study of extensions
of \emph{locally defined continuous and positive definite (p.d.) functions
$F$} on groups on the one hand, and, on the other, (ii) the question
of extensions for an associated \emph{system of unbounded Hermitian
operators with dense domain} in a reproducing kernel Hilbert space
(RKHS) $\mathscr{H}_{F}$ associated to $F$.\index{Hilbert space}
\index{Hermitian}

Because of the role of p.d. functions in harmonic analysis, in statistics,
and in physics, the connections in both directions are of interest,
i.e., from (i) to (ii), and vice versa. This means that the notion
of \textquotedblleft extension\textquotedblright{} for question (ii)
must be inclusive enough in order to encompass all the extensions
encountered in (i). For this reason enlargement of the initial Hilbert
space $\mathscr{H}_{F}$ is needed. In other words, it is necessary
to consider also operator extensions which are realized in a dilation-Hilbert
space; a new Hilbert space containing $\mathscr{H}_{F}$ isometrically,
and with the isometry intertwining the respective operators.

An overview of the extension correspondence is given in Section \ref{sub:GR},
Figure \ref{fig:extc}.\index{extension correspondence}

To appreciate these issues in concrete examples, readers may wish
to consult Chapter \ref{chap:ext}, especially the last two sections,
\ref{sec:circle} and \ref{sub:exp(-|x|)}. To help visualization,
we have included tables and figures in Sections \ref{sec:F3-Mercer},
\ref{sec:Green}, and \ref{sec:F2F3}.

\index{Lie group}

\paragraph{\textbf{Where to find it.}}

\textbf{Caption to the table in item (i) below:} A number of cases
of the extension problems (we treat) occur in increasing levels of
generality, interval vs open subsets of $\mathbb{R}$, $\mathbb{T}=\mathbb{R}/\mathbb{Z}$,
$\mathbb{R}^{n}$, locally compact Abelian group, and finally of a
Lie group, both the specific theorems and the parameters differ from
one to the other, and we have found it worthwhile to discuss the settings
in separate sections. To help readers make comparisons, we have outlined
a roadmap, with section numbers, where the cases can be found. 

\renewcommand{\arraystretch}{2}
\begin{enumerate}[label=(\roman{enumi})]
\item Extension of \emph{locally defined p.d. functions}.

\begin{enumerate}[label=,itemsep=3pt]
\item Type I vs Type II: Sections \ref{sub:ExtSpace} (p.\pageref{sub:ExtSpace}),
\ref{sec:mercer} (p.\pageref{sub:ExtSpace})
\item Level of generality: 
\end{enumerate}

\begin{center}
\begin{tabular}{|>{\raggedright}p{0.45\textwidth}|>{\raggedright}p{0.48\textwidth}|}
\hline 
Ambient Space & Section \tabularnewline
\hline 
interval $\subset\mathbb{R}$ & §\ref{sub:ExtSpace} (p.\pageref{sub:ExtSpace}), §\ref{sec:Green}
(p.\pageref{sec:Green}), §\ref{sec:F2F3} (p.\pageref{sec:F2F3})\tabularnewline
\hline 
$\Omega\subset\mathbb{R}^{n}$ & §\ref{sub:euclid} (p.\pageref{sub:euclid}), §\ref{sec:hdim} (p.\pageref{sub:euclid})\tabularnewline
\hline 
$\Omega\subset\mathbb{T}=\mathbb{R}/\mathbb{Z}$ & §\ref{sub:G=00003DT} (p.\pageref{sub:G=00003DT}), §\ref{sec:expT}
(p.\pageref{sec:expT}), §\ref{sec:circle} (p.\pageref{sec:circle})\tabularnewline
\hline 
$\Omega\subset G$, locally cpt Abelian group & §\ref{sub:lcg} (p.\pageref{sub:lcg})\tabularnewline
\hline 
$\Omega\subset G$, general Lie group & §\ref{sub:lie} (p.\pageref{sub:lie})\tabularnewline
\hline 
$\Omega\subset\mathcal{S}$, in a Gelfand-triple & p.\pageref{nt:gelfand}, §\ref{sec:stach} (p.\pageref{sec:stach}),
§\ref{sec:F2F3} (p.\pageref{sec:F2F3})\index{Gelfand triple}\tabularnewline
\hline 
\end{tabular}
\par\end{center}

\item Connections to extensions of \emph{system of unbounded Hermitian operators}

\begin{enumerate}[label=,itemsep=3pt]
\item Models for indices - $\left(1,1\right)$ operators: Section \ref{sec:index 11}
(p.\pageref{sec:index 11})
\item Models for indices - $\left(d,d\right)$, $d>1$, operators: Section
\ref{sec:index (d,d)} (p.\pageref{sec:index (d,d)})
\end{enumerate}
\end{enumerate}
\renewcommand{\arraystretch}{1}

\index{Hermitian}

While each of the two extension problems\index{extension problem}
has received a considerable amount of attention in the literature,
the emphasis here will be \emph{the interplay between the two}. The
aim is a duality theory; and, in the case $G=\mathbb{R}^{n}$, and
$G=\mathbb{T}^{n}=\mathbb{R}^{n}/\mathbb{Z}^{n}$, the theorems will
be stated in the language of Fourier duality \index{Fourier duality}
of Abelian groups: With the time frequency duality formulation of
Fourier duality for $G=\mathbb{R}^{n}$, both the time domain and
the frequency domain constitute a copy of $\mathbb{R}^{n}$. We then
arrive at a setup such that our extension questions (i) are in time
domain, and extensions from (ii) are in frequency domain. Moreover
we show that each of the extensions from (i) has a variant in (ii).
Specializing to $n=1$, we arrive of a spectral theoretic characterization
of all skew-Hermitian operators with dense domain in a separable Hilbert
space, having deficiency-indices $\left(1,1\right)$. \index{time/frequency duality}

A systematic study of densely defined Hermitian operators with \emph{deficiency
indices} $\left(1,1\right)$, and later $\left(d,d\right)$ for $d>1$,
was initiated by M. Krein\index{Krein, M.} \cite{Kre46}, and is
also part of de Branges\textquoteright{} model theory \cite{deB68,BrRo66}.
The direct connection between this theme and the problem of extending
continuous p.d. functions $F$ when they are only defined on a fixed
open subset to $\mathbb{R}^{n}$ was one of our motivations. One desires
continuous p.d. extensions to $\mathbb{R}^{n}$. 

If $F$ is given, we denote the set of such extensions $Ext\left(F\right)$.
If $n=1$, $Ext\left(F\right)$ is always non-empty, but for $n=2$,
Rudin gave examples when $Ext\left(F\right)$ may be empty \cite{Ru70,Ru63}.
Here we extend these results, and we also cover a number of classes
of p.d. functions on locally compact groups in general. 

Our results in the framework of locally compact Abelian (l.c.a) groups
are more complete than their counterparts for non-Abelian\index{group!non-Abelian}
Lie groups, one reason is the availability of Bochner\textquoteright s
duality theorem for l.c.a groups \cite{BC48,BC49,Boc46,Boc47}; --
not available for non-Abelian Lie groups.

\section{Quantum Physics}

The axioms of quantum physics (see e.g., \cite{BoMc13,OdHo13,KS02,CKS79,AAR13,Fan10,Maa10,Par09}
for relevant recent papers), are based on Hilbert space, and selfadjoint
operators. 

A brief sketch: A quantum mechanical observable is a Hermitian (selfadjoint)
linear operator mapping a Hilbert space, the space of states, into
itself. The values obtained in a physical measurement are in general
described by a probability distribution; and the distribution represents
a suitable ``average'' (or ``expectation'') in a measurement of
values of some quantum observable in a state of some prepared system.
The states are (up to phase) unit vectors in the Hilbert space, and
a measurement corresponds to a probability distribution (derived from
a projection-valued spectral measure). The particular probability
distribution used depends on both the state and the selfadjoint operator.
The associated spectral type may be continuous (such as position and
momentum; both unbounded) or discrete (such as spin); this depends
on the physical quantity being measured.

Symmetries are ubiquitous in physics, and in dynamics. Because of
the axioms of quantum theory, they take the form of unitary representations
of groups $G$ acting on Hilbert space; the groups are locally Euclidian,
(this means Lie groups). The tangent space at the neutral element
$e$ in $G$ acquires a Lie bracket, making it into a Lie algebra.
For describing dynamics from a Schrödinger wave equation, \textbf{$G=\mathbb{R}$}
(the real line, for time). In the general case, we consider strongly
continuous unitary representations $U$ of $G$; and if $G=\mathbb{R}$,
we say that $U$ is a unitary one-parameter group.

From \emph{unitary representation} $U$ of $G$ to \emph{positive
definite function}: Let 
\begin{equation}
U:G\rightarrow\mbox{\ensuremath{\begin{bmatrix}\mbox{unitary operators in }\\
\mbox{some Hilbert space \ensuremath{\mathscr{H}}}
\end{bmatrix}}},\label{eq:up1}
\end{equation}
let $v_{0}\in\mathscr{H}\backslash\left\{ 0\right\} $, and let 
\begin{equation}
F\left(x\right):=\left\langle v_{0},U\left(x\right)v_{0}\right\rangle _{\mathscr{H}},\quad x\in G;\label{eq:up2}
\end{equation}
then $F$ is positive definite (p.d.) on $G$. The converse is true
too, and is called the Gelfand-Naimark-Segal (GNS)-theorem, see Sections
\ref{sub:GNS}-\ref{sub:lgns}; i.e., from every p.d. function $F$
on some group $G$, there is a triple $\left(U,\mathscr{H},v_{0}\right)$,
as described above, such that (\ref{eq:up2}) holds. \index{representation!GNS-}

In the case where $G$ is a locally compact Abelian group with dual
character group $\widehat{G}$, and if $U$ is a unitary representation
(see (\ref{eq:up1})) then there is a projection valued measure $P_{U}$
on the Borel subsets of $\widehat{G}$ such that
\begin{equation}
U\left(x\right)=\int_{\widehat{G}}\left\langle \lambda,x\right\rangle dP_{U}\left(\lambda\right),\quad x\in G,\label{eq:up3}
\end{equation}
where $\left\langle \lambda,x\right\rangle =\lambda\left(x\right)$,
for all $\lambda\in\widehat{G}$, and $x\in G$. The assertion in
(\ref{eq:up3}) is a theorem of Stone,\index{Theorem!Stone's-} Naimark,
Ambrose, and Godement (SNAG), see Section \ref{sub:lcg}. In (\ref{eq:up3}),
$P_{U}\left(\cdot\right)$ is defined on the Borel subsets $S$ in
$\widehat{G}$, and $P_{U}\left(S\right)$ is a projection, $P_{U}(\widehat{G})=I$;
and $S\mapsto P_{U}\left(S\right)$ is countably additive. \index{Theorem!SNAG-}

Since the Spectral Theorem serves as the central tool in the study
of measurements, one must be precise about the distinction between
linear operators with dense domain which are only \emph{Hermitian}
as opposed to \emph{selfadjoint}\footnote{We refer to Section \ref{sec:Prelim} for details. A Hermitian operator,
also called \emph{formally} selfadjoint, may well be non-selfadjoint.}. This distinction is accounted for by von Neumann\textquoteright s
theory of deficiency indices (see e.g., \cite{vN32a,Kre46,DS88b,AG93,Ne69}).\index{distribution!probability-}

(Starting with \cite{vN32a,vN32b,vN32c}, J. von Neumann and M. Stone
did pioneering work in the 1930s on spectral theory for unbounded
operators in Hilbert space; much of it in private correspondence.
The first named author has from conversations with M. Stone, that
the notions \textquotedblleft deficiency-index,\textquotedblright{}
and \textquotedblleft deficiency space\textquotedblright{} are due
to them; suggested by MS to vN as means of translating more classical
notions of \textquotedblleft boundary values\textquotedblright{} into
rigorous tools in abstract Hilbert space: closed subspaces, projections,
and dimension count.)

\index{deficiency indices}

\index{von Neumann, John}

\index{operator!unbounded}

\index{unitary representation}

\index{group!non-commutative}

\index{group!simply connected}

\index{Kolmogorov, A.}

\index{group!covering}

\index{universal covering group}

\index{projection-valued measure (PVM)}

\index{representation!unitary-}

\index{Theorem!Spectral-}

\index{Spectral Theorem}

\index{spectral types}

\index{boundary condition}

\index{Hermitian}

\section{\label{sec:stach}Stochastic Processes}

\emph{\dots{} from its shady beginnings devising gambling strategies
and counting corpses in medieval London, probability theory and statistical
inference now emerge as better foundations for scientific models,
especially those of the process of thinking and as essential ingredients
of theoretical mathematics, \dots{}} 

--- David Mumford. From: \textquotedblleft The Dawning of the Age
of Stochasticity.\textquotedblright{} \cite{MR1754778}\\
\\

\paragraph{\textbf{Early Roots}}

The interest in positive definite (p.d.) functions has at least three
roots: 
\begin{enumerate}
\item Fourier analysis, and harmonic analysis more generally, including
the non-commutative variant where we study unitary representations
of groups. (See \cite{Dev59,Nus75,Ru63,Ru70,Jor86,MR1004167,MR1069255,Jor91,MR2247899,KL14,Kle74,MS12,Ors79,OS73,MR829589,MR847352,MR774726}
and the papers cited there.)
\item Optimization and approximation problems, involving for example spline\index{spline}
approximations as envisioned by I. Schöenberg\index{Schöenberg, Isac}.
(See \cite{pol49,Sch38a,Sch38b,Sch64,SZ07,SZ09,PS03} and the papers
cited there.)
\item \label{enu:s3}Stochastic processes. (See \cite{Boc47,Ber46,Ito04,PaSc75,App09,Hi80,GSS83}
and the papers cited there.)
\end{enumerate}
Below, we sketch a few details regarding (\ref{enu:s3}). A \emph{stochastic
process} is an indexed family of random variables based on a fixed
probability space. In our present analysis, the processes will be
indexed by some group $G$ or by a subset of $G$. For example, $G=\mathbb{R}$,
or $G=\mathbb{Z}$, correspond to processes indexed by real time,
respectively discrete time. A main tool in the analysis of stochastic
processes is an associated \emph{covariance function}, see (\ref{eq:stat1}).

A process $\{X_{g}\:|\:g\in G\}$ is called \emph{Gaussian} if each
random variable $X_{g}$ is Gaussian, i.e., its distribution is Gaussian.
For Gaussian processes\index{Gaussian processes} we only need two
moments. So if we normalize, setting the mean equal to $0$, then
the process is determined by its covariance function. In general the
covariance function is a function on $G\times G$, or on a subset,
but if the process is \emph{stationary}, the covariance function will
in fact be a p.d. \index{positive definite} function defined on $G$,
or a subset of $G$. \index{distribution!Gaussian-}

We will be using three stochastic processes, Brownian motion, Brownian
Bridge, and the Ornstein-Uhlenbeck\index{Ornstein-Uhlenbeck} process,
all Gaussian, and Ito integrals.\index{Ito-integral}\index{probability space}

We outline a brief sketch of these facts below.

A \emph{probability space} is a triple $\left(\Omega,\mathscr{F},\mathbb{P}\right)$
where $\Omega$ is a set (sample space), $\mathscr{F}$ is a (fixed)
sigma-algebra of subsets of $\Omega$, and $\mathbb{P}$ is a (sigma-additive)
probability measure defined on $\mathscr{F}$. (Elements $E$ in $\mathscr{F}$
are ``events'', and $\mathbb{P}\left(E\right)$ represents the probability
of the event $E$.) \index{sigma-algebra}

A real valued \emph{random variable} is a function $X:\Omega\rightarrow\mathbb{R}$
such that, for every Borel subset $A\subset\mathbb{R}$, we have that
$X^{-1}\left(A\right)=\left\{ \omega\in\Omega\:|\:X\left(\omega\right)\in A\right\} $
is in $\mathscr{F}$. Then
\begin{equation}
\mu_{X}\left(A\right)=\mathbb{P}\left(X^{-1}\left(A\right)\right),\quad A\in\mathscr{B}\label{eq:pr1}
\end{equation}
defines a positive measure on $\mathbb{R}$; here $\mathscr{B}$ denotes
the Borel sigma-algebra of subsets of $\mathbb{R}$. This measure
is called the \emph{distribution} of $X$. For examples of common
distributions, see Table \ref{tab:Table-3}.

The following notation for the $\mathbb{P}$ integral of random variables
$X\left(\cdot\right)$ will be used:
\[
\mathbb{E}\left(X\right):=\int_{\Omega}X\left(\omega\right)d\mathbb{P}\left(\omega\right),
\]
denoted \emph{expectation}. If $\mu_{X}$ is the distribution of $X$,
and $\psi:\mathbb{R}\rightarrow\mathbb{R}$ is a Borel function, then
\[
\int_{\mathbb{R}}\psi\:d\mu_{X}=\mathbb{E}\left(\psi\circ X\right).
\]

An example of a probability space is as follows:
\begin{eqnarray}
\Omega & = & \prod_{\mathbb{N}}\left\{ \pm1\right\} =\mbox{infinite Cartesian product}\nonumber \\
 & = & \left\{ \left\{ \omega_{i}\right\} _{i\in\mathbb{N}}\:\big|\:\omega_{i}\in\left\{ \pm1\right\} ,\;\forall i\in\mathbb{N}\right\} ,\;\mbox{and}\label{eq:pr2}
\end{eqnarray}

\begin{itemize}
\item[$\mathscr{F}$:]  subsets of $\Omega$ specified by a finite number of outcomes (called
``cylinder sets''.)
\item[$\mathbb{P}$:]  the infinite-product measure corresponding to a fair coin $\left(\frac{1}{2},\frac{1}{2}\right)$
measure for each outcome $\omega_{i}$.
\end{itemize}
The transform
\begin{equation}
\widehat{\mu_{X}}\left(\lambda\right)=\int_{\mathbb{R}}e^{i\lambda x}d\mu_{X}\left(x\right)\label{eq:pr3}
\end{equation}
is called the \emph{Fourier transform}, or the \emph{generating function}.
\index{transform!Fourier-}

Let $a$ be fixed, $0<a<1$. A \emph{random $a$-power series} is
the function
\begin{equation}
X_{a}\left(\omega\right)=\sum_{i=1}^{\infty}\omega_{i}\:a^{i},\quad\omega=\left(\omega_{i}\right)\in\Omega.\label{eq:pr4}
\end{equation}
One checks (see Chapter \ref{chap:conv} below) that the generating
function for $X_{a}$ is as follows:
\begin{equation}
\widehat{d\mu_{X_{a}}}\left(\lambda\right)=\prod_{k=1}^{\infty}\cos(a^{k}\lambda),\quad\lambda\in\mathbb{R}\label{eq:pr5}
\end{equation}
where the r.h.s. in (\ref{eq:pr5}) is an infinite product. Note that
it is easy to check independently that the r.h.s. in (\ref{eq:pr5}),
$F_{a}\left(\lambda\right)=\prod_{k=1}^{\infty}\cos(a^{k}\lambda)$
is \emph{positive definite} and continuous on $\mathbb{R}$, and so
it determines a measure. See also Example \ref{exa:aseries}. \index{infinite Cartesian product}

An indexed family of random variables is called a \emph{stochastic
process}.
\begin{example}[Brownian motion]
\label{exa:bm}
\begin{itemize}
\item[$\Omega$:]  all continuous real valued function on $\mathbb{R}$;
\item[$\mathscr{F}$:]  subsets of $\Omega$ specified by a finite number of sample-points; 
\item[$\mathbb{P}$:]  Wiener-measure on $\left(\Omega,\mathscr{F}\right)$, see \cite{Hi80}.
\end{itemize}

For $\omega\in\Omega$, $t\in\mathbb{R}$, set $X_{t}\left(\omega\right)=\omega\left(t\right)$;
then it is well known that $\{X_{t}\}_{t\in\mathbb{R}_{+}}$ is a
Gaussian-random variable with the property that:
\begin{eqnarray}
d\mu_{X_{t}}\left(x\right) & = & \frac{1}{\sqrt{2\pi t}}e^{-x^{2}/2t}dx,\;x\in\mathbb{R},\;t>0,\label{eq:gau}\\
X_{0} & = & 0,\;\mbox{and}\nonumber 
\end{eqnarray}
whenever $0\leq t_{1}<t_{2}<\cdots<t_{n}$, then the random variables
\begin{equation}
X_{t_{1}},X_{t_{2}}-X_{t_{1}},\cdots,X_{t_{n}}-X_{t_{n-1}}\label{eq:it0}
\end{equation}
are independent; see \cite{Hi80}. (The r.h.s. in (\ref{eq:gau})
is Gaussian distribution with mean $0$ and variance $t>0$. See Figure
\ref{fig:gauss}.)

In more detail, $X_{t}$ satisfies:
\begin{enumerate}
\item $\mathbb{E}\left(X_{t}\right)=0$, for all $t$; mean zero;
\item $\mathbb{E}\left(X_{t}^{2}\right)=t$, variance $=t$;
\item $\mathbb{E}\left(X_{s}X_{t}\right)=s\wedge t$, the covariance function;
and
\item $\mathbb{E}\left(\left(X_{b_{1}}-X_{a_{1}}\right)\left(X_{b_{2}}-X_{a_{2}}\right)\right)=\left|\left[a_{1},b_{1}\right]\cap\left[a_{2},b_{2}\right]\right|$,
for any pair of intervals.
\end{enumerate}

This stochastic process is called \emph{Brownian motion} (see Figure
\ref{fig:bm2}). 

\end{example}
\Needspace*{3\baselineskip}
\begin{lemma}[The Ito integral \cite{Hi80}]
\label{lem:ito} Let $\{X_{t}\}_{t\in\mathbb{R}_{+}}$ be Brownian
motion, and let $f\in L^{2}\left(\mathbb{R}_{+}\right)$. For partitions
of $\mathbb{R}_{+}$, $\pi:\{t_{i}\}$, $t_{i}\leq t_{i+1}$, consider
the sums 
\begin{equation}
S\left(\pi\right):=\sum_{i}f\left(t_{i}\right)\left(X_{t_{i}}-X_{t_{i-1}}\right)\in L^{2}\left(\Omega,\mathbb{P}\right).\label{eq:it1}
\end{equation}
Then the limit (in $L^{2}\left(\Omega,\mathbb{P}\right)$) of the
terms (\ref{eq:it1}) exists, taking limit on the net of all partitions
s.t. $\max_{i}\left(t_{i+1}-t_{i}\right)\rightarrow0$. The limit
is denoted 
\begin{equation}
\int_{0}^{\infty}f\left(t\right)dX_{t}\in L^{2}\left(\Omega,\mathbb{P}\right),\label{eq:it2}
\end{equation}
and it is called the \uline{Ito-integral}. The following \uline{isometric}
property holds:
\begin{equation}
\mathbb{E}\left(\left|\int_{0}^{\infty}f\left(t\right)dX_{t}\right|^{2}\right)=\int_{0}^{\infty}\left|f\left(t\right)\right|^{2}dt.\label{eq:it3}
\end{equation}
Eq (\ref{eq:it3}) is called the Ito-isometry.\end{lemma}
\begin{svmultproof2}
We refer to \cite{Hi80} for an elegant presentation, but the key
step in the proof involves the above mentioned properties of Brownian
motion.

The first step is the verification of 
\[
\mathbb{E}\left(\left|S\left(\pi\right)\right|^{2}\right)\underset{\text{\ensuremath{\left(\text{see }\left(\ref{eq:it1}\right)\right)}}}{=}\sum_{i}\left|f\left(t_{i}\right)\right|^{2}\left(t_{i}-t_{i-1}\right),
\]
which is based on (\ref{eq:it0}). 
\end{svmultproof2}

\paragraph{\textbf{An application of Lemma \ref{lem:ito}: A positive definite
function on an infinite dimensional vector space.}}

Let $\mathcal{S}$ denote the real valued Schwartz functions (see
\cite{Tre06}). For $\varphi\in\mathcal{S}$, set $X\left(\varphi\right)=\int_{0}^{\infty}\varphi\left(t\right)dX_{t}$,
the Ito integral from (\ref{eq:it2}). Then we get the following:
\begin{equation}
\mathbb{E}\left(e^{iX\left(\varphi\right)}\right)=e^{-\frac{1}{2}\int_{0}^{\infty}\left|\varphi\left(t\right)\right|^{2}dt}\left(=e^{-\frac{1}{2}\left\Vert \varphi\right\Vert _{L^{2}}^{2}}\right),\label{eq:sc1}
\end{equation}
where $\mathbb{E}$ is the expectation w.r.t. Wiener-measure. 

It is immediate that 
\begin{equation}
F\left(\varphi\right):=e^{-\frac{1}{2}\left\Vert \varphi\right\Vert _{L^{2}}^{2}},\label{eq:sc2}
\end{equation}
i.e., the r.h.s. in (\ref{eq:sc1}), is a positive definite function
on $\mathcal{S}$. To get from this an associated probability measure
(the Wiener measure $\mathbb{P}$) is non-trivial, see e.g., \cite{Hi80,AJ12,AJL11}:
The dual of $\mathcal{S}$, the tempered distributions $\mathcal{S}'$,
turns into a measure space, $\left(\mathcal{S}',\mathscr{F},\mathbb{P}\right)$
with the sigma-algebra $\mathscr{F}$ generated by the cylinder sets
in $\mathcal{S}'$. With this we get an equivalent realization of
Wiener measure (see the cited papers); now with the l.h.s. in (\ref{eq:sc1})
as $\mathbb{E}\left(\cdots\right)=\int_{\mathcal{S}'}\cdots d\mathbb{P}\left(\cdot\right)$.
But the p.d. function $F$ in (\ref{eq:sc2}) \emph{cannot} be realized
by a sigma-additive measure on $L^{2}$, one must pass to a ``bigger''
infinite-dimensional vector space, hence $\mathcal{S}'$. The system
\begin{equation}
\mathcal{S}\hookrightarrow L_{\mathbb{R}}^{2}\left(\mathbb{R}\right)\hookrightarrow\mathcal{S}'\label{eq:sc3}
\end{equation}
is called a \emph{Gelfand-triple}. The second right hand side inclusion
$L^{2}\hookrightarrow\mathcal{S}'$ in (\ref{eq:sc3}) is obtained
by dualizing $\mathcal{S}\hookrightarrow L^{2}$, where $\mathcal{S}$
is given its Fréchet topology, see \cite{Tre06}. \index{Gelfand triple}

\begin{figure}
\includegraphics[width=0.6\textwidth]{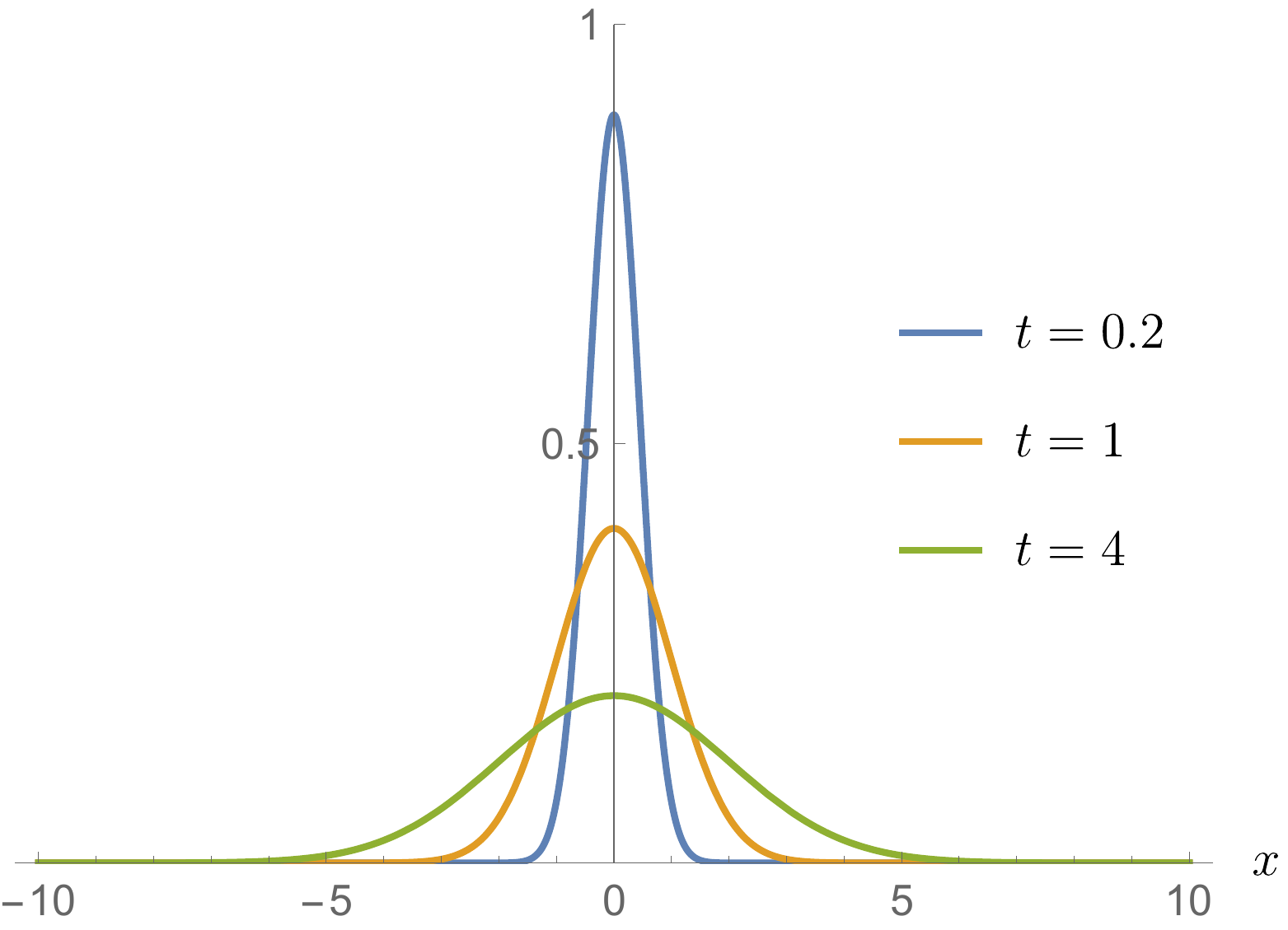}

\protect\caption{\label{fig:gauss}Gaussian distribution $d\mu_{X_{t}}\left(x\right)=\frac{1}{\sqrt{2\pi t}}e^{-x^{2}/2t}dx$,
$t>0$ (variance).}
\end{figure}

\begin{figure}
\includegraphics[width=0.7\textwidth]{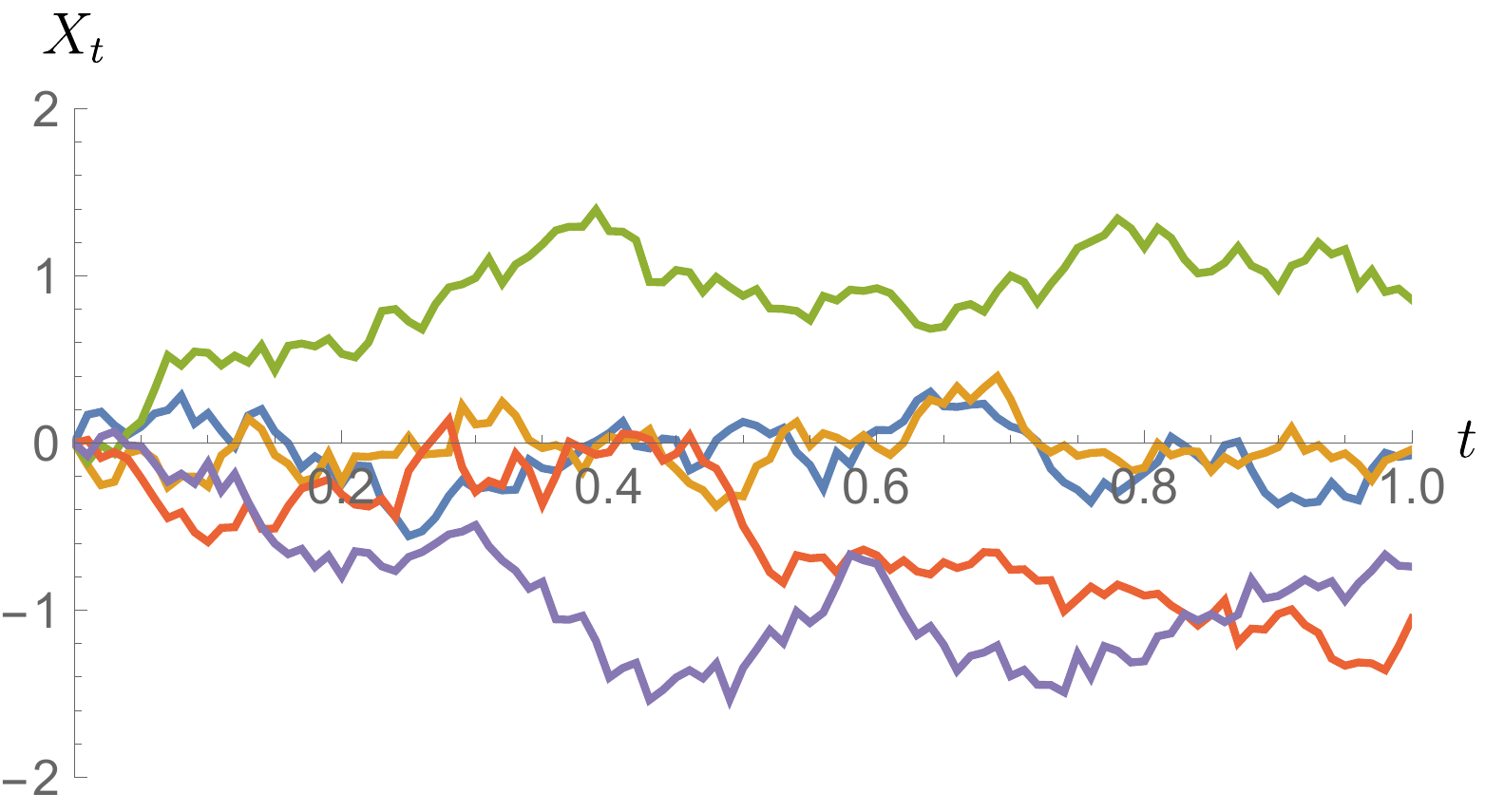}

\protect\caption{\label{fig:bm2}Monte-Carlo simulation of Brownian motion starting
at $x=0$, with 5 sample paths. (``Monte-Carlo simulation'' refers
to the use of computer-generated random numbers.) \index{Brownian motion}\index{Monte-Carlo simulation}}
\end{figure}

Let $G$ be a locally compact group, and let $\left(\Omega,\mathscr{F},\mathbb{P}\right)$
be a probability space, $\mathscr{F}$ a sigma-algebra, and $\mathbb{P}$
a probability measure \index{measure!probability} defined on $\mathscr{F}$.
A stochastic $L^{2}$-process is a system of random variables $\{X_{g}\}_{g\in G}$,
$X_{g}\in L^{2}\left(\Omega,\mathscr{F},\mathbb{P}\right)$. The \emph{covariance
function} $c_{X}:G\times G\rightarrow\mathbb{C}$ of the process is
given by\index{covariance} 
\begin{equation}
c_{X}\left(g_{1},g_{2}\right)=\mathbb{E}\left(\overline{X}_{g_{1}}X_{g_{2}}\right),\quad\forall\left(g_{1},g_{2}\right)\in G\times G.\label{eq:stat1}
\end{equation}
To simplify, we will assume that the mean $\mathbb{E}\left(X_{g}\right)=\int_{\Omega}X_{g}d\mathbb{P}\left(\omega\right)=0$
for all $g\in G$. 

We say that $\left(X_{g}\right)$ is stationary iff 
\begin{equation}
c_{X}\left(hg_{1},hg_{2}\right)=c_{X}\left(g_{1},g_{2}\right),\quad\forall h\in G.\label{eq:stat2}
\end{equation}
In this case $c_{X}$ is a function of $g_{1}^{-1}g_{2}$, i.e., 
\begin{equation}
\mathbb{E}\left(X_{g_{1}},X_{g_{2}}\right)=c_{X}\left(g_{1}^{-1}g_{2}\right),\quad\forall g_{1},g_{2}\in G;\label{eq:stat3}
\end{equation}
(setting $h=g_{1}^{-1}$ in (\ref{eq:stat2}).)\index{sigma-algebra}

The covariance function of Brownian motion $\mathbb{E}(X_{s}X_{t})$
is computed in Example \ref{exa:bm0} below.

\section{\label{sec:introRKHS}Overview of Applications of RKHSs}

In a general setup, reproducing kernel Hilbert spaces (RKHSs) were
pioneered by Aronszajn in the 1950s \cite{Aro50}; and subsequently
they have been used in a host of applications; e.g., \cite{SZ09,SZ07}. 

The key idea of Aronszajn is that a RKHS is a Hilbert space $\mathscr{H}_{K}$
of functions $f$ on a set such that the values $f(x)$ are \textquotedblleft reproduced\textquotedblright{}
from $f$ and a vector $K_{x}$ in $\mathscr{H}_{K}$, in such a way
that the inner product $\langle K_{x},K_{y}\rangle=:K\left(x,y\right)$
is a positive definite kernel. \index{RKHS}\index{Aronszajn, Nachman}

Since this setting is too general for many applications, it is useful
to restrict the very general framework for RKHSs to concrete cases
in the study of particular spectral theoretic problems; p.d. functions
on groups is a case in point. Such specific issues arise in physics
(e.g., \cite{Fal74,Jor07}) where one is faced with extending p.d.
functions $F$ which are only defined on a subset of a given group.
\index{Hilbert space}

\paragraph{\textbf{Connections to Gaussian Processes}}

By a theorem of Kolmogorov, every Hilbert space may be realized as
a (Gaussian) reproducing kernel Hilbert space (RKHS), see e.g., \cite{PaSc75,IM65,NF10},
and Theorem \ref{thm:kol} below.
\begin{definition}
A function $c$ defined on a subset of a group $G$ is said to be
\emph{positive definite} iff
\begin{equation}
\sum\nolimits _{i=1}^{n}\sum\nolimits _{j=1}^{n}\overline{\lambda_{i}}\lambda_{j}c\left(g_{i}^{-1}g_{j}\right)\geq0\label{eq:km1}
\end{equation}
for all $n\in\mathbb{N}$, and all $\left\{ \lambda_{i}\right\} _{i=1}^{n}\subset\mathbb{C}$,
$\left\{ g_{i}\right\} _{i=1}^{n}\subset G$ with $g_{i}^{-1}g_{j}$
in the domain of $c$. 
\end{definition}
From (\ref{eq:km1}), it follows that $F\left(g^{-1}\right)=\overline{F\left(g\right)}$,
and $\left|F\left(g\right)\right|\leq F\left(e\right)$, for all $g$
in the domain of $F$, where $e$ is the neutral element in $G$.

We recall the following theorem of Kolmogorov.\index{Kolmogorov, A.}
One direction is easy, and the other is the deep part:
\begin{theorem}[Kolmogorov]
\label{thm:kol} A function $c:G\rightarrow\mathbb{C}$ is positive
definite if and only if there is a stationary Gaussian process $\left(\Omega,\mathscr{F},\mathbb{P},X\right)$
with mean zero, such that $c=c_{X}$, i.e., $c\left(g_{1},g_{2}\right)=\mathbb{E}(\overline{X}_{g_{1}}X_{g_{2}})$;
see (\ref{eq:stat1}).\index{distribution!Gaussian-}\end{theorem}
\begin{svmultproof2}
We refer to \cite{PaSc75} for the non-trivial direction. To stress
the idea, we include a proof of the easy part of the theorem: 

Assume $c=c_{X}$. Let $\left\{ \lambda_{i}\right\} _{i=1}^{n}\subset\mathbb{C}$
and $\left\{ g_{i}\right\} _{i=1}^{n}\subset G$, then we have
\[
\sum\nolimits _{i}\sum\nolimits _{j}\overline{\lambda_{i}}\lambda_{j}c\left(g_{i}^{-1}g_{j}\right)=\mathbb{E}\big(\big|\sum\lambda_{i}X_{g_{i}}\big|^{2}\big)\geq0,
\]
i.e., $c$ is positive definite.\end{svmultproof2}

\begin{example}
\label{exa:bm0}Let $\Omega=\left[0,1\right]$, the closed unit interval,
and let $\mathscr{H}:=$ the space of continuous functions $\xi$
on $\Omega$ such that $\xi\left(0\right)=0$, and $\xi'\in L^{2}\left(0,1\right)$,
where $\xi'=\frac{d}{dx}\xi$ is the weak derivative of $\xi$, i.e.,
the derivative in the Schwartz-distribution sense. For $x,y\in\Omega$,
set \index{derivative!weak-} 
\begin{equation}
\begin{split}K\left(x,y\right) & =x\wedge y=\min\left(x,y\right);\;\mbox{and}\\
K_{x}\left(y\right) & =K\left(x,y\right).
\end{split}
\label{eq:bk1}
\end{equation}
Then in the sense of distribution, we have
\begin{equation}
\left(K_{x}\right)'=\chi_{\left[0,x\right]};\label{eq:bk2}
\end{equation}
i.e., the indicator function of the interval $\left[0,x\right]$,
see Figure \ref{fig:bk}.

For $\xi_{1},\xi_{2}\in\mathscr{H}$, set 
\[
\left\langle \xi_{1},\xi_{2}\right\rangle _{\mathscr{H}}:=\int_{0}^{1}\overline{\xi_{1}'\left(x\right)}\xi_{2}'\left(x\right)dx.
\]
Since $L^{2}\left(0,1\right)\subset L^{1}\left(0,1\right)$, and $\xi\left(0\right)=0$
for $\xi\in\mathscr{H}$, we see that 
\begin{equation}
\xi\left(x\right)=\int_{0}^{x}\xi'\left(y\right)dy,\quad\xi'\in L^{2}\left(0,1\right),\label{eq:bk3}
\end{equation}
and $\mathscr{H}$ consists of continuous functions on $\Omega$.\end{example}
\begin{claim}
The Hilbert space $\mathscr{H}$ is a RKHS with $\left\{ K_{x}\right\} _{x\in\Omega}$
as its kernel; see (\ref{eq:bk1}).\end{claim}
\begin{svmultproof2}
Let $\xi\in\mathscr{H}$, then by (\ref{eq:bk3}), we have:
\begin{align*}
\xi\left(x\right) & =\int_{0}^{1}\chi_{\left[0,x\right]}\left(y\right)\xi'\left(y\right)dy=\int_{0}^{1}K_{x}'\left(y\right)\xi'\left(y\right)dy\\
 & =\left\langle K_{x},\xi\right\rangle _{\mathscr{H}},\quad\forall x\in\Omega.
\end{align*}

\end{svmultproof2}

\begin{figure}
\begin{minipage}[t]{1\columnwidth}%
\subfloat[$K_{x}\left(\cdot\right)$]{\protect\includegraphics[width=0.35\textwidth]{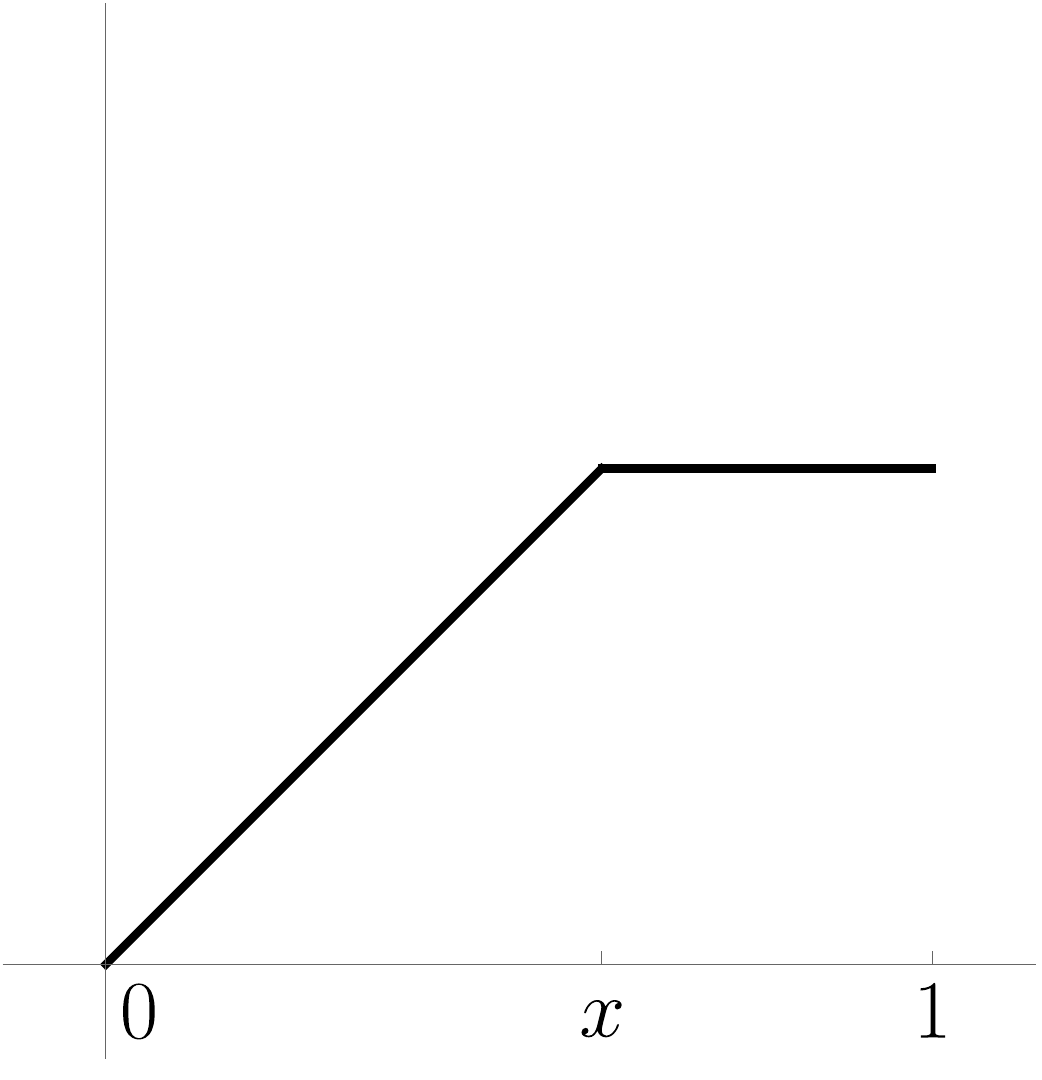}}\hfill{}\subfloat[{$\chi_{\left[0,x\right]}\left(\cdot\right)$}]{\protect\includegraphics[width=0.35\textwidth]{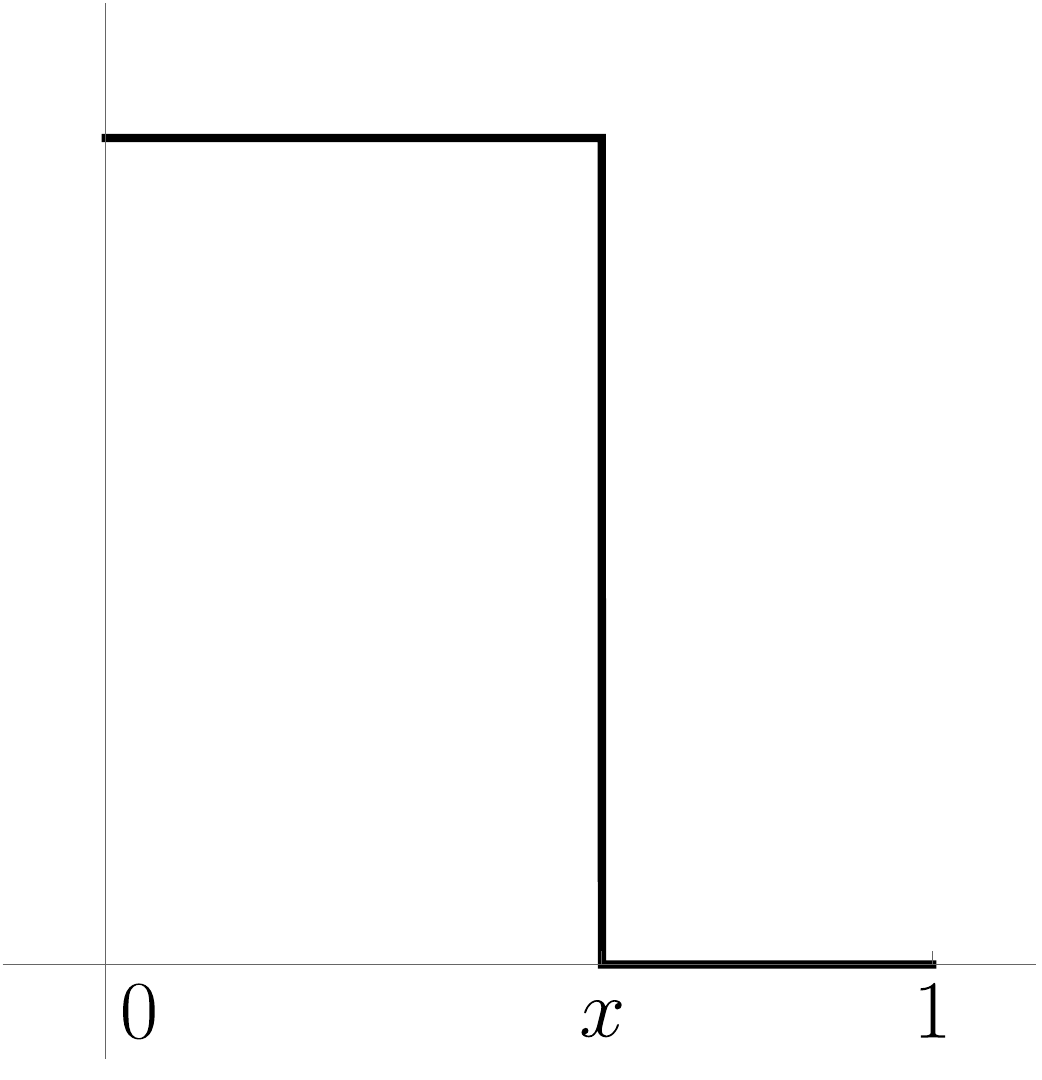}}%
\end{minipage}

\protect\caption{\label{fig:bk}$K_{x}$ and its distributional derivative. \index{derivative!distributional-}}
\end{figure}

\begin{remark}
In the case of Example \ref{exa:bm0} above, the Gaussian process
resulting from the p.d. kernel (\ref{eq:bk1}) is standard \emph{Brownian
motion} \cite{Hi80}, see Section \ref{sec:F2F3}. We shall return
to the p.d. kernel (\ref{eq:bk1}) in Chapters 2, 4, and 6. It is
the covariance kernel for standard Brownian motion in the interval
$\left[0,1\right]$.
\end{remark}

Subsequently, we shall be revisiting a number of specific instances
of these RKHSs. \index{RKHS}Our use of them ranges from the most
general case, when a continuous positive definite\index{positive definite}
(p.d.) function $F$ is defined on an open subset of a locally compact
group; the RKHS will be denoted $\mathscr{H}_{F}$. However, we stress
that the associated RKHS will depend on both the function $F$, and
on the subset of $G$ where $F$ is defined; hence on occasion, to
be specific about the subset, we shall index the RKHS by the pair
$\left(\Omega,F\right)$. If the choice of subset is implicit in the
context, we shall write simply $\mathscr{H}_{F}$. Depending on the
context, a particular RKHS typically will have any number of concrete,
hands-on realizations, allowing us thereby to remove the otherwise
obtuse abstraction entailed in its initial definition.

A glance at the Table of Contents indicates a large variety of the
classes of groups, and locally defined p.d. functions we consider,
and subsets. In each case, both the specific continuous, locally defined
p.d. function considered, \emph{and its domain} are important. Each
of the separate cases has definite applications. The most explicit
computations work best in the case when $G=\mathbb{R}$; and we offer
a number of applications in three areas: applications to stochastic
processes (Sections \ref{sec:Polya}-\ref{sec:hdim}), to harmonic
analysis (Sections \ref{sec:R^1}-\ref{sec:index11}), and to operator/spectral
theory (Section \ref{sec:mercer}).

\section{Earlier Papers}

Below we mention some earlier papers dealing with one or the other
of the two extension problems (i) or (ii) in Section \ref{sec:2ext}.
To begin with, there is a rich literature on (i), and a little on
(ii), but comparatively much less is known about \emph{their interconnections}.

As for positive definite (p.d.) functions, their use and applications
are extensive and includes such areas as stochastic processes, \index{stochastic processes}
see e.g., \cite{JorPea13,AJSV13,JP12,AJ12}; harmonic analysis (see
\cite{BCR84,JO00,JO98}, and the references there); potential theory
\cite{Fu74b,KL14}; operators in Hilbert space \cite{ADLR10,Al92,AD86,JN12};
and spectral theory \cite{AH13,Nus75,Dev72,Dev59}. We stress that
the literature is vast, and the above list is only a small sample.

Extensions of continuous p.d.\index{positive definite} functions
defined on subsets of Lie groups $G$ was studied in \cite{Jor91}.
In our present analysis of its connections to the extension questions
for associated operators in Hilbert space, we will be making use of
tools from spectral theory, and from the theory of reproducing kernel
Hilbert spaces, such as can be found in e.g., \cite{Ne69,Jor81,ABDS93,Aro50}.

There is a different kind of notion of positivity involving \emph{reflections},
\emph{restrictions}, and \emph{extensions}. It comes up in physics
and in stochastic processes, and is somewhat related to our present
theme. While they have several names, \textquotedblleft \emph{reflection
positivity}\textquotedblright{} is a popular term. In broad terms,
the issue is about realizing geometric reflections as \textquotedblleft conjugations\textquotedblright{}
in Hilbert space. When the program is successful, for a given unitary
representation $U$ of a Lie group $G$, it is possible to \emph{renormalize}
the Hilbert space on which $U$ is acting.

Now the Bochner\index{Bochner, S.} transform $F$ of a probability
measure (e.g., the distribution of a stochastic process) which further
satisfies reflection positivity, has two positivity properties: one
(i) because $F$ is the transform of a positive measure, so $F$ is
positive definite; and in addition the other, (ii) because of reflection
symmetry (see the discussion after Corollary \ref{cor:kernel}.) We
have not followed up below with structural characterizations of this
family of positive definite functions, but readers interested in the
theme will find details in \cite{JO00,JO98,Ar86,OS73}, and in the
references given there.

\index{representation!unitary-}

\index{Theorem!Bochner's-}

\index{operator!reflection-}

\index{reflection}

\index{Bochner's Theorem}

\index{reflection positivity}

\index{transform!Bochner-}

To help readers to appreciate other approaches, as well as current
research, we add below citation to some recent papers covering one
or the other of the many aspects of the extension problem, which is
the focus of our presentation here \cite{BT11,MR2422404,MR2316876,MR2247899,MR2215871,MR1876859}.

\section{\label{sec:Organization}Organization}

We begin with a quick summary of preliminaries, making clear our setting
and our choice of terminology. While we shall consider a host of variants
of related extension questions for positive definite (p.d.) functions,
the following theme is stressed in Chapter \ref{chap:ext} below:
Starting with a given and fixed, locally defined p.d. $F$ (say $F$
may be defined only in an open neighborhood in an ambient space),
then consider first the naturally associated Hilbert space, arising
from $F$ as a reproducing kernel Hilbert space (RKHS), $\mathscr{H}_{F}$.
The question is then: When is it possible to realize an extension
as a globally defined p.d. function $\widetilde{F}$, i.e., $\widetilde{F}$
extending $F$, in a construction which uses only spectral theory
for operators in the initial RKHS, $\mathscr{H}_{F}$? And when will
it be necessary to enlarge the initial Hilbert space, by passing to
a dilation Hilbert space $\mathscr{K}$ containing an isometric copy
of $\mathscr{H}_{F}$ itself? \index{dilation Hilbert space}

The monograph is organized around the following themes, some involving
\emph{dichotomies}; e.g., 
\begin{enumerate}
\item \label{enu:o-1}Abelian (sect \ref{sub:lcg}, ch \ref{chap:examples}-\ref{chap:Ext1})
vs non-Abelian (sect \ref{sub:lie}); 
\item \label{enu:o-2}simply connected (sect \ref{sub:exp(-|x|)}) vs connected
(sect \ref{sec:circle});
\item \label{enu:o-3}spectral theoretic (sect \ref{sub:lcg}, \ref{sub:lie},
\ref{sec:index 11}) vs geometric (sect \ref{sec:logz}); 
\item \label{enu:o-4}extensions of p.d. functions vs extensions of systems
of operators (sect \ref{sec:logz}); and
\item \label{enu:o-5}existence (ch \ref{chap:ext}) vs computation (ch
\ref{chap:Ext1}-\ref{chap:Greens}) and classification (ch \ref{chap:spbd}).
\end{enumerate}
Item (\ref{enu:o-1}) refers to the group $G$ under consideration.
In order to get started, we will need $G$ to be \emph{locally compact}
so it comes with \emph{Haar measure}, but it may be non-Abelian. It
may be a Lie group, or it may be non-locally Euclidean. In the other
end of this dichotomy, we look at $G=\mathbb{R}$, the real line.
In all cases, in the study of the themes from (\ref{enu:o-1}) it
is important whether the group is simply connected or not.

In order to quickly get to concrete examples, we begin the real line
$G=\mathbb{R}$ (Section \ref{sub:GR}), and $G=\mathbb{R}^{n}$,
$n>1$ (Section \ref{sub:euclid}); and the circle group, $G=\mathbb{T}=\mathbb{R}/\mathbb{Z}$
(Section \ref{sub:G=00003DT}).\index{Lie group} \index{Haar (see Haar measure)}

Of the other groups, we offer a systematic treatment of the classes
when $G$ is locally compact Abelian (Section \ref{sub:lcg}), and
the case of Lie groups (Section \ref{sub:lie}).

We note that the subdivision into classes of groups is necessary as
the theorems we prove in the case of $G=\mathbb{R}$ have a lot more
specificity than their counterparts do, for the more general classes
of groups. One reason for this is that our harmonic analysis relies
on unitary representations, and the \emph{non-commutative} theory
for \emph{unitary representations} is much more subtle than is the
Abelian counterpart.

Taking a choice of group $G$ as our starting point, we then study
continuous p.d.\index{positive definite} functions $F$ defined on
certain subsets in $G$. In the case of $G=\mathbb{R}$, our choice
of subset will be a finite open interval centered at $x=0$.

Our next step is to introduce a reproducing kernel Hilbert space (RKHS)\index{RKHS}
$\mathscr{H}_{F}$ that captures the properties of the given p.d.
function $F$. The nature and the harmonic analysis of this particular
RKHS are of independent interest; see Sections \ref{sec:introRKHS},
\ref{sec:embedding}, \ref{sub:euclid}, and \ref{sec:mercer}.

In Section \ref{sec:mercer}, we study a certain trace class \emph{integral
operator}, called the Mercer operator. A Mercer operator $T_{F}$
is naturally associated to a given a continuous and p.d. function
$F$ defined on the open interval, say $\left(-1,1\right)$. We use
$T_{F}$ in order to identify natural \index{Bessel frame}Bessel
frame in the corresponding RKHS $\mathscr{H}_{F}$. We then introduce
a notion of \emph{Shannon sampling} of finite Borel measures on $\mathbb{R}$,
sampling from integer points in $\mathbb{R}$. In Corollary \ref{cor:shan}
we then use this to give a necessary and sufficient condition for
a given finite Borel measure $\mu$ to fall in the convex set $Ext\left(F\right)$:
The measures in $Ext\left(F\right)$ are precisely those whose Shannon
sampling recover the given p.d. function $F$ on the interval $\left(-1,1\right)$.
\index{trace}\index{operator!integral-} \index{convex}\index{frame!Bessel-}

In the general case for $G$, the questions we address are as follows:
\begin{enumerate}[label=(\alph{enumi}),ref=\alph{enumi}]
\item \label{enu:pb1} What (if any) are the continuous p.d. functions
on $G$ which extend $F$? (Sections \ref{sec:embedding}, \ref{sub:exp(-|x|)},
and Chapters \ref{chap:types}, \ref{chap:Ext1}.) Denoting the set
of these extensions $Ext\left(F\right)$, then $Ext\left(F\right)$
is a compact convex set. Our next questions are:
\item \label{enu:pb2}What are the parameters for $Ext\left(F\right)$?
(Section \ref{sub:exp(-|x|)}, Chapters \ref{chap:types} and \ref{chap:question}.) 
\item How can we understand $Ext\left(F\right)$ from a generally non-commutative
extension problem for operators in $\mathscr{H}_{F}$? (See especially
Section \ref{sec:mercer}.)\index{extension problem}
\item \label{enu:pb4}We are further concerned with applications to scattering
theory (e.g., Theorem \ref{thm:R^n-spect}), and to commutative and
non-commutative harmonic analysis.\index{non-commutative harmonic analysis}
\item The unbounded operators we consider are defined naturally from given
p.d. functions $F$, and they have a common dense domain in the RKHS
$\mathscr{H}_{F}$. In studying possible selfadjoint operator extensions
in $\mathscr{H}_{F}$ we make use of von Neumann\textquoteright s
theory of deficiency indices\index{deficiency indices}. For concrete
cases, we must then find the deficiency indices; they must be equal,
but whether they are $\left(0,0\right)$, $\left(1,1\right)$, $\left(d,d\right)$,
$d>1$, is of great significance to the answers to the questions from
(\ref{enu:pb1})-(\ref{enu:pb4}).\index{von Neumann, John}
\item \label{enu:pb6}Finally, what is the relevance of the solutions in
(\ref{enu:pb1}) and (\ref{enu:pb2}) for the theory of operators
in Hilbert space and their harmonic (and spectral) analysis? (Sections
\ref{sub:euclid}, \ref{sec:index 11}, \ref{sec:index (d,d)}, and
Chapter \ref{chap:question}.)
\end{enumerate}
\emph{Citations for (\ref{enu:pb1})-(\ref{enu:pb6}) above.}
\begin{enumerate}[label=(\alph{enumi}),ref=\alph{enumi}]
\item \cite{Dev59,JN12,MR661628,Ru63,Ru70}.
\item \cite{JPT11-1,Sch38b,pol49,Nus75,MR1069255,Jor91}.
\item \cite{JM84,JT14,Nel59}.
\item \cite{LP89,JPT11-3,Maa10,JPT12}.
\item \cite{JPT11-1,DS88b,Ion01,MR661628,vN32a}.
\item \cite{MR755571,MR808690,JT14,Nel57}.
\end{enumerate}

\motto{It is nice to know that the computer understands the problem. But
I would like to understand it too. \textemdash{} Eugene Wigner}

\chapter{Extensions of Continuous Positive Definite Functions\label{chap:ext}}

We begin with a study of a family of reproducing kernel Hilbert spaces
(RKHSs) arising in connection with extension problems for positive
definite (p.d.) functions. While the extension problems make sense,
and are interesting, in a wider generality, we restrict attention
here to the case of continuous p.d. functions defined on open subsets
of groups $G$. We study two questions: 

(i) When is a given partially defined continuous p.d. function extendable
to the whole group $G$? In other words, when does it have continuous
p.d. extensions to $G$? 

(ii) When continuous p.d. extensions exist, what is the structure
of all continuous p.d. extensions? 

Because of available tools (mainly spectral theory for linear operators
in Hilbert space), we restrict here our focus as follows: In the Abelian
case, to when $G$ is locally compact; and if $G$ is non-Abelian,
we assume that it is a Lie group. But our most detailed results are
for the two cases $G=\mathbb{R}^{n}$, and $\mathbb{T}^{n}=\mathbb{R}^{n}/\mathbb{Z}^{n}$. 

Historically, the cases $n=1$ and $n=2$ are by far the most studied,
and they are also our main focus here. Available results are then
much more explicit, and the applications perhaps more far reaching.
Our results and examples for the case of $n=1$, we feel, are of independent
interest; and they are motivated by such applications as harmonic
analysis, sampling and interpolation theory, stochastic processes,
and Lax-Phillips scattering theory. One reason for the special significance
of the case $n=1$ is its connection to the theory of unbounded Hermitian
linear operators with prescribed dense domain in Hilbert space, and
their extensions. Indeed, if $n=1$, the possible continuous p.d.
extensions of given partially defined p.d. function $F$ are connected
with associated extensions of certain unbounded Hermitian linear operators;
in fact two types of such extensions: In one case, there are selfadjoint
extensions in the initial RKHS (Type I); and in another case, the
selfadjoint extensions necessarily must be realized in an enlargement
Hilbert space (Type II); so in a Hilbert space properly bigger than
the initial RKHS associated to $F$. 

\index{intertwining} \index{RKHS}\index{Type I}\index{Type II}\index{dilation Hilbert space}
\index{operator!intertwining-}\index{Hilbert space}

\section{\label{sec:Prelim}The RKHS $\mathscr{H}_{F}$}

In our theorems and proofs, we shall make use of reproducing kernel
Hilbert spaces (RKHSs), but the particular RKHSs we need here will
have additional properties (as compared to a general framework); which
allow us to give explicit formulas for our solutions. 

Our present setting is more restrictive in two ways: (i) we study
groups $G$, and translation-invariant kernels, and (ii) we further
impose continuity. By \textquotedblleft translation\textquotedblright{}
we mean relative to the operation in the particular group under discussion.
Our presentation below begins with the special case when $G$ is the
circle group $\mathbb{T}=\mathbb{R}/\mathbb{Z}$, or the real line
$\mathbb{R}$. 

For simplicity we focus on the case $G=\mathbb{R}$, indicating the
changes needed for $G=\mathbb{T}$. Modifications, if any, necessitated
by considering other groups $G$ will be described in the body of
the book. 
\begin{lemma}
\label{lem:Fdef}Let $\Omega$ be open subset of $\mathbb{R}^{n}$,
and let $F$ be a continuous function, defined on $\Omega-\Omega$;
then $F$ is positive definite (p.d.) if and only if the following
holds:\textup{ 
\[
\int_{\Omega}\int_{\Omega}\overline{\varphi(x)}\varphi(y)F(x-y)dxdy\geq0,\quad\forall\varphi\in C_{c}^{\infty}(\Omega).
\]
}\end{lemma}
\begin{svmultproof2}
Standard.
\end{svmultproof2}

Consider a continuous p.d. function $F:\Omega-\Omega\rightarrow\mathbb{C}$,
and set 
\begin{equation}
F_{y}\left(x\right):=F\left(x-y\right),\quad\forall x,y\in\Omega.\label{eq:H0}
\end{equation}
Let $\mathscr{H}_{F}$ be the \emph{reproducing kernel Hilbert space
(RKHS)}, which is the completion of 
\begin{equation}
\left\{ \sum\nolimits _{\text{finite}}c_{j}F_{x_{j}}\:\big|\ x_{j}\in\Omega,\:c_{j}\in\mathbb{C}\right\} \label{eq:H1}
\end{equation}
with respect to the inner product
\begin{equation}
\left\langle \sum\nolimits _{i}c_{i}F_{x_{i}},\sum\nolimits _{j}d_{j}F_{y_{j}}\right\rangle _{\mathscr{H}_{F}}:=\sum\nolimits _{i}\sum\nolimits _{j}\overline{c_{i}}d_{j}F\left(x_{i}-y_{j}\right);\label{eq:ip-discrete}
\end{equation}
modulo the subspace of functions with zero $\left\Vert \cdot\right\Vert _{\mathscr{H}_{F}}$-norm.
\begin{remark}
\noindent Throughout, we use the convention that the inner product
is conjugate linear in the first variable, and \uline{linear in
the second variable}. When more than one inner product is used, subscripts
will make reference to the Hilbert space. \end{remark}
\begin{lemma}
\label{lem:RKHS-def-by-integral}The RKHS, $\mathscr{H}_{F}$, is
the Hilbert completion of the functions 
\begin{equation}
F_{\varphi}\left(x\right)=\int_{\Omega}\varphi\left(y\right)F\left(x-y\right)dy,\quad\forall\varphi\in C_{c}^{\infty}\left(\Omega\right),\;x\in\Omega\label{eq:H2}
\end{equation}
with respect to the inner product
\begin{equation}
\left\langle F_{\varphi},F_{\psi}\right\rangle _{\mathscr{H}_{F}}=\int_{\Omega}\int_{\Omega}\overline{\varphi\left(x\right)}\psi\left(y\right)F\left(x-y\right)dxdy,\quad\forall\varphi,\psi\in C_{c}^{\infty}\left(\Omega\right).\label{eq:hi2}
\end{equation}
In particular, 
\begin{equation}
\left\Vert F_{\varphi}\right\Vert _{\mathscr{H}_{F}}^{2}=\int_{\Omega}\int_{\Omega}\overline{\varphi\left(x\right)}\varphi\left(y\right)F\left(x-y\right)dxdy.\label{eq:hn2}
\end{equation}
\end{lemma}
\begin{proof}
Indeed, setting $\varphi_{n,x}\left(t\right):=n\varphi\left(n\left(t-x\right)\right)$
(see Lemma \ref{lem:dense}), we have 
\begin{equation}
\left\Vert F_{\varphi_{n,x}}-F_{x}\right\Vert _{\mathscr{H}_{F}}\rightarrow0,\;\mbox{as }n\rightarrow\infty.\label{eq:approx-1}
\end{equation}
Hence $\left\{ F_{\varphi}\right\} _{\varphi\in C_{c}^{\infty}\left(\Omega\right)}$
spans a dense subspace in $\mathscr{H}_{F}$. For more details, see
\cite{Jor86,Jor87,MR1069255}.
\end{proof}

\index{measure!Dirac}
\begin{lemma}
\label{lem:dense}For $\Omega=\left(\alpha,\beta\right)$, let $a=\beta-\alpha$.
Let $\alpha<x<\beta$ and let $\varphi_{n,x}\left(t\right)=n\varphi\left(n\left(t-x\right)\right)$,
where $\varphi$ satisfies 
\begin{enumerate}
\item $\mathrm{supp}\left(\varphi\right)\subset\left(-a,a\right)$;
\item $\varphi\in C_{c}^{\infty}$, $\varphi\geq0$;
\item $\int\varphi\left(t\right)dt=1$. Note that $\lim_{n\rightarrow\infty}\varphi_{n,x}=\delta_{x}$,
the Dirac measure at $x$. 
\end{enumerate}
\begin{flushleft}
Then 
\begin{equation}
\lim_{n\rightarrow\infty}\left\Vert F_{\varphi_{n,x}}-F_{x}\right\Vert _{\mathscr{H}_{F}}=0\label{eq:approx}
\end{equation}
Hence $\left\{ F_{\varphi}\:|\:\varphi\in C_{c}^{\infty}\left(\Omega\right)\right\} $
spans a dense subspace in $\mathscr{H}_{F}$.
\par\end{flushleft}
\end{lemma}
\begin{figure}
\includegraphics[width=0.45\textwidth]{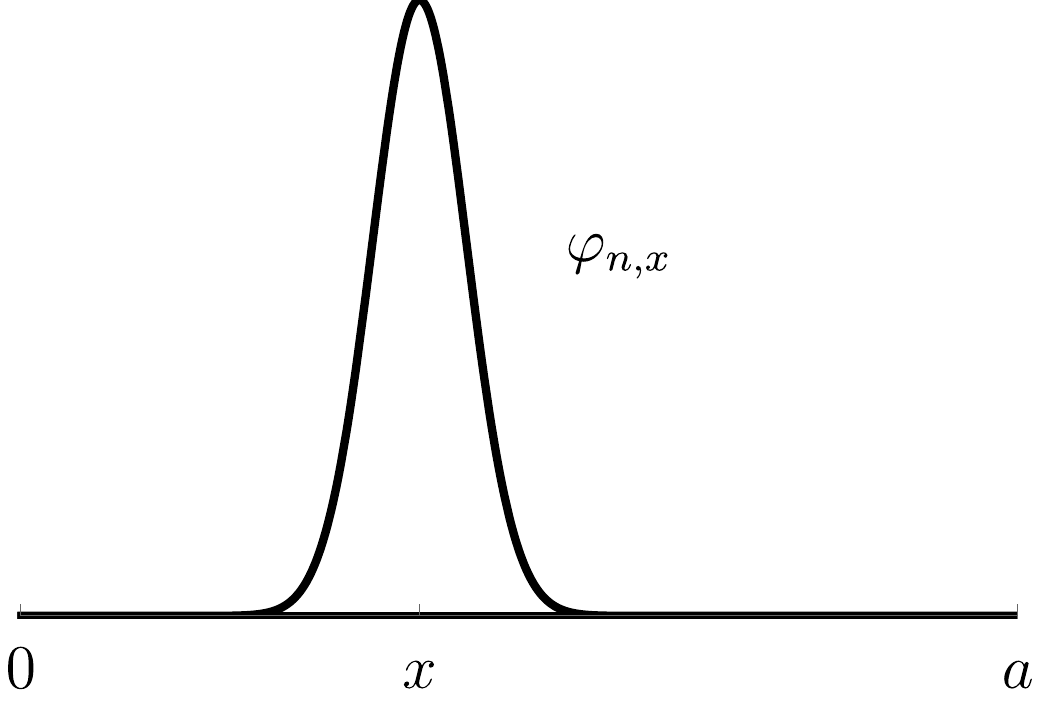}

\protect\caption{The approximate identity $\varphi_{n,x}\left(\cdot\right)$}
\end{figure}

The facts below about $\mathscr{H}_{F}$ follow from the general theory
of RKHS \cite{Aro50}:

\index{RKHS}
\begin{itemize}
\item For $F$ continuous, p.d., and non zero, $F(0)>0$, so we can always
arrange $F(0)=1.$
\item $F(-x)=\overline{F(x)}$, $\left|F\left(x\right)\right|\leq F\left(0\right)$
\item $\mathscr{H}_{F}$ consists of continuous functions $\xi:\Omega\rightarrow\mathbb{C}$.
\item The reproducing property holds:
\[
\left\langle F_{x},\xi\right\rangle _{\mathscr{H}_{F}}=\xi\left(x\right),\quad\forall\xi\in\mathscr{H}_{F},\:\forall x\in\Omega.
\]
This is a direct consequence of the definition of $\mathscr{H}_{F}$,
see (\ref{eq:ip-discrete}). 
\item If $F_{\phi_{n}}\to\xi$ in $\mathscr{H}_{F}$, then $F_{\phi_{n}}$
converges uniformly to $\xi$ in $\Omega$. In fact, the reproducing
property yields the estimate: 
\begin{align*}
\left|F_{\phi_{n}}\left(x\right)-\xi\left(x\right)\right| & =\left|\left\langle F_{x},F_{\phi_{n}}-\xi\right\rangle _{\mathscr{H}_{F}}\right|\\
 & \leq\left\Vert F_{x}\right\Vert _{\mathscr{H}_{F}}\left\Vert F_{\phi_{n}}-\xi\right\Vert _{\mathscr{H}_{F}}\\
 & =F\left(0\right)^{1/2}\left\Vert F_{\phi_{n}}-\xi\right\Vert _{\mathscr{H}_{F}}\xrightarrow{\;n\rightarrow\infty\;}0.
\end{align*}
\end{itemize}
\begin{theorem}
\label{thm:HF}A continuous function $\xi:\Omega\rightarrow\mathbb{C}$
is in $\mathscr{H}_{F}$ if and only if there exists $A_{0}>0$, such
that
\begin{equation}
\sum\nolimits _{i}\sum\nolimits _{j}\overline{c_{i}}c_{j}\overline{\xi\left(x_{i}\right)}\xi\left(x_{j}\right)\leq A_{0}\sum\nolimits _{i}\sum\nolimits _{j}\overline{c_{i}}c_{j}F\left(x_{i}-x_{j}\right)\label{eq:bdd}
\end{equation}
for all finite system $\left\{ c_{i}\right\} \subset\mathbb{C}$ and
$\left\{ x_{i}\right\} \subset\Omega$.

Equivalently, for all $\psi\in C_{c}^{\infty}\left(\Omega\right)$,
\begin{eqnarray}
\left|\int_{\Omega}\psi\left(y\right)\xi\left(y\right)dy\right|^{2} & \leq & A_{0}\int_{\Omega}\int_{\Omega}\overline{\psi\left(x\right)}\psi\left(y\right)F\left(x-y\right)dxdy.\label{eq:bdd2}
\end{eqnarray}
\end{theorem}
\begin{svmultproof2}
It suffices to check condition (\ref{eq:bdd2}). 

If (\ref{eq:bdd2}) holds, then $F_{\psi}\mapsto\int_{\Omega}\overline{\xi\left(x\right)}\psi\left(x\right)dx$
is a bounded linear functional on the dense subspace $\left\{ F_{\psi}\:|\:\psi\in C_{c}^{\infty}\left(\Omega\right)\right\} $
in $\mathscr{H}_{F}$ (Lemma \ref{lem:dense}), hence it extends to
$\mathscr{H}_{F}$. Therefore, $\exists!\:l_{\xi}\in\mathscr{H}_{F}$
s.t. 
\begin{eqnarray*}
\int_{\Omega}\overline{\xi\left(x\right)}\psi\left(x\right)dx & \overset{\left(\text{Riesz}\right)}{=} & \left\langle l_{\xi},F_{\psi}\right\rangle _{\mathscr{H}_{F}}\\
 & \overset{\left(\text{Fubini}\right)}{=} & \int_{\Omega}\psi\left(y\right)\left\langle l_{\xi},F_{y}\right\rangle _{\mathscr{H}_{F}}dy\\
 & = & \int_{\Omega}\overline{l_{\xi}\left(y\right)}\psi\left(y\right)dy,\quad\forall\psi\in C_{c}^{\infty}\left(\Omega\right);
\end{eqnarray*}
and this implies $\xi=l_{\xi}\in\mathscr{H}_{F}$. 

Conversely, assume $\xi\in\mathscr{H}_{F}$ then $\left|\langle\xi,F_{\psi}\rangle_{\mathscr{H}_{F}}\right|^{2}\leq\Vert\xi\Vert_{\mathscr{H}_{F}}^{2}\Vert F_{\psi}\Vert_{\mathscr{H}_{F}}^{2}$,
for all $\psi\in C_{c}^{\infty}\left(\Omega\right)$. Thus (\ref{eq:bdd2})
follows.
\end{svmultproof2}

These two conditions (\ref{eq:bdd})($\Longleftrightarrow$(\ref{eq:bdd2}))
are the best way to characterize elements in the Hilbert space $\mathscr{H}_{F}$;
see also Corollary \ref{cor:muHF}.

We will be using this when considering for example the \emph{deficiency-subspaces
for skew-symmetric operators} with dense domain in $\mathscr{H}_{F}$.

\index{skew-Hermitian operator; also called skew-symmetric}
\begin{example}
Let $G=\mathbb{T}=\mathbb{R}/\mathbb{Z}$, e.g., represented as $\left(-\frac{1}{2},\frac{1}{2}\right]$.
Fix $0<a<\frac{1}{2}$, then $\Omega-\Omega=\left(-a,a\right)$ mod
$\mathbb{Z}$. So, for example, (\ref{eq:F}) takes the form $F:(-a,a)\text{ mod }\mathbb{Z}\to\mathbb{C}$.
See Figure \ref{fig:Two-versions-of-T}. 
\end{example}
\begin{figure}[H]
\begin{tabular}{cc}
\includegraphics[width=0.4\textwidth]{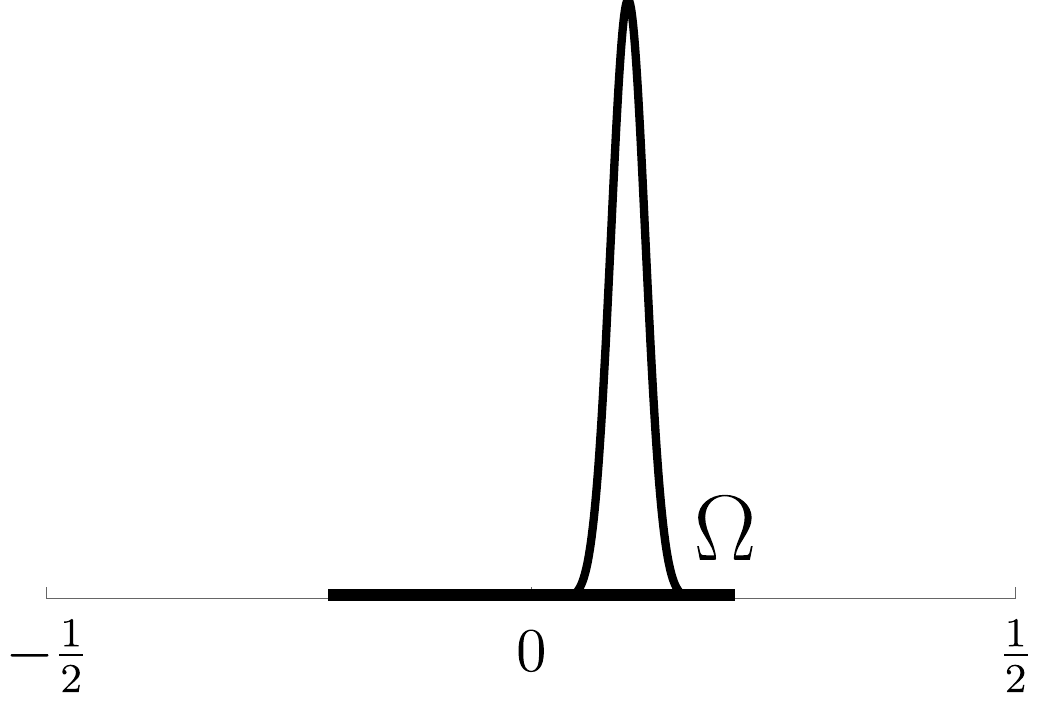}\hspace{1cm} & \includegraphics[width=0.3\textwidth]{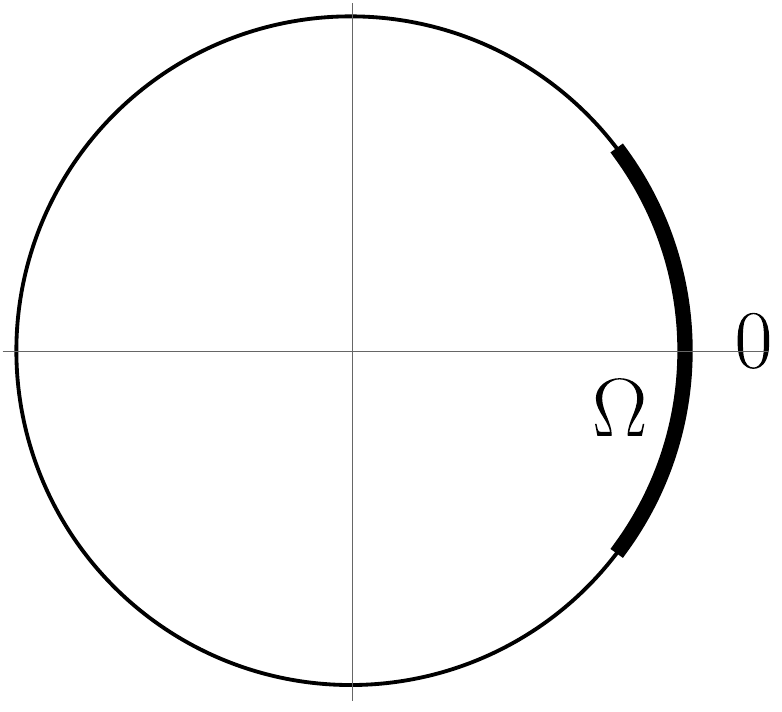}\tabularnewline
 & \tabularnewline
\end{tabular}

\protect\caption{\label{fig:Omega}\label{fig:Two-versions-of-T}Two versions of $\Omega=\left(0,a\right)\subset\mathbb{T}^{1}$. }
\end{figure}

\paragraph{\textbf{An Isometry}}

It is natural to extend the mapping from (\ref{eq:H2}) to measures,
i.e., extending 
\[
C_{c}^{\infty}\left(\Omega\right)\ni\varphi\longmapsto F_{\varphi}\in\mathscr{H}_{F}\left(\mbox{the RKHS}\right),
\]
by replacing $\varphi$ with a Borel measure $\mu$ on $\Omega$.
Specifically, set:
\begin{equation}
\left(S\mu\right)\left(x\right)=F_{\mu}\left(x\right):=\int_{\Omega}F\left(x-y\right)d\mu\left(y\right).\label{eq:tu1}
\end{equation}
Then the following result follows from the discussion above.
\begin{corollary}
\label{cor:muHF}~
\begin{enumerate}
\item Let $\mu$ be a Borel measure on $\Omega$ (possibly a signed measure),
then the following two conditions are equivalent:
\begin{align}
\int_{\Omega}\int_{\Omega} & F\left(x-y\right)\overline{d\mu\left(x\right)}d\mu\left(y\right)<\infty;\label{eq:tu1a}\\
 & \Updownarrow\nonumber \\
S\mu & \in\mathscr{H}_{F}\;\left(\text{see \ensuremath{\left(\ref{eq:tu1}\right)}}.\right)\label{eq:tu1b}
\end{align}

\item If the conditions hold, then 
\begin{equation}
\int_{\Omega}\int_{\Omega}F\left(x-y\right)d\mu\left(x\right)d\mu\left(y\right)=\left\Vert S\mu\right\Vert _{\mathscr{H}_{F}}^{2}\label{eq:tu2}
\end{equation}
(where the norm on the r.h.s. in (\ref{eq:tu2}) is the RKHS-norm.) 
\item And moreover, for all $\xi\in\mathscr{H}_{F}$, we have 
\begin{equation}
\left\langle \xi,S\mu\right\rangle _{\mathscr{H}_{F}}=\int_{\Omega}\overline{\xi\left(x\right)}d\mu\left(x\right).\label{eq:tu3}
\end{equation}

\item In particular, we note that all the Dirac measures $\mu=\delta_{x}$,
for $x\in\Omega$, satisfy (\ref{eq:tu1a}), and that 
\begin{equation}
S\left(\delta_{x}\right)=F\left(\cdot-x\right)\in\mathscr{H}_{F}.\label{eq:tu4}
\end{equation}

\item For every $x,y\in\Omega$, we have
\begin{equation}
\left\langle \delta_{x},\delta_{y}\right\rangle _{\mathfrak{M}_{2}\left(\Omega,F\right)}=F\left(x-y\right),\label{eq:tu41}
\end{equation}
where the r.h.s. in (\ref{eq:tu41}) refers to the $\mathfrak{M}_{2}$-Hilbert
inner product.
\end{enumerate}

In other words, the signed measures $\mu$ satisfying (\ref{eq:tu1a})
form a Hilbert space $\mathfrak{M}_{2}\left(\Omega,F\right)$, and
$S$ in (\ref{eq:tu1}) defines an \uline{isometry} of $\mathfrak{M}_{2}\left(\Omega,F\right)$
onto the RKHS $\mathscr{H}_{F}$. Here $\mathfrak{M_{2}}\left(\Omega,F\right)$
denotes the Hilbert space determined by condition (\ref{eq:tu1a})
above. 

\end{corollary}
\begin{example}
\label{exa:mexp}$G=\mathbb{R}$, $a>0$, $\Omega=\left(0,a\right)$,
\begin{equation}
F\left(x\right)=e^{-\left|x\right|},\quad\left|x\right|<a.\label{eq:m1}
\end{equation}
We now apply the isometry $S:\mathfrak{M}_{2}\left(F\right)\xrightarrow{\;\simeq\;}\mathscr{H}_{F}$
(onto) to this example. The two functions 
\[
\xi_{+}\left(x\right)=e^{-x},\;\mbox{and }\xi_{-}\left(x\right)=e^{x-a},\;x\in\Omega,
\]
will play an important role (defect-vectors) in Chapter \ref{chap:examples}
below. We have
\begin{eqnarray}
S^{*}\xi_{+} & = & \delta_{0}\in\mathfrak{M}_{2}\left(F\right),\;\mbox{and}\label{eq:m2}\\
S^{*}\xi_{-} & = & \delta_{a}\in\mathfrak{M}_{2}\left(F\right),\label{eq:m3}
\end{eqnarray}
where $\delta_{0}$ and $\delta_{a}$ denote the respective Dirac
measures. The proof of (\ref{eq:m2})-(\ref{eq:m3}) is immediate
from Corollary \ref{cor:muHF}.\end{example}
\begin{corollary}
\label{cor:Ddif}Let $F$ and $\Omega$ be as above; and assume further
that $F$ is $C^{\infty}$, then
\begin{equation}
F^{\left(n\right)}\left(\cdot-x\right)\in\mathscr{H}_{F},\quad\mbox{for all }x\in\Omega.\label{eq:tu5}
\end{equation}
\end{corollary}
\begin{svmultproof2}
We may establish (\ref{eq:tu5}) by induction, starting with the first
derivative.

Let $x\in\Omega$; then for sufficiently small $h$, $\left|h\right|<\varepsilon$,
we have $F\left(\cdot-x-h\right)\in\mathscr{H}_{F}$; and moreover,
\begin{eqnarray}
\left\Vert \frac{F\left(\cdot-x-h\right)-F\left(\cdot-x\right)}{h}\right\Vert _{\mathscr{H}_{F}}^{2} & = & \frac{1}{h^{2}}\left(F\left(0\right)+F\left(0\right)-2F\left(-h\right)\right)\label{eq:tu6}\\
 & = & \frac{1}{h}\left(\frac{F\left(0\right)-F\left(-h\right)}{h}-\frac{F\left(-h\right)-F\left(0\right)}{h}\right)\nonumber \\
 & \rightarrow & -F^{\left(2\right)}\left(0\right),\;\mbox{as}\;h\rightarrow0.\nonumber 
\end{eqnarray}
Hence, the limit of the difference quotient on the r.h.s. in (\ref{eq:tu6})
exists relative to the norm in $\mathscr{H}_{F}$, and so the limit
$-F'\left(\cdot-x\right)$ is in $\mathscr{H}_{F}$.
\end{svmultproof2}

We further get the following:
\begin{corollary}
If $F$ is p.d. and $C^{2}$ in a neighborhood of $0$, then $F^{\left(2\right)}\left(0\right)\leq0$. \end{corollary}
\begin{remark}
Note that, for $\xi\in\mathscr{H}_{F}$, (\ref{eq:tu1b}) is an integral
equation, i.e.,
\begin{equation}
\xi\left(x\right)=\int_{\Omega}F\left(x-y\right)d\nu\left(y\right);\;\mbox{solving for }\nu\in\mathfrak{M}_{2}\left(F\right).\label{eq:nu1}
\end{equation}
Using Corollary \ref{cor:Ddif}, we note that, when p.d. $F\in C^{\infty}\left(\Omega-\Omega\right)$,
then (\ref{eq:nu1}) also has distribution solutions $\nu$.
\end{remark}

\section{The Skew-Hermitian Operator $D^{(F)}$ in $\mathscr{H}_{F}$\label{sub:DF}}

In our discussion of the operator $D^{\left(F\right)}$ (Definition
\ref{def:D}), we shall make use of von Neumann's theory of symmetric
(or skew-symmetric) linear operators with dense domain in a fixed
Hilbert space; see e.g., \cite{vN32a,LP85,Kre46,JLW69,BrRo66,DS88b}. 

\index{skew-Hermitian operator; also called skew-symmetric}\index{Hilbert space}

The general setting is as follows: Let $\mathscr{H}$ be a complex
Hilbert space, and let $\mathscr{D}\subset\mathscr{H}$ be a dense
linear subspace. A linear operator $D$, defined on $\mathscr{D}$,
is said to be \emph{skew-symmetric} iff (Def) 
\begin{equation}
\left\langle Df,g\right\rangle _{\mathscr{H}}+\left\langle f,Dg\right\rangle _{\mathscr{H}}=0\label{eq:vn1}
\end{equation}
holds for all $f,g\in\mathscr{D}$. If we introduce the adjoint operator
$D^{*}$, then (\ref{eq:vn1}) is equivalent to the following containment
of graphs: \index{operator!graph of-}
\begin{equation}
D\subseteq-D^{*}.\label{eq:vn2}
\end{equation}
As in the setting of von Neumann, the domain of $D^{*}$, $dom(D^{*})$,
is as follows:\index{von Neumann, John}\index{operator!adjoint of an-}\index{operator!domain of-}
\begin{eqnarray}
dom(D^{*}) & = & \Big\{ g\in\mathscr{H}\:|\:\exists\:\mbox{const}\:C=C_{g}<\infty\;\mbox{s.t.}\;\nonumber \\
 &  & \left|\left\langle g,Df\right\rangle _{\mathscr{H}}\right|\leq C\left\Vert f\right\Vert _{\mathscr{H}},\;\forall f\in\mathscr{D}\Big\}.\label{eq:vn3}
\end{eqnarray}
And the vector $D^{*}g$ satisfies 
\begin{equation}
\left\langle D^{*}g,f\right\rangle _{\mathscr{H}}=\left\langle g,Df\right\rangle _{\mathscr{H}}\label{eq:vn4}
\end{equation}
for all $f\in\mathscr{D}$. 

An extension $A$ of $D$ is said to be skew-adjoint iff (Def)
\begin{equation}
A=-A^{*};\label{eq:vn5}
\end{equation}
which is ``$=$'', not merely containment, comparing (\ref{eq:vn2}).

We are interested in skew-adjoint extensions, since the Spectral Theorem\emph{
}applies to them; not to the operators which are merely skew-symmetric;
see \cite{DS88b}. From the Spectral Theorem, we then get solutions
to our original extension problem for locally defined p.d. functions.\index{Theorem!Spectral-}\index{skew-Hermitian operator; also called skew-symmetric}\index{Spectral Theorem}\index{operator!skew-adjoint-}
\index{locally defined}

But, in the general Hilbert space setting, skew-adjoint extensions
need not exist. However we have the following:
\begin{theorem}[von Neumann \cite{DS88b}]
A skew-symmetric operator $D$ has skew-adjoint extensions $A$,
i.e., 
\begin{equation}
D\subseteq A\subseteq-D^{*}\label{eq:vn6}
\end{equation}
if and only if the following two subspaces of $\mathscr{H}$ (deficiency
spaces, see Definition \ref{def:defsp}) have equal dimension:
\[
DEF{}^{\pm}\left(D\right)=\mathscr{N}\left(D^{*}\pm I\right),
\]
i.e., $\dim DEF{}^{+}\left(D\right)=\dim DEF{}^{-}\left(D\right)$,
where 
\[
DEF{}^{\pm}\left(D\right)=\left\{ g_{\pm}\in dom(D^{*})\;|\;D^{*}g_{\pm}=\mp g_{\pm}\right\} .
\]

\end{theorem}

Below, we introduce $D^{\left(F\right)}$ systematically and illustrate
how to use the corresponding skew-adjoint extensions to extend locally
defined p.d. functions. 

Fix a continuous p.d. function $F:\Omega-\Omega\rightarrow\mathbb{C}$,
where $\Omega=(0,a)$, a finite open interval in $\mathbb{R}$. Let
$\mathscr{H}_{F}$ be the corresponding RKHS. \index{positive definite}
\begin{definition}
\label{def:D}Set 
\begin{gather*}
dom(D^{\left(F\right)})=\left\{ F_{\psi}\:|\:\psi\in C_{c}^{\infty}\left(\Omega\right)\right\} ,\;\mbox{and}\\
D^{\left(F\right)}F_{\psi}=F_{\psi'},\quad\forall F_{\psi}\in dom(D^{\left(F\right)}),
\end{gather*}
where $\psi'\left(x\right)=d\psi/dx$, and $F_{\psi}$ is as in (\ref{eq:H2}).
\end{definition}
Note that the recipe for $D^{\left(F\right)}$ yields a well-defined
operator with dense domain in $\mathscr{H}_{F}$. To see this, use
Schwarz\textquoteright{} lemma to show that if $F_{\psi}=0$ in $\mathscr{H}_{F}$,
then it follows that the vector $F_{\psi'}\in\mathscr{H}_{F}$ is
0 as well. An alternative proof is given in Lemma \ref{lem:DF1}.
\begin{lemma}
\label{lem:DF}The operator $D^{\left(F\right)}$ is skew-symmetric
and densely defined in $\mathscr{H}_{F}$. \end{lemma}
\begin{svmultproof2}
By Lemma \ref{lem:RKHS-def-by-integral} $dom(D^{(F)})$ is dense
in $\mathscr{H}_{F}.$ If $\psi\in C_{c}^{\infty}\left(0,a\right)$
and $\left|t\right|<\mathrm{dist}\left(\mathrm{supp}\left(\psi\right),\mbox{endpoints}\right)$,
then
\begin{equation}
\left\Vert F_{\psi\left(\cdot+t\right)}\right\Vert _{\mathscr{H}_{F}}^{2}=\left\Vert F_{\psi}\right\Vert _{\mathscr{H}_{F}}^{2}=\int_{0}^{a}\int_{0}^{a}\overline{\psi\left(x\right)}\psi\left(y\right)F\left(x-y\right)dxdy
\end{equation}
by (\ref{eq:hn2}). Thus, 
\[
\frac{d}{dt}\left\Vert F_{\psi\left(\cdot+t\right)}\right\Vert _{\mathscr{H}_{F}}^{2}=0
\]
which is equivalent to 
\begin{equation}
\langle D^{\left(F\right)}F_{\psi},F_{\psi}\rangle_{\mathscr{H}_{F}}+\langle F_{\psi},D^{\left(F\right)}F_{\psi}\rangle_{\mathscr{H}_{F}}=0.
\end{equation}
It follows that $D^{\left(F\right)}$ is skew-symmetric. 

Lemma \ref{lem:DF1} below shows that $D^{\left(F\right)}$ is well-defined
on its dense domain in $\mathscr{H}_{F}$. This finishes the proof
of Lemma \ref{lem:DF}.\end{svmultproof2}

\begin{lemma}
\label{lem:DF1}The following implication holds:
\begin{eqnarray}
 & \left[\psi\in C_{c}^{\infty}\left(\Omega\right),\:F_{\psi}=0\mbox{ in }\mathscr{H}_{F}\right]\label{eq:H-1-1}\\
 & \Downarrow\nonumber \\
 & \left[F_{\psi'}=0\mbox{ in }\mathscr{H}_{F}\right]\label{eq:H-1-2}
\end{eqnarray}
\end{lemma}
\begin{svmultproof2}
Substituting (\ref{eq:H-1-1}) into 
\[
\left\langle F_{\varphi},F_{\psi'}\right\rangle _{\mathscr{H}_{F}}+\left\langle F_{\varphi'},F_{\psi}\right\rangle _{\mathscr{H}_{F}}=0
\]
we get 
\[
\left\langle F_{\varphi},F_{\psi'}\right\rangle _{\mathscr{H}_{F}}=0,\quad\forall\varphi\in C_{c}^{\infty}\left(\Omega\right).
\]
Taking $\varphi=\psi'$, yields
\[
\left\langle F_{\psi'},F_{\psi'}\right\rangle =\bigl\Vert F_{\psi'}\bigr\Vert_{\mathscr{H}_{F}}^{2}=0
\]
 which is the desired conclusion (\ref{eq:H-1-2}). Therefore, the
operator $D^{\left(F\right)}$ is well-defined.\end{svmultproof2}

\begin{definition}
\label{def:defsp}Let $(D^{\left(F\right)})^{*}$ be the adjoint of
$D^{\left(F\right)}$. The \emph{deficiency spaces} $DEF^{\pm}$ consists
of $\xi_{\pm}\in\mathscr{H}_{F}$, such that $(D^{\left(F\right)})^{*}\xi_{\pm}=\pm\xi_{\pm}$.
That is,
\[
DEF^{\pm}=\left\{ \xi_{\pm}\in\mathscr{H}_{F}:\left\langle F_{\psi'},\xi_{\pm}\right\rangle _{\mathscr{H}_{F}}=\left\langle F_{\psi},\pm\xi_{\pm}\right\rangle _{\mathscr{H}_{F}},\;\forall\psi\in C_{c}^{\infty}\left(\Omega\right)\right\} .
\]
Elements in $DEF^{\pm}$ are called \emph{defect vectors}. The dimensions
of $DEF^{\pm}$, i.e., the pair of numbers 
\[
d_{\pm}=\dim DEF^{\pm},
\]
are called the \emph{deficiency indices} of $D^{\left(F\right)}$.
See, e.g., \cite{vN32a,Kre46,DS88b,AG93,Ne69}. Example \ref{exa:mexp}
above is an instance of deficiency indices $\left(1,1\right)$, i.e.,
$d_{+}=d_{-}=1$. 
\end{definition}
von Neumann showed that a densely defined Hermitian operator in a
Hilbert space has equal deficiency indices if it commutes with a conjugation
operator \cite{DS88b}. This criterion is adapted to our setting in
Lemma \ref{lem:Conjugation-Operator} and its corollary. \index{von Neumann, John}\index{operator!conjugation-}\index{conjugation}\index{operator!adjoint of an-}\index{Hermitian}

\paragraph{\textbf{The Case of Conjugations}}

The purpose of the section below is to show that, for a large class
of locally defined p.d. functions $F$, it is possible to establish
existence of skew-adjoint extensions of the corresponding operator
$D^{\left(F\right)}$ with the use of a criterion of von Neumann:
It states that, if a skew-Hermitian operator anti-commutes with a
conjugation (a conjugate linear period-2 operator), then it must have
equal deficiency indices, and therefore have skew-adjoint extensions.
In the present case, the anti-commuting properties takes the form
of (\ref{eq:Conjugation}) below.
\begin{lemma}
\label{lem:Conjugation-Operator} Let $\Omega=(\alpha,\beta).$ Suppose
$F$ is a real-valued p.d. function defined on $\Omega-\Omega$. The
operator $J$ on $\mathscr{H}_{F}$ determined by
\[
JF_{\varphi}=\overline{F_{\varphi(\alpha+\beta-x)}},\quad\varphi\in C_{c}^{\infty}\left(\Omega\right)
\]
is a conjugation, i.e., $J$ is conjugate-linear, $J^{2}$ is the
identity operator, and 
\begin{equation}
\left\langle JF_{\phi},JF_{\psi}\right\rangle _{\mathscr{H}_{F}}=\left\langle F_{\psi},F_{\phi,}\right\rangle _{\mathscr{H}_{F}}.\label{eq:ConjugationIdentity}
\end{equation}
Moreover, 
\begin{equation}
D^{\left(F\right)}J=-JD^{\left(F\right)}.\label{eq:Conjugation}
\end{equation}
\end{lemma}
\begin{svmultproof2}
Let $a:=\alpha+\beta$ and $\phi\in C_{c}^{\infty}(\Omega).$ Since
$F$ is real-valued, we have 
\begin{align*}
JF_{\phi}(x) & =\int_{\alpha}^{\beta}\overline{\phi(a-y)}\,\overline{F(x-y)}dy\\
 & =\int_{\alpha}^{\beta}\psi(y)\,F(x-y)dy
\end{align*}
where $\psi(y):=\overline{\phi(a-y)}$ is in $C_{c}^{\infty}(\Omega)$.
It follows that $J$ maps the the operator domain $dom(D^{(F)})$
onto itself. For $\phi,\psi\in C_{c}^{\infty}(\Omega)$, 
\begin{align*}
\left\langle JF_{\phi},F_{\psi}\right\rangle _{\mathscr{H}_{F}} & =\int_{\alpha}^{\beta}F_{\phi(a-\cdot)}(x)\psi(x)dx\\
 & =\int_{\alpha}^{\beta}\int_{\alpha}^{\beta}\phi(a-y)F(x-y)\psi(x)dydx.
\end{align*}
Making the change of variables $(x,y)\to(a-x,a-y)$ and interchanging
the order of integration we see that
\begin{align*}
\left\langle JF_{\phi},F_{\psi}\right\rangle _{\mathscr{H}_{F}} & =\int_{\alpha}^{\beta}\int_{\alpha}^{\beta}\phi(y)F(y-x)\psi(a-x)dydx\\
 & =\int_{\alpha}^{\beta}\phi(y)F_{\psi(a-\cdot)}(y)dy\\
 & =\left\langle JF_{\psi},F_{\phi}\right\rangle _{\mathscr{H}_{F}},
\end{align*}
establishing (\ref{eq:ConjugationIdentity}). 

Finally, for all $\phi\in C_{c}^{\infty}(\Omega)$,
\[
JD^{(F)}F_{\phi}=\overline{F_{\phi'(a-\cdot)}}=-\overline{F_{\tfrac{d}{dx}\left(\phi\left(a-\cdot\right)\right)}}=-D^{(F)}JF_{\phi},
\]
hence (\ref{eq:Conjugation}) holds. \end{svmultproof2}

\begin{corollary}
If $F$ is real-valued, then $DEF^{+}$ and $DEF^{-}$ have the same
dimension.\end{corollary}
\begin{svmultproof2}
This follows from Lemma \ref{lem:Conjugation-Operator}, see e.g,
\cite{AG93} or \cite{DS88b}. 
\end{svmultproof2}

We proceed to characterize the deficiency spaces of $D^{\left(F\right)}$. 
\begin{lemma}
If $\xi\in DEF^{\pm}$ then $\xi(y)=\mathrm{constant}\,e^{\mp y}.$\end{lemma}
\begin{svmultproof2}
Specifically, $\xi\in DEF^{+}$ if and only if 
\[
\int_{0}^{a}\psi'\left(y\right)\xi\left(y\right)dy=\int_{0}^{a}\psi\left(y\right)\xi\left(y\right)dy,\quad\forall\psi\in C_{c}^{\infty}\left(0,a\right).
\]
Equivalently, $y\mapsto\xi\left(y\right)$ is a weak solution to the
ODE $\xi'=-\xi$, i.e., a strong solution in $C^{1}$. Thus, $\xi\left(y\right)=\mbox{constant }e^{-y}$.
The $DEF^{-}$ case is similar. \end{svmultproof2}

\begin{corollary}
Suppose $F$ is real-valued. Let $\xi_{\pm}(y):=e^{\mp y},$ for $y\in\Omega.$
Then $\xi_{+}\in\mathscr{H}_{F}$ iff $\xi_{-}\in\mathscr{H}_{F}.$
In the affirmative case $\left\Vert \xi_{-}\right\Vert _{\mathscr{H}_{F}}=e^{a}\left\Vert \xi_{+}\right\Vert _{\mathscr{H}_{F}}.$\end{corollary}
\begin{svmultproof2}
Let $J$ be the conjugation from Lemma \ref{lem:Conjugation-Operator}.
A short calculation:
\begin{align*}
\left\langle J\xi,F_{\phi}\right\rangle _{\mathscr{H_{F}}} & =\left\langle F_{\overline{\phi(a-\cdot)}},\xi\right\rangle _{\mathscr{H_{F}}}=\int\phi(a-x)\xi(x)dx\\
 & =\int\phi(x)\xi(a-x)dx=\left\langle \overline{\xi(a-\cdot)},F_{\phi}\right\rangle _{\mathscr{H_{F}}}
\end{align*}
shows that $\left(J\xi\right)(x)=\overline{\xi(a-x)}$, for $\xi\in\mathscr{H}_{F}$.
In particular, $J\xi_{-}=e^{a}\xi_{+}$. Since $\left\Vert J\xi_{-}\right\Vert _{\mathscr{H}_{F}}=\left\Vert \xi_{-}\right\Vert _{\mathscr{H}_{F}}$,
the proof is easily completed.\index{operator!conjugation-}\index{conjugation}\end{svmultproof2}

\begin{corollary}
\label{cor:DF}The deficiency indices of $D^{\left(F\right)}$, with
its dense domain in $\mathscr{H}_{F}$, are $\left(0,0\right)$, $(0,1),$
$(1,0),$ or $\left(1,1\right)$. 
\end{corollary}
The second case in the above corollary happens precisely when $y\mapsto e^{-y}\in\mathscr{H}_{F}$.
We can decide this with the use of (\ref{eq:bdd})($\Leftrightarrow$(\ref{eq:bdd2})). 

In Chapter \ref{chap:CompareFK} we will give some \emph{a priori}
estimates, which enable us to strengthen Corollary \ref{cor:DF}.
For this, see Corollary \ref{cor:RI}.
\begin{remark}
Note that deficiency indices $\left(1,1\right)$ is equivalent to
\begin{eqnarray}
\sum\nolimits _{i}\sum\nolimits _{j}\overline{c_{i}}c_{j}e^{-\left(x_{i}+x_{j}\right)} & \leq & A_{0}\sum\nolimits _{i}\sum\nolimits _{j}\overline{c_{i}}c_{j}F\left(x_{i}-x_{j}\right)\nonumber \\
 & \Updownarrow\label{eq:HF}\\
\left|\int_{0}^{a}\psi\left(y\right)e^{-y}dy\right|^{2} & \leq & A_{0}\int_{0}^{a}\int_{0}^{a}\overline{\psi\left(x\right)}\psi\left(y\right)F\left(x-y\right)dxdy\nonumber 
\end{eqnarray}
But it depends on $F$ (given on $\left(-a,a\right)$).\end{remark}
\begin{lemma}
\label{lem:UN}On $\mathbb{R}\times\mathbb{R}$, define the following
kernel $K_{+}\left(x,y\right)=e^{-\left|x+y\right|}$, $\left(x,y\right)\in\mathbb{R}\times\mathbb{R}$;
then this is a positive definite kernel on $\mathbb{R}_{+}\times\mathbb{R}_{+}$;
(see \cite{Aro50} for details on positive definite kernels.)\end{lemma}
\begin{svmultproof2}
Let $\{c_{j}\}\subset\mathbb{C}^{N}$ be a finite system of numbers,
and let $\left\{ x_{j}\right\} \subset\mathbb{R}_{+}^{N}$. Then 
\[
\sum\nolimits _{j}\sum\nolimits _{k}\overline{c_{j}}c_{k}e^{-\left(x_{j}+x_{k}\right)}=\left|\sum\nolimits _{j}c_{j}e^{-x_{j}}\right|^{2}\geq0.
\]
\end{svmultproof2}

\begin{corollary}
\label{cor:kernel}Let $F$, $\mathscr{H}_{F}$, and $D^{\left(F\right)}$
be as in Corollary \ref{cor:DF}; then $D^{\left(F\right)}$ has deficiency
indices $\left(1,1\right)$ if and only if the kernel $K_{+}\left(x,y\right)=e^{-\left|x+y\right|}$
is dominated by $K_{F}\left(x,y\right):=F\left(x-y\right)$ on $\left(0,a\right)\times\left(0,a\right)$,
i.e., there is a finite positive constant $A_{0}$ such that 
\[
A_{0}K_{F}-K_{+}
\]
is positive definite on $\left(0,a\right)\times\left(0,a\right)$. \end{corollary}
\begin{svmultproof2}
This is immediate from the lemma and (\ref{eq:HF}) above.
\end{svmultproof2}

In a general setting the kernels with the properties from Corollary
\ref{cor:kernel} are called \emph{reflection positive kernels}. See
Example \ref{exa:refpd}, and Section \ref{sub:exp(-|x|)}. Their
structure is accounted for by Theorem \ref{thm:bernstein}. Their
applications includes the study of Gaussian processes, the theory
of unitary representations of Lie groups, and quantum fields. The
following references give a glimpse into this area of analysis, \cite{JO98,JO00,Kle74}.
See also the papers and books cited there. \index{reflection}\index{operator!reflection-}
\begin{example}
\label{exa:refpd}The following are examples of positive definite
functions on $\mathbb{R}$ which are used in a variety of applications.
We discuss in Section \ref{sec:F2F3} how they arise as extensions
of locally defined p.d. functions, and some of the applications. In
the list below, $F:\mathbb{R}\rightarrow\mathbb{R}$ is p.d., and
$a,b>0$ are fixed constants. 
\begin{enumerate}
\item \label{enu:rp1}${\displaystyle F\left(x\right)=e^{-a\left|x\right|}}$;
Sections \ref{sec:circle} (p.\pageref{sec:circle}), \ref{sub:exp(-|x|)}
(p.\pageref{sub:exp(-|x|)}), \ref{sec:F2F3} (p.\pageref{sec:F2F3}).
\item ${\displaystyle F\left(x\right)=\frac{1-e^{-b\left|x\right|}}{b\left|x\right|}}$
\item \label{enu:rp3}${\displaystyle F\left(x\right)=\frac{1}{1+\left|x\right|}}$
\item ${\displaystyle F\left(x\right)=\frac{1}{\sqrt{1+\left|x\right|}}e^{-\frac{\left|x\right|}{1+\left|x\right|}}}$
\end{enumerate}

These are known to be generators of Gaussian reflection positive processes
\cite{JO98,JO00,Kle74}, i.e., having completely monotone covariance
functions. For details of these functions, see Theorem \ref{thm:bernstein}
(the Bernstein-Widder Theorem). The only one among these which is
also a Markov process is the one coming from $F\left(x\right)=e^{-a\left|x\right|}$,
$x\in\mathbb{R}$, where $a>0$ is a parameter; it is the Ornstein-Uhlenbeck
process. Note that $F\left(x\right)=e^{-\left|x\right|}$, $x\in\mathbb{R}$,
induces the positive definite kernel in Lemma \ref{lem:UN}, i.e.,
$K_{+}\left(x,y\right)=e^{-\left(x+y\right)}$, $\left(x,y\right)\in\mathbb{R}_{+}\times\mathbb{R}_{+}$.

For these reasons Example (\ref{enu:rp1}) in the above list shall
receive relatively more attention than the other three. Another reason
for Example (\ref{enu:rp1}) playing a prominent role is that almost
all the central questions dealing with locally defined p.d. functions
are especially transparent for this example.\index{locally defined}

\end{example}

More generally, a given p.d. function $F$ on $\mathbb{R}$ is said
to be \emph{reflection positive} if the induced kernel on $\mathbb{R}_{+}\times\mathbb{R}_{+}$,
given by
\[
K_{+}\left(x,y\right)=F\left(x+y\right),\quad\left(x,y\right)\in\mathbb{R}_{+}\times\mathbb{R}_{+},
\]
is positive definite.

In the case of a given p.d. function $F$ on $\mathbb{R}^{n}$, the
\emph{reflection positivity} making reference to some convex cone
$C_{+}$ in $\mathbb{R}^{n}$, and we say that $F$ is $C_{+}$-reflection
positive if 
\[
K_{+}\left(x,y\right):=F\left(x+y\right),\quad\left(x,y\right)\in C_{+}\times C_{+},
\]
is a positive definite kernel (on $C_{+}\times C_{+}$.)\index{reflection positivity}\index{convex}
\begin{example}
Let $\Omega=\left(-1,1\right)$ be the interval $-1<x<1$. The following
consideration illustrates the difference between positive definite
(p.d.) \emph{kernels} and positive definite \emph{functions}:
\begin{enumerate}
\item On $\Omega\times\Omega$, set $K\left(x,y\right):={\displaystyle \frac{1}{1-xy}}$.
\item On $\Omega$, set $F\left(x\right):={\displaystyle \frac{1}{1-x^{2}}}$. 
\end{enumerate}

Then $K$ is a p.d. kernel, but $F$ is \emph{not} a p.d. function. 

It is well known that $K$ is a p.d. kernel. The RKHS of $K$ consists
of analytic functions $\xi\left(x\right)=\sum_{n=0}^{\infty}c_{n}x^{n}$,
$\left|x\right|<1$, such that $\sum_{n=0}^{\infty}\left|c_{n}\right|^{2}<\infty$.
In fact, 
\[
\left\Vert \xi\right\Vert _{\mathscr{H}_{K}}^{2}=\sum_{n=0}^{\infty}\left|c_{n}\right|^{2}.
\]
To see that $F$ is \emph{not} a p.d. function, one checks that the
$2\times2$ matrix
\[
\begin{pmatrix}F\left(0\right) & F\left(x\right)\\
F\left(-x\right) & F\left(0\right)
\end{pmatrix}=\begin{pmatrix}1 & \frac{1}{1-x^{2}}\\
\frac{1}{1-x^{2}} & 1
\end{pmatrix}
\]
has a negative eigenvalue $\lambda\left(x\right)=\frac{-x^{2}}{1-x^{2}}$,
when $x\in\Omega\backslash\left\{ 0\right\} $. 
\end{example}

The following theorem, an inversion formula (see (\ref{eq:bi1})),
shows how the \emph{full} \emph{domain}, $-\infty<x<\infty$, of a
continuous positive definite function $F$ is needed in determining
the positive Borel measure $\mu$ which yields the Bochner-inversion,
i.e., $F=\widehat{d\mu}$.
\begin{theorem}[An inversion formula (see e.g., \cite{Akh65,LP89,DM76})]
Let $F$ be a continuous positive definite function on $\mathbb{R}$
such that $F\left(0\right)=1$. Let $\mu$ be the Borel probability
measure s.t. $F=\widehat{d\mu}$ (from Bochner's theorem), and let
$\left(a_{0},b_{0}\right)$ be a finite open interval, with the two-point
set $\left\{ a_{0},b_{0}\right\} $ of endpoints. Then
\begin{equation}
\mu\left(\left(a_{0},b_{0}\right)\right)+\frac{1}{2}\left(\left\{ a_{0},b_{0}\right\} \right)=\frac{1}{2\pi}\lim_{T\rightarrow\infty}\int_{-T}^{T}\frac{e^{-ixa_{0}}-e^{-ixb_{0}}}{ix}F\left(x\right)dx.\label{eq:bi1}
\end{equation}
(The second term on the l.h.s. in (\ref{eq:bi1}) is $\left(\mu\left(\left\{ a_{0}\right\} \right)+\mu\left(\left\{ b_{0}\right\} \right)\right)/2$.)\end{theorem}
\begin{svmultproof2}
The proof details are left to the reader. They are straightforward,
and also contained in many textbooks on harmonic analysis.\end{svmultproof2}

\begin{remark}
Suppose $F$ is continuous and positive definite, but is only known
on a finite centered interval $\left(-a,a\right)$, $a>0$. Formula
(\ref{eq:bi1}) now shows how \emph{distinct} positive definite extensions
$\widetilde{F}$ (to $\mathbb{R}$) for $F$ (on $\left(-a,a\right)$)
yields \emph{distinct measures} $\mu_{\widetilde{F}}$. However, a
measure $\mu$ cannot be determined (in general) from $F$ alone,
i.e., from $x$ in a finite interval $\left(-a,a\right)$.
\end{remark}

\subsection{\label{sub:GR}Illustration: $G=\mathbb{R}$, correspondence between
the two extension problems}

Extensions of continuous p.d. functions vs extensions of operators.
Illustration for the case of $\left(-a,a\right)\subset\mathbb{R}$,
i.e., given $F:\left(-a,a\right)\rightarrow\mathbb{C}$, continuous
and positive definite: 

We illustrate how to use the correspondence 
\[
\boxed{\mbox{extensions of p.d. functions}}\longleftrightarrow\boxed{\mbox{extensions of operators}}
\]
to get from $F\left(t\right)$ on $\left|t\right|<a$ to $\widetilde{F}\left(t\right)$,
$t\in\mathbb{R}$, with $\widetilde{F}=\widehat{d\mu}$. 

Figure \ref{fig:extc} illustrates the extension correspondence (p.d.
function vs extension operator) in the case of Type I, but each step
in the correspondence carries over to the Type II case. The main difference
is that for Type II, one must pass to a \emph{dilation Hilbert space},
i.e., a larger Hilbert space containing $\mathscr{H}_{F}$ as an isometric
copy.

\begin{flushleft}
Notations:
\par\end{flushleft}
\begin{enumerate}[label=,itemsep=5pt]
\item $\left\{ U\left(t\right)\right\} _{t\in\mathbb{R}}$: unitary one-parameter
group with generator $A^{\left(F\right)}\supset D^{\left(F\right)}$\index{infinitesimal generator}\index{unitary one-parameter group}
\item $P_{U}\left(\cdot\right)=$ the corresponding projection-valued measure
(PVM)
\item $U\left(t\right)=\int_{\mathbb{R}}e^{it\lambda}P_{U}\left(d\lambda\right)$
\item $d\mu\left(x\right)=\left\Vert P_{U}\left(dx\right)F_{0}\right\Vert _{\mathscr{H}_{F}}^{2}$
\item An extension: $\widetilde{F}\left(t\right)=\widehat{d\mu}=\int_{\mathbb{R}}e^{itx}d\mu\left(x\right)=\left\langle F_{0},U\left(t\right)F_{0}\right\rangle _{\mathscr{H}_{F}}$
\end{enumerate}
\begin{figure}
\[
\xymatrix{\left(F,\left|t\right|<a\right)\ar[d]\ar[rrr]\sp-{\mbox{p.d. extension}} &  &  & \widetilde{F}\left(t\right)\underset{t\in\mathbb{R}}{=}\left\langle F_{0},U\left(t\right)F_{0}\right\rangle _{\mathscr{H}_{F}}\\
(\mathscr{H}_{F},D^{\left(F\right)})\ar[d] &  &  & U\left(t\right)=\int_{\mathbb{R}}e^{it\lambda}P_{U}\left(d\lambda\right)\ar[u]\\
\mbox{\ensuremath{\begin{matrix}A^{\left(F\right)}\supset D^{\left(F\right)}\\
(A^{\left(F\right)})^{*}=-A^{\left(F\right)}
\end{matrix}}}\ar[rrr]\sb-{\mbox{operator extension}} &  &  & U\left(t\right)=e^{tA^{\left(F\right)}},\:t\in\mathbb{R}\ar[u]_{\mbox{resolution \ensuremath{P_{U}}}}^{\mbox{spectral}}
}
\]

\protect\caption{\label{fig:extc}Extension correspondence. From locally defined p.d
$F$ to $D^{\left(F\right)}$, to a skew-adjoint extension, and the
unitary one-parameter group, to spectral resolution, and finally to
an associated element in $Ext(F)$.\index{extension correspondence}}
\end{figure}
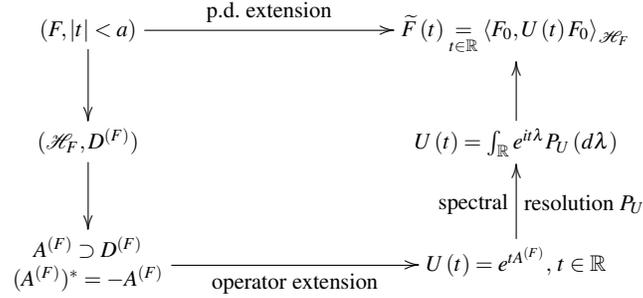

For the more general non-commutative correspondence (extension of
p.d. vs operator extension), we refer to Section \ref{sub:GNS}, the
GNS-construction. Compare with Figure \ref{fig:gns}.

Next, we flesh out in detail the role of $D^{\left(F\right)}$ in
extending a given p.d. function $F$, defined on $\Omega-\Omega$.
This will be continued in Section \ref{sec:embedding} in a more general
setting.

By Corollary \ref{cor:DF}, we conclude that there exists skew-adjoint
extension $A^{\left(F\right)}\supset D^{\left(F\right)}$ in $\mathscr{H}_{F}$.
That is, $dom(D^{\left(F\right)})\subset dom(A^{\left(F\right)})\subset\mathscr{H}_{F}$,
$(A^{\left(F\right)})^{*}=-A^{\left(F\right)}$, and 
\[
D^{\left(F\right)}=A^{\left(F\right)}\Big|_{dom(D^{\left(F\right)})}.
\]

Given a skew-adjoint extension $A^{\left(F\right)}\supset D^{\left(F\right)}$,
set $U\left(t\right)=e^{tA^{\left(F\right)}}:\mathscr{H}_{F}\rightarrow\mathscr{H}_{F}$,
and get the unitary one-parameter group 
\[
\left\{ U\left(t\right):t\in\mathbb{R}\right\} ,\;U\left(s+t\right)=U\left(s\right)U\left(t\right),\;\forall s,t\in\mathbb{R};
\]
and if 
\[
\xi\in dom(A^{\left(F\right)})=\left\{ \xi\in\mathscr{H}_{F}\:|\:s.t.-\lim_{t\rightarrow0}\frac{U\left(t\right)\xi-\xi}{t}\:\mbox{exists}\right\} 
\]
then 
\begin{equation}
A^{\left(F\right)}\xi=\lim_{t\rightarrow0}\frac{U\left(t\right)\xi-\xi}{t}.
\end{equation}

Now let 
\begin{equation}
\widetilde{F}_{A}\left(t\right):=\left\langle F_{0},U\left(t\right)F_{0}\right\rangle _{\mathscr{H}_{F}},\quad\forall t\in\mathbb{R}.\label{eq:Fext}
\end{equation}
Using (\ref{eq:approx}), we see that $\widetilde{F}_{A}\left(t\right)$,
defined on $\mathbb{R}$, is a continuous p.d. extension of $F$.
\index{operator!unitary one-parameter group}\index{unitary one-parameter group}

\index{skew-Hermitian operator; also called skew-symmetric}
\begin{lemma}
\label{lem:Ut}$\widetilde{F}_{A}\left(t\right)$ as in (\ref{eq:Fext})
is a continuous bounded p.d. function of $\mathbb{R}$, and
\begin{equation}
\widetilde{F}_{A}\left(t\right)=F\left(t\right),\quad t\in\left(-a,a\right).
\end{equation}
\end{lemma}
\begin{svmultproof2}
Since $\left\{ U\left(t\right)\right\} $ is a strongly continuous
unitary group acting on $\mathscr{H}_{F}$, we have 
\[
\left|\widetilde{F}_{A}\left(t\right)\right|=\left|\left\langle F_{0},U\left(t\right)F_{0}\right\rangle \right|\leq\left\Vert F_{0}\right\Vert _{\mathscr{H}_{F}}\left\Vert U\left(t\right)F_{0}\right\Vert _{\mathscr{H}_{F}}=\left\Vert F_{0}\right\Vert _{\mathscr{H}_{F}}^{2}=F\left(0\right),
\]
by (\ref{eq:ip-discrete}). Recall that $F_{x}=F\left(\cdot-x\right)$,
$x\in\Omega=\left(0,a\right)$. This shows that every $F^{ext}$ is
bounded and continuous.

The proof that $\widetilde{F}_{A}\left(t\right)$ indeed extends $F$
to $\mathbb{R}$ holds in a more general context, see e.g., Theorem
\ref{thm:pd-extension-bigger-H-space} and \cite{MR1004167,MR1069255,Jor91}. 
\end{svmultproof2}

\index{measure!projection-valued}

\index{projection-valued measure (PVM)} \index{strongly continuous}

Recall that $F$ can always be normalized by $F\left(0\right)=1$.
Consider the spectral representation:\index{representation!spectral-}
\begin{equation}
U\left(t\right)=e^{tA^{\left(F\right)}}=\int_{-\infty}^{\infty}e^{i\lambda t}P\left(d\lambda\right)\label{eq:Ut}
\end{equation}
where $P\left(\cdot\right)$ is the projection-valued measure of $A^{\left(F\right)}$.
Thus, $P\left(B\right):\mathscr{H}_{F}\rightarrow\mathscr{H}_{F}$
is a projection, for all Borel subsets $B$ in $\mathbb{R}$. Setting
\begin{equation}
d\mu\left(\lambda\right)=\left\Vert P\left(d\lambda\right)F_{0}\right\Vert _{\mathscr{H}_{F}}^{2}\label{eq:Ut1}
\end{equation}
then the corresponding extension is as follows: 
\begin{equation}
\widetilde{F}_{A}\left(t\right)=\int_{-\infty}^{\infty}e^{i\lambda t}d\mu\left(\lambda\right)=\widehat{d\mu}\left(t\right),\quad\forall t\in\mathbb{R}.\label{eq:Fext1}
\end{equation}

\textbf{Conclusion.} The extension $\widetilde{F}_{A}\left(t\right)$
from (\ref{eq:Fext}) has nice transform properties, and via (\ref{eq:Fext1})
we get 
\[
\mathscr{H}_{\widetilde{F}_{A}}\simeq L^{2}\left(\mathbb{R},\mu\right)
\]
with the transformation $T_{\mu}$ of Theorem \ref{thm:pd-extension-bigger-H-space};
where $\mathscr{H}_{\widetilde{F}_{A}}$ is the RKHS of the extended
p.d. function $\widetilde{F}_{A}$. The explicit transform $T_{\mu}$
realizing $\mathscr{H}_{F}\xrightarrow{\;\simeq\;}L^{2}\left(\mu\right)$
is given in Corollary \ref{cor:lcg-isom}. 
\begin{corollary}
\label{cor:sp}Let a partially defined p.d. function $F$ be given
as in Figure \ref{fig:extc}. Let $D^{\left(F\right)}$ be the associated
skew-Hermitian operator in $\mathscr{H}_{F}$, and assume that it
has a skew-adjoint extension $A$. Let $P_{A}\left(\cdot\right)$
and $\mu_{A}\left(\cdot\right)$ be the measures from (\ref{eq:Ut})
and (\ref{eq:Ut1}), then $\widetilde{F}_{A}=\widehat{d\mu_{A}}$
is a Type I extension, and 
\[
suppt\left(\mu_{A}\right)=i\cdot spect\left(A\right),\quad i=\sqrt{-1}.
\]
\end{corollary}
\begin{svmultproof2}
Immediate from the proof of Lemma \ref{lem:Ut}.\end{svmultproof2}

\begin{remark}
In the circle case $G=\mathbb{T}$, the extension $\widetilde{F}_{A}$
in (\ref{eq:Fext}) needs not be $\mathbb{Z}$-periodic.
\end{remark}
Consider $a>0$, $\Omega=\left(0,a\right)$, and a continuous p.d.
function $F$ on $\left(-a,a\right)$. Let $D^{\left(F\right)}$ be
the corresponding skew-Hermitian operator in $\mathscr{H}_{F}$. We
proved that for every skew-adjoint extension $A\supset D^{\left(F\right)}$
in $\mathscr{H}_{F}$, the corresponding p.d. function 
\begin{equation}
F_{A}\left(t\right)=\left\langle F_{0},e^{tA}F_{0}\right\rangle _{\mathscr{H}_{F}}=\int_{\mathbb{R}}e^{it\lambda}\left\Vert P_{A}\left(d\lambda\right)F_{0}\right\Vert _{\mathscr{H}_{F}}^{2},\quad t\in\mathbb{R}\label{eq:ta1}
\end{equation}
is a Type I p.d. extension of $F$; see formulas (\ref{eq:Ut1})-(\ref{eq:Fext1}).
\begin{proposition}
Let $F$ be continuous and p.d. on $\left(-a,a\right)$; and let $\widetilde{F}$
be a Type I positive definite extension to $\mathbb{R}$, i.e., 
\begin{equation}
F\left(t\right)=\widetilde{F}\left(t\right),\quad\forall t\in\left(-a,a\right).\label{eq:ta2}
\end{equation}
Then there is a skew-adjoint extension $A$ of $D^{\left(F\right)}$
such that $\widetilde{F}=F_{A}$ on $\mathbb{R}$; see (\ref{eq:ta1})
above.\end{proposition}
\begin{svmultproof2}
By the definition of Type I, we know that there is a strongly continuous
unitary one-parameter group $\{U\left(t\right)\}_{t\in\mathbb{R}}$,
acting in $\mathscr{H}_{F}$; such that 
\begin{equation}
F\left(t\right)=\left\langle F_{0},U\left(t\right)F_{0}\right\rangle _{\mathscr{H}_{F}},\quad\forall t,\;\left|t\right|<a.\label{eq:ta3}
\end{equation}
By Stone's Theorem, there is a unique skew-adjoint operator $A$ in
$\mathscr{H}_{F}$ such that $U\left(t\right)=e^{tA}$, $\forall t\in\mathbb{R}$;
and so the r.h.s. is $F_{A}=\widetilde{F}$; i.e., the given Type
I extension $\widetilde{F}$ has the form $F_{A}$. But differentiation,
$d/dt$ at $t=0$, in (\ref{eq:ta3}) shows that $A$ is an extension
of $D^{\left(F\right)}$ which is the desired conclusion.\index{Theorem!Stone's-}
\end{svmultproof2}

\section{\label{sec:embedding}Enlarging the Hilbert Space}

The purpose of this section is to describe the dilation-Hilbert space
in detail, and to prove some lemmas which will then be used in Chapter
\ref{chap:types}. In Chapter \ref{chap:types}, we identify extensions
of the initial positive definite (p.d.) function $F$ which are associated
with operator extensions in $\mathscr{H}_{F}$ (Type I), and those
which require an enlargement of $\mathscr{H}_{F}$ (Type II).\index{dilation Hilbert space}\index{Hilbert space}

To simplify notations, results in this section are formulated for
$G=\mathbb{R}$. The modification for general Lie groups is straightforward,
and are left for the reader. 

Fix $a>0$, and $\Omega=\left(0,a\right)\subset\mathbb{R}$. Let $F:\Omega-\Omega\to\mathbb{C}$
be a continuous p.d. function. Recall the corresponding reproducing
kernel Hilbert space (RKHS) $\mathscr{H}_{F}$ is the completion of
$span\left\{ F_{x}:=F\left(\cdot-x\right)\:|\:x\in\Omega\right\} $
with respect to the inner product 
\begin{equation}
\left\langle F_{x},F_{y}\right\rangle _{\mathscr{H}_{F}}:=F_{y}\left(x\right)=F\left(x-y\right),\quad\forall x,y\in\Omega,\label{eq:e.1}
\end{equation}
modulo elements of zero $\mathscr{H}_{F}$-norm.

Following Section \ref{sec:Prelim}, let $F_{\varphi}=\varphi*F$
be the convolution, where 
\[
F_{\varphi}\left(x\right)=\int_{\Omega}\varphi\left(y\right)F\left(x-y\right)dy,\quad x\in\Omega,\;\varphi\in C_{c}^{\infty}\left(\Omega\right).
\]
Setting 
\[
\pi\left(\varphi\right)F_{0}=F_{\varphi},\;\mbox{and}\;\varphi^{\#}\left(x\right)=\overline{\varphi\left(-x\right)},
\]
we may write 
\begin{equation}
\left\langle F_{\varphi},F_{\psi}\right\rangle _{\mathscr{H}_{F}}=\left\langle F_{0},\pi\left(\varphi^{\#}*\psi\right)F_{0}\right\rangle _{\mathscr{H}_{F}}=\left\langle \pi\left(\varphi\right)F_{0},\pi\left(\psi\right)F_{0}\right\rangle _{\mathscr{H}_{F}}.\label{eq:e.1.1}
\end{equation}

The following theorem also holds in $\mathbb{R}^{n}$ with $n>1$.
It is stated here for $n=1$ to illustrate the ``enlargement'' of
$\mathscr{H}_{F}$ question. 
\begin{theorem}
\label{thm:pd-extension-bigger-H-space}The following two conditions
are equivalent:\index{isometry}
\begin{enumerate}
\item \label{enu:eng1}$F$ is extendable to a continuous p.d. function
$\widetilde{F}$ defined on $\mathbb{R}$, i.e., $\widetilde{F}$
is a continuous p.d. function defined on $\mathbb{R}$ and $F\left(x\right)=\widetilde{F}\left(x\right)$
for all $x$ in $\Omega-\Omega$. 
\item \label{enu:eng2}There is a Hilbert space $\mathscr{K}$, an isometry
$W:\mathscr{H}_{F}\to\mathscr{K}$, and a strongly continuous unitary
group $U_{t}:\mathscr{K}\to\mathscr{K}$, $t\in\mathbb{R}$, such
that if $A$ is the skew-adjoint generator of $U_{t}$, i.e.,\index{strongly continuous}
\begin{equation}
\lim_{t\rightarrow0}\dfrac{1}{t}\left(U_{t}k-k\right)=Ak,\quad\forall k\in dom\left(A\right);\label{eq:e.3a}
\end{equation}
then $\forall\varphi\in C_{c}^{\infty}\left(\Omega\right)$, we have
\begin{equation}
WF_{\varphi}\in dom\left(A\right),\quad\mbox{and}\label{eq:e.3b}
\end{equation}
\begin{equation}
AWF_{\varphi}=WF_{\varphi'}.\label{eq:e.4}
\end{equation}

\end{enumerate}
\end{theorem}
The rest of this section is devoted to the proof of Theorem \ref{thm:pd-extension-bigger-H-space}.\index{infinitesimal generator}
\begin{lemma}
\label{lem:abc}Let $W,\mathscr{H}_{F},\mathscr{K}$ and $U_{t}$
be as in the theorem. If $s,t\in\Omega$, and $\varphi\in C_{c}^{\infty}\left(\Omega\right)$,
then
\begin{equation}
U_{t}WF_{\varphi}=WF_{\varphi_{t}},\;\mbox{and}\label{eq:e.9}
\end{equation}
\begin{equation}
\left\langle WF_{\varphi},U_{t}WF_{\varphi}\right\rangle _{\mathscr{K}}=\left\langle F_{0},\pi\left(\varphi^{\#}*\varphi_{t}\right)F_{0}\right\rangle _{\mathscr{H}_{F}},\label{eq:e.6}
\end{equation}
where $\varphi_{t}\left(\cdot\right)=\varphi\left(\cdot+t\right)$.
\index{approximate identity}\end{lemma}
\begin{svmultproof2}
We first establish (\ref{eq:e.9}). Consider 
\begin{equation}
U_{t-s}WF_{\varphi_{s}}=\begin{cases}
U_{t}WF_{\varphi} & \text{at }s=0\\
WF_{\varphi_{t}} & \text{at }s=t
\end{cases}\label{eq:e.10}
\end{equation}
and
\begin{equation}
\int_{0}^{t}\tfrac{d}{ds}\left(U_{t-s}WF_{\varphi_{s}}\right)ds=WF_{\varphi_{t}}-U_{t}WF_{\varphi}.\label{eq:e.11}
\end{equation}
Note that the left-side of (\ref{eq:e.11}) equals zero. Indeed, (\ref{eq:e.3a})
and (\ref{eq:e.3b}) imply that 
\[
\tfrac{d}{ds}\left(U_{t-s}WF_{\varphi_{s}}\right)=-U_{t-s}AWF_{\varphi_{s}}+U_{t-s}WF_{\varphi_{s}'}.
\]
But (\ref{eq:e.4}) applied to $\varphi_{s}$ yields 
\begin{equation}
AWF_{\varphi_{s}}=WF_{\varphi_{s}'},\label{eq:e.12}
\end{equation}
and so (\ref{eq:e.11}) is identically zero. The desired conclusion
(\ref{eq:e.9}) follows. 

Moreover, we have 
\begin{eqnarray*}
l.h.s.\left(\ref{eq:e.6}\right) & \overset{\text{by \ensuremath{\left(\ref{eq:e.9}\right)}}}{=} & \left\langle WF_{\varphi},WF_{\varphi_{t}}\right\rangle _{\mathscr{K}}\\
 & \overset{\text{\ensuremath{W} is isom.}}{=} & \left\langle F_{\varphi},F_{\varphi_{t}}\right\rangle _{\mathscr{H}_{F}}\\
 & \overset{\text{by \ensuremath{\left(\ref{eq:e.1.1}\right)}}}{=} & \left\langle F_{0},\pi\left(\varphi^{\#}*\varphi_{t}\right)F_{0}\right\rangle _{\mathscr{H}_{F}}\\
 & = & r.h.s.\left(\ref{eq:e.6}\right).
\end{eqnarray*}

\end{svmultproof2}

\begin{svmultproof2}
(The proof of Theorem \ref{thm:pd-extension-bigger-H-space})

(\ref{enu:eng2})$\Longrightarrow$(\ref{enu:eng1}) Assume there
exist $\mathscr{K}$, $W$, $U_{t}$ and $A$ as in (\ref{enu:eng2}).
Let 
\begin{equation}
\widetilde{F}(t)=\left\langle WF_{0},U_{t}WF_{0}\right\rangle ,\quad t\in\mathbb{R}.\label{eq:e.5}
\end{equation}
By the Spectral Theorem, $U_{t}=\int_{\mathbb{R}}e^{i\lambda t}P\left(d\lambda\right)$,
where $P\left(\cdot\right)$ is the corresponding projection-valued
measure. Setting\index{projection-valued measure (PVM)}\index{Theorem!Spectral-}\index{Spectral Theorem}
\[
d\mu\left(\lambda\right):=\left\Vert P\left(d\lambda\right)WF_{0}\right\Vert _{\mathscr{K}}^{2}=\left\langle WF_{0},P\left(d\lambda\right)WF_{0}\right\rangle _{\mathscr{K}},
\]
then $\widetilde{F}=\widehat{d\mu}$, i.e., $\widetilde{F}$ is the
Bochner transform of the Borel measure $d\mu$ on $\mathbb{R}$. \index{Bochner, S.}
\index{Bochner transform}

Let $\phi^{\left(\epsilon\right)}$, $\epsilon>0$, be an approximate
identity at $x=0$. (That is, $\phi^{\left(\epsilon\right)}\rightarrow\delta_{0}$,
as $\epsilon\rightarrow0^{+}$; see Lemma \ref{lem:dense} for details.)
Then\index{Theorem!Bochner's-}\index{Bochner's Theorem} 
\begin{eqnarray}
\widetilde{F}\left(t\right) & = & \left\langle WF_{0},U_{t}WF_{0}\right\rangle _{\mathscr{K}}\nonumber \\
 & = & \lim_{\epsilon\to0^{+}}\left\langle WF_{0},U_{t}WF_{\phi^{\left(\epsilon\right)}}\right\rangle _{\mathscr{K}}\nonumber \\
 & \overset{\text{by \ensuremath{\left(\ref{eq:e.9}\right)}}}{=} & \lim_{\epsilon\to0^{+}}\left\langle WF_{0},WF_{\phi_{t}^{\left(\epsilon\right)}}\right\rangle _{\mathscr{K}}\nonumber \\
 & = & \lim_{\epsilon\to0^{+}}\left\langle F_{0},F_{\phi_{t}^{\left(\epsilon\right)}}\right\rangle _{\mathscr{H}_{F}}\nonumber \\
 & = & \left\langle F_{0},F_{-t}\right\rangle _{\mathscr{H}_{F}}\overset{\text{by \ensuremath{\left(\ref{eq:e.1}\right)}}}{=}F\left(t\right),\quad t\in\Omega-\Omega;\label{eq:e.8}
\end{eqnarray}
Therefore, $\widetilde{F}$ is a continuous p.d. extension of $F$
to $\mathbb{R}$. 

(\ref{enu:eng1})$\Longrightarrow$(\ref{enu:eng2}) Let $\widetilde{F}=\widehat{d\mu}$
be a p.d. extension and Bochner transform. Define $W:\mathscr{H}_{F}\to\mathscr{H}_{\widetilde{F}}$,
by 
\begin{equation}
WF_{\varphi}=\widetilde{F}_{\varphi},\quad\varphi\in C_{c}^{\infty}\left(\Omega\right).\label{eq:e.13}
\end{equation}
Then $W$ is an isometry and $\mathscr{H}_{\widetilde{F}}\backsimeq L^{2}(\mu)$.
Indeed, for all $\varphi\in C_{c}^{\infty}(\Omega)$, since $\widetilde{F}$
is an extension of $F$, we have 
\begin{eqnarray*}
\left\Vert F_{\varphi}\right\Vert _{\mathscr{H}_{F}}^{2} & = & \int_{\Omega}\int_{\Omega}\overline{\varphi\left(s\right)}\varphi\left(t\right)F\left(s-t\right)dsdt\\
 & = & \int_{\mathbb{R}}\int_{\mathbb{R}}\overline{\varphi\left(s\right)}\varphi\left(t\right)\widetilde{F}\left(s-t\right)dsdt\\
 & = & \int_{\mathbb{R}}\int_{\mathbb{R}}\overline{\varphi\left(s\right)}\varphi\left(t\right)\left(\int_{-\infty}^{\infty}e^{i\lambda\left(s-t\right)}d\mu\left(\lambda\right)\right)dsdt\\
 & \overset{\left(\text{Fubini}\right)}{=} & \int_{\mathbb{R}}\left|\widehat{\varphi}\left(\lambda\right)\right|^{2}d\mu\left(\lambda\right)=\Vert\widetilde{F}_{\varphi}\Vert{}_{\mathscr{H}_{\widetilde{F}}}^{2}.
\end{eqnarray*}

Now let 
\[
L^{2}\left(\mu\right)\ni f\xrightarrow{\quad U_{t}\quad}e^{i\lambda t}f\left(\lambda\right)\in L^{2}\left(\mu\right)
\]
be the unitary group acting in $\mathscr{H}_{\widetilde{F}}\backsimeq L^{2}(\mu)$,
and let $A$ be its generator. Then, 
\begin{equation}
\left(WF_{\varphi}\right)\left(x\right)=\int_{\mathbb{R}}e^{i\lambda x}\widehat{\varphi}\left(\lambda\right)d\mu\left(\lambda\right),\quad\forall x\in\Omega,\:\forall\varphi\in C_{c}^{\infty}\left(\Omega\right);\label{eq:e.14}
\end{equation}
and it follows that 

\begin{eqnarray*}
\left(WF_{\varphi'}\right)\left(x\right) & = & \int_{\mathbb{R}}e^{i\lambda x}i\lambda\widehat{\varphi}\left(\lambda\right)d\mu\left(\lambda\right)\\
 & = & \frac{d}{dt}\Big|_{t=0}\left(U_{t}WF_{\varphi}\right)\left(x\right)=\left(AWF_{\varphi}\right)\left(x\right)
\end{eqnarray*}
as claimed. This proves part (\ref{enu:eng2}) of the theorem.
\end{svmultproof2}

An early instance of dilations (i.e., enlarging the Hilbert space)
is the theorem by Sz.-Nagy \cite{RSN56,Muh74} on \emph{unitary dilations}
of strongly \emph{continuous semigroups}. 

We mention this result here to stress that p.d. functions on $\mathbb{R}$
(and subsets of $\mathbb{R}$) often arise from contraction semigroups.
\index{unitary dilations}\index{strongly continuous}
\begin{theorem}[Sz.-Nagy]
 Let $\left\{ S_{t}\:|\:t\in\mathbb{R}_{+}\right\} $ be a strongly
continuous semigroup of contractive operator in a Hilbert space $\mathscr{H}$;
then there is
\begin{enumerate}
\item a Hilbert space $\mathscr{K}$, 
\item an isometry $V:\mathscr{H}\rightarrow\mathscr{K}$, 
\item a strongly continuous one-parameter unitary group $\left\{ U\left(t\right)\:|\:t\in\mathbb{R}\right\} $
acting on $\mathscr{K}$ such that
\begin{equation}
VS_{t}=U\left(t\right)V,\quad\forall t\in\mathbb{R}_{+}\label{eq:nagy-1}
\end{equation}

\end{enumerate}
\end{theorem}
Sz.-Nagy also proved the following:
\begin{theorem}[Sz.-Nagy]
Let $\left(S_{t},\mathscr{H}\right)$ be a contraction semigroup,
$t\geq0$, (such that $S_{0}=I_{\mathscr{H}}$;) and let $f_{0}\in\mathscr{H}\backslash\left\{ 0\right\} $;
then the following function $F$ on $\mathbb{R}$ is positive definite:
\index{positive definite}
\begin{equation}
F\left(t\right)=\begin{cases}
\left\langle f_{0},S_{t}f_{0}\right\rangle _{\mathscr{H}} & \mbox{if }t\geq0,\\
\left\langle f_{0},S_{-t}^{*}f_{0}\right\rangle _{\mathscr{H}} & \mbox{if }t<0.
\end{cases}\label{eq:nagy-2}
\end{equation}
\end{theorem}
\begin{corollary}
Every p.d. function as in (\ref{eq:nagy-2}) has the form:
\begin{equation}
F\left(t\right):=\left\langle k_{0},U\left(t\right)k_{0}\right\rangle _{\mathscr{K}},\quad t\in\mathbb{R}\label{eq:nagy-3}
\end{equation}
where $\left(U\left(t\right),\mathscr{K}\right)$ is a unitary representation
of $\mathbb{R}$. 
\end{corollary}
\index{unitary representation}\index{representation!unitary-}

\section{\label{sub:ExtSpace}$Ext_{1}(F)$ and $Ext_{2}(F)$}

Let $G$ be a locally compact group, and $\Omega$ an open connected
subset of $G$. Consider a continuous p.d. function $F:\Omega^{-1}\cdot\Omega\rightarrow\mathbb{C}$.
\index{positive definite}

We shall study the two sets of extensions in the title of this section. 
\begin{definition}
\label{def:Ext}We say that $\left(U,\mathscr{K},k_{0}\right)\in Ext\left(F\right)$
iff
\begin{enumerate}
\item $U$ is a strongly continuous unitary representation of $G$ in the
Hilbert space $\mathscr{K}$, containing the RKHS $\mathscr{H}_{F}$;
and 
\item there exists $k_{0}\in\mathscr{K}$ such that
\begin{equation}
F\left(g\right)=\left\langle k_{0},U\left(g\right)k_{0}\right\rangle _{\mathscr{K}},\quad\forall g\in\Omega^{-1}\cdot\Omega.\label{eq:ext-1-1}
\end{equation}

\end{enumerate}
\end{definition}

\begin{definition}
Let $Ext_{1}\left(F\right)\subset Ext\left(F\right)$ consisting of
$\left(U,\mathscr{H}_{F},F_{e}\right)$ with 
\begin{equation}
F\left(g\right)=\left\langle F_{e},U\left(g\right)F_{e}\right\rangle _{\mathscr{H}_{F}},\quad\forall g\in\Omega^{-1}\cdot\Omega;\label{eq:ext-1-2}
\end{equation}
where $F_{e}\in\mathscr{H}_{F}$ satisfies $\left\langle F_{e},\xi\right\rangle _{\mathscr{H}_{F}}=\xi\left(e\right)$,
$\forall\xi\in\mathscr{H}_{F}$, and $e$ denotes the neutral (unit)
element in $G$, i.e., $e\,g=g$, $\forall g\in G$. 
\end{definition}

\begin{definition}
Let $Ext_{2}\left(F\right):=Ext\left(F\right)\backslash Ext_{1}\left(F\right)$,
consisting of the solutions to problem (\ref{eq:ext-1-1}) for which
$\mathscr{K}\supsetneqq\mathscr{H}_{F}$, i.e., unitary representations
realized in an enlargement Hilbert space. \index{dilation Hilbert space}\end{definition}
\begin{remark}
\label{rem:measRn}When $G=\mathbb{R}^{n}$, and $\Omega\subset\mathbb{R}^{n}$
is open and connected, we consider continuous p.d. functions $F:\Omega-\Omega\rightarrow\mathbb{C}$.
In this special case, we have 
\begin{align}
Ext\left(F\right)= & \Bigl\{\mu\in\mathscr{M}_{+}\left(\mathbb{R}^{n}\right)\:\big|\:\widehat{\mu}\left(x\right)=\int_{\mathbb{R}^{n}}e^{i\lambda\cdot x}d\mu\left(\lambda\right)\label{eq:ext-1-4}\\
 & \mbox{ is a p.d. extensiont of \ensuremath{F}}\Bigr\}.\nonumber 
\end{align}
Note that (\ref{eq:ext-1-4}) is consistent with (\ref{eq:ext-1-1}).
In fact, if $\left(U,\mathscr{K},k_{0}\right)$ is a unitary representation
of $G=\mathbb{R}^{n}$, such that (\ref{eq:ext-1-1}) holds; then,
by a theorem of Stone, there is a projection-valued measure (PVM)
$P_{U}\left(\cdot\right)$, defined on the Borel subsets of $\mathbb{R}^{n}$
s.t. \index{representation!unitary-}\index{Theorem!Stone's-} 
\begin{equation}
U\left(x\right)=\int_{\mathbb{R}^{n}}e^{i\lambda\cdot x}P_{U}\left(d\lambda\right),\quad x\in\mathbb{R}^{n}.\label{eq:ex-1-5}
\end{equation}
Setting\index{projection-valued measure (PVM)} 
\begin{equation}
d\mu\left(\lambda\right):=\left\Vert P_{U}\left(d\lambda\right)k_{0}\right\Vert _{\mathscr{K}}^{2},\label{eq:ext-1-6-7}
\end{equation}
it is then immediate that $\mu\in\mathscr{M}_{+}\left(\mathbb{R}^{n}\right)$,
and the finite measure $\mu$ satisfies 
\begin{equation}
\widehat{\mu}\left(x\right)=F\left(x\right),\quad\forall x\in\Omega-\Omega.\label{eq:ext-1-6}
\end{equation}

\end{remark}

\paragraph{\textbf{The Case of $n=1$}}

Fix $a>0$, and $\Omega=\left(0,a\right)\subset\mathbb{R}$. Start
with a local continuous p.d. function $F:\Omega-\Omega\rightarrow\mathbb{C}$,
and let $\mathscr{H}_{F}$ be the corresponding RKHS. Let $Ext(F)$
be the compact convex set of probability measures on $\mathbb{R}$
defining extensions of $F$; see (\ref{eq:ext-1-4}). \index{RKHS}

We see in Section \ref{sec:embedding} that all continuous p.d. extensions
of $F$ come from strongly continuous unitary representations. So
in the case of 1D, from unitary one-parameter groups of course, say
$U(t)$. Further recall that some of the p.d. extensions of $F$ may
entail a bigger Hilbert space, say $\mathscr{K}$. By this we mean
that $\mathscr{K}$ creates a dilation (enlargement) of $\mathscr{H}_{F}$
in the sense that $\mathscr{H}_{F}$ is isometrically embedded in
$\mathscr{K}$. Via the embedding we may therefore view $\mathscr{H}_{F}$
as a closed subspace in $\mathscr{K}$. \index{dilation Hilbert space}

We now divide $Ext(F)$ into two parts, say $Ext_{1}(F)$ and $Ext_{2}(F)$.
$Ext_{1}(F)$ is the subset of $Ext(F)$ corresponding to extensions
when the unitary representation $U(t)$ acts in $\mathscr{H}_{F}$
(internal extensions), and $Ext_{2}(F)$ is the part of $Ext(F)$
associated to unitary representations $U(t)$ acting in a proper enlargement
Hilbert space $\mathscr{K}$ (if any), i.e., acting in a Hilbert space
$\mathscr{K}$ corresponding to a proper dilation. For example, the
Pólya extensions in Chapter \ref{chap:types} account for a part of
$Ext_{2}(F)$. 

\index{operator!unitary one-parameter group}

\index{Pólya extensions}

\index{extensions!Pólya-}

\index{operator!Pólya's-}

\index{strongly continuous}\index{unitary one-parameter group}

Now consider the canonical skew-Hermitian operator $D^{\left(F\right)}$
in the RKHS $\mathscr{H}_{F}$ (Definition \ref{def:D}), i.e., 
\begin{gather}
D^{\left(F\right)}\left(F_{\varphi}\right)=F_{\varphi'}\quad\mbox{where}\label{eq:D-3-1}\\
F_{\varphi}\left(x\right)=\int_{\Omega}\varphi\left(y\right)F\left(x-y\right)dy,\quad\varphi\in C_{c}^{\infty}\left(\Omega\right).\label{eq:D-3-2}
\end{gather}
As shown in Section \ref{sec:Prelim}, $D^{\left(F\right)}$ defines
a skew-Hermitian operator with dense domain in $\mathscr{H}_{F}$.
Moreover, the deficiency indices for $D^{\left(F\right)}$ can be
only $\left(0,0\right)$ or $\left(1,1\right)$. The role of deficiency
indices in the RKHS $\mathscr{H}_{F}$ is as follows: \index{deficiency indices}\index{Pólya, G.}\index{deficiency indices}
\index{operator!skew-Hermitian}\index{skew-Hermitian operator; also called skew-symmetric}
\begin{theorem}
\label{thm:DFind}The deficiency indices computed in $\mathscr{H}_{F}$
are $\left(0,0\right)$ if and only if $Ext_{1}(F)$ is a singleton. \end{theorem}
\begin{remark}
Even if $Ext_{1}(F)$ is a singleton, we can still have non-empty
$Ext_{2}(F)$. In Chapter \ref{chap:types}, we include a host of
examples, including one with a Pólya extension where $\mathscr{K}$
is infinite dimensional, while $\mathscr{H}_{F}$ is 2 dimensional.
(When $\dim\mathscr{H}_{F}=2$, obviously we must have deficiency
indices $\left(0,0\right)$.) In other examples we have $\mathscr{H}_{F}$
infinite dimensional, non-trivial Pólya extensions, and yet deficiency
indices $\left(0,0\right)$.

\index{Pólya extensions}\index{extensions!Pólya-}\index{operator!Pólya's-}
\end{remark}

\paragraph{\textbf{Comparison of p.d. Kernels}}

We conclude the present section with some results on comparing positive
definite kernels.

We shall return to the comparison of positive definite functions in
Chapter \ref{chap:CompareFK}. More detailed results on comparison,
and in wider generality, will be included there.
\begin{definition}
Let $K_{i}$, $i=1,2$, be two p.d. kernels defined on some product
$S\times S$ where $S$ is a set. We say that $K_{1}\ll K_{2}$ iff
there is a finite constant $A$ such that 
\begin{equation}
\sum\nolimits _{i}\sum\nolimits _{j}\overline{c_{i}}c_{j}K_{1}\left(s_{i},s_{j}\right)\leq A\sum\nolimits _{i}\sum\nolimits _{j}\overline{c_{i}}c_{j}K_{2}\left(s_{i},s_{j}\right)\label{eq:o-2-1-1}
\end{equation}
for all finite system $\left\{ c_{i}\right\} $ of complex numbers.

If $F_{i}$, $i=1,2$, are p.d. functions defined on a subset of a
group, then we say that $F_{1}\ll F_{2}$ iff the two kernels 
\[
K_{i}\left(x,y\right):=K_{F_{i}}\left(x,y\right)=F_{i}\left(x^{-1}y\right),\;i=1,2
\]
satisfy the condition in (\ref{eq:o-2-1-1}).
\end{definition}
\index{measure!Borel}\index{absolutely continuous}\index{Radon-Nikodym}
\begin{lemma}
\label{lem:li-meas-1}Let $\mu_{i}\in\mathscr{M}_{+}\left(\mathbb{R}^{n}\right)$,
$i=1,2$, i.e., two finite positive Borel measures on $\mathbb{R}^{n}$,
and let $F_{i}:=\widehat{d\mu_{i}}$ be the corresponding Bochner
transforms. Then the following two conditions are equivalent:\index{transform!Fourier-}
\begin{enumerate}
\item \label{enu:order1-1}$\mu_{1}\ll\mu_{2}$ (meaning absolutely continuous)
with $\frac{d\mu_{1}}{d\mu_{2}}\in L^{1}\left(\mu_{2}\right)\cap L^{\infty}\left(\mu_{2}\right)$. 
\item \label{enu:order2-1}$F_{1}\ll F_{2}$, referring to the order of
positive definite functions on $\mathbb{R}^{n}$.
\end{enumerate}
\end{lemma}
\begin{svmultproof2}
(\ref{enu:order1-1})$\Longrightarrow$(\ref{enu:order2-1}) By assumption,
there exists $g\in L_{+}^{2}\left(\mathbb{R}^{n},\mu_{2}\right)$,
the Radon-Nikodym derivative, s.t.\index{Bochner transform}\index{derivative!Radon-Nikodym-}
\begin{equation}
d\mu_{1}=gd\mu_{2}.\label{eq:o-1-1-1}
\end{equation}
Let $\left\{ c_{i}\right\} _{1}^{N}\subset\mathbb{C}$ and $\left\{ x_{i}\right\} _{1}^{N}\subset\mathbb{R}^{n}$,
then 
\begin{eqnarray*}
\sum\nolimits _{j}\sum\nolimits _{k}\overline{c_{j}}c_{k}F_{1}\left(x_{j}-x_{k}\right) & = & \int_{\mathbb{R}^{n}}\Bigl|\sum\nolimits _{j}c_{j}e^{ix_{j}\lambda}\Bigr|^{2}d\mu_{1}\left(\lambda\right)\\
 & = & \int_{\mathbb{R}^{n}}\Bigl|\sum\nolimits _{j}c_{j}e^{ix_{j}\lambda}\Bigr|^{2}g\left(\lambda\right)d\mu_{2}\left(\lambda\right)\quad\text{\ensuremath{\left(\text{by }\left(\ref{eq:o-1-1-1}\right)\&\left(\ref{enu:order1-1}\right)\right)}}\\
 & \overset{\text{\ensuremath{\left(\ref{enu:order1-1}\right)}}}{\leq} & \left\Vert g\right\Vert _{L^{\infty}\left(\mu_{2}\right)}\int_{\mathbb{R}^{n}}\Bigl|\sum\nolimits _{j}c_{j}e^{ix_{j}\lambda}\Bigr|^{2}d\mu_{2}\left(\lambda\right)\\
 & = & \left\Vert g\right\Vert _{L^{\infty}\left(\mu_{2}\right)}\sum\nolimits _{j}\sum\nolimits _{k}\overline{c_{j}}c_{k}F_{2}\left(x_{j}-x_{k}\right);
\end{eqnarray*}
which is the desired estimate in (\ref{eq:o-2-1-1}), with $A=\left\Vert g\right\Vert _{L^{\infty}\left(\mu_{2}\right)}$.

(\ref{enu:order2-1})$\Longrightarrow$(\ref{enu:order1-1}) Conversely,
$\exists A<\infty$ s.t. for all $\varphi\in C_{c}\left(\mathbb{R}^{n}\right)$,
we have 
\begin{equation}
\iint\overline{\varphi\left(x\right)}\varphi\left(y\right)F_{1}\left(x-y\right)dxdy\leq A\iint\overline{\varphi\left(x\right)}\varphi\left(y\right)F_{2}\left(x-y\right)dxdy.\label{eq:o-1-2-1}
\end{equation}
Using that $F_{i}=\widehat{d\mu_{i}}$, eq. (\ref{eq:o-1-2-1}) is
equivalent to 
\begin{equation}
\int_{\mathbb{R}^{n}}\left|\widehat{\varphi}\left(\lambda\right)\right|^{2}d\mu_{1}\left(\lambda\right)\leq A\int_{\mathbb{R}^{n}}\left|\widehat{\varphi}\left(\lambda\right)\right|^{2}d\mu_{2}\left(\lambda\right),\label{eq:o-1-3-1}
\end{equation}
where $\left|\widehat{\varphi}\left(\lambda\right)\right|^{2}=\widehat{\varphi\ast\varphi^{\#}}\left(\lambda\right)$,
$\lambda\in\mathbb{R}^{n}$. 

Since $\left\{ \widehat{\varphi}\:|\:\varphi\in C_{c}\left(\mathbb{R}^{n}\right)\right\} \cap L^{1}\left(\mathbb{R}^{n},\mu\right)$
is dense in $L^{1}\left(\mathbb{R}^{n},\mu\right)$, for all $\mu\in\mathscr{M}_{+}\left(\mathbb{R}^{n}\right)$,
we conclude from (\ref{eq:o-1-3-1}) that $\mu_{1}\ll\mu_{2}$. Moreover,
$\frac{d\mu_{1}}{d\mu_{2}}\in L_{+}^{1}\left(\mathbb{R}^{n},\mu_{2}\right)\cap L^{\infty}\left(\mathbb{R}^{n},\mu_{2}\right)$
by the argument in the first half of the proof.\end{svmultproof2}

\begin{definition}
Let $K_{i}$, $i=1,2$, be two p.d. kernels defined on some product
$S\times S$ where $S$ is a set. We say that $K_{1}\ll K_{2}$ iff
there is a finite constant $A$ such that 
\begin{equation}
\sum\nolimits _{i}\sum\nolimits _{j}\overline{c_{i}}c_{j}K_{1}\left(s_{i},s_{j}\right)\leq A\sum\nolimits _{i}\sum\nolimits _{j}\overline{c_{i}}c_{j}K_{2}\left(s_{i},s_{j}\right)\label{eq:o-2-1}
\end{equation}
for all finite system $\left\{ c_{i}\right\} $ of complex numbers.

If $F_{i}$, $i=1,2$, are p.d. functions defined on a subset of a
group then we say that $F_{1}\ll F_{2}$ iff the two kernels 
\[
K_{i}\left(x,y\right):=K_{F_{i}}\left(x,y\right)=F_{i}\left(x^{-1}y\right),\;i=1,2
\]
satisfies the condition in (\ref{eq:o-2-1}).
\end{definition}
\index{measure!Borel}\index{absolutely continuous}\index{Radon-Nikodym}
\begin{remark}
Note that (\ref{eq:o-2-1}) holds iff we have containment $\mathscr{H}\left(K_{2}\right)\hookrightarrow\mathscr{H}\left(K_{1}\right)$,
\emph{and} the inclusion operator is \emph{bounded}.\end{remark}
\begin{lemma}
\label{lem:li-meas}Let $\mu_{i}\in\mathscr{M}_{+}\left(\mathbb{R}^{n}\right)$,
$i=1,2$, i.e., two finite positive Borel measures on $\mathbb{R}^{n}$,
and let $F_{i}:=\widehat{d\mu_{i}}$ be the corresponding Bochner
transforms. Then the following two conditions are equivalent:
\begin{enumerate}
\item \label{enu:order1}$\mu_{1}\ll\mu_{2}$ (meaning absolutely continuous)
with $\frac{d\mu_{1}}{d\mu_{2}}\in L^{1}\left(\mu_{2}\right)\cap L^{\infty}\left(\mu_{2}\right)$. 
\item \label{enu:order2}$F_{1}\ll F_{2}$, referring to the order of positive
definite functions on $\mathbb{R}^{n}$.
\end{enumerate}
\end{lemma}
\begin{svmultproof2}
(\ref{enu:order1})$\Longrightarrow$(\ref{enu:order2}) By assumption,
there exists $g\in L_{+}^{2}\left(\mathbb{R}^{n},\mu_{2}\right)$,
the Radon-Nikodym derivative, s.t.\index{Bochner transform} 
\begin{equation}
d\mu_{1}=gd\mu_{2}.\label{eq:o-1-1}
\end{equation}
Let $\left\{ c_{i}\right\} _{1}^{N}\subset\mathbb{C}$ and $\left\{ x_{i}\right\} _{1}^{N}\subset\mathbb{R}^{n}$,
then 
\begin{eqnarray*}
\sum\nolimits _{j}\sum\nolimits _{k}\overline{c_{j}}c_{k}F_{1}\left(x_{j}-x_{k}\right) & = & \int_{\mathbb{R}^{n}}\Bigl|\sum\nolimits _{j}c_{j}e^{ix_{j}\lambda}\Bigr|^{2}d\mu_{1}\left(\lambda\right)\\
 & = & \int_{\mathbb{R}^{n}}\Bigl|\sum\nolimits _{j}c_{j}e^{ix_{j}\lambda}\Bigr|^{2}g\left(\lambda\right)d\mu_{2}\left(\lambda\right)\quad\text{\ensuremath{\left(\text{by }\left(\ref{eq:o-1-1}\right)\&\left(\ref{enu:order1}\right)\right)}}\\
 & \overset{\text{\ensuremath{\left(\ref{enu:order1}\right)}}}{\leq} & \left\Vert g\right\Vert _{L^{\infty}\left(\mu_{2}\right)}\int_{\mathbb{R}^{n}}\Bigl|\sum\nolimits _{j}c_{j}e^{ix_{j}\lambda}\Bigr|^{2}d\mu_{2}\left(\lambda\right)\\
 & = & \left\Vert g\right\Vert _{L^{\infty}\left(\mu_{2}\right)}\sum\nolimits _{j}\sum\nolimits _{k}\overline{c_{j}}c_{k}F_{2}\left(x_{j}-x_{k}\right);
\end{eqnarray*}
which is the desired estimate in (\ref{eq:o-2-1}), with $A=\left\Vert g\right\Vert _{L^{\infty}\left(\mu_{2}\right)}$.

(\ref{enu:order2})$\Longrightarrow$(\ref{enu:order1}) Conversely,
$\exists A<\infty$ s.t. for all $\varphi\in C_{c}\left(\mathbb{R}^{n}\right)$,
we have 
\begin{equation}
\iint\overline{\varphi\left(x\right)}\varphi\left(y\right)F_{1}\left(x-y\right)dxdy\leq A\iint\overline{\varphi\left(x\right)}\varphi\left(y\right)F_{2}\left(x-y\right)dxdy.\label{eq:o-1-2}
\end{equation}
Using that $F_{i}=\widehat{d\mu_{i}}$, eq. (\ref{eq:o-1-2}) is equivalent
to 
\begin{equation}
\int_{\mathbb{R}^{n}}\left|\widehat{\varphi}\left(\lambda\right)\right|^{2}d\mu_{1}\left(\lambda\right)\leq A\int_{\mathbb{R}^{n}}\left|\widehat{\varphi}\left(\lambda\right)\right|^{2}d\mu_{2}\left(\lambda\right),\label{eq:o-1-3}
\end{equation}
where $\left|\widehat{\varphi}\left(\lambda\right)\right|^{2}=\widehat{\varphi\ast\varphi^{\#}}\left(\lambda\right)$,
$\lambda\in\mathbb{R}^{n}$. 

Since $\left\{ \widehat{\varphi}\:|\:\varphi\in C_{c}\left(\mathbb{R}^{n}\right)\right\} \cap L^{1}\left(\mathbb{R}^{n},\mu\right)$
is dense in $L^{1}\left(\mathbb{R}^{n},\mu\right)$, for all $\mu\in\mathscr{M}_{+}\left(\mathbb{R}^{n}\right)$,
we conclude from (\ref{eq:o-1-3}) that $\mu_{1}\ll\mu_{2}$. Moreover,
$\frac{d\mu_{1}}{d\mu_{2}}\in L_{+}^{1}\left(\mathbb{R}^{n},\mu_{2}\right)\cap L^{\infty}\left(\mathbb{R}^{n},\mu_{2}\right)$
by the argument in the first half of the proof.
\end{svmultproof2}

\section{Spectral Theory of $D^{\left(F\right)}$ and its Extensions}

\index{complex exponential}\index{skew-Hermitian operator; also called skew-symmetric}\index{operator!skew-Hermitian}\index{atom}\index{spectrum}

In this section we return to $n=1$, so a given continuous positive
definite function $F$, defined in an interval $(-a,a)$ where $a>0$
is fixed. We shall study spectral theoretic properties of the associated
skew-Hermitian operator $D^{\left(F\right)}$ in the RKHS $\mathscr{H}_{F}$
from Definition \ref{def:D} in Section \ref{sub:DF}, and Theorem
\ref{thm:DFind} above.
\begin{proposition}
Fix $a>0$, and set $\Omega=\left(0,a\right)$. Let $F:\Omega-\Omega\rightarrow\mathbb{C}$
be continuous, p.d. and $F\left(0\right)=1$. Let $D^{\left(F\right)}$
be the skew-Hermitian operator, i.e., $D^{\left(F\right)}\left(F_{\varphi}\right)=F_{\varphi'}$,
for all $\varphi\in C_{c}^{\infty}\left(0,a\right)$. Suppose $D^{\left(F\right)}$
has a skew-adjoint extension $A\supset D^{\left(F\right)}$ (in the
RKHS $\mathscr{H}_{F}$), such that $A$ has simple and purely atomic
spectrum, $\left\{ i\lambda_{n}\:|\:\lambda_{n}\in\mathbb{R}\right\} _{n\in\mathbb{N}}$.
Then the complex exponentials, restricted to $\Omega$, \index{purely atomic}
\begin{equation}
e_{\lambda_{n}}\left(x\right)=e^{i\lambda_{n}x},\quad x\in\Omega\label{eq:exp-15}
\end{equation}
are orthogonal and total in $\mathscr{H}_{F}$. \index{operator!rank-one-}\end{proposition}
\begin{svmultproof2}
By the Spectral Theorem, and the assumption on the spectrum of $A$,
there exists an orthonormal basis (ONB) $\left\{ \xi_{n}\right\} $
in $\mathscr{H}_{F}$, such that 
\begin{equation}
U_{A}\left(t\right)=e^{tA}=\sum_{n\in\mathbb{N}}e^{it\lambda_{n}}\left|\xi_{n}\left\rangle \right\langle \xi_{n}\right|,\quad t\in\mathbb{R},\label{eq:eFa}
\end{equation}
where $\left|\xi_{n}\left\rangle \right\langle \xi_{n}\right|$ is
Dirac's term for the rank-1 projection onto the $\mathbb{C}\xi_{n}$
in $\mathscr{H}_{F}$. 

Recall that $F_{x}\left(y\right):=F\left(x-y\right)$, $\forall x,y\in\Omega=\left(0,a\right)$.
Then, we get that 
\[
F_{0}\left(\cdot\right)=\sum_{n\in\mathbb{N}}\left\langle \xi_{n},F_{0}\right\rangle \xi_{n}\left(\cdot\right)=\sum_{n\in\mathbb{N}}\overline{\xi_{n}\left(0\right)}\xi_{n}\left(\cdot\right),
\]
by the reproducing property in $\mathscr{H}_{F}$; and with $0<t<a$,
we have: 
\begin{equation}
F_{-t}\left(\cdot\right)=U_{A}\left(t\right)F_{0}\left(\cdot\right)\underset{\text{by \ensuremath{\left(\ref{eq:eFa}\right)}}}{=}\sum_{n\in\mathbb{N}}e^{it\lambda_{n}}\overline{\xi_{n}\left(0\right)}\xi_{n}\left(\cdot\right)\label{eq:eF0}
\end{equation}
holds on $\Omega=\left(0,a\right)$. 

Now fix $n\in\mathbb{N}$, and take the inner-product $\left\langle \xi_{n},\cdot\right\rangle _{\mathscr{H}_{F}}$
on both sides in (\ref{eq:eF0}). Using again the reproducing property,
we get 
\begin{equation}
\xi_{n}\left(t\right)=e^{it\lambda_{n}}\xi_{n}\left(0\right),\quad t\in\Omega;\label{eq:exp-20}
\end{equation}
which yields the desired conclusion. 

Note that the functions $\{e_{\lambda_{n}}\}_{n\in\mathbb{N}}$ in
(\ref{eq:exp-15}) are orthogonal, and total in $\mathscr{H}_{F}$;
but they are not normalized. In fact, it follows from (\ref{eq:exp-20}),
that $\Vert e_{\lambda_{n}}\Vert_{\mathscr{H}_{F}}=\left|\xi_{n}\left(0\right)\right|^{-1}$. \end{svmultproof2}

\begin{theorem}
\label{thm:Eigenspaces-for-the-adjoint}Let $a>0$, and $\Omega=(0,a)$.
Let $F$ be a continuous p.d. \index{positive definite} function
on $\Omega-\Omega=\left(-a,a\right)$. Let $D=D^{\left(F\right)}$
be the canonical skew-Hermitian operator acting in $\mathscr{H}_{F}$. 

Fix $z\in\mathbb{C}$; then the function $\xi_{z}:y\mapsto e^{-zy}$,
restricted to $\Omega$, is in $\mathscr{H}_{F}$ iff $z$ is an eigenvalue
for the adjoint operator $D^{*}$. In the affirmative case, the corresponding
eigenspace is $\mathbb{C}\xi_{z}$, in particular, the eigenspace
has dimension one. 
\end{theorem}
\index{eigenvalue(s)}\index{operator!adjoint of an-}\index{operator!domain of-}
\begin{svmultproof2}
Suppose $D^{*}\xi=z\xi$, $\xi\in\mathscr{H}_{F}$, then 
\[
\left\langle DF_{\varphi},\xi\right\rangle _{\mathscr{H}_{F}}=\left\langle F_{\varphi},z\xi\right\rangle _{\mathscr{H}_{F}},\quad\forall\varphi\in C_{c}^{\infty}\left(\Omega\right).
\]
Equivalently, 
\[
\int_{\Omega}\varphi'\left(y\right)\xi\left(y\right)dy=\int_{\Omega}z\varphi\left(y\right)\xi\left(y\right)dy,\quad\forall\varphi\in C_{c}^{\infty}\left(\Omega\right).
\]
Hence, $\xi$ is a weak solution to 
\[
-\xi'(y)=z\xi(y),\quad y\in\Omega.
\]
and so $\xi\left(y\right)=\mathrm{const}\cdot e^{-zy}$. 

Conversely, suppose $\xi_{z}\left(y\right)=e^{-zy}\big|_{\Omega}$
is in $\mathscr{H}_{F}.$ It is sufficient to show $\xi_{z}\in dom(D^{*})$;
i.e., we must show that there is a finite constant $C$, such that
\begin{equation}
\left|\left\langle DF_{\varphi},\xi_{z}\right\rangle _{\mathscr{H}_{F}}\right|\leq C\left\Vert F_{\varphi}\right\Vert _{\mathscr{H}_{F}},\quad\forall\varphi\in C_{c}^{\infty}\left(\Omega\right).\label{eq:te-23}
\end{equation}
But, we have 
\begin{eqnarray*}
\left|\left\langle DF_{\varphi},\xi_{z}\right\rangle _{\mathscr{H}_{F}}\right| & = & \left|\int_{\Omega}\overline{\varphi'\left(y\right)}\xi_{z}\left(y\right)dy\right|=\left|z\right|\left|\int_{\Omega}\overline{\varphi\left(y\right)}\xi_{z}\left(y\right)dy\right|\\
 & = & \left|z\right|\left|\left\langle F_{\varphi},\xi_{z}\right\rangle _{\mathscr{H}_{F}}\right|\leq\left|z\right|\left\Vert \xi_{z}\right\Vert _{\mathscr{H}_{F}}\left\Vert F_{\varphi}\right\Vert _{\mathscr{H}_{F}};
\end{eqnarray*}
which is the desired estimate in (\ref{eq:te-23}). (The final step
follows from the Cauchy-Schwarz inequality.)
\end{svmultproof2}

\begin{theorem}
Let $a>0$, and $\lambda_{1}\in\mathbb{R}$ be given. Let $F:\left(-a,a\right)\rightarrow\mathbb{C}$
be a fixed continuous p.d. function. Then there following two conditions
(\ref{enu:en2-1}) and (\ref{enu:en-2-2}) are equivalent:
\begin{enumerate}
\item \label{enu:en2-1}$\exists\mu_{1}\in Ext\left(F\right)$ such that
$\mu_{1}(\{\lambda_{1}\})>0$ (i.e., $\mu_{1}$ has an atom at $\lambda_{1}$);
and
\item \label{enu:en-2-2}$e_{\lambda_{1}}\left(x\right):=e^{i\lambda_{1}x}\big|_{\left(-a,a\right)}\in\mathscr{H}_{F}$. 
\end{enumerate}
\end{theorem}
\begin{svmultproof2}
The implication (\ref{enu:en2-1})$\Longrightarrow$(\ref{enu:en-2-2})
is already contained in the previous discussion. 

Proof of (\ref{enu:en-2-2})$\Longrightarrow$(\ref{enu:en2-1}).
We first consider the skew-Hermitian operator $D^{\left(F\right)}\left(F_{\varphi}\right):=F_{\varphi'}$,
$\varphi\in C_{c}^{\infty}\left(0,a\right)$. Using an idea of M.
Krein\index{Krein, M.} \cite{Kre46,KL14}, we may always find a Hilbert
space $\mathscr{K}$, an isometry $J:\mathscr{H}\rightarrow\mathscr{K}$,
and a strongly continuous unitary one-parameter group $U_{A}\left(t\right)=e^{tA}$,
$t\in\mathbb{R}$, with $A^{*}=-A$; $U_{A}\left(t\right)$ acting
in $\mathscr{K}$, such that\index{unitary one-parameter group} 
\begin{equation}
JD^{\left(F\right)}=AJ\mbox{ on }\label{eq:en-4-1}
\end{equation}
\begin{equation}
dom(D^{\left(F\right)})=\left\{ F_{\varphi}\:\big|\:\varphi\in C_{c}^{\infty}\left(0,a\right)\right\} ;\label{eq:en-4-2}
\end{equation}
see also Theorem \ref{thm:gEn}. Since 
\begin{equation}
e_{1}\left(x\right)=e^{i\lambda_{1}x}\big|_{\left(-a,a\right)}\label{eq:en-4-3}
\end{equation}
is in $\mathscr{H}_{F}$, we can form the following measure $\mu_{1}\in\mathscr{M}_{+}\left(\mathbb{R}\right)$,
now given by 
\begin{equation}
d\mu_{1}\left(\lambda\right):=\left\Vert P_{A}\left(d\lambda\right)Je_{1}\right\Vert _{\mathscr{K}}^{2},\quad\lambda\in\mathbb{R},\label{eq:en-4-4}
\end{equation}
where $P_{A}\left(\cdot\right)$ is the PVM of $U_{A}\left(t\right)$,
i.e., 
\begin{equation}
U_{A}\left(t\right)=\int_{\mathbb{R}}e^{it\lambda}P_{A}\left(d\lambda\right),\quad t\in\mathbb{R}.\label{eq:en-4-5}
\end{equation}

We claim the following two assertions: 
\begin{enumerate}[label=(\roman{enumi})]
\item \label{enu:en-2-3}$\mu_{1}\in Ext\left(F\right)$; and 
\item \label{enu:en-2-4}$\lambda_{1}$ is an atom in $\mu_{1}$, i.e.,
$\mu_{1}\left(\left\{ \lambda_{1}\right\} \right)>0$. 
\end{enumerate}

This is the remaining conclusion in the theorem.

The proof of of \ref{enu:en-2-3} is immediate from the construction
above; using the intertwining\index{intertwining} isometry\index{isometry}
$J$ from (\ref{eq:en-4-1}), and formulas (\ref{eq:en-4-4})-(\ref{eq:en-4-5}).\index{operator!intertwining-}

To prove \ref{enu:en-2-4}, we need the following:
\begin{lemma}
Let $F$, $\lambda_{1}$, $e_{1}$, $\mathscr{K}$, $J$ and $\left\{ U_{A}\left(t\right)\right\} _{t\in\mathbb{R}}$
be as above; then we have the identity:
\begin{equation}
\left\langle Je_{1},U_{A}\left(t\right)Je_{1}\right\rangle _{\mathscr{K}}=e^{it\lambda_{1}}\left\Vert e_{1}\right\Vert _{\mathscr{H}_{F}}^{2},\quad\forall t\in\mathbb{R}.\label{eq:en-4-6}
\end{equation}
\end{lemma}
\begin{svmultproof2}
It is immediate from (\ref{eq:en-4-1})-(\ref{eq:en-4-4}) that (\ref{eq:en-4-6})
holds for $t=0$. To get it for all $t$, fix $t$, say $t>0$ (the
argument is the same if $t<0$); and we check that 
\begin{equation}
\frac{d}{ds}\left(\left\langle Je_{1},U_{A}\left(t-s\right)Je_{1}\right\rangle _{\mathscr{K}}-e^{i\left(t-s\right)\lambda_{1}}\left\Vert e_{1}\right\Vert _{\mathscr{H}_{F}}^{2}\right)\equiv0,\;\forall s\in\left(0,t\right).\label{eq:en-4-7}
\end{equation}
But this, in turn, follow from the assertions above: First 
\[
D_{\left(F\right)}^{*}e_{1}=D_{\left(F\right)}^{*}J^{*}Je_{1}=J^{*}AJe_{1}
\]
holds on account of (\ref{eq:en-4-1}). We get: $e_{1}\in dom(D_{\left(F\right)}^{*})$,
and $D_{\left(F\right)}^{*}e_{1}=-i\lambda_{1}e_{1}$. 

Using this, the verification of is (\ref{eq:en-4-6}) now immediate.
\end{svmultproof2}

As a result, we get:
\[
U_{A}\left(t\right)Je_{1}=e^{it\lambda_{1}}Je_{1},\quad\forall t\in\mathbb{R},
\]
and by (\ref{eq:en-4-5}):
\[
P_{A}\left(\left\{ \lambda_{1}\right\} \right)Je_{1}=Je_{1}
\]
where $\left\{ \lambda_{1}\right\} $ denotes the desired $\lambda_{1}$-atom.
Hence, by (\ref{eq:en-4-4}), $\mu_{1}\left(\left\{ \lambda_{1}\right\} \right)=\left\Vert Je_{1}\right\Vert _{\mathscr{K}}^{2}=\left\Vert e_{1}\right\Vert _{\mathscr{H}_{F}}^{2}$,
which is the desired conclusion in (\ref{enu:en-2-2}).

\end{svmultproof2}

\motto{Nowadays, group theoretical methods --- especially those involving
characters and representations, pervade all branches of quantum mechanics.
--- George Mackey.\\ \vspace{1em}The universe is an enormous direct
product of representations of symmetry groups. --- Hermann Weyl}

\chapter{The Case of More General Groups}

\section{\label{sub:lcg}Locally Compact Abelian Groups}

We are concerned with extensions of locally defined continuous and
positive definite (p.d.) functions $F$ on Lie groups, say $G$, but
some results apply to locally compact groups as well. However in the
case of locally compact Abelian groups, we have stronger theorems,
due to the powerful Fourier analysis theory in this specific setting.

First, we fix some notations: \index{group!locally compact Abelian}\index{measure!Haar}
\index{group!dual}\index{Hilbert space}
\begin{enumerate}
\item[$G$]  a given locally compact Abelian group; group operation is written
additively. 
\item[$dx$]  the \emph{Haar measure} of $G$, unique up to a scalar multiple.
\item[$\widehat{G}\:$]  the \emph{dual group}, consisting of all continuous homomorphisms
$\lambda:G\rightarrow\mathbb{T}$, s.t. 
\[
\lambda\left(x+y\right)=\lambda\left(x\right)\lambda\left(y\right),\;\lambda\left(-x\right)=\overline{\lambda\left(x\right)},\quad\forall x,y\in G.
\]
Occasionally, we shall write $\left\langle \lambda,x\right\rangle $
for $\lambda\left(x\right)$. Note that $\widehat{G}$ also has its
Haar measure. 
\end{enumerate}

The Pontryagin duality theorem below is fundamental for locally compact
Abelian groups. \index{duality!Pontryagin-}\index{Theorem!Pontryagin-}
\begin{theorem}[Pontryagin \cite{Ru90}]
\label{thm:lcg-dualG}$\widehat{\widehat{G}}\simeq G$, and we have
the following:
\[
\left[G\mbox{ is compact}\right]\Longleftrightarrow[\widehat{G}\mbox{ is discrete}]
\]

\end{theorem}

Let $\emptyset\neq\Omega\subset G$ be an open connected subset, and
let 
\begin{equation}
F:\Omega-\Omega\rightarrow\mathbb{C}\label{eq:lcgf}
\end{equation}
be a fixed continuous p.d. function. We choose the normalization $F\left(0\right)=1$.
Set 
\[
F_{y}\left(x\right)=F\left(x-y\right),\quad\forall x,y\in\Omega.
\]

The corresponding RKHS $\mathscr{H}_{F}$ is defined almost verbatim
as in Section \ref{sec:Prelim}. Its ``continuous'' version is recast
in Lemma \ref{lem:lcg-F_varphi} with slight modifications (see Lemma
\ref{lem:RKHS-def-by-integral}.) Functions in $\mathscr{H}_{F}$
are characterized in Lemma \ref{lem:lcg-bdd} (see Theorem \ref{thm:HF}.)

\index{RKHS}\index{positive definite}
\begin{lemma}
\label{lem:lcg-F_varphi}For $\varphi\in C_{c}\left(\Omega\right)$,
set 
\begin{equation}
F_{\varphi}\left(\cdot\right)=\int_{\Omega}\varphi\left(y\right)F\left(\cdot-y\right)dy,\label{eq:lcg-1}
\end{equation}
then $\mathscr{H}_{F}$ is the Hilbert completion of $\left\{ F_{\varphi}\:|\:\varphi\in C_{c}\left(\Omega\right)\right\} $
in the inner product:
\begin{equation}
\left\langle F_{\varphi},F_{\psi}\right\rangle _{\mathscr{H}_{F}}=\int_{\Omega}\int_{\Omega}\overline{\varphi\left(x\right)}\psi\left(y\right)F\left(x-y\right)dxdy.\label{eq:lcg-2}
\end{equation}
Here $C_{c}\left(\Omega\right):=$ all continuous compactly supported
functions in $\Omega$.
\end{lemma}

\begin{lemma}
\label{lem:lcg-bdd}The Hilbert space $\mathscr{H}_{F}$ is also a
Hilbert space of continuous functions on $\Omega$ as follows:

If $\xi:\Omega\rightarrow\mathbb{C}$ is a fixed continuous function,
then $\xi\in\mathscr{H}_{F}$ if and only if $\exists$ $K=K_{\xi}<\infty$
such that
\begin{equation}
\left|\int_{\Omega}\overline{\xi\left(x\right)}\varphi\left(x\right)dx\right|^{2}\leq K\int_{\Omega}\int_{\Omega}\overline{\varphi\left(y_{1}\right)}\varphi\left(y_{2}\right)F\left(y_{1}-y_{2}\right)dy_{1}dy_{2}.\label{eq:lcg-3}
\end{equation}
When (\ref{eq:lcg-3}) holds, then 
\[
\left\langle \xi,F_{\varphi}\right\rangle _{\mathscr{H}_{F}}=\int_{\Omega}\overline{\xi\left(x\right)}\varphi\left(x\right)dx,\quad\forall\varphi\in C_{c}\left(\Omega\right).
\]
\end{lemma}
\begin{svmultproof2}
We refer to the basics on the theory of RKHSs; e.g., \cite{Aro50}.
Also see Section \ref{sec:Prelim}.\end{svmultproof2}

\begin{definition}
\label{def:lcadual}Let $\mathscr{M}(\widehat{G})$ be the set of
all Borel measures on $\widehat{G}$. Given $\mu\in\mathscr{M}(\widehat{G})$,
let $\widehat{\mu}$ be the Fourier transform, i.e., 
\begin{equation}
\widehat{\mu}\left(x\right)=\int_{\widehat{G}}\lambda\left(x\right)d\mu\left(\lambda\right)=\int_{\widehat{G}}\left\langle \lambda,x\right\rangle d\mu\left(\lambda\right),\quad\forall x\in G.\label{eq:lcad1}
\end{equation}
Given $F$ as in (\ref{eq:lcgf}), set 
\begin{equation}
Ext\left(F\right)=\left\{ \mu\in\mathscr{M}(\widehat{G})\:\big|\:F\left(x\right)=\widehat{\mu}\left(x\right),\;\forall x\in\Omega-\Omega\right\} .\label{eq:lcad2}
\end{equation}

\end{definition}

\begin{theorem}
$Ext\left(F\right)$ is weak $\ast$-compact and convex.\end{theorem}
\begin{svmultproof2}
See e.g., \cite{Rud73}.\index{convex}\end{svmultproof2}

\begin{remark}
We shall extend the discussion for the case of $G=\mathbb{R}^{n}$
to locally compact Abelian groups in general (Section \ref{sub:ExtSpace},
especially Remark \ref{rem:measRn}.)

Note that $Ext\left(F\right)$ may be empty. For $G=\mathbb{R}^{2}$,
Rudin gave examples where $Ext\left(F\right)=\emptyset$ \cite{Ru70,Ru63}.
See Section \ref{sub:lgns}, and the example of logarithmic Riemann
surface in Section \ref{sec:logz}.\end{remark}
\begin{question}
Suppose $Ext\left(F\right)\neq\emptyset$, then what are its extreme
points? Equivalently, characterize $ext\left(Ext\left(F\right)\right)$. 

The reader is referred to Section \ref{sub:ext(Ext)}, especially,
the discussion of $Ext_{1}(F)$ as a set of extreme points in $Ext(F)$;
and the direct integral decomposition.
\end{question}
\index{measure!probability}

We now assume that $F\left(0\right)=1$; normalization.
\begin{lemma}
\label{lem:lcg-Bochner}There is a bijective correspondence between
all continuous p.d. extensions $\widetilde{F}$ to $G$ of the given
p.d. function $F$ on $\Omega-\Omega$, on the one hand; and all Borel
probability measures $\mu$ on $\widehat{G}$, on the other, i.e.,
all $\mu\in\mathscr{M}(\widehat{G})$ s.t.
\begin{equation}
F\left(x\right)=\widehat{\mu}\left(x\right),\quad\forall x\in\Omega-\Omega,\label{eq:lcg-bochner}
\end{equation}
where $\widehat{\mu}$ is as in (\ref{eq:lcad2}). \end{lemma}
\begin{svmultproof2}
This is an immediate application of Bochner's characterization of
the continuous p.d. functions on locally compact Abelian groups \cite{BC48,BC49,Boc46,Boc47}.
\end{svmultproof2}

\index{Theorem!Bochner's-}\index{Bochner's Theorem}
\begin{theorem}
\label{thm:lcg-1}Let $F$ and $\mathscr{H}_{F}$ be as above. 
\begin{enumerate}
\item Let $\mu\in\mathscr{M}(\widehat{G})$; then there is a positive Borel
function $h$ on $\widehat{G}$ s.t. $h^{-1}\in L^{\infty}(\widehat{G})$,
and $hd\mu\in Ext\left(F\right)$, if and only if $\exists K_{\mu}<\infty$
such that
\begin{equation}
\int_{\widehat{G}}\left|\widehat{\varphi}\left(\lambda\right)\right|^{2}d\mu\left(\lambda\right)\leq K_{\mu}\int_{\Omega}\int_{\Omega}\overline{\varphi\left(y_{1}\right)}\varphi\left(y_{2}\right)F\left(y_{1}-y_{2}\right)dy_{1}dy_{2},\label{eq:lcg-4}
\end{equation}
for all $\varphi\in C_{c}\left(\Omega\right)$.
\item Assume $\mu\in Ext\left(F\right)$, then 
\begin{equation}
\chi_{\overline{\Omega}}\left(fd\mu\right)^{\vee}\in\mathscr{H}_{F},\quad\forall f\in L^{2}(\widehat{G},\mu).\label{eq:lcg-5}
\end{equation}

\end{enumerate}
\end{theorem}
\begin{svmultproof2}
The assertion in (\ref{eq:lcg-4}) is immediate from Lemma \ref{lem:lcg-bdd}.
The remaining computations are left to the reader.\end{svmultproof2}

\begin{remark}
Our conventions for the two transforms used in (\ref{eq:lcg-4}) and
(\ref{eq:lcg-5}) are as follows:\index{transform!Fourier-}
\begin{eqnarray}
\widehat{\varphi}\left(\lambda\right) & = & \int_{G}\overline{\left\langle \lambda,x\right\rangle }\varphi\left(x\right)dx,\quad\varphi\in C_{c}\left(\Omega\right);\;\mbox{and}\label{eq:lcg-6}\\
\left(fd\mu\right)^{\vee}\left(x\right) & = & \int_{\widehat{G}}\left\langle \lambda,x\right\rangle f\left(\lambda\right)d\mu\left(\lambda\right),\quad f\in L^{2}(\widehat{G},\mu).\label{eq:lcg-7}
\end{eqnarray}

\end{remark}
\Needspace*{3\baselineskip}
\begin{corollary}
\label{cor:lcg-isom}~
\begin{enumerate}
\item \label{enu:1}Let $F$ be as above; then $\mu\in Ext\left(F\right)$
iff the operator 
\[
T(F_{\varphi})=\widehat{\varphi},\quad\forall\varphi\in C_{c}\left(\Omega\right),
\]
is well-defined and extends to an isometric operator $T:\mathscr{H}_{F}\rightarrow L^{2}(\widehat{G},\mu)$.
\index{operator!bounded}
\item \label{enu:2}If $\mu\in Ext\left(F\right)$, the adjoint operator
$T^{*}:L^{2}(\widehat{G},\mu)\rightarrow\mathscr{H}_{F}$ is given
by
\begin{equation}
T^{*}\left(f\right)=\chi_{\overline{\Omega}}\left(fd\mu\right)^{\vee},\quad\forall f\in L^{2}(\widehat{G},\mu).\label{eq:lcg-8}
\end{equation}

\end{enumerate}
\end{corollary}
\begin{remark}
Note that we have fixed $F$, and a positive measure $\mu\in Ext\left(F\right)$,
both fixed at the outset. We then show that $T\left(=T_{\mu}\right)$
as defined in \ref{enu:1} is an isometry relative to $\mu$, i.e.,
mapping from $\mathscr{H}_{F}$ into $L^{2}\left(\mu\right)$. But
in the discussion below, we omit the subscript $\mu$ in order to
lighten notation.\end{remark}
\begin{svmultproof2}
If $\mu\in Ext\left(F\right)$, then for all $\varphi\in C_{c}\left(\Omega\right)$,
and $x\in\Omega$, we have (see (\ref{eq:lcg-1}))
\begin{eqnarray*}
F_{\varphi}\left(x\right) & = & \int_{\Omega}\varphi\left(y\right)F\left(x-y\right)dy\\
 & = & \int_{\Omega}\varphi\left(y\right)\widehat{\mu}\left(x-y\right)dy\\
 & = & \int_{\Omega}\varphi\left(y\right)\int_{\widehat{G}}\left\langle \lambda,x-y\right\rangle d\mu\left(\lambda\right)dy\\
 & \overset{\left(\text{Fubini}\right)}{=} & \int_{\widehat{G}}\left\langle \lambda,x\right\rangle \widehat{\varphi}\left(\lambda\right)d\mu\left(\lambda\right).
\end{eqnarray*}

By Lemma \ref{lem:lcg-bdd}, we note that $\chi_{\overline{\Omega}}\left(\widehat{\varphi}d\mu\right)^{\vee}\in\mathscr{H}_{F}$,
see (\ref{eq:lcg-7}). Hence $\exists K<\infty$ such that the estimate
(\ref{eq:lcg-4}) holds. To see that $T\left(F_{\varphi}\right)=\widehat{\varphi}$
is well-defined on $\mathscr{H}_{F}$, we must check the implication:
\[
\Bigl(F_{\varphi}=0\mbox{ in }\mathscr{H}_{F}\Bigr)\Longrightarrow\Bigl(\widehat{\varphi}=0\mbox{ in }L^{2}(\widehat{G},\mu)\Bigr)
\]
but this now follows from estimate (\ref{eq:lcg-4}).

Using the definition of the respective inner products in $\mathscr{H}_{F}$
and in $L^{2}(\widehat{G},\mu)$, we check directly that, if $\varphi\in C_{c}\left(\Omega\right)$,
and $f\in L^{2}(\widehat{G},\mu)$ then we have:
\begin{equation}
\left\langle \widehat{\varphi},f\right\rangle _{L^{2}\left(\mu\right)}=\left\langle F_{\varphi},\left(fd\mu\right)^{\vee}\right\rangle _{\mathscr{H}_{F}}.\label{eq:lcg-9}
\end{equation}
On the r.h.s. in (\ref{eq:lcg-9}), we note that, when $\mu\in Ext\left(F\right)$,
then $\chi_{\overline{\Omega}}\left(fd\mu\right)^{\vee}\in\mathscr{H}_{F}$.
This last conclusion is a consequence of Lemma \ref{lem:lcg-bdd}:

Indeed, since $\mu$ is finite, $L^{2}(\widehat{G},\mu)\subset L^{1}(\widehat{G},\mu)$,
so $\left(fd\mu\right)^{\vee}$ in (\ref{eq:lcg-7}) is continuous
on $G$ by Riemann-Lebesgue; and so is its restriction to $\Omega$.
If $\mu$ is further assumed absolutely continuous, then $\left(fd\mu\right)^{\vee}\rightarrow0$
at $\infty$. \index{absolutely continuous}\index{operator!adjoint of an-}
With a direct calculation, using the reproducing property in $\mathscr{H}_{F}$,
and Fubini's theorem, we check directly that the following estimate
holds:
\[
\left|\int_{\Omega}\overline{\varphi\left(x\right)}\left(fd\mu\right)^{\vee}\left(x\right)dx\right|^{2}\leq\left(\int_{\Omega}\int_{\Omega}\overline{\varphi\left(y_{1}\right)}\varphi\left(y_{2}\right)F\left(y_{1}-y_{2}\right)dy_{1}dy_{2}\right)\left\Vert f\right\Vert _{L^{2}\left(\mu\right)}^{2}
\]
and so Lemma \ref{lem:lcg-bdd} applies; we get $\chi_{\overline{\Omega}}\left(fd\mu\right)^{\vee}\in\mathscr{H}_{F}$.

It remains to verify the formula (\ref{eq:lcg-9}) for all $\varphi\in C_{c}\left(\Omega\right)$
and all $f\in L^{2}(\widehat{G},\mu)$; but this now follows from
the reproducing property in $\mathscr{H}_{F}$, and Fubini. 

Once we have this, both assertions in (\ref{enu:1}) and (\ref{enu:2})
in the Corollary follow directly from the definition of the adjoint
operator $T^{*}$ with respect to the two Hilbert space inner products
in $\mathscr{H}_{F}\overset{T}{\longrightarrow}L^{2}(\widehat{G},\mu)$.
Indeed then (\ref{eq:lcg-8}) follows. \end{svmultproof2}

\begin{remark}
\label{rem:HFL2}The transform $T=T_{\mu}$ from Corollary \ref{cor:lcg-isom}
is \emph{not} onto $L^{2}\left(\mu\right)$ in general. We illustrate
the cases with $G=\mathbb{R}$; and we note that then:
\begin{enumerate}[label=(\roman{enumi}), ref=\roman{enumi}]
\item \label{enu:Tu1}If $\mu$ is of compact support (in $\mathbb{R}$),
then $T_{\mu}$ maps onto $L^{2}\left(\mu\right)$. 
\item \label{enu:Tu2}An example, where $T_{\mu}$ is \emph{not} onto $L^{2}\left(\mu\right)$,
may be obtained as follows.\\
Let $F=F_{2}$ in Tables \ref{tab:F1-F6}-\ref{tab:Table-2}, i.e.,
\begin{equation}
F\left(x\right)=\left(1-\left|x\right|\right)\big|_{\left(-\frac{1}{2},\frac{1}{2}\right)}.\label{eq:t0}
\end{equation}
Let 
\begin{equation}
d\mu\left(\lambda\right)=\frac{1}{2\pi}\left(\frac{\sin\left(\lambda/2\right)}{\lambda/2}\right)^{2}d\lambda,\;(\mbox{Table \ref{tab:Table-3}})\label{eq:t3}
\end{equation}
where $\left(d\mu\right)^{\vee}=\widetilde{F}\in Ext\left(F\right)$
is the tent function on $\mathbb{R}$ given by 
\begin{equation}
\widetilde{F}\left(x\right)=\begin{cases}
1-\left|x\right| & \left|x\right|\leq1\\
0 & \left|x\right|>1.
\end{cases}\label{eq:5}
\end{equation}
We say that $\mu\in Ext\left(F\right)$. Take a function $f$ on $\mathbb{R}$
such that
\begin{equation}
\check{f}\left(y\right)=\begin{cases}
e^{-\left|y\right|} & \left|y\right|>2\\
0 & \left|y\right|\leq2,
\end{cases}\label{eq:t1}
\end{equation}
($\check{f}$ denotes inverse Fourier transform.) Then (see (\ref{eq:lcg-8}))\index{transform!Fourier-}
\begin{equation}
T_{\mu}^{*}\left(f\right)=0\;\mbox{in}\;\mathscr{H}_{F}.\label{eq:t2}
\end{equation}
(Note that $f\in L^{2}\left(\mathbb{R},\mu\right)\backslash\left\{ 0\right\} $
on account of (\ref{eq:t3}) and (\ref{eq:t1}).)
\end{enumerate}
\end{remark}
\begin{svmultproof2}
(\ref{enu:Tu1}) Fix some finite positive measure $\mu$. Since $Ran\left(T_{\mu}\right)^{\perp}=Null\left(T_{\mu}^{*}\right)$,
we need only consider the homogeneous equation $T_{\mu}^{*}f=0$ for
$f\in L^{2}\left(\mu\right)$. By (\ref{eq:lcg-8}), this is equivalent
to 
\begin{equation}
\int e^{i2\pi x\lambda}f\left(\lambda\right)d\mu\left(\lambda\right)=0,\quad\forall x\in\Omega,\label{eq:t4}
\end{equation}
where $\Omega\subset\mathbb{R}$ is chosen open and bounded such that
the initial p.d. function $F$ is defined on $\Omega-\Omega$. But,
if $supp\left(\mu\right)$ is assumed compact, then the function $\Psi_{f,\mu}\left(x\right)=\mbox{l.h.s.}$
in (\ref{eq:t4}) has an entire analytic extension to $\mathbb{C}$.
And then (\ref{eq:t4}) implies that $\Psi_{f,\mu}=0$ in some interval;
and therefore for all $x\in\mathbb{R}$. The conclusion $f=0$ in
$L^{2}\left(\mathbb{R},\mu\right)$ then follows by a standard Fourier
uniqueness theorem. 

(\ref{enu:Tu2}) Let $f$, $F$, $\Omega:=\left(0,\frac{1}{2}\right)$,
and $\mu$ be as specified in (\ref{eq:t0})-(\ref{eq:t1}). Now writing
out $T_{\mu}^{*}f$, we get for all $x\in\Omega$ (see Table \ref{tab:F1-F6}):
\begin{eqnarray}
\left(fd\mu\right)^{\vee}\left(x\right) & \underset{\text{by \ensuremath{\left(\ref{eq:lcg-8}\right)}}}{=} & \left(\check{f}*\left(d\mu\right)^{\vee}\right)\left(x\right)=\left(\check{f}*\widetilde{F}\right)\left(x\right)\nonumber \\
 & = & \int_{\left|y\right|>2}\check{f}\left(y\right)\widetilde{F}\left(x-y\right)dy=0,\quad x\in\Omega.\label{eq:t6}
\end{eqnarray}
Note that, by (\ref{eq:5}), for all $x\in\Omega$, the function $\widetilde{F}\left(x-\cdot\right)$
is supported inside $\left(-2,2\right)$; and so the integral on the
r.h.s. in (\ref{eq:t6}) is zero. See Figure \ref{fig:tu}. We proved
that $Null(T_{\mu}^{*})\neq\left\{ 0\right\} $, and so $T_{\mu}$
is not onto.
\end{svmultproof2}

\begin{figure}
\includegraphics[width=0.9\textwidth]{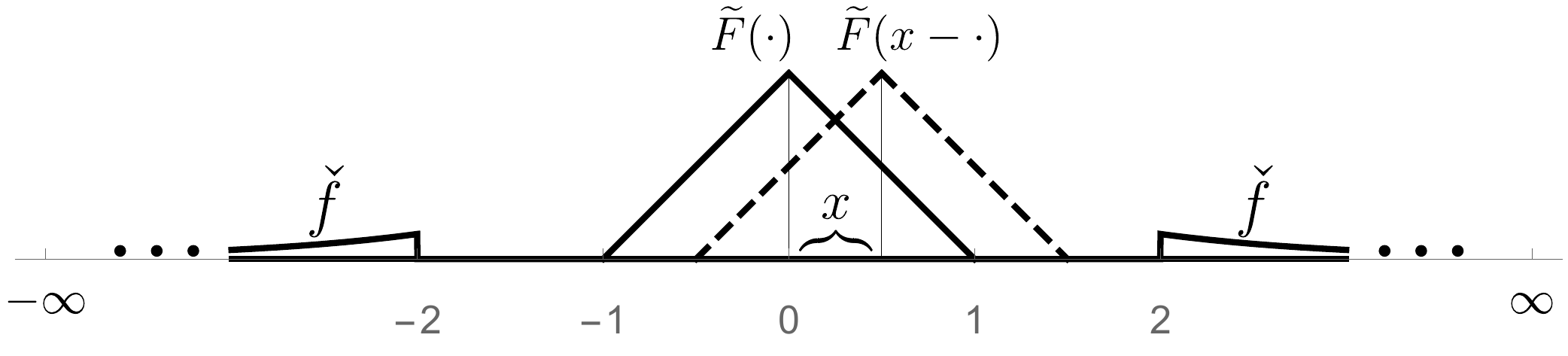}

\protect\caption{\label{fig:tu}The tent function. See also (\ref{eq:t0}) and (\ref{eq:5}).}

\end{figure}

\begin{remark}
In Section \ref{sec:Polya}, we study a family of extensions in the
the case $G=\mathbb{R}$, where $F:\left(-a,a\right)\rightarrow\mathbb{R}$
is symmetric ($F\left(x\right)=F\left(-x\right)$, $\left|x\right|<a$),
and $F\big|_{[0,a)}$ is assumed convex. By a theorem of Pólya, these
functions are positive definite. The positive definite extensions
$F^{\left(ext\right)}$ to $\mathbb{R}$ of given $F$ as above are
discussed in Section \ref{sec:Polya}. They are called \emph{spline
extensions}; see Figures \ref{fig:spline0}-\ref{fig:spline2}, and
Figure \ref{fig:spline3}. They will all have $F^{\left(ext\right)}$
of compact support. \index{convex}

As a result, the above argument from Remark \ref{rem:HFL2} shows
(with slight modification) that the spectral transform $T_{\mu}=T_{\mu^{\left(ext\right)}}$
(Corollary \ref{cor:lcg-isom}) will not map $\mathscr{H}_{F}$ onto
$L^{2}(\mathbb{R},\mu^{\left(ext\right)})$. Here, with $\mu^{\left(ext\right)}$,
we refer to the unique positive measure on $\mathbb{R}$ such that
$F^{\left(ext\right)}=\widehat{d\mu^{\left(ext\right)}}$, $F^{\left(ext\right)}$
is one of the spline extensions. For further details, see Figure \ref{fig:pexp}
and Section \ref{sec:Polya}.\end{remark}
\begin{example}
Let $F\left(x\right)=e^{-\left|x\right|}$, $\left|x\right|<1$, then
\[
F^{\left(ext\right)}\left(x\right)=\begin{cases}
e^{-\left|x\right|} & \text{if }\left|x\right|<1,\;F\mbox{ itself}\\
e^{-1}\left(2-\left|x\right|\right) & \mbox{if }1\leq\left|x\right|<2\\
0 & \mbox{if }\left|x\right|\ge2
\end{cases}
\]
is an example of a positive definite Pólya spline extension; see Figure
\ref{fig:pexp}.
\end{example}
\begin{figure}
\includegraphics[width=0.7\textwidth]{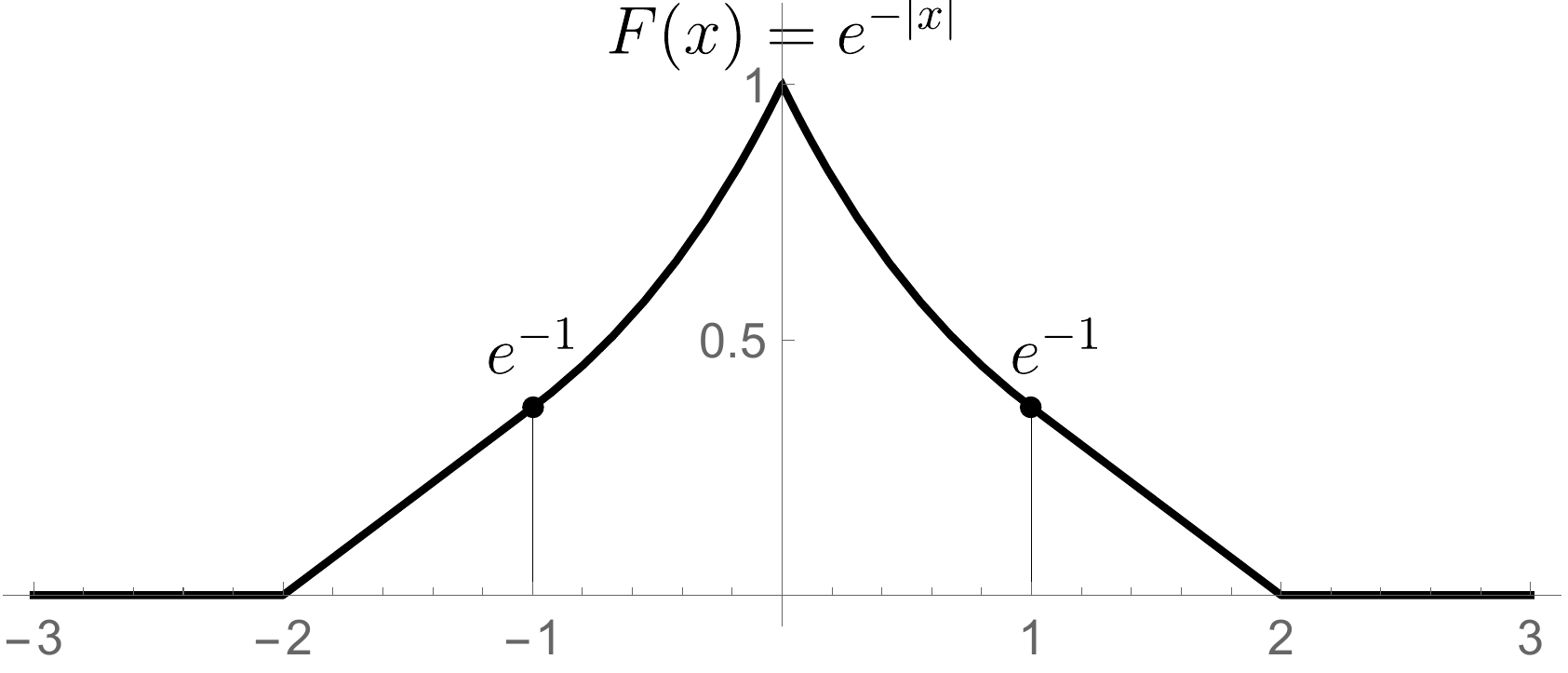}

\protect\caption{\label{fig:pexp}An example of a positive definite Pólya spline extension.
(For visual effect we use different scales on the two coordinate axes.)}

\end{figure}

\begin{theorem}
\label{thm:TT*}Let $G$ be a locally compact Abelian group. Let $\Omega\subset G$,
$\Omega\neq\emptyset$, open and connected. Let $F:\Omega-\Omega\rightarrow\mathbb{C}$
be continuous, positive definite; and assume $Ext\left(F\right)\neq\emptyset$.
Let $\mu\in Ext\left(F\right)$, and let $T_{\mu}\left(F_{\phi}\right):=\widehat{\varphi}$,
defined initially only for $\varphi\in C_{c}\left(\Omega\right)$,
be the isometry $T_{\mu}:\mathscr{H}_{F}\rightarrow L^{2}\left(\mu\right):=L^{2}(\widehat{G},\mu)$. 

Then $Q_{\mu}:=T_{\mu}T_{\mu}^{*}$ is a projection in \textup{$L^{2}\left(\mu\right)$
onto $ran\left(T_{\mu}\right)$;} and $ran\left(T_{\mu}\right)$ is
the $L^{2}\left(\mu\right)$-closure of the span of the functions
$\left\{ e_{x}\:|\:x\in\overline{\Omega}\right\} $, where $e_{x}\left(\lambda\right):=\left\langle \lambda,x\right\rangle $,
$\forall\lambda\in\widehat{G}$, and $\left\langle \lambda,x\right\rangle $
denotes the $\widehat{G}\leftrightarrow G$ duality. \end{theorem}
\begin{svmultproof2}
By Theorem \ref{thm:lcg-1}, $T_{\mu}:\mathscr{H}_{F}\rightarrow L^{2}\left(\mu\right)$
is isometric, and so $Q_{\mu}:=T_{\mu}T_{\mu}^{*}$ is the projection
in $L^{2}\left(\mu\right)$ onto $ran\left(T_{\mu}\right)$. 

We shall need the following three facts:
\begin{enumerate}
\item \label{enu:io1}$ran\left(T_{\mu}\right)=\left(null\left(T_{\mu}^{*}\right)\right)^{\perp}$,
where ``$\perp$'' refers to the standard inner product in $L^{2}\left(\mu\right)$.
\item \label{enu:io2}$\left(T_{\mu}^{*}f\right)\left(x\right)=\chi_{\overline{\Omega}}\left(x\right)\left(fd\mu\right)^{\vee}\left(x\right)$,
$x\in G$.
\item \label{enu:io3}Setting $e_{x}\left(\lambda\right)=\lambda\left(x\right)=\left\langle \lambda,x\right\rangle $,
$x\in\overline{\Omega}$, $\forall\lambda\in\widehat{G}$, and 
\[
E_{\mu}\left(\overline{\Omega}\right)=\overline{span}^{L^{2}\left(\mu\right)\mbox{-closure}}\left\{ e_{x}\left(\cdot\right)\:|\:x\in\overline{\Omega}\right\} ,
\]
we note that $e_{x}\in ran\left(T_{\mu}\right)$, for all $x\in\overline{\Omega}$. 
\end{enumerate}

\begin{flushleft}
\emph{Comments}. (\ref{enu:io1}) is general for isometries in Hilbert
space. (\ref{enu:io2}) is from Corollary \ref{cor:lcg-isom}. Finally,
(\ref{enu:io3}) follows from Corollary \ref{cor:muHF}.
\par\end{flushleft}

Now, let $f\in L^{2}\left(\mu\right)$; we then have the following
double-implications:
\begin{eqnarray*}
f & \in & \left(ran\left(T_{\mu}\right)\right)^{\perp}\\
 & \Updownarrow & \text{by \ensuremath{\left(\ref{enu:io1}\right)}}\\
\left(fd\mu\right)^{\vee}\left(x\right) & = & 0,\;\forall x\in\overline{\Omega}\\
 & \Updownarrow & \text{by \ensuremath{\left(\ref{enu:io2}\right)}}\\
\int_{\widehat{G}}\overline{\left\langle \lambda,x\right\rangle }f\left(\lambda\right)d\mu\left(\lambda\right) & = & 0,\;\forall x\in\overline{\Omega}\\
 & \Updownarrow\\
\left\langle e_{x},f\right\rangle _{L^{2}\left(\mu\right)} & = & 0,\;\forall x\in\overline{\Omega}.
\end{eqnarray*}
Now applying ``$\perp$'' one more time, we get that 
\[
ran\left(T_{\mu}\right)=\left\{ e_{x}\:|\:x\in\overline{\Omega}\right\} ^{\perp\perp}=E_{\mu}\left(\overline{\Omega}\right),
\]
which is the desired conclusion.
\end{svmultproof2}

\section{\label{sub:lie}Lie Groups}

\index{positive definite}

\index{extension problem}

\index{Lie group}

In this section, we consider \emph{the extension problem for continuous
positive definite functions} in Lie groups. 

A key ingredient in our approach to extensions of \emph{locally defined}
positive definite functions, in the context of Lie groups, is the
use of a correspondence between the following two aspects of representation
theory: 
\begin{enumerate}[label=(\roman{enumi}),ref=\roman{enumi}]
\item Unitary representations $U$ of a particular \emph{Lie group} $G$,
with $U$ acting in some Hilbert space $\mathscr{H}$; and 
\item a derived representations $dU$ of the associated \emph{Lie algebra}
$La(G)$. 
\end{enumerate}
This viewpoint originate with two landmark papers by Nelson \cite{Nel59},
and Nelson-Stinespring \cite{NS59}.

The following observations about (ii) will be used without further
mention in our discussion below, and elsewhere: First, given $U$,
the derived representation $dU$ of $La(G)$ is a representation by
unbounded skew-Hermitian operators $dU(X)$, for $X\in La\left(G\right)$,
with common dense domain in the Hilbert space $\mathscr{H}$.

To say that $dU$ is a representation refers to the Lie bracket in
$La(G)$ on one side, and the commutator bracket for operators on
the other. Further, $dU$, as a representation of $La(G)$, automatically
induces a representation of the corresponding associative enveloping
algebra to $La(G)$; this is the one referenced in Nelson-Stinespring
\cite{NS59}. But the most important point is that the correspondence
from (i) to (ii) is automatic, while the reverse from (ii) to (i)
is not. For example, the Lie algebra can only account for the subgroup
of $G$ which is the connected component of the unit element in $G$.
And, similarly, loops in $G$ cannot be accounted for by the Lie algebra.
So, at best, a given representation $\rho$ of the Lie algebra will
give information only about possible associated unitary representations
of the universal covering group to $G$; the two have the same Lie
algebra. But even then, there may not be a unitary representation
of this group coming as an integral of a Lie algebra representation
$\rho$: This is where the Laplace operator $\Delta$ of $G$ comes
in. Nelson\textquoteright s theorem states that a given Lie algebra
representation $\rho$ is integrable (to a unitary representation)
if and only if $\rho\left(\Delta\right)$ is essentially selfadjoint.

For a systematic account of the theory of integration of representations
of Lie algebras in infinite dimensional Hilbert space, and their applications
to physics, see \cite{JM84}.\index{Lie bracket} \index{commutator}
\begin{definition}
\label{def:li-1}Let $G$ be a Lie group. Assume $\Omega\subset G$
is a non-empty, connected and open subset. A continuous 
\begin{equation}
F:\Omega^{-1}\Omega\rightarrow\mathbb{C}\label{eq:li-1}
\end{equation}
is said to be \emph{positive definite}, iff (Def) 
\begin{equation}
\sum\nolimits _{i}\sum\nolimits _{i}\overline{c_{i}}c_{j}F\left(x_{i}^{-1}x_{j}\right)\geq0,\label{eq:li-2}
\end{equation}
for all finite systems $\left\{ c_{i}\right\} \subset\mathbb{C}$,
and points $\left\{ x_{i}\right\} \subset\Omega$. Equivalent,
\begin{equation}
\int_{\Omega}\overline{\varphi\left(x\right)}\varphi\left(y\right)F\left(x^{-1}y\right)dxdy\geq0,\label{eq:li-3}
\end{equation}
for all $\varphi\in C_{c}\left(\Omega\right)$; where $dx$ denotes
a choice of \emph{left-invariant} Haar measure on $G$.\index{measure!Haar}\end{definition}
\begin{lemma}
~
\begin{enumerate}
\item \label{enu:la1}Let $F$ be as in (\ref{eq:li-1})-(\ref{eq:li-2}).
For all $X\in La$$\left(G\right)=$ the Lie algebra of $G$, set
\begin{equation}
(\widetilde{X}\varphi)\left(g\right)=\frac{d}{dt}\Big|_{t=0}\varphi\left(\exp_{G}\left(-tX\right)g\right),\quad\varphi\in C_{c}^{\infty}\left(\Omega\right).\label{eq:li-4}
\end{equation}
Let 
\begin{equation}
F_{\varphi}\left(x\right)=\int_{\Omega}\varphi\left(y\right)F\left(x^{-1}y\right)dy\,;\label{eq:li-5}
\end{equation}
then 
\begin{equation}
S_{X}^{\left(F\right)}:F_{\varphi}\longmapsto F_{\widetilde{X}\varphi},\quad\varphi\in C_{c}^{\infty}\left(\Omega\right)\label{eq:li-6}
\end{equation}
defines a representation of the Lie algebra $La\left(G\right)$ by
skew-Hermitian operators in the RKHS $\mathscr{H}_{F}$, with the
operator in (\ref{eq:li-6}) defined on the common dense domain $\left\{ F_{\varphi}\:\big|\:\varphi\in C_{c}^{\infty}\left(\Omega\right)\right\} \subset\mathscr{H}_{F}$.
\index{skew-Hermitian operator; also called skew-symmetric}
\item \label{enu:la2}Let $F$ be as in part (\ref{enu:la1}) above, and
define $F_{\varphi}$ for $\varphi\in C_{c}^{\infty}\left(\Omega\right)$
as in (\ref{eq:li-5}), then $F_{\varphi}\in C^{\infty}\left(\Omega\right)$,
and 
\begin{equation}
\widetilde{X}(F_{\varphi})=F_{\widetilde{X}\varphi}\label{eq:li-61}
\end{equation}
holds on $\Omega$, where $\widetilde{X}$ for $X\in La\left(G\right)$
is the vector field defined in (\ref{eq:li-4}).
\end{enumerate}
\end{lemma}
\begin{svmultproof2}
Part (\ref{enu:la1}). The arguments here follow those of the proof
of Lemma \ref{lem:DF} \emph{mutatis mutandis}.

Part (\ref{enu:la2}). Let $F,\Omega,\varphi$, and $X$ be as in
the statement of the theorem, and let $g\in\Omega$ be fixed. Then
the r.h.s. in (\ref{eq:li-61}) exists as a pointwise limit as follows:

Let $t\in\mathbb{R}\backslash\left\{ 0\right\} $ be sufficiently
small, i.e., $\left|t\right|<a$; then 
\begin{eqnarray*}
 &  & \frac{1}{t}\left(F_{\varphi}\left(\exp_{G}\left(-tX\right)g\right)-F_{\varphi}\left(g\right)\right)\\
 & \overset{\begin{pmatrix}\mbox{invariance of}\\
\mbox{Haar measure}
\end{pmatrix}}{=} & \int_{\Omega}\frac{1}{t}\left(\varphi\left(\exp_{G}\left(-tX\right)y\right)-\varphi\left(y\right)\right)F\left(g^{-1}y\right)dy\\
 & \xrightarrow{\quad t\rightarrow0\quad} & \int_{\Omega}(\widetilde{X}\varphi)\left(y\right)F\left(g^{-1}y\right)dy=F_{\widetilde{X}\varphi}\left(g\right),
\end{eqnarray*}
and the desired conclusion (\ref{eq:li-61}) follows, i.e., the function
$F_{\varphi}$ is differentiable on the open set $\Omega$, with derivative
as specified w.r.t. the vector field $\widetilde{X}$. 

Note that $X\in La\left(G\right)$ is arbitrary. Since $\widetilde{X}\varphi\in C_{c}^{\infty}\left(\Omega\right)$,
the argument may be iterated; so we get $F_{\varphi}\in C^{\infty}\left(\Omega\right)$
as desired. 
\end{svmultproof2}

\index{unitary representation}

\index{operator!skew-Hermitian}

\index{Lie algebra}

\index{representation!-of Lie algebra}

\index{strongly continuous}
\begin{definition}
\label{def:li-2}Let $\Omega\subset G$, non-empty, open and connected. 
\begin{enumerate}
\item \label{enu:li-1}A continuous p.d. function $F:\Omega^{-1}\Omega\rightarrow\mathbb{C}$
is said to be extendable, if there is a continuous p.d. function $F_{ex}:G\rightarrow G$
such that
\begin{equation}
F_{ex}\big|{}_{\Omega^{-1}\Omega}=F.\label{eq:li-7}
\end{equation}

\item \label{enu:li-2}Let $U\in Rep\left(G,\mathscr{K}\right)$ be a strongly
continuous unitary representation of $G$ acting in some Hilbert space
$\mathscr{K}$. We say that $U\in Ext\left(F\right)$ (Section \ref{sub:ExtSpace}),
if there is an isometry $J:\mathscr{H}_{F}\hookrightarrow\mathscr{K}$
such that the function
\begin{equation}
G\ni g\longmapsto\left\langle JF_{e},U\left(g\right)JF_{e}\right\rangle _{\mathscr{K}}\label{eq:li-8}
\end{equation}
satisfies the condition (\ref{eq:li-7})
\end{enumerate}
\end{definition}
\begin{theorem}
Every extension of some continuous p.d. function $F$ on $\Omega^{-1}\Omega$
as in (\ref{enu:li-1}) arises from a unitary representation of $G$
as specified in (\ref{enu:li-2}).\end{theorem}
\begin{svmultproof2}
First assume some unitary representation $U$ of $G$ satisfies (\ref{enu:li-2}),
then (\ref{eq:li-8}) is an extension of $F$. This follows from the
argument in our proof of Theorem \ref{thm:pd-extension-bigger-H-space};
also see Lemma \ref{lem:DF} and \ref{lem:DF1}. 

For the converse; assume some continuous p.d. function $F_{ex}$ on
$G$ satisfies (\ref{eq:li-7}). Now apply the GNS (Theorem \ref{thm:gns})
to $F_{ex}$; and, as a result, we get a cyclic representation $\left(U,\mathscr{K},v_{0}\right)$
where 
\begin{itemize}
\item $\mathscr{K}$ is a Hilbert space;
\item $U$ is a strongly continuous unitary representation of $G$ acting
on $\mathscr{K}$; and
\item $v_{0}\in\mathscr{K}$ is a cyclic vector, $\left\Vert v_{0}\right\Vert =1$;
and
\begin{equation}
F_{ex}\left(g\right)=\left\langle v_{0},U\left(g\right)v_{0}\right\rangle ,\quad g\in G.\label{eq:li-9}
\end{equation}

\end{itemize}

Defining $J:\mathscr{H}_{F}\rightarrow\mathscr{K}$ by
\[
J\left(F\left(\cdot g\right)\right):=U\bigl(g^{-1}\bigr)v_{0},\quad\forall g\in\Omega;
\]
and extend by limit. We check that $J$ is isometric and satisfies
the condition from (\ref{enu:li-2}) in Definition \ref{def:li-2}.
We omit details as they parallel the arguments already contained in
Sections \ref{sec:Prelim}-\ref{sub:ExtSpace}.

\end{svmultproof2}

\index{GNS}

\index{Theorem!GNS-}

\index{representation!unitary-}

\index{representation!GNS-}\index{representation!cyclic-}
\begin{theorem}
\label{thm:gEn}Let $\Omega$, $G$, $La\left(G\right)$, and $F:\Omega^{-1}\Omega\rightarrow\mathbb{C}$
be as above. Let $\widetilde{G}$ be the simply connected universal
covering group for $G$. Then $F$ has an extension to a p.d. continuous
function on $\widetilde{G}$ iff there is a unitary representation
$U$ of $\widetilde{G}$, and an isometry $\mathscr{H}_{F}\overset{J}{\longrightarrow}\mathscr{K}$,
such that
\begin{equation}
JS_{X}^{\left(F\right)}=dU\left(X\right)J\label{eq:li-11}
\end{equation}
holds on $\left\{ F_{\varphi}\:\big|\:\varphi\in C_{c}^{\infty}\left(\Omega\right)\right\} $,
for all $X\in La\left(G\right)$; where 
\[
dU\left(X\right)U\left(\varphi\right)v_{0}=U\bigl(\widetilde{X}\varphi\bigr)v_{0}.
\]
\end{theorem}
\begin{svmultproof2}
Details are contained in Sections \ref{sec:embedding}, \ref{sub:G=00003DT},
and Chapter \ref{chap:types}.
\end{svmultproof2}

\index{group!covering}

\index{group!simply connected}

\index{representation!unitary-}

\index{group!universal covering-}

\index{universal covering group}

\index{Lie group}

\index{local translation}

Assume $G$ is connected. On $C_{c}^{\infty}\left(\Omega\right)$,
the Lie group $G$ acts locally by $\varphi\mapsto\varphi_{g}$, where
$\varphi_{g}:=\varphi\left(g^{-1}\cdot\right)$ denotes translation
of $\varphi$ by $g\in G$, such that $\varphi_{g}$ is also supported
in $\Omega$. Note that 
\begin{equation}
\left\Vert F_{\varphi}\right\Vert _{\mathscr{H}_{F}}=\left\Vert F_{\varphi_{g}}\right\Vert _{\mathscr{H}_{F}};\label{eq:li-1-1}
\end{equation}
but only for elements $g\in G$ in a neighborhood of $e\in G$ , and
with the neighborhood depending on $\varphi$. 
\begin{corollary}
Given 
\begin{equation}
F:\Omega^{-1}\cdot\Omega\rightarrow\mathbb{C}\label{eq:li-1-2}
\end{equation}
continuous and positive definite, then set 
\begin{equation}
L_{g}\left(F_{\varphi}\right):=F_{\varphi_{g}},\quad\varphi\in C_{c}^{\infty}\left(\Omega\right),\label{eq:li-1-3}
\end{equation}
defining a local representation of $G$ in $\mathscr{H}_{F}$.\end{corollary}
\begin{svmultproof2}
See \cite{Jor87,Jor86,JT14}. Equivalently, introducing the Hilbert
space $\mathfrak{M}_{2}\left(\Omega,F\right)$ from Corollary \ref{cor:muHF},
i.e., all  complex measures $\nu$ on $\Omega$ s.t. 
\[
\left\Vert \nu\right\Vert _{\mathfrak{M}_{2}\left(\Omega,F\right)}^{2}=\int_{\Omega}\int_{\Omega}F\left(x^{-1}y\right)\overline{d\nu\left(x\right)}d\nu\left(y\right)<\infty,
\]
we get the following formula for the local representation $L$:

For elements $x,y\in\Omega$ s.t. $xy\in\Omega$, we have
\[
L\left(x\right)\delta_{y}=\delta_{xy},\quad\mbox{and}\quad\left\langle \delta_{e},L\left(x\right)\delta_{e}\right\rangle _{\mathfrak{M}_{2}\left(\Omega,F\right)}=F\left(x\right)
\]
holds.
\end{svmultproof2}

\begin{corollary}
Given $F$, positive definite and continuous, as in (\ref{eq:li-1-2}),
and let $L$ be the corresponding local representation of $G$ acting
on $\mathscr{H}_{F}$. Then $Ext\left(F\right)\neq\emptyset$ if and
only if the local representation (\ref{eq:li-1-3}) extends to a global
unitary representation acting in some Hilbert space $\mathscr{K}$,
containing $\mathscr{H}_{F}$ isometrically.\end{corollary}
\begin{svmultproof2}
We refer to \cite{Jor87,Jor86} for details, as well as Chapter \ref{chap:Ext1}.
\end{svmultproof2}

\index{representation!local-}\index{Hilbert space}

\subsection{\label{sub:GNS}The GNS Construction}

Given a group $G$, or $C^{*}$-algebra $\mathfrak{A}$, the Gelfand\textendash Naimark\textendash Segal
(GNS for short) construction establishes an explicit correspondence
between representations on the one side and positive definite functionals
(states) in the case of $C^{*}$-algebras, and positive definite functions
in the case of groups. The non-trivial part is the construction of
the representation from the given positive definite object. In the
case of $C^{*}$-algebras the correspondence is between cyclic $*$-representations
of $\mathfrak{A}$ and certain linear functionals on $\mathfrak{A}$
(called states). The correspondence is shown by an explicit construction
of the $*$-representation from the state. In the case of groups $G$,
the correspondence is between unitary representations of $G$, with
cyclic vectors, on the one side, and positive definite functions $F$
on $G$ on the other. One further shows that the given p.d. function
$F$ on $G$ is continuous iff the corresponding unitary representation
is strongly continuous. 

The following is the version of the GNS construction most relevant
for our preset purpose. The construction is named after Israel Gelfand,
Mark Naimark, and Irving Segal; and it is abbreviated the GNS construction
\cite{Dix96}. 

\index{strongly continuous}\index{representation!cyclic-}
\begin{theorem}
\label{thm:gns}Let $G$ be a group, and let $F$ be a positive definite
(p.d.) function on $G$, i.e., for all $\left\{ x_{i}\right\} _{i=1}^{n}\subset G$,
the $n\times n$ matrix $\left[F\left(x_{i}^{-1}x_{j}\right)\right]_{i,j=1}^{n}$
is assumed positive semidefinite. Assume $F\left(e\right)=1$.
\begin{enumerate}
\item \label{enu:gns1}Then there is a Hilbert space $\mathscr{H}$, a vector
$v_{0}\in\mathscr{H}$, $\left\Vert v_{0}\right\Vert =1$, and a unitary
representation $U$, 
\begin{equation}
U:G\rightarrow\left(\mbox{unitary operators on }\mathscr{H}\right),\;\mbox{s.t.}\label{eq:gns1}
\end{equation}
\begin{equation}
F\left(x\right)=\left\langle v_{0},U\left(x\right)v_{0}\right\rangle _{\mathscr{H}},\quad\forall x\in G.\label{eq:gns2}
\end{equation}

\item \label{enu:gns2}The triple $\left(\mathscr{H},v_{0},U\right)$ may
be chosen such that $v_{0}$ is a \uline{cyclic vector} for the
unitary representation in (\ref{eq:gns1}). (We shall do this.)
\item If $(\mathscr{H}_{i},v_{0}^{\left(i\right)},U_{i})$, $i=1,2$, are
two solutions to the problem in (\ref{enu:gns1})-(\ref{enu:gns2}),
i.e., 
\begin{equation}
F\left(x\right)=\left\langle v_{0}^{\left(i\right)},U_{i}\left(x\right)v_{0}^{\left(i\right)}\right\rangle _{\mathscr{H}_{i}},\;i=1,2\label{eq:gns3}
\end{equation}
holds for all $x\in G$; then there is an isometric intertwining isomorphism
\\
$W:\mathscr{H}_{1}\rightarrow\mathscr{H}_{2}$ (onto $\mathscr{H}_{2}$)
such that\index{operator!intertwining-}
\begin{equation}
WU_{1}\left(x\right)=U_{2}\left(x\right)W\label{eq:gns4}
\end{equation}
holds on $\mathscr{H}_{1}$, for all $x\in G$, the intertwining property.
\end{enumerate}
\end{theorem}
\begin{svmultproof2}
We shall refer to \cite{Dix96} for details. The main ideas are sketched
below:

Starting with a given $F:G\rightarrow\mathbb{C}$, p.d., as in the
statement of the theorem, we build the Hilbert space (from part (\ref{enu:gns1}))
as follows:

Begin with the space of all finitely supported functions $\psi:G\rightarrow\mathbb{C}$,
denoted by $Fun_{fin}\left(G\right)$. For $\varphi,\psi\in Fun_{fin}\left(G\right)$,
set 
\begin{equation}
\left\langle \varphi,\psi\right\rangle _{\mathscr{H}}:=\sum_{x}\sum_{y}\overline{\varphi\left(x\right)}\psi\left(y\right)F\left(x^{-1}y\right).\label{eq:gns5}
\end{equation}
From the definition, it follows that 
\[
\left\langle \psi,\psi\right\rangle _{\mathscr{H}}\ge0,\quad\forall\psi\in Fun_{fin}\left(G\right);
\]
and moreover, that 
\[
N_{F}=\left\{ \psi\:|\:\left\langle \psi,\psi\right\rangle _{\mathscr{H}}=0\right\} 
\]
is a subspace of $Fun_{fin}\left(G\right)$. To get $\mathscr{H}\left(=\mathscr{H}_{F}\right)$
as a Hilbert space, we do the following two steps:
\begin{equation}
Fun_{fin}\left(G\right)\xrightarrow[\left(\text{quotient}\right)]{}Fun_{fin}\left(G\right)/N_{F}\xrightarrow[\left(\text{Hilbert completion}\right)]{}\mathscr{H}.\label{eq:gns6}
\end{equation}
For $v_{0}\in\mathscr{H}$, we choose 
\begin{equation}
v_{0}=\mbox{class}\left(\delta_{e}\right)=\left(\delta_{e}+N_{F}\right)\in\mathscr{H},\;\mbox{see}\;\left(\ref{eq:gns6}\right).\label{eq:gns7}
\end{equation}
In particular, for all $\varphi,\psi\in Fun_{fin}\left(G\right)$
and $g\in G$, we have
\begin{eqnarray*}
\left\langle \varphi\left(g^{-1}\cdot\right),\psi\left(g^{-1}\cdot\right)\right\rangle _{\mathscr{H}} & = & \sum_{x}\sum_{y}\overline{\varphi\left(x\right)}\psi\left(y\right)F\left(x^{-1}y\right)\\
 & = & \left\langle \varphi,\psi\right\rangle _{\mathscr{H}},\;\mbox{see}\;\left(\ref{eq:gns5}\right).
\end{eqnarray*}

Set $U\left(g\right)\psi=\psi\left(g^{-1}\cdot\right)$, i.e., 
\begin{equation}
\left(U\left(g\right)\psi\right)\left(x\right)=\psi\left(g^{-1}x\right),\quad x\in G.\label{eq:gns8}
\end{equation}
Using (\ref{eq:gns5})-(\ref{eq:gns7}), it is clear that $U$ as
defined in (\ref{eq:gns8}) passes to the quotient and the Hilbert
completion from (\ref{eq:gns6}), and that it will be the solution
to our problem.

Finally, the uniqueness part of the theorem is left for the reader.
See, e.g., \cite{Dix96}.\end{svmultproof2}

\begin{corollary}
\label{cor:gns}Assume that the group $G$ is a topological group.
Let $F$ be a p.d. function specified as in Theorem \ref{thm:gns}
(\ref{enu:gns1}), and let $\left(\mathscr{H}_{F},v_{0},U\right)$
be the representation triple (cyclic, see part (\ref{enu:gns2}) of
the theorem), solving the GNS problem. Then $F$ is continuous if
and only if the representation $U$ is strongly continuous, i.e.,
if, for all $\psi\in\mathscr{H}_{F}$, the function\index{group!topological-}
\[
G\ni x\longmapsto U\left(x\right)\psi\in\mathscr{H}_{F}
\]
is continuous in the norm of $\mathscr{H}_{F}$; iff for all $\varphi,\psi\in\mathscr{H}_{F}$,
the function
\[
F_{\varphi,\psi}\left(x\right)=\left\langle \varphi,U\left(x\right)\psi\right\rangle _{\mathscr{H}_{F}}
\]
is continuous on $G$.\index{representation!GNS-}\index{representation!cyclic-}
\end{corollary}

\subsection{\label{sub:lgns}Local Representations}

\index{strongly continuous}
\begin{definition}
Let $G$ be a Lie group and let $\mathscr{O}$ be an open connected
neighborhood of $e\in G$. A \emph{local representation} (defined
on $\mathscr{O}$) is a function:
\[
L:\mathscr{O}\longrightarrow\left(\mbox{isometric operators in some Hilbert space \ensuremath{\mathscr{H}}}\right)
\]
such that 
\begin{equation}
L\left(e\right)=I_{\mathscr{H}}\left(=\mbox{the identity operator in \ensuremath{\mathscr{H}}}\right);\label{eq:lo2}
\end{equation}
and if $x,y\in\mathscr{O}$, and $xy\in\mathscr{O}$ as well, then
\begin{equation}
L\left(xy\right)=L\left(x\right)L\left(y\right),\label{eq:lo3}
\end{equation}
where $xy$ denotes the group operation from $G$. \end{definition}
\begin{theorem}[GNS for local representations]
\label{thm:lgns} Let $G$ be as above, and let $\Omega\subset G$
be an open connected subset. Set
\begin{equation}
\mathscr{O}=\Omega^{-1}\Omega=\left\{ x^{-1}y\:|\:x,y\in\Omega\right\} .\label{eq:lo4}
\end{equation}
Let $F$ be a continuous (locally defined) p.d. function on $\mathscr{O}=\Omega^{-1}\Omega$;
and assume $F\left(e\right)=1$. \index{locally defined}
\begin{enumerate}
\item Then there is a triple $\left(\mathscr{H},v_{0},L\right)$, where
$\mathscr{H}$ is a Hilbert space, $v_{0}\in\mathscr{H}$, $\left\Vert v_{0}\right\Vert =1$,
and $L$ is an $\mathscr{O}$-local representation, strongly continuous,
such that
\begin{equation}
F\left(x\right)=\left\langle v_{0},L\left(x\right)v_{0}\right\rangle _{\mathscr{H}},\quad\forall x\in\mathscr{O}.\label{eq:lo5}
\end{equation}

\item If $L$ is cyclic w.r.t. $v_{0}$, then the solution $\left(\mathscr{H},v_{0},L\right)$
is unique up to unitary equivalence.
\item Assume $L$ is cyclic; then on the dense subspace $\mathscr{D}_{L}$
of vectors spanned by 
\begin{equation}
L_{\varphi}v_{0}=\int_{\mathscr{O}}\varphi\left(x\right)L\left(x\right)v_{0}dx,\quad\varphi\in C_{c}^{\infty}\left(\mathscr{O}\right)\label{eq:lo6}
\end{equation}
(integration w.r.t. a choice of a left-Haar measure on $G$), we get
a representation of the Lie algebra, $La\left(G\right)$ of $G$,
acting as follows: \\
Set $\lambda:=dL$, and 
\begin{equation}
\lambda\left(X\right)L_{\varphi}:=L_{\widetilde{X}\varphi}\label{eq:lo7}
\end{equation}
where $\widetilde{X}=$ the vector field (\ref{eq:li-4}). Specifically,
on $\mathscr{D}_{L}$, we have, 
\begin{eqnarray}
\lambda\left(\left[X,Y\right]\right) & = & \left[\lambda\left(X\right),\lambda\left(X\right)\right]\label{eq:lo8}\\
 & = & \lambda\left(X\right)\lambda\left(Y\right)-\lambda\left(Y\right)\lambda\left(X\right)\nonumber 
\end{eqnarray}
(commutator), for all $X,Y\in La\left(G\right)$, where $\left[X,Y\right]$
on the l.h.s. in (\ref{eq:lo8}) denotes the Lie bracket from $La\left(G\right)$.
\index{Lie bracket}\index{commutator}
\end{enumerate}
\end{theorem}
\begin{svmultproof2}
The reader will notice that the proof of these assertions follows
closely those details outlined in the proofs of the results from Sections
\ref{sub:GR}, \ref{sub:lie} and \ref{sub:GNS}, especially Theorem
\ref{thm:gns} and Corollary \ref{cor:gns}. The reader will easily
be able to fill in details; but see also \cite{JT14}.\end{svmultproof2}

\begin{corollary}
Let $G$ be a connected locally compact group, and let $\mathscr{O}$
be an open neighborhood of $e$ in $G$. Let $F$ be a continuous
positive definite function, defined on $\mathscr{O}$, and let $\widetilde{F}$
be a positive definite extension of $F$ to $G$; then $\widetilde{F}$
is continuous on $G$.\end{corollary}
\begin{svmultproof2}
Using Theorems \ref{thm:gns} and \ref{thm:lgns}, we may realize
both $F$ and $\widetilde{F}$ in GNS-representations, yielding a
local representation $L$ for $F$, and a unitary representation $U$
for the p.d. extension $\widetilde{F}$. This can be done in such
a way that $U$ is an extension of $L$, i.e., extending from $\mathscr{O}$
to $G$. Since $F$ is continuous, it follows that $L$ is strongly
continuous. But it is known that a unitary representation (of a connected
$G$) is strongly continuous iff it is strongly continuous in a neighborhood
of $e$. The result now follows from Corollary \ref{cor:gns}.\index{representation!GNS-}\end{svmultproof2}

\begin{remark}
It is not always possible to extend continuous p.d. functions $F$
defined on open subsets $\mathscr{O}$ in Lie groups. In Section \ref{sec:logz}
we give examples of locally defined p.d. functions $F$ on $G=\mathbb{R}^{2}$,
i.e., $F$ defined on some open subset $\mathscr{O}\subset\mathbb{R}^{2}$
for which $Ext\left(F\right)=\emptyset$. But when extension is possible,
the following algorithm applies, see Figure \ref{fig:gns}; also compare
with Figure \ref{fig:extc} in the case of $G=\mathbb{R}$.
\end{remark}
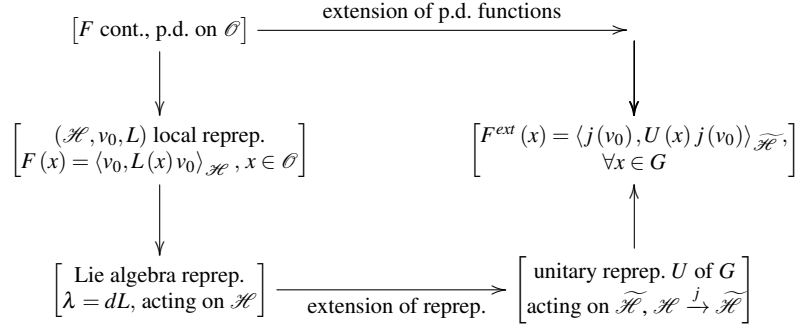
\begin{figure}
\[
\xymatrix{\begin{bmatrix}\mbox{\ensuremath{F} cont., p.d. on \ensuremath{\mathscr{O}}}\end{bmatrix}\ar[d]\ar[rr]\sp-{\mbox{extension of p.d. functions}} &  & \text{\ar[d]}\\
\text{\ensuremath{\begin{bmatrix}\left(\mathscr{H},v_{0},L\right)\:\mbox{local reprep.}\\
F\left(x\right)=\left\langle v_{0},L\left(x\right)v_{0}\right\rangle _{\mathscr{H}},\:x\in\mathscr{O}
\end{bmatrix}}}\ar[d] &  & \ensuremath{\begin{bmatrix}F^{ext}\left(x\right)=\left\langle j\left(v_{0}\right),U\left(x\right)j\left(v_{0}\right)\right\rangle _{\widetilde{\mathscr{H}}},\\
\forall x\in G
\end{bmatrix}}\\
\text{\ensuremath{\begin{bmatrix}\mbox{Lie algebra reprep.}\\
\lambda=dL,\:\mbox{acting on }\mathscr{H}
\end{bmatrix}}}\ar[rr]_{\mbox{extension of reprep.}} &  & \text{\ensuremath{\begin{bmatrix}\mbox{unitary reprep. \ensuremath{U} of \ensuremath{G}}\\
\mbox{acting on \ensuremath{\widetilde{\mathscr{H}}}, \ensuremath{\mathscr{H}\xrightarrow{\:j\:}\widetilde{\mathscr{H}}}}
\end{bmatrix}}}\ar[u]
}
\]

\protect\caption{\label{fig:gns}Extension correspondence. From locally defined p.d.
function $F$ to local representation $L$, to representation of the
Lie algebra, to a unitary representation $U$ of $G$, and finally
to an element in $Ext(F)$.}
\index{extension correspondence}
\end{figure}

\subsection{\label{sub:ext(Ext)}The Convex Operation in $Ext\left(F\right)$}
\begin{definition}
\label{def:eExt}Let $G$, $\Omega$, $\mathscr{O}=\Omega^{-1}\Omega$,
and $F$ be as in the statement of Theorem \ref{thm:lgns}; i.e.,
$F$ is a given continuous p.d. function, locally defined in $G$.
Assume $F\left(e\right)=1$. \index{convex}

Using Theorem \ref{thm:gns}, we consider ``points'' in $Ext\left(F\right)$
to be triples $\pi:=(\mathscr{H},U,v^{\left(0\right)})$ specified
as in (\ref{enu:gns1})-(\ref{enu:gns2}) of the theorem; i.e., 
\begin{itemize}
\item[$\mathscr{H}$:]  Hilbert space
\item[$U$:]  strongly continuous unitary representation of $G$
\item[$v^{\left(0\right)}$:]  vector in $\mathscr{H}$ s.t. $\Vert v^{\left(0\right)}\Vert_{\mathscr{H}}=1$.
\end{itemize}

Moreover, the restriction property
\begin{equation}
F\left(x\right)=\left\langle v^{\left(0\right)},U\left(x\right)v^{\left(0\right)}\right\rangle _{\mathscr{H}},\quad x\in\mathscr{O},\label{eq:eE1}
\end{equation}
holds. Denote the r.h.s. in (\ref{eq:eE1}), $\widetilde{F}_{\pi}$.

\end{definition}
Now let $\alpha\in\mathbb{R}$ satisfying $0<\alpha<1$, and let $\pi_{i}=(\mathscr{H}_{i},U_{i},v_{i}^{\left(0\right)})$,
$i=1,2$, be two ``points'' in $Ext\left(F\right)$; in particular,
(\ref{eq:eE1}) holds for both $\pi_{1}$ and $\pi_{2}$. The convex
combination is 
\begin{equation}
\pi=\alpha\pi_{1}+\left(1-\alpha\right)\pi_{2}.\label{eq:eE2}
\end{equation}
It is realized as follows:

Set $\mathscr{H}:=\mathscr{H}_{1}\oplus\mathscr{H}_{2}$, $U:=U_{1}\oplus U_{2}$,
and 
\begin{equation}
v^{\left(0\right)}=\sqrt{\alpha}v_{1}^{\left(0\right)}+\sqrt{1-\alpha}v_{2}^{\left(0\right)}\in\mathscr{H}=\mathscr{H}_{1}\oplus\mathscr{H}_{2};\label{eq:eE3}
\end{equation}
those are the standard direct sum operations in Hilbert space. \index{orthogonal sum}
\begin{claim}
The function
\[
\widetilde{F}\left(x\right)=\left\langle v^{\left(0\right)},U\left(x\right)v^{\left(0\right)}\right\rangle _{\mathscr{H}},\quad x\in G
\]
satisfies (\ref{eq:eE1}), and so it is in $Ext\left(F\right)$. \end{claim}
\begin{svmultproof2}
Setting 
\begin{equation}
\widetilde{F}_{i}\left(x\right)=\left\langle v_{i}^{\left(0\right)},U_{i}\left(x\right)v_{i}^{\left(0\right)}\right\rangle _{\mathscr{H}_{i}},\;i=1,2,\;\mbox{and }x\in G.\label{eq:eE5}
\end{equation}
Then, using (\ref{eq:eE3}), one easily verifies that 
\begin{equation}
\widetilde{F}\left(x\right)=\alpha\widetilde{F}_{1}\left(x\right)+\left(1-\alpha\right)\widetilde{F}_{2}\left(x\right)\label{eq:eE6}
\end{equation}
holds for all $x\in G$; and the desired conclusion follows.\end{svmultproof2}

\begin{remark}
\label{rem:eE}Using Definition \ref{def:eExt}, one now checks that,
for Examples 1-5 in Table \ref{tab:F1-F6} below, we get that
\begin{equation}
Ext_{1}\left(F\right)\subseteq ext\left(Ext\left(F\right)\right)\label{eq:eE7}
\end{equation}
where the ``ext'' on the r.h.s. in (\ref{eq:eE7}) denotes the set
of extreme points in $Ext\left(F\right)$, in the sense of Krein-Milman
\cite{Phe66}. 

In general, (\ref{eq:eE7}) is a \emph{proper} containment.\end{remark}
\begin{example}
Below we illustrate (\ref{eq:eE7}) in Remark \ref{rem:eE} with $F=F_{3}$
from Table \ref{tab:F1-F6}, i.e., $F\left(t\right)=e^{-\left|t\right|}$
in $\left|t\right|<1$; $G=\mathbb{G}$, and $\Omega=\left(0,1\right)$.
In this case, the skew-Hermitian operator $D^{\left(F\right)}$ from
Section \ref{sub:DF} has deficiency $\left(1,1\right)$. Using the
discussion above, one sees that every ``point'' $\pi=(\mathscr{H}^{\left(\pi\right)},U^{\left(\pi\right)},v^{\left(\pi\right)})$
in the closed convex hull, $\overline{conv}Ext_{1}\left(F\right)$,
of $Ext_{1}\left(F\right)$ has the following form. (See, e.g., \cite{Phe66}.)\index{convex}

Fix such a $\pi$. There is then a unique probability measure $\rho^{\pi}$
on $\mathbb{T}=\mathbb{R}/\mathbb{Z}$ such that, up to unitary equivalence,
the three components of $\pi$ are as follows:\end{example}
\begin{itemize}
\item $\mathscr{H}^{\left(\pi\right)}$: measurable functions $f$ on $\mathbb{R}\times\mathbb{T}$
such that
\begin{align}
 & \chi_{\left[0,1\right]}\left(\cdot\right)f\left(\cdot,\alpha\right)\in\mathscr{H}_{F},\quad\alpha\in\mathbb{T};\label{eq:eEx1}\\
 & \left(f+f'\right)\left(x+n,\alpha\right)=\alpha^{n}\left(f-f'\right)\left(x,\alpha\right),\quad\forall x\in\mathbb{R},\:n\in\mathbb{Z},\:\alpha\in\mathbb{T};\label{eq:eEx2}
\end{align}
where $f$ is $C^{1}$ in the $x$-variable (in $\mathbb{R}$), and
measurable in $\alpha$ (in $\mathbb{T}$). And
\begin{equation}
\left\Vert f\right\Vert _{\mathscr{H}^{\left(\pi\right)}}^{2}:=\int_{\mathbb{T}}\left\Vert \chi_{\left[0,1\right]}\left(\cdot\right)f\left(\cdot,\alpha\right)\right\Vert _{\mathscr{H}_{F}}^{2}d\alpha^{\left(Haar\right)}<\infty,\label{eq:eEx3}
\end{equation}
where $d\alpha$ is the Haar measure on $\mathbb{T}$, i.e, $\alpha=e^{i2\pi\theta}$,
$0\leq\theta\leq1$, yields 
\[
d\alpha^{\left(Haar\right)}=d\theta.
\]
Set 
\begin{equation}
\mathscr{H}^{\left(\pi\right)}=\int_{\mathbb{T}}^{\oplus}\mathscr{H}_{F}d\rho^{\left(\pi\right)},\label{eq:eEx5}
\end{equation}
as a direct integral. \end{itemize}
\begin{remark}
Eq (\ref{eq:eEx2}) above is a \textquotedblleft boundary-condition.\textquotedblright{}
In detail: For fixed $\alpha\in\mathbb{T}$, this condition comes
from the von Neumann extension corresponding to the deficiency analysis
of our skew-Hermitian operator $D^{\left(F\right)}$. More specifically,
recall that, as an operator in $\mathscr{H}_{F}$ (with dense domain),
$D^{\left(F\right)}$ has deficiency indices $\left(1,1\right)$ and
so, by our analysis in Section \ref{sub:exp(-|x|)}, every $\alpha\in\mathbb{T}$
corresponds to a unique skew-adjoint extension $A^{\left(\alpha\right)}$
of $D^{\left(F\right)}$. The conditions in (\ref{eq:eEx2}) simply
allow us to realize the unitary one-parameter group $U\left(t\right)$
generated by $A^{\left(\alpha\right)}$, with $U\left(t\right)$ realized
as per the specification in (\ref{eq:eEx7}). For further details
on boundary conditions in reproducing kernel Hilbert spaces of type
$\mathscr{H}_{F}$, see \cite{JT14a}.\end{remark}
\begin{itemize}
\item $U^{\left(\pi\right)}$: Up to unitary equivalence, we get 
\begin{equation}
U_{t}^{\left(\pi\right)}\left(f\left(\cdot,\alpha\right)\right)_{\alpha\in\mathbb{T}}\left(x\right):=\left(f\left(x+t,\alpha\right)\right)_{\alpha\in\mathbb{T}},\quad\forall t\in\mathbb{R};\label{eq:eEx4}
\end{equation}
where we use the representation (\ref{eq:eEx5}) above for $\mathscr{H}^{\left(\pi\right)}$.
\item Finally, we have the following direct integral representation of the
respective rank-one operators (using Dirac notation): \index{operator!rank-one-}\index{unitary one-parameter group}
\begin{equation}
\left|v^{\left(\pi\right)}\left\rangle \right\langle v^{\left(\pi\right)}\right|=\int_{\mathbb{T}}^{\oplus}d\rho^{\left(\pi\right)}\left(\alpha\right)\left|F_{0}\left\rangle \right\langle F_{0}\right|\label{eq:eEx6}
\end{equation}
The ket-bra notation $\left|v\left\rangle \right\langle v\right|$
means rank-one operator.
\end{itemize}
For $\alpha\in\mathbb{T}$ fixed, we note that 
\begin{equation}
\left(U_{t}^{\left(\alpha\right)}f\left(\cdot,\alpha\right)\right)\left(x\right)=f\left(x+t,\alpha\right),\quad t\in\mathbb{R},\label{eq:eEx7}
\end{equation}
represents a unique element in $Ext_{1}\left(F\right)$, and 
\begin{equation}
U_{t}^{\left(\pi\right)}=\int_{\mathbb{T}}^{\oplus}U_{t}^{\left(\alpha\right)}d\rho^{\left(\pi\right)}\left(\alpha\right),\quad t\in\mathbb{\mathbb{R}}\label{eq:eEx8}
\end{equation}
is the unitary representation corresponding to $\pi$, i.e., 
\begin{equation}
\widetilde{F^{\left(\pi\right)}}\left(t\right)=\int_{\mathbb{T}}\widetilde{F^{\left(\alpha\right)}}\left(t\right)d\rho^{\left(\pi\right)}\left(\alpha\right),\quad t\in\mathbb{R};\label{eq:eEx9}
\end{equation}
and $\widetilde{F^{\left(\pi\right)}}\left(\cdot\right)$ is the continuous
p.d. function on $\mathbb{R}$ which extends $F$ (on $\left(-1,1\right)$),
and corresponding to the ``point'' $\pi\in\overline{conv}Ext_{1}\left(F\right)$
\cite{Phe66}. 

The respective $\mathbb{R}$ p.d. functions in (\ref{eq:eEx9}) are
as follows:
\begin{equation}
\widetilde{F^{\left(\pi\right)}}\left(t\right)=\left\langle v^{\left(\pi\right)},U_{t}^{\left(\pi\right)}v^{\left(\pi\right)}\right\rangle _{\mathscr{H}^{\left(\pi\right)}},\quad t\in\mathbb{R},\label{eq:eEx10}
\end{equation}
and 
\begin{equation}
\widetilde{F^{\left(\alpha\right)}}\left(t\right)=\left\langle F_{0},U_{t}^{\left(\alpha\right)}F_{0}\right\rangle _{\mathscr{H}_{F}},\label{eq:eEx11}
\end{equation}
defined for all $t\in\mathbb{R}$, and $\alpha\in\mathbb{T}$. For
each $\alpha$ fixed, we have 
\begin{equation}
U_{t}^{\left(\alpha\right)}=e^{tA^{\left(\alpha\right)}},\quad t\in\mathbb{R},\label{eq:eEx12}
\end{equation}
where $A^{\left(\alpha\right)}=-(A^{\left(\alpha\right)})^{*}$ is
a \emph{skew-adjoint} extension of $D^{\left(F\right)}$, acting in
$\mathscr{H}_{F}$.
\begin{remark}
In the direct-integral formulas (\ref{eq:eEx5}) and (\ref{eq:eEx6}),
we stress that it is a direct integral of copies of the Hilbert space
$\mathscr{H}_{F}$; more precisely, an isomorphic copy of $\mathscr{H}_{F}$
for every $\alpha\in\mathbb{T}$. For fixed $\alpha\in\mathbb{T}$,
when we define our Hilbert space via eq (\ref{eq:eEx2}), then we
will automatically get the right boundary condition. They are built
in. And on this Hilbert space, then the action of $U(t)$ will simplify;
it will just be the translation formula (\ref{eq:eEx7}). In fact,
the construction is analogous to an induced representation \cite{Mac88}.
We must also need to justify ``up to unitary equivalence.'' It is
on account of Mackey's imprimitivity theorem, see \cite{Ors79}.

Our present direct integral representation should be compared with
the Zak transform from signal processing \cite{Jan03,ZM95}. 
\end{remark}

\motto{It was mathematics, the non-empirical science par excellence, wherein
the mind appears to play only with itself, that turned out to be the
science of sciences, delivering the key to those laws of nature and
the universe that are concealed by appearances.\\
\hfill--- Hannah Arendt, The Life of the Mind (1971), p.7.}

\chapter{Examples\label{chap:examples}}

While the present chapter is made up of examples, we emphasis that
the selection of examples is carefully chosen; -- chosen and presented
in such a way that the details involved, do in fact cover and illustrate
general and important ideas, which in turn serve to bring out the
structure of more general theorems (in the second half of our monograph.)
Moreover, the connection from the examples to more general contexts
will be mentioned inside the present chapter, on a case-by-case basis.
We further emphasize that some of the examples have already been used
to illustrate key ideas in Chapters \ref{chap:intro} and \ref{chap:ext}
above. Below we flesh out the details. And our present examples will
be used again in Chapters \ref{chap:Ext1} through \ref{chap:question}
below, dealing with theorems that apply to any number of a host of
general settings.

\section{\label{sub:euclid}The Case of $G=\mathbb{R}^{n}$}

Of course, $G=\mathbb{R}^{n}$ is a special case of locally compact
Abelian groups (Section \ref{sub:lcg}), but the results available
for $\mathbb{R}^{n}$ are more refined. We focus on this in the present
section. This is also the setting of the more classical studies of
the extension questions.

Let $\Omega\subset\mathbb{R}^{n}$ be a fixed open and connected subset;
and let $F:\Omega-\Omega\rightarrow\mathbb{C}$ be a given continuous
p.d. function, where 
\begin{equation}
\Omega-\Omega:=\left\{ x-y\in\mathbb{R}^{n}\:|\:x,y\in\Omega\right\} .\label{eq:rn1}
\end{equation}
Let $\mathscr{H}_{F}$ be the corresponding reproducing kernel Hilbert
space (RKHS). \index{RKHS} 

We showed (Theorem \ref{thm:pd-extension-bigger-H-space}) that $Ext\left(F\right)\not\neq\emptyset$
if and only if there is a strongly continuous unitary representation
$\{U\left(t\right)\}_{t\in\mathbb{R}^{n}}$ acting on $\mathscr{H}_{F}$
such that 
\begin{equation}
\mathbb{R}^{n}\ni t\longmapsto\left\langle F_{0},U\left(t\right)F_{0}\right\rangle _{\mathscr{H}_{F}}\label{eq:rn2}
\end{equation}
is a p.d. extension of $F$, extending from (\ref{eq:rn1}) to $\mathbb{R}^{n}$. 

Now if $U$ is a unitary representation of $G=\mathbb{R}^{n}$, we
denote by $P_{U}\left(\cdot\right)$ the associated projection valued
measure (PVM) on $\mathscr{B}\left(\mathbb{R}^{n}\right)$ (= the
sigma--algebra of all Borel subsets in $\mathbb{R}^{n}$). We have
\begin{equation}
U\left(t\right)=\int_{\mathbb{R}^{n}}e^{it\cdot\lambda}P_{U}\left(d\lambda\right),\quad\forall t\in\mathbb{R}^{n};\label{eq:rn3}
\end{equation}
where $t=\left(t_{1},\ldots,t_{n}\right)$, $\lambda=\left(\lambda_{1},\ldots,\lambda_{n}\right)$,
and $t\cdot\lambda=\sum_{j=1}^{n}t_{j}\lambda_{j}$. Setting
\begin{equation}
d\mu\left(\cdot\right)=\left\Vert P_{U}\left(\cdot\right)F_{0}\right\Vert _{\mathscr{H}_{F}}^{2},\label{eq:rn4}
\end{equation}
then the p.d. function on r.h.s. in (\ref{eq:rn2}) satisfies 
\begin{equation}
r.h.s.\left(\ref{eq:rn2}\right)=\int_{\mathbb{R}^{n}}e^{it\cdot\lambda}d\mu\left(\lambda\right),\quad\forall t\in\mathbb{R}^{n}.\label{eq:rn5}
\end{equation}

The purpose of the next theorem is to give an orthogonal splitting
of the RKHS $\mathscr{H}_{F}$ associated to a fixed $\left(\Omega,F\right)$
when it is assumed that $Ext\left(F\right)$ is non-empty. This orthogonal
splitting of $\mathscr{H}_{F}$ depends on a choice of $\mu\in Ext\left(F\right)$,
and the splitting is into three orthogonal subspaces of $\mathscr{H}_{F}$,
correspond a splitting of spectral types into atomic, absolutely continuous
(with respect to Lebesgue measure), and singular. 

\index{measure!PVM}

\index{projection-valued measure (PVM)}

\index{measure!Lebesgue}

\index{measure!singular}

\index{absolutely continuous}

\index{atom}

\index{orthogonal}

\index{spectral types}

\index{strongly continuous}\index{purely atomic}
\begin{theorem}
\label{thm:R^n-spect}Let $\Omega\subset\mathbb{R}^{n}$ be given,
$\Omega\neq\emptyset$, open and connected. Suppose $F$ is given
p.d. and continuous on $\Omega-\Omega$, and assume $Ext\left(F\right)\neq\emptyset$.
Let $U$ be the corresponding unitary representations of $G=\mathbb{R}^{n}$,
and let $P_{U}\left(\cdot\right)$ be its associated PVM acting on
$\mathscr{H}_{F}$.
\begin{enumerate}
\item \label{enu:rn1}Then $\mathscr{H}_{F}$ splits up as an orthogonal
sum of three closed and $U\left(\cdot\right)$-invariant subspaces\index{orthogonal sum}
\begin{equation}
\mathscr{H}_{F}=\mathscr{H}_{F}^{\left(atom\right)}\oplus\mathscr{H}_{F}^{\left(ac\right)}\oplus\mathscr{H}_{F}^{\left(sing\right)},\label{eq:rn6}
\end{equation}
characterized by the PVM $P_{U}$ as follows: 

\begin{enumerate}
\item[(a)] $P_{U}$ restricted to $\mathscr{H}_{F}^{\left(atom\right)}$ is
purely atomic;
\item[(b)] $P_{U}$ restricted to $\mathscr{H}_{F}^{\left(ac\right)}$ is absolutely
continuous with respect to the Lebesgue measure on $\mathbb{R}^{n}$;
and
\item[(c)] $P_{U}$ is continuous, purely singular, when restricted to $\mathscr{H}_{F}^{\left(sing\right)}$.
\end{enumerate}
\item \textbf{\label{enu:rn2}Case $\mathscr{H}_{F}^{\left(atom\right)}$.}
If $\lambda\in\mathbb{R}^{n}$ is an atom in $P_{U}$, i.e., $P_{U}\left(\left\{ \lambda\right\} \right)\neq0$,
where $\left\{ \lambda\right\} $ denotes the singleton with $\lambda$
fixed; then $P_{U}$$\left(\left\{ \lambda\right\} \right)\mathscr{H}_{F}$
is one-dimensional. Moreover, 
\begin{equation}
P_{U}\left(\left\{ \lambda\right\} \right)\mathscr{H}_{F}=\mathbb{C}e_{\lambda}\big|_{\Omega};\label{eq:rn7}
\end{equation}
where $e_{\lambda}\left(x\right):=e^{i\lambda\cdot x}$ is the complex
exponential. In particular, $e_{\lambda}\big|_{\Omega}\in\mathscr{H}_{F}$.\textbf{}\\
\textbf{Case $\mathscr{H}_{F}^{\left(ac\right)}$.} If $\xi\in\mathscr{H}_{F}^{\left(ac\right)}$,
then it is represented as a continuous function on $\Omega$, and
\begin{equation}
\left\langle \xi,F_{\varphi}\right\rangle _{\mathscr{H}_{F}}=\int_{\Omega}\overline{\xi\left(x\right)}\varphi\left(x\right)dx_{\text{(Lebesgue)}},\quad\forall\varphi\in C_{c}\left(\Omega\right).\label{eq:rn8}
\end{equation}
Further, there is a $f\in L^{2}\left(\mathbb{R}^{n},\mu\right)$ ($\mu$
as in (\ref{eq:rn4})) such that, for all $\varphi\in C_{c}\left(\Omega\right)$,
\begin{equation}
\int_{\Omega}\overline{\xi\left(x\right)}\varphi\left(x\right)dx=\int_{\mathbb{R}^{n}}\overline{f\left(\lambda\right)}\widehat{\varphi}\left(\lambda\right)d\mu\left(\lambda\right),\;\mbox{and}\label{eq:rn9}
\end{equation}
\begin{equation}
\xi=\left(fd\mu\right)^{\vee}\big|_{\Omega}.\label{eq:rn10}
\end{equation}
The r.h.s. of (\ref{eq:rn10}) is a $\mu$-extension of $\xi$.\textbf{
}Consequently, every $\mu$-extension $\widetilde{\xi}$ of $\xi$
is continuos on $\mathbb{R}^{n}$, and $\lim_{\left|x\right|\rightarrow\infty}\widetilde{\xi}\left(x\right)=0$,
i.e., $\widetilde{\xi}\in C_{0}\left(\mathbb{R}^{n}\right)$. \\
\\
\textbf{Case $\mathscr{H}_{F}^{\left(sing\right)}$.} Vectors $\xi\in\mathscr{H}_{F}^{\left(sing\right)}$
are characterized by the following property: The measure
\begin{equation}
d\mu_{\xi}\left(\cdot\right):=\left\Vert P_{U}\left(\cdot\right)\xi\right\Vert _{\mathscr{H}_{F}}^{2}\label{eq:rn11}
\end{equation}
is continuous and purely singular.
\end{enumerate}
\end{theorem}
\begin{svmultproof2}
Most of the proof details are contained in the previous discussion.
See Section \ref{sub:lcg}, Theorem \ref{thm:lcg-1}. For the second
part of the theorem:

Case $\mathscr{H}_{F}^{\left(atom\right)}$. Suppose $\lambda\in(\mathbb{R}^{n})$
is an atom, and that $\xi\in\mathscr{H}_{F}\backslash\left\{ 0\right\} $
satisfies $P_{U}\left(\left\{ \lambda\right\} \right)\xi=\xi$; then
\begin{equation}
U\left(t\right)\xi=e^{it\cdot\lambda}\xi,\quad\forall t\in\mathbb{R}^{n}.\label{eq:rn13}
\end{equation}
By (\ref{eq:rn2})-(\ref{eq:rn3}), we conclude that $\xi$ is a continuous,
weak solution to the elliptic system
\begin{equation}
\frac{\partial}{\partial x_{j}}\xi=\sqrt{-1}\lambda_{j}\xi\:\left(\mbox{on }\Omega\right),\quad1\leq j\leq n.\label{eq:rn14}
\end{equation}
Hence $\xi=\mbox{const}\cdot e_{\lambda}|_{\Omega}$ as asserted in
(\ref{eq:rn7}).

Case $\mathscr{H}_{F}^{\left(ac\right)}$; this follows from (\ref{eq:rn10})
and the Riemann-Lebesgue theorem applied to $\mathbb{R}^{n}$; and
the case $\mathscr{H}_{F}^{\left(sing\right)}$ is immediate.\end{svmultproof2}

\begin{example}
\label{ex:splitting}Consider the following continuous p.d. function
$F$ on $\mathbb{R}$, or on some bounded interval $\left(-a,a\right)$,
$a>0$. 
\begin{equation}
F\left(x\right)=\frac{1}{3}\left(e^{-ix}+\prod_{n=1}^{\infty}\cos\left(\frac{2\pi x}{3^{n}}\right)+e^{i3x/2}\frac{\sin\left(x/2\right)}{\left(x/2\right)}\right).\label{eq:r-2-1}
\end{equation}
The RKHS $\mathscr{H}_{F}$ has the decomposition (\ref{eq:rn6}),
where all three subspaces $\mathscr{H}_{F}^{\left(atom\right)}$,
$\mathscr{H}_{F}^{\left(ac\right)}$, and $\mathscr{H}_{F}^{\left(sing\right)}$
are non-zero; the first one is one-dimensional, and the other two
are infinite-dimensional. Moreover, the operator 
\begin{gather}
D^{\left(F\right)}\left(F_{\varphi}\right):=F_{\varphi'}\mbox{ on}\label{eq:r-2-2}\\
dom\big(D^{\left(F\right)}\big)=\left\{ F_{\varphi}\:\big|\:\varphi\in C_{c}^{\infty}\left(0,a\right)\right\} \nonumber 
\end{gather}
is bounded, and so extends by \emph{closure} to a skew-adjoint operator,
satisfying 
\[
\overline{D^{\left(F\right)}}=-(D^{\left(F\right)})^{*}.
\]
\index{operator!skew-adjoint-}\end{example}
\begin{svmultproof2}
Using infinite convolutions of operators (Chapter \ref{chap:conv}),
and results from \cite{DJ12}, we conclude that $F$ defined in (\ref{eq:r-2-1})
is entire analytic, and $F=\widehat{d\mu}$ (Bochner-transform) where
\index{Bochner, S.}\index{Theorem!Bochner's-}\index{Bochner's Theorem}
\begin{equation}
d\mu\left(\lambda\right)=\frac{1}{3}\left(\delta_{-1}+\mu_{c}+\chi_{\left[1,2\right]}\left(\lambda\right)d\lambda\right).\label{eq:r-2-3}
\end{equation}
The measures on the r.h.s. in (\ref{eq:r-2-3}) are as follows (Figure
\ref{fig:cantor}-\ref{fig:cantor2}): 
\begin{itemize}
\item $\delta_{-1}$ is the Dirac mass at $-1$, i.e., $\delta\left(\lambda+1\right)$.
\index{measure!Dirac}\index{transform!Bochner-}
\item $\mu_{c}=$ the middle-third Cantor measure; determined as the unique
solution in $\mathscr{M}_{+}^{prob}\left(\mathbb{R}\right)$ to \index{measure!Cantor}\index{Cantor measure}
\begin{eqnarray*}
 &  & \int f\left(\lambda\right)d\mu_{c}\left(\lambda\right)\\
 & = & \frac{1}{2}\left(\int f\left(\frac{\lambda+1}{3}\right)d\mu_{c}\left(\lambda\right)+\int f\left(\frac{\lambda-1}{3}\right)d\mu_{c}\left(\lambda\right)\right),\quad\forall f\in C_{c}\left(\mathbb{R}\right).
\end{eqnarray*}

\item $\chi_{\left[1,2\right]}\left(\lambda\right)d\lambda$ is restriction
to the closed interval $\left[1,2\right]$ of Lebesgue measure.
\end{itemize}

It follows from the literature (e.g. \cite{DJ12}) that $\mu_{c}$
is supported in $\left[-\frac{1}{2},\frac{1}{2}\right]$. Thus, the
three measures on the r.h.s. in (\ref{eq:r-2-3}) have disjoint compact
supports, with the three supports positively separately.

The conclusions asserted in Example \ref{ex:splitting} follow from
this, in particular the properties for $D^{\left(F\right)}$. In fact,
\begin{equation}
\mbox{spectrum}(iD^{\left(F\right)})\subseteq\left\{ -1\right\} \cup\left[-\tfrac{1}{2},\tfrac{1}{2}\right]\cup\left[1,2\right]\quad(i=\sqrt{-1})\label{eq:r-2-5}
\end{equation}

\end{svmultproof2}

\begin{figure}
\includegraphics[width=0.7\textwidth]{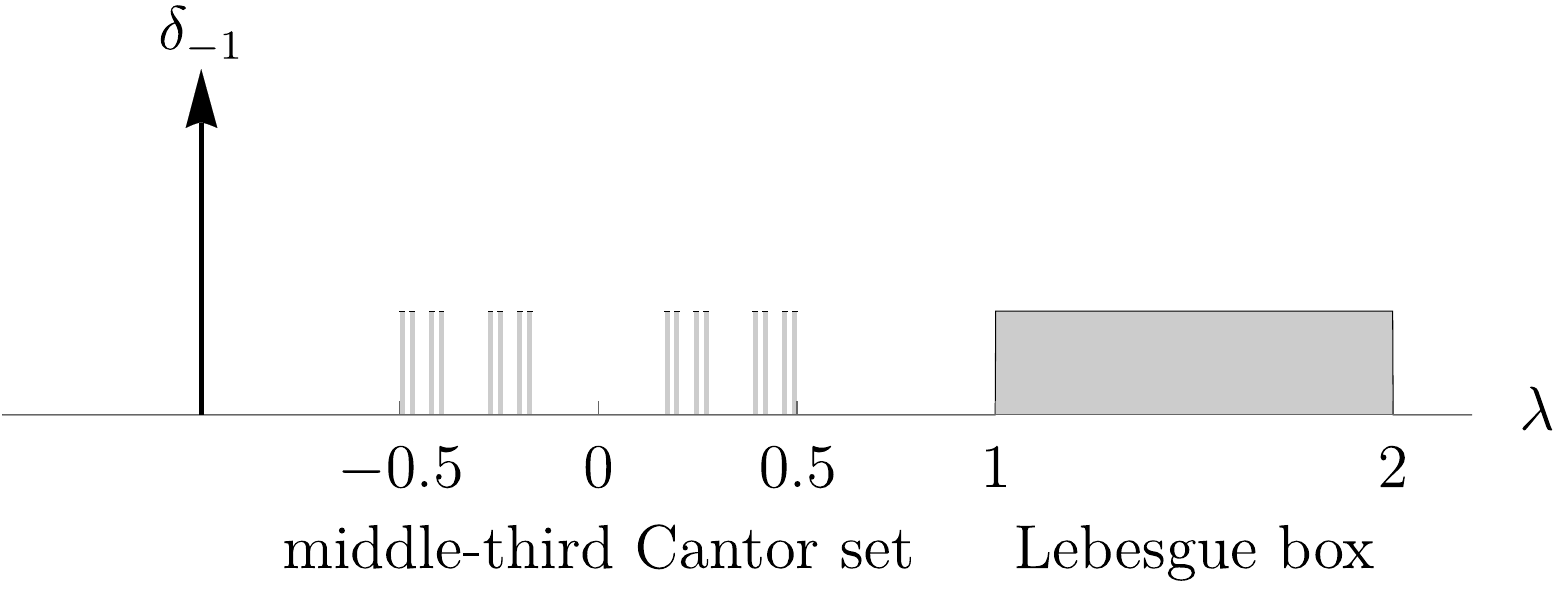}

\protect\caption{\label{fig:cantor}The measure $d\mu\left(\lambda\right)$ in Example
\ref{ex:splitting}. }
\end{figure}

\begin{figure}
\includegraphics[width=0.6\textwidth]{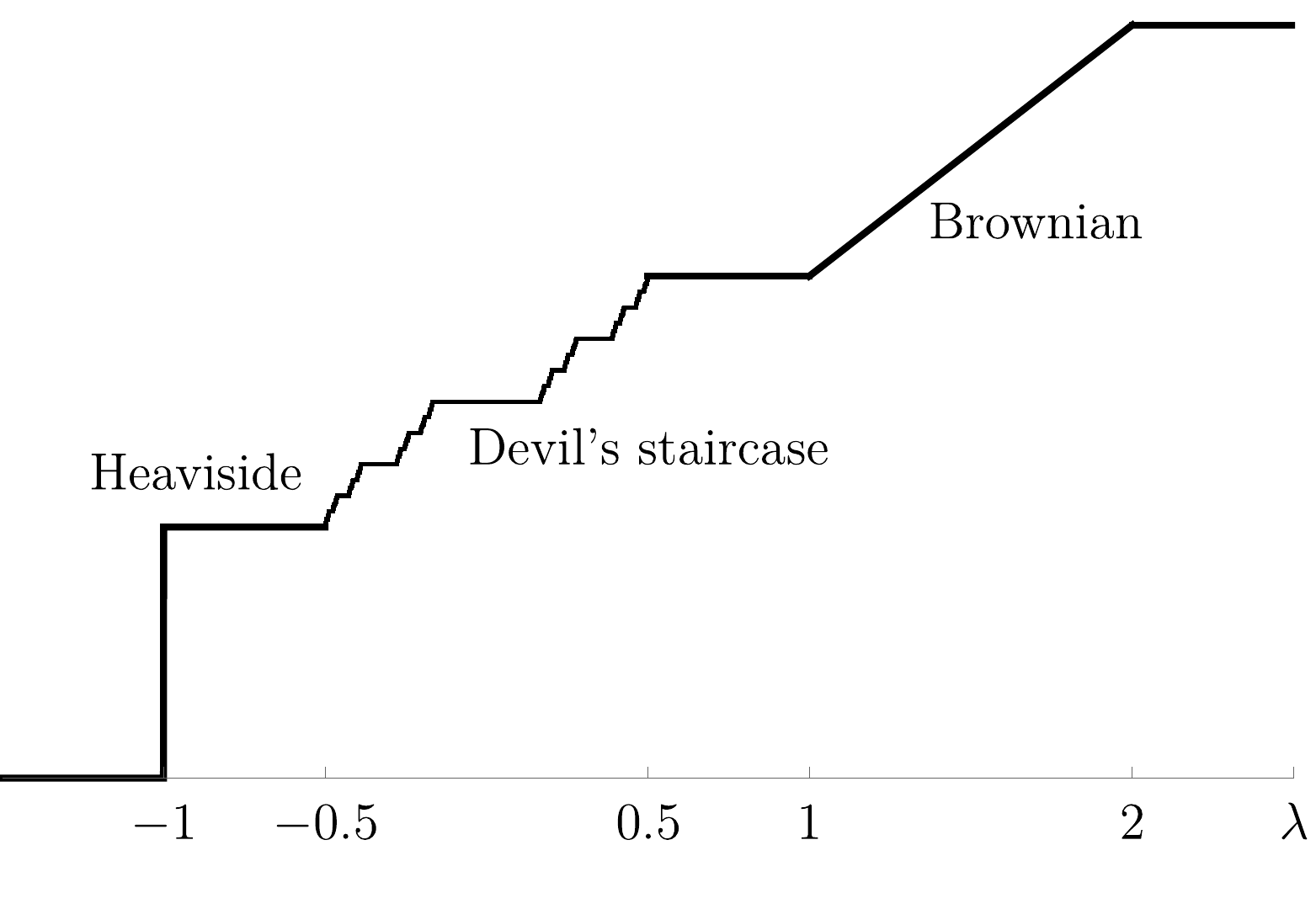}

\protect\caption{\label{fig:cantor2}Cumulative distribution $=\int_{-\infty}^{\lambda}d\mu\left(\lambda\right)$,
as in Example \ref{ex:splitting}.}
\end{figure}

\begin{corollary}
\label{cor:Tadj}Let $\Omega\subset\mathbb{R}^{n}$ be non-empty,
open and connected. Let $F:\Omega-\Omega\rightarrow\mathbb{C}$ be
a fixed continuous and p.d. function; and let $\mathscr{H}_{F}$ be
the corresponding RKHS (of functions on $\Omega$.)
\begin{enumerate}
\item \label{enu:rn-1}If there is a compactly supported measure $\mu\in Ext\left(F\right)$,
then every function $\xi$ on $\Omega$, which is in $\mathscr{H}_{F}$,
has an entire analytic extension to $\mathbb{C}^{n}$, i.e., extension
from $\Omega\subset\mathbb{R}^{n}$ to $\mathbb{C}^{n}$. 
\item \label{enu:rn-2}If, in addition, it is assumed that $\mu\ll d\lambda\left(=d\lambda_{1}\cdots d\lambda_{n}\right)=$
the Lebesgue measure on $\mathbb{R}^{n}$, where ``$\ll$'' means
``absolutely continuous;'' then the functions in $\mathscr{H}_{F}$
are restrictions to $\Omega$ of $C_{c}\left(\mathbb{R}^{n}\right)$-functions. 
\end{enumerate}
\end{corollary}
\begin{svmultproof2}
Part (\ref{enu:rn-1}). Let $\Omega$, $F$, $\mathscr{H}_{F}$ and
$\mu$ be as stated. Let $\mathscr{H}_{F}\overset{T}{\longrightarrow}L^{2}\left(\mathbb{R}^{n},\mu\right)$
be the isometry from Corollary \ref{cor:lcg-isom}, where 
\begin{equation}
T\left(F_{\varphi}\right):=\widehat{\varphi},\quad\varphi\in C_{c}\left(\Omega\right).\label{eq:rn15}
\end{equation}
By (\ref{eq:lcg-8}), for all $f\in L^{2}\left(\mathbb{R}^{n},\mu\right)$,
we have
\begin{equation}
\left(T^{*}f\right)\left(x\right)=\int_{\mathbb{R}^{n}}e^{ix\cdot\lambda}f\left(\lambda\right)d\mu\left(\lambda\right),\quad x\in\Omega;\label{eq:rn16}
\end{equation}
Equivalently, 
\begin{equation}
T^{*}f=\chi_{\overline{\Omega}}\left(fd\mu\right)^{\vee},\quad\forall f\in L^{2}\left(\mu\right).\label{eq:rn16a}
\end{equation}
And further that
\begin{equation}
T^{*}\left(L^{2}\left(\mathbb{R}^{n},\mu\right)\right)=\mathscr{H}_{F}.\label{eq:rn17}
\end{equation}
Now, if $\mu$ is of compact support, then so is the complex measure
$fd\mu$. This measure is finite since $L^{2}\left(\mu\right)\subset L^{1}\left(\mu\right)$.
Hence the desired conclusion follows from (\ref{eq:rn16}), (\ref{eq:rn17}),
and the Paley-Wiener theorem; see e.g., \cite{Rud73}.

Part (\ref{enu:rn-2}) of the corollary follows from Theorem \ref{thm:R^n-spect},
case $\mathscr{H}_{F}^{\left(ac\right)}$. See also Proposition \ref{prop:exp}. 

Details about (\ref{eq:rn16a}): For all $f\in L^{2}\left(\mu\right)$,
and all $\varphi\in C_{c}\left(\Omega\right)$, we have 
\begin{eqnarray*}
\left\langle T(F_{\varphi}),f\right\rangle _{L^{2}\left(\mu\right)} & = & \int_{\mathbb{R}}\overline{\widehat{\varphi}\left(\lambda\right)}f\left(\lambda\right)d\mu\left(\lambda\right)\\
 & = & \int_{\mathbb{R}}\left(\overline{\int_{\Omega}\varphi\left(x\right)e^{-i\lambda\cdot x}dx}\right)f\left(\lambda\right)d\mu\left(\lambda\right)\\
 & = & \int_{\mathbb{R}}\left(\overline{\int_{\mathbb{R}}\chi_{\overline{\Omega}}\left(x\right)\varphi\left(x\right)e^{-i\lambda\cdot x}dx}\right)f\left(\lambda\right)d\mu\left(\lambda\right)\\
 & \underset{\left(\text{Fubini}\right)}{=} & \int_{\mathbb{R}}\chi_{\overline{\Omega}}\left(x\right)\overline{\varphi\left(x\right)}\left(\int_{\mathbb{R}}e^{i\lambda\cdot x}f\left(\lambda\right)d\mu\left(\lambda\right)\right)dx\\
 & = & \int_{\mathbb{R}}\overline{\varphi\left(x\right)}\left(\chi_{\overline{\Omega}}\left(fd\mu\right)^{\vee}\right)\left(x\right)dx\\
 & = & \left\langle F_{\varphi},\chi_{\overline{\Omega}}\left(fd\mu\right)^{\vee}\right\rangle _{\mathscr{H}_{F}},
\end{eqnarray*}
which is the desired conclusion.
\end{svmultproof2}

\index{isometry}\index{Theorem!Paley-Wiener-}\index{Paley-Wiener}

\section{\label{sub:G=00003DT}The Case of $G=\mathbb{R}/\mathbb{Z}$}

While we consider extensions of locally defined continuous and positive
definite (p.d.) functions $F$ on groups, say $G$, the question of
whether $G$ is \emph{simply connected} or not plays an essential
role in the analysis, and in the possible extensions. 

It turns out that the geometric issues arising for general Lie groups
can be illustrated well in a simple case. To illustrate this point,
we isolate below the two groups, the circle group, and its universal
cover, the real line. We study extensions defined on a small arc in
the circle group $G=\mathbb{T}=\mathbb{R}/\mathbb{Z}$, versus extensions
to the universal covering group $\mathbb{R}$.

Let $G=\mathbb{T}$ represented as $\left(-\tfrac{1}{2},\tfrac{1}{2}\right]$.
Pick $0<\varepsilon<\tfrac{1}{2}$, and let $\Omega=\left(0,\varepsilon\right)$,
so that $\Omega-\Omega=(-\varepsilon,\varepsilon)$ mod $\mathbb{Z}$. 

Let $F:\Omega-\Omega\rightarrow\mathbb{C}$ be a continuous p.d. function;
$D^{\left(F\right)}$ is the canonical skew-Hermitian operator, s.t.
$dom(D^{\left(F\right)})=\{F_{\varphi}\:|\:\varphi\in C_{c}^{\infty}\left(\Omega\right)\}$,
$D^{\left(F\right)}(F_{\varphi})=F_{\varphi'}$, acting in the RKHS
$\mathscr{H}_{F}$. As shown in Section \ref{sec:Prelim}, $D^{\left(F\right)}$
has deficiency indices $\left(0,0\right)$ or $\left(1,1\right)$. 

\index{group!circle}

\index{universal covering group}

\index{dilation Hilbert space}

\index{selfadjoint extension}

\index{skew-Hermitian operator; also called skew-symmetric}

\index{boundary condition}

\index{locally defined}\index{Hilbert space}

\index{group!universal covering-}
\begin{lemma}
If $D^{\left(F\right)}$ has deficiency indices $(1,1)$, there is
a skew-adjoint extension of $D^{\left(F\right)}$ acting in $\mathscr{H}_{F}$,
such that the corresponding p.d. extension $\widetilde{F}$ of $F$
has period one; then $\varepsilon$ is rational. \end{lemma}
\begin{svmultproof2}
In view of Section \ref{sub:ExtSpace} (also see Theorem \ref{thm:pd-extension-bigger-H-space}),
our assumptions imply that that the dilation Hilbert space is identical
to $\mathscr{H}_{F}$, thus an ``internal'' extension. 

Since the deficiency indices of $D^{\left(F\right)}$ are $(1,1)$,
the skew-adjoint extensions of $D^{(F)}$ are determined by boundary
conditions of the form 
\begin{equation}
\xi(\varepsilon)=e^{i\theta}\xi(0),\label{eq:bda}
\end{equation}
where $0\leq\theta<1$ is fixed. 

Let $A_{\theta}$ be the extension corresponding to $\theta$, and
\[
U_{A_{\theta}}\left(t\right)=e^{tA_{\theta}}=\int_{\mathbb{R}}e^{it\lambda}P_{A_{\theta}}(d\lambda),\quad t\in\mathbb{R}
\]
be the corresponding unitary group, where $P_{A}\left(\cdot\right)$
is the projection-valued measure of $A_{\theta}$. Set 
\[
F_{A_{\theta}}(t)=\left\langle F_{0},U_{A_{\theta}}\left(t\right)F_{0}\right\rangle _{\mathscr{H}_{F}}=\int_{\mathbb{R}}e^{it\lambda}d\mu_{A_{\theta}}(\lambda)
\]
where 
\[
d\mu_{A_{\theta}}(\lambda)=\left\Vert P_{A_{\theta}}(d\lambda)F_{0}\right\Vert _{\mathscr{H}_{F}}^{2}.
\]

Repeating the calculation of the defect vectors and using that eigenfunctions
of $A_{\theta}$ must satisfy the boundary condition (\ref{eq:bda}),
it follows that the spectrum of $A_{\theta}$ is given by 
\[
\sigma\left(iA_{\theta}\right)=\left\{ \tfrac{\theta+n}{\varepsilon}\:|\:n\in\mathbb{Z}\right\} .
\]

Note that $supp(\mu_{A_{\theta}})\subset\sigma\left(iA_{\theta}\right)$,
and the containment becomes equal if $F_{0}$ is a cyclic vector for
$U_{A_{\theta}}\left(t\right)$. Since $F_{A}$ is assumed to have
period one, $supp(\mu_{A_{\theta}})$ consists of integers. Now if
$\lambda_{n},\lambda_{m}\in supp(\mu_{A_{\theta}})$, then 

\[
\varepsilon=\frac{n-m}{\lambda_{n}-\lambda_{m}}
\]
which is rational.
\end{svmultproof2}

\index{spectrum}\index{projection-valued measure (PVM)}
\begin{remark}
If 
\begin{equation}
F\left(x\right)=e^{-\left|x\right|},\quad-\pi<x\leq\pi\label{eq:ce1}
\end{equation}
is considered a function on $\mathbb{T}$ via $\mathbb{T}=\mathbb{R}/2\pi\mathbb{Z}$
(see Figure \ref{fig:circle5}), then its spectral representation
is 
\[
F\left(x\right)=\sum_{n\in\mathbb{Z}}e^{ixn}\frac{1}{\pi\left(1+n^{2}\right)},\quad\forall x\in\mathbb{T}=\mathbb{R}/2\pi\mathbb{Z}.
\]
\end{remark}
\begin{svmultproof2}
We must compute the Fourier \emph{series} expansion of $F$ as given
by (\ref{eq:ce1}) and the Figure \ref{fig:circle5} coordinate system.
Using normalized Haar measure on $\mathbb{T}$, the Fourier coefficients
are as follows:
\begin{eqnarray*}
Fourier\left(F,n\right) & = & \int_{\mathbb{T}}e^{-inx}F\left(x\right)dx\\
 & = & \int_{\mathbb{T}}e^{-inx}\left(\int_{\mathbb{R}}e^{i\lambda x}\frac{d\lambda}{\pi\left(1+\lambda^{2}\right)}\right)dx\\
 & \underset{\left(\text{by Fubini}\right)}{=} & \int_{\mathbb{R}}\left(\frac{1}{2\pi}\int_{-\pi}^{\pi}e^{i\left(\lambda-n\right)x}dx\right)\frac{d\lambda}{\pi\left(1+\lambda^{2}\right)}\\
 & = & \int_{\mathbb{R}}\frac{\sin\pi\left(\lambda-n\right)}{\pi\left(\lambda-n\right)}\frac{d\lambda}{\pi\left(1+\lambda^{2}\right)}\\
 & = & \frac{1}{\pi\left(1+n^{2}\right)},\quad\forall n\in\mathbb{Z},
\end{eqnarray*}
where we used Shannon's interpolation formula in the last step, and
using Figure \ref{fig:circle5}, to establish the band-limit property
required for Shannon; see also Sections \ref{sec:Polya}, \ref{sec:mercer},
especially Theorem \ref{thm:shannon}.
\end{svmultproof2}

\begin{figure}
\includegraphics[width=0.8\textwidth]{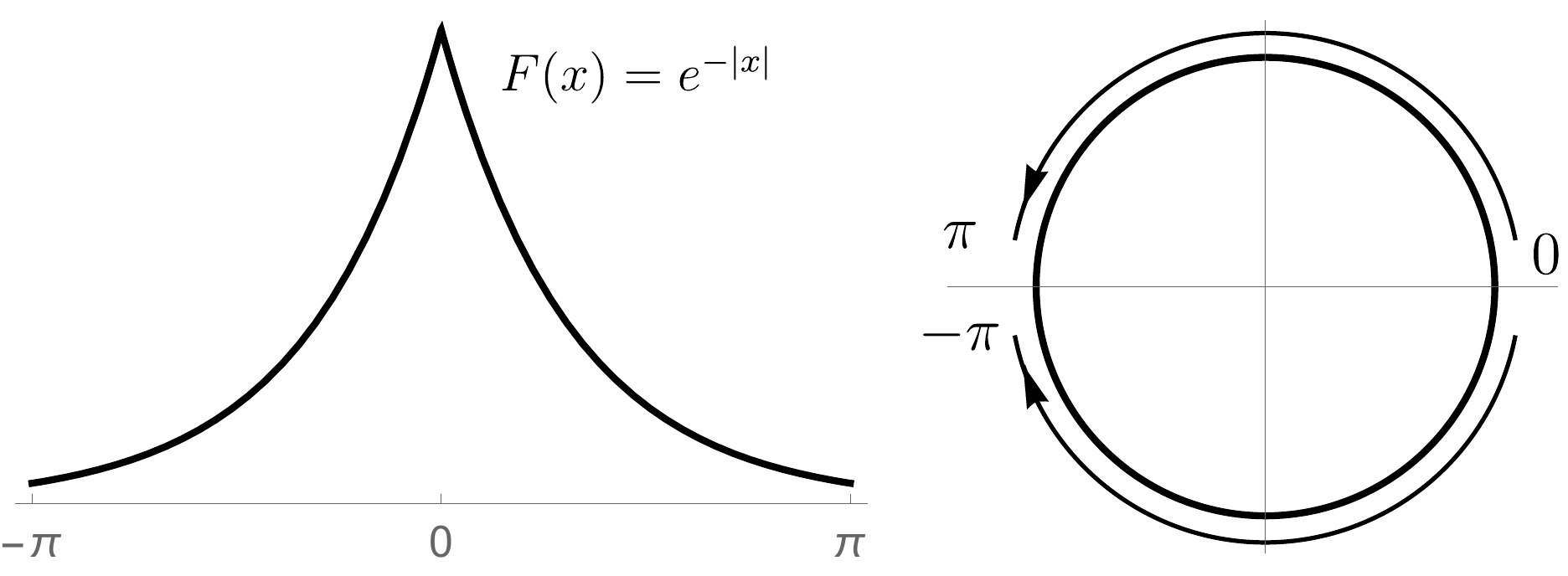}

\protect\caption{\label{fig:circle5}$F\left(x\right)=e^{-\left|x\right|}$ on $\mathbb{R}/2\pi\mathbb{Z}$}

\end{figure}

\section{\label{sec:expT}Example: $e^{i2\pi x}$}

This is a trivial example, but it is helpful to understand subtleties
of extensions of positive definite functions. In a different context,
one studies classes of unbounded Hermitian operators which arise in
scattering theory for wave equations in physics; see e.g., \cite{PeTi13,JPT12,LP89,LP85}.
For a discussion, see e.g., Section \ref{sub:momentum}, and Lemma
\ref{lem:(01)}.

\index{positive definite}

Consider the quantum mechanical momentum operator \index{operator!momentum-}
\[
P=-\frac{1}{2\pi i}\frac{d}{dx},\quad dom\left(P\right)=C_{c}^{\infty}\left(\mathbb{T}\right),
\]
acting in $L^{2}\left(\mathbb{T}\right)$, where $\mathbb{T}\simeq\mathbb{R}/\mathbb{Z}\simeq\left(-\frac{1}{2},\frac{1}{2}\right]$.
$P$ is Hermitian, with a dense domain in $L^{2}\left(\mathbb{T}\right)$.
It has deficiency indices $\left(1,1\right)$, and the family of selfadjoint
extensions is determined by the boundary condition (\ref{eq:bda}).
\index{boundary condition}\index{Hermitian}

Setting $\theta=0$ in (\ref{eq:bda}), the corresponding selfadjoint
(s.a.) extension $\widetilde{P}\supset P$ has spectrum 
\[
\sigma(\widetilde{P})=\left\{ \xi_{n}:=e^{i2\pi nx}\:|\:n\in\mathbb{Z}\right\} ;
\]
so that if $U\left(t\right)=e^{it\widetilde{P}}$, $t\in\mathbb{R}$,
then $\left(U\left(t\right)f\right)\left(x\right)=f\left(x+t\right)$,
for all $f\in L^{2}\left(\mathbb{T}\right)$. That is, $U\left(t\right)$
is the translation group. 
\begin{example}
Now fix $0<\epsilon<\frac{1}{2}$, and set 
\begin{equation}
F\left(x\right):=e^{i2\pi x}\big|{}_{\left(-\epsilon,\epsilon\right)}.\label{eq:e1}
\end{equation}
\end{example}
\begin{lemma}
The RKHS $\mathscr{H}_{F}$ of $F$ in (\ref{eq:e1}) is one-dimensional.
Moreover, for all $\varphi,\psi\in C_{c}^{\infty}\left(0,\epsilon\right)$,
we have
\begin{equation}
\left\langle F_{\varphi},F_{\psi}\right\rangle _{\mathscr{H}_{F}}=\overline{\widehat{\varphi}\left(1\right)}\widehat{\psi}\left(1\right).\label{eq:T0}
\end{equation}
In particular, 
\begin{equation}
\left\Vert F_{\varphi}\right\Vert _{\mathscr{H}_{F}}^{2}=\left|\widehat{\varphi}\left(1\right)\right|^{2}.\label{eq:T1}
\end{equation}
Here, $\widehat{\cdot}$ denotes Fourier transform. \index{RKHS}\end{lemma}
\begin{svmultproof2}
Recall that $\mathscr{H}_{F}$ is the completion of $span\left\{ F_{\varphi}:\varphi\in C_{c}^{\infty}\left(0,\epsilon\right)\right\} $,
where 
\begin{eqnarray*}
F_{\varphi}\left(x\right) & = & \int_{0}^{\epsilon}\varphi\left(y\right)F\left(x-y\right)dy\\
 & = & e^{i2\pi x}\int_{0}^{\epsilon}\varphi\left(y\right)e^{-i2\pi y}dy=e^{i2\pi x}\widehat{\varphi}\left(1\right);
\end{eqnarray*}
it follows that $\dim\mathscr{H}_{F}=1$.

Moreover, 
\begin{eqnarray*}
\left\langle F_{\varphi},F_{\psi}\right\rangle _{\mathscr{H}_{F}} & = & \int_{0}^{\epsilon}\int_{0}^{\epsilon}\overline{\varphi\left(x\right)}\psi\left(y\right)F\left(x-y\right)dxdy\\
 & = & \overline{\widehat{\varphi}\left(1\right)}\widehat{\psi}\left(1\right),\quad\forall\varphi,\psi\in C_{c}^{\infty}\left(0,\epsilon\right);
\end{eqnarray*}
which is the assertion (\ref{eq:T0}). 
\end{svmultproof2}

Note that $F$ in (\ref{eq:e1}) is the restriction of a p.d. function,
so it has at least one extension, i.e., $e^{i2\pi x}$ on $\mathbb{T}$.
In fact, this is also the unique continuous p.d. extension to $\mathbb{T}$.
\begin{lemma}
\label{lem:T}Let $F$ be as in (\ref{eq:e1}). If $\widetilde{F}$
is a continuous p.d. extension of $F$ on $\mathbb{T}$, then $\widetilde{F}\left(x\right)=e^{i2\pi x}$,
$x\in\left(-\frac{1}{2},\frac{1}{2}\right]$. \end{lemma}
\begin{svmultproof2}
Assume $\widetilde{F}$ is a continuous p.d. extension of $F$ to
$\mathbb{T}$, then by Bochner's theorem, we have\index{Theorem!Bochner's-}\index{Bochner's Theorem}
\begin{equation}
e^{i2\pi x}=\sum_{n\in\mathbb{Z}}\mu_{n}e^{i2\pi nx},\quad\forall x\in\left(-\epsilon,\epsilon\right)\label{eq:e-2-1}
\end{equation}
where $\mu_{n}\geq0$, and $\sum_{n\in\mathbb{Z}}\mu_{n}=1$. Now
each side of the above equation extends continuously in $x$. By uniqueness
of the Fourier expansion in (\ref{eq:e-2-1}), we get $\mu_{1}=1$,
and $\mu_{n}=0$, $n\in\mathbb{Z}\backslash\left\{ 1\right\} $. \end{svmultproof2}

\begin{remark}
In fact, Lemma \ref{lem:T} holds in a more general context. (We thank
the anonymous referee who kindly suggested the following abstract
argument.)\end{remark}
\begin{proposition}
Let $G$ be a connected locally compact Abelian group, $u$ a positive
definite function on $G$ and $\chi$ a character for which $\chi|_{U}=u|_{U}$
for some neighborhood of $U$ of the origin $1$. Then $u=\chi$.\end{proposition}
\begin{svmultproof2}
Indeed let $v=\overline{\chi}u$ which is $1$ on $U$. An easy argument
shows that then $v$ is also p.d. on $G$. 

We now pass to the Gelfand-Naimark-Segal (GNS) representation (see
Theorem \ref{thm:gns}), i.e., to a Hilbert space $\mathscr{H}$,
$\sigma$ a cyclic representation of $G$ acting on $\mathscr{H}$,
and a unit vector $\xi\in\mathscr{H}$ s.t. $v\left(x\right)=\left\langle \xi,\sigma\left(x\right)\xi\right\rangle _{\mathscr{H}}$,
for all $x\in G$. 

Notice that $H=\left\{ x\in G:v(x)=1\right\} $ is a subgroup. Indeed,
since $v=\langle\xi,\sigma(\cdot)\xi\rangle$ where $\|\xi\|=1$,
then $H=\{x\in G:\sigma(x)\xi=\xi\}$, thanks to uniform convexity.
But $H\supseteq U$ so $H\supseteq\bigcup_{n=1}^{\infty}U^{n}=G$,
where we used that $G$ is assumed connected. \index{representation!GNS-}\index{representation!cyclic-}
\end{svmultproof2}

\section{\label{sec:circle}Example: $e^{-\left|x\right|}$ in $\left(-a,a\right)$,
extensions to $\mathbb{T}=\mathbb{R}/\mathbb{Z}$}

Fix $0<a<\frac{1}{2}$. Let $F:\Omega-\Omega\rightarrow\mathbb{C}$
be a continuous p.d. function, where $\Omega=\left(0,a\right)$. 

For the extension problem (continuous, p.d.), we consider two cases:\index{positive definite}
\begin{equation}
\Omega\subset\mathbb{R}\quad\mbox{vs.}\quad\Omega\subset\mathbb{T}=\mathbb{R}/\mathbb{Z}\label{eq:circ}
\end{equation}
First, recall the Pontryagin duality (see \cite{Ru90}) for locally
compact Abelian (l.c.a) groups:\index{duality!Pontryagin-}\index{Theorem!Pontryagin-}
\[
\xymatrix{\boxed{\mbox{cont. p.d. functions \ensuremath{\Phi} on \ensuremath{G}}}\ar@2{<->}[rrr]\sp-{\mbox{Bochner}}\sb-{\mbox{transform}} &  &  & \mbox{\boxed{\begin{matrix}\mbox{finite positive Borel}\\
\mbox{measures \ensuremath{\mu} on \ensuremath{\widehat{G}}}
\end{matrix}}}}
\]
\begin{equation}
\Phi\left(x\right)=\int_{\widehat{G}}\chi\left(x\right)d\mu\left(\chi\right),\quad\chi\in\widehat{G},\:x\in G\label{eq:pjd1}
\end{equation}
\textbf{\uline{Application}}: In our current setting, $G=\mathbb{T}=\mathbb{R}/\mathbb{Z}$,
$\widehat{G}=\mathbb{Z}$, and 
\begin{equation}
\chi_{n}\left(x\right)=e^{i2\pi nx},\quad n\in\mathbb{Z},\:x\in\mathbb{T};\label{eq:pjd2}
\end{equation}
so that 
\begin{equation}
\Phi\left(x\right)=\sum_{n\in\mathbb{Z}}w_{n}e^{i2\pi nx},\quad w_{n}\geq0,\:\sum_{n\in\mathbb{Z}}w_{n}<\infty.\label{eq:pjd3}
\end{equation}
The weights $\left\{ w_{n}\right\} $ in (\ref{eq:pjd3}) determines
a measure $\mu_{w}$ on $\mathbb{Z}$, where
\begin{equation}
\mu_{w}\left(E\right):=\sum_{n\in E}w_{n},\quad\forall E\subset\mathbb{Z}.\label{eq:pjd4}
\end{equation}
Conclusions: Formula (\ref{eq:pjd3}) is a special case of formula
(\ref{eq:pjd1}); and (\ref{eq:pjd1}) is the general l.c.a. Pontryagin-Bochner
duality. (If $\Phi\left(0\right)=1$, then $\mu_{w}$ is a probability
measure on $\mathbb{Z}$.)
\begin{remark}
Distributions as in (\ref{eq:pjd4}) arise in many applications; indeed
\emph{stopping times} of Markov chains is a case in point \cite{Sok13,Du12}. 

Let $\left(\Omega,\mathscr{F},\mathbb{P}\right)$ be a probability
space, and let $X_{n}:\Omega\rightarrow\mathbb{R}$ be a stochastic
process indexed by $\mathbb{Z}$. Let $M\subset\mathbb{R}$ be a bounded
Borel subset. The corresponding stopping time $\tau_{M}$ is as follows:

For $\omega\in\Omega$, set 
\begin{equation}
\tau_{M}\left(\omega\right)=\inf\left\{ n\in\mathbb{Z}\:|\:X_{n}\left(\omega\right)\in M\right\} .\label{eq:st}
\end{equation}
The distribution $\mu_{\tau_{M}}$ of $\tau_{M}$ is an example of
a measure on $\mathbb{Z}$, i.e., an instance of the setting in (\ref{eq:pjd4}). 

For $n\in\mathbb{Z}$, set 
\[
w_{n}^{M}:=\mathbb{P}\left(\left\{ \omega\in\Omega\:|\:\tau_{M}\left(\omega\right)=n\right\} \right),\quad n\in\mathbb{Z}.
\]
Assume that the values of $\tau_{M}$ are finite so $\tau_{M}:\Omega\rightarrow\mathbb{Z}$.
Then it follows from the definition of stopping time (\ref{eq:st})
that $w_{n}^{M}\geq0$, $\forall n\in\mathbb{Z}$; and $\sum_{n\in\mathbb{Z}}w_{n}^{M}=1$.
\end{remark}

Back to the extension problem (\ref{eq:circ}). In fact, there are
many more solutions to the $\mathbb{R}$-problem than periodic solutions
on $\mathbb{T}$. By the Pontryagin duality (see \cite{Ru90}), any
periodic solution $\widetilde{F}_{per}$ has the representation (\ref{eq:pjd3}).\index{Theorem!Pontryagin-}
\begin{example}[The periodic extension continued]
\label{exa:expp} Recall that $e^{-\left|x\right|}$, $x\in\mathbb{R}$,
is positive definite on $\mathbb{R}$, where 
\begin{equation}
e^{-\left|x\right|}=\int_{-\infty}^{\infty}\frac{2}{1+4\pi^{2}\lambda^{2}}e^{i2\pi\lambda x}d\lambda,\quad x\in\mathbb{R}.\label{eq:Fper1}
\end{equation}
By the Poisson summation formula, we have:
\begin{equation}
\begin{split}\widetilde{F}_{per}\left(x\right) & =\sum_{n\in\mathbb{Z}}e^{-\left|x-n\right|}=\sum_{n\in\mathbb{Z}}w_{n}e^{i2\pi nx}\\
w_{n} & =\frac{2}{1+4\pi^{2}n^{2}}
\end{split}
\label{eq:poisson}
\end{equation}
$\widetilde{F}_{per}\left(x\right)$ is a continuous p.d. function
on $\mathbb{T}$; and obviously, an extension of the restriction $\widetilde{F}_{per}\big|_{\left(-a,a\right)}$.
Note that 
\begin{equation}
\sum_{n\in\mathbb{Z}}w_{n}=\sum_{n\in\mathbb{Z}}\frac{2}{1+4\pi^{2}n^{2}}=\coth\left(1/2\right)<\infty.\label{eq:poisson2}
\end{equation}

\end{example}

\paragraph{\textbf{General Consideration}}

There is a bijection between (i) functions on $\mathbb{T}$, and (ii)
$1$-periodic functions on $\mathbb{R}$. Note that (ii) includes
p.d. functions $F$ defined only on a subset $\left(-a,a\right)\subset\mathbb{T}$. 

Given $0<a<\frac{1}{2}$, if $F$ is continuous and p.d. on $\left(-a,a\right)\subset\mathbb{T}$,
then the analysis of the corresponding RKHS $\mathscr{H}_{F}$ is
totally independent of considerations of periods. $\mathscr{H}_{F}$
does not see global properties. To understand $F$ on $\left(-a,a\right)$,
only one period interval is needed; but when passing to $\mathbb{R}$,
things can be complicated. Hence, it is tricky to get extensions $\widetilde{F}_{per}$
to $\mathbb{T}$.
\begin{example}
\label{exa:expT}Let $F\left(x\right)=e^{-\left|x\right|}$, for all
$x\in\left(-a,a\right)$. See Figure \ref{fig:ext}. Then $\widetilde{F}\left(x\right)=e^{-\left|x\right|}$,
$x\in\mathbb{R}$, is a continuous p.d. extension to $\mathbb{R}$.
However, the periodic version in (\ref{eq:poisson}) is NOT a p.d.
extension of $F$. See Figure \ref{fig:pext}. 
\end{example}
\begin{figure}
\begin{tabular}{c}
\includegraphics[scale=0.4]{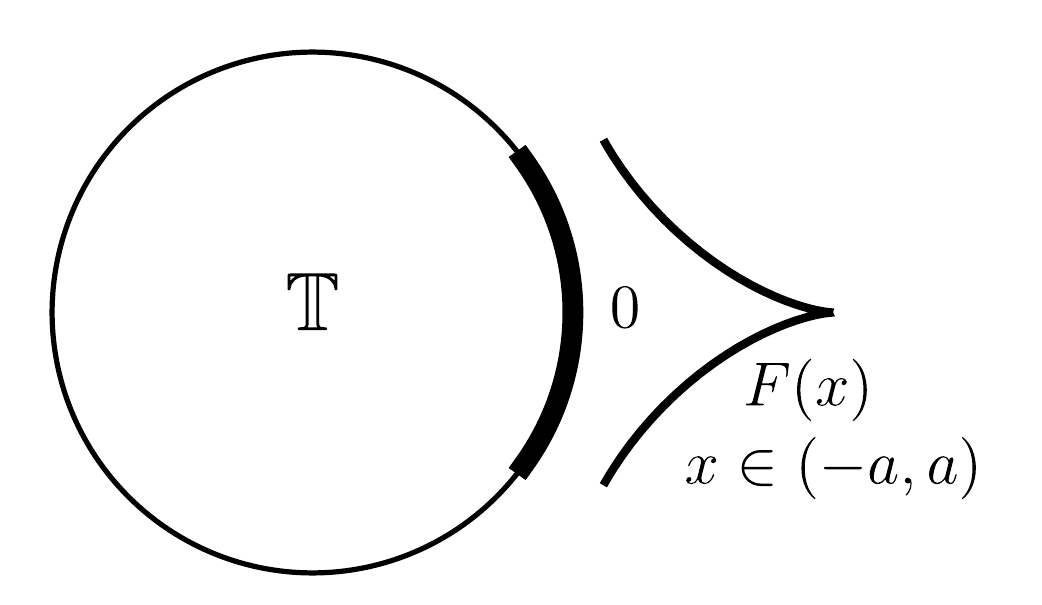}\tabularnewline
\includegraphics[scale=0.5]{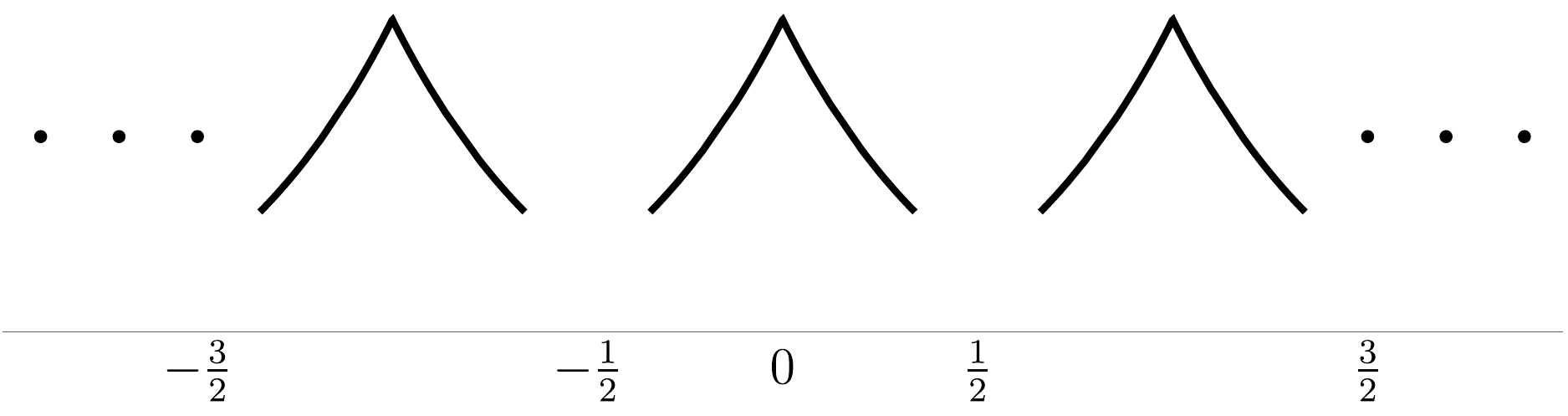}\tabularnewline
\end{tabular}

\begin{minipage}[t]{0.9\columnwidth}%
Note this bijection also applies if a continuous p.d. function $F$
is only defined on a subset $\left(-a,a\right)\subset\mathbb{T}$,
with $0<a<\frac{1}{2}$.%
\end{minipage}

\protect\caption{\label{fig:ext}Functions on $\mathbb{T}=\mathbb{R}/\mathbb{Z}\protect\longleftrightarrow$
$\bigl(\mbox{1-periodic functions on \ensuremath{\mathbb{R}} }\bigr)$. }
\end{figure}

\begin{figure}
\includegraphics[width=0.7\textwidth]{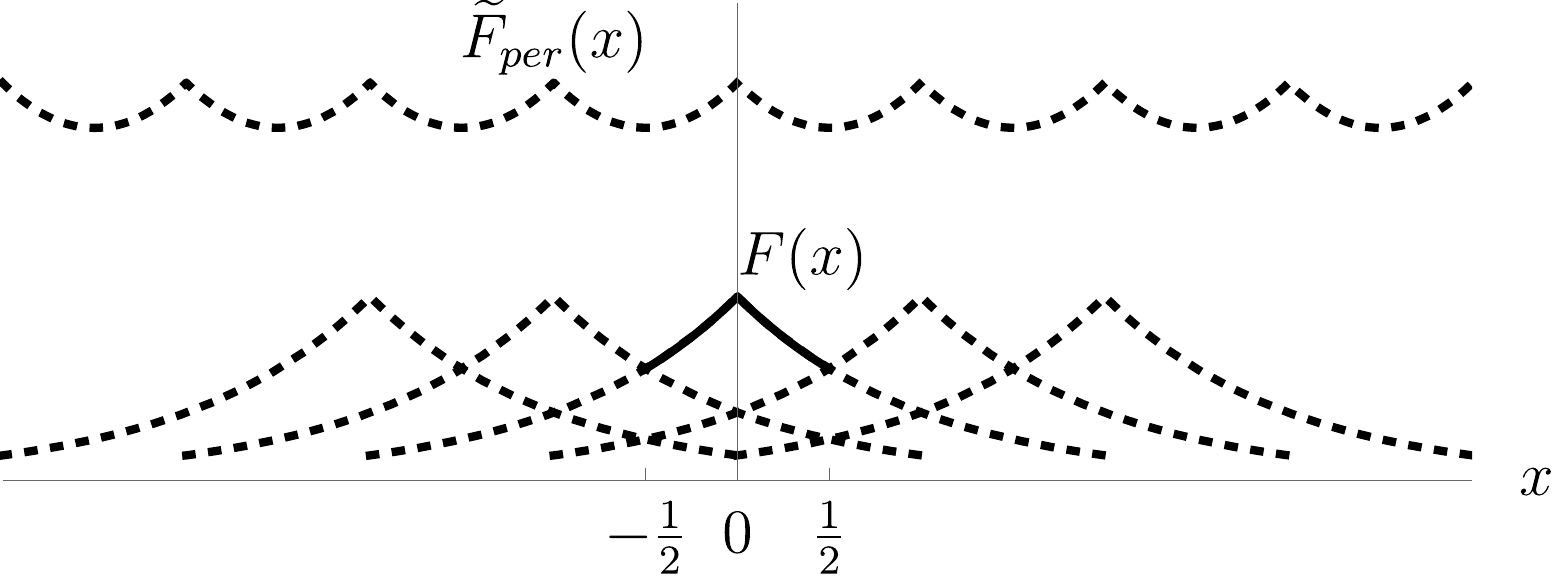}

\protect\caption{\label{fig:pext}$\widetilde{F}_{per}\left(x\right)>F\left(x\right)$
on $\left[-\frac{1}{2},\frac{1}{2}\right]$}
\end{figure}

\begin{example}
\label{exa:expT1}Fix $a\in\left(0,\frac{1}{2}\right)$. Let $F\left(x\right)=e^{-\left|x\right|}$,
$x\in\left(-a,a\right)$. Note that 
\begin{eqnarray}
\left(e^{-\left|\cdot\right|}\chi_{\left[-\frac{1}{2},\frac{1}{2}\right]}\right)^{\wedge}\left(\lambda\right) & = & \int_{\mathbb{R}}\frac{\sin\pi\left(\lambda-t\right)}{\pi\left(\lambda-t\right)}\cdot\frac{2}{1+4\pi^{2}t^{2}}dt\nonumber \\
 & = & \frac{2\left(2\pi\lambda\sin(\pi\lambda)-\cos(\pi\lambda)+\sqrt{e}\right)}{\sqrt{e}\left(1+4\pi^{2}\lambda^{2}\right)},\label{eq:FperW}
\end{eqnarray}
then, using Shannon's interpolation formula, we get 
\begin{eqnarray}
\widetilde{F}_{cir}\left(x\right) & := & \sum_{n\in\mathbb{Z}}e^{-\left|x+n\right|}\chi_{\left[-\frac{1}{2},\frac{1}{2}\right]}\left(x+n\right)=\sum_{n\in\mathbb{Z}}\widehat{W}\left(n\right)e^{i2\pi nx};\;\mbox{where}\label{eq:Fper2}
\end{eqnarray}
\begin{eqnarray}
\widehat{W}\left(n\right) & :=\left(e^{-\left|\cdot\right|}\chi_{\left[-\frac{1}{2},\frac{1}{2}\right]}\right)^{\wedge}\left(n\right)\overset{\text{\ensuremath{\left(\ref{eq:FperW}\right)}}}{=} & \frac{2\left(1-e^{-1/2}\left(-1\right)^{n}\right)}{1+4\pi^{2}n^{2}},\quad\forall n\in\mathbb{Z}.\label{eq:Fper3}
\end{eqnarray}
See Figures \ref{fig:circle3}-\ref{fig:circle4}. 

It follows from (\ref{eq:Fper1})-(\ref{eq:poisson}), that 
\begin{eqnarray*}
\sum_{n\in\mathbb{Z}}\widehat{W}\left(n\right) & = & \sum_{n\in\mathbb{Z}}e^{-\left|n\right|}-e^{-1/2}\sum_{n\in\mathbb{Z}}e^{-\left|n-\frac{1}{2}\right|}=1,\;\mbox{and}\\
\widehat{W}\left(n\right) & > & 0,\quad\forall n\in\mathbb{Z}.
\end{eqnarray*}
Therefore, the weights $\{\widehat{W}\left(n\right)\}$ in (\ref{eq:Fper3})
determines a probability measure on $\mathbb{Z}$ (using the normalization
$F\left(0\right)=1$.) By Bochner's theorem, $\widetilde{F}_{cir}$
in (\ref{eq:Fper2}) is a continuous p.d. extension of $F$, thus
a solution to the $\mathbb{T}$-problem. 
\end{example}
\begin{figure}
\includegraphics[width=0.7\textwidth]{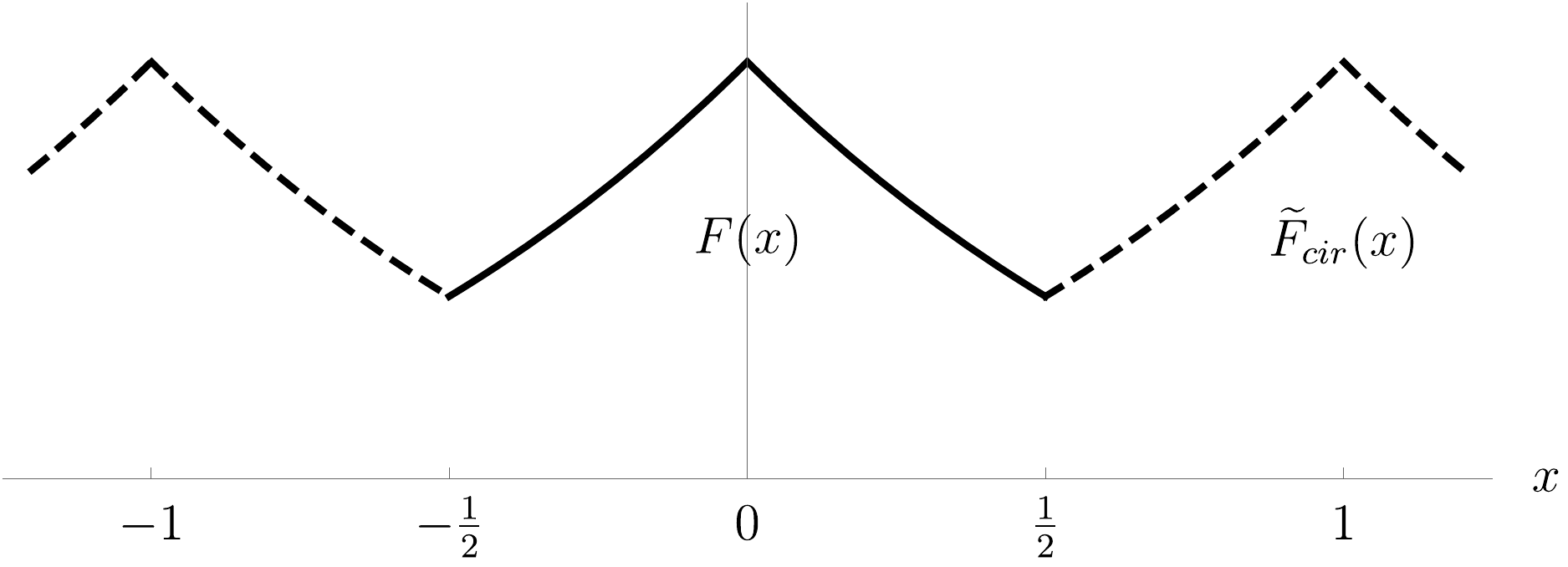}

\protect\caption{\label{fig:circle3}A p.d. extension from $F$ to $\widetilde{F}_{cir}$
defined on $\mathbb{T}$.}
\end{figure}

\begin{figure}
\includegraphics[width=0.7\textwidth]{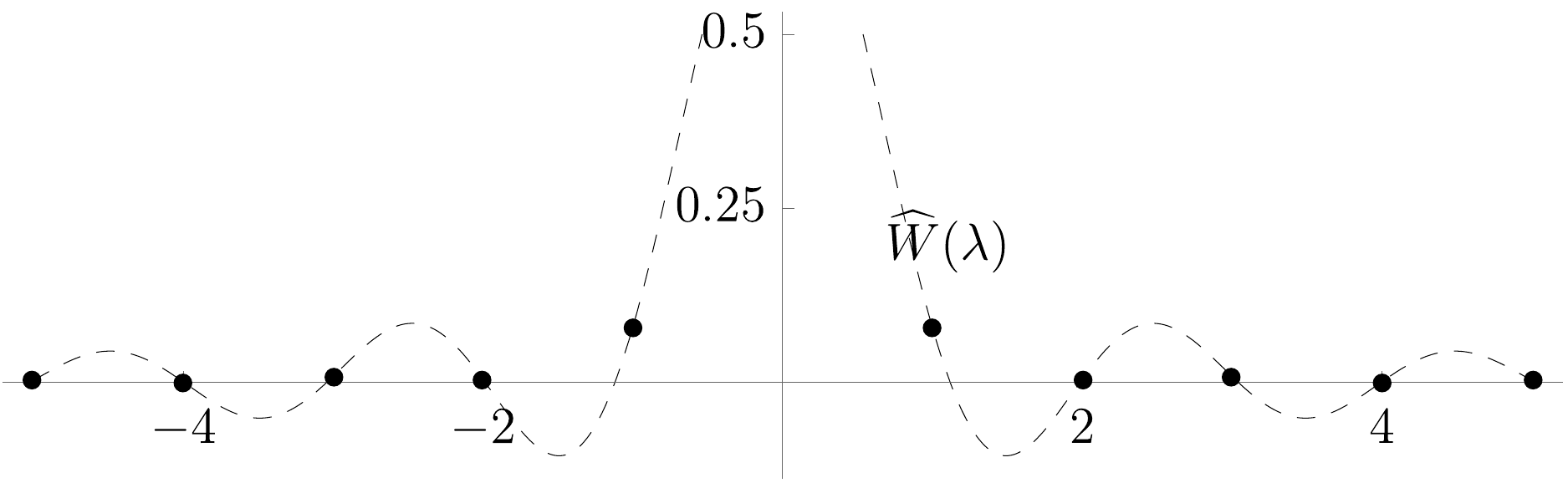}

Set $\widehat{W}\left(\lambda\right):=\left(e^{-\left|\cdot\right|}\chi_{\left(-1/2,1/2\right)}\right)^{\wedge}\left(\lambda\right)$,
then $\{\widehat{W}\left(n\right)\}_{n\in\mathbb{Z}}$ is a probability
distribution on $\mathbb{Z}$. \index{distribution!probability-}

\protect\caption{\label{fig:circle4}A periodic extension of $e^{-\left|\cdot\right|}\chi_{\left[-1/2,1/2\right]}$
on $\mathbb{R}$ by Shannon's interpolation. }
\end{figure}

\begin{example}
\label{exa:absT}Fix $0<a<\frac{1}{2}$, let $F\left(x\right)=1-\left|x\right|$,
$x\in\left(-a,a\right)$, a locally defined p.d. function. This is
another interesting example that we will study in Chapter  \ref{chap:types}.

Using the method from Example \ref{exa:expT}, we have 
\[
\widehat{W}\left(\lambda\right):=\left(\left(1-\left|\cdot\right|\right)\chi_{\left[-\frac{1}{2},\frac{1}{2}\right]}\right)^{\wedge}\left(\lambda\right)=\frac{\pi\lambda\sin(\pi\lambda)-\cos(\pi\lambda)+1}{2\pi^{2}\lambda^{2}}
\]
and so 
\[
\widetilde{F}_{cir}\left(x\right):=\sum_{n\in\mathbb{Z}}\left(1-\left|x+n\right|\right)\chi_{\left[-\frac{1}{2},\frac{1}{2}\right]}\left(x+n\right)=\sum_{n\in\mathbb{Z}}\widehat{W}\left(n\right)e^{i2\pi nx},
\]
where
\[
\widehat{W}\left(n\right)=\frac{1-\left(-1\right)^{n}}{2\pi^{2}\lambda^{2}}\geq0,\;\forall n\in\mathbb{Z},\;\mbox{and }\sum_{n\in\mathbb{Z}}\widehat{W}\left(n\right)=1.
\]
Therefore, $\widetilde{F}_{cir}$ is a periodic p.d. extension of
$F$.\end{example}
\begin{remark}
The method used in Examples \ref{exa:expT1}-\ref{exa:absT} does
not hold in general. For example, a periodization of the restricted
Gaussian distribution, $e^{-x^{2}/2}\chi_{\left[-\frac{1}{2},\frac{1}{2}\right]}$,
is not positive definite on $\mathbb{R}$. The reason is that the
sampling 
\[
\left\{ \left(e^{-\left(\cdot\right)^{2}/2}\chi_{\left[-\frac{1}{2},\frac{1}{2}\right]}\right)^{\wedge}\left(n\right)\right\} _{n\in\mathbb{Z}}
\]
does NOT give a positive measure on $\mathbb{Z}$. \index{distribution!Gaussian-}
\end{remark}

\section{\label{sub:exp(-|x|)}Example: $e^{-\left|x\right|}$ in $\left(-a,a\right)$,
extensions to $\mathbb{R}$}

In this section, we consider the following partially defined p.d.
function: 
\begin{equation}
F\left(x\right)=e^{-\left|x\right|},\quad x\in\left(-a,a\right)\label{eq:exp-1}
\end{equation}
where $0<a<\infty$, fixed. It is the restriction of
\begin{equation}
\begin{split}e^{-\left|x\right|} & =\int_{\mathbb{R}}e^{i\lambda x}d\mu\left(\lambda\right),\quad x\in\mathbb{R},\;\mbox{with}\\
d\mu & =\frac{d\lambda}{\pi\left(1+\lambda^{2}\right)}\left(=\mbox{prob. measure}\right),
\end{split}
\label{eq:ep2}
\end{equation}
thus, $\mu\in Ext\left(F\right)$. (In particular, $Ext\left(F\right)\neq\emptyset$.)
\begin{remark}
This special case is enlightening in connection with deficiency indices
considerations. In general, there is a host of other interesting 1D
examples (Chapters \ref{chap:types}-\ref{chap:Ext1}), some have
indices $\left(1,1\right)$ and others $\left(0,0\right)$. In case
of indices $\left(1,1\right)$, the convex set $Ext\left(F\right)$
is parameterized by $\mathbb{T}$; while, in case of $\left(0,0\right)$,
$Ext(F)$ is a singleton. Further spectral theoretic results are given
in Chapter \ref{chap:spbd}, and we obtain a spectral classification
of all Hermitian operators with dense domain in a separable Hilbert
space, having deficiency indices $\left(1,1\right)$. 
\end{remark}
\index{deficiency indices}

\index{positive definite}

\index{measure!probability}

\index{RKHS}\index{convex}

Now define $D^{\left(F\right)}$ (the canonical skew-Hermitian operator),
by $D^{\left(F\right)}F_{\varphi}=F_{\varphi'}$, with $dom(D^{\left(F\right)})=\left\{ F_{\varphi}:\varphi\in C_{c}^{\infty}\left(0,a\right)\right\} $.
We show below that $D^{\left(F\right)}$ has deficiency indices $\left(1,1\right)$. 

Let $\mu\in Ext\left(F\right)$ be as in (\ref{eq:ep2}). By Corollary
\ref{cor:lcg-isom}, there is an isometry $\mathscr{H}_{F}\hookrightarrow L^{2}\left(\mu\right)$,
determined by 
\[
\mathscr{H}_{F}\ni F_{\varphi}\longmapsto\hat{\varphi}\in L^{2}\left(\mu\right),
\]
where $\mathscr{H}_{F}$ is the corresponding RKHS. Indeed, we have
\begin{eqnarray*}
\left\Vert F_{\varphi}\right\Vert _{\mathscr{H}_{F}}^{2} & = & \int_{0}^{a}\int_{0}^{a}\overline{\varphi\left(x\right)}\varphi\left(y\right)F\left(x-y\right)dxdy\\
 & = & \int_{0}^{a}\int_{0}^{a}\overline{\varphi\left(x\right)}\varphi\left(y\right)\left(\int_{\mathbb{R}}e^{i\lambda\left(x-y\right)}d\mu\right)dxdy\\
 & \overset{\text{Fubini}}{=} & \int_{\mathbb{R}}\left|\widehat{\varphi}\left(\lambda\right)\right|^{2}d\mu\left(\lambda\right)=\left\Vert \widehat{\varphi}\right\Vert _{L^{2}\left(\mu\right)}^{2},\quad\forall F_{\varphi}\in dom(D^{\left(F\right)}).
\end{eqnarray*}

\begin{lemma}
Fix $0<a<\infty$, and let $F$ be a continuous p.d. function on $\left(-a,a\right)$;
assume $F\left(0\right)=1$. Pick $\mu\in Ext\left(F\right)$, i.e.,
a probability measure on $\mathbb{R}$ s.t. 
\begin{equation}
F\left(x\right)=\widehat{d\mu}\left(x\right),\quad\forall x\in\left(-a,a\right).\label{eq:exp-7}
\end{equation}
(We know that $Ext\left(F\right)\neq\emptyset$ from Section \ref{sec:Prelim}.)
Then, $D^{\left(F\right)}$ has deficiency indices $\left(1,1\right)$
if and only if
\[
\left(e^{\pm\left(\cdot\right)}\chi_{\left(0,a\right)}\right)^{\wedge}\left(\cdot\right)\in L^{2}\left(\mu\right).
\]
 
\begin{equation}
\int_{-\infty}^{\infty}\frac{e^{2a}+1-2e^{a}\cos\left(\lambda a\right)}{1+\lambda^{2}}d\mu\left(\lambda\right)<\infty.\label{eq:exp-10}
\end{equation}
\end{lemma}
\begin{svmultproof2}
Recall the defect vectors of $D^{\left(F\right)}$ are multiples of
$e^{\pm x}$, $x\in\left(0,a\right)$. The assertion follows from
the isometry $\mathscr{H}_{F}\hookrightarrow L^{2}\left(\mu\right)$. \end{svmultproof2}

\begin{corollary}
Let $F$ be as in (\ref{eq:exp-1}), then the associated $D^{\left(F\right)}$
has deficiency indices $\left(1,1\right)$. In particular, the defect
vectors $\xi_{+}=e^{-x}$, $\xi_{-}=e^{a-x}$, $x\in\left(0,a\right)$,
satisfy $\left\Vert \xi_{-}\right\Vert _{\mathscr{H}_{F}}=\left\Vert \xi_{+}\right\Vert _{\mathscr{H}_{F}}=1$.\end{corollary}
\begin{lemma}
We have:
\begin{equation}
\left|\left(\chi_{\left(0,a\right)}\left(x\right)e^{x}\right)^{\wedge}\left(\lambda\right)\right|^{2}=\frac{e^{2a}+1-2e^{a}\cos\left(\lambda a\right)}{1+\lambda^{2}}\;\left(\in L^{2}\left(\mathbb{R},d\lambda\right)\right),\label{eq:exp-3}
\end{equation}
where $d\lambda$ denotes the Lebesgue measure on $\mathbb{R}$.\end{lemma}
\begin{svmultproof2}
Direct calculation shows that
\[
\left(\chi_{\left(0,a\right)}\left(x\right)e^{x}\right)^{\wedge}\left(\lambda\right)=\int_{0}^{a}e^{-i\lambda x}e^{x}dx=\frac{e^{\left(1-i\lambda\right)a}-1}{1-i\lambda};
\]
which is the assertion in (\ref{eq:exp-3}).
\end{svmultproof2}

Now return to the RKHS $\mathscr{H}_{F}$, defined from the function
$F$ in (\ref{eq:exp-1}). Recall that $span\left\{ F_{\varphi}\:\big|\:\varphi\in C_{c}^{\infty}\left(0,a\right)\right\} $
is a dense subspace in $\mathscr{H}_{F}$ (Lemma \ref{lem:dense}),
where
\begin{equation}
F_{\varphi}\left(x\right)=\int_{0}^{a}\varphi\left(y\right)e^{-\left|x-y\right|}dy\label{eq:exp-4}
\end{equation}

\begin{lemma}
We have 
\begin{eqnarray}
\left\Vert F_{\varphi}\right\Vert _{\mathscr{H}_{F}}^{2} & = & \int_{0}^{a}\int_{0}^{a}\overline{\varphi\left(y\right)}\varphi\left(x\right)e^{-\left|x-y\right|}dxdy\nonumber \\
 & = & \int_{-\infty}^{\infty}\left|\widehat{\varphi}\left(\lambda\right)\right|^{2}\frac{d\lambda}{\pi\left(1+\lambda^{2}\right)}\label{eq:exp-5}
\end{eqnarray}
\end{lemma}
\begin{svmultproof2}
Immediate.
\end{svmultproof2}

\index{measure!probability}

\index{absolutely continuous}
\begin{lemma}[continuation of Example \ref{exa:mexp}]
\label{lem:exp-1} Fix $a$, $0<a<\infty$, and let $F\left(\cdot\right):=e^{-\left|\cdot\right|}\big|_{\left(-a,a\right)}$
as in (\ref{eq:exp-1}). For all $0\leq x_{0}\leq a$, let 
\begin{equation}
F_{x_{0}}\left(x\right):=F\left(x-x_{0}\right)\Big|_{\left(0,a\right)}\left(\in C\left(0,a\right)\right).\label{eq:exp-11}
\end{equation}
With 
\begin{align}
DEF^{+} & =\left\{ \xi:\bigl(D^{\left(F\right)}\bigr)^{*}\xi=\xi\right\} =span\left\{ \xi_{+}\left(x\right):=e^{-x}\Big|_{\left(0,a\right)}\right\} \label{eq:exp-12}\\
DEF^{-} & =\left\{ \xi:\bigl(D^{\left(F\right)}\bigr)^{*}\xi=-\xi\right\} =span\left\{ \xi_{-}\left(x\right):=e^{-a}e^{+x}\Big|_{\left(0,a\right)}\right\} \label{eq:exp-13}
\end{align}
we get 
\begin{equation}
\left\Vert \xi_{+}\right\Vert _{\mathscr{H}_{F}}^{2}=\left\Vert \xi_{-}\right\Vert _{\mathscr{H}_{F}}^{2}=1.\label{eq:exp-14}
\end{equation}
\end{lemma}
\begin{svmultproof2}
Note that $x=0$ and $x=a$ are the endpoints in the open interval
$\left(0,a\right)$: 
\begin{align*}
\xi_{+}\left(x\right) & =e^{-x}\Big|_{\left(0,a\right)}=F_{0}\left(x\right)\\
\xi_{-}\left(x\right) & =e^{-a}e^{x}\Big|_{\left(0,a\right)}=F_{a}\left(x\right).
\end{align*}
Let $\psi_{n}\in C_{c}^{\infty}\left(0,a\right)$ be an approximate
identity\index{approximate identity}, such that 
\begin{enumerate}
\item $\psi_{n}\geq0$, $\int\psi_{n}=1$; 
\item $\psi_{n}\rightarrow\delta_{a}$, as $n\rightarrow\infty$. 
\end{enumerate}

\begin{flushleft}
Then 
\[
\xi_{-}\left(x\right)=F_{a}\left(x\right)=\lim_{n\rightarrow\infty}\int_{0}^{a}\psi_{n}\left(y\right)F_{y}\left(x\right)dy.
\]
This shows that $\xi_{-}\in\mathscr{H}_{F}$. Also, 
\begin{eqnarray*}
\left\Vert \xi_{-}\right\Vert _{\mathscr{H}_{F}}^{2} & = & \left\Vert F_{a}\right\Vert _{\mathscr{H}_{F}}^{2}\\
 & = & \lim_{n\rightarrow\infty}\frac{1}{2\pi}\int_{-\infty}^{\infty}\left|\widehat{\psi_{n}}\left(y\right)\right|^{2}\widehat{F}\left(y\right)dy\\
 & = & \frac{1}{2\pi}\int_{-\infty}^{\infty}\widehat{F}\left(y\right)dy=1.
\end{eqnarray*}
Similarly, if instead, $\psi_{n}\rightarrow\delta_{0}$, then 
\[
\xi_{+}\left(x\right)=F_{0}\left(x\right)=\lim_{n\rightarrow\infty}\int_{0}^{a}\psi_{n}\left(y\right)F_{y}\left(x\right)dy
\]
and $\left\Vert \xi_{+}\right\Vert _{\mathscr{H}_{F}}^{2}=1$. 
\par\end{flushleft}

\end{svmultproof2}

\begin{lemma}
Let $F$ be as in (\ref{eq:exp-1}). We have the following for its
Fourier transform:\index{transform!Fourier-}
\[
\widehat{F}\left(y\right)=\frac{2-2e^{-a}\left(\cos\left(ay\right)-y\sin\left(ay\right)\right)}{1+y^{2}}.
\]
\end{lemma}
\begin{svmultproof2}
Let $y\in\mathbb{R}$, then 
\begin{eqnarray*}
\widehat{F}\left(y\right) & = & \int_{-a}^{a}e^{iyx}e^{-\left|x\right|}dx\\
 & = & \int_{-a}^{0}e^{iyx}e^{x}dx+\int_{0}^{a}e^{iyx}e^{-x}dx\\
 & = & \frac{1-e^{-a(1+iy)}}{1+iy}+\frac{e^{ia(y+i)}-1}{-1+iy}\\
 & = & \frac{2-2e^{-a}\left(\cos\left(ay\right)-y\sin\left(ay\right)\right)}{1+y^{2}}
\end{eqnarray*}
which is the assertion.\end{svmultproof2}

\begin{remark}
If $\left(F,\Omega\right)$ is such that $D^{\left(F\right)}$ has
deficiency indices $\left(1,1\right)$, then by von Neumann's theory
\cite{DS88b}, the family of selfadjoint extensions is characterized
by \index{von Neumann, John}\index{selfadjoint extension} 
\begin{align*}
dom\left(A_{\theta}^{\left(F\right)}\right) & =\left\{ F_{\psi}+c\left(\xi_{+}+e^{i\theta}\xi_{-}\right):\psi\in C_{c}^{\infty}\left(0,a\right),c\in\mathbb{C}\right\} \\
A_{\theta}^{\left(F\right)}: & F_{\psi}+c\left(\xi_{+}+e^{i\theta}\xi_{-}\right)\mapsto F_{i\psi'}+c\,i\left(\xi_{+}-e^{i\theta}\xi_{-}\right),\mbox{ where }i=\sqrt{-1}.
\end{align*}
\end{remark}
\begin{proposition}
\label{prop:exp}Let $F\left(x\right)=e^{-\left|x\right|}$, $x\in\left(-a,a\right)$,
be the locally defined p.d. function (see (\ref{eq:exp-1})), and
$\mathscr{H}_{F}$ be the associated RKHS. Consider $e_{\lambda}\left(x\right)=\chi_{\left(0,a\right)}\left(x\right)e^{i\lambda x}$
with $\lambda\in\mathbb{C}$ fixed. Then the function $e_{\lambda}\in\mathscr{H}_{F}$,
for all $\lambda\in\mathbb{C}$. \index{complex exponential}\end{proposition}
\begin{svmultproof2}
For simplicity, set $a=1$. Fix $\lambda\in\mathbb{C}$, and consider
two cases:

Case 1. $\Im\left\{ \lambda\right\} =0$, i.e. , $\lambda\in\mathbb{R}$. 

Recall the isometry $T:\mathscr{H}_{F}\rightarrow L^{2}\left(\mu\right)$,
with $d\mu$ as in (\ref{eq:ep2}); and by Corollary \ref{cor:Tadj},
the adjoint operator is given by 
\[
L^{2}\left(\mu\right)\ni f\xrightarrow{\quad T^{*}\;}\chi_{\left[0,1\right]}\left(fd\mu\right)^{\vee}\in\mathscr{H}_{F}.
\]
We shall show that there exists $f\in L^{2}\left(\mu\right)$ s.t.
$T^{*}f=\chi_{\left[0,1\right]}e_{\lambda}$. 

Let 
\begin{eqnarray}
\varphi\left(x\right) & = & \left(\chi_{\left[-1,1\right]}*\chi_{\left[0,1\right]}\right)\left(x\right)\nonumber \\
 & = & \left|\left[x-1,x\right]\cap\left[-1,1\right]\right|,\quad x\in\mathbb{R},\label{eq:convphi}
\end{eqnarray}
where $\left|\cdots\right|$ denotes Lebesgue measure. See Figure
\ref{fig:pyr}. Note that 
\begin{eqnarray}
\widehat{\varphi}\left(t\right) & = & \frac{2\sin(t)}{t}\cdot\frac{2\sin\left(t/2\right)}{t}e^{-it/2}\\
\left|\widehat{\varphi}\left(t\right)\right|^{2} & = & \frac{16\sin^{2}\left(t\right)\sin^{2}\left(t/2\right)}{t^{4}},\quad t\in\mathbb{R}.
\end{eqnarray}
Now, set $e_{\lambda}\left(x\right):=e^{i\lambda x}$, and 
\begin{eqnarray}
\psi\left(x\right) & = & e_{\lambda}\left(x\right)\varphi\left(x\right),\;\mbox{and}\\
f\left(t\right) & = & \frac{1}{2}\widehat{\psi}\left(t\right)\left(1+t^{2}\right),\quad t\in\mathbb{R}.\label{eq:laf}
\end{eqnarray}
Then, we have: 
\begin{eqnarray*}
\left(T^{*}f\right)\left(x\right) & = & \left(\chi_{\left[0,1\right]}\left(fd\mu\right)^{\vee}\right)\left(x\right)\\
 & = & \chi_{\left[0,1\right]}\left(x\right)\int_{\mathbb{R}}e^{itx}\frac{1}{2}\hat{\psi}\left(t\right)\left(1+t^{2}\right)\frac{dt}{\pi\left(1+t^{2}\right)}\\
 & = & \chi_{\left[0,1\right]}\left(x\right)\frac{1}{2\pi}\int_{\mathbb{R}}e^{itx}\widehat{\psi}\left(t\right)dt\\
 & = & \chi_{\left[0,1\right]}\left(x\right)\psi\left(x\right)=\chi_{\left[0,1\right]}\left(x\right)e_{\lambda}\left(x\right);\;\mbox{and}
\end{eqnarray*}
\begin{eqnarray*}
\int_{\mathbb{R}}\left|f\right|^{2}d\mu & = & \frac{1}{4}\int_{\mathbb{R}}\left|\hat{\psi}\left(t\right)\right|^{2}\left(1+t^{2}\right)^{2}\frac{1}{\pi\left(1+t^{2}\right)}dt\\
 & = & \frac{1}{4\pi}\int_{\mathbb{R}}\left|\hat{\psi}\left(t\right)\right|^{2}\left(1+t^{2}\right)dt\\
 & = & \frac{4}{\pi^{2}}\int_{\mathbb{R}}\frac{\sin^{2}\left(\left(t-\lambda\right)/2\right)\sin^{2}\left(t-\lambda\right)}{\left(t-\lambda\right)^{4}}\left(1+t^{2}\right)dt<\infty.
\end{eqnarray*}
Thus, the function $f$ in (\ref{eq:laf}) is in $L^{2}\left(\mu\right)$,
and satisfies $T^{*}f=\chi_{\left[0,1\right]}e_{\lambda}$. This is
the desired conclusion.

Case 2. $\Im\left\{ \lambda\right\} \neq0$. 

Using general theory \cite{DS88b}, it is enough to prove that $e_{\lambda}\left(x\right)$
is in $\mathscr{H}_{F}$ for $\lambda=\pm i$. So consider the following:
\begin{equation}
\begin{split}\xi_{+}\left(x\right) & =e^{-x}\chi_{\left(0,1\right)}\left(x\right)\\
\xi_{-}\left(x\right) & =e^{x-1}\chi_{\left(0,1\right)}\left(x\right).
\end{split}
\label{eq:l5}
\end{equation}
Let $S$ be the isometry in Corollary \ref{cor:muHF}, then 
\begin{equation}
S^{*}\underset{\in\mathscr{H}_{F}}{\underbrace{\left(\nu*F\right)}}=\nu\in\mathfrak{M}_{2}\left(F\right)\label{eq:l6}
\end{equation}
for all signed measures $\nu\in\mathfrak{M}_{2}\left(F\right)$. Fixing
$\mu$ positive s.t. $\widehat{d\mu}\left(x\right)=F\left(x\right)$
for all $x\in\left[-1,1\right]$, we get 
\[
\nu*F\in\mathscr{H}_{F}\Longleftrightarrow\int_{\mathbb{R}}\left|\widehat{\nu}\left(t\right)\right|^{2}d\mu\left(t\right)<\infty,
\]
and 
\begin{equation}
\left\Vert \nu*F\right\Vert _{\mathscr{H}_{F}}^{2}=\int_{\mathbb{R}}\left|\widehat{\nu}\left(t\right)\right|^{2}d\mu\left(t\right).\label{eq:la7}
\end{equation}
Here, $\left(\nu*F\right)\left(x\right)=\int_{0}^{1}F\left(x-y\right)d\nu\left(y\right)$,
$\forall x\in\left[0,1\right]$. 

In Example \ref{exa:mexp}, we proved that
\begin{equation}
\begin{split}\left(\delta_{0}*F\right)\left(x\right) & =\xi_{+}\left(x\right)\\
\left(\delta_{1}*F\right)\left(x\right) & =\xi_{-}\left(x\right)
\end{split}
\label{eq:la8}
\end{equation}
Using (\ref{eq:la7}), we then get 
\[
\left\Vert \xi_{+}\right\Vert _{\mathscr{H}_{F}}^{2}=\Vert\widehat{\delta_{0}}\Vert_{L^{2}\left(\mathbb{R},\mu\right)}^{2}=1
\]
and 
\[
\left\Vert \xi_{-}\right\Vert _{\mathscr{H}_{F}}^{2}=\Vert\widehat{\delta_{1}}\Vert_{L^{2}\left(\mathbb{R},\mu\right)}^{2}=1;
\]
where of course $\widehat{\delta_{0}}\equiv1$, $\forall t\in\mathbb{R}$,
and $\widehat{\delta_{1}}=e^{it}$. 
\end{svmultproof2}

\begin{figure}
\includegraphics[width=0.6\textwidth]{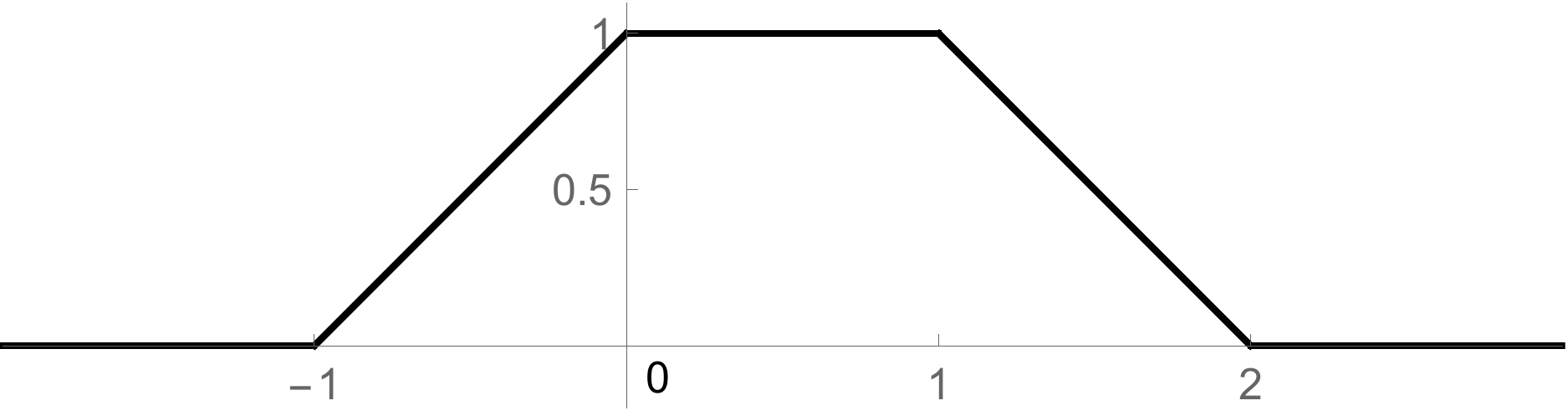}

\protect\caption{\label{fig:pyr}Convolution product $\varphi=\chi_{\left[0,1\right]}*\chi_{\left[-1,1\right]}$
from (\ref{eq:convphi}).}

\end{figure}

\index{distribution!-solution}

\index{distribution!-derivation}
\begin{remark}
In Tables \ref{tab:F1-F6}-\ref{tab:Table-3}, we consider six examples
of locally defined continuous p.d. functions, all come from restrictions
of p.d. functions on $\mathbb{R}$. Except for the trivial case $F_{6}$,
for each $F_{j}$, $j=1,\dots,5$, we need to decide whether $e_{\lambda}\left(x\right)=e^{i\lambda x}$,
$\lambda\in\mathbb{C}$, are in the corresponding RKHS $\mathscr{H}_{F_{j}}$. 

For illustration, it suffices to consider the case $\lambda=0$, i.e.,
whether the constant functions $\chi_{\left[0,a\right]}$ are in $\mathscr{H}_{F_{j}}$.
Equivalently, if there exists $f\in L^{2}(\mu)$, s.t. 
\begin{eqnarray}
\varphi & = & T^{*}f=\chi_{\left[0,a\right]}\left(fd\mu\right)^{\vee},\;\mbox{s.t.}\label{eq:te1}\\
\varphi & \equiv & 1\;\mbox{on }\left[0,a\right].\label{eq:te2}
\end{eqnarray}
Here, $d\mu$ denotes one of the measures introduced above. Compare
with Proposition \ref{prop:exp}.\index{locally defined}

Note that in all cases, $d\mu$ is absolutely continuous w.r.t. the
Lebesgue measure, i.e., $d\mu=m\:dx$, where $m$ is a Radon-Nikodym
derivative. Thus, a possible solution $f$ to (\ref{eq:te1})-(\ref{eq:te2})
must be given by\index{derivative!Radon-Nikodym-} 
\begin{equation}
\widehat{f}=\frac{\widehat{\varphi}}{m},\quad\mbox{s.t.}\quad\int_{\mathbb{R}}\frac{\left|\widehat{\varphi}\right|^{2}\left(t\right)}{m\left(t\right)}dt<\infty.\label{eq:te3}
\end{equation}
By the splitting, $\varphi=\varphi\chi_{\left[0,a\right]}+\varphi_{2}$,
and using (\ref{eq:te2}), we then get 
\begin{equation}
\widehat{\varphi}\left(t\right)=\mbox{sinc}\left(t\right)+\widehat{\varphi}_{2}\left(t\right),\label{eq:te4}
\end{equation}
where, in (\ref{eq:te4}), we use that $\varphi_{2}$ is supported
in $\mathbb{R}\backslash\left(0,a\right)$. In the applications mentioned
above, the $\mbox{sinc}$ function in (\ref{eq:te4}) will already
account for the divergence of the integral in (\ref{eq:te3}). 

\uline{Conclusion}: One checks that for $F_{1},F_{4}$, and $F_{5}$,
the function $\chi_{\left[0,a\right]}$ is \emph{not} in the corresponding
RKHS. \end{remark}
\begin{example}
$F_{1}\left(x\right)=\frac{1}{1+x^{2}}$, $\left|x\right|<1$, $d\mu_{1}\left(t\right)=\frac{1}{2}e^{-\left|t\right|}dt$.
Then, we have 
\[
\int_{\mathbb{R}}\frac{\left|\widehat{\varphi}\right|^{2}\left(t\right)}{m\left(t\right)}dt\sim\int_{\mathbb{R}}\mbox{sinc}^{2}\left(t\right)e^{\left|t\right|}dt+\mbox{other positive terms}
\]
which is divergent. Thus, $\chi_{\left[0,1\right]}\notin\mathscr{H}_{F_{1}}$. \end{example}
\begin{remark}
In the example above, we cannot have 
\begin{equation}
\left|\widehat{\varphi}\right|^{2}\left(t\right)\sim O(e^{-\left(1+\epsilon\right)\left|t\right|}),\;\mbox{as }\left|t\right|\rightarrow\infty;\label{eq:te5}
\end{equation}
for some $\epsilon>0$. Should (\ref{eq:te5}) hold, then By the Paley-Wiener
theorem, $x\mapsto\varphi\left(x\right)$ would have an analytic continuation
to a strap, $\left|\Im z\right|<\alpha$, $z=x+iy$, $\left|y\right|<\alpha$.
But this contradicts the assumption that $\varphi\left(x\right)\equiv1$,
for all $x\in\left(0,1\right)$; see (\ref{eq:te2}). \index{analytic continuation}
\end{remark}

\motto{If people do not believe that mathematics is simple, it is only because
they do not realize how complicated life is. \\
\hfill\textemdash{} John von Neumann}

\section{\label{sec:logz}Example: A non-extendable p.d. function in a neighborhood
of zero in $G=\mathbb{R}^{2}$}

Let $M$ denote the Riemann surface of the complex $\log z$ function.
$M$ is realized as a covering space for $\mathbb{R}^{2}\backslash\left\{ \left(0,0\right)\right\} $
with an infinite number of sheets indexed by $\mathbb{Z}$, see Figure
\ref{fig:logz}.\index{Riemann surface}\index{covering space}

\begin{figure}
\begin{tabular}{cc}
\includegraphics[width=0.45\textwidth]{logz1} & \includegraphics[width=0.45\textwidth]{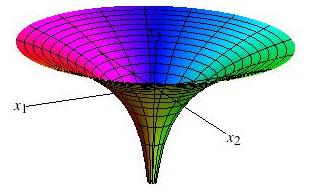}\tabularnewline
\end{tabular}

\protect\caption{\label{fig:logz}$M$ the Riemann surface of $\log z$ as an $\infty$
cover of $\mathbb{R}^{2}\backslash\left\{ \left(0,0\right)\right\} $. }
\end{figure}

The 2D-Lebesgue measure lifts to a unique measure on $M$; hence $L^{2}\left(M\right)$.
Here the two skew-symmetric operators $\frac{\partial}{\partial x_{j}}$,
$j=1,2$, with domain $C_{c}^{\infty}\left(M\right)$ define an Abelian
2-dimensional Lie algebra of densely defined operators in the Hilbert
space $L^{2}\left(M\right)$. 
\begin{proposition}
\label{prop:logz}~
\begin{enumerate}
\item \label{enu:lg1}Each operator $\frac{\partial}{\partial x_{j}}$,
defined on $C_{c}^{\infty}\left(M\right)$, is essentially skew-adjoint
in $L^{2}\left(M\right)$, i.e.,\index{essentially skew-adjoint}
\begin{equation}
-\left(\frac{\partial}{\partial x_{j}}\Big|_{C_{c}^{\infty}\left(M\right)}\right)^{*}=\overline{\frac{\partial}{\partial x_{j}}\Big|_{C_{c}^{\infty}\left(M\right)}},\quad j=1,2;\label{eq:a1}
\end{equation}
where the r.h.s. in (\ref{eq:a1}) denotes operator closure.
\item \label{enu:lg2}The $\left\{ \frac{\partial}{\partial x_{j}}\right\} _{j=1,2}$
Lie algebra with domain $C_{c}^{\infty}\left(M\right)\subset L^{2}\left(M\right)$
does not extend to a unitary representation of $G=\mathbb{R}^{2}$. 
\item The two skew-adjoint operators in (\ref{eq:a1}) are NOT strongly
commuting.
\item \label{enu:lg4}The Laplace operator\index{operator!strongly commuting-}\index{operator!Laplace-}
\begin{equation}
L:=\left(\frac{\partial}{\partial x_{1}}\right)^{2}+\left(\frac{\partial}{\partial x_{2}}\right)^{2},\quad dom\left(L\right)=C_{c}^{\infty}\left(M\right)\label{eq:a2}
\end{equation}
has deficiency indices $\left(\infty,\infty\right)$. 
\end{enumerate}
\end{proposition}
\begin{svmultproof2}
For $j=1,2$, the operators $\frac{\partial}{\partial x_{j}}\big|_{C_{c}^{\infty}\left(M\right)}$
generate unitary one-parameter groups $U_{j}\left(t\right)$ on $L^{2}\left(M\right)$,
lifted from the coordinate translations:\index{unitary one-parameter group}\index{representation!translation-}
\begin{eqnarray}
\left(x_{1},x_{2}\right) & \longmapsto & \left(x_{1}+s,x_{2}\right),\quad x_{2}\neq0\label{eq:a3}\\
\left(x_{1},x_{2}\right) & \longmapsto & \left(x_{1},x_{2}+t\right),\quad x_{1}\neq0\label{eq:a4}
\end{eqnarray}
It is immediate that the respective infinitesimal generators are the
closed operators $\overline{\frac{\partial}{\partial x_{j}}\big|_{C_{c}^{\infty}\left(M\right)}}$.
Part (\ref{enu:lg1}) follows from this.\index{infinitesimal generator}

Part (\ref{enu:lg2})-(\ref{enu:lg4}). Suppose $\varphi\in C_{c}^{\infty}\left(M\right)$
is supported over some open set in $\mathbb{R}^{2}\backslash\left\{ \left(0,0\right)\right\} $,
for example, $\left(x_{1}-2\right)^{2}+x_{2}^{2}<1$. If $1<s,t<2$,
then the functions 
\begin{equation}
U_{1}\left(s\right)U_{2}\left(t\right)\varphi\quad\mbox{vs.}\quad U_{2}\left(t\right)U_{1}\left(s\right)\varphi\label{eq:a5}
\end{equation}
are supported on two opposite levels in the covering space $M$; see
Figure \ref{fig:trans}. (One is over the other; the two are on ``different
floors of the parking garage.'') Hence the unitary groups $U_{j}\left(t\right)$
do not commute. It follows from Nelson's theorem \cite{Nel59} that
$L$ in (\ref{eq:a2}) is not essentially selfadjoint, equivalently,
the two skew-adjoint operators are not strongly commuting. \index{Hermitian}

Since $L\leq0$ (in the sense of Hermitian operators), so it has equal
deficiency indices. In fact, $L$ has indices $\left(\infty,\infty\right)$;
see \cite{JT14,Tia11}.

This completes the proof of the proposition.
\end{svmultproof2}

\begin{figure}[H]
\includegraphics[width=0.6\textwidth]{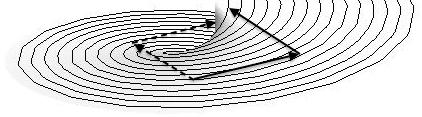}

\protect\caption{\label{fig:trans}Translation of $\varphi$ to different sheets.}
\end{figure}

\begin{remark}
The two unitary groups from Proposition \ref{prop:logz} define a
local representation of $G=\mathbb{R}^{2}$ on $L^{2}\left(M\right)$,
but not a global one. That is, the skew-adjoint operators $\overline{\frac{\partial}{\partial x_{j}}\big|_{C_{c}^{\infty}\left(M\right)}}$
in (\ref{eq:a1}) are \emph{not} strongly commuting, and so 
\begin{equation}
\rho:\left(s,t\right)\longmapsto U_{1}\left(s\right)U_{2}\left(t\right)\label{eq:lg0}
\end{equation}
is \emph{not} a unitary representation of $\mathbb{R}^{2}$ on $L^{2}\left(M\right)$.
This is different from the 1D examples.

The harmonic analysis of the Riemann surface $M$ of $\log z$ is
of independent interest, but it will involve von Neumann algebras
and non-commutative geometry. Note that to study this, one is faced
with two non-commuting unitary one-parameter groups acting on $L^{2}\left(M\right)$
(corresponding to the two coordinates for $M$). Of interest here
is the von Neumann algebra generated by these two non-commuting unitary
one-parameter groups. It is likely that this von Neumann algebra is
a type III factor. There is a sequence of interesting papers by K.
Schmudgen on dealing with some of this \cite{MR755571,MR774726,MR808690,MR847352,MR829589}.\index{Riemann surface}\index{von Neumann algebra}\index{unitary one-parameter group}
\end{remark}

\paragraph{\textbf{A locally defined p.d. functions $F$ on $G=\mathbb{R}^{2}$
with $Ext\left(F\right)=\emptyset$.}}

Let $M$ be the Riemann surface of the complex $\log z$; see Figure
\ref{fig:logz}. 

Denote $B_{r}\left(x\right):=\left\{ y\in\mathbb{R}^{2}\::\:\left|x-y\right|<r,\;r>0\right\} $,
$x\in\mathbb{R}^{2}$, the open neighborhood of points $x\in\mathbb{R}^{2}$
of radius $r$. Pick $\xi\in C_{c}^{\infty}\left(B_{1/2}\left(1,1\right)\right)$,
such that 
\begin{equation}
\iint\xi=1,\quad\xi\geq0.\label{eq:lg1}
\end{equation}
Note that $\xi$ is supported on level 1 of the covering surface $M$,
i.e., on the sheet $\mathbb{R}^{2}\backslash\left\{ \left(x,y\right):x\geq0\right\} $;
assuming the branch cuts are along the positive $x$-axis.

Let $\pi:M\rightarrow\mathbb{R}^{2}\backslash\left\{ 0\right\} $
be the covering mapping. Set 
\begin{equation}
\left(\sigma\left(t\right)\xi\right)\left(m\right):=\xi\left(\pi^{-1}\left(t+\pi\left(m\right)\right)\right),\quad\forall t\in B_{1/2}\left(0\right),\;\forall m\in M.\label{eq:lg2}
\end{equation}

\begin{lemma}
The function 
\begin{equation}
F\left(t\right):=\left\langle \xi,\sigma\left(t\right)\xi\right\rangle _{L^{2}\left(M\right)},\quad t\in B_{1/2}\left(0\right)\subset\mathbb{R}^{2}\label{eq:lg3}
\end{equation}
is defined on a local neighborhood of $0\in\mathbb{R}^{2}$, continuous
and p.d.; but cannot be extended to a continuous p.d. function $F^{ext}$
on $\mathbb{R}^{2}$. \end{lemma}
\begin{svmultproof2}
If $F$ has an extension, then there exists $\rho\in Rep\left(\mathbb{R}^{2},L^{2}\left(M\right)\right)$,
see (\ref{eq:lg0}), s.t. $\rho$ extends the local representation
$\sigma$ in (\ref{eq:lg2}), and this implies that we get strongly
commuting vector fields on $M$ generated by $\overline{\frac{\partial}{\partial x_{j}}\big|_{C_{c}^{\infty}\left(M\right)}},$
$j=1,2$, as in (\ref{eq:a1}). But we know that this is impossible
from Proposition \ref{prop:logz}. 

Indeed, if $U_{j}\left(t\right)$ are the one-parameter unitary groups
from (\ref{eq:a5}), then in our current setting, $U_{2}\left(1\right)U_{1}\left(1\right)\xi$
is supported on level 1 of the covering surface $M$, but $U_{1}\left(1\right)U_{2}\left(1\right)\xi$
is on level $-1$.
\end{svmultproof2}

\index{operator!strongly commuting-}
\begin{theorem}
\label{thm:log1}Let $F$ be the local p.d. function in (\ref{eq:lg3}),
then $Ext\left(F\right)=\emptyset$. \end{theorem}
\begin{svmultproof2}
The proof uses the lemma above. Assuming $F^{ext}\in Ext\left(F\right)$;
we then get the following correspondences (Figure \ref{fig:logz3})
which in turn leads to a contradiction. 

The contradiction is based on the graph in Figure \ref{fig:trans}.
We show that any $F^{ext}\in Ext\left(F\right)$ would lead to the
existence of two (globally) commuting and unitary one-parameter groups
of translations. Following supports, this is then shown to be inconsistent
(Figure \ref{fig:logz}-\ref{fig:trans}). \index{representation!translation-}
\end{svmultproof2}

\begin{figure}
\[
G=\mathbb{R}^{2},\;\mathscr{O}=B_{\frac{1}{2}}\left(0\right)\subset\mathbb{R}^{2}
\]

\[
\xymatrix{\begin{bmatrix}\mbox{\ensuremath{F} cont., p.d. in \ensuremath{B_{\frac{1}{2}}\left(0\right)\subset\mathbb{R}^{2}}}\end{bmatrix}\ar[d] &  & \begin{bmatrix}F^{ext}\left(t\right)=\left\langle \xi,U\left(t\right)\xi\right\rangle _{L^{2}\left(M\right)}\end{bmatrix}\\
\mbox{\ensuremath{\begin{bmatrix}\exists\:\mbox{local reprep.}\\
\mbox{\ensuremath{\mathscr{H}=L^{2}\left(M\right)}, \ensuremath{v_{0}=\xi},}\\
\{\sigma\left(t\right):t\in B_{\frac{1}{2}}\left(0\right)\}\\
\mbox{see }\left(\ref{eq:lg3}\right) 
\end{bmatrix}}}\ar[d] &  & \mbox{\ensuremath{\begin{bmatrix}\mbox{unitary reprep. \ensuremath{U} of}\\
\mathbb{R}^{2}\mbox{ acting on \ensuremath{L^{2}\left(M\right)}} 
\end{bmatrix}}}\ar[u]\\
\mbox{\ensuremath{\begin{bmatrix}\mbox{reprep. of the}\\
\mbox{comm. Lie algebra} 
\end{bmatrix}}}\ar[rr]_{\mbox{if extendable}} &  & \mbox{\ensuremath{\begin{bmatrix}\mbox{\ensuremath{\exists}\ two comm. unitary}\\
\mbox{one-parameter groups} 
\end{bmatrix}}}\ar[u]_{\text{CONTRADICTION}}
}
\]

\protect\caption{\label{fig:logz3}Extension correspondence in the $\log z$ example.
From locally defined p.d. function $F$ in a neighborhood of $(0,0)$
in $G=\mathbb{R}^{2}$, to a representation of the 2-dimensional Abelian
Lie algebra by operators acting on a dense domain in $L^{2}(M)$.
This Lie algebra representation does not exponentiate to a unitary
representation of $G=\mathbb{R}^{2}$.\index{extension correspondence} }
\end{figure}
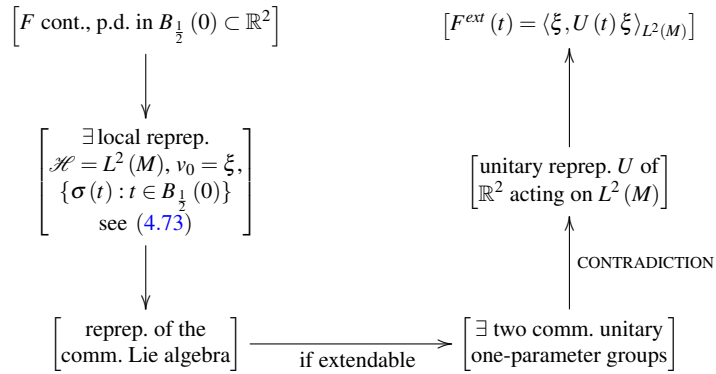

\motto{``Mathematics links the abstract world of mental concepts to the
real world of physical things without being located completely in
either.\textquotedblright{} --- Ian Stewart, Preface to second edition
of What is Mathematics? by Richard Courant and Herbert Robbins}

\chapter{\label{chap:types}Type I vs. Type II Extensions}

In this chapter, we identify extensions of the initially give positive
definite (p.d.) functions $F$ which are associated with operator
extensions in the RKHS $\mathscr{H}_{F}$ itself (Type I), and those
which require an enlargement of $\mathscr{H}_{F}$, Type II. In the
case of $G=\mathbb{R}$ (the real line) some of these continuous p.d.
extensions arising from the second construction involve a \emph{spline-procedure},
and a \emph{theorem of G. Pólya}, which leads to p.d. extensions of
$F$ that are \emph{symmetric} around $x=0$, and \emph{convex on
the left and right half-lines}. Further these extensions are supported
in a compact interval, symmetric around $x=0$. 

A main result in this chapter (Theorem \ref{thm:Ext2}) concerns the
set $Ext\left(F\right)$. Our theorem applies to any positive definite
function $F$ which is defined in an interval, centered at $0$, such
that $F$ is also \emph{analytic in a neighborhood} of $0$. Let $D^{\left(F\right)}$
be the associated skew-Hermitian operator in $\mathscr{H}_{F}$. Under
these assumptions, Theorem \ref{thm:Ext2} states that the operator
$D^{\left(F\right)}$ will automatically be essentially skew-adjoint,
i.e., it has indices $(0,0)$, and, moreover we conclude that $Ext(F)$
is \emph{a singleton}. In particular, under these assumptions, we
get that the subset $Ext_{2}(F)$ of $Ext(F)$ is empty. (Note, however,
that our non-trivial Pólya spline-extensions are convex on the positive
and the negative half-lines. And further that, for those, the initially
given p.d. function $F$ will \emph{not} be analytic in a neighborhood
of 0. Rather, for this class of examples, the initial p.d. function
$F$ is convex on the two finite intervals, on either side of $0$.
In this setting, we show that there are many non-trivial Pólya spline-extensions,
and they are in $Ext_{2}(F)$.)

\index{spline}\index{RKHS}\index{positive definite}\index{Pólya, G.}\index{dilation Hilbert space}\index{Pólya extensions}\index{extensions!Pólya-}\index{extensions!spline-}\index{essentially skew-adjoint}\index{convex}

\section{\label{sec:Polya}Pólya Extensions}

We need to recall Pólya's theorem \cite{pol49} regarding positive
definite (p.d.) functions. For splines and p.d. functions, we refer
to \cite{Sch83,GSS83}.
\begin{theorem}[Pólya]
 Let $f:\mathbb{R}\rightarrow\mathbb{R}$ be a continuous function,
and assume that 
\begin{enumerate}
\item $f\left(0\right)=1$, 
\item $\lim_{t\rightarrow\infty}f\left(t\right)=0$,
\item $f\left(-t\right)=f\left(t\right)$, $\forall t\in\mathbb{R}$, and
\item $f|_{\mathbb{R}_{+}}$ is convex. 
\end{enumerate}

Then it follows that $f$ is positive definite; and as a result that
there is a probability measure $\mu$ on $\mathbb{R}$ such that 
\begin{equation}
f\left(t\right)=\int_{\mathbb{R}}e^{i\lambda t}d\mu\left(\lambda\right),\quad\forall t\in\mathbb{R}.\label{eq:pa1}
\end{equation}

\end{theorem}

\begin{proposition}
Let $F:\left(-a,a\right)\rightarrow\mathbb{C}$ be p.d., continuous.
Assume $F$ has a Pólya extension $F_{ex}$ supported in $\left[-c,c\right]$,
with $c>a>0$; then the corresponding measure $\mu_{ex}\in Ext\left(F\right)$
has the following form: 
\[
d\mu_{ex}\left(\lambda\right)=\Phi_{ex}\left(\lambda\right)d\lambda,\;\mbox{where}
\]
\[
\Phi_{ex}\left(\lambda\right)=\frac{1}{2\pi}\int_{-c}^{c}e^{-i\lambda y}F_{ex}\left(y\right)dy
\]
is entire analytic in $\lambda$.\index{Pólya extensions}\index{extensions!Pólya-}\index{operator!Pólya's-}\end{proposition}
\begin{svmultproof2}
An application of Fourier inversion, and the Paley-Wiener theorem.
\end{svmultproof2}

The construction of Pólya extensions is as follow: Starting with a
\emph{convex} p.d. function $F$ on a finite interval $\left(-a,a\right)$,
we create a new function $F_{ex}$ on $\mathbb{R}$, such that $F_{ex}\big|_{\mathbb{R}_{+}}$
is convex, and $F_{ex}\left(-x\right)=F_{ex}\left(x\right)$. Pólya's
theorem \cite{pol49} states that $F_{ex}$ is a p.d. extension of
$F$. 

\index{Pólya extensions}

\index{extensions!spline-}

\index{extensions!Pólya-}

As illustrated in Figure \ref{fig:polya}, after extending $F$ from
$\left(-a,a\right)$ by adding one or more line segments over $\mathbb{R}_{+}$,
and using symmetry by $x=0$, there will be a constant $c$, with
$0<a<c$, such that the extension $F_{ex}$ satisfies $F_{ex}\left(x\right)=0$,
for all $\left|x\right|\geq c$. 

\begin{figure}[H]
\includegraphics[width=0.7\textwidth]{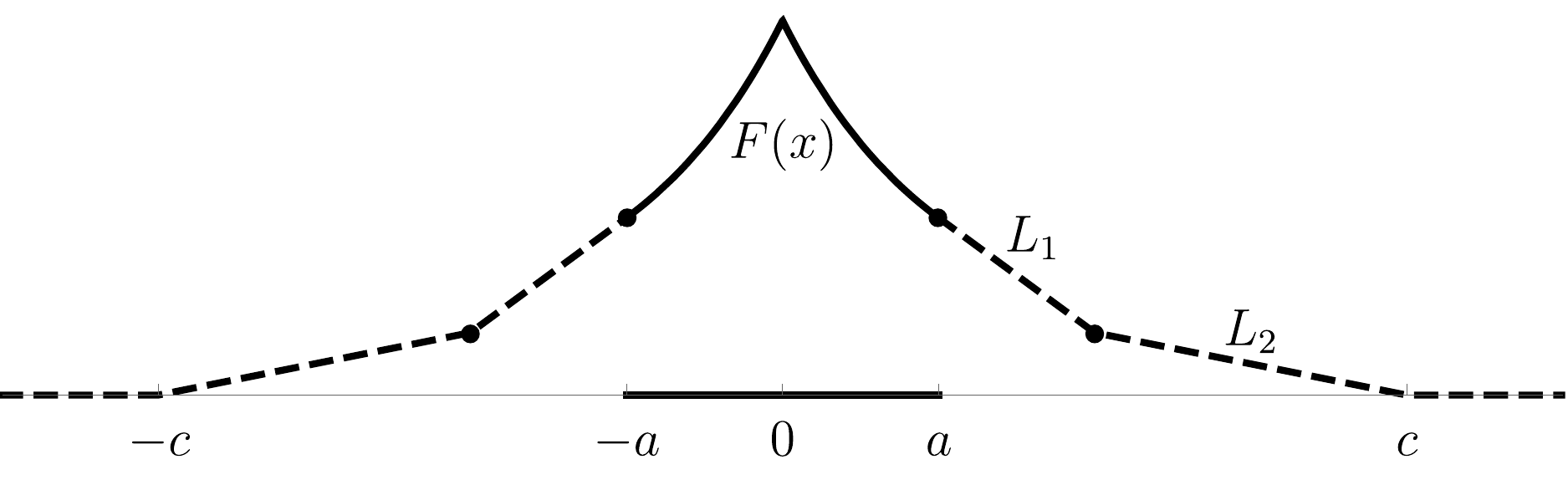}

\protect\caption{\label{fig:polya}An example of Pólya extension of $F$ on $\left(-a,a\right)$.
On $\mathbb{R}_{+}$, extend $F$ by adding line segments $L_{1}$
and $L_{2}$, and then use symmetry by $x=0$ to get $F_{ex}$. By
Pólya's theorem, the extension $F_{ex}$ is p.d. on $\mathbb{R}$. }
\end{figure}

In order to apply Pólya\textquoteright s theorem, the spline extended
function $F_{ex}$ to $\mathbb{R}$ must be \emph{convex} on $\mathbb{R}_{+}$.
In that case, $F_{ex}$ will be positive definite. However, we may
also start with a continuous p.d. function $F$ on $\left(-a,a\right)$,
which is \emph{concave} near $x=0$, and consider spline extensions
that are supported in $\left[-c,c\right]$, for some $c>a$. In Figure
\ref{fig:spline0}, $F$ is concave in a neighborhood of $0$; choose
the slope of $L_{+}=F'\left(a\right)$, and the slope of $L_{-}=F'\left(-a\right)=-F'\left(a\right)$,
using mirror symmetry around $x=0$. The extension $F_{ex}$ does
not satisfy the premise in Pólya\textquoteright s theorem (not convex
on $\mathbb{R}_{+}$), and so it may not be positive definite. 

\begin{figure}[H]
\includegraphics[width=0.7\textwidth]{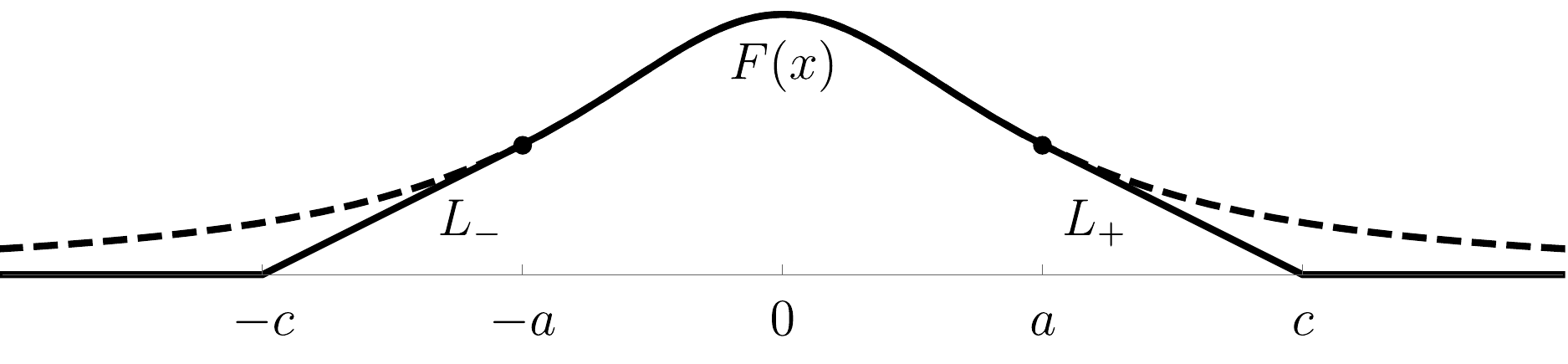}

\protect\caption{\label{fig:spline0}A spline extension of $F:\left(-a,a\right)\rightarrow\mathbb{R}$,
where $F$ is \emph{concave} around $x=0$. The extension $F_{ex}$
is not p.d. on $\mathbb{R}$.}
\end{figure}

Examples \ref{eg:F1}-\ref{eg:F6} contain six cases of locally defined
continuous p.d. functions $F_{i}$, where $F_{i}=\chi_{\left(-a_{i},a_{i}\right)}\widehat{d\mu_{i}}$,
$i=1,\cdots,6$; i.e., each $F_{i}$ is the restriction to a given
finite interval $\left(-a_{i},a_{i}\right)\subset\mathbb{R}$, of
a p.d. function $\widehat{d\mu_{i}}\left(x\right)$, $x\in\mathbb{R}$.
The corresponding measures $\mu_{i}$ are listed in Table \ref{tab:F1-F6}.
In the study of the p.d. extension problem, the following hold (see
Section \ref{sub:ExtSpace} for the definitions):
\begin{enumerate}
\item Each $F_{i}$ has a trivial p.d. extension to $\mathbb{R}$, $\widehat{d\mu_{i}}$.
In particular, $\mu_{i}\in Ext\left(F_{i}\right)\neq\emptyset$. 
\item $F_{1}$, $F_{4}$, $F_{5}$ and $F_{6}$ are \emph{concave} around
$x=0$, so they do \emph{not} yield spline extensions which are convex
when restricted to $\mathbb{R}_{+}$. 
\item Moreover, $F_{1}$, $F_{4}$, $F_{5}$ and $F_{6}$, are analytic
in a neighborhood of $0$. By Theorem \ref{thm:Ext2}, the skew-Hermitian
operators $D^{\left(F_{i}\right)}$ have deficiency indices $\left(0,0\right)$,
and $Ext\left(F_{i}\right)=Ext_{1}\left(F_{i}\right)=\left\{ \mu_{i}\right\} =$
singleton. Therefore, each $F_{i}$ has a unique p.d. extension to
$\mathbb{R}$, i.e., $\widehat{d\mu_{i}}$, which is Type I. As a
result, there will be no p.d. spline extensions for these four cases.
(The spline extensions $F_{ex}^{\left(i\right)}$ illustrated below
are supported in some $\left[-c_{i},c_{i}\right]$, but $F_{ex}^{\left(i\right)}$
are \emph{not} p.d. on $\mathbb{R}$.)
\item $F_{2}$ \& $F_{3}$ are \emph{convex} in a neighborhood of $0$,
and they have p.d. spline extensions in the sense of Pólya, which
are Type II. For these two cases, $D^{\left(F_{i}\right)}$ has deficiency
indices $\left(1,1\right)$, and $Ext_{1}\left(F_{i}\right)$ contains
atomic measures. As a result, we conclude that $\mu_{i}\in Ext_{2}\left(F_{i}\right)$,
i.e., the trivial extensions $\widehat{d\mu_{i}}$ are Type II.\index{convex}\end{enumerate}
\begin{example}[Cauchy distribution]
\label{eg:F1} $F_{1}\left(x\right)={\displaystyle \frac{1}{1+x^{2}}}$;
$\left|x\right|<1$. Note that $F_{1}$ is concave, and analytic in
a neighborhood of $0$.\index{Cauchy distribution} It follows from
Theorem \ref{thm:Ext2} that the spline extension in Figure \ref{fig:spline1}
is not positive definite.

\begin{figure}[H]
\includegraphics[width=0.7\textwidth]{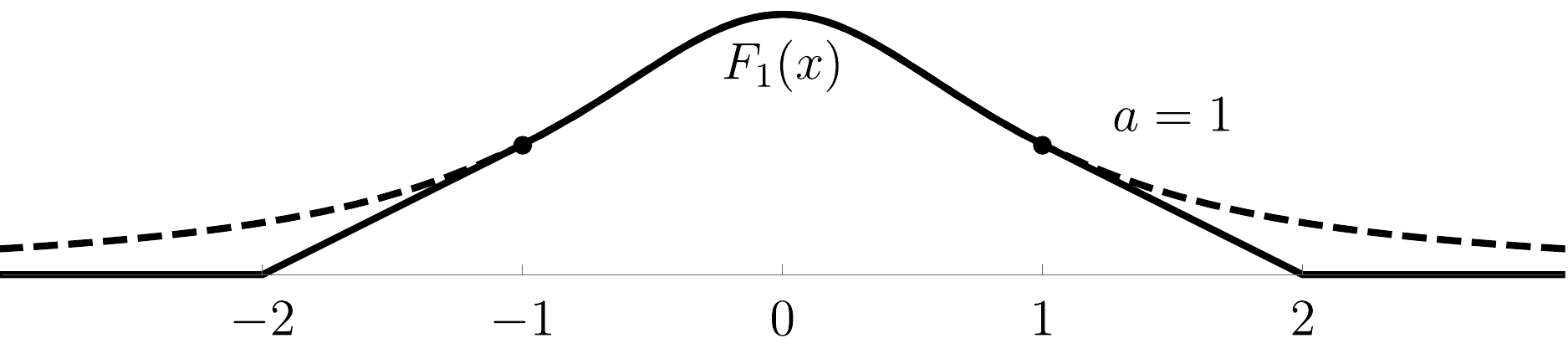}

\protect\caption{\label{fig:spline1}A spline extension of $F_{1}\left(x\right)={\displaystyle \tfrac{1}{1+x^{2}}}$;
$\Omega=\left(0,1\right)$.}
\end{figure}

\end{example}

\begin{example}
\label{eg:F2}$F_{2}\left(x\right)=1-\left|x\right|$; $\left|x\right|<\frac{1}{2}$.
Consider the following Pólya extension
\[
F\left(x\right)=\begin{cases}
1-\left|x\right| & \mbox{if }\left|x\right|<\frac{1}{2}\\
{\displaystyle \frac{1}{3}\left(2-\left|x\right|\right)} & \mbox{if }\frac{1}{2}\leq\left|x\right|<2\\
0 & \mbox{if }\left|x\right|\geq2
\end{cases}
\]

\begin{figure}[H]
\includegraphics[width=0.7\textwidth]{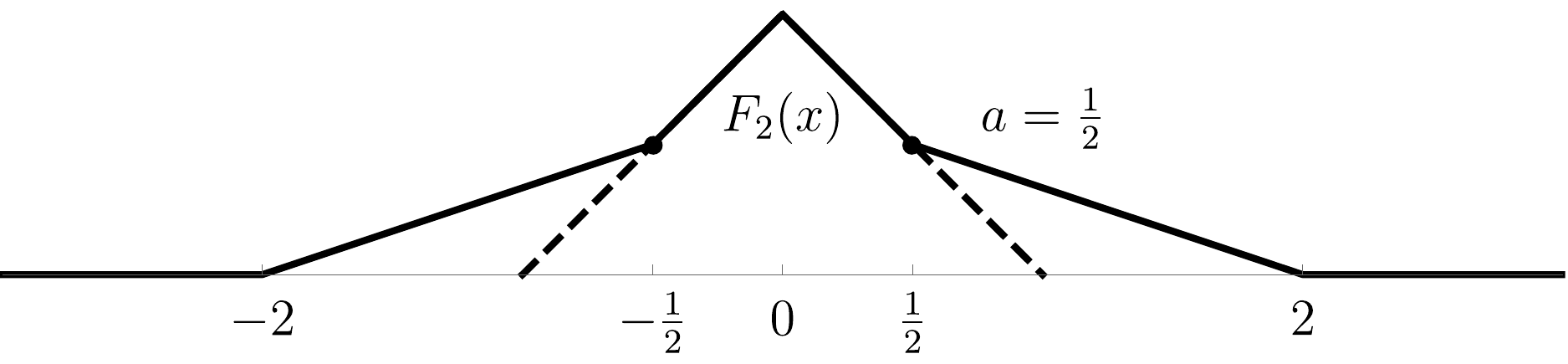}

\protect\caption{\label{fig:spline2}A positive definite spline extension of $F_{2}\left(x\right)=1-\left|x\right|$;
$\Omega=\left(0,\frac{1}{2}\right)$}
\end{figure}

This is a p.d. spline extension which is \emph{convex} on $\mathbb{R}_{+}$.
The corresponding measure $\mu\in Ext\left(F\right)$ has the following
form $d\mu\left(\lambda\right)=\Phi\left(\lambda\right)d\lambda$,
where $d\lambda=$ Lebesgue measure on $\mathbb{R}$, and 
\[
\Phi\left(\lambda\right)=\begin{cases}
\frac{3}{4\pi} & \mbox{if }\lambda=0\\
\frac{1}{3\pi\lambda^{2}}\left(3-2\cos\left(\lambda/2\right)-\cos\left(2\lambda\right)\right) & \mbox{if }\lambda\neq0.
\end{cases}
\]
This solution $\mu$ is in $Ext_{2}\left(F\right)$, and similarly,
the measure $\mu_{2}$ from Table \ref{tab:Table-3} is in $Ext_{2}\left(F\right)$.
\index{extensions!spline-}
\end{example}

\begin{example}[Ornstein-Uhlenbeck]
\label{eg:F3}$F_{3}\left(x\right)=e^{-\left|x\right|}$; $\left|x\right|<1$.
A positive definite spline extension which is convex on $\mathbb{R}_{+}$.\index{Ornstein-Uhlenbeck}

\begin{figure}[H]
\includegraphics[width=0.7\textwidth]{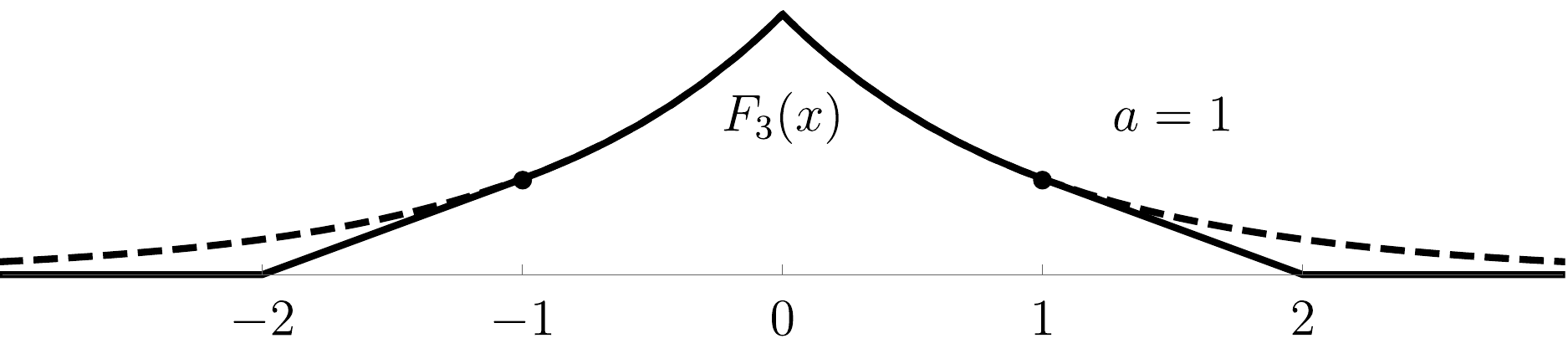}

\protect\caption{\label{fig:spline3}A positive definite spline extension of $F_{3}\left(x\right)=e^{-\left|x\right|}$;
$\Omega=\left(0,1\right)$}
\end{figure}

\end{example}

\begin{remark}[Ornstein-Uhlenbeck]
\label{rem:OUprocess}The p.d. function $F\left(t\right):=e^{-\left|t\right|}$,
$t\in\mathbb{R}$, is of special significance in connection with the
Ornstein-Uhlenbeck process. See Figure \ref{fig:oup}, and Section
\ref{sub:OUprocess}.

To make the stated connection more direct, consider the standard Brownian
motion $B_{t}$, i.e., $\{B_{t}\}$ is Gaussian with mean zero and
covariance function $\mathbb{E}(B_{t_{1}}B_{t_{2}})=t_{1}\wedge t_{2}$,
for all $t_{1},t_{2}\geq0$. (See, e.g., \cite{Hi80}.) Fix $\alpha\in\mathbb{R}_{+}$,
and set 
\begin{equation}
X_{t}:=\frac{1}{\sqrt{\alpha}}e^{-\frac{\alpha t}{2}}B\left(e^{\alpha t}\right),\quad t\in\mathbb{R};\label{eq:un1-1}
\end{equation}
then
\begin{equation}
\mathbb{E}\left(X_{t_{1}}X_{t_{2}}\right)=\frac{1}{\alpha}e^{-\frac{\alpha}{2}\left|t_{1}-t_{2}\right|},\quad\forall t_{1},t_{2}\in\mathbb{R}.\label{eq:un1-2}
\end{equation}
The reader will be able to verify (\ref{eq:un1-2}) directly by using
the covariance kernel for Brownian motion. See also Example \ref{exa:bm0}.\index{distribution!Gaussian-}
\end{remark}
\begin{figure}
\includegraphics[width=0.6\columnwidth]{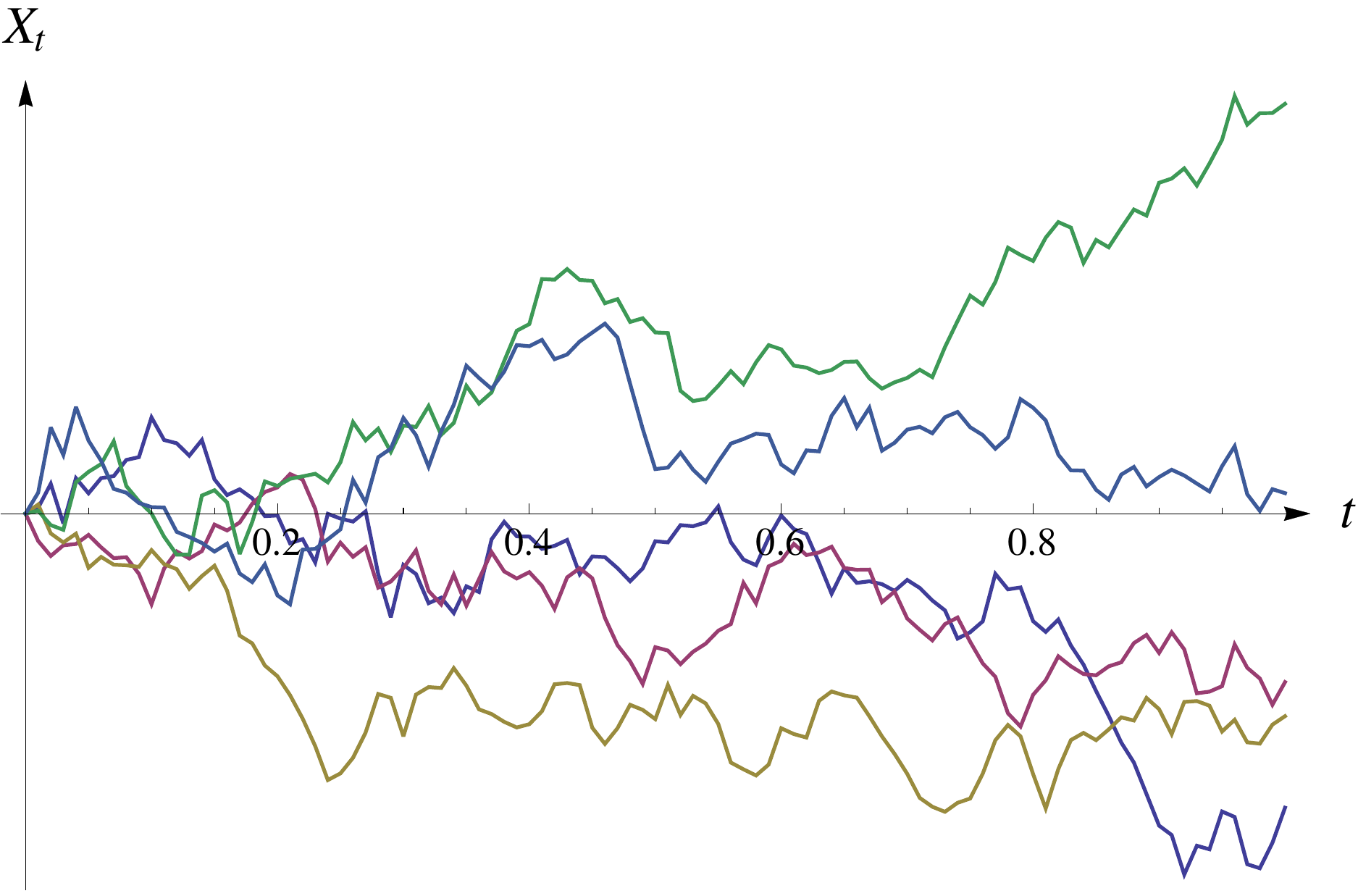}

\protect\caption{\label{fig:oup}Simulation of the Ornstein-Uhlenbeck process. Five
sample paths starting at $\left(0,0\right)$ with mean $\mu=0$, and
standard deviation $\sigma=0.3$.}
\end{figure}

\begin{example}[Shannon]
\label{eg:F4}$F_{4}\left(x\right)=\left({\displaystyle \frac{\sin\left(x/2\right)}{x/2}}\right)^{2}$;
$\left|x\right|<\frac{1}{2}$. $F_{4}$ is concave, and analytic in
a neighborhood of $x=0$.\index{Shannon sampling} By Theorem \ref{thm:Ext2},
$Ext\left(F_{4}\right)=Ext_{1}\left(F_{4}\right)=$ singleton. Therefore,
the spline extension in Figure \ref{fig:spline4} is not positive
definite.

\begin{figure}[H]
\includegraphics[width=0.7\textwidth]{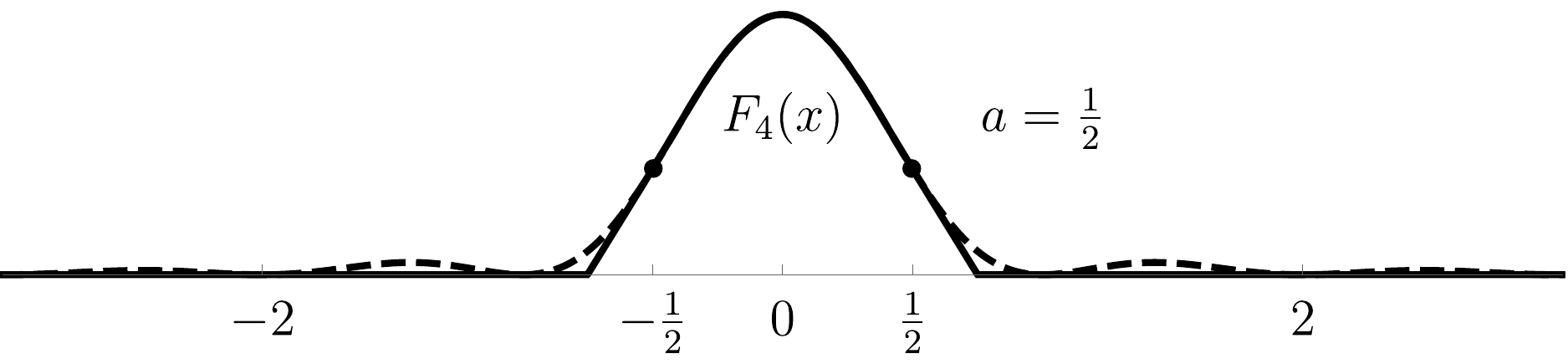}

\protect\caption{\label{fig:spline4}A spline extension of $F_{4}\left(x\right)=\left(\frac{\sin\left(x/2\right)}{x/2}\right)^{2}$;
$\Omega=\left(0,\frac{1}{2}\right)$}
\end{figure}

\end{example}

\begin{example}[Gaussian distribution]
\label{eg:F5}$F_{5}\left(x\right)=e^{-x^{2}/2}$; $\left|x\right|<1$.
$F_{5}$ is concave, analytic in $-1<x<1$. $Ext\left(F_{5}\right)=Ext_{1}\left(F_{5}\right)=$
singleton. The spline extension in Figure \ref{fig:spline5} is not
positive definite (see Theorem \ref{thm:Ext2}). \index{Gaussian distribution}\index{distribution!Gaussian-}

\begin{figure}[H]
\includegraphics[width=0.7\textwidth]{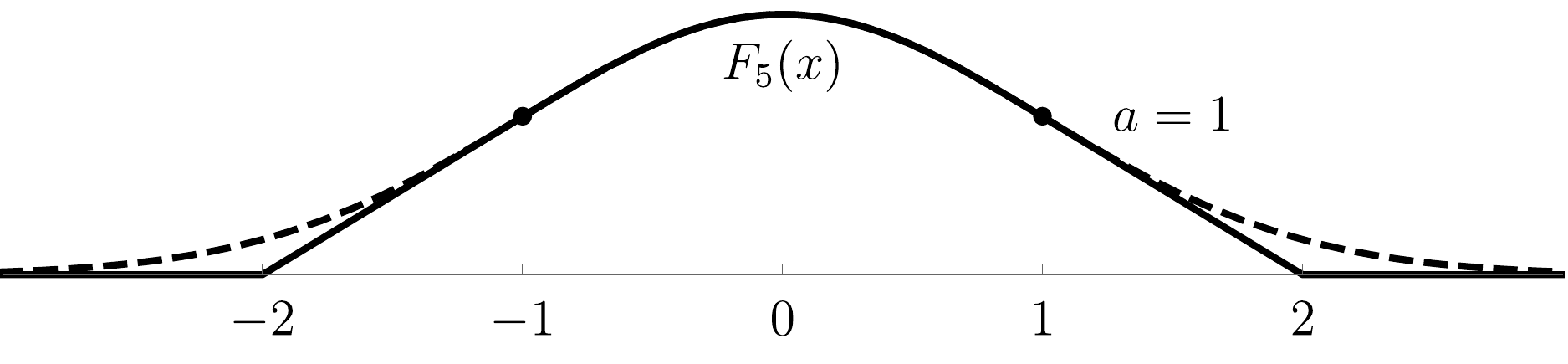}

\protect\caption{\label{fig:spline5}A spline extension of $F_{5}\left(x\right)=e^{-x^{2}/2}$;
$\Omega=\left(0,1\right)$}
\end{figure}

\end{example}

\begin{example}
\label{eg:F6}$F_{6}\left(x\right)=\cos\left(x\right)$; $\left|x\right|<\frac{\pi}{4}$.
$F_{6}$ is concave, analytic around $x=0$. By Theorem \ref{thm:Ext2},
the spline extension in Figure \ref{fig:spline6} is not positive
definite.

\begin{figure}[H]
\includegraphics[width=0.7\textwidth]{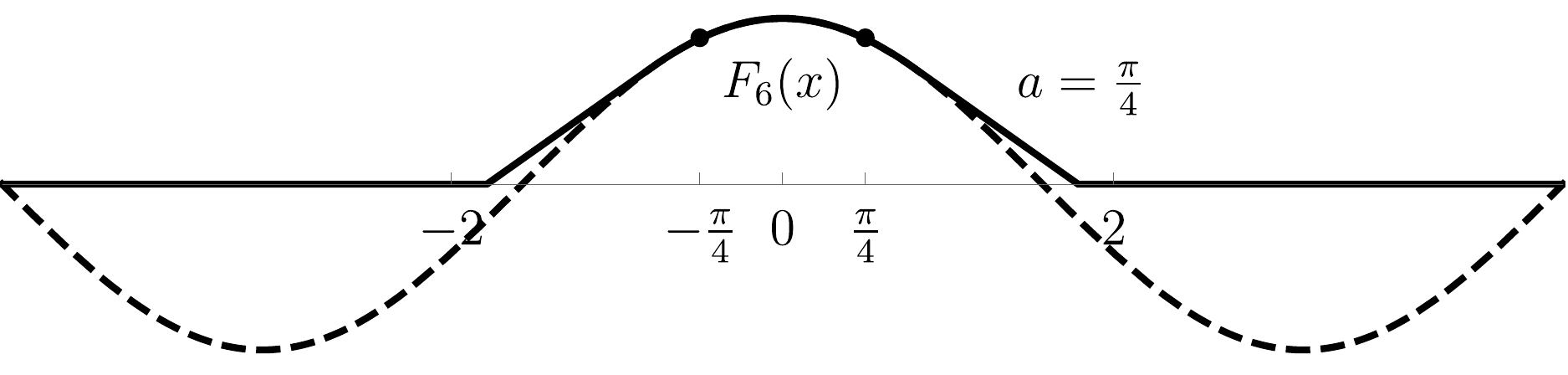}

\protect\caption{\label{fig:spline6}A spline extension of $F_{6}\left(x\right)=\cos\left(x\right)$;
$\Omega=\left(0,\frac{\pi}{4}\right)$}
\end{figure}
\end{example}
\begin{lemma}
\label{lem:cos}The function $\cos x$ is positive definite on $\mathbb{R}$. \end{lemma}
\begin{svmultproof2}
For all finite system of coefficients $\left\{ c_{j}\right\} $ in
$\mathbb{C}$, we have 
\begin{eqnarray}
\sum_{j}\sum_{k}\overline{c_{j}}c_{k}\cos\left(x_{j}-x_{k}\right) & = & \sum_{j}\sum_{k}\overline{c_{j}}c_{k}\left(\cos x_{j}\cos x_{k}+\sin x_{j}\sin x_{k}\right)\nonumber \\
 & = & \Bigl|\sum_{j}c_{j}\cos x_{j}\Bigr|^{2}+\Bigl|\sum_{j}c_{j}\sin x_{j}\Bigr|^{2}\geq0.\label{eq:ext-1}
\end{eqnarray}

\end{svmultproof2}

\index{extensions!spline-}\index{Pólya extensions}\index{extensions!Pólya-}\index{operator!Pólya's-}\index{extensions!spline-}
\begin{lemma}
\label{lem:distd}Consider the following two functions $F_{2}\left(x\right)=1-\left|x\right|$,
and $F_{3}\left(x\right)=e^{-\left|x\right|}$, each defined on a
finite interval $\left(-a,a\right)$, possibly a different value of
$a$ from one to the next. The distributional double derivatives are
as follows:
\begin{eqnarray*}
F_{2}'' & = & -2\delta_{0}\\
F_{3}'' & = & F_{3}-2\delta_{0}
\end{eqnarray*}
where $\delta_{0}$ is Dirac's delta function (i.e., point mass at
$x=0$, $\delta_{0}=\delta\left(x-0\right)$.)\end{lemma}
\begin{svmultproof2}
The conclusion follows from a computation, making use of L. Schwartz'
theory of distributions; see \cite{Tre06}.\index{Schwartz!distribution}\index{distribution!-solution}\index{distribution!-derivation}
\end{svmultproof2}

\section{Main Theorems}

Now, we shall establish the main result in this chapter (Theorem \ref{thm:Ext2})
concerning the set $Ext\left(F\right)$. We begin with two preliminary
theorems (one is a theorem by Carleman on moments, and the other by
Nelson on analytic vectors), and a lemma, which will all be used in
our proof of Theorem \ref{thm:Ext2}.
\begin{theorem}[Carleman \cite{Akh65}]
 \label{thm:carleman}Let $\mu$ be a positive Borel measure on $\mathbb{R}$
such that $t^{n}\in L^{1}\left(\mu\right)$ for all $n\in\mathbb{Z}_{+}\cup\left\{ 0\right\} $,
and set $m_{n}:=\int_{\mathbb{R}}t^{n}d\mu\left(t\right)$. If 
\begin{equation}
\sum_{k=1}^{\infty}m_{2k}^{-\frac{1}{2k}}=\infty,\label{eq:carleman}
\end{equation}
then the set of positive measures with these moments is a singleton.
We say that the moment problem is \uline{determinate}. \index{determinate}\index{moment problem}\index{moments}
\index{Carleman condition}
\end{theorem}
We recall E. Nelson's theorem on analytic vectors:\index{analytic vector}
\begin{theorem}[Nelson \cite{Nel59}]
\label{thm:av}Let $D$ be a skew-Hermitian operator with dense domain
$dom\left(D\right)$ in a Hilbert space. Suppose $dom\left(D\right)$
contains a dense set of vectors $v$ s.t. $\exists\:C_{0},C_{1}<\infty$,
\begin{equation}
\left\Vert D^{n}v\right\Vert \leq C_{0}\:n!\:C_{1}^{n},\quad n=0,1,2,\cdots\label{eq:ne4}
\end{equation}
where the constants $C_{0}$, $C_{1}$ depend on $v$; then $D$ is
essentially skew-adjoint, and so it has deficiency indices $\left(0,0\right)$.
Vectors satisfying (\ref{eq:ne4}) are called analytic vectors.\end{theorem}
\begin{lemma}
\label{lem:ne}Let $F$ be a p.d. function defined on $\left(-a,a\right)$,
$a>0$, and assume that $F$ is analytic in a neighborhood of $x=0$;
then $D^{\left(F\right)}$ has deficiency indices $\left(0,0\right)$,
i.e., $D^{\left(F\right)}$ is essentially skew-adjoint, and so $Ext_{1}\left(F\right)$
is a singleton. \end{lemma}
\begin{svmultproof2}
We introduce $\mathscr{H}_{F}$ and $D^{\left(F\right)}$ as before.
Recall that $dom(D^{\left(F\right)})$ consists of 
\begin{equation}
F_{\varphi}\left(x\right)=\int_{0}^{a}\varphi\left(y\right)F\left(x-y\right)dy,\quad\varphi\in C_{c}^{\infty}\left(0,a\right)\label{eq:ne1}
\end{equation}
and 
\begin{equation}
D^{\left(F\right)}(F_{\varphi})\left(x\right)=F_{\varphi'}\left(x\right)=\frac{d}{dx}F_{\varphi}\left(x\right).\;(\mbox{integration by parts})\label{eq:ne2}
\end{equation}
Since $F$ is locally analytic, we get 
\begin{equation}
(D^{\left(F\right)})^{n}F_{\varphi}=\varphi*F^{\left(n\right)},\quad\varphi\in C_{c}^{\infty}\left(0,a\right).\label{eq:ne3}
\end{equation}
The lemma will be established by the following:
\begin{claim}
The vectors $\{F_{\varphi}:\varphi\in C_{c}^{\infty}\left(0,\varepsilon\right)\}$,
for sufficiently small $\varepsilon>0$, satisfy the condition in
(\ref{eq:ne4}). Note that $\left\{ F_{\varphi}\right\} $ is dense
in $\mathscr{H}_{F}$ by Section \ref{sec:Prelim}, so these are analytic
vectors in the sense of Nelson (see Theorem \ref{thm:av}, and \cite{Nel59}).
\index{analytic vector}\end{claim}
\begin{svmultproof2}
Since $F$ is assumed to be analytic around $x=0$, there exists $\varepsilon>0$
s.t. 
\begin{equation}
F\left(x\right)=\sum_{n=0}^{\infty}\frac{F^{\left(n\right)}\left(y\right)}{n!}\left(x-y\right)^{n}\label{eq:ne6}
\end{equation}
holds for $\left|x\right|\ll2\varepsilon$, $\left|y\right|<\varepsilon$;
and this implies that 
\begin{equation}
\left|F^{\left(n\right)}\left(y\right)\right|\leq C_{0}\:n!\varepsilon^{-n}.\label{eq:ne7}
\end{equation}
We now use (\ref{eq:ne3}) to establish the estimate (\ref{eq:ne4})
for $\varepsilon$ as in (\ref{eq:ne6})-(\ref{eq:ne7}). We pick
$\varepsilon$, and consider $\varphi\in C_{c}^{\infty}\left(0,\varepsilon\right)$,
then 
\begin{eqnarray}
(D^{\left(F\right)})^{n}F_{\varphi}\left(x\right) & = & \int_{0}^{\varepsilon}\varphi\left(y\right)F^{\left(n\right)}\left(x-y\right)dy\nonumber \\
 & = & \int_{0}^{\varepsilon}\varphi\left(x-y\right)F^{\left(n\right)}\left(y\right)dy.\label{eq:ne8}
\end{eqnarray}
Using Corollary \ref{cor:Ddif}, we get 
\begin{equation}
\left\Vert F^{\left(m\right)}\right\Vert _{\mathscr{H}_{F}}^{2}=-F^{\left(2m\right)}\left(0\right)\label{eq:ne9}
\end{equation}
and in particular $F^{\left(2m\right)}\left(0\right)\leq0$. So, 
\begin{eqnarray*}
\left\Vert (D^{\left(F\right)})^{n}\left(F_{\varphi}\right)\right\Vert _{\mathscr{H}_{F}}^{2} & \underset{\text{\ensuremath{\left(\ref{eq:ne8}\right)}}}{\leq} & \left(\int_{0}^{\varepsilon}\left|\varphi\right|\right)^{2}\left\Vert F^{\left(n\right)}\right\Vert _{\mathscr{H}_{F}}^{2}\\
 & \underset{\text{\ensuremath{\left(\ref{eq:ne9}\right)}}}{\leq} & \left(\int_{0}^{\varepsilon}\left|\varphi\right|\right)^{2}\left(-F^{\left(2n\right)}\left(0\right)\right)\\
 & \underset{\text{\ensuremath{\left(\ref{eq:ne7}\right)}}}{\leq} & \left(\int_{0}^{\varepsilon}\left|\varphi\right|\right)^{2}C_{0}^{2}\left(\frac{2}{\varepsilon}\right)^{2n}\left(n!\right)^{2};
\end{eqnarray*}
i.e., 
\[
\left\Vert (D^{\left(F\right)})^{n}\left(F_{\varphi}\right)\right\Vert _{\mathscr{H}_{F}}\leq\left(\int_{0}^{\varepsilon}\left|\varphi\right|\right)C_{0}\left(\frac{2}{\varepsilon}\right)^{n}n!,
\]
which is the Nelson estimate (\ref{eq:ne4}). This completes the proof
of the claim. 

Note that in Corollary \ref{cor:Ddif}, we proved (\ref{eq:ne9})
for $m=1$, but it follows in general by induction: $m=1$, $\left\Vert F'\right\Vert _{\mathscr{H}_{F}}^{2}=-F^{\left(2\right)}\left(0\right)$.

The fact that the set $\{F_{\varphi}:\varphi\in C_{c}^{\infty}\left(0,\varepsilon\right)\}$
is dense in $\mathscr{H}_{F}$ follows from an argument in \cite{Nel59,Jor87}:

Starting with $F_{\varphi}$, $\varphi\in C_{c}^{\infty}\left(0,\varepsilon\right)$,
using a local translation, we can move $F_{\varphi}$ towards the
endpoint $x=a$. This local translation preserves the $\mathscr{H}_{F}$-norm
(see, e.g., the proof of Lemma \ref{lem:DF}), and so $F_{\varphi}\left(\cdot-s\right)$,
with sufficiently small $s$, will also satisfy the analytic estimates
with the same constants $C$ as that of $F_{\varphi}$. Thus the analytic
estimates established for $F_{\varphi}$ with $\varphi\in C_{c}^{\infty}\left(0,\varepsilon\right)$
carries over to the rest of the interval up to the endpoint $x=a$.
It remains to note that $\left\{ F_{\varphi}:\varphi\in C_{c}^{\infty}\left(0,a\right)\right\} $
is dense in $\mathscr{H}_{F}$; see Lemma \ref{lem:RKHS-def-by-integral}. 
\end{svmultproof2}

Therefore, we get a dense set of analytic vectors for $D^{\left(F\right)}$
as claimed. By Nelson's theorem, we conclude that $D^{\left(F\right)}$
is essentially skew-adjoint, so it has deficiency indices $\left(0,0\right)$.
\index{essentially skew-adjoint}\end{svmultproof2}

\begin{theorem}
\label{thm:Ext2}Let $F$ be a continuous p.d. function given in some
finite interval $\left(-a,a\right)$, $a>0$. Assume that $F$ is
analytic in a neighborhood of 0; then $Ext_{2}\left(F\right)$ is
empty.\end{theorem}
\begin{svmultproof2}
Since $F$ is assumed to be analytic in a neighborhood of $0$, by
Lemma \ref{lem:ne}, we see that $D^{\left(F\right)}$ has deficiency
indices $\left(0,0\right)$, i.e., it is essentially skew-adjoint;
and $Ext_{1}\left(F\right)=\left\{ \mu\right\} $, a singleton. Thus,
\begin{equation}
F\left(x\right)=\widehat{d\mu}\left(x\right),\quad\forall x\in\left(-a,a\right).\label{eq:un1}
\end{equation}
Also, by the analytic assumption, there exists $c>0$ such that 
\begin{equation}
F\left(x\right)=\sum_{n=0}^{\infty}\frac{F^{\left(n\right)}\left(0\right)}{n!}x^{n},\quad\forall\left|x\right|<c;\label{eq:un2}
\end{equation}
and using (\ref{eq:un1}), we get that 
\begin{equation}
\int_{\mathbb{R}}e^{itx}d\mu\left(t\right)=\sum_{n=0}^{\infty}\frac{\left(ix\right)^{n}}{n!}\int_{\mathbb{R}}t^{n}d\mu\left(t\right),\quad x\in\mathbb{R},\:\left|x\right|<c.\label{eq:un3}
\end{equation}
Combining (\ref{eq:un2}) and (\ref{eq:un3}), we see that $\exists0<c_{1}<c$
such that the $n^{th}$-moment of $\mu$, i.e., 
\begin{equation}
m_{n}=\int_{\mathbb{R}}t^{n}d\mu\left(t\right)=F^{\left(n\right)}\left(0\right)\left(-i\right)^{n}\label{eq:un4}
\end{equation}
satisfies the following estimate: $\exists c_{0}<\infty$ such that
\[
\left|\frac{x^{n}}{n!}m_{n}\right|\leq c_{0},\quad\forall\left|x\right|\leq c_{1};
\]
that is, 
\begin{equation}
\left|m_{n}\right|\leq c_{0}\frac{n!}{c_{1}^{n}},\quad\forall n=0,1,2,\cdots.\label{eq:un5}
\end{equation}
In particular, all the moments of $\mu$ are finite.

Now, we may apply Carleman's condition \cite{Akh65} with the estimate
in (\ref{eq:un5}), and conclude that there is a unique measure with
the moments specified by the r.h.s. of (\ref{eq:un4}); i.e., the
measure is uniquely determined by $F$. Therefore, $Ext\left(F\right)=\left\{ \mu\right\} $,
and so $Ext_{2}\left(F\right)=Ext\left(F\right)\backslash Ext_{1}\left(F\right)=\emptyset$,
which is the desired conclusion. (To apply Carleman's condition, we
make use of Stirling\textquoteright s asymptotic formula for the factorials
in (\ref{eq:un5})). \index{moments}

By Carleman\textquoteright s theorem (see Theorem \ref{thm:carleman}),
there is a unique measure with these moments, but by (\ref{eq:un2})
and (\ref{eq:un4}) these moments also determine $F$ uniquely, and
vice versa. Hence the conclusion that $Ext(F)$ is a singleton.\index{moment problem}\index{Carleman condition}
\end{svmultproof2}

\subsection{\label{sec:exApp}Some Applications}

Below we discuss a family of examples of positive definite functions
$F$, defined initially only in the interval $-1<x<1$, but allows
analytic continuation to a complex domain, 
\[
\mathscr{O}^{\mathbb{C}}:=\left\{ z\in\mathbb{C}\:|\:\Im\left\{ z\right\} >-a\right\} ,
\]
see Figure \ref{fig:analytic}. Here, $a\geq0$ is fixed. \index{analytic continuation}

What is special for this family is that, for each $F$ in the family,
the convex set $Ext\left(F\right)$ is a singleton. It has the following
form: There is a probability measure $\mu$ on $\mathbb{R}$, such
that 
\begin{equation}
\widehat{d\mu}\left(\lambda\right)=M\left(\lambda\right)d\lambda,\quad\lambda\in\mathbb{R};\label{eq:ea1}
\end{equation}
where $M$ is supported on $[0,\infty)$, and $Ext\left(F\right)=\left\{ \mu\right\} $,
the singleton.

The list of these measures $d\mu=M\left(\lambda\right)d\lambda$ includes
the following distributions from statistics:
\begin{equation}
M\left(\lambda\right)=M_{p}\left(\lambda\right)=\frac{\lambda^{p-1}}{\Gamma\left(p\right)}e^{-\lambda},\quad\lambda\geq0;\label{eq:ea2}
\end{equation}
where $p$ is fixed, $p>0$, and $\Gamma\left(p\right)$ is the Gamma
function. (The case when $p=\frac{1}{2}\mathbb{Z}_{+}$, i.e., $p=\frac{n}{2}$,
is called the $\chi_{2}$-distribution of $n$ degrees of freedom.
See \cite{AMS13}.) \index{distribution!Gamma-}\index{Gamma (p) distribution}

Another example is the \emph{log-normal distribution} (see \cite{KM04}),
where 
\begin{equation}
M\left(\lambda\right)=\frac{1}{\sqrt{2\pi}}\frac{1}{\lambda}e^{-\left(\log\lambda-\mu_{0}\right)^{2}/2},\label{eq:ea3}
\end{equation}
and where $\mu_{0}\in\mathbb{R}$ is fixed. The domain in $\lambda$
is $\lambda>0$. See Figure \ref{fig:dist}. \index{log-normal distribution}

\begin{figure}
\begin{minipage}[t]{1\columnwidth}%
\subfloat[\label{fig:gam(p)}The Gamma (p) distribution with parameter $p$,
see (\ref{eq:ea2}).]{\protect\includegraphics[width=0.45\textwidth]{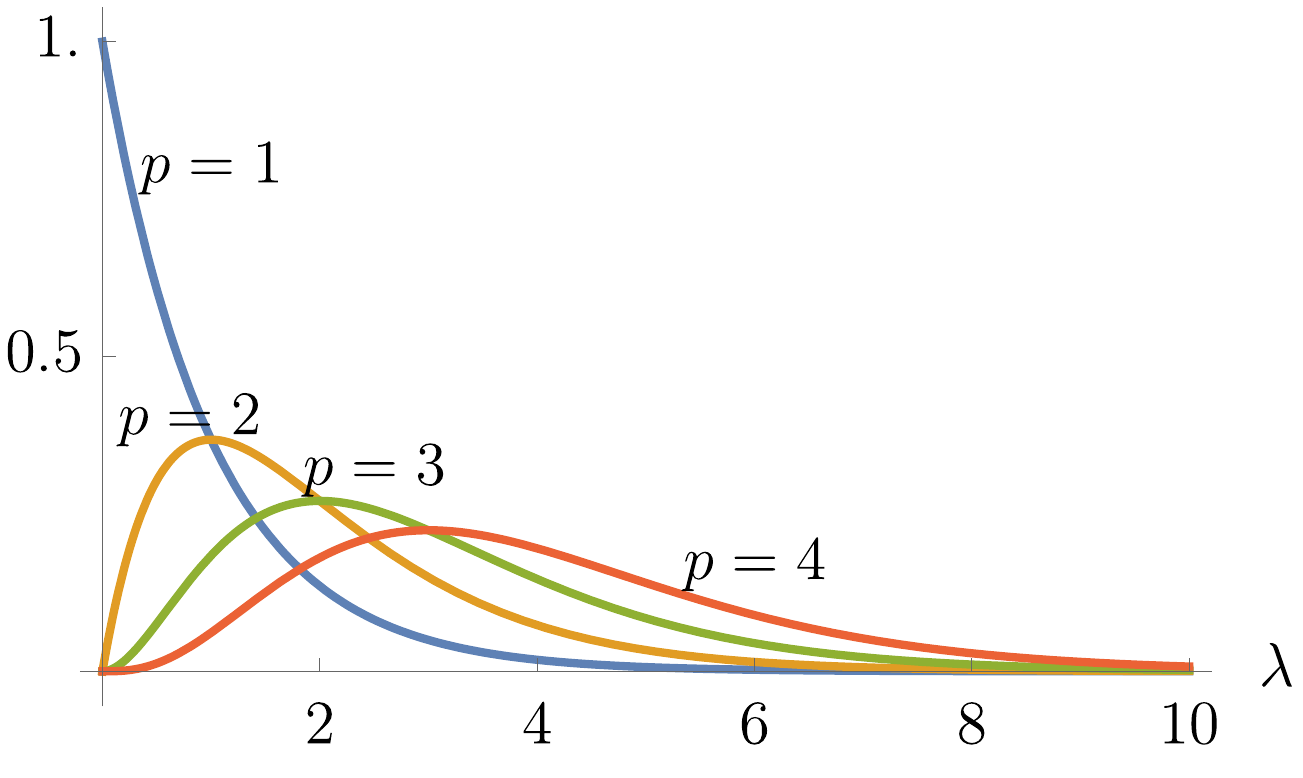}

}\hfill{}\subfloat[The log-normal distribution with parameter $\mu_{0}$, see (\ref{eq:ea3}).]{\protect\includegraphics[width=0.45\textwidth]{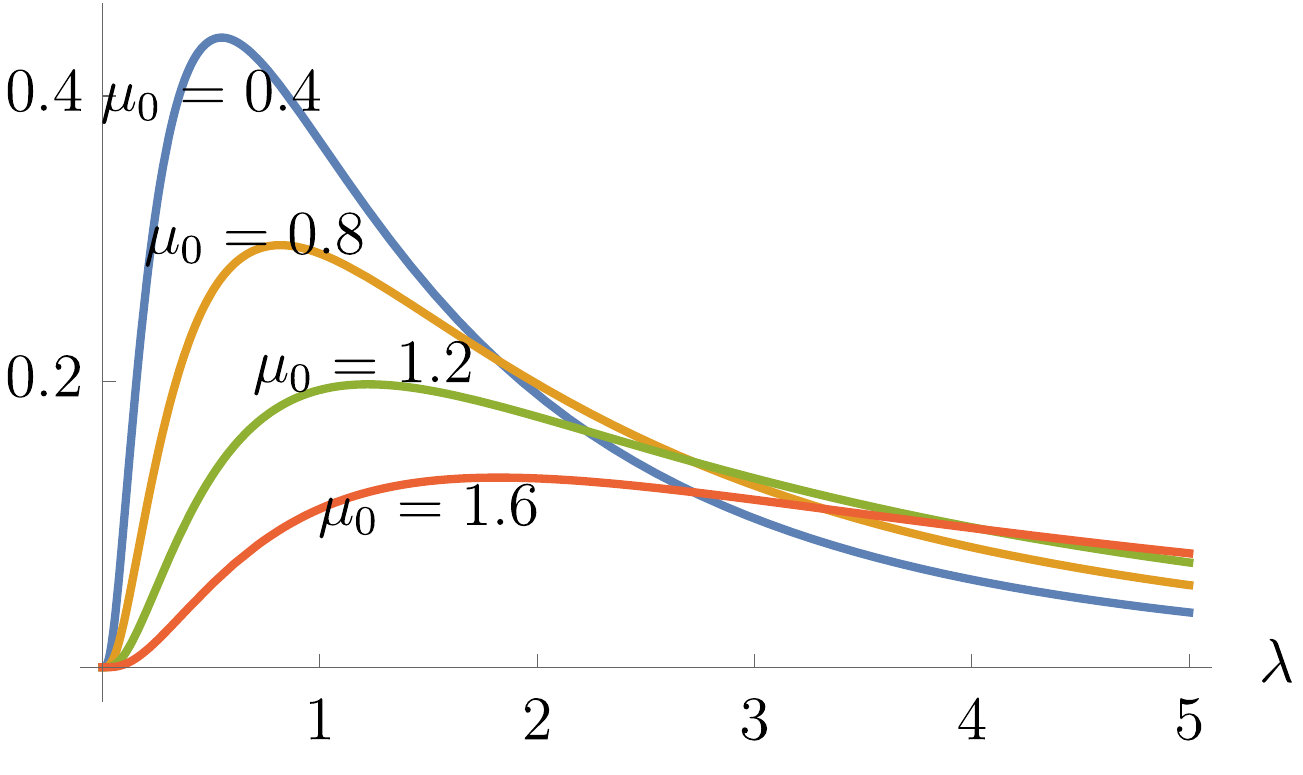}

}%
\end{minipage}

\protect\caption{\label{fig:dist}Probability densities of Gamma (p) and log-normal
distributions.}

\end{figure}

Both of the families are given by parameters, but different from one
to the other, see (\ref{eq:ea2}), Gamma; and \ref{eq:ea3}, log-normal.
They have in common the index $(0,0)$ conclusion. More specifically,
in each of the two classes, when an initial p.d. function $F$ is
specified in a finite interval $(-a,a)$, we compute the associated
skew-Hermitian operator $D^{\left(F\right)}$ in $\mathscr{H}_{F}$,
and this operator will have deficiency indices $(0,0)$. The big difference
between the two classes is that only in the first family will $F$
be analytic in a neighborhood of $x=0$; not for the other. We also
stress that these families are very important in applications, see
e.g., \cite{KT08,KRAA15,Lin14}.

We shall consider extension theory for the p.d. functions connected
with the first family of these distributions, i.e., (\ref{eq:ea2}).
The conclusions we list in Example \ref{exa:ea} below all follow
from a direct application of Theorem \ref{thm:Ext2}. \index{distribution!log-normal-}
\begin{example}
\label{exa:ea}Fix $p\in\mathbb{R}_{+}$, and let $F=F_{p}$ be given
by
\begin{equation}
F\left(x\right):=\left(1-ix\right)^{-p},\quad\left|x\right|<1;\label{eq:ea4}
\end{equation}
i.e., $F$ is defined in the interval $\left(-1,1\right)$. For the
study of the corresponding $\mathscr{H}_{F}$, we write $\Omega=\left(0,1\right)$,
so that $\Omega-\Omega=\left(-1,1\right)$. The following properties
hold:
\begin{enumerate}
\item \label{enu:ea1}$F$ is positive definite (p.d.) in $\left(-1,1\right)$.
\item \label{enu:ea2}$F$ has a unique continuous p.d. extension $\widetilde{F}$
to $\mathbb{R}$.
\item \label{enu:ea3}The p.d. extension $\widetilde{F}$ in (\ref{enu:ea2})
has an analytic continuation, $x\rightarrow z\in\mathbb{C}$, to the
complex domain (Figure \ref{fig:analytic})
\begin{equation}
\mathscr{O}^{\mathbb{C}}=\left\{ z\in\mathbb{C}\:|\:\Im\left\{ z\right\} >-1\right\} .\label{eq:ea5}
\end{equation}

\item $Ext\left(F\right)$ is a singleton, and $Ext_{2}\left(F\right)=\emptyset$.
\item For the skew-Hermitian operator $D^{\left(F\right)}$ (with dense
domain in $\mathscr{H}_{F}$), we get deficiency indices $\left(0,0\right)$. 
\item In case of $p=1$, for the corresponding p.d. function $F$ in (\ref{eq:ea4}),
we have:
\begin{eqnarray*}
\Re F\left(x\right) & = & \frac{1}{1+x^{2}},\;\mbox{and}\\
\Im F\left(x\right) & = & \frac{x}{1+x^{2}},\;\mbox{for }\left|x\right|<1.
\end{eqnarray*}

\end{enumerate}

A systematic discussion or the real and the imaginary parts of positive
definite functions, and their local properties, is included in Section
\ref{sec:imgF}.

\end{example}
\begin{figure}
\includegraphics[width=0.6\textwidth]{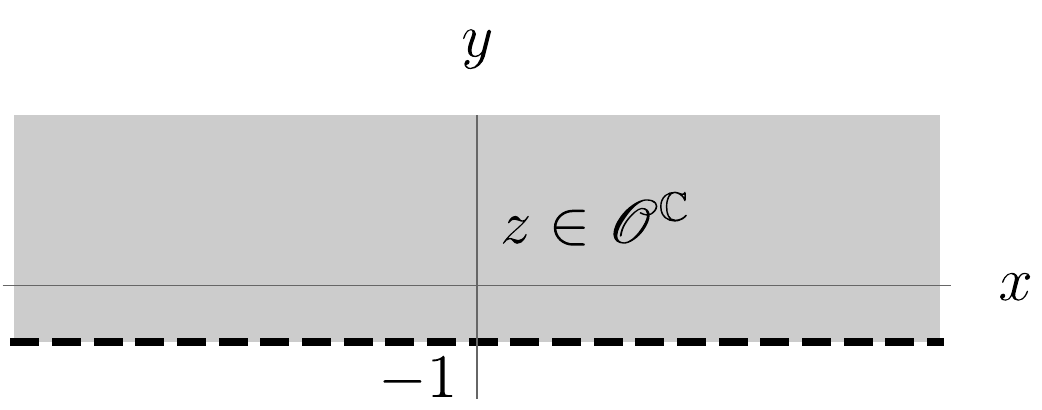}

\protect\caption{\label{fig:analytic}The complex domain $\mathscr{O}^{\mathbb{C}}:=\left\{ z\in\mathbb{C}\:|\:\Im\left\{ z\right\} >-a\right\} $
, $a=1$.}
\end{figure}

\begin{remark}
While both the Gamma distribution (\ref{eq:ea2}), $p>0$, and the
log-normal distribution (\ref{eq:ea3}) are supported on the half-line
$0\leq\lambda<\infty$, there is an important distinction connected
to moments and the Carleman condition (\ref{eq:carleman}). \index{distribution!log-normal-}\index{log-normal distribution}\index{Gamma (p) distribution}

The moments,\index{moments}
\[
m_{n}:=\int_{0}^{\infty}\lambda^{n}d\mu\left(\lambda\right)=\int_{0}^{\infty}\lambda^{n}M\left(\lambda\right)d\lambda,
\]
for the two cases are as follows:
\begin{enumerate}
\item Gamma (p): 
\begin{equation}
m_{n}=\left(p+n-1\right)\left(p+n-2\right)\cdots\left(p+1\right)p,\quad n\in\mathbb{Z}_{+}\label{eq:ea6}
\end{equation}
 
\item log-normal ($\mu_{0}$):
\begin{equation}
m_{n}=e^{\mu_{0}n+n^{2}/2},\quad n\in\mathbb{Z}_{+}.\label{eq:ea7}
\end{equation}

\end{enumerate}

Since both are Stieltjes-problems, the Carleman condition is \index{Carleman condition}
\begin{equation}
\sum_{n=1}^{\infty}m_{n}^{-\frac{1}{2n}}=\infty.\label{eq:ea8}
\end{equation}
One checks that (\ref{eq:ea8}) is satisfied for Gamma (p), see (\ref{eq:ea6});
but \emph{not} for log-normal ($\mu_{0}$) case, see (\ref{eq:ea7}).
A closer inspection shows that the generating function, 
\begin{equation}
F_{LN}\left(x\right)=\widehat{d\mu_{LN}}\left(x\right),\quad x\in\mathbb{R},\label{eq:ea9}
\end{equation}
for log-normal ($\mu_{0}$) is \emph{not} analytic in a neighborhood
of $x=0$, see (\ref{eq:ea3}). In fact, we have 
\begin{equation}
F_{LN}\left(x\right)=\frac{1}{\sqrt{2\pi}}\int_{-\infty}^{\infty}\exp\left[-\frac{\lambda^{2}}{2}+ixe^{\lambda}\right]d\lambda\label{eq:ea10}
\end{equation}
which has a complex analytic continuation: $x\mapsto z$, $\Im\left\{ z\right\} >0$,
the open upper half-plane (UH); but fails to be analytic on $x$-axis,
the boundary of $UH$. \index{analytic continuation}

Note the \emph{formal} power series, 
\begin{equation}
F_{LN}\left(x\right)\:\mbox{"="}\:\sum_{n=0}^{\infty}e^{n^{2}/2}\frac{\left(ix\right)^{n}}{n!},\label{eq:ea11}
\end{equation}
is divergent.\index{distribution!Gamma-}More precisely, the radius
of convergence in the ``power series'' on the r.h.s. (\ref{eq:ea11})
is $rad=0$. And, as a result $F_{LN}\left(x\right)$ is \emph{not}
analytic in a neighborhood of $x=0$. 

\end{remark}

\section{The Deficiency-Indices of $D^{\left(F\right)}$}

This section is devoted to the index problem for the skew-Hermitian
operator $D^{\left(F\right)}$ (with dense domain in the RKHS $\mathscr{H}_{F}$),
but not in the general case. Rather we compute $D^{\left(F\right)}$
for the particular choice of partially defined p.d. functions $F$
from the list in Table \ref{tab:F1-F6}. There are six of them in
all, numbered $F_{1}$ through $F_{6}$. The last one $F_{6}$ is
the easiest, and we begin with it. While these are only special cases
(for example, they are all continuous p.d. functions defined initially
only in a fixed finite interval, centered at $x=0$), a closer analysis
of them throws light on much more general cases, including domains
in $\mathbb{R}^{n}$, and even in non-abelian groups. But even for
the case of $G=\mathbb{R}$, we stress that applications to statistics
dictates a study of extension theory for much bigger families of interesting
p.d. functions, defined on finite intervals; see for example Section
\ref{sec:exApp}.

\renewcommand{\arraystretch}{2.5}

\begin{table}
\begin{tabular}{|>{\raggedright}p{0.3\textwidth}|>{\centering}p{0.1\textwidth}|>{\centering}p{0.5\textwidth}|}
\hline 
$F:\left(-a,a\right)\rightarrow\mathbb{C}$ & Indices & The Operator $D^{\left(F\right)}$\tabularnewline
\hline 
\hline 
$F_{1}\left(x\right)=\frac{1}{1+x^{2}}$, $\left|x\right|<1$  & $\left(0,0\right)$ & $D^{\left(F\right)}$ unbounded, skew-adjoint\tabularnewline
\hline 
$F_{2}\left(x\right)=1-\left|x\right|$, $\left|x\right|<\frac{1}{2}$ & $\left(1,1\right)$ & $D^{\left(F\right)}$ has unbounded skew-adjoint extensions\tabularnewline
\hline 
$F_{3}\left(x\right)=e^{-\left|x\right|}$, $\left|x\right|<1$ & $\left(1,1\right)$ & $D^{\left(F\right)}$ has unbounded skew-adjoint extensions\tabularnewline
\hline 
$F_{4}\left(x\right)=\left(\frac{\sin\left(x/2\right)}{x/2}\right)^{2}$,
$\left|x\right|<\frac{1}{2}$ & $\left(0,0\right)$ & $D^{\left(F\right)}$ bounded, skew-adjoint\tabularnewline
\hline 
$F_{5}\left(x\right)=e^{-x^{2}/2}$, $\left|x\right|<1$ & $\left(0,0\right)$ & $D^{\left(F\right)}$ unbounded, skew-adjoint\tabularnewline
\hline 
$F_{6}\left(x\right)=\cos x$, $\left|x\right|<\frac{\pi}{4}$ & $\left(0,0\right)$ & $D^{\left(F\right)}$ is rank-one, $\dim\left(\mathscr{H}_{F_{6}}\right)=2$ \tabularnewline
\hline 
$F_{7}\left(x\right)=\left(1-ix\right)^{-p}$, $\left|x\right|<1$ & $\left(0,0\right)$ & $D^{\left(F\right)}$ is unbounded, skew-adjoint, but semibounded;
see Corollary (\ref{cor:sp})\tabularnewline
\hline 
\end{tabular}

\protect\caption{\label{tab:F1-F6}The deficiency indices\index{deficiency indices}
of $D^{\left(F\right)}:F_{\varphi}\protect\mapsto F_{\varphi'}$,
with $F=F_{i}$, $i=1,\cdots,7$.\index{skew-Hermitian operator; also called skew-symmetric}\index{operator!rank-one-}}
\end{table}

\renewcommand{\arraystretch}{1}

We have introduced a special class of positive definite (p.d.)  extensions
using a spline technique based on a theorem by Pólya \cite{pol49}.

We then get a deficiency index-problem (see e.g., \cite{vN32a,Kre46,DS88b,AG93,Ne69})
in the RKHSs $\mathscr{H}_{F_{i}}$, $i=1,\ldots,6$, for the operator
$D^{\left(F_{i}\right)}F_{\varphi}^{\left(i\right)}=F_{\varphi'}^{\left(i\right)}$,
$\forall\varphi\in C_{c}^{\infty}\left(0,a\right)$. As shown in Tables
\ref{tab:F1-F6}, $D^{\left(F_{2}\right)}$ and $D^{\left(F_{3}\right)}$
have indices $\left(1,1\right)$, and the other four all have indices
$\left(0,0\right)$. \index{operator!skew-Hermitian} \index{deficiency indices}
\index{skew-Hermitian operator; also called skew-symmetric} \index{RKHS}

Following is an example with deficiency indices $\left(0,0\right)$
\begin{lemma}
$\mathscr{H}_{F_{6}}$ is 2-dimensional.\end{lemma}
\begin{svmultproof2}
For all $\varphi\in C_{c}^{\infty}\left(0,\frac{\pi}{4}\right)$,
we have:
\begin{align}
 & \int_{\Omega}\int_{\Omega}\overline{\varphi\left(x\right)}\varphi\left(y\right)F_{6}\left(x-y\right)dxdy\nonumber \\
= & \left|\int_{\Omega}\varphi\left(x\right)\cos xdx\right|^{2}+\left|\int_{\Omega}\varphi\left(x\right)\sin xdx\right|^{2}\nonumber \\
= & \left|\widehat{\varphi}^{\left(c\right)}\left(1\right)\right|^{2}+\left|\widehat{\varphi}^{\left(s\right)}\left(1\right)\right|^{2}\label{eq:ext-2}
\end{align}
where $\widehat{\varphi}^{\left(c\right)}=$ the cosine-transform,
and $\widehat{\varphi}^{\left(s\right)}=$ the sine-transform.
\end{svmultproof2}

So the deficiency indices only account from some of the extension
of a given positive definite function $F$ on $\Omega-\Omega$, the
Type I extensions. 

Except for $\mathscr{H}_{F_{6}}$ (2-dimensional), in all the other
six examples, $\mathscr{H}_{F_{i}}$ is infinite-dimensional. 

In the given seven examples, we have p.d. extensions to $\mathbb{R}$
of the following form, $\widehat{d\mu_{i}}\left(\cdot\right)$, $i=1,\ldots,7$,
where these measures are as follows:

\renewcommand{\arraystretch}{2.5}

\begin{table}
\begin{tabular}{|>{\raggedright}m{0.4\textwidth}|>{\raggedright}m{0.55\textwidth}|}
\hline 
$d\mu_{1}\left(\lambda\right)=\frac{1}{2}e^{-\left|\lambda\right|}d\lambda$ & Fig \ref{fig:meas} (pg. \pageref{fig:meas})\tabularnewline
\hline 
$d\mu_{2}\left(\lambda\right)=\frac{1}{2\pi}\left(\frac{\sin\left(\lambda/2\right)}{\lambda/2}\right)^{2}d\lambda$ & Fig \ref{fig:meas} (pg. \pageref{fig:meas}); Table \ref{tab:F23}
(pg. \pageref{tab:F23})\tabularnewline
\hline 
$d\mu_{3}\left(\lambda\right)=\frac{d\lambda}{\pi\left(1+\lambda^{2}\right)}$ & Fig \ref{fig:meas} (pg. \pageref{fig:meas}); Table \ref{tab:F23}
(pg. \pageref{tab:F23})\tabularnewline
\hline 
$d\mu_{4}\left(\lambda\right)=\chi_{\left(-1,1\right)}\left(\lambda\right)\left(1-\left|\lambda\right|\right)d\lambda$  & Fig \ref{fig:meas} (pg. \pageref{fig:meas}); Ex \ref{exa:cp} (pg.
\pageref{exa:cp}); Ex \ref{eg:F4} (pg. \pageref{eg:F4})\tabularnewline
\hline 
$d\mu_{5}\left(\lambda\right)=\frac{1}{\sqrt{2\pi}}e^{-\lambda^{2}/2}d\lambda$ & Fig \ref{fig:meas} (pg. \pageref{fig:meas}); Ex \ref{exa:bm} (pg.
\pageref{exa:bm})\tabularnewline
\hline 
$d\mu_{6}\left(\lambda\right)=\frac{1}{2}\left(\delta_{1}+\delta_{-1}\right)$ & Fig \ref{fig:meas} (pg. \pageref{fig:meas}); Lemma \ref{lem:cos}
(pg. \pageref{lem:cos}); Thm \ref{thm:F2-bd} (pg. \pageref{thm:F2-bd});
Thm \ref{thm:F3bd} (pg. \ref{thm:F3bd}); Rem \ref{rem:cos} (pg.
\pageref{rem:cos}) \tabularnewline
\hline 
$d\mu_{7}\left(\lambda\right)=\frac{\lambda^{p}}{\Gamma\left(p\right)}e^{-\lambda}d\lambda$,
$\lambda\geq0$ & Fig \ref{fig:meas} (pg. \pageref{fig:meas}); Fig \ref{fig:gam(p)}
(pg. \pageref{fig:gam(p)})\tabularnewline
\hline 
\end{tabular}

\protect\caption{\label{tab:Table-3}The canonical isometric embeddings: $\mathscr{H}_{F_{i}}\protect\hookrightarrow L^{2}\left(\mathbb{R},d\mu_{i}\right)$,
$i=1,\ldots,7$. In each case, we have $\mu_{i}\in Ext\left(F_{i}\right)$,
and, by Corollary \ref{cor:lcg-isom} therefore, the corresponding
isometric embeddings $T_{\mu_{i}}$ mapping into the respective $L^{2}\left(\mu_{i}\right)$-Hilbert
spaces. \index{purely atomic}\index{Cauchy distribution}\index{atom}\index{measure!Lebesgue}}
\end{table}

\renewcommand{\arraystretch}{1}

See also Table \ref{tab:meas} and Figure \ref{fig:meas} below.
\begin{corollary}
For $i=1,\cdots,6$, we get isometries $T^{\left(i\right)}:\mathscr{H}_{F_{i}}\rightarrow L^{2}\left(\mathbb{R},\mu_{i}\right)$,
determined by 
\[
T^{\left(i\right)}(F_{\varphi}^{\left(i\right)})=\widehat{\varphi},\quad\forall\varphi\in C_{c}^{\infty}\left(\Omega_{i}\right);\;\mbox{where}
\]
\[
\Vert F_{\varphi}^{\left(i\right)}\Vert_{\mathscr{H}_{F_{i}}}^{2}=\left\Vert \widehat{\varphi}\right\Vert _{L^{2}\left(\mu_{i}\right)}^{2}=\int_{\mathbb{R}}\left|\widehat{\varphi}\right|^{2}d\mu_{i}.
\]
Note that, except for $i=6$, $T^{\left(i\right)}$ is only isometric
into $L^{2}\left(\mu_{i}\right)$. 

The adjoint operator, $(T^{\left(i\right)})^{*}:L^{2}\left(\mathbb{R},\mu_{i}\right)\rightarrow\mathscr{H}_{F_{i}}$,
is given by 
\[
(T^{\left(i\right)})^{*}f=\chi_{\overline{\Omega_{i}}}\left(fd\mu_{i}\right)^{\vee},\quad\forall f\in L^{2}\left(\mathbb{R},\mu_{i}\right).
\]
\end{corollary}
\begin{svmultproof2}
We refer to Corollary \ref{cor:lcg-isom}. \end{svmultproof2}

\begin{example}[An infinite-dimensional example as a version of $F_{6}$]
\label{exa:aseries} Fix $p$, $0<p<1$, and set 
\[
F_{p}\left(x\right):=\prod_{n=1}^{\infty}\cos\left(2\pi p^{n}x\right).
\]
Then $F_{p}=\widehat{d\mu_{p}}$, where $\mu_{p}$ is the Bernoulli
measure, and so $F_{p}$ is a continuous positive definite function
on $\mathbb{R}$. Note that some of those measures $\mu_{p}$ are
fractal measures.\index{measure!Bernoulli}\index{measure!fractal}\index{infinite Cartesian product}

For fixed $p\in\left(0,1\right)$, the measure $\mu_{p}$ is the law
(i.e., distribution) of the following random power series
\[
X_{p}\left(\omega\right):=\sum_{n=1}^{\infty}\left(\pm\right)p^{n},
\]
where $\omega\in\prod{}_{1}^{\infty}\left\{ \pm1\right\} $ (= infinite
Cartesian product) and where the distribution of each factor is $\left\{ -\frac{1}{2},\frac{1}{2}\right\} $,
and statically independent. For relevant references on random power
series, see \cite{Neu13,Lit99}.\index{Hilbert space}
\end{example}

\paragraph{\textbf{Pólya-extensions}}

The extensions we generate with the application of Pólya's theorem
are realized in a bigger Hilbert space. The deficiency indices of
the skew-Hermitian operator $D^{\left(F\right)}$ are computed w.r.t.
the RKHS $\mathscr{H}_{F}$, i.e., for the ``small'' p.d. function
$F:\Omega-\Omega\rightarrow\mathbb{C}$. \index{RKHS}\index{deficiency indices}
\begin{example}
$F_{2}\left(x\right)=1-\left|x\right|$, $\left|x\right|<\frac{1}{2}$,
has the obvious p.d. extension to $\mathbb{R}$, i.e., $\left(1-\left|x\right|\right)\chi_{\left[-1,1\right]}\left(x\right)$,
which corresponds to the measure $\mu_{2}$ from Table \ref{tab:Table-3}.
It also has other p.d. extensions, e.g., the Pólya extension in Figure
\ref{fig:spline2}. All these extensions are of Type II, realized
in $\infty$-dimensional dilation-Hilbert spaces.
\end{example}
We must make a distinction between two classes of p.d. extensions
of $F:\Omega-\Omega\rightarrow\mathbb{C}$ to continuous p.d. functions
on $\mathbb{R}$. \index{unitary representation}\index{representation!unitary-}\index{Pólya extensions}\index{extensions!Pólya-}\index{operator!Pólya's-}

\textbf{Case 1.} There exists a unitary representation $U\left(t\right):\mathscr{H}_{F}\rightarrow\mathscr{H}_{F}$
such that
\begin{equation}
F\left(t\right)=\left\langle \xi_{0},U\left(t\right)\xi_{0}\right\rangle _{\mathscr{H}_{F}},\:t\in\Omega-\Omega\label{eq:ext-3}
\end{equation}

\textbf{Case 2.} (e.g., Pólya extension) There exist a dilation-Hilbert
space $\mathscr{K}$, and an isometry\index{isometry} $J:\mathscr{H}_{F}\rightarrow\mathscr{K}$,
and a unitary representation $U\left(t\right)$ of $\mathbb{R}$ acting
in $\mathscr{K}$, such that 
\begin{equation}
F\left(t\right)=\left\langle J\xi_{0},U\left(t\right)J\xi_{0}\right\rangle _{\mathscr{K}},\:t\in\Omega-\Omega\label{eq:ext-4}
\end{equation}

In both cases, $\xi_{0}=F\left(0-\cdot\right)\in\mathscr{H}_{F}$.

In case 1, the unitary representation is realized in $\mathscr{H}_{\left(F,\Omega-\Omega\right)}$,
while, in case 2, the unitary representation $U\left(t\right)$ lives
in the expanded Hilbert space $\mathscr{K}$.

Note that the r.h.s. in both (\ref{eq:ext-3}) and (\ref{eq:ext-4})
is defined for all $t\in\mathbb{R}$.\index{GNS}
\begin{lemma}
Let $F_{ex}$ be one of the Pólya extensions if any. Then by the Galfand-Naimark-Segal
(GNS) construction applied to $F_{ext}:\mathbb{R}\rightarrow\mathbb{R}$,
there is a Hilbert space $\mathscr{K}$ and a vector $v_{0}\in\mathscr{K}$
and a unitary representation $\left\{ U\left(t\right)\right\} _{t\in\mathbb{R}};$
$U\left(t\right):\mathscr{K}\rightarrow\mathscr{K}$, such that\index{representation!GNS-}
\begin{equation}
F_{ex}\left(t\right)=\left\langle v_{0},U\left(t\right)v_{0}\right\rangle _{\mathscr{K}},\:\forall t\in\mathbb{R}.\label{eq:ext-5}
\end{equation}

Setting $J:\mathscr{H}_{F}\rightarrow\mathscr{K}$, $J\xi_{0}=v_{0}$,
then $J$ defines (by extension) an isometry such that 
\begin{equation}
U\left(t\right)J\xi_{0}=J\left(\mbox{local translation in }\Omega\right)\label{eq:ext-6}
\end{equation}
holds locally (i.e., for $t$ sufficiently close to $0$.)

Moreover, the function 
\begin{equation}
\mathbb{R}\ni t\mapsto U\left(t\right)J\xi_{0}=U\left(t\right)v_{0}\label{eq:ext-7}
\end{equation}
is compactly supported.\end{lemma}
\begin{svmultproof2}
The existence of $\mathscr{K}$, $v_{0}$, and $\left\{ U\left(t\right)\right\} _{t\in\mathbb{R}}$
follows from the GNS-construction.

The conclusions in (\ref{eq:ext-6}) and (\ref{eq:ext-7}) follow
from the given data, i.e., $F:\Omega-\Omega\rightarrow\mathbb{R}$,
and the fact that $F_{ex}$ is a spline-extension, i.e., it is of
compact support; but by (\ref{eq:ext-5}), this means that (\ref{eq:ext-7})
is also compactly supported.
\end{svmultproof2}

Example \ref{eg:F2} gives a p.d. $F$ in $\left(-\tfrac{1}{2},\tfrac{1}{2}\right)$
with $D^{(F)}$ of index $(1,1)$ and explicit measures in $Ext_{1}(F)$
and in $Ext_{2}(F).$

We have the following:

\textbf{Deficiency $\left(0,0\right)$:} The p.d. extension of Type
I\index{Type I} is unique; see (\ref{eq:ext-3}); but there may still
be p.d. extensions of Type II\index{Type II}; see (\ref{eq:ext-4}).

\textbf{Deficiency $\left(1,1\right)$: }This is a one-parameter family
of extensions of Type I; and some more p.d. extensions are Type II.

So we now divide 
\[
Ext\left(F\right)=\left\{ \mu\in\mbox{Prob}\left(\mathbb{R}\right)\:\big|\:\widehat{d\mu}\mbox{ is an extension of }F\right\} 
\]
up in subsets
\[
Ext\left(F\right)=Ext_{type1}\left(F\right)\cup Ext_{type2}\left(F\right);
\]
where $Ext_{2}\left(F\right)$ corresponds to the Pólya extensions.\index{operator!unbounded}\index{operator!skew-adjoint-}

\renewcommand{\arraystretch}{2}

\begin{table}
\begin{tabular}{cc}
\includegraphics[width=0.45\textwidth]{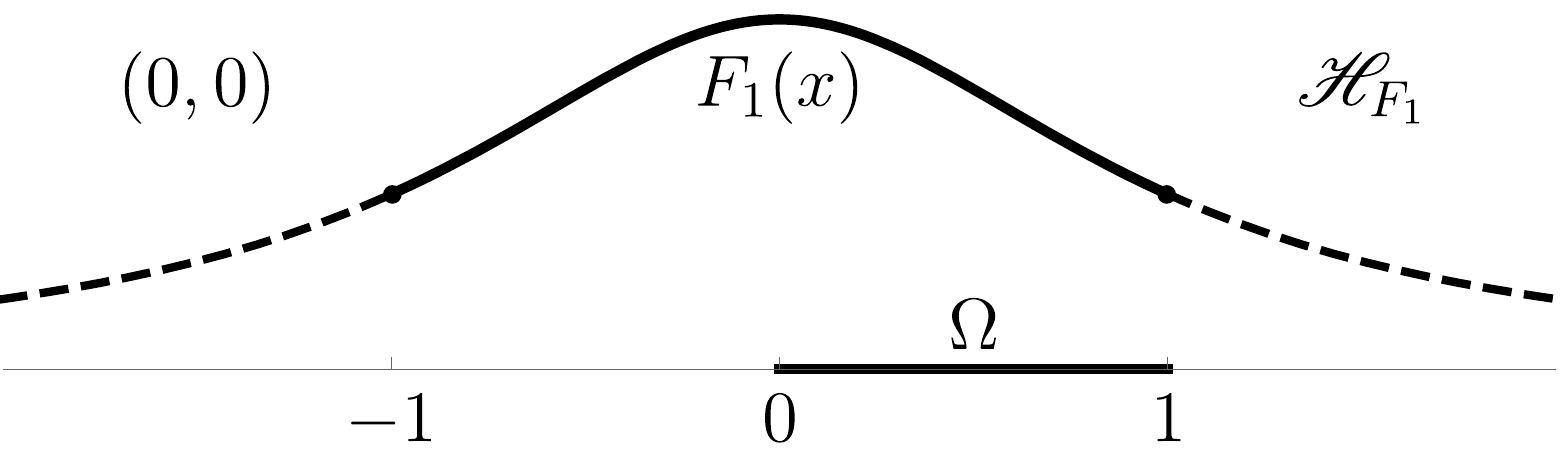} & \includegraphics[width=0.45\textwidth]{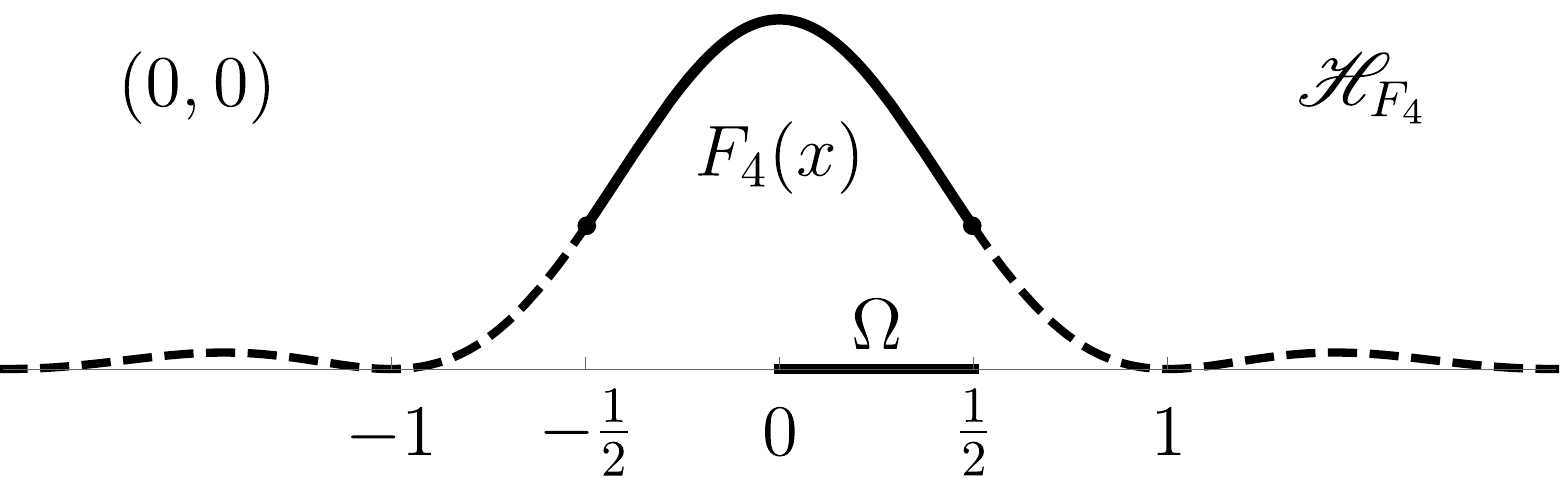}\tabularnewline
\includegraphics[width=0.45\textwidth]{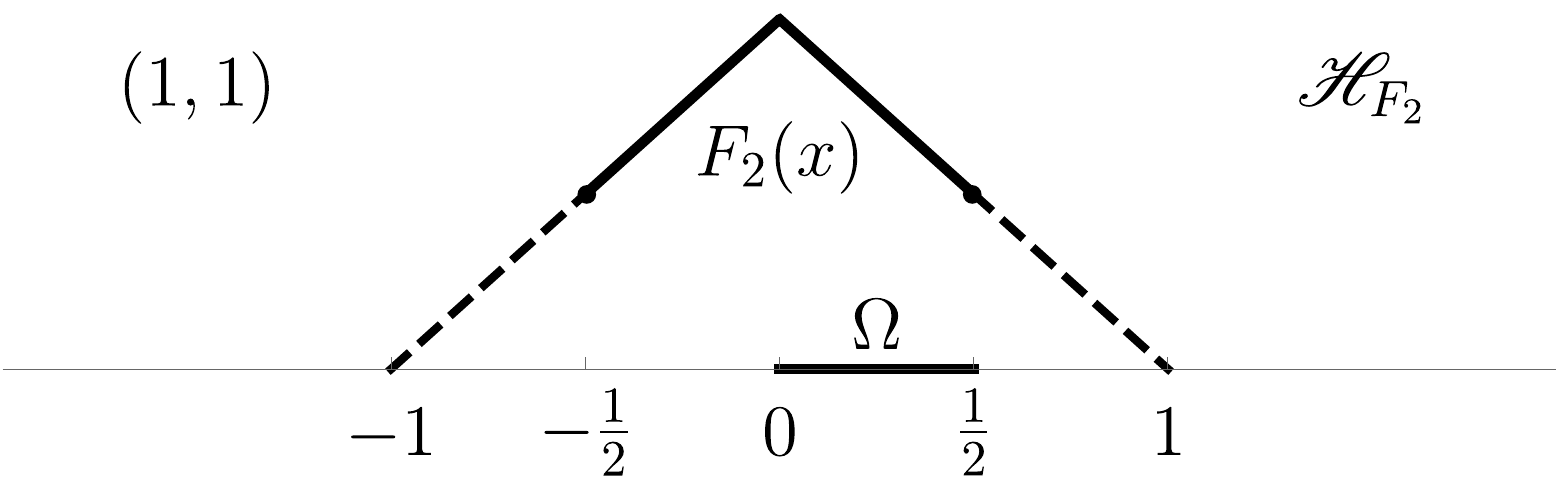} & \includegraphics[width=0.45\textwidth]{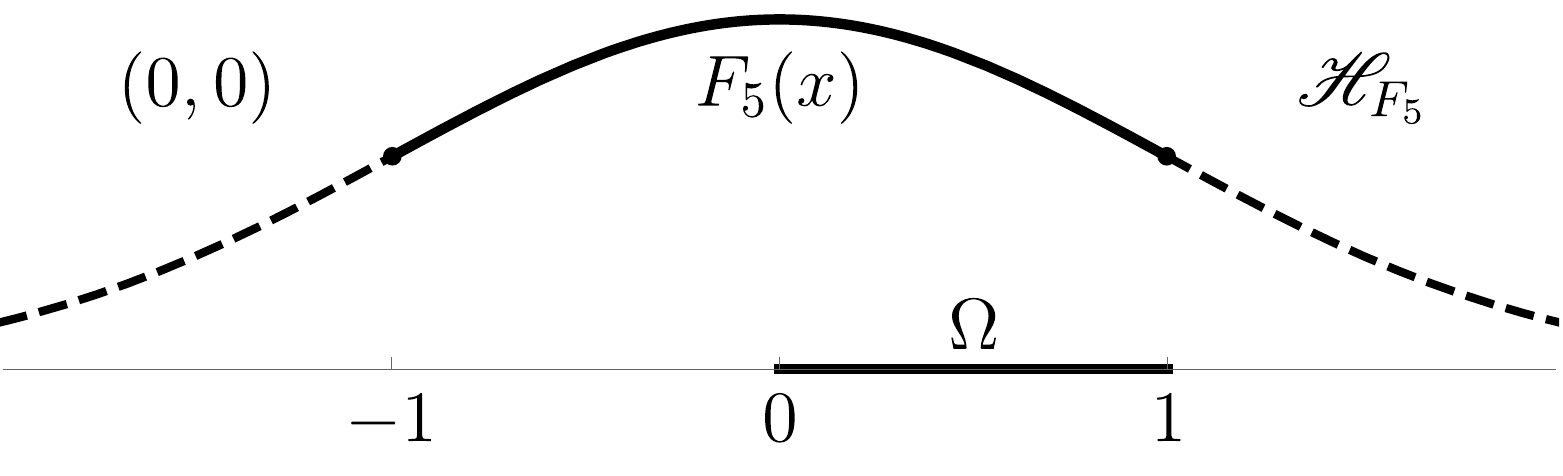}\tabularnewline
\includegraphics[width=0.45\textwidth]{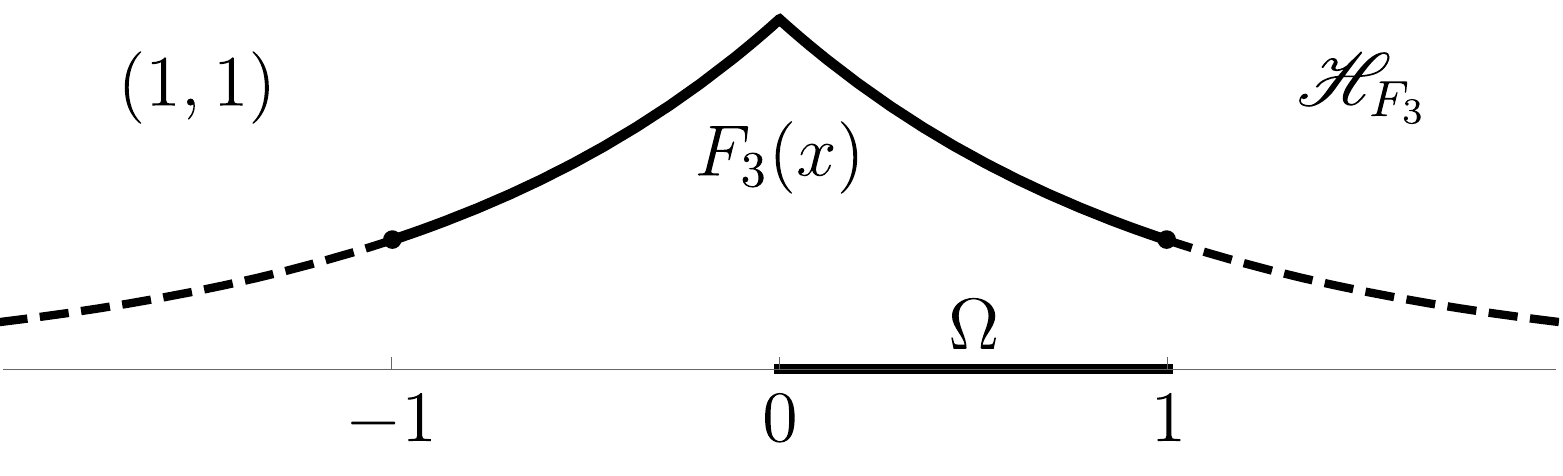} & \includegraphics[width=0.45\textwidth]{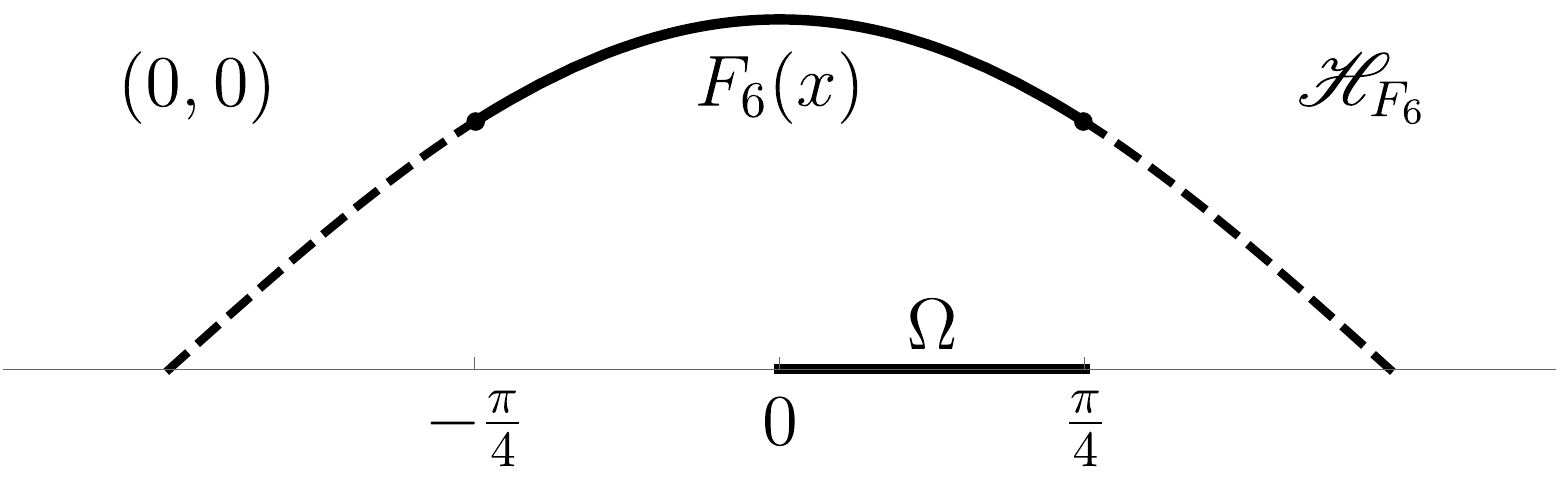}\tabularnewline
\end{tabular}

\protect\caption{\label{tab:Table-2}Type II extensions. Six cases of p.d. continuous
functions defined on a finite interval. Additional properties of these
functions will be outlined in Section \ref{sec:R^1}. For each of
the six graphs (of p.d. function), the numbers $(0,0)$ or $(1,1)$
indicate the deficiency indices. The case $(0,0)$ means \textquotedblleft unique
skew-adjoint extension\textquotedblright ; while $(1,1)$ means that
there is a one-parameter family of distinct skew-adjoint extensions.
For each of the cases of the locally defined p.d. functions $F_{1}$,
$F_{4}$, $F_{5}$, and $F_{6}$, we know that sets $Ext(F)$ are
singletons. This follows from Theorem \ref{thm:Ext2}.}
\end{table}

\renewcommand{\arraystretch}{1}

Return to a continuous p.d. function $F:\left(-a,a\right)\rightarrow\mathbb{C}$,
we take for the RKHS $\mathscr{H}_{F}$, and the skew-Hermitian operator
\[
D\left(F_{\varphi}\right)=F_{\varphi'},\:\varphi'=\frac{d\varphi}{dx}
\]
 If $D\subseteq A$, $A^{*}=-A$ in $\mathscr{H}_{F}$ then there
exists an isometry $J:\mathscr{H}_{F}\rightarrow L^{2}\left(\mathbb{R},\mu\right)$,
where $d\mu\left(\cdot\right)=\left\Vert P_{U}\left(\cdot\right)\xi_{0}\right\Vert ^{2}$,
\[
U_{A}\left(t\right)=e^{tA}=\int_{\mathbb{R}}e^{it\lambda}P_{U}\left(d\lambda\right),
\]
$\xi_{0}=F\left(\cdot-0\right)\in\mathscr{H}_{F}$, $J\xi_{0}=1\in L^{2}\left(\mu\right)$.

\section{\label{sec:F3-Mercer}The Example \ref{eg:F3}, Green's function,
and an $\mathscr{H}_{F}$-ONB}

Here, we study Example \ref{eg:F3} in more depth, and we compute
the spectral date of the corresponding Mercer operator. Recall that
\index{Green's function} 
\begin{equation}
F\left(x\right):=\begin{cases}
e^{-\left|x\right|} & \left|x\right|<1\\
e^{-1}\left(2-\left|x\right|\right) & 1\leq\left|x\right|<2\\
0 & \left|x\right|\geq2
\end{cases}\label{eq:et1}
\end{equation}
See \figref{PF} below.\index{operator!Mercer} 

\begin{figure}[H]
\includegraphics[width=0.7\textwidth]{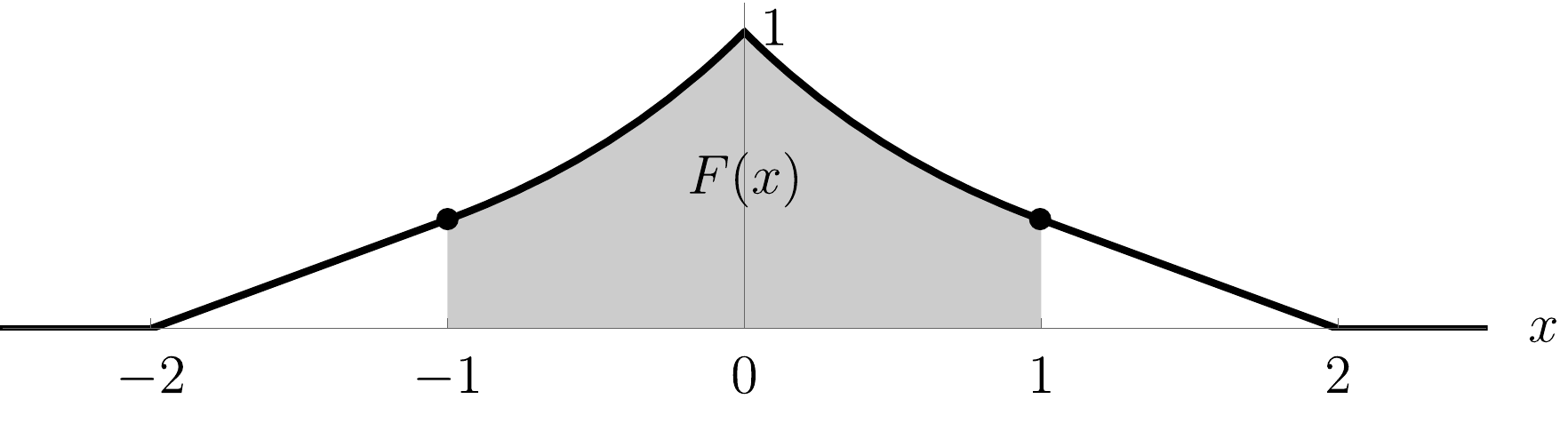}

\protect\caption{\label{fig:PF} A Pólya (spline) extension of $F_{3}\left(x\right)=e^{-\left|x\right|}$;
$\Omega=\left(-1,1\right)$.}
\end{figure}

\begin{proposition}
The RKHS $\mathscr{H}_{F}$ for the spline extension $F$ in (\ref{eq:et1})
has the following ONB:
\[
\left\{ \left(\frac{2}{1+\left(\frac{\pi n}{2}\right)^{2}}\right)^{\frac{1}{2}}\sin\left(\frac{n\pi x}{2}\right)\right\} _{n\in\mathbb{Z}_{+}.}
\]
\end{proposition}
\begin{svmultproof2}
The distributional derivative of $F$ satisfies\index{derivative!distributional-}
\begin{eqnarray*}
F'' & = & F-2\delta_{0}+e^{-1}\left(\delta_{2}+\delta_{-2}\right)\\
 & \Updownarrow\\
\left(1-\triangle\right)F & = & 2\delta_{0}-e^{-1}\left(\delta_{2}+\delta_{-2}\right)
\end{eqnarray*}
This can be verified directly, using Schwartz' theory of distributions.
See \lemref{distd} and \cite{Tre06}. \index{Schwartz!distribution}\index{distribution!-derivation}\index{extensions!spline-}

Thus, for 
\[
F_{x}\left(\cdot\right):=F\left(x-\cdot\right)
\]
we have the translation 
\begin{equation}
\left(1-\triangle\right)F_{x}=2\delta_{x}-e^{-1}\left(\delta_{x-2}+\delta_{x+2}\right)\label{eq:FP-1}
\end{equation}

Consider the Mercer operator (see Section \ref{sec:mercer}): 
\[
L^{2}\left(0,2\right)\ni g\longmapsto\int_{0}^{2}F_{x}\left(y\right)g\left(y\right)dy=\left\langle F_{x},g\right\rangle .
\]
Suppose $\lambda\in\mathbb{R}$, and 
\begin{equation}
\left\langle F_{x},g\right\rangle =\lambda g\left(x\right).\label{eq:FP-2}
\end{equation}
Applying $\left(1-\triangle\right)$ to both sides of (\ref{eq:FP-2}),
we get
\begin{equation}
\left\langle \left(1-\triangle\right)F_{x},g\right\rangle =\lambda\left(g\left(x\right)-g''\left(x\right)\right),\;0<x<2.\label{eq:FP-3}
\end{equation}

By (\ref{eq:FP-1}), we also have 
\begin{equation}
\left\langle \left(1-\triangle\right)F_{x},g\right\rangle =2g\left(x\right),\;0<x<2;\label{eq:FP-4}
\end{equation}
using the fact the two Dirac masses in (\ref{eq:FP-1}), i.e., $\delta_{x\pm2}$,
are supported outside the open interval $\left(0,2\right)$. 

Therefore, combining (\ref{eq:FP-3})-(\ref{eq:FP-4}), we have 
\[
g''=\frac{\lambda-2}{\lambda}g,\;\forall x\in\left(0,2\right).
\]
By Mercer's theorem, $0<\lambda<2$. Setting \index{Theorem!Mercer's-}
\[
k:=\sqrt{\frac{2-\lambda}{\lambda}}\;\left(\Leftrightarrow\lambda=\frac{2}{1+k^{2}}\right)
\]
we have 
\[
g''=-k^{2}g,\;\forall x\in\left(0,2\right).
\]

\uline{Boundary conditions:}

In (\ref{eq:FP-3}), set $x=0$, and $x=2$, we get 
\begin{eqnarray}
2g\left(0\right)-e^{-1}g\left(2\right) & = & \lambda\left(g\left(0\right)-g''\left(0\right)\right)\label{eq:FP-bd-1}\\
2g\left(2\right)-e^{-1}g\left(0\right) & = & \lambda\left(g\left(2\right)-g''\left(2\right)\right)\label{eq:FP-bd-2}
\end{eqnarray}
Now, assume 
\[
g\left(x\right)=Ae^{ikx}+Be^{-ikx};
\]
where 
\begin{eqnarray*}
g\left(0\right) & = & A+B\\
g\left(2\right) & = & Ae^{ik2}+Be^{-ik2}\\
g''\left(0\right) & = & -k^{2}\left(A+B\right)\\
g''\left(2\right) & = & -k^{2}\left(Ae^{ik2}+Be^{-ik2}\right).
\end{eqnarray*}
Therefore, for (\ref{eq:FP-bd-1}), we have 
\begin{eqnarray*}
2\left(A+B\right)-e^{-1}\left(Ae^{ik2}+Be^{-ik2}\right) & = & \lambda\left(1+k^{2}\right)\left(A+B\right)\\
 & = & 2\left(A+B\right)
\end{eqnarray*}
i.e.,
\begin{equation}
Ae^{ik2}+Be^{-ik2}=0.\label{eq:FP-bd-3}
\end{equation}
Now, from (\ref{eq:FP-bd-2}) and using (\ref{eq:FP-bd-3}), we have
\begin{equation}
A=-B.\label{eq:FP-bd-4}
\end{equation}
Combining (\ref{eq:FP-bd-3})-(\ref{eq:FP-bd-4}), we conclude that
\[
\sin2k=0\Longleftrightarrow k=\frac{\pi n}{2},\;n\in\mathbb{Z};
\]
i.e., 
\begin{equation}
\lambda_{n}:=\frac{2}{1+\left(\dfrac{n\pi}{2}\right)^{2}},\;n\in\mathbb{N}.\label{eq:FP-5}
\end{equation}
The associated ONB in $L^{2}\left(0,2\right)$ is 
\begin{equation}
\xi_{n}\left(x\right)=\sin\left(\frac{n\pi x}{2}\right),\;n\in\mathbb{N}.\label{eq:FP-6}
\end{equation}
And the corresponding ONB in $\mathscr{H}_{F}$ consists of the functions
$\left\{ \sqrt{\lambda_{n}}\xi_{n}\right\} _{n\in\mathbb{N}}$, i.e.,
\begin{equation}
\left\{ \frac{\sqrt{2}}{\bigl(1+\bigl(\frac{n\pi}{2}\bigr)^{2}\bigr)^{1/2}}\sin\left(\frac{n\pi x}{2}\right)\right\} _{n\in\mathbb{N}}\label{eq:FP-7}
\end{equation}
which is the desired conclusion.
\end{svmultproof2}

\chapter{\label{chap:Ext1}Spectral Theory for Mercer Operators, and Implications
for $Ext\left(F\right)$}

Given a continuous positive definite (p.d.) function $F$ on the open
interval $\left(-1,1\right)$, we are concerned with the set $Ext\left(F\right)$
of its extensions to p.d. functions defined on all of $\mathbb{R}$,
as well as a certain subset $Ext_{1}\left(F\right)$ of $Ext\left(F\right)$.
Since every such p.d. extension of $F$ is a Bochner transform of
a unique positive and finite Borel measure $\mu$ on $\mathbb{R}$,
i.e., $\widehat{d\mu}\left(x\right)$, $x\in\mathbb{R}$ and $\mu\in\mathscr{M}_{+}\left(\mathbb{R}\right)$,
we will speak of $Ext\left(F\right)$ as a subset of $\mathscr{M}_{+}\left(\mathbb{R}\right)$.
The purpose of this chapter is to gain insight into the nature and
properties of $Ext_{1}\left(F\right)$.\index{transform!Bochner-}

In Section \ref{sec:mercer}, we study the Mercer operator $T_{F}$
associated with $F$, which is a certain trace class integral operator.
We use it (see Section \ref{sec:shannon}) to identify a natural Bessel
frame in the RKHS $\mathscr{H}_{F}$. We further introduce a notion
of Shannon sampling of finite Borel measures on $\mathbb{R}$, then
use this in Corollary \ref{cor:shan} to give a necessary and sufficient
condition: $\mu\in\mathscr{M}_{+}\left(\mathbb{R}\right)$ is in $Ext\left(F\right)$
if and only if the Shannon sampling of $\mu$ recovers the p.d. function
$F$ on the interval $\left(-1,1\right)$.

\index{Bessel frame}\index{Bochner transform}\index{trace}\index{operator!integral-}\index{frame!Bessel-}
\begin{definition}
Let $F$ be a continuous positive definite (p.d.) function on a finite
interval $\left(-a,a\right)\subset\mathbb{R}$. By a Mercer operator,
we mean the integral operator $T_{F}:L^{2}\left(0,a\right)\rightarrow L^{2}\left(0,a\right)$,
given by \index{operator!Mercer}
\begin{equation}
\left(T_{F}\varphi\right)\left(x\right):=\int_{0}^{a}\varphi\left(y\right)F\left(x-y\right)dy,\quad\varphi\in L^{2}\left(0,a\right),\:x\in\left(0,a\right).\label{eq:mer-1}
\end{equation}
\end{definition}
\begin{lemma}
\label{lem:mer1}Let $T_{F}$ be the Mercer operator in (\ref{eq:mer-1}),
and $\mathscr{H}_{F}$ be the RKHS of $F$. 
\begin{enumerate}
\item Then there is a sequence $\left(\lambda_{n}\right)_{n\in\mathbb{N}}$,
$\lambda_{n}>0$, $\sum_{n\in\mathbb{N}}\lambda_{n}<\infty$; and
a system of orthogonal functions\index{orthogonal} $\left\{ \xi_{n}\right\} \subset L^{2}\left(0,a\right)\cap\mathscr{H}_{F}$,
such that
\begin{equation}
F\left(x-y\right)=\sum_{n\in\mathbb{N}}\lambda_{n}\xi_{n}\left(x\right)\overline{\xi_{n}\left(y\right)},\mbox{ and}\label{eq:mer-3}
\end{equation}
\begin{equation}
\int_{0}^{a}\overline{\xi_{n}\left(x\right)}\xi_{m}\left(x\right)dx=\delta_{n,m},\quad n,m\in\mathbb{N}.\label{eq:mer-4}
\end{equation}

\item In particular, $T_{F}$ is trace class in $L^{2}\left(0,a\right)$.
If $F\left(0\right)=1$, then 
\begin{equation}
trace\left(T_{F}\right)=a.\label{eq:mer-2}
\end{equation}

\end{enumerate}
\end{lemma}
\begin{svmultproof2}
This is an application of Mercer's theorem \cite{LP89,FR42,FM13}.
Note that we must check that $F$, initially defined on $\left(-a,a\right)$,
extends uniquely by limit to a continuous p.d. function on the closed
interval $\left[-a,a\right]$. This is easy to verify; also see Lemma
\ref{lem:F-bd}.
\end{svmultproof2}

\index{Theorem!Mercer's-}\index{operator!integral-}
\begin{corollary}
Let $F:\left(-a,a\right)\rightarrow\mathbb{C}$ be a continuous p.d.
function, assume that $F\left(0\right)=1$. Let $D^{\left(F\right)}$
be the skew-Hermitian\index{operator!skew-Hermitian} operator in
$\mathscr{H}_{F}$. Let $z\in\mathbb{C}\backslash\left\{ 0\right\} $,
and $DEF_{F}\left(z\right)\subset\mathscr{H}_{F}$ be the corresponding
deficiency space. 

Let $\left\{ \xi_{n}\right\} _{n\in\mathbb{N}}\subset\mathscr{H}_{F}\cap L^{2}\left(\Omega\right)$,
$\Omega:=\left(0,a\right)$; and $\left\{ \lambda_{n}\right\} _{n\in\mathbb{N}}$
s.t. $\lambda_{n}>0$, and $\sum_{n=1}^{\infty}\lambda_{n}=a$, be
one Mercer--system as in Lemma \ref{lem:mer1}.

Then $DEF_{F}\left(z\right)\neq0$ if and only if 
\begin{equation}
\sum_{n=1}^{\infty}\frac{1}{\lambda_{n}}\left|\int_{0}^{a}\overline{\xi_{n}\left(x\right)}e^{zx}dx\right|^{2}<\infty.\label{eq:mer-1-1}
\end{equation}
\end{corollary}
\begin{svmultproof2}
We saw in Lemma \ref{lem:mer1} that if $\left\{ \xi_{n}\right\} _{n\in\mathbb{N}}$,
$\left\{ \lambda_{n}\right\} _{n\in\mathbb{N}}$, is a Mercer system
(i.e., the spectral data for the Mercer operator $T_{F}$ in $L^{2}\left(0,a\right)$),
then $\xi_{n}\in\mathscr{H}_{F}\cap L^{2}\left(\Omega\right)$; and
$\left(\sqrt{\lambda_{n}}\xi_{n}\left(\cdot\right)\right)_{n\in\mathbb{N}}$
is an ONB in $\mathscr{H}_{F}$.

But (\ref{eq:mer-1-1}) in the corollary is merely stating that the
function $e_{z}\left(x\right):=e^{zx}$ has a finite $l^{2}$-expansion
relative to this ONB. The rest is an immediate application of Parseval's
identity (w.r.t. this ONB.) 
\end{svmultproof2}

\index{Theorem!Mercer's-} \index{Parseval's identity}
\begin{remark}
The conclusion in the corollary applies more generally: It shows that
a continuous function $f$ on $\left[0,a\right]$ is in $\mathscr{H}_{F}$
iff 
\[
\sum_{n=1}^{\infty}\frac{1}{\lambda_{n}}\left|\int_{0}^{a}\overline{\xi_{n}\left(x\right)}f\left(x\right)dx\right|^{2}<\infty.
\]

\end{remark}

\section{\label{sec:mercer}Groups, Boundary Representations, and Renormalization}

In this section, we study a boundary representation for the operator
$T_{F}$ (Corollary \ref{cor:Liebd}), and an associated renormalization
for the Hilbert space $\mathscr{H}_{F}$ in Corollary \ref{cor:mer2},
in Theorems \ref{thm:mer2} and \ref{thm:Liebd}, with corollaries.

\begin{definition}
\label{def:LMer}Let $G$ be a Lie group, and let $\Omega$ be a subset
of $G$, satisfying the following:\index{trace}
\begin{enumerate}[label=B\arabic{enumi}.,ref=B\arabic{enumi}]
\item \label{enu:mer1} $\Omega\neq\emptyset$;
\item \label{enu:mer2}$\Omega$ is open and connected;
\item \label{enu:mer3}the closure $\overline{\Omega}$ is compact;
\item \label{enu:mer4}the boundary of $\overline{\Omega}$ has Haar measure
zero.
\end{enumerate}

Then, for any continuous positive definition function $F:\Omega^{-1}\cdot\Omega\rightarrow\mathbb{C}$,
a trace class Mercer operator is defined by 
\begin{equation}
\left(T_{F}\varphi\right)\left(x\right)=\int_{\Omega}\varphi\left(y\right)F\left(y^{-1}x\right)dy,\quad\varphi\in L^{2}\left(\Omega\right)\label{eq:mer0}
\end{equation}
where $dy$ is the restriction of Haar measure on $G$ to $\Omega$,
or equivalently to $\overline{\Omega}$.

\end{definition}
\index{measure!Haar}

Note that with Lemma \ref{lem:F-bd} and assumption \ref{enu:mer3},
we conclude that 
\begin{equation}
\mathscr{H}_{F}\subseteq C\left(\overline{\Omega}\right)\subset L^{2}\left(\Omega\right).\label{eq:m-3-1}
\end{equation}
It is natural to ask: ``What is the orthogonal complement of $\mathscr{H}_{F}$
in the larger Hilbert space $L^{2}\left(\Omega\right)$?'' The answer
is given in Corollary \ref{cor:HFc}. We will need the following:
\begin{lemma}
There is a finite constant $C_{1}$ such that 
\begin{equation}
\left\Vert \xi\right\Vert _{L^{2}\left(\Omega\right)}\leq C_{1}\left\Vert \xi\right\Vert _{\mathscr{H}_{F}},\quad\forall\xi\in\mathscr{H}_{F}.\label{eq:m-3-3}
\end{equation}
\end{lemma}
\begin{svmultproof2}
Since $\mathscr{H}_{F}\subset L^{2}\left(\Omega\right)$ by (\ref{eq:m-3-1}),
the inclusion mapping, $\mathscr{H}_{F}\rightarrow L^{2}\left(\Omega\right)$,
is closed and therefore bounded; by the Closed-Graph Theorem \cite{DS88b};
and so the estimate (\ref{eq:m-3-3}) follows. (See also Theorem \ref{thm:mer4}
below for an explicit bound $C_{1}$.)\index{Theorem!Closed-Graph-}\end{svmultproof2}

\begin{corollary}
\label{cor:HFc}Let $T_{F}$ be the Mercer operator in (\ref{eq:mer0}),
then 
\begin{equation}
L^{2}\left(\Omega\right)\ominus\mathscr{H}_{F}=ker\left(T_{F}\right);\label{eq:m-3-2}
\end{equation}
and as an operator in $L^{2}\left(\Omega\right)$, $T_{F}$ takes
the form:
\end{corollary}
\[ 
T_F =\quad \begin{blockarray}{ccc}
\scriptstyle{\mathscr{H}_{F}} & \scriptstyle{ker(T_{F})} &   \\
\begin{block}{(cc)c}
  T_{F} & 0 &  \quad\scriptstyle{\mathscr{H}_{F}} \\
  0 & 0 &  \quad\scriptstyle{ker(T_{F})} \\ 
\end{block} 
\end{blockarray} 
\]
\begin{svmultproof2}
Let $\xi\in\mathscr{H}_{F}$, and pick $\varphi_{n}\in C_{c}\left(\Omega\right)$
such that 
\begin{equation}
\left\Vert \xi-F_{\varphi_{n}}\right\Vert _{\mathscr{H}_{F}}\rightarrow0,\quad n\rightarrow\infty.\label{eq:m-3-4}
\end{equation}
Since $F_{\varphi_{n}}=T_{F}\varphi_{n}$, and so we get 
\begin{equation}
\left\Vert \xi-T_{F}\varphi_{n}\right\Vert _{L^{2}\left(\Omega\right)}\underset{\left(\ref{eq:m-3-3}\right)}{\leq}C_{1}\left\Vert \xi-F_{\varphi_{n}}\right\Vert _{\mathscr{H}_{F}}\underset{\left(\ref{eq:m-3-4}\right)}{\longrightarrow}0,\quad n\rightarrow\infty.\label{eq:m-3-4-a}
\end{equation}
Therefore, if $f\in L^{2}\left(\Omega\right)$ is given, we conclude
that the following properties are equivalent:
\begin{eqnarray*}
\left\langle f,\xi\right\rangle _{L^{2}\left(\Omega\right)} & = & 0,\;\forall\xi\in\mathscr{H}_{F}\\
 & \Updownarrow & \text{by (\ref{eq:m-3-4-a})}\\
\left\langle f,F_{\varphi}\right\rangle _{L^{2}\left(\Omega\right)} & = & 0,\;\forall\varphi\in C_{c}\left(\Omega\right)\\
 & \Updownarrow\\
\left\langle f,T_{F}\varphi\right\rangle _{L^{2}\left(\Omega\right)} & = & 0,\;\forall\varphi\in C_{c}\left(\Omega\right)\\
 & \Updownarrow & \text{(\ensuremath{T_{F}} is selfadjoint, as an operator in \ensuremath{L^{2}\left(\Omega\right)}.)}\\
\left\langle T_{F}f,\varphi\right\rangle _{L^{2}\left(\Omega\right)} & = & 0,\;\forall\varphi\in C_{c}\left(\Omega\right)\\
 & \Updownarrow & \text{(\ensuremath{C_{c}}\ensuremath{\left(\Omega\right)} is dense in \ensuremath{L^{2}\left(\Omega\right)})}\\
T_{F}f & = & 0
\end{eqnarray*}
\end{svmultproof2}

\begin{remark}
Note that $ker\left(T_{F}\right)$ may be infinite dimensional. This
happens for example, in the cases studied in Section \ref{sec:expT},
and in the case of $F_{6}$ in Table \ref{tab:F1-F6}. On the other
hand, these examples are rather degenerate since the corresponding
RKHSs $\mathscr{H}_{F}$ are finite dimensional.

\textbf{Convention.} When computing $T_{F}\left(f\right)$, $f\in L^{2}\left(\Omega\right)$,
we may use (\ref{eq:m-3-1}) to write $f=f^{\left(F\right)}+f^{\left(K\right)}$,
where $f^{\left(F\right)}\in\mathscr{H}_{F}$, $f^{\left(K\right)}\in ker\left(T_{F}\right)$,
and $\langle f^{\left(F\right)},f^{\left(K\right)}\rangle_{L^{2}\left(\Omega\right)}=0$.
As a result,
\[
T_{F}\left(f\right)=T_{F}f^{\left(F\right)}+T_{F}f^{\left(K\right)}=T_{F}f^{\left(F\right)}.
\]
For the inverse operator, by $T_{F}^{-1}\left(f\right)$ we mean $T_{F}^{-1}f^{\left(F\right)}$.\end{remark}
\begin{corollary}
Let $G$, $\Omega$, $F$ and $\mathscr{H}_{F}$ be as above, assume
that $\Omega\subset G$ satisfies \ref{enu:mer1}-\ref{enu:mer4}.
Suppose in addition that $G=\mathbb{R}^{n}$, and that there exists
$\mu\in Ext\left(F\right)$ such that the support of $\mu$, $suppt\left(\mu\right)$,
contains a non-empty open subset in $\mathbb{R}^{n}$; then $ker\left(T_{F}\right)=0$.\end{corollary}
\begin{svmultproof2}
By Corollary \ref{cor:lcg-isom}, there is an isometry $\mathscr{H}_{F}\rightarrow L^{2}\left(\mathbb{R}^{n},\mu\right)$,
by $F_{\varphi}\mapsto\widehat{\varphi}$, for all $\varphi\in C_{c}\left(\Omega\right)$,
and so $\left\Vert F_{\varphi}\right\Vert _{\mathscr{H}_{F}}=\left\Vert \widehat{\varphi}\right\Vert _{L^{2}(\mathbb{R}^{k},\mu)}$.
(Here, $\widehat{\cdot}$ denotes Fourier transform in $\mathbb{R}^{n}$.)

Since $F_{\varphi}=T_{F}\left(\varphi\right)$, for all $\varphi\in C_{c}\left(\Omega\right)$,
and $C_{c}\left(\Omega\right)$ is dense in $L^{2}\left(\Omega\right)$,
we get 
\begin{equation}
\bigl\Vert T_{F}f\bigr\Vert_{\mathscr{H}_{F}}=\bigl\Vert\widehat{f}\bigr\Vert_{L^{2}(\mathbb{R}^{n},\mu)}\label{eq:m-3-6}
\end{equation}
holds also for a dense subspace of $f\in L^{2}\left(\Omega\right)$,
containing $C_{c}\left(\Omega\right)$. (In fact, for all $\xi\in\mathfrak{M}_{2}\left(F,\Omega\right)$,
we have $\left\Vert \xi\right\Vert _{\mathfrak{M}_{2}\left(F,\Omega\right)}=\Vert\widehat{\xi}\Vert_{L^{2}\left(\mathbb{R}^{n},\mu\right)}$;
see Corollary \ref{cor:muHF}.)

It follows that $f\in ker\left(T_{F}\right)\Rightarrow\widehat{f}\equiv0$
on $suppt\left(\mu\right)$, by (\ref{eq:m-3-6}). But since $\overline{\Omega}$
is compact, we conclude by Paley-Wiener, that $\widehat{f}$ is entire
analytic. Since $suppt\left(\mu\right)$ contains a non-empty open
set, we conclude that $\widehat{f}\equiv0$; and by (\ref{eq:m-3-6}),
that therefore $f=0$. 
\end{svmultproof2}

\index{Paley-Wiener}

\index{measure!Haar}

\index{Theorem!Paley-Wiener-}

\index{transform!Fourier-}
\begin{theorem}
\label{thm:mer2}Let $G$, $\Omega$, $F:\Omega^{-1}\Omega\rightarrow\mathbb{C}$,
and $\mathscr{H}_{F}$ be as in Definition \ref{def:LMer}, i.e.,
we assume that \ref{enu:mer1}-\ref{enu:mer4} hold.
\begin{enumerate}
\item \label{enu:mer-1} Let $T_{F}$ denote the corresponding Mercer operator
$L^{2}\left(\Omega\right)\rightarrow\mathscr{H}_{F}$. Then $T_{F}^{*}$
is also a densely defined operator on $\mathscr{H}_{F}$ as follows:
\begin{equation}
\xymatrix{\mathscr{H}_{F}\ar@/^{1pc}/[rr]^{j=T_{F}^{*}} &  & L^{2}\left(\Omega\right)\ar@/^{1pc}/[ll]^{T_{F}}}
\label{eq:mer18}
\end{equation}

\item \label{enu:mer-2}Moreover, every $\xi\in\mathscr{H}_{F}$ has a realization
$j\left(\xi\right)$ as a bounded uniformly continuous function on
$\Omega$, and 
\begin{equation}
T_{F}^{*}\xi=j\left(\xi\right),\quad\forall\xi\in\mathscr{H}_{F}.\label{eq:mer19}
\end{equation}

\item \label{enu:mer-3}Finally, the operator $j$ from (\ref{enu:mer-1})-(\ref{enu:mer-2})
satisfies 
\begin{equation}
ker\left(j\right)=0.\label{eq:mer20}
\end{equation}

\end{enumerate}
\end{theorem}
\begin{svmultproof2}
We begin with the claims relating (\ref{eq:mer18}) and (\ref{eq:mer19})
in the theorem.

Using the reproducing property in the RKHS $\mathscr{H}_{F}$, we
get the following two estimates, valid for all $\xi\in\mathscr{H}_{F}$:
\begin{equation}
\left|\xi\left(x\right)-\xi\left(y\right)\right|^{2}\leq2\left\Vert \xi\right\Vert _{\mathscr{H}_{F}}^{2}\left(F\left(e\right)-\Re\left(F\left(x^{-1}y\right)\right)\right),\quad\forall x,y\in\Omega;\label{eq:mer21}
\end{equation}
and
\begin{equation}
\left|\xi\left(x\right)\right|\leq\left\Vert \xi\right\Vert _{\mathscr{H}_{F}}\sqrt{F\left(e\right)},\quad\forall x\in\Omega.\label{eq:mer22}
\end{equation}

In the remaining part of the proof, we shall adopt the normalization
$F\left(e\right)=1$; so the assertion in (\ref{eq:mer22}) states
that point-evaluation for $\xi\in\mathscr{H}_{F}$ is contractive.
\begin{remark}
In the discussion below, we give conditions which yield boundedness
of the operators in (\ref{eq:mer18}). If bounded, then, by general
theory, we have
\begin{equation}
\left\Vert j\right\Vert _{\mathscr{H}_{F}\rightarrow L^{2}\left(\Omega\right)}=\left\Vert T_{F}\right\Vert _{L^{2}\left(\Omega\right)\rightarrow\mathscr{H}_{F}}\label{eq:mer23}
\end{equation}
for the respective operator-norms. While boundedness holds in ``most''
cases, it does not in general.
\end{remark}
The assertion in (\ref{eq:mer19}) is a statement about the adjoint
of an operator mapping between different Hilbert spaces, meaning:
\begin{equation}
\int_{\Omega}\overline{j\left(\xi\right)\left(x\right)}\varphi\left(x\right)dx=\left\langle \xi,T_{F}\varphi\right\rangle _{\mathscr{H}_{F}}\left(=\left\langle \xi,F_{\varphi}\right\rangle _{\mathscr{H}_{F}}\right)\label{eq:mer24}
\end{equation}
for all $\xi\in\mathscr{H}_{F}$, and $\varphi\in C_{c}\left(\Omega\right)$.
But eq. (\ref{eq:mer24}), in turn, is immediate from (\ref{eq:mer0})
and the reproducing property in $\mathscr{H}_{F}$.

Part (\ref{enu:mer-3}) of the theorem follows from:
\begin{eqnarray*}
ker\left(j\right) & = & \left(ran\left(j^{*}\right)\right)^{\perp}\\
 & \underset{\left(\text{by \ensuremath{\left(\mbox{\ref{enu:mer-2}}\right)}}\right)}{\begin{array}[t]{c}
=\end{array}} & \left(ran\left(T_{F}\right)\right)^{\perp}=0;
\end{eqnarray*}
where we used that $ran\left(T_{F}\right)$ is dense in $\mathscr{H}_{F}$.\index{operator!adjoint of an-}\end{svmultproof2}

\begin{remark}
The r.h.s. in (\ref{eq:mer19}) is subtle because of boundary values
for the function $\xi:\overline{\Omega}\rightarrow\mathbb{C}$ which
represent some vector (also denoted $\xi$) in $\mathscr{H}_{F}$.
We refer to Lemma \ref{lem:F-bd} for specifics.

We showed that if $\varphi\in C_{c}\left(\Omega\right)$, then (\ref{eq:mer24})
holds; but if some $\varphi\in L^{2}\left(\Omega\right)$ is not in
$C_{c}\left(\Omega\right)$, then we pick up a boundary term when
computing $\left\langle \xi,T_{F}\varphi\right\rangle _{\mathscr{H}_{F}}$.
Specifically, one can show, with the techniques in Section \ref{sub:lie},
extended to Lie groups, that the following boundary representation
holds:\end{remark}
\begin{corollary}
\label{cor:Liebd}For a connected Lie group $G$ and $\Omega\subset G$
open, there is a measure $\beta$ on the boundary $\partial\Omega=\overline{\Omega}\backslash\Omega$
such that \index{measure!boundary} 
\begin{equation}
\left\langle \xi,T_{F}\varphi\right\rangle _{\mathscr{H}_{F}}=\int_{\Omega}\overline{\xi\left(x\right)}\varphi\left(x\right)dx+\int_{\partial\Omega}\overline{\xi\big|_{\partial\Omega}\left(\sigma\right)}\left(T_{F}\varphi\right)_{n}\left(\sigma\right)d\beta\left(\sigma\right)\label{eq:m-2-1}
\end{equation}
for all $\varphi\in L^{2}\left(\Omega\right)$. \end{corollary}
\begin{remark}
On the r.h.s. of (\ref{eq:m-2-1}), we must make the following restrictions:
\begin{enumerate}
\item $\overline{\Omega}$ is compact ($\Omega$ is assumed open, connected);
\item $\partial\Omega$ is a differentiable manifold of dimension $\dim\left(G\right)-1$;
\item The function $T_{F}\varphi\in\mathscr{H}_{F}\subset C\left(\Omega\right)$
has a well-defined inward normal vector field; and $\left(T_{F}\varphi\right)_{n}$
denotes the corresponding normal derivative. \index{derivative!normal-}
\end{enumerate}

It follows from Lemma \ref{lem:F-bd} that the term $b_{\xi}\left(\sigma\right):=\xi\big|_{\partial\Omega}\left(\sigma\right)$
on the r.h.s. of (\ref{eq:m-2-1}) satisfies 
\begin{equation}
\left|b_{\xi}\left(\sigma\right)\right|\leq\left\Vert \xi\right\Vert _{\mathscr{H}_{F}},\quad\forall\sigma\in\partial\Omega.\label{eq:m-2-2}
\end{equation}

\end{remark}
\begin{corollary}[Renormalization]
\label{cor:mer2} Let $G$, $\Omega$, $F$, $\mathscr{H}_{F}$,
and $T_{F}$ be as in the statement of Theorem \ref{thm:mer2}. Let
$\left\{ \xi_{n}\right\} _{n\in\mathbb{N}}$, $\left\{ \lambda_{n}\right\} _{n\in\mathbb{N}}$
be the spectral data for the Mercer operator $T_{F}$, i.e., $\lambda_{n}>0$,\index{Renormalization}
\begin{equation}
\sum_{n=1}^{\infty}\lambda_{n}=\left|\Omega\right|,\mbox{ and }\left\{ \xi_{n}\right\} _{n\in\mathbb{N}}\subset L^{2}\left(\Omega\right)\cap\mathscr{H}_{F},\mbox{ satisfying }\label{eq:mer25}
\end{equation}
\begin{equation}
T_{F}\xi_{n}=\lambda_{n}\xi_{n},\:\int_{\Omega}\overline{\xi_{n}\left(x\right)}\xi_{m}\left(x\right)dx=\delta_{n,m},\;n,m\in\mathbb{N},\label{eq:mer26}
\end{equation}
and 
\begin{equation}
F\left(x^{-1}y\right)=\sum_{n=1}^{\infty}\lambda_{n}\overline{\xi_{n}\left(x\right)}\xi_{n}\left(y\right),\;\forall x,y\in\Omega;\label{eq:mer27}
\end{equation}
then
\begin{equation}
\left\{ \sqrt{\lambda_{n}}\xi_{n}\left(\cdot\right)\right\} _{n\in\mathbb{N}}\label{eq:mer28}
\end{equation}
\textup{is an ONB in $\mathscr{H}_{F}$.}\end{corollary}
\begin{svmultproof2}
This is immediate from the theorem. To stress the idea, we include
the proof that $\left\Vert \sqrt{\lambda_{n}}\xi_{n}\right\Vert _{\mathscr{H}_{F}}=1$,
$\forall n\in\mathbb{N}$. 

Clearly, 
\[
\left\Vert \sqrt{\lambda_{n}}\xi_{n}\right\Vert _{\mathscr{H}_{F}}^{2}=\lambda_{n}\left\langle \xi_{n},\xi_{n}\right\rangle _{\mathscr{H}_{F}}=\left\langle \xi_{n},T_{F}\xi_{n}\right\rangle _{\mathscr{H}_{F}}.
\]
But since $\xi_{n}\in L^{2}\left(\Omega\right)\cap\mathscr{H}_{F}$,
\[
\left\langle \xi_{n},T_{F}\xi_{n}\right\rangle _{\mathscr{H}_{F}}=\int_{\Omega}\overline{\xi_{n}\left(x\right)}\xi_{n}\left(x\right)dx=\left\Vert \xi_{n}\right\Vert _{L^{2}\left(\Omega\right)}^{2}=1\;\mbox{by }(\ref{eq:mer26}),
\]
and the result follows.\end{svmultproof2}

\begin{theorem}
\label{thm:mer4}Let $G$, $\Omega$, $F$, $\mathscr{H}_{F}$, $\left\{ \xi_{n}\right\} $,
and $\left\{ \lambda_{n}\right\} $ be as specified in Corollary \ref{cor:mer2}. 
\begin{enumerate}
\item \label{enu:m-1-1}Then $\mathscr{H}_{F}\subseteq L^{2}\left(\Omega\right)$,
and there is a finite constant $C_{1}$ such that 
\begin{equation}
\int_{\Omega}\left|g\left(x\right)\right|^{2}dx\leq C_{1}\left\Vert g\right\Vert _{\mathscr{H}_{F}}^{2}\label{eq:m-1-1}
\end{equation}
holds for all $g\in\mathscr{H}_{F}$. Indeed, $C_{1}=\left\Vert \lambda_{1}\right\Vert _{\infty}$
will do.
\item \label{enu:m-1-2}Let $\left\{ \xi_{k}\right\} $, $\left\{ \lambda_{k}\right\} $
be as above, let $N\in\mathbb{N}$, set 
\begin{equation}
\mathscr{H}_{F}\left(N\right)=span\left\{ \xi_{k}\:\big|\:k=1,2,3,\ldots,N\right\} ;\label{eq:mer-4-1}
\end{equation}
and let $Q_{N}$ be the $\mathscr{H}_{F}$-orthogonal projection onto
$\mathscr{H}_{F}\left(N\right)$; and let $P_{N}$ be the $L^{2}\left(\Omega\right)$-orthogonal
projection onto $span_{1\leq k\leq N}\left\{ \xi_{k}\right\} $; then
\index{projection} 
\begin{equation}
Q_{N}\geq\left(\frac{1}{\lambda_{1}}\right)P_{N}\label{eq:mer-4-2}
\end{equation}
where $\leq$ in (\ref{eq:mer-4-2}) is the order of Hermitian operators,
and where $\lambda_{1}$ is the largest eigenvalue. \index{eigenvalue(s)}\index{Hermitian}
\end{enumerate}
\end{theorem}
\begin{svmultproof2}
Part (\ref{enu:m-1-1}). Pick $\lambda_{n}$, $\xi_{n}$ as in (\ref{eq:mer26})-(\ref{eq:mer28}).
We saw that $\mathscr{H}_{F}\subset L^{2}\left(\Omega\right)$; recall
$\overline{\Omega}$ is compact. Hence
\begin{eqnarray*}
\left\Vert g\right\Vert _{L^{2}\left(\Omega\right)}^{2} & = & \sum_{n=1}^{\infty}\left|\left\langle \xi_{n},g\right\rangle _{2}\right|^{2}\\
 & = & \sum_{n=1}^{\infty}\left|\left\langle T_{F}\xi_{n},g\right\rangle _{\mathscr{H}_{F}}\right|^{2}\\
 & = & \sum_{n=1}^{\infty}\left|\left\langle \lambda_{n}\xi_{n},g\right\rangle _{\mathscr{H}_{F}}\right|^{2}\\
 & = & \sum_{n=1}^{\infty}\lambda_{n}\left|\left\langle \sqrt{\lambda_{n}}\xi_{n},g\right\rangle _{\mathscr{H}_{F}}\right|^{2}\\
 & \leq & \left(\sup_{n\in\mathbb{N}}\left\{ \lambda_{n}\right\} \right)\sum_{n=1}^{\infty}\left|\left\langle \sqrt{\lambda_{n}}\xi_{n},g\right\rangle _{\mathscr{H}_{F}}\right|^{2}\\
 & = & \left(\sup_{n\in\mathbb{N}}\left\{ \lambda_{n}\right\} \right)\left\Vert g\right\Vert _{\mathscr{H}_{F}}^{2},\mbox{ by (\ref{eq:mer28}) and Parseval.}
\end{eqnarray*}
This proves (\ref{enu:m-1-1}) with $C_{1}=\sup_{n\in\mathbb{N}}\left\{ \lambda_{n}\right\} $. 

Part (\ref{enu:m-1-2}). Let $f\in L^{2}\left(\Omega\right)\cap\mathscr{H}_{F}$.
Arrange the eigenvalues $\left(\lambda_{n}\right)$ s.t. 
\begin{equation}
\lambda_{1}\geq\lambda_{2}\geq\lambda_{3}\cdots>0.\label{eq:m-1-3}
\end{equation}
Then 
\begin{eqnarray*}
\left\langle f,Q_{N}f\right\rangle _{\mathscr{H}_{F}}=\left\Vert Q_{N}f\right\Vert _{\mathscr{H}_{F}}^{2} & \overset{\text{(\ref{eq:mer28})}}{=} & \sum_{n=1}^{N}\left|\left\langle \sqrt{\lambda_{n}}\xi_{n},f\right\rangle _{\mathscr{H}_{F}}\right|^{2}\\
 & = & \sum_{n=1}^{N}\frac{1}{\lambda_{n}}\left|\left\langle \lambda_{n}\xi_{n},f\right\rangle _{\mathscr{H}_{F}}\right|^{2}\\
 & = & \sum_{n=1}^{N}\frac{1}{\lambda_{n}}\left|\left\langle T_{F}\xi_{n},f\right\rangle _{\mathscr{H}_{F}}\right|^{2}\\
 & = & \sum_{n=1}^{N}\frac{1}{\lambda_{n}}\left|\left\langle \xi_{n},f\right\rangle _{L^{2}\left(\Omega\right)}\right|^{2}\\
 & \overset{\text{(\ref{eq:m-1-3})}}{\geq} & \frac{1}{\lambda_{1}}\sum_{n=1}^{N}\left|\left\langle \xi_{n},f\right\rangle _{L^{2}\left(\Omega\right)}\right|^{2}\\
 & \overset{\text{(Parseval)}}{=} & \frac{1}{\lambda_{1}}\left\Vert P_{N}f\right\Vert _{L^{2}\left(\Omega\right)}^{2}=\frac{1}{\lambda_{1}}\left\langle f,P_{N}f\right\rangle _{L^{2}\left(\Omega\right)};
\end{eqnarray*}
and so the \emph{a priori} estimate (\ref{enu:m-1-2}) holds.\end{svmultproof2}

\begin{remark}
The estimate (\ref{eq:m-1-1}) in Theorem \ref{thm:mer4} is related
to, but different from a classical Poincaré-inequality \cite{Maz11}.
The latter \emph{a priori }is as follows:

\index{Poincaré-inequality}

Let $\Omega\subset\mathbb{R}^{n}$ satisfying in \ref{enu:mer1}-\ref{enu:mer4}
in Definition \ref{def:LMer}, and let $\left|\Omega\right|_{n}$
denote the $n$-dimensional Lebesgue measure of $\Omega$, i.e., 
\[
\left|\Omega\right|_{n}=\int_{\mathbb{R}^{n}}\chi_{\Omega}\left(x\right)\underset{dx}{\underbrace{dx_{1}\cdots dx_{n}}};
\]
let $\nabla=\left(\frac{\partial}{\partial x_{1}},\ldots,\frac{\partial}{\partial x_{n}}\right)$
be the gradient, and 
\begin{equation}
\left\Vert \nabla f\right\Vert _{L^{2}\left(\Omega\right)}^{2}:=\sum_{i=1}^{n}\int_{\Omega}\left|\frac{\partial f}{\partial x_{i}}\right|^{2}dx.\label{eq:f2-3-1}
\end{equation}
Finally, let $\lambda_{1}\left(N\right)=$ the finite eigenvalue for
the Neumann problem on $\Omega$ (NBP$\Omega$). Then,
\begin{equation}
\left\Vert f-\frac{1}{\left|\Omega\right|_{n}}\int_{\Omega}fdx\right\Vert _{L^{2}\left(\Omega\right)}^{2}\leq\frac{1}{\lambda_{1}\left(N\right)}\left\Vert \nabla f\right\Vert _{L^{2}\left(\Omega\right)}^{2}\label{eq:f2-3-2}
\end{equation}
holds for all $f\in L^{2}\left(\Omega\right)$ such that $\frac{\partial f}{\partial x_{i}}\in L^{2}\left(\Omega\right)$,
$1\leq i\leq n$.
\end{remark}
Since we shall not have occasion to use this more general version
of the Mercer-operators we omit details below, and restrict attention
to the case of finite interval in $\mathbb{R}$.
\begin{remark}
Some of the conclusions in Theorem \ref{thm:mer2} hold even if conditions
\ref{enu:mer1}-\ref{enu:mer4} are relaxed. But condition \ref{enu:mer3}
ensures that $L^{2}\left(\Omega\right)$ has a realization as a subspace
of $\mathscr{H}_{F}$; see eq. (\ref{eq:mer18}). By going to unbounded
sets $\Omega$ we give up this. 

Even if $\Omega\subset G$ is unbounded, then the operator $T_{F}$
in (\ref{eq:mer0}) is still well-defined; and it may be considered
as a possibly unbounded linear operator as follow:
\begin{equation}
L^{2}\left(\Omega\right)\stackrel{T_{F}}{\longrightarrow}\mathscr{H}_{F}\label{eq:mer-2-1}
\end{equation}
with dense domain $C_{c}\left(\Omega\right)$ in $L^{2}\left(\Omega\right)$.
(If $G$ is a Lie group, we may take $C_{c}^{\infty}\left(\Omega\right)$
as dense domain for $T_{F}$ in (\ref{eq:mer-2-1}).)\index{Lie group}\end{remark}
\begin{lemma}
\label{lem:mer-2-2}Let $\left(\Omega,F\right)$ be as above, but
now assume only conditions \ref{enu:mer1} and \ref{enu:mer2} for
the subset $\Omega\subset G$. 

Then the operator $T_{F}$ in (\ref{eq:mer-2-1}) is a \uline{closable}
operator from $L^{2}\left(\Omega\right)$ into $\mathscr{H}_{F}$;
i.e., the closure of the graph of $T_{F}$, as a subspace in $L^{2}\left(\Omega\right)\times\mathscr{H}_{F}$,
is the graph of a (closed) operator $\overline{T_{F}}$ from $L^{2}\left(\Omega\right)$
into $\mathscr{H}_{F}$; still with $dom\left(\overline{T_{F}}\right)$
dense in $L^{2}\left(\Omega\right)$. \index{operator!closed} \index{operator!graph of-}\end{lemma}
\begin{svmultproof2}
Using a standard lemma on unbounded operators, see \cite[ch.13]{Rud73},
we need only show that the following implication holds:\index{operator!unbounded}

Given $\left\{ f_{n}\right\} \subset C_{c}\left(\Omega\right)\left(\subset L^{2}\left(\Omega\right)\right)$,
suppose $\exists\,\xi\in\mathscr{H}_{F}$; and suppose the following
two limits holds:
\begin{align}
 & \lim_{n\rightarrow\infty}\left\Vert f_{n}\right\Vert _{L^{2}\left(\Omega\right)}=0,\mbox{ and}\label{eq:mer-2-2}\\
 & \lim_{n\rightarrow\infty}\left\Vert \xi-T_{F}\left(f_{n}\right)\right\Vert _{\mathscr{H}_{F}}=0.\label{eq:mer-2-3}
\end{align}
Then, it follows that $\xi=0$ in $\mathscr{H}_{F}$.

Now assume $\left\{ f_{n}\right\} $ and $\xi$ satisfying (\ref{eq:mer-2-2})-(\ref{eq:mer-2-3});
the by the reproducing property in $\mathscr{H}_{F}$, we have 
\begin{equation}
\left\langle \xi,T_{F}f_{n}\right\rangle _{\mathscr{H}_{F}}=\int_{\Omega}\overline{\xi\left(x\right)}f_{n}\left(x\right)dx\label{eq:mer-2-4}
\end{equation}
Using (\ref{eq:mer-2-3}), we get 
\[
\lim_{n\rightarrow\infty}l.h.s.\left(\ref{eq:mer-2-4}\right)=\left\langle \xi,\xi\right\rangle _{\mathscr{H}_{F}}=\left\Vert \xi\right\Vert _{\mathscr{H}_{F}}^{2};
\]
and using (\ref{eq:mer-2-2}), we get 
\[
\lim_{n\rightarrow\infty}r.h.s.\left(\ref{eq:mer-2-4}\right)=0.
\]
The domination here is justified by (\ref{eq:mer-2-3}). Indeed, if
(\ref{eq:mer-2-3}) holds, $\exists\,n_{0}$ s.t.
\[
\left\Vert \xi-T_{F}\left(f_{n}\right)\right\Vert _{\mathscr{H}_{F}}\leq1,\;\forall n\geq n_{0},
\]
and therefore,
\begin{equation}
\sup_{n\in\mathbb{N}}\left\Vert T_{F}\left(f_{n}\right)\right\Vert _{\mathscr{H}_{F}}\leq\max_{n\leq n_{0}}\Bigl(1+\left\Vert T_{F}\left(f_{n}\right)\right\Vert _{\mathscr{H}_{F}}\Bigr)<\infty.\label{eq:mer-2-5}
\end{equation}
The desired conclusion follows; we get $\xi=0$ in $\mathscr{H}_{F}$.\end{svmultproof2}

\begin{remark}
The conclusion in Lemma \ref{lem:mer-2-2} is the assertion that the
closure of the graph of $T_{F}$ is again the graph of a closed operator,
called the closure. Hence the importance of \textquotedblleft closability.\textquotedblright{}
Once we have existence of the closure of the operator $T_{F}$, as
a closed operator, we will denote this closed operator also by the
same $T_{F}$. This helps reduce the clutter in operator symbols to
follow. From now on, $T_{F}$ will be understood to be the closed
operator obtained in Lemma \ref{lem:mer-2-2}.\index{operator!closed}\index{operator!graph of-}
\end{remark}

\begin{remark}
If in Lemma \ref{lem:mer-2-2}, for $\left(F,\Omega\right)$ the set
$\Omega$ also satisfies \ref{enu:mer3}-\ref{enu:mer4}, then the
operator $T_{F}$ in (\ref{eq:mer-2-1}) is in fact bounded; but in
general it is not; see the example below with $G=\mathbb{R}$, and
$\Omega=\mathbb{R}_{+}$.\end{remark}
\begin{corollary}
\label{cor:mer-2-4}Let $\Omega$, $G$ and $F$ be as in Lemma \ref{lem:mer-2-2},
i.e., with $\Omega$ possibly unbounded, and let $T_{F}$ denote the
closed operator obtained from (\ref{eq:mer-2-1}), and the conclusion
in the lemma. Then we get the following two conclusions:\index{operator!selfadjoint}
\begin{enumerate}
\item \label{enu:mer-2-1}$T_{F}^{*}T_{F}$ is selfadjoint with dense domain
in $L^{2}\left(\Omega\right)$, and 
\item \label{enu:mer-2-2}$T_{F}T_{F}^{*}$ is selfadjoint with dense domain
in $\mathscr{H}_{F}$.
\end{enumerate}
\end{corollary}
\begin{svmultproof2}
This is an application of the fundamental theorem for closed operators;
see \cite[Theorem 13.13]{Rud73}.\end{svmultproof2}

\begin{remark}
\label{rem:mer-2-5}The significance of the conclusions (\ref{enu:mer-2-1})-(\ref{enu:mer-2-2})
in the corollary is that we may apply the Spectral Theorem to the
respective selfadjoint operators in order to get that $\left(T_{F}^{*}T_{F}\right)^{1/2}$
is a well-defined selfadjoint operator in $L^{2}\left(\Omega\right)$;
and that $\left(T_{F}T_{F}^{*}\right)^{1/2}$ well-defined and selfadjoint
in $\mathscr{H}_{F}$. 

Moreover, by the polar decomposition applied to $T_{F}$ (see \cite[ch. 13]{Rud73}),
we conclude that:\index{polar decomposition}\index{Theorem!Spectral-}\index{Spectral Theorem}
\begin{equation}
\mbox{spec}\left(T_{F}^{*}T_{F}\right)\backslash\left\{ 0\right\} =\mbox{spec}\left(T_{F}T_{F}^{*}\right)\backslash\left\{ 0\right\} .\label{eq:mer-2-6}
\end{equation}

\end{remark}
\index{Lie algebra}
\begin{theorem}
\label{thm:mer3}Assume $F$, $G$, $\Omega$, and $T_{F}$, are as
above, where $T_{F}$ denotes the closed operator $L^{2}\left(\Omega\right)\stackrel{T}{\longrightarrow}\mathscr{H}_{F}$.
We are assuming that $G$ is a Lie group, $\Omega$ satisfies \ref{enu:mer1}-\ref{enu:mer2}.
Let $X$ be a vector in the Lie algebra of $G$, $X\in La$$\left(G\right)$,
and define $D_{X}^{\left(F\right)}$ as a skew-Hermitian operator
in $\mathscr{H}_{F}$ as follows:\index{skew-Hermitian operator; also called skew-symmetric}
\begin{align}
dom\left(D_{X}^{\left(F\right)}\right) & =\left\{ T_{F}\varphi\:\big|\:\varphi\in C_{c}^{\infty}\left(\Omega\right)\right\} ,\mbox{ and}\label{eq:mer-2-7}\\
D_{X}^{\left(F\right)}\left(F_{\varphi}\right) & =F_{\widetilde{X}\varphi}\nonumber 
\end{align}
where
\begin{align}
\bigl(\widetilde{X}\varphi\bigr)\left(g\right)= & \lim_{t\rightarrow0}\frac{1}{t}\left(\varphi\left(\exp\left(-tX\right)g\right)-\varphi\left(g\right)\right)\label{eq:mer-2-8}
\end{align}
for all $\varphi\in C_{c}^{\infty}\left(\Omega\right)$, and all $g\in\Omega$. 

Then $\widetilde{X}$ defines a skew-Hermitian operator in $L^{2}\left(\Omega\right)$
with dense domain $C_{c}^{\infty}\left(\Omega\right)$. (It is closable,
and we shall denote its closure also by $\widetilde{X}$.)\index{operator!skew-Hermitian}

We get 
\begin{equation}
D_{X}^{\left(F\right)}T_{F}=T_{F}\widetilde{X}\label{eq:mer-2-9}
\end{equation}
on the domain of $\widetilde{X}$; or equivalently\index{operator!domain of-}
\begin{equation}
D_{X}^{\left(F\right)}=T_{F}\widetilde{X}T_{F}^{-1}.\label{eq:mer-2-10}
\end{equation}
\end{theorem}
\begin{svmultproof2}
By definition, for all $\varphi\in C_{c}^{\infty}\left(\Omega\right)$,
we have 
\[
\left(D_{X}^{\left(F\right)}T_{F}\right)\left(\varphi\right)=D_{X}^{\left(F\right)}F_{\varphi}=F_{\widetilde{X}\varphi}=\left(T_{F}\widetilde{X}\right)\left(\varphi\right).
\]
Since $\left\{ F_{\varphi}\:\big|\:\varphi\in C_{c}^{\infty}\left(\Omega\right)\right\} $
is a core-domain, (\ref{eq:mer-2-9}) follows. Then the conclusions
in the theorem follow from a direct application of Lemma \ref{lem:mer-2-2},
and Corollary \ref{cor:mer-2-4}; see also Remark \ref{rem:mer-2-5}.\end{svmultproof2}

\begin{corollary}
\label{cor:mer-2-10}For the respective adjoint operators in (\ref{eq:mer-2-10}),
we have 
\begin{equation}
\bigl(D_{X}^{\left(F\right)}\bigr)^{*}=T_{F}^{*-1}\widetilde{X}^{*}T_{F}^{*}.\label{eq:mer-2-11}
\end{equation}
\end{corollary}
\begin{svmultproof2}
The formula (\ref{eq:mer-2-11}) in the corollary results from applying
the adjoint operation to both sides of eq (\ref{eq:mer-2-10}), and
keeping track of the domains of the respective operators in the product
on the r.h.s. in eq (\ref{eq:mer-2-10}). Only after checking domains
of the three respective operators, occurring as factors in the product
on the r.h.s. in eq (\ref{eq:mer-2-10}), may we then use the algebraic
rules for adjoint of a product of operators. In this instance, we
conclude that adjoint of the product on the r.h.s. in eq (\ref{eq:mer-2-10})
is the product of the adjoint of the factors, but now composed in
the reverse order; so product from left to right, becomes product
of the adjoints from right to left; hence the result on the r.h.s.
in eq (\ref{eq:mer-2-11}).

Now the fact that the domain issues work out follows from application
of Corollary \ref{cor:mer-2-4}, Remark \ref{rem:mer-2-5}, and Theorem
\ref{thm:mer3}; see especially eqs (\ref{eq:mer-2-6}), and (\ref{eq:mer-2-7}).
The rules for adjoint of a product of operators, where some factors
are unbounded are subtle, and we refer to \cite[chapter 13]{Rud73}
and \cite[Chapter 11-12]{DS88b}. Care must be exercised when the
unbounded operators in the product map between different Hilbert spaces.
The fact that our operator $T_{F}$ is closed as a linear operator
from $L^{2}\left(\Omega\right)$ into $\mathscr{H}_{F}$ is crucial
in this connection; see Lemma \ref{lem:mer-2-2}.\index{operator!adjoint of an-}
\end{svmultproof2}

\begin{corollary}
\label{cor:mer3}Let $G$, $\Omega$, $F$, and $T_{F}$ be as above;
then the RKHS $\mathscr{H}_{F}$ consists precisely of the continuous
functions $\xi$ on $\Omega$ such that $\xi\in dom\bigl(\bigl(T_{F}^{*}T_{F}\bigr)^{-1/4}\bigr)$,
and then 
\[
\bigl\Vert\xi\bigr\Vert_{\mathscr{H}_{F}}=\bigl\Vert\bigl(T_{F}^{*}T_{F}\bigr)^{-1/4}\xi\bigr\Vert_{L^{2}\left(\Omega\right)}.
\]
\end{corollary}
\begin{svmultproof2}
An immediate application of Corollary \ref{cor:mer-2-4}; and the
polar decomposition, applied to the closed operator $T_{F}$ from
Lemma \ref{lem:mer-2-2}.\end{svmultproof2}

\begin{example}[Application]
Let $G=\mathbb{R}$, $\Omega=\mathbb{R}_{+}=\left(0,\infty\right)$;
so that $\Omega-\Omega=\mathbb{R}$; let $F\left(x\right)=e^{-\left|x\right|}$,
$\forall x\in\mathbb{R}$, and let $D^{\left(F\right)}$ be the skew-Hermitian
operator from Corollary \ref{cor:mer-2-10}. Then $D^{\left(F\right)}$
has deficiency indices $\left(1,0\right)$ in $\mathscr{H}_{F}$.
\index{deficiency indices}\end{example}
\begin{svmultproof2}
From Corollary \ref{cor:mer-2-4}, we conclude that $\mathscr{H}_{F}$
consists of all continuous functions $\xi$ on $\mathbb{R}_{+}\left(=\Omega\right)$
such that $\xi$ and $\xi'=\frac{d\xi}{dx}$ are in $L^{2}\left(\mathbb{R}_{+}\right)$;
and then 
\begin{equation}
\bigl\Vert\xi\bigr\Vert_{\mathscr{H}_{F}}^{2}=\frac{1}{2}\left(\int_{0}^{\infty}\left|\xi\left(x\right)\right|^{2}dx+\int_{0}^{\infty}\left|\xi'\left(x\right)\right|^{2}dx\right)+\int_{0}^{1}\overline{\xi_{n}}\xi\,d\beta;\label{eq:mer-2-13}
\end{equation}
where $\xi_{n}$ denote its inward normal derivative, and $d\beta$
is the corresponding boundary measure. Indeed, $d\beta=-\frac{1}{2}\delta_{0}$,
with $\delta_{0}:=\delta\left(\cdot-0\right)=$ Dirac mass at $x=0$.
See Sections \ref{sub:F2}-\ref{sub:F3} for details. \index{derivative!normal-}

We now apply Corollary \ref{cor:mer-2-10} to the operator $D_{0}=\dfrac{d}{dx}$
in $L^{2}\left(\mathbb{R}_{+}\right)$ with $dom\left(D_{0}\right)=C_{c}^{\infty}\left(\mathbb{R}_{+}\right)$.
It is well known that $D_{0}$ has deficiency indices $\left(1,0\right)$;
and the $+$ deficiency space is spanned by $\xi_{+}\left(x\right):=e^{-x}\in L^{2}\left(\mathbb{R}_{+}\right)$,
i.e., $x>0$.

Hence, using (\ref{eq:mer-2-11}), we only need to show that $\xi_{+}\in\mathscr{H}_{F}$;
but this is immediate from (\ref{eq:mer-2-13}); in fact
\[
\bigl\Vert\xi_{+}\bigr\Vert_{\mathscr{H}_{F}}^{2}=1.
\]
Setting $\xi_{-}\left(x\right):=e^{x}$, the same argument shows that
$r.h.s.\left(\ref{eq:mer-2-13}\right)=\infty$, so the index conclusion
$\left(1,0\right)$ follows.
\end{svmultproof2}

We now return to the case for $G=\mathbb{R}$, and $F$ is fixed continuous
positive definite function on some finite interval $\left(-a,a\right)$,
i.e., the case where $\Omega=\left(0,a\right)$.
\begin{corollary}
If $G=\mathbb{R}$ and if $\Omega=\left(0,a\right)$ is a bounded
interval, $a<\infty$, then the operator $D^{\left(F\right)}$ has
equal indices for all given $F:\left(-a,a\right)\rightarrow\mathbb{C}$
which is p.d. and continuous.\end{corollary}
\begin{svmultproof2}
We showed in Theorem \ref{thm:mer2}, and Corollary \ref{cor:mer2}
that if $\Omega=\left(0,a\right)$ is bounded, then $T_{F}:L^{2}\left(0,a\right)\rightarrow\mathscr{H}_{F}$
is bounded. By Corollary \ref{cor:mer-2-4}, we get that $T_{F}^{-1}:\mathscr{H}_{F}\rightarrow L^{2}\left(0,a\right)$
is closed. Moreover, as an operator in $L^{2}\left(0,a\right)$, $T_{F}$
is positive and selfadjoint. 

Since 
\begin{equation}
D_{0}=\dfrac{d}{dx}\Big|{}_{C_{c}^{\infty}\left(0,a\right)}\label{eq:mer-2-14}
\end{equation}
has indices $\left(1,1\right)$ in $L^{2}\left(0,a\right)$, it follows
from (\ref{eq:mer-2-11}) applied to (\ref{eq:mer-2-14}) that $D^{\left(F\right)}$,
as a skew-Hermitian operator in $\mathscr{H}_{F}$, must have indices
$\left(0,0\right)$ or $\left(1,1\right)$.\index{skew-Hermitian operator; also called skew-symmetric}

To finish the proof, use that a skew Hermitian operator with indices
$\left(1,0\right)$ must generate a semigroup of isometries; one that
is non-unitary. If such an isometry semigroup were generated by the
particular skew Hermitian operator $D^{\left(F\right)}$ then this
would be inconsistent with Corollary \ref{cor:mer3}; see especially
the formula for the norm in $\mathscr{H}_{F}$.
\end{svmultproof2}

To simplify notation, we now assume that the endpoint $a$ in \eqref{mer-1}
is $a=1$.
\begin{proposition}
\label{prop:mer1}Let $F$ be p.d. continuos on $I=\left(-1,1\right)\subset\mathbb{R}$.
Assume $\mu\in Ext\left(F\right)$, and $\mu\ll d\lambda$, i.e.,
$\exists M\in L^{1}\left(\mathbb{R}\right)$ s.t.
\begin{equation}
d\mu\left(\lambda\right)=M\left(\lambda\right)d\lambda,\;\mbox{where }d\lambda=\mbox{Lebesgue measure on }\mathbb{R}.\label{eq:mer1}
\end{equation}
Set $\mathscr{L}=\left(2\pi\mathbb{Z}\right)$ (period lattice), and
\begin{equation}
\widehat{\varphi_{I}}\left(\xi\right)=\int_{0}^{1}e^{-i\xi y}\varphi\left(y\right)dy,\;\forall\varphi\in C_{c}\left(0,1\right);\label{eq:mer2}
\end{equation}
then the Mercer operator is as follows:
\begin{equation}
\left(T_{F}\varphi\right)\left(x\right)=\sum_{l\in\mathscr{L}}M\left(l\right)\widehat{\varphi_{I}\left(l\right)}e^{ilx}.\label{eq:mer3}
\end{equation}
\end{proposition}
\begin{svmultproof2}
Let $x\in\left(0,1\right)$, then 
\begin{eqnarray*}
r.h.s.\left(\ref{eq:mer3}\right) & = & \sum_{l\in\mathscr{L}}M\left(l\right)\left(\int_{0}^{1}e^{-ily}\varphi\left(y\right)dy\right)e^{ilx}\\
 & \underset{\left(\mbox{Fubini}\right)}{=} & \int_{0}^{1}\varphi\left(y\right)\underset{\mbox{Poisson summation}}{\underbrace{\left(\sum_{l\in\mathscr{L}}M\left(l\right)e^{il\left(x-y\right)}\right)}}dy\\
 & = & \int_{0}^{1}\varphi\left(y\right)F\left(x-y\right)dy\\
 & = & \left(T_{F}\varphi\right)\left(x\right)\\
 & = & l.h.s.\left(\ref{eq:mer3}\right),
\end{eqnarray*}
where we use that $F=\widehat{d\mu}\Big|_{\left(-1,1\right)}$, and
(\ref{eq:mer1}).\end{svmultproof2}

\begin{example}
Application to Table \ref{tab:F1-F6}: $\mathscr{L}=2\pi\mathbb{Z}$. \end{example}
\begin{svmultproof2}
Application of Proposition \ref{prop:mer1}. See Table \ref{tab:mercer}
below.
\end{svmultproof2}

\renewcommand{\arraystretch}{3}

\begin{table}[H]
\begin{tabular}{|>{\centering}p{0.2\textwidth}|>{\centering}p{0.3\textwidth}|>{\centering}p{0.3\textwidth}|}
\hline 
p.d. Function & $\left(T_{F}\varphi\right)\left(x\right)$ & $M\left(\lambda\right)$, $\lambda\in\mathbb{R}$\tabularnewline
\hline 
$F_{1}$ & ${\displaystyle \frac{1}{2}\sum_{l\in\mathscr{L}}e^{-\left|l\right|}\widehat{\varphi}_{I}\left(l\right)e^{ilx}}$ & ${\displaystyle \frac{1}{2}e^{-\left|l\right|}}$ \tabularnewline
\hline 
$F_{3}$ & ${\displaystyle \sum_{l\in\mathscr{L}}\frac{1}{\pi\left(1+l^{2}\right)}\widehat{\varphi}_{I}\left(l\right)e^{ilx}}$ & ${\displaystyle \frac{1}{\pi\left(1+l^{2}\right)}}$\tabularnewline
\hline 
$F_{5}$ & ${\displaystyle \sum_{l\in\mathscr{L}}\frac{1}{\sqrt{2\pi}}e^{-l^{2}/2}\widehat{\varphi}_{I}\left(l\right)e^{ilx}}$ & ${\displaystyle \frac{1}{\sqrt{2\pi}}e^{-l^{2}/2}}$\tabularnewline
\hline 
\end{tabular}

\protect\caption{\label{tab:mercer}Application of Proposition \ref{prop:mer1} to
Table \ref{tab:F1-F6}.}
\end{table}

\renewcommand{\arraystretch}{1}
\begin{corollary}
Let $F:\left(-1,1\right)\rightarrow\mathbb{C}$ be a continuous positive
definite function on the interval $\left(-1,1\right)$, and assume:
\begin{enumerate}[label=(\roman{enumi}),ref=\roman{enumi}]
\item  $F\left(0\right)=1$
\item $\exists\mu\in Ext_{1}\left(F\right)$ s.t. $\mu\ll d\lambda$, i.e.,
$\exists M\in L^{1}\left(\mathbb{R}\right)$ s.t. $d\mu\left(\lambda\right)=M\left(\lambda\right)d\lambda$
on $\mathbb{R}$. 
\end{enumerate}
Now consider the Mercer operator 
\begin{equation}
\left(T_{F}\varphi\right)\left(x\right)=\int_{0}^{1}\varphi\left(y\right)F\left(x-y\right)dy,\;\varphi\in L^{2}\left(0,1\right),x\in\left(0,1\right).\label{eq:mer11}
\end{equation}
Then the following two conditions (bd-1) and (bd-2) are equivalent,
where 
\begin{equation}
\mathscr{L}=\left(2\pi\mathbb{Z}\right)=\widehat{\mathbb{T}},\mbox{ and}\label{eq:mer12}
\end{equation}
\begin{eqnarray*}
 & \text{(bd-1)} & \qquad b_{M}:=\sup_{\lambda\in\left[0,1\right]}\sum_{l\in\mathscr{L}}M\left(\lambda+l\right)<\infty,\;\mbox{and}\\
 & \Updownarrow\\
 & \text{(bd-2)} & \qquad T_{F}\left(L^{2}\left(0,1\right)\right)\subseteq\mathscr{H}_{F}.
\end{eqnarray*}
If (bd-1) ($\Leftrightarrow$ (bd-2)) holds, then, for the corresponding
operator-norm, we then have
\begin{equation}
\left\Vert T_{F}\right\Vert _{L^{2}\left(0,1\right)\rightarrow\mathscr{H}_{F}}=\sqrt{b_{M}}\;\mbox{in }\text{(bd-1)}.\label{eq:mer13}
\end{equation}
\end{corollary}
\begin{remark}
Condition (bd-1) is automatically satisfied in all interesting cases
(at least from the point of view of our present note.)\end{remark}
\begin{svmultproof2}
The key step in the proof of ``$\Longleftrightarrow$'' was the
Parseval duality,\index{duality!Parseval-} 
\begin{equation}
\int_{0}^{1}\left|f\left(x\right)\right|^{2}dx=\sum_{l\in\mathscr{L}}\left|\widehat{f_{I}}\left(l\right)\right|^{2},\mbox{ where}\label{eq:mer14}
\end{equation}
$\left[0,1\right)\simeq\mathbb{T}=\mathbb{R}/\mathbb{Z}$, $\widehat{\mathbb{T}}\simeq\mathscr{L}$.
\index{Parseval's identity}

Let $F$, $T_{F}$, and $M$ be as in the statement of the corollary.
Then for $\varphi\in C_{c}\left(0,1\right)$, we compute the $\mathscr{H}_{F}$-norm
of 
\begin{equation}
T_{F}\left(\varphi\right)=F_{\varphi}\label{eq:mer15}
\end{equation}
with the use of (\ref{eq:mer11}), and Proposition \ref{prop:mer1}.

We return to 
\begin{equation}
\widehat{\varphi_{I}}\left(l\right)=\int_{0}^{1}e^{-ily}\varphi\left(y\right)dy,\;l\in\mathscr{L};\label{eq:mer16}
\end{equation}
and we now compute $\widehat{\left(T_{F}\varphi\right)_{I}}\left(l\right)$,
$l\in\mathscr{L}$; starting with $T_{F}\varphi$ from (\ref{eq:mer11}).
The result is
\[
\widehat{\left(T_{F}\varphi\right)_{I}}\left(l\right)=M\left(l\right)\widehat{\varphi_{I}}\left(l\right),\;\forall l\in\mathscr{L}\left(=2\pi\mathbb{Z}.\right)
\]
And further, using \chapref{conv}, we have: 
\begin{eqnarray*}
\left\Vert F_{\varphi}\right\Vert _{\mathscr{H}_{F}}^{2} & = & \left\Vert T_{F}\varphi\right\Vert _{\mathscr{H}_{F}}^{2}\\
 & \begin{array}[t]{c}
=\\
\text{(Cor. \ensuremath{\left(\ref{cor:lcg-isom}\right)})}
\end{array} & \int_{\mathbb{R}}\left|\widehat{\varphi}\left(\lambda\right)\right|^{2}M\left(\lambda\right)d\lambda\\
 & = & \int_{0}^{1}\sum_{l\in\mathscr{L}}\left|\widehat{\varphi}\left(\lambda+l\right)\right|^{2}M\left(\lambda+l\right)d\lambda\\
 & \begin{array}[t]{c}
\leq\\
\text{\ensuremath{\left(\ref{eq:mer16}\right)}}
\end{array} & \left(\sum_{l\in\mathscr{L}}\left|\widehat{\varphi}\left(l\right)\right|^{2}\right)\sup_{\lambda\in\left[0,1\right)}\sum_{l\in\mathscr{L}}M\left(\lambda+l\right)\\
 & \begin{array}[t]{c}
=\\
\text{\ensuremath{\left(\ref{eq:mer14}\right)}}\\
\text{and (bd-1)}
\end{array} & \left(\int_{0}^{1}\left|\varphi\left(x\right)\right|^{2}dx\right)\cdot b_{M}=\left\Vert \varphi\right\Vert _{L^{2}\left(0,1\right)}^{2}\cdot b_{M}.
\end{eqnarray*}

Hence, if $b_{M}<\infty$, (bd-2) holds, with
\begin{equation}
\Bigl\Vert T_{F}\Bigr\Vert_{L^{2}\left(0,1\right)\rightarrow\mathscr{H}_{F}}\leq\sqrt{b_{M}}.\label{eq:mer17}
\end{equation}
Using standard Fourier duality, one finally sees that ``$\leq$''
in (\ref{eq:mer17}) is in fact ``$=$''. \index{Fourier duality}\index{duality!Fourier-}\end{svmultproof2}

\begin{remark}
A necessary condition for boundedness of $T_{F}:L^{2}\left(0,1\right)\rightarrow\mathscr{H}_{F}$,
is $M\in L^{\infty}\left(\mathbb{R}\right)$ when the function $M\left(\cdot\right)$
is as specified in (ii) of the corollary.\end{remark}
\begin{svmultproof2}
Let $\varphi\in C_{c}\left(0,1\right)$, then

\begin{eqnarray*}
\Bigl\Vert T_{F}\varphi\Bigr\Vert_{\mathscr{H}_{F}}^{2}=\Bigl\Vert F_{\varphi}\Bigr\Vert_{\mathscr{H}_{F}}^{2} & = & \int_{\mathbb{R}}\left|\widehat{\varphi}\left(\lambda\right)\right|^{2}M\left(\lambda\right)d\lambda\\
 & \leq & \bigl\Vert M\bigr\Vert{}_{\infty}\cdot\int_{\mathbb{R}}\left|\widehat{\varphi}\left(\lambda\right)\right|^{2}d\lambda\\
 & = & \bigl\Vert M\bigr\Vert{}_{\infty}\cdot\int_{\mathbb{R}}\left|\varphi\left(x\right)\right|^{2}d\lambda\,\,\,\,\,\,(\mbox{Parseval})\\
 & = & \bigl\Vert M\bigr\Vert{}_{\infty}\bigl\Vert\varphi\bigr\Vert{}_{L^{2}\left(0,1\right)}^{2}.
\end{eqnarray*}
\end{svmultproof2}

\begin{theorem}
Let $F$ be as in Proposition \ref{prop:mer1}, and $\mathscr{H}_{F}$
the corresponding RKHS. Define the Hermitian operator $D^{\left(F\right)}\left(F_{\varphi}\right)=\frac{1}{i}F_{\varphi'}$
on 
\[
dom\bigl(D^{\left(F\right)}\bigr)=\left\{ F_{\varphi}\:\Big|\:\varphi\in C_{c}^{\infty}\left(0,1\right)\right\} \subset\mathscr{H}_{F}
\]
as before. Let $A\supset D^{\left(F\right)}$ be a selfadjoint extension
of $D^{\left(F\right)}$, i.e., 
\[
D^{\left(F\right)}\subset A\subset\bigl(D^{\left(F\right)}\bigr)^{*},\;A=A^{*}.
\]
Let $P=P_{A}$ be the projection valued measure (PVM) of $A$, and
\begin{equation}
U_{t}^{\left(A\right)}=e^{tA}=\int_{\mathbb{R}}e^{it\lambda}P_{A}\left(d\lambda\right),\:t\in\mathbb{R}\label{eq:mer4}
\end{equation}
be the one-parameter unitary group; and for all $f$ measurable on
$\mathbb{R}$, set (the Spectral Theorem applied to $A$)
\begin{equation}
f\left(A\right)=\int_{\mathbb{R}}f\left(\lambda\right)P_{A}\left(d\lambda\right);\label{eq:mer5}
\end{equation}
then we get the following
\begin{equation}
\left(T_{F}\varphi\right)\left(x\right)=\left(M\widehat{\varphi_{I}}\right)\left(A\right)U_{x}^{\left(A\right)}=U_{x}^{\left(A\right)}\left(M\widehat{\varphi_{I}}\right)\left(A\right)\label{eq:mer6}
\end{equation}
for the Mercer operator $\left(T_{F}\varphi\right)\left(x\right)=\int_{0}^{1}F\left(x-y\right)\varphi\left(y\right)$,
$\varphi\in L^{2}\left(0,1\right)$.

\index{measure!PVM}\index{selfadjoint extension}\index{Hermitian}
\end{theorem}
\index{Theorem!Spectral-}\index{Spectral Theorem}
\begin{svmultproof2}
Using (\ref{eq:mer5}), we get 
\begin{eqnarray}
\widehat{\varphi_{I}}\left(A\right)U_{x}^{\left(A\right)} & = & \int_{\mathbb{R}}\int_{0}^{1}\varphi\left(y\right)e^{-i\lambda y}P_{A}\left(d\lambda\right)U^{A}\left(x\right)\nonumber \\
 & \underset{\text{(Fubini)}}{=} & \int_{0}^{1}\varphi\left(y\right)\left(\int_{\mathbb{R}}e^{i\lambda\left(x-y\right)}P_{A}\left(d\lambda\right)\right)dy\nonumber \\
 & \underset{\text{\ensuremath{\left(\ref{eq:mer4}\right)}}}{=} & \int_{0}^{1}U^{A}\left(x-y\right)\varphi\left(y\right)dy\nonumber \\
 & = & U_{x}^{\left(A\right)}U^{\left(A\right)}\left(\varphi\right),\;\mbox{where}\label{eq:mer7}
\end{eqnarray}
\begin{equation}
U^{\left(A\right)}\left(\varphi\right)=\int_{0}^{1}\varphi\left(y\right)U^{\left(A\right)}\left(-y\right)dy,\label{eq:mer8}
\end{equation}
all operators in the RKHS $\mathscr{H}_{F}$.

We have a selfadjoint extension $A$ corresponding to $\mu=\mu^{\left(A\right)}\in Ext\left(F\right)$,
and a cyclic vector $v_{0}$:\index{representation!cyclic-} 
\begin{eqnarray}
F^{\left(A\right)}\left(t\right) & = & \left\langle v_{0},U^{\left(A\right)}\left(t\right)v_{0}\right\rangle \nonumber \\
 & = & \int_{\mathbb{R}}e^{i\lambda t}d\mu_{A}\left(\lambda\right),\mbox{ where }d\mu_{A}\left(\lambda\right)=\left\Vert P_{A}\left(d\lambda\right)v_{0}\right\Vert ^{2},\label{eq:mer9}
\end{eqnarray}
and from (\ref{eq:mer8}):
\begin{equation}
\left(T_{F}\varphi\right)\left(x\right)=\bigl(F_{\varphi}^{\left(A\right)}\bigr)\left(x\right),\;\forall\varphi\in C_{c}\left(0,1\right),\mbox{ and }\forall x\in\left(0,1\right).\label{eq:mer10}
\end{equation}

\end{svmultproof2}

\section{\label{sec:shannon}Shannon Sampling, and Bessel Frames}

The notion \textquotedblleft sampling\textquotedblright{} refers to
the study of certain function spaces, consisting of functions that
allow effective \textquotedblleft reconstruction\textquotedblright{}
from sampled values, usually from discrete subsets; and where a reconstruction
formula is available. The best known result of this nature was discovered
by Shannon (see e.g., \cite{DM76,DM70}), and Shannon\textquoteright s
variant refers to spaces of band-limited functions. \textquotedblleft Band\textquotedblright{}
in turn refers to a bounded interval in the Fourier-dual domain. For
time functions, the Fourier dual is a frequency-variable. Shannon
showed that band-limited functions allow perfect reconstruction from
samples in an arithmetic progression of time-points. After rescaling,
the sample points may be chosen to be the integers $\mathbb{Z}$;
referring to the case when the time-function is defined on $\mathbb{R}$.
Eq (\ref{eq:m-4-5}) below illustrates this point. Our purpose here
is to discuss sampling as a part of the central theme of our positive
definite (p.d.) extension analysis.

However, there is a rich literature dealing with sampling in a host
of other areas of both pure and applied mathematics.
\begin{theorem}
\label{thm:shannon}Let $F:\left(-1,1\right)\rightarrow\mathbb{C}$
be a continuous p.d. function, and 
\begin{equation}
T_{F}:L^{2}\left(0,1\right)\rightarrow\mathscr{H}_{F}\subset L^{2}\left(0,1\right)\label{eq:m-4-1}
\end{equation}
be the corresponding Mercer operator. Let $\mu\in Ext\left(F\right)$;
then the range of $T_{F}$, $ran\left(T_{F}\right)$, as a subspace
of $\mathscr{H}_{F}$, admits the following representation:
\begin{equation}
ran\left(T_{F}\right)=\left\{ \sum\nolimits _{n\in\mathbb{Z}}c_{n}f_{n}\;\big|\;\left(c_{n}\right)\in l^{2}\left(\mathbb{Z}\right),\;f_{n}\in\mathscr{H}_{F}\right\} ,\label{eq:m-4-2}
\end{equation}
where $f_{n}$ in the r.h.s. of (\ref{eq:m-4-2}) is as follows: 
\begin{equation}
f_{n}\left(x\right)=\int_{\mathbb{R}}e^{i2\pi\lambda x}\frac{\sin\pi\left(\lambda-n\right)}{\pi\left(\lambda-n\right)}e^{-i\pi\left(\lambda-n\right)}d\mu\left(\lambda\right),\quad x\in\left[0,1\right].\label{eq:m-4-3}
\end{equation}
Setting 
\begin{equation}
Sha\left(\xi\right)=e^{i\xi}\frac{\sin\xi}{\xi},\label{eq:m-5-1}
\end{equation}
then 
\begin{equation}
f_{n}\left(x\right)=\int_{\mathbb{R}}Sha\left(\pi\left(\lambda-n\right)\right)e^{i2\pi\lambda x}d\mu\left(x\right).\label{eq:m-5-2}
\end{equation}
\end{theorem}
\begin{svmultproof2}
We first show $\subseteq$ in (\ref{eq:m-4-2}). Suppose $\varphi\in L^{2}\left(0,1\right)$,
then it extends to a 1-periodic function on $\mathbb{R}$, so that
\begin{eqnarray}
\varphi\left(x\right) & = & \sum_{n\in\mathbb{Z}}\widehat{\varphi}\left(n\right)e^{i2\pi nx},\quad\forall x\in\left(0,1\right);\;\mbox{where}\label{eq:m-4-4a}\\
\widehat{\varphi}\left(\lambda\right) & = & \int_{0}^{1}e^{-i2\pi\lambda x}\varphi\left(x\right)dx,\quad\lambda\in\mathbb{R}.\label{eq:m-4-4}
\end{eqnarray}
Applying Fourier transform to (\ref{eq:m-4-4a}) yields \index{Shannon sampling}\index{representation!spectral-}
\begin{equation}
\widehat{\varphi}\left(\lambda\right)=\sum_{n\in\mathbb{Z}}\widehat{\varphi}\left(n\right)\frac{\sin\pi\left(\lambda-n\right)}{\pi\left(\lambda-n\right)}e^{-i\pi\left(\lambda-n\right)},\quad\lambda\in\mathbb{R},\label{eq:m-4-5}
\end{equation}
and by Parseval's identity, we have\index{Parseval's identity} 
\[
\int_{0}^{1}\left|\varphi\left(x\right)\right|^{2}dx=\sum_{n\in\mathbb{Z}}\left|\widehat{\varphi}\left(n\right)\right|^{2}.
\]
Note the r.h.s. in (\ref{eq:m-4-5}) uses the Shannon integral kernel
(see, e.g., \cite{KT09}.)

But, for $\mu\in Ext\left(F\right)$, we also have \index{integral kernel}\index{operator!integral-}
\begin{equation}
\left(T_{F}\varphi\right)\left(x\right)=\int_{\mathbb{R}}e^{i2\pi\lambda x}\widehat{\varphi}\left(\lambda\right)d\mu\left(\lambda\right),\label{eq:m-4-6}
\end{equation}
which holds for all $\varphi\in L^{2}\left(0,1\right)$, $x\in\left[0,1\right]$.
Now substituting (\ref{eq:m-4-5}) into the r.h.s. of (\ref{eq:m-4-6}),
we get 
\begin{eqnarray}
\left(T_{F}\varphi\right)\left(x\right) & = & \sum_{n\in\mathbb{Z}}\widehat{\varphi}\left(n\right)\int_{\mathbb{R}}e^{i2\pi\lambda x}\frac{\sin\pi\left(\lambda-n\right)}{\pi\left(\lambda-n\right)}e^{-i\pi\left(\lambda-n\right)}d\mu\left(\lambda\right)\nonumber \\
 & = & \sum_{n\in\mathbb{Z}}\widehat{\varphi}\left(n\right)\underset{=:f_{n}}{\underbrace{\int_{\mathbb{R}}e^{i2\pi\lambda x}Sha\left(\pi\left(\lambda-n\right)\right)d\mu\left(\lambda\right)}},\label{eq:m-4-6-a}
\end{eqnarray}
where ``$Sha$'' in (\ref{eq:m-4-6-a}) refers to the Shannon kernel
from (\ref{eq:m-5-1}). Thus, the desired conclusion (\ref{eq:m-4-2})
is established. (The interchange of integration and summation is justified
by (\ref{eq:m-4-6}) and Fubini.)

It remains to prove $\supseteq$ in (\ref{eq:m-4-2}). Let $\left(c_{n}\right)\in l^{2}\left(\mathbb{Z}\right)$
be given. By Parseval's identity, it follows that 
\[
\varphi\left(x\right)=\sum_{n\in\mathbb{Z}}c_{n}e^{i2\pi nx}\in L^{2}\left(0,1\right),
\]
and that (\ref{eq:m-4-6}) holds. Now the same argument (as above)
with Fubini shows that 
\[
\left(T_{F}\varphi\right)\left(x\right)=\sum_{n\in\mathbb{Z}}c_{n}f_{n}\left(x\right),\quad x\in\left[0,1\right];
\]
with $f_{n}$ given by (\ref{eq:m-4-3}).\end{svmultproof2}

\begin{remark}
While the system $\left\{ f_{n}\right\} _{n\in\mathbb{Z}}$ in (\ref{eq:m-4-2})
is explicit, it has a drawback compared to the eigenfunctions for
the Mercer operator, in that such a system is \emph{not} orthogonal. \end{remark}
\begin{example}
Let $F\left(x\right)=e^{-\left|x\right|}$, $\left|x\right|<1$, i.e.,
$F=F_{3}$ in Table \ref{tab:F1-F6}. Then the generating function
system $\left\{ f_{n}\right\} _{n\in\mathbb{Z}}$ in $\mathscr{H}_{F}$
(from Shannon sampling) is as follows:
\begin{eqnarray*}
\Re\left\{ f_{n}\right\} \left(x\right) & = & \frac{e^{x-1}+e^{-x}-2\cos\left(2\pi nx\right)}{1+\left(2\pi n\right)^{2}},\;\mbox{and}\\
\Im\left\{ f_{n}\right\} \left(x\right) & = & \frac{\left(e^{x-1}-e^{-x}\right)2\pi n-2\sin\left(2\pi nx\right)}{1+\left(2\pi n\right)^{2}},\quad\forall n\in\mathbb{Z},\:x\in\left[0,1\right].
\end{eqnarray*}
The boundary values of $f_{n}$ are as in Table \ref{tab:shan}. \index{boundary condition}
\end{example}
\renewcommand{\arraystretch}{3}

\begin{table}
\begin{tabular}[t]{|c|>{\centering}m{0.3\textwidth}|>{\centering}m{0.3\textwidth}|}
\hline 
 & $\Re\left\{ f_{n}\right\} \left(x\right)$ & $\Im\left\{ f\right\} \left(x\right)$\tabularnewline
\hline 
$x=0$ & $\dfrac{e^{-1}+1-2}{1+\left(2\pi n\right)^{2}}$ & $-\dfrac{2\pi n\left(1-e^{-1}\right)}{1+\left(2\pi n\right)^{2}}$\tabularnewline
\hline 
$x=1$ & $\dfrac{e^{-1}+1-2}{1+\left(2\pi n\right)^{2}}$ & $\dfrac{2\pi n\left(1-e^{-1}\right)}{1+\left(2\pi n\right)^{2}}$\tabularnewline
\hline 
\end{tabular}

\bigskip{}
\protect\caption{\label{tab:shan}Boundary values of the Shannon functions, s.t. $\Re\left\{ f_{n}\right\} \left(1\right)=\Re\left\{ f_{n}\right\} \left(0\right)$,
and $\Im\left\{ f_{n}\right\} \left(1\right)=-\Im\left\{ f_{n}\right\} \left(0\right)$. }
\end{table}

\renewcommand{\arraystretch}{1}
\begin{corollary}
\label{cor:shan}Let $F:\left(-1,1\right)\rightarrow\mathbb{C}$ be
continuous and positive definite, and assume $F\left(0\right)=1$.
Let $\mu\in\mathscr{M}_{+}\left(\mathbb{R}\right)$. Then the following
two conditions are equivalent:
\begin{enumerate}
\item \label{enu:m-6-1}$\mu\in Ext\left(F\right)$
\item \label{enu:m-6-2}For all $x\in\left(-1,1\right)$, 
\begin{equation}
\sum_{n\in\mathbb{Z}}\int_{\mathbb{R}}e^{i2\pi\lambda x}Sha\left(\pi\left(\lambda-n\right)\right)d\mu\left(\lambda\right)=F\left(x\right).\label{eq:m-6-1}
\end{equation}

\end{enumerate}
\end{corollary}
\begin{svmultproof2}
The implication (\ref{enu:m-6-2})$\Rightarrow$(\ref{enu:m-6-1})
follows since the l.h.s. of (\ref{eq:m-6-1}) is a continuous p.d.
function defined on all of $\mathbb{R}$, and so it is an extension
as required. Indeed, this is (\ref{eq:m-4-6-a}), with $\widehat{\varphi}\left(n\right)=1$,
for all $n\in\mathbb{Z}$. 

Conversely, assume (\ref{enu:m-6-1}) holds. Let $\varphi_{\epsilon}\xrightarrow{\;\epsilon\rightarrow0^{+}\;}\delta_{0}$
(Dirac mass at $0$) in (\ref{eq:m-4-6-a}), and then (\ref{enu:m-6-2})
follows.
\end{svmultproof2}

\index{Theorem!Bochner's-}\index{Bochner transform}\index{Bochner's Theorem}
\begin{definition}
A system of functions $\left\{ f_{n}\right\} _{n\in\mathbb{Z}}\subset\mathscr{H}_{F}$
is said to be a \emph{Bessel frame} \cite{CM13} if there is a finite
constant $A$ such that \index{Bessel frame} 
\begin{equation}
\sum_{n\in\mathbb{Z}}\left|\left\langle f_{n},\xi\right\rangle _{\mathscr{H}_{F}}\right|^{2}\leq A\left\Vert \xi\right\Vert _{\mathscr{H}_{F}}^{2},\;\forall\xi\in\mathscr{H}_{F}.\label{eq:f2-5-1}
\end{equation}
\end{definition}
\begin{theorem}
\label{thm:Liebd}Let $F:\left(-1,1\right)\rightarrow\mathbb{C}$
be continuous and p.d., and let $\left\{ f_{n}\right\} _{n\in\mathbb{Z}}\subset\mathscr{H}_{F}$
be the system (\ref{eq:m-4-3}) obtained by Shannon sampling in Fourier-domain;
then $\left\{ f_{n}\right\} $ is a Bessel frame, where we may take
$A=\lambda_{1}$ as frame bound in (\ref{eq:f2-5-1}). ($\lambda_{1}=$
the largest eigenvalue of the Mercer operator $T_{F}$.)\index{eigenvalue(s)}\index{frame!Bessel-}\end{theorem}
\begin{svmultproof2}
Let $\xi\in\mathscr{H}_{F}$, then 
\begin{eqnarray*}
\sum_{n\in\mathbb{Z}}\left|\left\langle f_{n},\xi\right\rangle _{\mathscr{H}_{F}}\right|^{2} & = & \sum_{n\in\mathbb{Z}}\left|\left\langle Te_{n},\xi\right\rangle _{\mathscr{H}_{F}}\right|^{2}\\
 & \underset{\text{Cor. \ensuremath{\left(\ref{cor:mer2}\right)}}}{=} & \sum_{n\in\mathbb{Z}}\left|\left\langle e_{n},\xi\right\rangle _{L^{2}\left(0,1\right)}\right|^{2}\\
 & \underset{\text{(Parseval)}}{=} & \left\Vert \xi\right\Vert _{L^{2}\left(0,1\right)}^{2}\\
 & \underset{\text{(Thm. \ensuremath{\left(\ref{thm:mer4}\right)})}}{\leq} & \lambda_{1}\left\Vert \xi\right\Vert _{\mathscr{H}_{F}}^{2}
\end{eqnarray*}
which is the desired conclusion.

At the start of the estimate above we used the Fourier basis $e_{n}\left(x\right)=e^{i2\pi nx}$,
$n\in\mathbb{Z}$, an ONB in $L^{2}\left(0,1\right)$; and the fact
that $f_{n}=T_{F}\left(e_{n}\right)$, $n\in\mathbb{Z}$; see the
details in the proof of Theorem \ref{thm:shannon}. \end{svmultproof2}

\begin{corollary}
For every $f\in L^{2}\left(\Omega\right)\cap\mathscr{H}_{F}$, and
every $x\in\overline{\Omega}$ (including boundary points), we have
the following estimate:
\begin{equation}
\left|f\left(x\right)\right|\leq\left\Vert f\right\Vert _{\mathscr{H}_{F}}\label{eq:m-1-4}
\end{equation}
(Note that this is independent of $x$ and of $f$.)\end{corollary}
\begin{svmultproof2}
The estimate in (\ref{eq:m-1-4}) follows from Lemma \ref{lem:F-bd};
and (\ref{eq:mer-4-2}) in part (\ref{enu:m-1-2}) of Theorem \ref{thm:mer4}. \end{svmultproof2}

\begin{corollary}
Let $\left\{ \xi_{n}\right\} $, $\left\{ \lambda_{n}\right\} $,
$T_{F}$, $\mathscr{H}_{F}$ and $L^{2}\left(\Omega\right)$ be as
above; and assume $\lambda_{1}\geq\lambda_{2}\geq\cdots$; then we
have the following details for operator norms 
\begin{eqnarray}
\left\Vert T_{F}^{-1}\Big|_{span\left\{ \xi_{k}:k=1,\ldots,N\right\} }\right\Vert _{\mathscr{H}_{F}\rightarrow\mathscr{H}_{F}} & = & \left\Vert T_{F}^{-1}\Big|_{span\left\{ \xi_{k}:k=1,\ldots,N\right\} }\right\Vert _{L^{2}\left(\Omega\right)\rightarrow L^{2}\left(\Omega\right)}\label{eq:mer-1-5}\\
 & = & \frac{1}{\lambda_{N}}\longrightarrow\infty,\mbox{ as }N\rightarrow\infty.\nonumber 
\end{eqnarray}

\end{corollary}
\index{operator!-norm}

\section{Application: The Case of $F_{2}$ and Rank-1 Perturbations }

In Theorem \ref{thm:mer1} below, we revisit the example $F_{2}$
(in $\left|x\right|<\frac{1}{2}$) from Table \ref{tab:F1-F6}. This
example has a number of intriguing properties that allow us to compute
the eigenvalues and the eigenvectors for the Mercer operator from
Lemma \ref{lem:mer1}.
\begin{theorem}
\label{thm:mer1}Set $E\left(x,y\right)=x\wedge y=\min\left(x,y\right)$,
$x,y\in\left(0,\frac{1}{2}\right)$, and 
\begin{equation}
\left(T_{E}\varphi\right)\left(x\right)=\int_{0}^{\frac{1}{2}}\varphi\left(y\right)x\wedge y\:dy\label{eq:thm:mer1-1}
\end{equation}
then the spectral resolution of $T_{E}$ in $L^{2}\left(0,\frac{1}{2}\right)$
is as follows:
\begin{equation}
E\left(x,y\right)=\sum_{n=1}^{\infty}\frac{4}{\left(\pi\left(2n-1\right)\right)^{2}}\sin\left(\left(2n-1\right)\pi x\right)\sin\left(\left(2n-1\right)\pi y\right)\label{eq:thm:mer1-2}
\end{equation}
for all $\forall\left(x,y\right)\in\left(0,\frac{1}{2}\right)\times\left(0,\frac{1}{2}\right)$.

Setting
\begin{equation}
u\left(x\right)=\left(T_{E}\varphi\right)\left(x\right),\;\mbox{for }\varphi\in C_{c}^{\infty}\left(0,\tfrac{1}{2}\right),\;\mbox{we get}\label{eq:thm:mer1-3}
\end{equation}
\begin{eqnarray}
u'\left(x\right) & = & \int_{x}^{\frac{1}{2}}\varphi\left(y\right)dy,\mbox{ and}\label{eq:thm:mer1-4}\\
u''\left(x\right) & = & -\varphi\left(x\right);\label{eq:mer-5-1}
\end{eqnarray}
moreover, $u$ satisfies the boundary condition\index{boundary condition}
\begin{equation}
\begin{cases}
u\left(0\right) & =0\\
u'\left(\tfrac{1}{2}\right) & =0
\end{cases}\label{eq:mer-6-1}
\end{equation}

\end{theorem}
Note that $E\left(x,y\right)$ is a p.d. kernel, but not a positive
definite function in the sense of Definition \ref{def:pdf}. \index{positive definite}

In particular, $E$ is a function of two variables, as opposed to
one. The purpose of Theorem \ref{thm:MerF2-1}, and Lemmas \ref{lem:F2kernel},
and \ref{lem:F2kernel2} is to show that the Mercer operator $T_{F}$
defined from $F=F_{2}$ (see Table \ref{tab:F1-F6}) is a rank-1 perturbation\index{perturbation}
of the related operator $T_{E}$ defined from $E\left(x,y\right)$.
The latter is of significance in at least two ways: $T_{E}$ has an
explicit spectral representation, and $E\left(x,y\right)$ is the
covariance\index{covariance} kernel for Brownian motion;\index{Brownian motion}
see also Figure \ref{fig:bm1}. By contrast, we show in Lemma \ref{lem:bbridge}
that the Mercer operator $T_{F}$ is associated with pinned Brownian
motion. \index{representation!spectral-}

The connection between the two operators $T_{F}$ and $T_{E}$ is
reflecting a more general feature of boundary conditions for domains
$\Omega$ in Lie groups; -- a topic we consider in Chapters \ref{chap:types},
and \ref{chap:spbd}.\index{operator!rank-one-}

Rank-one perturbations play a role in spectral theory in different
problems; see e.g., \cite{Yos12,DJ10,Ion01,DRSS94,TW86}.
\begin{svmultproof2}
(Theorem \ref{thm:mer1}) We verify directly that (\ref{eq:thm:mer1-4})-(\ref{eq:mer-5-1})
hold. 

Consider the Hilbert space $L^{2}\left(0,\frac{1}{2}\right)$. The
operator $\triangle_{E}:=T_{E}^{-1}$ is a selfadjoint extension\index{selfadjoint extension}
of $-\left(\frac{d}{dx}\right)^{2}\big|_{C_{c}^{\infty}\left(0,\frac{1}{2}\right)}$
in $L^{2}\left(0,\frac{1}{2}\right)$; and under the ONB 
\[
f_{n}\left(x\right)=2\sin\left(\left(2n-1\right)\pi x\right),\;n\in\mathbb{N},
\]
and the boundary condition (\ref{eq:mer-6-1}), $\triangle_{E}$ is
diagonalized as \index{boundary condition} 
\begin{equation}
\triangle_{E}f_{n}=\left(\left(2n-1\right)\pi\right)^{2}f_{n},\;n\in\mathbb{N}.\label{eq:thm:mer1-5}
\end{equation}
We conclude that
\begin{align}
\triangle_{E} & =\sum_{n=1}^{\infty}\left(\left(2n-1\right)\pi\right)^{2}\left|f_{n}\left\rangle \right\langle f_{n}\right|\label{eq:thm:mer1-6}\\
T_{E} & =\triangle_{E}^{-1}=\sum_{n=1}^{\infty}\frac{1}{\left(\left(2n-1\right)\pi\right)^{2}}\left|f_{n}\left\rangle \right\langle f_{n}\right|\label{eq:thm:mer1-7}
\end{align}
where $P_{n}:=\left|f_{n}\left\rangle \right\langle f_{n}\right|=$
Dirac's rank-1 projection\index{projection} in $L^{2}\left(0,\frac{1}{2}\right)$,
and\index{operator!rank-one-} 
\begin{eqnarray}
\left(P_{n}\varphi\right)\left(x\right) & = & \left\langle f_{n},\varphi\right\rangle f_{n}\left(x\right)\nonumber \\
 & = & \left(\int_{0}^{\frac{1}{2}}f_{n}\left(y\right)\varphi\left(y\right)dy\right)f_{n}\left(x\right)\nonumber \\
 & = & 4\sin\left(\left(2n-1\right)\pi x\right)\int_{0}^{\frac{1}{2}}\sin\left(\left(2n-1\right)\pi y\right)\varphi\left(y\right)dy.\label{eq:thm:mer1-8}
\end{eqnarray}
Combing (\ref{eq:thm:mer1-7}) and (\ref{eq:thm:mer1-8}), we get
\begin{equation}
\left(T_{E}\varphi\right)\left(x\right)=\sum_{n=1}^{\infty}\frac{4\sin\left(\left(2n-1\right)\pi x\right)}{\left(\left(2n-1\right)\pi\right)^{2}}\int_{0}^{\frac{1}{2}}\sin\left(\left(2n-1\right)\pi y\right)\varphi\left(y\right)dy.\label{eq:thm:mer1-9}
\end{equation}

Note the normalization considered in (\ref{eq:thm:mer1-2}) and (\ref{eq:thm:mer1-9})
is consistent with the condition:
\[
\sum_{n=1}^{\infty}\lambda_{n}=trace\left(T_{E}\right)
\]
in Lemma \ref{lem:mer1} for the Mercer eigenvalues $\left(\lambda_{n}\right)_{n\in\mathbb{N}}$.
Indeed,
\begin{gather*}
trace\left(T_{E}\right)=\int_{0}^{\frac{1}{2}}x\wedge x\:dx=\sum_{n=1}^{\infty}\frac{1}{\left(\left(2n-1\right)\pi\right)^{2}}=\frac{1}{8}.
\end{gather*}
 \end{svmultproof2}

\begin{theorem}
\label{thm:MerF2-1}Set $F\left(x-y\right)=1-\left|x-y\right|$, $x,y\in\left(0,\frac{1}{2}\right)$;
$K^{\left(E\right)}\left(x,y\right)=x\wedge y=\min\left(x,y\right)$,
and let 
\begin{equation}
\left(T_{E}\varphi\right)\left(x\right)=\int_{0}^{\frac{1}{2}}\varphi\left(y\right)K^{\left(E\right)}\left(x,y\right)dy;\label{eq:f2-1}
\end{equation}
then 
\begin{equation}
K^{\left(E\right)}\left(x,y\right)=\sum_{n=1}^{\infty}\frac{4\sin\left(\left(2n-1\right)\pi x\right)\sin\left(\left(2n-1\right)\pi y\right)}{\left(\pi\left(2n-1\right)\right)^{2}}\label{eq:f2-2}
\end{equation}
and 
\begin{equation}
F\left(x-y\right)=1-x-y+2\sum_{n=1}^{\infty}\frac{4\sin\left(\left(2n-1\right)\pi x\right)\sin\left(\left(2n-1\right)\pi y\right)}{\left(\pi\left(2n-1\right)\right)^{2}}.\label{eq:f2-3}
\end{equation}
That is, 
\begin{equation}
F\left(x-y\right)=1-x-y+2K^{\left(E\right)}\left(x,y\right).\label{eq:f2-2-1}
\end{equation}
\end{theorem}
\begin{remark}
Note the trace normalization $trace\left(T_{F}\right)=\frac{1}{2}$
holds. Indeed, from (\ref{eq:f2-3}), we get 
\begin{eqnarray*}
trace\left(T_{F}\right) & = & \int_{0}^{\frac{1}{2}}\left(1-2x+2\sum_{n=1}^{\infty}\frac{4\sin^{2}\left(\left(2n-1\right)\pi x\right)}{\left(\pi\left(2n-1\right)\right)^{2}}\right)dx\\
 & = & \frac{1}{2}-\frac{1}{4}+2\sum_{n=1}^{\infty}\frac{1}{\left(\pi\left(2n-1\right)\right)^{2}}\\
 & = & \frac{1}{2}-\frac{1}{4}+2\cdot\frac{1}{8}=\frac{1}{2};
\end{eqnarray*}
where $\frac{1}{2}$ on the r.h.s. is the right endpoint of the interval
$\left[0,\frac{1}{2}\right]$. \end{remark}
\begin{svmultproof2}
The theorem follows from lemma \ref{lem:F2kernel} and lemma \ref{lem:F2kernel2}.\end{svmultproof2}

\begin{lemma}
\label{lem:F2kernel}Consider the two integral kernels:
\begin{eqnarray}
F\left(x-y\right) & = & 1-\left|x-y\right|,\;x,y\in\Omega;\label{eq:f2-4}\\
K^{\left(E\right)}\left(x,y\right) & = & x\wedge y=\min\left(x,y\right),\;x,y\in\Omega.\label{eq:f2-5}
\end{eqnarray}
We take $\Omega=\left(0,\frac{1}{2}\right)$. Then
\begin{equation}
F\left(x-y\right)=2K^{\left(E\right)}\left(x,y\right)+1-x-y;\mbox{ and}\label{eq:f2-6}
\end{equation}
\begin{equation}
\left(F_{x}-2K_{x}^{\left(E\right)}\right)''\left(y\right)=-2\delta\left(0-y\right).\label{eq:f2-7}
\end{equation}
\end{lemma}
\begin{svmultproof2}
A calculation yields:
\[
x\wedge y=\frac{x+y-1+F\left(x-y\right)}{2},
\]
and therefore, solving for $F\left(x-y\right)$, we get (\ref{eq:f2-6}).

To prove (\ref{eq:f2-7}), we calculate the respective Schwartz derivatives
(in the sense of distributions). Let $H_{x}=$ the Heaviside function
at $x$, with $x$ fixed. Then \index{Heaviside function}\index{derivative!Schwartz-}
\begin{eqnarray}
\left(K_{x}^{\left(E\right)}\right)' & = & H_{0}-H_{x},\mbox{ and}\label{eq:f2-8}\\
\left(K_{x}^{\left(E\right)}\right)'' & = & \delta_{0}-\delta_{x}\label{eq:f2-9}
\end{eqnarray}
combining (\ref{eq:f2-9}) with $\left(F_{x}\right)''=-2\delta_{x}$,
we get 
\begin{equation}
\left(F_{x}-2K_{x}^{\left(E\right)}\right)''=-2\delta_{x}-2\left(\delta_{0}-\delta_{x}\right)=-2\delta_{0}\label{eq:f2-10}
\end{equation}
which is the desired conclusion (\ref{eq:f2-7}).\end{svmultproof2}

\begin{remark}
From (\ref{eq:f2-6}), we get the following formula for three integral
operators
\begin{equation}
T_{F}=2T_{E}+L,\mbox{ where}\label{eq:f2-4-1}
\end{equation}
\begin{equation}
\left(L\varphi\right)\left(x\right)=\int_{0}^{\frac{1}{2}}\varphi\left(y\right)\left(1-x-y\right)dy.\label{eq:f2-4-2}
\end{equation}
Now in Lemma \ref{lem:F2kernel}, we diagonalize $T_{E}$, but the
two Hermitian operators on the r.h.s. in (\ref{eq:f2-4-1}), do \emph{not}
commute. But the perturbation\index{perturbation} $L$ in (\ref{eq:f2-4-1})
is still relatively harmless; it is a rank-1 operator with just one
eigenfunction: $\varphi\left(x\right)=a+bx$, where $a$ and $b$
are determined from $\left(L\varphi\right)\left(x\right)=\lambda\varphi\left(x\right)$;
and 
\begin{eqnarray*}
\left(L\varphi\right)\left(x\right) & = & \left(1-x\right)\frac{a}{2}-\frac{1}{8}\left(a+\frac{b}{3}\right)\\
 & = & \left(\frac{3}{8}a-\frac{b}{24}\right)-\left(\frac{a}{2}\right)x=\lambda\left(a+bx\right)
\end{eqnarray*}
thus the system of equations\index{Hermitian} 
\[
\begin{cases}
\left(\frac{3}{8}-\lambda\right)a-\frac{1}{24}b=0\\
-\frac{1}{2}a-\lambda b=0.
\end{cases}
\]
It follows that 
\begin{eqnarray*}
\lambda & = & \frac{1}{48}\left(9+\sqrt{129}\right)\\
b & = & -\frac{1}{2}\left(\sqrt{129}-9\right)a.
\end{eqnarray*}

\end{remark}

\begin{remark}[A dichotomy for integral kernel operators]
 Note the following dichotomy for the two integral kernel-operators:
\index{operator!rank-one-}
\begin{enumerate}[resume]
\item one with the kernel $L\left(x,y\right)=1-x-y$, a rank-one operator;
and the other $T_{F}$ corresponding to $F=F_{2}$, i.e., with kernel
$F\left(x\lyxmathsym{\textendash}y\right)=1-\left|x\lyxmathsym{\textendash}y\right|$. 
\end{enumerate}

And by contrast:
\begin{enumerate}[resume]
\item $T_{F}$ is an infinite dimensional integral-kernel operator. Denoting
both the kernel $L$, and the rank-one operator, by the same symbol,
we then establish the following link between the two integral operators:
The two operators $T_{F}$ and $L$ satisfy the identity $T_{F}=L+2T_{E}$,
where $T_{E}$ is the integral kernel-operator defined from the covariance\index{covariance}
function of Brownian motion. For more applications of rank-one perturbations,
see e.g., \cite{Yos12,DJ10,Ion01,DRSS94,TW86}.
\end{enumerate}
\end{remark}
\index{integral kernel}\index{operator!rank-one-}

\index{operator!integral-}
\begin{lemma}
\label{lem:F2kernel2}Set 
\begin{equation}
\left(T_{E}\varphi\right)\left(x\right)=\int_{0}^{\frac{1}{2}}K^{\left(E\right)}\left(x,y\right)\varphi\left(y\right)dy,\;\varphi\in L^{2}\left(0,\tfrac{1}{2}\right),x\in\left(0,\tfrac{1}{2}\right)\label{eq:f2-11}
\end{equation}
Then $s_{n}\left(x\right):=\sin\left(\left(2n-1\right)\pi x\right)$
satisfies 
\begin{equation}
T_{E}s_{n}=\frac{1}{\left(\left(2n-1\right)\pi\right)^{2}}s_{n};\label{eq:f2-12}
\end{equation}
and we have
\begin{equation}
K^{\left(E\right)}\left(x,y\right)=\sum_{n\in\mathbb{N}}\frac{4}{\left(\left(2n-1\right)\pi\right)^{2}}\sin\left(\left(2n-1\right)\pi x\right)\sin\left(\left(2n-1\right)\pi y\right)\label{eq:f2-13}
\end{equation}
\end{lemma}
\begin{svmultproof2}
Let $\Omega=\left(0,\frac{1}{2}\right)$. Setting 
\begin{equation}
s_{n}\left(x\right):=\sin\left(\left(2n-1\right)\pi x\right);\;x\in\Omega,n\in\mathbb{N}.\label{eq:f2-14}
\end{equation}
Using (\ref{eq:f2-9}), we get 
\[
\left(T_{E}s_{n}\right)\left(x\right)=\frac{1}{\left(\left(2n-1\right)\pi\right)^{2}}s_{n}\left(x\right),\;x\in\Omega,n\in\mathbb{N},
\]
where $T_{E}$ is the integral operator with kernel $K^{\left(E\right)}$,
and $s_{n}$ is as in (\ref{eq:f2-14}). Since 
\[
\int_{0}^{\frac{1}{2}}\sin^{2}\left(\left(2n-1\right)\pi x\right)dx=\frac{1}{4}
\]
the desired formula (\ref{eq:f2-13}) holds.\end{svmultproof2}

\begin{corollary}
Let $T_{F}$ and $T_{E}$ be the integral operators in $L^{2}\left(0,\frac{1}{2}\right)$
defined in the lemmas; i.e., $T_{F}$ with kernel $F\left(x-y\right)$;
and $T_{E}$ with kernel $x\wedge y$. Then the selfadjoint operator
$\left(T_{F}-2T_{E}\right)^{-1}$ is well-defined, and it is the Friedrichs
extension of 
\[
-\frac{1}{2}\left(\frac{d}{dx}\right)^{-1}\Big|_{C_{c}^{\infty}\left(0,\frac{1}{2}\right)}
\]
as a Hermitian and semibounded operator in $L^{2}\left(0,\frac{1}{2}\right)$.
\index{Hermitian}
\end{corollary}
\index{operator!selfadjoint}\index{Friedrichs extension}\index{operator!semibounded}\index{operator!integral-}
\begin{svmultproof2}
Formula (\ref{eq:f2-7}) in the lemma translates into 
\begin{equation}
\left(T_{F}\varphi-2T_{E}\varphi\right)''=-2\varphi\label{eq:f2-15}
\end{equation}
for all $\varphi\in C_{c}^{\infty}\left(0,\frac{1}{2}\right)$. Hence
$\left(T_{F}-2T_{E}\right)^{-1}$ is well-defined as an unbounded
selfadjoint operator in $L^{2}\left(0,\frac{1}{2}\right)$; and 
\[
\left(T_{F}-2T_{E}\right)^{-1}\varphi=-\frac{1}{2}\varphi'',\;\forall\varphi\in C_{c}^{\infty}\left(0,\tfrac{1}{2}\right).
\]
Since the Friedrichs extension is given by Dirichlet boundary condition
in $L^{2}\left(0,\frac{1}{2}\right)$, the result follows. \index{boundary condition}
\index{Dirichlet boundary condition}
\end{svmultproof2}

\section{\label{sec:Green}Positive Definite Functions, Green's Functions,
and Boundary}

In this section, we consider a correspondence and interplay between
a class of boundary value problems on the one hand, and spectral theoretic
properties of extension operators on the other.

Fix a bounded domain $\Omega\subset\mathbb{R}^{n}$, open and connected.
Let $F:\Omega-\Omega\rightarrow\mathbb{C}$ be a continuous positive
definite (p.d.) function. We consider a special case when $F$ occurs
as the Green's function of certain linear operator. \index{positive definite}\index{operator!selfadjoint}\index{Schwartz!test function}\index{Green's function}
\begin{lemma}
\label{lem:gr-1}Let $\mathscr{D}$ be a Hilbert space, a Fréchet
space or an LF-space (see \cite{Tre06}), such that $\mathscr{D}\underset{j}{\hookrightarrow}L^{2}\left(\Omega\right)$;
and such that the inclusion mapping $j$ is continuous relative to
the respective topologies on $\mathscr{D}$, and on $L^{2}\left(\Omega\right)$.
Let $\mathscr{D}^{*}:=$ the dual of $\mathscr{D}$ when $\mathscr{D}$
is given its Fréchet (LF, or Hilbert) topology; then there is a natural
``inclusion'' mapping $j^{*}$ from $L^{2}\left(\Omega\right)$
to $\mathscr{D}^{*}$, i.e., we get 
\begin{equation}
\mathscr{D}\underset{j}{\hookrightarrow}L^{2}\left(\Omega\right)\underset{j^{*}}{\hookrightarrow}\mathscr{D}^{*}.\label{eq:gr-2-1}
\end{equation}
\end{lemma}
\begin{svmultproof2}
It is immediate from the assumptions, and the fact that $L^{2}\left(\Omega\right)$
is its own dual. See also \cite{Tre06}.\end{svmultproof2}

\begin{remark}
In the following we shall use Lemma \ref{lem:gr-1} in two cases:
\begin{enumerate}[leftmargin=*]
\item Let $A$ be a selfadjoint operator (unbounded in the non-trivial
cases) acting in $L^{2}\left(\Omega\right)$; and with dense domain.
For $\mathscr{D}=\mathscr{D}_{A}$, we may choose the domain of $A$
with its graph topology. \index{operator!domain of-}
\item Let $\mathscr{D}$ be a space of Schwartz test functions, e.g., $C_{c}^{\infty}\left(\Omega\right)$,
given its natural LF-topology, see \cite{Tre06}; then the inclusion
\begin{equation}
C_{c}^{\infty}\left(\Omega\right)\underset{j}{\hookrightarrow}L^{2}\left(\Omega\right)\label{eq:gr-2-2}
\end{equation}
satisfies the condition in Lemma \ref{lem:gr-1}.
\end{enumerate}
\end{remark}
\begin{corollary}
\label{cor:gr-1}Let $\mathscr{D}\subset L^{2}\left(\Omega\right)$
be a subspace satisfying the conditions in Lemma \ref{lem:gr-1};
and consider the triple of spaces (\ref{eq:gr-2-1}); then the inner
product in $L^{2}\left(\Omega\right)$, here denoted $\left\langle \cdot,\cdot\right\rangle _{2}$,
extends by closure to a sesquilinear function $\left\langle \cdot,\cdot\right\rangle $
(which we shall also denote by $\left\langle \cdot,\cdot\right\rangle _{2}$):
\index{sesquilinear function}
\begin{equation}
\left\langle \cdot,\cdot\right\rangle :L^{2}\left(\Omega\right)\times\mathscr{D}^{*}\rightarrow\mathbb{C}.\label{eq:gr-2-3}
\end{equation}
\end{corollary}
\begin{svmultproof2}
This is a standard argument based on dual topologies; see \cite{Tre06}.\end{svmultproof2}

\begin{example}[Application]
If $\mathscr{D}=C_{c}^{\infty}\left(\Omega\right)$ in (\ref{eq:gr-2-1}),
then $\mathscr{D}^{*}=$ the space of all Schwartz-distributions on
$\Omega$, including the Dirac masses. Referring to (\ref{eq:gr-2-3}),
we shall write $\left\langle \delta_{x},f\right\rangle _{2}$ to mean
$f\left(x\right)$, when $f\in C(\overline{\Omega})\cap L^{2}(\Omega)$. 
\end{example}
Adopting the constructions from Lemma \ref{lem:gr-1} and Corollary
\ref{cor:gr-1}, we now turn to calculus of positive definite functions: 
\begin{definition}
If $F:\Omega-\Omega\rightarrow\mathbb{C}$ is a function, or a distribution,
then we say that $F$ is positive definite iff\index{positive definite!-distribution}
\begin{equation}
\left\langle F,\overline{\varphi}\otimes\varphi\right\rangle \geq0\label{eq:gr-1-1}
\end{equation}
for all $\varphi\in C_{c}^{\infty}\left(\Omega\right)$. The meaning
of (\ref{eq:gr-1-1}) is the distribution $K_{F}:=F\left(x-y\right)$
acting on $\left(\overline{\varphi}\otimes\varphi\right)\left(x,y\right):=\overline{\varphi\left(x\right)}\varphi\left(y\right)$,
$x,y\in\Omega$. 
\end{definition}
Let 
\begin{equation}
\triangle:=\sum_{j=1}^{n}\Bigl(\frac{\partial}{\partial x_{j}}\Bigr)^{2}\label{eq:gr-1-2}
\end{equation}
and consider an open domain $\Omega\subset\mathbb{R}^{n}$.

In $\mathscr{H}_{F}$, set 
\begin{equation}
D_{j}^{\left(F\right)}\left(F_{\varphi}\right):=F_{\frac{\partial\varphi}{\partial x_{j}}},\;\varphi\in C_{c}^{\infty}\left(\Omega\right),\;j=1,\ldots,n.\label{eq:gr-1-3}
\end{equation}
Then this is a system of commuting skew-Hermitian operators with dense
domain in $\mathscr{H}_{F}$. \index{skew-Hermitian operator; also called skew-symmetric}
\begin{lemma}
Let $F:\Omega-\Omega\rightarrow\mathbb{C}$ be a positive definite
function (or a distribution); and set\index{distribution!-derivation}

\begin{equation}
M:=-\triangle F\label{eq:gr-1-4}
\end{equation}
where $\triangle F$ on the r.h.s. in (\ref{eq:gr-1-4}) is in the
sense of distributions. Then $M$ is also positive definite and 
\begin{equation}
\left\langle M_{\varphi},M_{\psi}\right\rangle _{\mathscr{H}_{M}}=\sum_{j=1}^{n}\left\langle D_{j}^{\left(F\right)}F_{\varphi},D_{j}^{\left(F\right)}F_{\psi}\right\rangle _{\mathscr{H}_{F}}\label{eq:gr-1-5}
\end{equation}
for all $\varphi,\psi\in C_{c}^{\infty}\left(\Omega\right)$. In particular,
setting $\varphi=\psi$ in (\ref{eq:gr-1-5}), we have 
\begin{equation}
\Bigl\Vert M_{\varphi}\Bigr\Vert_{\mathscr{H}_{M}}^{2}=\sum_{j=1}^{k}\Bigl\Vert D_{j}^{\left(F\right)}F_{\varphi}\Bigr\Vert_{\mathscr{H}_{F}}^{2}.\label{eq:gr-1-6}
\end{equation}
\end{lemma}
\begin{svmultproof2}
We must show that $M$ satisfies (\ref{eq:gr-1-1}), i.e., that 
\begin{equation}
\left\langle M,\overline{\varphi}\otimes\varphi\right\rangle \geq0;\label{eq:gr-1-7}
\end{equation}
and moreover that (\ref{eq:gr-1-5}), or equivalently (\ref{eq:gr-1-4}),
holds.

For l.h.s of (\ref{eq:gr-1-7}), we have 
\begin{eqnarray*}
\left\langle M,\overline{\varphi}\otimes\varphi\right\rangle  & = & \left\langle -\triangle F,\overline{\varphi}\otimes\varphi\right\rangle =-\sum_{j=1}^{n}\left\langle \Bigl(\frac{\partial}{\partial x_{j}}\Bigr)^{2}F,\overline{\varphi}\otimes\varphi\right\rangle \\
 & = & \sum_{j=1}^{n}\left\langle F,\overline{\frac{\partial\varphi}{\partial x_{j}}}\otimes\frac{\partial\varphi}{\partial x_{j}}\right\rangle \underset{\left(\text{by \ensuremath{\left(\ref{eq:gr-1-3}\right)}}\right)}{=}\sum_{j=1}^{n}\Bigl\Vert D_{j}^{\left(F\right)}\left(F_{\varphi}\right)\Bigr\Vert_{\mathscr{H}_{F}}^{2}\geq0,
\end{eqnarray*}
using the action of $\frac{\partial}{\partial x_{j}}$ in the sense
of distributions. This is the desired conclusion.\end{svmultproof2}

\begin{example}
For $n=1$, consider the functions $F_{2}$ and $F_{3}$ from Table
\ref{tab:F1-F6}. 
\begin{enumerate}
\item Let $F=F_{2}$, $\Omega=\left(-\frac{1}{2},\frac{1}{2}\right)$, then
\begin{equation}
M=-F''=2\delta\label{eq:gr-1-8}
\end{equation}
where $\delta$ is the Dirac mass at $x=0$, i.e., $\delta=\delta\left(x-0\right)$. 
\item Let $F=F_{3}$, $\Omega=\left(-1,1\right)$, then 
\begin{equation}
M=-F''=2\delta-F\label{eq:gr-1-9}
\end{equation}
 \end{enumerate}
\begin{svmultproof2}
The proof of the assertions in the two examples follows directly from
Sections \ref{sub:F2} and \ref{sub:F3}.
\end{svmultproof2}

\end{example}
Now we return to the p.d. function $F:\Omega-\Omega\rightarrow\mathbb{C}$.
Suppose $A:L^{2}\left(\Omega\right)\rightarrow L^{2}\left(\Omega\right)$
is an unbounded positive linear operator, i.e., $A\geq c>0$, for
some constant $c$. Further assume that $A^{-1}$ has the integral
kernel (Green's function) $F$, i.e.,\index{Green's function}\index{integral kernel}\index{operator!integral-}
\begin{equation}
\left(A^{-1}f\right)\left(x\right)=\int_{\Omega}F\left(x-y\right)f\left(y\right)dy,\;\forall f\in L^{2}\left(\Omega\right).\label{eq:gr-1}
\end{equation}
For all $x\in\Omega$, define 
\begin{equation}
F_{x}\left(\cdot\right):=F\left(x-\cdot\right).\label{eq:gr-2}
\end{equation}

Here $F_{x}$ is the fundamental solution to the following equation
\[
Au=f
\]
where $u\in dom\left(A\right)$, and $f\in L^{2}\left(\Omega\right)$.
Hence, in the sense of distribution, we have
\begin{eqnarray*}
AF_{x}\left(\cdot\right) & = & \delta_{x}\\
 & \Updownarrow\\
A\left(\int_{\Omega}F\left(x,y\right)f\left(y\right)dy\right) & = & \int\left(AF_{x}\left(y\right)\right)f\left(y\right)dy\\
 & = & \int\delta_{x}\left(y\right)f\left(y\right)dy\\
 & = & f\left(x\right).
\end{eqnarray*}
 Note that $A^{-1}\geq0$ iff $F$ is a p.d. kernel. 

Let $\mathscr{H}_{A}=$ the completion of $C_{c}^{\infty}\left(\Omega\right)$
in the bilinear form 
\begin{equation}
\left\langle f,g\right\rangle _{A}:=\left\langle Af,g\right\rangle _{2};\label{eq:gr-3}
\end{equation}
where the r.h.s. extends the inner product in $L^{2}\left(\Omega\right)$
as in (\ref{eq:gr-2-3}).

\index{distribution!-solution}

\index{distribution!-derivation}\index{fundamental solution}\index{Hilbert space}
\begin{lemma}
$\mathscr{H}_{A}$ is a RKHS and the reproducing kernel is $F_{x}$.
\index{RKHS}\end{lemma}
\begin{svmultproof2}
Since $A\geq c>0$, in the usual ordering of Hermitian operator, (\ref{eq:gr-3})
is a well-defined inner product, so $\mathscr{H}_{A}$ is a Hilbert
space. For the reproducing property, we check that 
\[
\left\langle F_{x},g\right\rangle _{A}=\left\langle AF_{x},g\right\rangle _{2}=\left\langle \delta_{x},g\right\rangle _{2}=g\left(x\right).
\]
\end{svmultproof2}

\begin{lemma}
Let $\mathscr{H}_{F}$ be the RKHS corresponding to $F$, i.e., the
completion of $span\left\{ F_{x}\::\:x\in\Omega\right\} $ in the
inner product 
\begin{equation}
\left\langle F_{y},F_{x}\right\rangle _{F}:=F_{x}\left(y\right)=F\left(x-y\right)\label{eq:gr-4}
\end{equation}
extending linearly. Then we have the isometric embedding $\mathscr{H}_{F}\hookrightarrow\mathscr{H}_{A}$,
via the map,
\begin{equation}
F_{x}\mapsto F_{x}.\label{eq:gr-5}
\end{equation}
\end{lemma}
\begin{svmultproof2}
We check directly that 
\begin{align*}
\left\Vert F_{x}\right\Vert _{F}^{2} & =\left\langle F_{x},F_{x}\right\rangle _{F}=F_{x}\left(x\right)=F\left(0\right)\\
\left\Vert F_{x}\right\Vert _{A}^{2} & =\left\langle F_{x},F_{x}\right\rangle _{A}=\left\langle AF_{x},F_{x}\right\rangle _{L^{2}}=\left\langle \delta_{x},F_{x}\right\rangle _{L^{2}}=F_{x}\left(x\right)=F\left(0\right).
\end{align*}
\end{svmultproof2}

\begin{remark}
Now consider $\mathbb{R}$, and let $\Omega=\left(0,a\right)$. Recall
the \emph{Mercer operator} 
\[
T_{F}:L^{2}\left(\Omega\right)\rightarrow L^{2}\left(\Omega\right),\mbox{ by}
\]
\begin{eqnarray}
\left(T_{F}g\right)\left(x\right) & := & \int_{0}^{a}F_{x}\left(y\right)g\left(y\right)dy\label{eq:gr-10}\\
 & = & \left\langle F_{x},g\right\rangle _{2},\;\forall g\in L^{2}\left(0,a\right).\nonumber 
\end{eqnarray}
By Lemma \ref{lem:mer1}, $T_{F}$ can be diagonalized in $L^{2}\left(0,a\right)$
by
\[
T_{F}\xi_{n}=\lambda_{n}\xi_{n},\;\lambda_{n}>0
\]
where $\left\{ \xi_{n}\right\} _{n\in\mathbb{N}}$ is an ONB in $L^{2}\left(0,a\right)$;
further $\xi_{n}\subset\mathscr{H}_{F}$, for all $n\in\mathbb{N}$.

From (\ref{eq:gr-10}), we then have 
\begin{equation}
\left\langle F_{x},\xi_{n}\right\rangle _{2}=\lambda_{n}\xi_{n}\left(x\right).\label{eq:gr-11}
\end{equation}
Applying $A$ on both sides of (\ref{eq:gr-11}) yields 
\begin{eqnarray*}
l.h.s.\left(\ref{eq:gr-11}\right) & = & \left\langle AF_{x},\xi_{n}\right\rangle _{2}=\left\langle \delta_{x},\xi_{n}\right\rangle _{2}=\xi_{n}\left(x\right)\\
r.h.s.\left(\ref{eq:gr-11}\right) & = & \lambda_{n}\left(A\xi_{n}\right)\left(x\right)
\end{eqnarray*}
therefore, $A\xi_{n}=\frac{1}{\lambda_{n}}\xi_{n}$, i.e., 
\begin{equation}
A=T_{F}^{-1}.\label{eq:fr-12}
\end{equation}
Consequently, 
\[
\left\langle \xi_{n},\xi_{m}\right\rangle _{A}=\left\langle A\xi_{n},\xi_{m}\right\rangle _{2}=\frac{1}{\lambda_{n}}\left\langle \xi_{n},\xi_{m}\right\rangle _{2}=\frac{1}{\lambda_{n}}\delta_{n,m}.
\]
And we conclude that $\left\{ \sqrt{\lambda_{n}}\xi_{n}\right\} _{n\in\mathbb{N}}$
is an ONB in $\mathscr{H}_{A}=\mathscr{H}_{T_{F}^{-1}}$. 

See Section \ref{sec:F3-Mercer}, where $F=$ Pólya extension of $F_{3}$,
and a specific construction of $\mathscr{H}_{T_{F}^{-1}}$. \index{Pólya extensions}\index{extensions!Pólya-}\index{operator!Pólya's-}
\end{remark}

\subsection{Connection to the Energy Form Hilbert Spaces}

Now consider $A=1-\triangle$ defined on $C_{c}^{\infty}\left(\Omega\right)$.
There is a connection between the RKHS $\mathscr{H}_{A}$ and the
energy space as follows:\index{Energy Form}

For $f,g\in\mathscr{H}_{A}$, we have (restricting to real-valued
functions), 
\begin{eqnarray*}
\left\langle f,g\right\rangle _{A} & = & \left\langle \left(1-\triangle\right)f,g\right\rangle _{L^{2}}\\
 & = & \int_{\Omega}fg-\int_{\Omega}\left(\triangle f\right)g\\
 & = & \underset{\mbox{energy inner product}}{\underbrace{\int_{\Omega}fg+\int_{\Omega}Df\cdot Dg}}+\text{boundary corrections};
\end{eqnarray*}
So we define 
\begin{equation}
\left\langle f,g\right\rangle _{Energy}:=\int_{\Omega}fg+\int_{\Omega}Df\cdot Dg;\label{eq:gr-6}
\end{equation}
and then 
\begin{equation}
\left\langle f,g\right\rangle _{A}=\left\langle f,g\right\rangle _{Energy}+\mbox{boundary corrections.}\label{eq:gr-7}
\end{equation}

\begin{remark}
The $A$-inner product on the r.h.s. of (\ref{eq:gr-7}) incorporates
the boundary information.\end{remark}
\begin{example}
Consider $L^{2}\left(0,1\right)$, $F\left(x\right)=e^{-\left|x\right|}\big|_{\left(-1,1\right)}$,
and $A=\frac{1}{2}\bigl(1-\bigl(\frac{d}{dx}\bigr)^{2}\bigr)$. We
have 

\begin{eqnarray*}
\left\langle f,g\right\rangle _{A} & = & \frac{1}{2}\left\langle f-f'',g\right\rangle _{L^{2}}\\
 & = & \frac{1}{2}\int_{0}^{1}fg-\frac{1}{2}\int_{0}^{1}f''g\\
 & = & \frac{1}{2}\left(\int_{0}^{1}fg+\int_{0}^{1}f'g'\right)+\frac{\left(f'g\right)\left(0\right)-\left(f'g\right)\left(1\right)}{2}\\
 & = & \left\langle f,g\right\rangle _{Energy}+\frac{\left(f'g\right)\left(0\right)-\left(f'g\right)\left(1\right)}{2}.
\end{eqnarray*}
Here, the boundary term
\begin{equation}
\frac{\left(f'g\right)\left(0\right)-\left(f'g\right)\left(1\right)}{2}\label{eq:gr-8}
\end{equation}
contains the inward normal derivative of $f'$ at $x=0$ and $x=1$.\index{derivative!normal-}
\begin{enumerate}[leftmargin=*]
\item \begin{flushleft}
We proceed to check the reproducing property w.r.t. the $A$-inner
product:
\[
2\left\langle e^{-\left|x-\cdot\right|},g\right\rangle _{Energy}=\int_{0}^{1}e^{-\left|x-y\right|}g\left(y\right)dy+\int_{0}^{1}\left(\frac{d}{dy}e^{-\left|x-y\right|}\right)g'\left(y\right)dy
\]
where
\begin{eqnarray*}
 &  & \int_{0}^{1}\left(\frac{d}{dy}e^{-\left|x-y\right|}\right)g'\left(y\right)dy\\
 & = & \int_{0}^{x}e^{-\left(x-y\right)}g'\left(y\right)dy-\int_{x}^{1}e^{-\left(y-x\right)}g'\left(y\right)dy\\
 & = & 2g\left(x\right)-g\left(0\right)e^{-x}-g\left(1\right)e^{-\left(1-x\right)}-\int_{0}^{1}e^{-\left|x-y\right|}g\left(y\right)dy;
\end{eqnarray*}
it follows that 
\begin{equation}
\left\langle e^{-\left|x-\cdot\right|},g\right\rangle _{Energy}=g\left(x\right)-\frac{g\left(0\right)e^{-x}+g\left(1\right)e^{-\left(1-x\right)}}{2}\label{eq:f3-2-1}
\end{equation}

\par\end{flushleft}
\item \begin{flushleft}
It remains to check the boundary term in (\ref{eq:f3-2-1}) comes
from the inward normal derivative of $e^{-\left|x-\cdot\right|}$.
Indeed, set $f\left(\cdot\right)=e^{-\left|x-\cdot\right|}$ in (\ref{eq:gr-8}),
then 
\[
f'\left(0\right)=e^{-x},\qquad f'\left(1\right)=-e^{-\left(1-x\right)}
\]
therefore,
\par\end{flushleft}
\end{enumerate}
\[
\frac{\left(f'g\right)\left(0\right)-\left(f'g\right)\left(1\right)}{2}=\frac{e^{-x}g\left(0\right)+e^{-\left(1-x\right)}g\left(1\right)}{2}.
\]

\end{example}

\begin{example}
Consider $L^{2}\left(0,\frac{1}{2}\right)$, $F\left(x\right)=1-\left|x\right|$
with $\left|x\right|<\frac{1}{2}$, and let $A=-\frac{1}{2}\left(\frac{d}{dx}\right)^{2}$.
Then the $A$-inner product yields
\begin{eqnarray*}
\left\langle f,g\right\rangle _{A} & = & -\frac{1}{2}\left\langle f'',g\right\rangle _{L^{2}}\\
 & = & \frac{1}{2}\int_{0}^{\frac{1}{2}}f'g'-\frac{\left(f'g\right)\left(\frac{1}{2}\right)-\left(f'g\right)\left(0\right)}{2}\\
 & = & \left\langle f,g\right\rangle _{Energy}+\frac{\left(f'g\right)\left(0\right)-\left(f'g\right)\left(\frac{1}{2}\right)}{2}
\end{eqnarray*}
where we set 
\[
\left\langle f,g\right\rangle _{Energy}:=\frac{1}{2}\int_{0}^{\frac{1}{2}}f'g';
\]
and the corresponding boundary term is 
\begin{equation}
\frac{\left(f'g\right)\left(0\right)-\left(f'g\right)\left(\frac{1}{2}\right)}{2}\label{eq:gr-9}
\end{equation}

\begin{enumerate}[leftmargin=*]
\item \begin{flushleft}
To check the reproducing property w.r.t. the $A$-inner product:
Set $F_{x}\left(y\right):=1-\left|x-y\right|$, $x,y\in\left(0,\frac{1}{2}\right)$;
then 
\begin{eqnarray}
\left\langle F_{x},g\right\rangle _{Energy} & = & \frac{1}{2}\int_{0}^{\frac{1}{2}}F_{x}\left(y\right)'g'\left(y\right)dy\nonumber \\
 & = & =\frac{1}{2}\left(\int_{0}^{x}g'\left(y\right)dy-\int_{x}^{\frac{1}{2}}g'\left(y\right)dy\right)\nonumber \\
 & = & g\left(x\right)-\frac{g\left(0\right)+g\left(\frac{1}{2}\right)}{2}.\label{eq:F2-3-1}
\end{eqnarray}

\par\end{flushleft}
\item \begin{flushleft}
Now we check the second term on the r.h.s. of (\ref{eq:F2-3-1}) contains
the inward normal derivative of $F_{x}$. Note that 
\begin{eqnarray*}
F_{x}'\left(0\right) & = & \frac{d}{dy}\Big|_{y=0}\left(1-\left|x-y\right|\right)=1\\
F_{x}'\left(\frac{1}{2}\right) & = & \frac{d}{dy}\Big|_{y=\frac{1}{2}}\left(1-\left|x-y\right|\right)=-1
\end{eqnarray*}
Therefore, 
\[
\frac{\left(f'g\right)\left(0\right)-\left(f'g\right)\left(\frac{1}{2}\right)}{2}=\frac{g\left(0\right)+g\left(\frac{1}{2}\right)}{2};
\]
which verifies the boundary term in (\ref{eq:gr-9}).
\par\end{flushleft}
\end{enumerate}
\end{example}

\chapter{\label{chap:Greens}Green\textquoteright s Functions }

The focus of this chapter is a detailed analysis of two specific positive
definite functions, each one defined in a fixed finite interval, centered
at $x=0$. Rationale: The examples serve to make explicit some of
the many connections between our general theme (locally defined p.d.
functions and their extensions), on the one hand; and probability
theory and stochastic processes on the other.

\section{\label{sec:F2F3}The RKHSs for the Two Examples $F_{2}$ and $F_{3}$
in Table \ref{tab:F1-F6}}

In this section, we revisit cases $F_{2}$, and $F_{3}$ (from Table
\ref{tab:F1-F6}) and their associated RKHSs. The two examples are
\begin{eqnarray*}
F_{2}\left(x\right) & = & 1-\left|x\right|,\;\mbox{in \ensuremath{\left|x\right|<\frac{1}{2}}; and}\\
F_{3}\left(x\right) & = & e^{-\left|x\right|},\:\mbox{in \ensuremath{\left|x\right|<1}.}
\end{eqnarray*}

We show that they are (up to isomorphism) also the Hilbert spaces
used in stochastic integration for Brownian motion, and for the Ornstein-Uhlenbeck
process (see e.g., \cite{Hi80}), respectively. As reproducing kernel
Hilbert spaces, they have an equivalent and more geometric form, of
use in for example analysis of Gaussian processes. Analogous results
for the respective RKHSs also hold for other positive definite function
systems $\left(F,\Omega\right)$, but for the present two examples
$F_{2}$, and $F_{3}$, the correspondences involved are explicit.
As a bonus, we get an easy and transparent proof that the deficiency-indices
for the respective operators $D^{\left(F\right)}$ are $\left(1,1\right)$
in both these examples.\index{distribution!Gaussian-}

The purpose of the details below are two-fold. First we show that
the respective RKHSs corresponding to $F_{2}$ and $F_{3}$ in Table
\ref{tab:F1-F6} are naturally isomorphic to more familiar RKHSs which
are used in the study of Gaussian processes, see e.g., \cite{AJL11,AL10,AJ12};
and secondly, to give an easy (and intuitive) proof that the deficiency
indices in these two cases are $\left(1,1\right)$. Recall for each
p.d. function $F$ in an interval $\left(-a,a\right)$, we study 
\begin{equation}
D^{\left(F\right)}\left(F_{\varphi}\right):=F_{\varphi'},\;\varphi\in C_{c}^{\infty}\left(0,a\right)\label{eq:RKHS-eg-1}
\end{equation}
as a skew-Hermitian operator in $\mathscr{H}_{F}$; see Lemma \ref{lem:lcg-F_varphi}.

\index{Brownian motion}

\index{Ornstein-Uhlenbeck}

\index{RKHS}

\index{positive definite}

\index{Gaussian processes}

\index{deficiency indices}

\index{operator!skew-Hermitian}

\index{distribution!Gaussian-}\index{Hilbert space}

\subsection{\label{sub:green}Green's Functions}

\index{Green's function}

The term \textquotedblleft Green\textquoteright s function\textquotedblright{}
(also called \textquotedblleft fundamental solution,\textquotedblright{}
see e.g., \cite{Tre06}) is used generally in connection with inverses
of elliptic operators; but, more importantly, for the solution of
elliptic boundary value problems. We show that for the particular
case of one dimension, for the two elliptic operators (\ref{eq:F2-D})
and (\ref{eq:F3-D}), and for the corresponding boundary value problems
for finite intervals, the two above mentioned p.d. functions play
the role of Green\textquoteright s functions.
\begin{lemma}
~
\begin{enumerate}
\item \label{enu:F2-D}For $F_{2}\left(x\right)=1-\left|x\right|$, $\left|x\right|<\frac{1}{2}$,
let $\varphi\in C_{c}^{\infty}\left(0,\frac{1}{2}\right)$, then $u\left(x\right):=\left(T_{F_{2}}\varphi\right)\left(x\right)$
satisfies 
\begin{equation}
\varphi=-\frac{1}{2}\bigl(\frac{d}{dx}\bigr)^{2}u.\label{eq:F2-D}
\end{equation}
Hence, 
\begin{equation}
T_{F_{2}}^{-1}\supset-\frac{1}{2}\bigl(\frac{d}{dx}\bigr)^{2}\Big|_{C_{c}^{\infty}\left(0,\frac{1}{2}\right)}.\label{eq:F2-D-ext}
\end{equation}

\item \label{enu:F3-D}For $F_{3}\left(x\right)=e^{-\left|x\right|}$, $\left|x\right|<1$,
let $\varphi\in C_{c}^{\infty}\left(0,1\right)$, then 
\begin{equation}
\varphi=\frac{1}{2}\bigl(I-\bigl(\frac{d}{dx}\bigr)^{2}\bigr)u.\label{eq:F3-D}
\end{equation}
Hence,
\begin{equation}
T_{F_{3}}^{-1}\supset\frac{1}{2}\bigl(I-\bigl(\frac{d}{dx}\bigr)^{2}\bigr)\Big|_{C_{c}^{\infty}\left(0,1\right)}.\label{eq:F3-D-ext}
\end{equation}

\end{enumerate}
\end{lemma}
\begin{svmultproof2}
The computation for $F=F_{2}$ is as follows: Let $\varphi\in C_{c}^{\infty}\left(0,\frac{1}{2}\right)$,
then 
\begin{align*}
u\left(x\right)=\left(T_{F_{2}}\varphi\right)\left(x\right) & =\int_{0}^{\frac{1}{2}}\varphi\left(y\right)\left(1-\left|x-y\right|\right)dy\\
 & =\int_{0}^{x}\varphi\left(y\right)\left(1-\left(x-y\right)\right)dy+\int_{x}^{\frac{1}{2}}\varphi\left(y\right)\left(1-\left(y-x\right)\right)dy;
\end{align*}
and 
\[
\int_{0}^{x}\varphi\left(y\right)\left(1-\left(x-y\right)\right)dy=\int_{0}^{x}\varphi\left(y\right)dy-x\int_{0}^{x}\varphi\left(y\right)dy+\int_{0}^{x}y\varphi\left(y\right)dy
\]
\begin{eqnarray*}
u'\left(x\right) & = & -\int_{0}^{x}\varphi\left(y\right)+\varphi\left(x\right)+\int_{x}^{\frac{1}{2}}\varphi\left(y\right)dy-\varphi\left(x\right)\\
u''\left(x\right) & = & -2\varphi\left(x\right).
\end{eqnarray*}
Thus, $\varphi=-\frac{1}{2}u''$, and the desired result follows.

For $F_{3}$, let $\varphi\in C_{c}^{\infty}\left(0,1\right)$, then
\begin{eqnarray*}
u\left(x\right)=\left(T_{F_{3}}\varphi\right)\left(x\right) & = & \int_{0}^{1}e^{-\left|x-y\right|}\varphi\left(y\right)dy\\
 & = & \int_{0}^{x}e^{-\left(x-y\right)}\varphi\left(y\right)dy+\int_{x}^{1}e^{-\left(y-x\right)}\varphi\left(y\right)dy.
\end{eqnarray*}
Now, 
\begin{eqnarray*}
u'\left(x\right)=\left(T_{F_{3}}\varphi\right)'\left(x\right) & = & -e^{-x}\int_{0}^{x}e^{-y}\varphi\left(y\right)dy+\varphi\left(x\right)\\
 &  & +e^{x}\int_{x}^{1}e^{-y}\varphi\left(y\right)dy-\varphi\left(x\right)\\
 & = & -e^{-x}\int_{0}^{x}e^{y}\varphi\left(y\right)dy+e^{x}\int_{x}^{1}e^{-y}\varphi\left(y\right)dy\\
u'' & = & e^{-x}\int_{0}^{x}e^{y}\varphi\left(y\right)dy-\varphi\left(x\right)\\
 &  & +e^{x}\int_{x}^{1}e^{-y}\varphi\left(y\right)dy-\varphi\left(x\right)\\
 & = & -2\varphi+\int_{0}^{1}e^{-\left|x-y\right|}\varphi\left(y\right)dy\\
 & = & -2\varphi+T_{F_{3}}\left(\varphi\right);
\end{eqnarray*}
and then 
\begin{eqnarray*}
u''\left(x\right) & = & e^{-x}\int_{0}^{x}e^{y}\varphi\left(y\right)dy-\varphi\left(x\right)\\
 &  & +e^{x}\int_{x}^{1}e^{-y}\varphi\left(y\right)dy-\varphi\left(x\right)\\
 & = & -2\varphi\left(x\right)+\int_{0}^{1}e^{-\left|x-y\right|}\varphi\left(y\right)dy\\
 & = & -2\varphi+u\left(x\right).
\end{eqnarray*}
Thus, $\varphi=\frac{1}{2}$$\left(u-u''\right)=\frac{1}{2}\left(I-\frac{1}{2}\left(\frac{d}{dx}\right)^{2}\right)u$.
This proves (\ref{eq:F3-D}).
\end{svmultproof2}

\textbf{Summary: Conclusions for the two examples.} 

The computation for $F=F_{2}$ is as follows: If $\varphi\in L^{2}\left(0,\frac{1}{2}\right)$,
then $u\left(x\right):=\left(T_{F}\varphi\right)\left(x\right)$ satisfies
\[
\begin{array}[t]{cccc}
(F_{2}) &  &  & \varphi=\frac{1}{2}\bigl(-\bigl(\frac{d}{dx}\bigr)^{2}\bigr)u;\end{array}
\]
while, for $F=F_{3}$, the corresponding computation is as follows:
If $\varphi\in L^{2}\left(0,1\right)$, then $u\left(x\right)=\left(T_{F}\varphi\right)\left(x\right)$
satisfies
\[
\begin{array}[t]{cccc}
(F_{3}) &  &  & \varphi=\frac{1}{2}\bigl(I-\bigl(\frac{d}{dx}\bigr)^{2}\bigr)u;\end{array}
\]
For the operator $D^{\left(F\right)}$, in the case of $F=F_{2}$,
it follows that the Mercer operator $T_{F}$ plays the following role:
$T_{F}^{-1}$ is a selfadjoint extension of $-\frac{1}{2}\bigl(D^{\left(F\right)}\bigr)^{2}$.
In the case of $F=F_{3}$ the corresponding operator $T_{F}^{-1}$
(in the RKHS $\mathscr{H}_{F_{3}}$) is a selfadjoint extension of
$\frac{1}{2}\bigl(I-\bigl(D^{\left(F\right)}\bigr)^{2}\bigr)$; in
both cases, they are the Friedrichs extensions.
\begin{remark}
When solving boundary values for elliptic operators in a bounded domain,
say $\Omega\subset\mathbb{R}^{n}$, one often ends up with Green's
functions (= integral kernels) which are positive definite kernels,
so $K\left(x,y\right)$, defined on $\Omega\times\Omega$, not necessarily
of the form $K\left(x,y\right)=F\left(x\lyxmathsym{\textendash}y\right)$. 

But many of the questions we ask in the special case of p.d. functions,
so when the kernel is $K\left(x,y\right)=F\left(x-y\right)$ will
also make sense for p.d. kernels.
\end{remark}
\index{Friedrichs extension}

\index{operator!Mercer}

\index{Green's function}

\subsection{\label{sub:F2}The Case of $F_{2}\left(x\right)=1-\left|x\right|$,
$x\in\left(-\frac{1}{2},\frac{1}{2}\right)$}

Let $F=F_{2}$. Fix $x\in\left(0,\frac{1}{2}\right)$, and set 
\begin{equation}
F_{x}\left(y\right)=F\left(x-y\right),\;\mbox{for }x,y\in{\textstyle \left(0,\frac{1}{2}\right)};\label{eq:RKHS-eg-2}
\end{equation}
where $F_{x}\left(\cdot\right)$ and its derivative (in the sense
of distributions) are as in Figure \ref{fig:Fx} (sketched for two
values of $x$).

Consider the Hilbert space
\begin{align}
\mathscr{H}_{F} & :=\begin{Bmatrix}h; & \text{continuous on \ensuremath{\left(0,\tfrac{1}{2}\right)}, and }h'=\tfrac{dh}{dx}\in L^{2}\left(0,\tfrac{1}{2}\right)\\
 & \text{where the derivative is in the weak sense}\qquad
\end{Bmatrix}\label{eq:RKHS-eg-3}
\end{align}
modulo constants; and let the norm, and inner-product, in $\mathscr{H}_{F}$
be given by 
\begin{equation}
\left\Vert h\right\Vert _{\mathscr{H}_{F}}^{2}=\frac{1}{2}\int_{0}^{\frac{1}{2}}\left|h'\left(x\right)\right|^{2}dx+\int_{\partial\Omega}\overline{h_{n}}h\,d\beta.\label{eq:RKHS-eg-4}
\end{equation}
On the r.h.s. of (\ref{eq:RKHS-eg-4}), $d\beta$ denotes the corresponding
boundary measure, and $h_{n}$ is the inward normal derivative of
$h$. See Theorem \ref{thm:F2-bd} below. \index{derivative!normal-}

Then the reproducing kernel property is as follows:
\begin{equation}
\left\langle F_{x},h\right\rangle _{\mathscr{H}_{F}}=h\left(x\right),\;\forall h\in\mathscr{H}_{F},\forall x\in\left(0,\tfrac{1}{2}\right);\label{eq:RKHS-eg-5}
\end{equation}
and it follows that $\mathscr{H}_{F}$ is naturally isomorphic to
the RKHS for $F_{2}$ from  Section \ref{sub:lcg}.

\begin{figure}
\begin{tabular}{cc}
\includegraphics[scale=0.6]{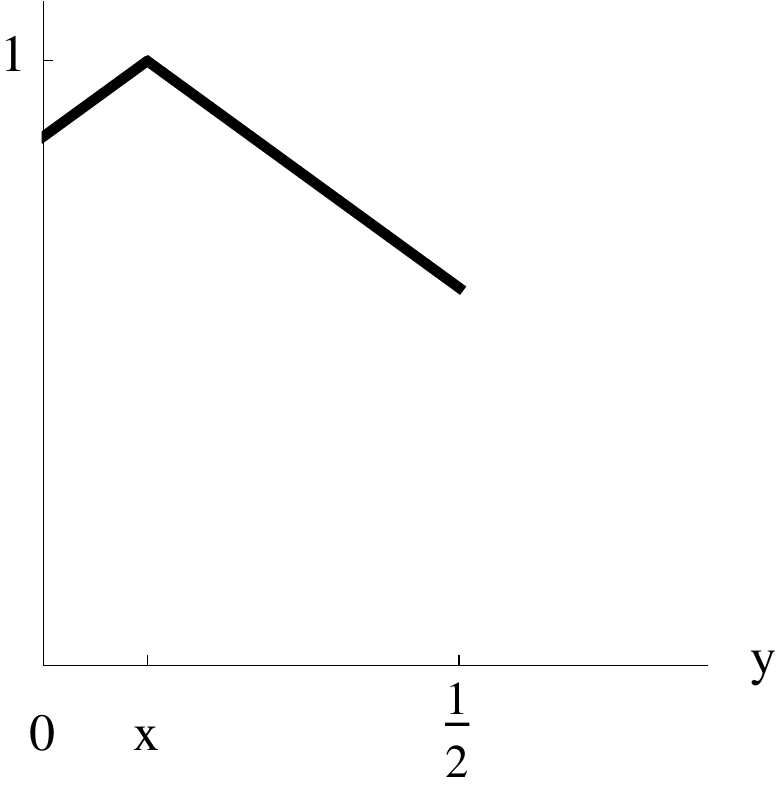} & \includegraphics[scale=0.6]{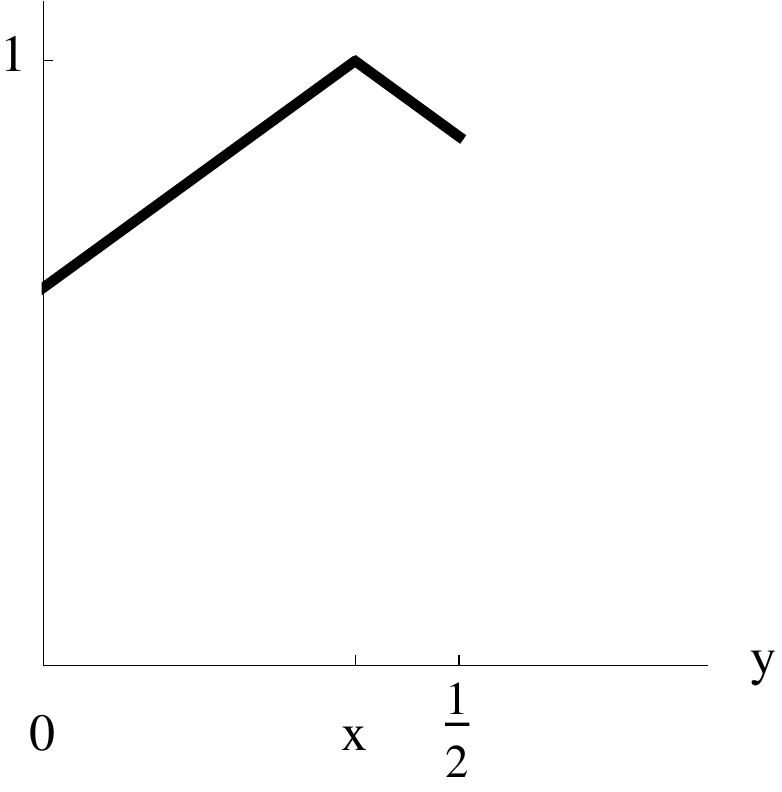}\tabularnewline
\includegraphics[scale=0.6]{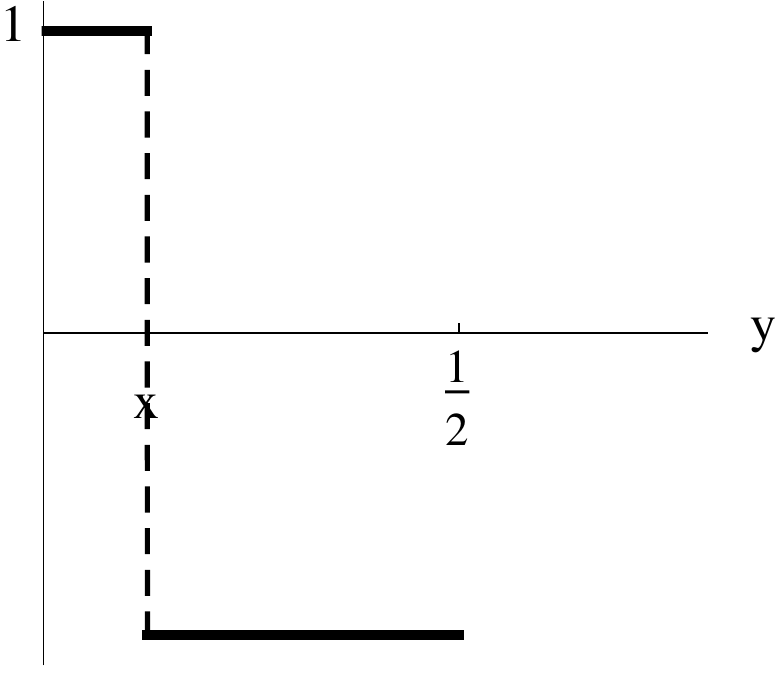} & \includegraphics[scale=0.6]{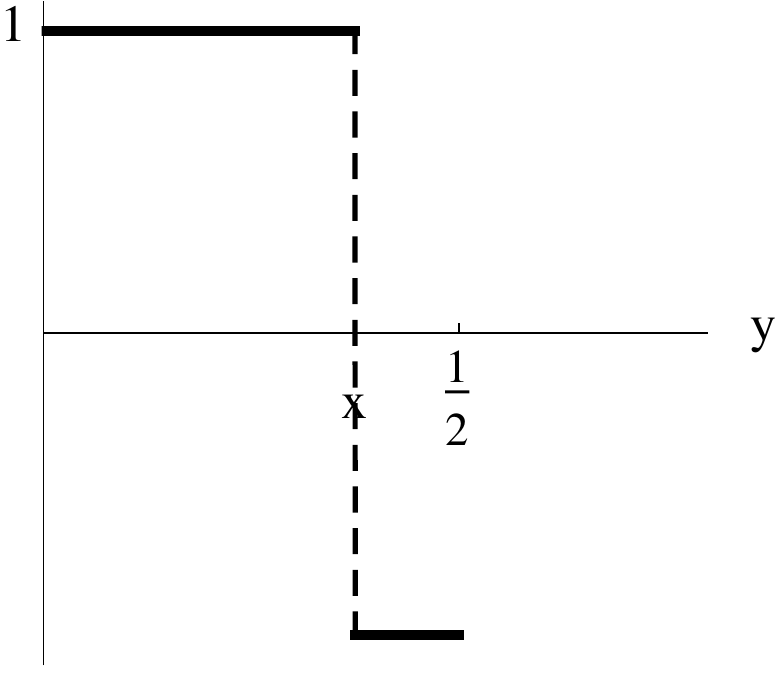}\tabularnewline
\end{tabular}

\protect\caption{\label{fig:Fx}The kernel $F_{x}$ and its derivative (the case of
$F_{2}$)}
\end{figure}

\begin{theorem}
\label{thm:F2-bd}The boundary measure for $F=F_{2}$ (see (\ref{eq:RKHS-eg-4}))
is 
\[
\beta=\frac{1}{2}\left(\delta_{0}+\delta_{1/2}\right).
\]
\end{theorem}
\begin{svmultproof2}
Set 
\begin{equation}
\mathscr{E}\left(\xi\right):=\frac{1}{2}\int_{0}^{\frac{1}{2}}\left|\xi'\left(x\right)\right|^{2}dx,\;\forall\xi\in\mathscr{H}_{F}.\label{eq:en-2-1}
\end{equation}
And let $F_{x}\left(\cdot\right):=1-\left|x-\cdot\right|$ defined
on $\left[0,\frac{1}{2}\right]$, for all $x\in\left(0,\frac{1}{2}\right)$.
Then 
\begin{eqnarray*}
E\left(F_{x},\xi\right) & = & \frac{1}{2}\int_{0}^{\frac{1}{2}}F_{x}'\left(y\right)\xi'\left(y\right)dy\\
 & = & \frac{1}{2}\left(\int_{0}^{x}\xi'\left(y\right)dy-\int_{x}^{\frac{1}{2}}\xi'\left(y\right)dy\right)\\
 & = & \xi\left(x\right)-\frac{\xi\left(0\right)+\xi\left(\frac{1}{2}\right)}{2}.\;(\mbox{see Fig. \ref{fig:Fx}})
\end{eqnarray*}
Since
\[
\left\Vert \xi\right\Vert _{\mathscr{H}_{F}}^{2}=E\left(\xi\right)+\int\left|\xi\right|^{2}d\beta
\]
we get 
\[
\left\langle F_{x},\xi\right\rangle _{\mathscr{H}_{F}}=\xi\left(x\right),\;\forall\xi\in\mathscr{H}_{F}.
\]
We conclude that 
\begin{equation}
\left\langle \xi,\eta\right\rangle _{\mathscr{H}_{F}}=E\left(\xi,\eta\right)+\int_{\partial\Omega}\overline{\xi_{n}}\eta d\beta;\label{eq:en-3-1}
\end{equation}
note the boundary in this case consists two points, $x=0$ and $x=\frac{1}{2}$.\end{svmultproof2}

\begin{remark}
\label{rem:bm}The energy form in (\ref{eq:en-2-1}) also defines
a RKHS (Wiener's energy form for Brownian motion \cite{Hi80}, see
Figure \ref{fig:bm2}) as follows: \index{Energy Form}

On the space of all continuous functions, $\mathscr{C}\left([0,\frac{1}{2}]\right)$,
set 
\begin{equation}
\mathscr{H}_{\mathscr{E}}:=\left\{ f\in\mathscr{C}\left([0,\tfrac{1}{2}]\right)\:\big|\:\mathscr{E}\left(f\right)<\infty\right\} \label{eq:en-2-2}
\end{equation}
modulo constants, where
\begin{equation}
\mathscr{E}\left(f\right)=\int_{0}^{\frac{1}{2}}\left|f'\left(x\right)\right|^{2}dx.\label{eq:en-2-3}
\end{equation}
For $x\in\left[0,\frac{1}{2}\right]$, set 
\[
E_{x}\left(y\right)=x\wedge y=\min\left(x,y\right),\;y\in\left(0,\tfrac{1}{2}\right);
\]
see \figref{bm1}; then $E_{x}\in\mathscr{H}_{\mathscr{E}}$, and
\begin{equation}
\left\langle E_{x},f\right\rangle _{\mathscr{H}_{\mathscr{E}}}=f\left(x\right),\;\forall f\in\mathscr{H}_{\mathscr{E}},\forall x\in\left[0,\tfrac{1}{2}\right].\label{eq:en-2-4}
\end{equation}

\end{remark}
\begin{figure}[H]
\includegraphics[width=0.5\textwidth]{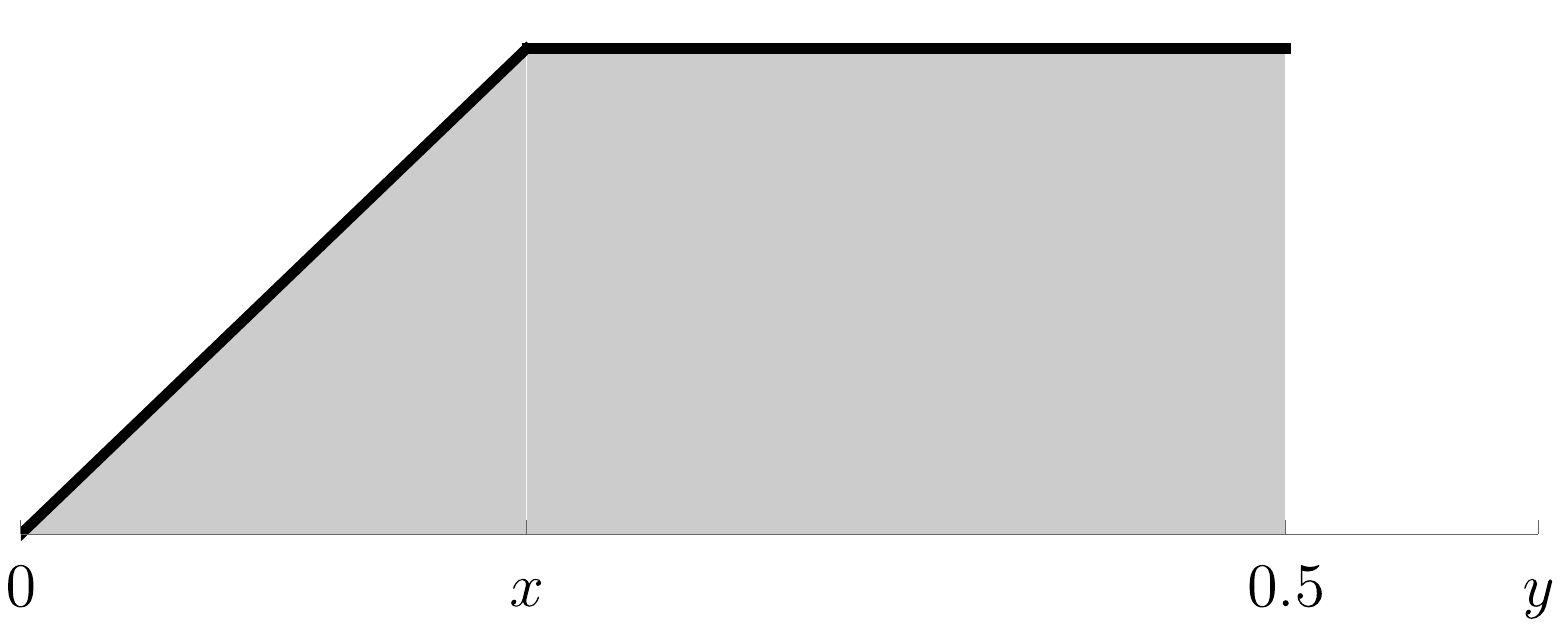}

\protect\caption{\label{fig:bm1}The covariance function $E_{x}\left(\cdot\right)=\min\left(x,\cdot\right)$
of Brownian motion.}
\end{figure}

\begin{svmultproof2}
For the reproducing property (\ref{eq:en-2-4}): Let $f$ and $x$
be as stated; then 
\begin{eqnarray*}
\left\langle E_{x},f\right\rangle _{\mathscr{H}_{\mathscr{E}}} & = & \int_{0}^{\frac{1}{2}}E_{x}'\left(y\right)f'\left(y\right)dy\\
 & \underset{\text{Fig. \ensuremath{\left(\ref{fig:bm1}\right)}}}{=} & \int_{0}^{\frac{1}{2}}\chi_{\left[0,x\right]}\left(y\right)f'\left(y\right)dy\\
 & = & \int_{0}^{x}f'\left(y\right)dy=f\left(x\right)-f\left(0\right).
\end{eqnarray*}
Note in (\ref{eq:en-2-2}) we define $\mathscr{H}_{\mathscr{E}}$
modulo constants; alternatively, we may stipulate $f\left(0\right)=0$. 
\end{svmultproof2}

Note that the Brownian motion\index{Brownian motion} RKHS is not
defined by a p.d. function, but rather by a p.d. kernel. Nonetheless
the remark explains its connection to our present RKHS $\mathscr{H}_{F}$
which is defined by the p.d. function, namely the p.d. function $F_{2}$.

\textbf{Pinned Brownian Motion. }\index{Brownian bridge}

We illustrate the boundary term in Theorem \ref{thm:F2-bd}, eq. (\ref{eq:en-3-1})
with pinned Brownian motion (also called ``Brownian bridge.'') In
order to simplify constructions, we pin the Brownian motion at the
following two points in $\left(t,x\right)$ space, $\left(t,x\right)=\left(0,0\right)$,
and $\left(t,x\right)=\left(1,1\right)$; see \figref{bbridge}. To
simplify computations further, we restrict attention to real-valued
functions only. 
\begin{remark}
The literature on Gaussian processes, Ito-calculus, and Brownian motion
is vast, but for our purposes, the reference \cite{Hi80} will do.
For background on stochastic processes, see especially \cite[p.243-244]{IkWa89}.
\index{stochastic processes}
\end{remark}
For the pinning down the process $X_{t}$ at the two points $\left(0,0\right)$
and $\left(1,1\right)$ as in \figref{bbridge}, we have
\begin{proposition}
The Brownian bridge $X_{t}$, connecting (or pinning down) the two
points $(0,0)$ and $(1,1)$, satisfies $X_{0}=0$, $X_{1}=1$, and
it has the following covariance function $c$,
\begin{equation}
c\left(s,t\right)=s\wedge t-st.\label{eq:bb-1-a}
\end{equation}
Moreover, 
\begin{equation}
X_{t}=t+\left(1-t\right)\int_{0}^{t}\frac{dB_{s}}{1-s},\;0<t<1,\label{eq:bb-1}
\end{equation}
where $dB_{s}$ on the r.h.s. of (\ref{eq:bb-1}) refers to the standard
Brownian motion $dB_{s}$, and the second term in (\ref{eq:bb-1})
is the corresponding Ito-integral. \end{proposition}
\begin{svmultproof2}
We have\index{Ito-integral}
\begin{eqnarray}
\mathbb{E}\left(\int_{0}^{t}\frac{dB_{s}}{1-s}\right) & = & 0,\mbox{ and}\label{eq:bb-2}\\
\mathbb{E}\left(\left|\int_{0}^{t}\frac{dB_{s}}{1-s}\right|^{2}\right) & = & \int_{0}^{t}\frac{ds}{\left(1-s\right)^{2}}=\frac{t}{1-t}\label{eq:bb-3}
\end{eqnarray}
where $\mathbb{E}\left(\cdots\right)=\int_{\Omega}\cdots d\mathbb{P}=$
expectation with respect to the underlying probability space $\left(\Omega,\mathscr{F},\mathbb{P}\right)$.
\index{probability space} Hence, for the mean and covariance function\index{covariance}
of the process $X_{t}$ in (\ref{eq:bb-1}), we get 
\begin{equation}
\mathbb{E}\left(X_{t}\right)=t,\;\mbox{and}\label{eq:bb-4}
\end{equation}
\begin{equation}
Cov\left(X_{t}X_{s}\right)=\mathbb{E}\left(\left(X_{t}-t\right)\left(X_{s}-s\right)\right)=s\wedge t-st;\label{eq:bb-5}
\end{equation}
where $s\wedge t=\min\left(s,t\right)$, and $s,t\in\left(0,1\right)$. 
\end{svmultproof2}

\begin{figure}
\includegraphics[width=0.7\textwidth]{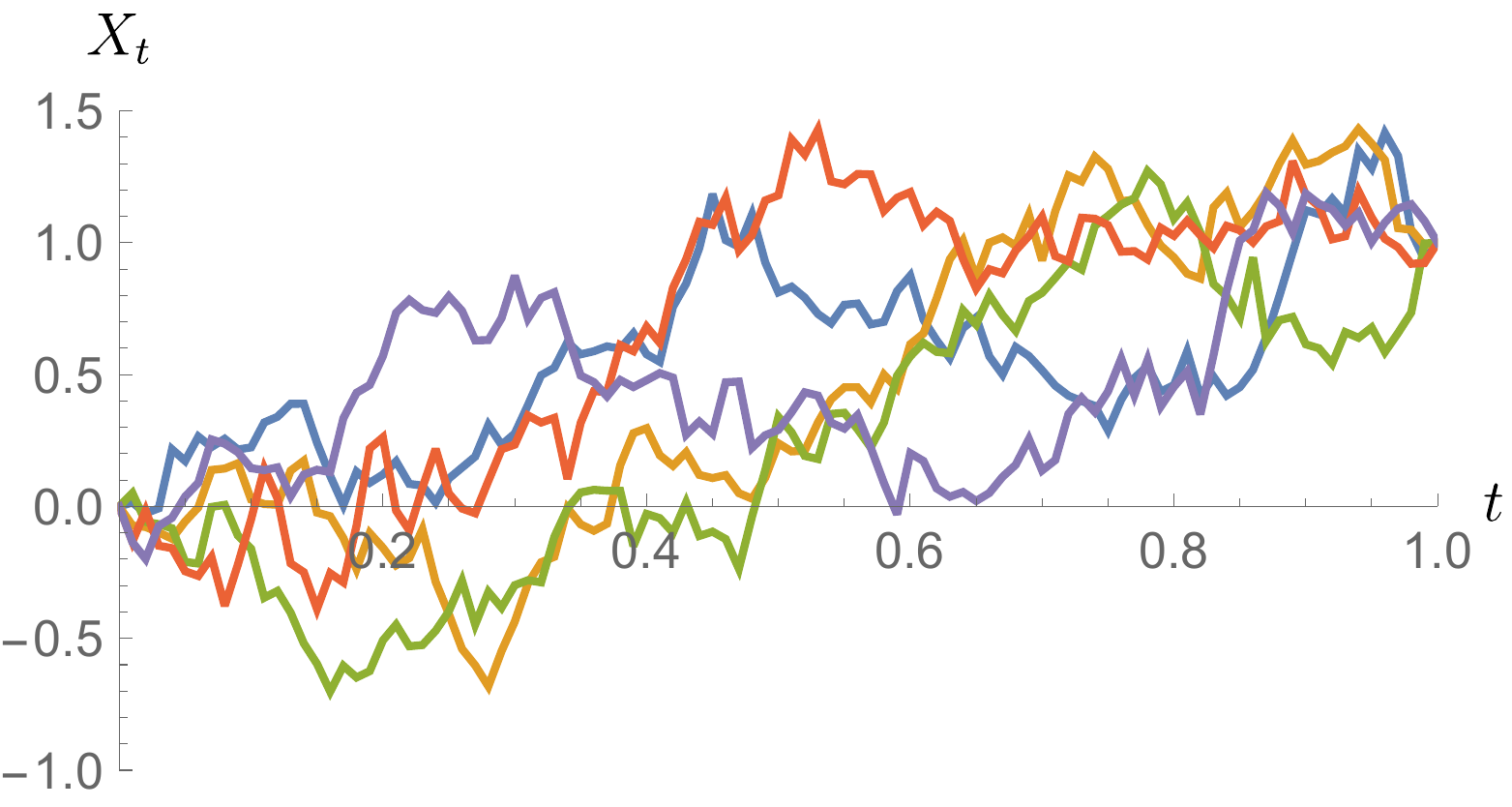}

\protect\caption{\label{fig:bbridge}Monte-Carlo simulation of Brownian bridge pinned
at $\left(0,0\right)$ and $\left(1,1\right)$, with 5 sample paths.
\index{Monte-Carlo simulation}}
\end{figure}

And it follows in particular that the function on the r.h.s. in eq
(\ref{eq:bb-5}) is a positive definite kernel. Its connection to
$F_{2}$ is given in the next lemma.

Now return to the p.d. function $F=F_{2}$; i.e., $F\left(t\right)=1-\left|t\right|$,
and therefore, $F\left(s-t\right)=1-\left|s-t\right|$, $s,t\in\left(0,1\right)$,
and let $\mathscr{H}_{F}$ be the corresponding RKHS. 

Let $\left\{ F_{t}\right\} _{t\in\left(0,1\right)}$ denote the kernels
in $\mathscr{H}_{F}$, i.e., $\left\langle F_{t},g\right\rangle _{\mathscr{H}_{F}}=g\left(t\right)$,
$t\in\left(0,1\right)$, $g\in\mathscr{H}_{F}$. We then have the
following:
\begin{lemma}
\label{lem:bbridge}Let $\left(X_{t}\right)_{t\in\left(0,1\right)}$
denote the pinned Brownian motion (\ref{eq:bb-1}); then 
\begin{equation}
\left\langle F_{s},F_{t}\right\rangle _{\mathscr{H}_{F}}=Cov\left(X_{s}X_{t}\right);\label{eq:bb-6}
\end{equation}
see (\ref{eq:bb-5}); and 
\begin{equation}
\left\langle F_{s},F_{t}\right\rangle _{energy}=s\wedge t;\label{eq:bb-7}
\end{equation}
while the boundary term 
\begin{equation}
bdr\left(s,t\right)=-st.\label{eq:bb-8}
\end{equation}
 \end{lemma}
\begin{svmultproof2}
With our normalization from Figure \ref{fig:bbridge}, we must take
the energy form as follows:\index{Energy Form}
\begin{equation}
\left\langle f,g\right\rangle _{energy}=\int_{0}^{1}f'\left(x\right)g'\left(x\right)dx.\label{eq:bb-9}
\end{equation}

Set $F_{s}\left(y\right)=s\wedge y$, see Figure \ref{fig:bm1}. For
the distributional derivative we have\index{derivative!distributional-}
\[
F_{s}'\left(y\right)=\chi_{\left[0,s\right]}\left(y\right)=\begin{cases}
1 & \mbox{if }0\leq y\leq s\\
0 & \mbox{else},
\end{cases}
\]
then 
\begin{eqnarray*}
\left\langle F_{s},F_{t}\right\rangle _{energy} & = & \int_{0}^{1}F_{s}'\left(y\right)F_{t}'\left(y\right)dy\\
 & = & \int_{0}^{1}\chi_{\left[0,s\right]}\left(y\right)\chi_{\left[0,t\right]}\left(y\right)dy\\
 & = & \left|\left[0,s\right]\cap\left[0,t\right]\right|_{\text{Lebesgue measure}}\\
 & = & s\wedge t.
\end{eqnarray*}
The desired conclusions (\ref{eq:bb-7})-(\ref{eq:bb-8}) in the lemma
now follow. See Remark \ref{rem:bm}.

The verification of (\ref{eq:bb-6}) uses Ito-calculus \cite{Hi80}
as follows: Note that (\ref{eq:bb-1}) for $X_{t}$ is the solution
to the following Ito-equation:\index{Ito-integral}
\begin{equation}
dX_{t}=\left(\frac{X_{t}-1}{t-1}\right)dt+dB_{t};\label{eq:bb-11}
\end{equation}
and by Ito's lemma \cite{Hi80}, therefore, 
\begin{equation}
\left(dX_{t}\right)^{2}=\left(dB_{t}\right)^{2}.\label{eq:bb-12}
\end{equation}

As a result, if $f:\mathbb{R}\rightarrow\mathbb{R}$ is a function
in the energy-Hilbert space defined from (\ref{eq:bb-9}), then for
$0<T<1$, we have, \index{energy-Hilbert space} 
\[
\mathbb{E}\left(\left|f\left(X_{T}\right)\right|^{2}\right)=\mathbb{E}\left(\left|f\left(B_{T}\right)\right|^{2}\right)=\int_{0}^{T}\left|f'\left(t\right)\right|^{2}dt=\mathscr{E}_{2}\left(f\right),
\]
the energy form of the RKHS (\ref{eq:bb-9}).\end{svmultproof2}

\begin{remark}
In (\ref{eq:bb-11})-(\ref{eq:bb-12}), we use standard conventions
for Brownian motions $B_{t}$ (see, e.g., \cite{Hi80}): Let $\left(\Omega,\mathscr{F},\mathbb{P}\right)$
be a choice of probability space for $\left\{ B_{t}\right\} _{t\in\mathbb{R}}$,
(or $t\in\left[0,1\right]$.) With $E\left(\cdots\right)=\int_{\Omega}\cdots d\mathbb{P}$,
we have \index{probability space} 
\begin{equation}
E\left(B_{s}B_{t}\right)=s\wedge t=\min\left(s,t\right),\;t\in\left[0,1\right].\label{eq:bm-1-1}
\end{equation}
If $f:\mathbb{R}\rightarrow\mathbb{R}$ is a $C^{1}$-function, we
write $f\left(B_{t}\right)$ for $f\circ B_{t}$; and $df\left(B_{t}\right)$
refers to Ito-calculus.
\end{remark}

\subsection{\label{sub:F3}The Case of $F_{3}\left(x\right)=e^{-\left|x\right|}$,
$x\in\left(-1,1\right)$}

Recall, if $\{B(x)\}_{x\in\mathbb{R}}$ denotes the standard Brownian
motion (see, e.g., Sections \ref{sec:stach}, \ref{sec:Polya}, and
\ref{sec:mercer}): then if we set 
\begin{equation}
X(x):=e^{-x}B\left(e^{2x}\right),\;x\in\mathbb{R};\label{eq:un2-1}
\end{equation}
then $X\left(x\right)$ is the Ornstein-Uhlenbeck process with 
\begin{equation}
\mathbb{E}\left(X\left(x\right)X\left(y\right)\right)=e^{-\left|x-y\right|},\;\forall x,y\in\mathbb{R}.\label{eq:un2-2}
\end{equation}

Let $F_{x}\left(y\right)=F\left(x-y\right)$, for all $x,y\in\left(0,1\right)$;
and consider the Hilbert space 
\begin{align}
\mathscr{H}_{F} & :=\begin{Bmatrix}h; & \text{continuous on \ensuremath{\left(0,1\right)}, }h\in L^{2}\left(0,1\right)\text{, and}\\
 & \text{the weak derivative }h'\in L^{2}\left(0,1\right)\,\,\,\,\,\,\,\,\,\,\,\,\,\,\,
\end{Bmatrix};\label{eq:RKHS-eg-6}
\end{align}
and let the $\mathscr{H}_{F}$-norm, and inner product, be given by
\begin{equation}
\left\Vert h\right\Vert _{\mathscr{H}_{F}}^{2}=\frac{1}{2}\left(\int_{0}^{1}\left|h'\left(x\right)\right|^{2}dx+\int_{0}^{1}\left|h\left(x\right)\right|^{2}dx\right)+\int_{\partial\Omega}\overline{h_{n}}h\,d\beta.\label{eq:RKHS-eg-7}
\end{equation}
Here, $d\beta$ on the r.h.s. of (\ref{eq:RKHS-eg-7}) denotes the
corresponding boundary measure, and $h_{n}$ is the inward normal
derivative of $h$. See Theorem \ref{thm:F3bd} below.

Then a direct verification yields:
\begin{equation}
\left\langle F_{x},h\right\rangle _{\mathscr{H}_{F}}=h\left(x\right),\;\forall h\in\mathscr{H}_{F},\forall x\in\left(0,1\right);\label{eq:RKHS-eg-8}
\end{equation}
and it follows that $\mathscr{H}_{F}$ is naturally isomorphic to
RKHS for $F_{3}$ (see Section \ref{sub:lcg}).

For details of (\ref{eq:RKHS-eg-8}), see also \cite{Jor81}.
\begin{corollary}
\label{cor:RKHS-eg-1}In both $\mathscr{H}_{F_{i}}$, $i=2,3$, the
deficiency indices are $\left(1,1\right)$.\end{corollary}
\begin{svmultproof2}
In both cases, we are referring to the skew-Hermitian operator $D^{\left(F_{i}\right)}$
in $\mathscr{H}_{Fi}$, $i=2,3$; see (\ref{eq:RKHS-eg-1}) above.
But it follows from (\ref{eq:RKHS-eg-4}) and (\ref{eq:RKHS-eg-7})
for the respective inner products, that the functions $e^{\pm x}$
have finite positive norms in the respective RKHSs.\end{svmultproof2}

\begin{theorem}
\label{thm:F3bd}Let $F=F_{3}$ as before. Consider the energy-Hilbert
space
\begin{equation}
\mathscr{H}_{E}:=\begin{Bmatrix}\ensuremath{f\in C\left[0,1\right]}\,\big|\,f,f'\in L^{2}\left(0,1\right)\text{ where }\\
f'\text{ is the weak derivative of }f
\end{Bmatrix}\label{eq:en-1}
\end{equation}
with\index{energy-Hilbert space} 
\begin{align}
\left\langle f,g\right\rangle _{E} & =\frac{1}{2}\left(\int_{0}^{1}\overline{f'\left(x\right)}g'\left(x\right)dx+\int_{0}^{1}\overline{f\left(x\right)}g\left(x\right)dx\right);\;\text{and so}\label{eq:en-2}\\
\left\Vert f\right\Vert _{E}^{2} & =\frac{1}{2}\left(\int_{0}^{1}\left|f'\left(x\right)\right|^{2}dx+\int_{0}^{1}\left|f\left(x\right)\right|^{2}dx\right),\;\forall f,g\in\mathscr{H}_{E}.\label{eq:en-3}
\end{align}
Set 
\begin{align}
P\left(f,g\right) & =\int_{0}^{1}\overline{f_{n}\left(x\right)}g\left(x\right)d\beta\left(x\right),\;\text{where }\label{eq:en-4}\\
\beta & :=\frac{\delta_{0}+\delta_{1}}{2},\text{ i.e., Dirac masses at endpoints}\label{eq:en-5}
\end{align}
where $g_{n}$ denotes the inward normal derivative at endpoints. 

Let 
\begin{equation}
\mathscr{H}_{F}:=\begin{Bmatrix}f\in C\left[0,1\right]\,\big|\,\left\Vert f\right\Vert _{E}^{2}+P_{2}\left(f\right)<\infty\text{ where}\\
P_{2}\left(f\right):=P\left(f,f\right)=\int_{0}^{1}\left|f\left(x\right)\right|^{2}d\beta\left(x\right)
\end{Bmatrix}\label{eq:en-6}
\end{equation}
Then we have following:\index{derivative!normal-}
\begin{enumerate}
\item As a Green-Gauss-Stoke principle, we have \index{Theorem!Green-Gauss-Stoke-}
\begin{align}
\left\Vert \cdot\right\Vert _{F}^{2} & =\left\Vert \cdot\right\Vert _{E}^{2}+P_{2},\;\text{i.e.,}\label{eq:en-7}\\
\left\Vert f\right\Vert _{F}^{2} & =\left\Vert f\right\Vert _{E}^{2}+\int_{0}^{1}\left|f\left(x\right)\right|^{2}d\beta\left(x\right).\label{eq:en-8}
\end{align}

\item Moreover, 
\begin{equation}
\left\langle F_{x},g\right\rangle _{E}=g\left(x\right)-\frac{e^{-x}g\left(0\right)+e^{-\left(1-x\right)}g\left(1\right)}{2}.\label{eq:en-9}
\end{equation}

\item As a result of (\ref{eq:en-9}), eq. (\ref{eq:en-7})-(\ref{eq:en-8})
is the reproducing property in $\mathscr{H}_{F}$. Specifically, we
have
\begin{equation}
\left\langle F_{x},g\right\rangle _{\mathscr{H}_{F}}=g\left(x\right),\;\forall g\in\mathscr{H}_{F}.\:\text{(see (\ref{eq:en-6}))}\label{eq:en-10}
\end{equation}

\end{enumerate}
\end{theorem}
\begin{svmultproof2}
Given $g\in\mathscr{H}_{F}$, one checks directly that 
\begin{eqnarray*}
P\left(F_{x},g\right) & = & \int_{0}^{1}F_{x}\left(y\right)g\left(y\right)d\beta\left(y\right)\\
 & = & \frac{F_{x}\left(0\right)g\left(0\right)+F_{x}\left(1\right)g\left(1\right)}{2}\\
 & = & \frac{e^{-x}g\left(0\right)+e^{-\left(1-x\right)}g\left(1\right)}{2};
\end{eqnarray*}
and, using integration by parts, we have 
\begin{eqnarray*}
\left\langle F_{x},g\right\rangle _{E} & = & g\left(x\right)-\frac{e^{-x}g\left(0\right)+e^{-\left(1-x\right)}g\left(1\right)}{2}\\
 & = & \left\langle F_{x},g\right\rangle _{F}-P\left(F_{x},g\right).
\end{eqnarray*}

Now, using 
\begin{equation}
\int_{\partial\Omega}\left(F_{x}\right)_{n}gd\beta:=\frac{e^{-x}g\left(0\right)+e^{-\left(1-x\right)}g\left(1\right)}{2},\label{eq:en-13}
\end{equation}
and $\left\langle F_{x},g\right\rangle _{\mathscr{H}_{F}}=g\left(x\right)$,
$\forall x\in\left(0,1\right)$, and using (\ref{eq:en-9}), the desired
conclusion
\begin{equation}
\left\langle F_{x},g\right\rangle _{\mathscr{H}_{F}}=\underset{\text{energy term}}{\underbrace{\left\langle F_{x},g\right\rangle _{E}}}+\int_{\partial\Omega}\left(F_{x}\right)_{n}gd\beta\label{eq:en-1-1}
\end{equation}
follows. Since the $\mathscr{H}_{F}$-norm closure of the span of
$\left\{ F_{x}\:\big|\:x\in\left(0,1\right)\right\} $ is all of $\mathscr{H}_{F}$,
from (\ref{eq:en-1-1}), we further conclude that
\begin{equation}
\left\langle f,g\right\rangle _{\mathscr{H}_{F}}=\left\langle f,g\right\rangle _{E}+\int_{\partial\Omega}\overline{f_{n}}gd\beta\label{eq:en-1-2}
\end{equation}
holds for all $f,g\in\mathscr{H}_{F}$.

In (\ref{eq:en-1-1}) and (\ref{eq:en-1-2}), we used $f_{n}$ to
denote the inward normal derivative, i.e., $f_{n}\left(0\right)=f'\left(0\right)$,
and $f_{n}\left(1\right)=-f'\left(1\right)$, $\forall f\in\mathscr{H}_{E}$. \end{svmultproof2}

\begin{example}
Fix $p$, $0<p\leq1$, and set 
\begin{equation}
F_{p}\left(x\right):=1-\left|x\right|^{p},\;x\in\left(-\tfrac{1}{2},\tfrac{1}{2}\right);\label{eq:Fp-1}
\end{equation}
see \figref{Fp}. Then $F_{p}$ is positive definite and continuous;
and so $-F''_{p}\left(x-y\right)$ is a p.d. kernel, so $-F_{p}''$
is a positive definite distribution on $\left(-\frac{1}{2},\frac{1}{2}\right)$. 

We saw that if $p=1$, then 
\begin{equation}
-F_{1}''=2\delta\label{eq:Fp-2}
\end{equation}
where $\delta$ is the Dirac mass at $x=0$. But for $0<p<1$, $-F_{p}''$
does not have the form (\ref{eq:Fp-2}). We illustrate this if $p=\frac{1}{2}$.
Then 
\begin{equation}
-F_{\frac{1}{2}}''=\chi_{\left\{ x\neq0\right\} }\frac{1}{4}\left|x\right|^{-\frac{3}{2}}+\frac{1}{4}\delta'',\label{eq:Fp-3}
\end{equation}
where $\delta''$ is the double derivative of $\delta$ in the sense
of distributions. 

\begin{figure}
\includegraphics[width=0.5\textwidth]{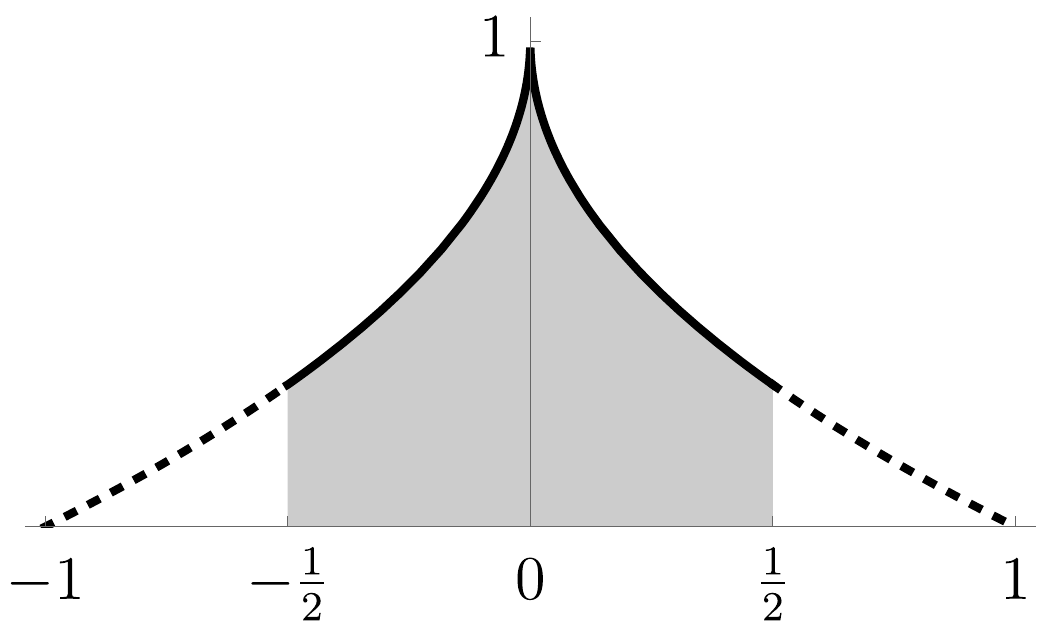}

\protect\caption{\label{fig:Fp}$F_{p}\left(x\right)=1-\left|x\right|^{p}$, $x\in\left(-\tfrac{1}{2},\tfrac{1}{2}\right)$
and $0<p\leq1$.}
\end{figure}
\end{example}
\begin{remark}
There is a notion of boundary measure in potential theory. Boundary
measures exist for any potential theory. In our case it works even
more generally, whenever p.d. $F_{ex}$ for bounded domains $\Omega\subset\mathbb{R}^{n}$,
and even Lie groups. But in the example of $F_{3}$, the boundary
is two points.\index{measure!boundary}

In all cases, we get $\mathscr{H}_{F}$ as a perturbation of the energy
form:\index{perturbation} 
\begin{equation}
\left\Vert \cdot\right\Vert _{\mathscr{H}_{F}}^{2}=\text{energy form }+\text{perturbation}\label{eq:en-11}
\end{equation}
It is a Green-Gauss-Stoke principle. There is still a boundary measure
for ALL bounded domains $\Omega\subset\mathbb{R}^{n}$, and even Lie
groups. \index{Theorem!Green-Gauss-Stoke-}\index{Energy Form}

And RKHS form 
\begin{equation}
\left\Vert \cdot\right\Vert _{\mathscr{H}_{F}}^{2}=\text{energy form }+\int_{\partial\Omega}f_{n}f\,d\mu_{\text{bd meas.}}\label{eq:en-12}
\end{equation}

For the general theory of boundary measures and their connection to
the Green-Gauss-Stoke principle, we refer readers to \cite{JorPea13,JoPe13b,Mo12,Bat90,Tel83,ACF09}. 

The approach via $\left\Vert \cdot\right\Vert _{\mathscr{H}_{F}}^{2}=$
\textquotedbl{}energy term + boundary term\textquotedbl{}, does not
fail to give a RKHS, but we must replace \textquotedbl{}energy term\textquotedbl{}
with an abstract Dirichlet form; see Refs \cite{HS12,Tre88}. \index{Dirichlet boundary condition}
\end{remark}

\subsection{Integral Kernels and Positive Definite Functions}

Let $0<H<1$ be given, and set \index{fractional Brownian motion}
\begin{equation}
K_{H}\left(x,y\right)=\frac{1}{2}\left(\left|x\right|^{2H}+\left|y\right|^{2H}-\left|x-y\right|^{2H}\right),\;\forall x,y\in\mathbb{R}.\label{eq:pdH-1-1}
\end{equation}
It is known that $K_{H}$$\left(\cdot,\cdot\right)$ is the covariance
kernel for \uline{fractional} Brownian motion \cite{AJ12,AJL11,AL08,Aur11}.
The special case $H=\frac{1}{2}$ is Brownian motion; and if $H=\frac{1}{2}$,
then 
\[
K_{\frac{1}{2}}\left(x,y\right)=\left|x\right|\wedge\left|y\right|=\min\left(\left|x\right|,\left|y\right|\right)
\]
if $x$ and $y$ have the same signs. Otherwise the covariance vanishes.
See Figure \ref{fig:Hker}. 

Set 
\begin{eqnarray*}
\widetilde{F}_{H}\left(x,y\right) & := & 1-\left|x-y\right|^{2H}\\
F_{H}\left(x\right) & := & 1-\left|x\right|^{2H}
\end{eqnarray*}
and we recover $F\left(x\right)=1-\left|x\right|\left(=F_{2}\right)$
as a special case of $H=\frac{1}{2}$. 

\begin{figure}[H]
\includegraphics[width=0.5\textwidth]{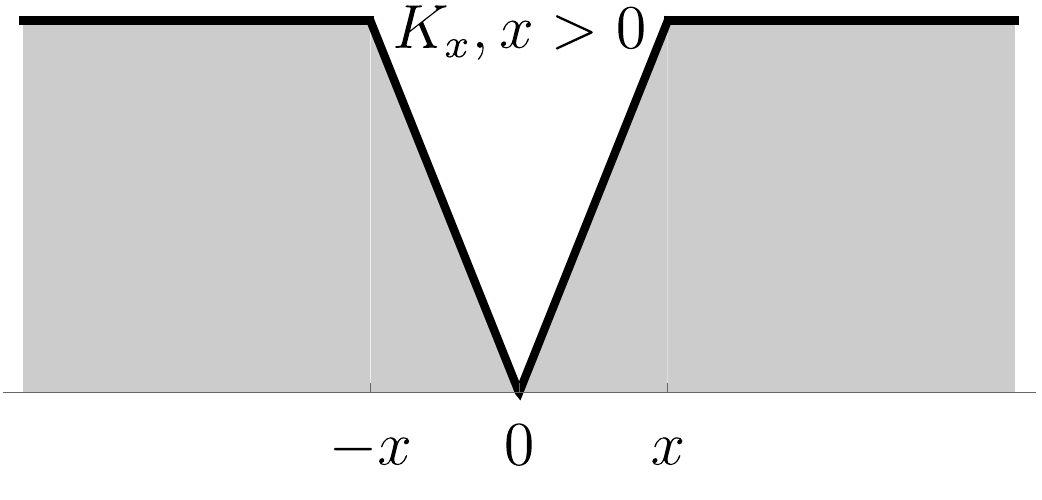}

\protect\caption{\label{fig:Hker}The integral kernel $K_{\frac{1}{2}}\left(x,y\right)=\left|x\right|\wedge\left|y\right|$.}
\end{figure}
\index{integral kernel}\index{operator!integral-}

Solving for $F_{H}$ in (\ref{eq:pdH-1-1}) ($0<H<1$ fixed), we then
get 
\begin{equation}
F_{H}\left(x-y\right)=\underset{L_{H}}{\underbrace{1-\left|x\right|^{2H}-\left|y\right|^{2H}}}+2K_{H}\left(x,y\right);\label{eq:pdH-1-2}
\end{equation}
and specializing to $x,y\in\left[0,1\right]$; we then get 
\[
F\left(x-y\right)=\underset{L}{\underbrace{1-x-y}}+2\left(x\wedge y\right);
\]
or written differently:
\begin{equation}
K_{F}\left(x,y\right)=L\left(x,y\right)+2E\left(x,y\right)\label{eq:pdH-1-3}
\end{equation}
where $E\left(x,y\right)=x\wedge y$ is the familiar covariance kernel
for Brownian motion.

Introducing the Mercer integral kernels corresponding to (\ref{eq:pdH-1-2}),
we therefore get:
\begin{eqnarray}
\left(T_{F_{H}}\varphi\right)\left(x\right) & = & \int_{0}^{1}\varphi\left(y\right)F_{H}\left(x-y\right)dy\label{eq:pdH-1-5}\\
\left(T_{L_{H}}\varphi\right)\left(x\right) & = & \int_{0}^{1}\varphi\left(y\right)L_{H}\left(x,y\right)dy\label{eq:pdH-1-6}\\
\left(T_{K_{H}}\varphi\right)\left(x\right) & = & \int_{0}^{1}\varphi\left(y\right)K_{H}\left(x,y\right)dy\label{eq:pdH-1-7}
\end{eqnarray}
for all $\varphi\in L^{2}\left(0,1\right)$, and all $x\in\left[0,1\right]$.
Note the special case of (\ref{eq:pdH-1-7}) for $H=\frac{1}{2}$
is 
\[
\left(T_{E}\varphi\right)\left(x\right)=\int_{0}^{1}\varphi\left(y\right)\left(x\wedge y\right)\,dy,\;\varphi\in L^{2}\left(0,1\right),x\in\left[0,1\right].
\]

We have the following lemma for these Mercer operators:
\begin{lemma}
Let $0<H<1$, and let $F_{H}$, $L_{H}$ and $K_{H}$ be as in (\ref{eq:pdH-1-2}),
then the corresponding Mercer operators satisfy:
\begin{equation}
T_{F_{H}}=T_{L_{H}}+2T_{K_{H}}.\label{eq:pdH-1-9}
\end{equation}
\end{lemma}
\begin{svmultproof2}
This is an easy computation, using (\ref{eq:pdH-1-2}), and (\ref{eq:pdH-1-5})-(\ref{eq:pdH-1-7}).
\end{svmultproof2}

\subsection{\label{sub:OUprocess}The Ornstein-Uhlenbeck Process Revisited}

The reproducing kernel Hilbert space $\mathscr{H}_{F_{2}}$ in (\ref{eq:RKHS-eg-3})
is used in computations of Ito-integrals of Brownian motion; while
the corresponding RKHS $\mathscr{H}_{F3}$ from (\ref{eq:RKHS-eg-6})-(\ref{eq:RKHS-eg-7})
is used in calculations of stochastic integration with the Ornstein-Uhlenbeck
process.\index{Ito-integral}

Motivated by Newton's second law of motion, the Ornstein-Uhlenbeck\index{Ornstein-Uhlenbeck}
velocity process is proposed to model a random external driving force.
In 1D, the process is the solution to the following stochastic differential
equation 
\begin{equation}
dv_{t}=-\gamma v_{t}dt+\beta dB_{t},\:\gamma,\beta>0.\label{eq:ou-1}
\end{equation}
Here, $-\gamma v_{t}$ is the dissipation, $\beta dB_{t}$ denotes
a random fluctuation, and $B_{t}$ is the standard Brownian motion.

Assuming the particle starts at $t=0$. The solution to (\ref{eq:ou-1})
is a Gaussian stochastic process\index{stochastic processes} such
that
\begin{eqnarray*}
\mathbb{E}\left[v_{t}\right] & = & v_{0}e^{-\gamma t}\\
var\left[v_{t}\right] & = & \frac{\beta^{2}}{2\gamma}\left(1-e^{-2\gamma t}\right);
\end{eqnarray*}
with $v_{0}$ being the initial velocity. See Figure \ref{fig:OU}.
Moreover, the process has the following covariance function 
\[
c\left(s,t\right)=\frac{\beta^{2}}{2\gamma}\left(e^{-\gamma\left|t-s\right|}-e^{-\gamma\left|s+t\right|}\right).
\]
If we wait long enough, it turns to a stationary process such that
\[
c\left(s,t\right)\sim\frac{\beta^{2}}{2\gamma}e^{-\gamma\left|t-s\right|}.
\]
This corresponds to the function $F_{3}$. See also Remark \ref{rem:OUprocess}.
\index{distribution!Gaussian-}

\begin{figure}
\includegraphics[width=0.7\textwidth]{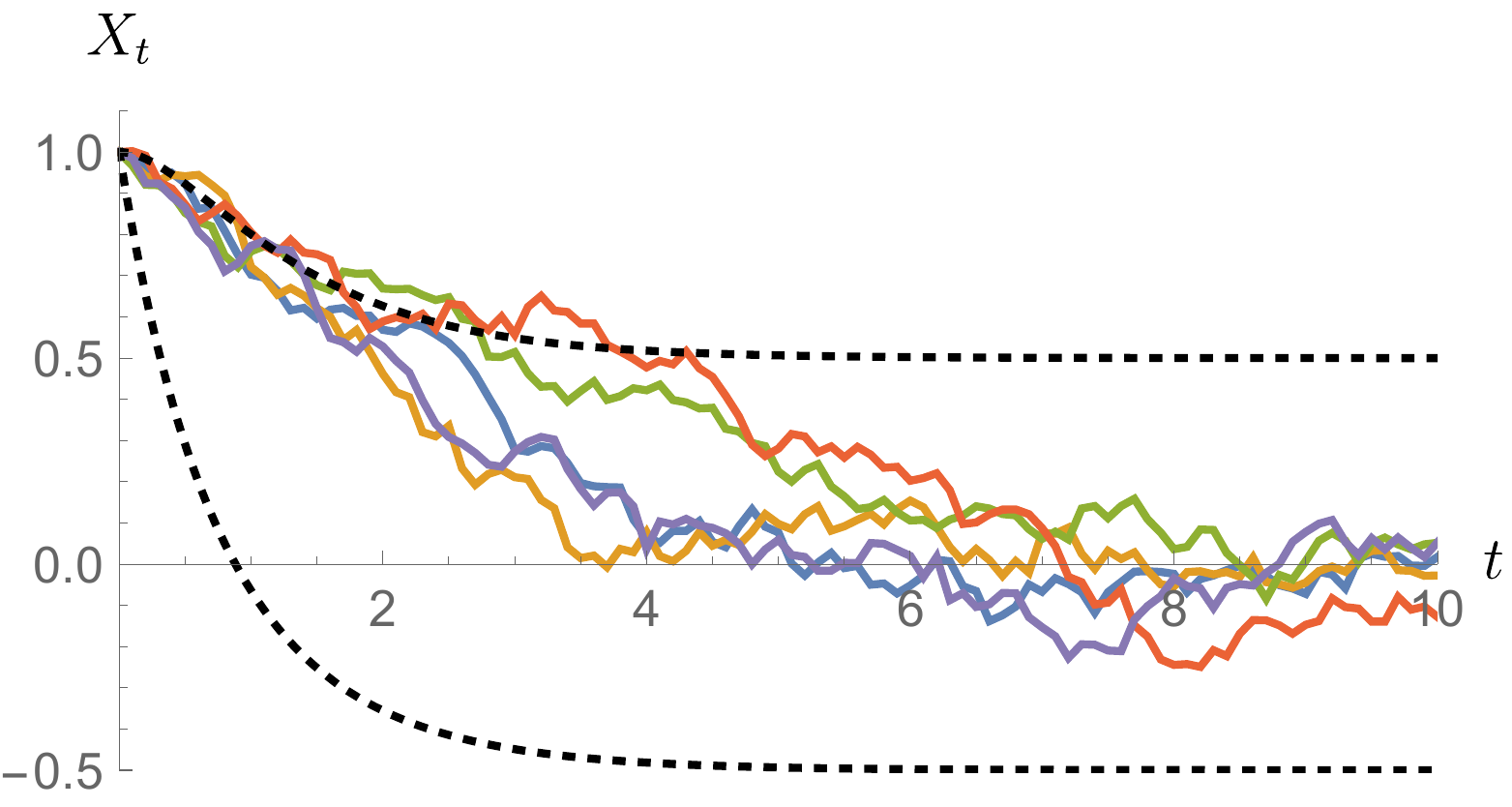}

\protect\caption{\label{fig:OU}Monte-Carlo simulation of Ornstein-Uhlenbeck process
with 5 sample paths. $\beta=\gamma=1$, $v_{0}=1$, $t\in\left[0,10\right]$\index{Monte-Carlo simulation}}
\end{figure}

\subsection{An Overview of the Two Cases: $F_{2}$ and $F_{3}$.}

In Table \ref{tab:F23} below, we give a list of properties for the
particular two cases $F_{2}$ and $F_{3}$ from table \ref{tab:F1-F6}.
Special attention to these examples is merited; first they share a
number of properties with much wider families of locally defined positive
definite functions; and these properties are more transparent in their
simplest form. Secondly, there are important differences between cases
$F_{2}$ and $F_{3}$, and the table serves to highlight both differences
and similarities. A particular feature that is common for the two
is that, when the Mercer operator $T_{F}$ is introduced, then its
inverse $T_{F}^{-1}$ exists as an unbounded positive and selfadjoint
operator in $\mathscr{H}_{F}$. Moreover, in each case, this operator
coincides with the Friedrichs extension of a certain second order
Hermitian semibounded operator (calculated from $D^{\left(F\right)}$),
with dense domain in $\mathscr{H}_{F}$. 

\index{Friedrichs extension}

\index{operator!unbounded}

\index{operator!Mercer}

\index{operator!semibounded}

\index{locally defined}

\section{\label{sec:hdim}Higher Dimensions}

Our function above for $F_{3}$ (Sections \ref{sub:green} and \ref{sub:F3})
admits a natural extension to $\mathbb{R}^{n}$, $n>1$, as follows.

Let $\Omega\subset\mathbb{R}^{n}$ be a subset satisfying $\Omega\neq\emptyset$,
$\Omega$ open and connected; and assume $\overline{\Omega}$ is compact.
Let $\triangle=\sum_{i=1}^{n}\left(\partial/\partial x_{i}\right)^{2}$
be the usual Laplacian in $n$-variables.
\begin{lemma}
Let $F:\Omega-\Omega\rightarrow\mathbb{C}$ be continuous and positive
definite; and let $\mathscr{H}_{F}$ be the corresponding RKHS. In
$\mathscr{H}_{F}$, we set
\begin{align}
L^{\left(F\right)}\left(F_{\varphi}\right) & :=F_{\triangle\varphi},\;\forall\varphi\in C_{c}^{\infty}\left(\Omega\right),\:\mbox{with}\\
dom\bigl(L^{\left(F\right)}\bigr) & =\left\{ F_{\varphi}\:\big|\:\varphi\in C_{c}^{\infty}\left(\Omega\right)\right\} ;
\end{align}
then $L^{\left(F\right)}$ is semibounded in $\mathscr{H}_{F}$, $L^{\left(F\right)}\leq0$
in the order of Hermitian operators.

Let $T_{F}$ be the Mercer operator; then there is a positive constant
$c\:(\:=c_{\left(F,\Omega\right)}>0)$ such that $T_{F}^{-1}$ \textup{is
a selfadjoint extension of the densely defined operator $c\left(I-L^{\left(F\right)}\right)$
in $\mathscr{H}_{F}$. }\end{lemma}
\begin{svmultproof2}
The ideas for the proof are contained in Sections \ref{sub:green}
and \ref{sub:F3} above. To get the general conclusions, we may combine
these considerations with the general theory of Green's functions
for elliptic linear PDEs (partial differential equations); see also
\cite{Nel57,JLW69}.
\end{svmultproof2}

\index{operator!Mercer}\index{operator!partial differential}\index{Hermitian}

\index{Green's function}

\index{trace}

\begin{table}
\begin{tabular}{|>{\centering}p{0.5\textwidth}|>{\centering}p{0.5\textwidth}|}
\hline 
\begin{minipage}[t]{0.47\columnwidth}%
\medskip{}

1. $F_{2}\left(x\right)=1-\left|x\right|$, $\left|x\right|<\frac{1}{2}$

2. Mercer operator 
\[
T_{F_{2}}:L^{2}\left(0,\tfrac{1}{2}\right)\rightarrow L^{2}\left(0,\tfrac{1}{2}\right)
\]

3. $T_{F_{2}}^{-1}$ is unbounded, selfadjoint \index{operator!selfadjoint}
\begin{svmultproof2}
Since $T_{F}$ is positive, bounded, and trace class, it follows that
$T_{F_{2}}^{-1}$ is positive, unbounded, and selfadjoint.
\end{svmultproof2}

4. $T_{F_{2}}^{-1}=$ Friedrichs extension of 
\[
-\tfrac{1}{2}\bigl(\tfrac{d}{dx}\bigr)^{2}\Big|_{C_{c}^{\infty}\left(0,\frac{1}{2}\right)}
\]
as a selfadjoint operator on $L^{2}\left(0,\frac{1}{2}\right)$. 

\emph{Sketch of proof.} 

Setting 
\[
u\left(x\right)=\int_{0}^{\frac{1}{2}}\varphi\left(y\right)\left(1-\left|x-y\right|\right)dy
\]
then 
\begin{eqnarray*}
u'' & = & -2\varphi\\
 & \Updownarrow\\
u & = & \left(-\tfrac{1}{2}\bigl(\tfrac{d}{dx}\bigr)^{2}\right)^{-1}\varphi
\end{eqnarray*}
and so 
\[
T_{F_{2}}=\left(-\tfrac{1}{2}\bigl(\tfrac{d}{dx}\bigr)^{2}\right)^{-1}.
\]
And we get a selfadjoint extension 
\[
T_{F_{2}}^{-1}\supset-\tfrac{1}{2}\bigl(\tfrac{d}{dx}\bigr)^{2}
\]
in $L^{2}\left(0,\frac{1}{2}\right)$, where the containment refers
to operator graphs.\index{operator!graph of-}

\medskip{}
\end{minipage} & %
\begin{minipage}[t]{0.47\columnwidth}%
\medskip{}

1. $F_{3}\left(x\right)=e^{-\left|x\right|}$, $\left|x\right|<1$

2. Mercer operator 
\[
T_{F_{3}}:L^{2}\left(0,1\right)\rightarrow L^{2}\left(0,1\right)
\]

3. $T_{F_{3}}^{-1}$ is unbounded, selfadjoint
\begin{svmultproof2}
Same argument as in the proof for $T_{F_{2}}^{-1}$; also follows
from Mercer's theorem.
\end{svmultproof2}

4. $T_{F_{2}}^{-1}=$ Friedrichs extension of 
\[
\tfrac{1}{2}\bigl(I-\bigl(\tfrac{d}{dx}\bigr)^{2}\bigr)\Big|_{C_{c}^{\infty}\left(0,1\right)}
\]
as a selfadjoint operator on $L^{2}\left(0,1\right)$. 

\emph{Sketch of proof.} 

Setting
\[
u\left(x\right)=\int_{0}^{1}\varphi\left(y\right)e^{-\left|x-y\right|}dy
\]
then 
\begin{eqnarray*}
u'' & = & u-2\varphi\\
 & \Updownarrow\\
u & = & \left(\tfrac{1}{2}\bigl(1-\bigl(\tfrac{d}{dx}\bigr)^{2}\bigr)\right)^{-1}\varphi
\end{eqnarray*}
and so 
\[
T_{F_{3}}=\left(\tfrac{1}{2}\left(I-\bigl(\tfrac{d}{dx}\bigr)^{2}\right)\right)^{-1}.
\]
Now, a selfadjoint extension\index{selfadjoint extension} 
\[
T_{F_{3}}^{-1}\supset\tfrac{1}{2}\left(I-\bigl(\tfrac{d}{dx}\bigr)^{2}\right)
\]
in $L^{2}\left(0,1\right)$.

\medskip{}
\end{minipage}\tabularnewline
\hline 
\end{tabular}

\protect\caption{\label{tab:F23}An overview of two cases: $F_{2}$ v.s. $F_{3}$.}
\end{table}

\chapter{\label{chap:CompareFK}Comparing the Different RKHSs $\mathscr{H}_{F}$
and $\mathscr{H}_{K}$}

In the earlier chapters we have studied extension problems for particular
locally defined positive definite functions $F$. In each case, the
p.d. function $F$ was fixed. Many classes of p.d. functions were
studied but not compared. Indeed, the cases we studied have varied
between different levels of generality, and varied between cases separated
by different technical assumptions. All our results, in turn, have
been motivated by explicit applications.

In the present chapter, we turn to a comparison between the results
we obtain for pairs of locally defined positive definite functions.
We answer such questions as this: Given an open domain $\Omega$ (in
an ambient group $G$), and two continuous positive functions $F$
and $K$, both defined on $\Omega^{-1}\Omega$, how do the two associated
RKHSs compare? (Theorem \ref{thm:FK-1}.) How to compare spectral
theoretic information for the two? How does the operator theory compare,
for the associated pair of skew-Hermitian operators $D^{\left(F\right)}$
and $D^{\left(K\right)}$? (Theorem \ref{thm:FK-app-1}.) What if
$F$ is fixed, and $K$ is taken to be the complex conjugate of $F$?
Theorem \ref{thm:FKc}.) What are the properties of the imaginary
part of $F$? (Section \ref{sec:imgF}.)

Before we turn to comparison of pairs of RKHSs, we will prove a lemma
which accounts for two uniformity principles for positive definite
continuous functions defined on open subsets in locally compact groups:\index{RKHS}\index{positive definite}
\begin{lemma}
\label{lem:F-bd}Let $G$ be a locally compact group, and let $\Omega\subset G$
be an open and connected subset, $\Omega\neq\emptyset$. Let $F:\Omega^{-1}\Omega\rightarrow\mathbb{C}$
be a continuous positive definite function satisfying $F\left(e\right)=1$,
where $e\in G$ is the unit for the group operation. 
\begin{enumerate}
\item Then $F$ extends by limit to a continuous p.d. function
\begin{equation}
F^{\left(ex\right)}:\overline{\Omega}^{-1}\overline{\Omega}\longrightarrow\mathbb{C}.
\end{equation}

\item Moreover, the two p.d. functions $F$ and $F^{\left(ex\right)}$ have
the same RKHS consisting of continuous functions $\xi$ on $\overline{\Omega}$
such that, $\exists\,0<A<\infty$, $A=A_{\xi}$, s.t.
\begin{equation}
\left|\int_{\overline{\Omega}}\overline{\xi\left(x\right)}\varphi\left(x\right)dx\right|^{2}\leq A\left\Vert F_{\varphi}\right\Vert _{\mathscr{H}_{F}}^{2},\;\forall\varphi\in C_{c}\left(\Omega\right)\label{eq:lcg-1-1}
\end{equation}
where $dx=$ Haar measure on $G$, \index{measure!Haar} 
\[
F_{\varphi}\left(\cdot\right)=\int_{\Omega}\varphi\left(y\right)F\left(y^{-1}\cdot\right)dy,\;\varphi\in C_{c}\left(\Omega\right);
\]
and $\left\Vert F_{\varphi}\right\Vert _{\mathscr{H}_{F}}$ denotes
the $\mathscr{H}_{F}$-norm of $F_{\varphi}$. 
\item Every $\xi\in\mathscr{H}_{F}$ satisfies the following a priori estimate:
\begin{equation}
\left|\xi\left(x\right)-\xi\left(y\right)\right|^{2}\leq2\left\Vert \xi\right\Vert _{\mathscr{H}_{F}}^{2}\left(1-\Re\left\{ F\left(x^{-1}y\right)\right\} \right)
\end{equation}
for all $\xi\in\mathscr{H}_{F}$, and all $x,y\in\overline{\Omega}$. 
\end{enumerate}
\end{lemma}
\begin{svmultproof2}
The arguments in the proof only use standard tools from the theory
of reproducing kernel Hilbert spaces. We covered a special case in
Section \ref{sec:mercer}, and so we omit the details here.
\end{svmultproof2}

Now add the further assumption on the subset $\Omega\subset G$ from
the lemma: Assume in addition that $\Omega$ has compact closure,
so $\overline{\Omega}$ is compact. Let $\partial\Omega=\overline{\Omega}\backslash\Omega$
be the boundary. With this assumption, we get automatically the inclusion
\[
C\left(\overline{\Omega}\right)\subset L^{2}\left(\overline{\Omega}\right)
\]
since continuous functions on $\overline{\Omega}$ are automatically
uniformly bounded, and Haar measure on $G$ has the property that
$\left|\overline{\Omega}\right|<\infty$.
\begin{definition}
Let $\left(\Omega,F\right)$ be as above, i.e., 
\[
F:\overline{\Omega}^{-1}\overline{\Omega}\longrightarrow\mathbb{C}
\]
is continuous and positive definite\index{positive definite}. Assume
$G$ is a Lie group.\index{group!Lie} Recall, extension by limit
to $\overline{\Omega}$ is automatic by the limit. Let $\beta\in\mathscr{M}\left(\partial\Omega\right)$
be a positive finite measure on the boundary $\partial\Omega$. We
say that $\beta$ is a boundary measure \index{measure!boundary}iff
\index{Lie group}
\begin{equation}
\left\langle T_{F}f,\xi\right\rangle _{\mathscr{H}_{F}}-\int_{\Omega}\overline{f\left(x\right)}\xi\left(x\right)dx=\int_{\partial\Omega}\overline{\left(T_{F}f\right)_{n}\left(\sigma\right)}\xi\left(\sigma\right)d\beta\left(\sigma\right)
\end{equation}
holds for all $f\in C\left(\overline{\Omega}\right)$, $\forall\xi\in\mathscr{H}_{F}$,
where $\left(\cdot\right)_{n}=$ normal derivative computed on the
boundary $\partial\Omega$ of $\Omega$. \index{derivative!normal-}\end{definition}
\begin{remark}
\label{rem:cos}For the example $G=\mathbb{R}$, $\overline{\Omega}=\left[0,1\right]$,
$\partial\Omega=\left\{ 0,1\right\} $, and $F=F_{3}$ on $\left[0,1\right]$,
where 
\[
F\left(x\right)=e^{-\left|x\right|},\;\forall x\in\left[-1,1\right],
\]
the boundary measure is
\begin{equation}
\beta=\frac{1}{2}\left(\delta_{0}+\delta_{1}\right).
\end{equation}

\end{remark}
Let $G$ be a locally compact group with left-Haar measure, and let
$\Omega\subset G$ be a non-empty subset satisfying: $\Omega$ is
open and connected; and is of finite Haar measure; write $\left|\Omega\right|<\infty$.
The Hilbert space $L^{2}\left(\Omega\right)=L^{2}\left(\Omega,dx\right)$
is the usual $L^{2}$-space of measurable functions $f$ on $\Omega$
such that 
\begin{equation}
\left\Vert f\right\Vert _{L^{2}\left(\Omega\right)}^{2}:=\int_{\Omega}\left|f\left(x\right)\right|^{2}dx<\infty.\label{eq:FK-1}
\end{equation}

\begin{definition}
\label{def:FK-1}Let $F$ and $K$ be two continuous and positive
definite functions defined on 
\begin{equation}
\Omega^{-1}\cdot\Omega:=\left\{ x^{-1}y\:\Big|\:x,y\in\Omega\right\} .\label{eq:FK-2}
\end{equation}
We say that $K\ll F$ iff there is a finite constant $A$ such that
for all finite subsets $\left\{ x_{i}\right\} _{i=1}^{N}\subset\Omega$,
and all systems $\left\{ c_{i}\right\} _{i=1}^{N}\subset\mathbb{C}$,
we have:
\begin{equation}
\sum_{i}\sum_{j}\overline{c_{i}}c_{j}K\left(x_{i}^{-1}x_{j}\right)\leq A\sum_{i}\sum_{j}\overline{c_{i}}c_{j}F\left(x_{i}^{-1}x_{j}\right).\label{eq:FK-3}
\end{equation}
\end{definition}
\begin{lemma}
\label{lem:FK-1}Let $F$ and $K$ be as above; then $K\ll F$ if
and only if there is a finite constant $A\in\mathbb{R}_{+}$ such
that
\begin{equation}
\int_{\Omega}\int_{\Omega}\overline{\varphi\left(x\right)}\varphi\left(y\right)K\left(x^{-1}y\right)dxdy\leq A\int_{\Omega}\int_{\Omega}\overline{\varphi\left(x\right)}\varphi\left(y\right)F\left(x^{-1}y\right)dxdy\label{eq:FK-4}
\end{equation}
holds for all $\varphi\in C_{c}\left(\Omega\right)$. The constant
$A$ in (\ref{eq:FK-3}) and (\ref{eq:FK-4}) will be the same.\end{lemma}
\begin{svmultproof2}
Easy; use an approximate identity\index{approximate identity} in
$G$, see e.g., \cite{Rud73,Ru90}.
\end{svmultproof2}

Setting 
\begin{equation}
F_{\varphi}\left(x\right)=\int_{\Omega}\varphi\left(y\right)F\left(y^{-1}x\right)dy,\label{eq:FK-5}
\end{equation}
and similarly for $K_{\varphi}=\int_{\Omega}\varphi\left(y\right)K\left(y^{-1}\cdot\right)dy$,
we note that $K\ll F$ if and only if:
\begin{equation}
\left\Vert K_{\varphi}\right\Vert _{\mathscr{H}_{K}}\leq\sqrt{A}\left\Vert F_{\varphi}\right\Vert _{\mathscr{H}_{F}},\;\forall\varphi\in C_{c}\left(\Omega\right).\label{eq:FK-6}
\end{equation}
Further, note that, if $G$ is also a \emph{Lie group}, then (\ref{eq:FK-6})
follows from checking it only for all $\varphi\in C_{c}^{\infty}\left(G\right)$.
See Lemma \ref{lem:li-meas}.\index{Lie group}
\begin{theorem}
\label{thm:FK-1}Let $\Omega$, $F$ and $K$ be as in Definition
\ref{def:FK-1}, i.e., both continuous and p.d. on the set $\Omega^{-1}\cdot\Omega$
in (\ref{eq:FK-2}); then the following two conditions are equivalent:
\begin{enumerate}[label=(\roman{enumi}), ref=\roman{enumi}]
\item \label{enu:cp1} $K\ll F$, and
\item \label{enu:cp2}$\mathscr{H}_{K}$ \uline{is} a closed subspace
of $\mathscr{H}_{F}$.
\end{enumerate}
\end{theorem}
\begin{svmultproof2}
$\Downarrow$ Assume $K\ll F$, we then define a linear operator $l:\mathscr{H}_{F}\rightarrow\mathscr{H}_{K}$,
setting 
\begin{equation}
l\left(F_{\varphi}\right):=K_{\varphi},\;\forall\varphi\in C_{c}\left(\Omega\right).\label{eq:FK-7}
\end{equation}
We now use (\ref{eq:FK-6}), plus the fact that $\mathscr{H}_{F}$
is the $\left\Vert \cdot\right\Vert _{\mathscr{H}_{F}}$-completion
of 
\[
\left\{ F_{\varphi}\:\Big|\:\varphi\in C_{c}\left(\Omega\right)\right\} ;
\]
and similarly for $\mathscr{H}_{K}$.

As a result of (\ref{eq:FK-7}) and (\ref{eq:FK-6}), we get a canonical
extension of $l$ to a bounded linear operator, also denoted $l:\mathscr{H}_{F}\rightarrow\mathscr{H}_{K}$,
and 
\begin{equation}
\left\Vert l\left(\xi\right)\right\Vert _{\mathscr{H}_{K}}\leq\sqrt{A}\left\Vert \xi\right\Vert _{\mathscr{H}_{F}},\mbox{ for all }\xi\in\mathscr{H}_{F}.\label{eq:FK-8}
\end{equation}

\index{operator!bounded}\index{Hilbert space}

We interrupt the proof to give a lemma:\end{svmultproof2}

\begin{lemma}
Let $F$, $K$, $\Omega$ be as above. Assume $K\ll F$, and let $l:\mathscr{H}_{F}\rightarrow\mathscr{H}_{K}$
be the bounded operator introduced in (\ref{eq:FK-7}) and (\ref{eq:FK-8}).
Then the adjoint operator $l^{*}:\mathscr{H}_{K}\rightarrow\mathscr{H}_{F}$
satisfies:
\begin{equation}
\left(l^{*}\left(\xi\right)\right)\left(x\right)=\xi\left(x\right),\mbox{ for all }\xi\in\mathscr{H}_{K},x\in\Omega.\label{eq:FK-9}
\end{equation}
And of course, $l^{*}\left(\xi\right)\in\mathscr{H}_{F}$.
\begin{svmultproof2}
By the definite of $l^{*}$, as the adjoint of a bounded linear operator
between Hilbert spaces, we get\index{operator!adjoint of an-}\index{operator!-norm}
\begin{equation}
\left\Vert l^{*}\right\Vert _{\mathscr{H}_{K}\rightarrow\mathscr{H}_{F}}=\left\Vert l\right\Vert _{\mathscr{H}_{F}\rightarrow\mathscr{H}_{K}}\leq\sqrt{A}\label{eq:FK-10}
\end{equation}
for the respective operator norms; and 
\begin{equation}
\left\langle l^{*}\left(\xi\right),F_{\varphi}\right\rangle _{\mathscr{H}_{F}}=\left\langle \xi,K_{\varphi}\right\rangle _{\mathscr{H}_{K}},\;\forall\varphi\in C_{c}\left(\Omega\right).\label{eq:FK-11}
\end{equation}

Using now the reproducing property in the two RKHSs, we get:
\begin{eqnarray*}
l.h.s.\left(\ref{eq:FK-11}\right) & = & \int_{\Omega}\overline{l^{*}\left(\xi\right)\left(x\right)}\varphi\left(x\right)dx,\mbox{ and}\\
r.h.s.\left(\ref{eq:FK-11}\right) & = & \int_{\Omega}\overline{\xi\left(x\right)}\varphi\left(x\right)dx,\;\mbox{ for all }\varphi\in C_{c}\left(\Omega\right).
\end{eqnarray*}
Taking now approximations in $C_{c}\left(\Omega\right)$ to the Dirac
masses $\{\delta_{x}\:|\:x\in\Omega\}$, the desired conclusion (\ref{eq:FK-9})
follows.
\end{svmultproof2}

\end{lemma}
\begin{svmultproof2}
(of Theorem \ref{thm:FK-1} continued.) Assume $K\ll F$, the lemma
proves that $\mathscr{H}_{K}$ identifies with a linear space of continuous
functions $\xi$ on $\Omega$, and if $\xi\in\mathscr{H}_{K}$, then
it is also in $\mathscr{H}_{F}$.

We claim that this is a closed subspace in $\mathscr{H}_{F}$ relative
to the $\mathscr{H}_{F}$-norm. 

\uline{Step 1}. Let $\left\{ \xi_{n}\right\} \subset\mathscr{H}_{K}$
satisfying 
\[
\lim_{n,m\rightarrow\infty}\left\Vert \xi_{n}-\xi_{m}\right\Vert _{\mathscr{H}_{F}}=0.
\]
By (\ref{eq:FK-8}) and (\ref{eq:FK-9}), the lemma; we get
\[
\lim_{n,m\rightarrow\infty}\left\Vert \xi_{n}-\xi_{m}\right\Vert _{\mathscr{H}_{K}}=0.
\]

\uline{Step 2}. Since $\mathscr{H}_{K}$ is complete, we get $\chi\in\mathscr{H}_{K}$
such that 
\begin{equation}
\lim_{n\rightarrow\infty}\left\Vert \xi_{n}-\chi\right\Vert _{\mathscr{H}_{K}}=0.\label{eq:FK-12}
\end{equation}

\uline{Step 3}. We claim that this $\mathscr{H}_{K}$-limit $\chi$
also defines a unique element in $\mathscr{H}_{F}$, and it is therefore
the $\mathscr{H}_{F}$-limit. 

We have for all $\varphi\in C_{c}\left(\Omega\right)$:
\begin{eqnarray*}
\left|\int_{\Omega}\overline{\chi\left(x\right)}\varphi\left(x\right)dx\right| & \leq & \left\Vert \chi\right\Vert _{\mathscr{H}_{K}}\left\Vert K_{\varphi}\right\Vert _{\mathscr{H}_{K}}\\
 & \leq & \left\Vert \chi\right\Vert _{\mathscr{H}_{K}}\sqrt{A}\left\Vert F_{\varphi}\right\Vert _{\mathscr{H}_{F}};
\end{eqnarray*}
and so $\chi\in\mathscr{H}_{F}$.

We now turn to the converse implication of Theorem \ref{thm:FK-1}:

$\Uparrow$ Assume $F$ and $K$ are as in the statement of the theorem;
and that $\mathscr{H}_{K}$ is a close subspace in $\mathscr{H}_{F}$
via identification of the respective continuous functions on $\Omega$.
We then prove that $K\ll F$. 

Now let $P_{K}$ denote the orthogonal projection\index{projection}
of $\mathscr{H}_{F}$ onto the closed subspace $\mathscr{H}_{K}$.
We claim that 
\begin{equation}
P_{K}\left(F_{\varphi}\right)=K_{\varphi},\;\forall\varphi\in C_{c}\left(\Omega\right).\label{eq:FK-13}
\end{equation}
Using the uniqueness of the projection $P_{K}$, we need to verify
that $F_{\varphi}-K_{\varphi}\in\mathscr{H}_{F}\ominus\mathscr{H}_{K}$;
i.e., that 
\begin{equation}
\left\langle F_{\varphi}-K_{\varphi},\xi_{K}\right\rangle _{\mathscr{H}_{F}}=0,\;\mbox{for all }\xi_{K}\in\mathscr{H}_{K}.\label{eq:FK-14}
\end{equation}
But since $\mathscr{H}_{K}\subset\mathscr{H}_{F}$, we have 
\[
l.h.s.\left(\ref{eq:FK-14}\right)=\int_{\Omega}\overline{\varphi\left(x\right)}\xi_{K}\left(x\right)dx-\int_{\Omega}\overline{\varphi\left(x\right)}\xi_{K}\left(x\right)dx=0,
\]
for all $\varphi\in C_{c}\left(\Omega\right)$. This proves (\ref{eq:FK-13}).

To verify $K\ll F$, we use the criterion (\ref{eq:FK-6}) from Lemma
\ref{lem:FK-1}. Indeed, consider $K_{\varphi}\in\mathscr{H}_{K}$.
Since $\mathscr{H}_{K}\subset\mathscr{H}_{F}$, we get 
\[
l\left(F_{\varphi}\right)=P_{K}\left(F_{\varphi}\right)=K_{\varphi},\;\mbox{and}
\]
\[
\left\Vert K_{\varphi}\right\Vert _{\mathscr{H}_{K}}=\left\Vert l\left(F_{\varphi}\right)\right\Vert _{\mathscr{H}_{K}}\leq\sqrt{A}\left\Vert F_{\varphi}\right\Vert _{\mathscr{H}_{F}}
\]
which is the desired estimate (\ref{eq:FK-6}).\end{svmultproof2}

\begin{example}
\label{exa:cp}To illustrate the conclusion in Theorem \ref{thm:FK-1},
take 
\begin{equation}
K\left(x\right)=\left(\frac{\sin x}{x}\right)^{2},\;\mbox{and }F\left(x\right)=\frac{1}{1+x^{2}};\label{eq:cp1}
\end{equation}
both defined on the fixed interval $-\frac{1}{4}<x<\frac{1}{4}$;
and so the interval for $\Omega$ in both cases can be take to be
$\Omega=\left(0,\frac{1}{4}\right)$; compare with $F_{4}$ and $F_{1}$
from Table \ref{tab:F1-F6}.\end{example}
\begin{claim}
$K\ll F$, and $\mathscr{H}_{K}\subseteq\mathscr{H}_{F}$.\end{claim}
\begin{svmultproof2}
Using Table \ref{tab:Table-3}, we can identify measures $\mu_{K}\in Ext\left(K\right)$,
and $\mu_{F}\in Ext\left(F\right)$. With the use of Corollary \ref{cor:muHF},
and Corollary \ref{cor:lcg-isom}, it is now easy to prove that $K\ll F$;
i.e., that condition (\ref{enu:cp1}) in Theorem \ref{thm:FK-1} is
satisfied. Also, recall from Lemma \ref{lem:RKHS-def-by-integral}
that the RKHSs are generated by the respective kernels, i.e., the
functions $K$ and $F$: 

For $x_{0}\in\left[0,\frac{1}{4}\right]=\overline{\Omega}$, set 
\begin{equation}
K_{x_{0}}\left(y\right)=\left(\frac{\sin\left(y-x_{0}\right)}{y-x_{0}}\right)^{2},\quad y\in\Omega.\label{eq:cp2}
\end{equation}
To show directly that $\mathscr{H}_{K}$ \emph{is} a subspace of $\mathscr{H}_{F}$,
we use Corollary \ref{cor:muHF}: We find a \emph{signed measure}
$v_{x_{0}}\in\mathfrak{M}_{2}\left(F,\Omega\right)$ such that 
\begin{equation}
K_{x_{0}}\left(\cdot\right)=\left(v_{v_{0}}*F\right)\left(\cdot\right)\;\mbox{on \ensuremath{\Omega} holds.}\label{eq:cp3}
\end{equation}
The r.h.s. in (\ref{eq:cp3}) is short hand for 
\begin{equation}
K_{x_{0}}\left(x\right)=\int_{0}^{\frac{1}{4}}\left(\frac{1}{1+\left(x-y\right)^{2}}d\nu_{x_{0}}\left(y\right)\right),\label{eq:cp4}
\end{equation}
using Fourier transform, eq (\ref{eq:cp4}) is easily solved by the
signed measure $v_{x_{0}}$ s.t. 
\begin{equation}
\widehat{dv_{0}}\left(\lambda\right)=\chi_{\left[-\frac{1}{2},\frac{1}{2}\right]}\left(\lambda\right)\left(1-2\left|\lambda\right|\right)e^{\left|\lambda\right|}e^{ix_{0}\lambda}d\lambda.\label{eq:cp5}
\end{equation}
Note that the r.h.s. in (\ref{eq:cp5}) is of compact support; and
so an easy Fourier inversion then yields the signed measure $v_{x_{0}}\in\mathfrak{M}_{2}\left(F,\Omega\right)$
which solves eq (\ref{eq:cp3}). Hence $K_{x_{0}}\left(\cdot\right)\in\mathscr{H}_{F}$;
and so $\mathscr{H}_{K}$ \emph{is} a subspace of $\mathscr{H}_{F}$.
\end{svmultproof2}

\section{Applications}

Below we give an application of Theorem \ref{thm:FK-1} to the deficiency-index
problem, and to the computation of the deficiency spaces; see also
Lemma \ref{lem:exp-1}, and Lemma \ref{lem:def1}.

As above, we will consider two given continuous p.d. functions $F$
and $K$, but the group now is $G=\mathbb{R}$: We pick $a,b\in\mathbb{R}_{+}$,
$0<a<b$, such that $F$ is defined on $\left(-b,b\right)$, and $K$
on $\left(-a,a\right)$. The corresponding two RKHSs will be denoted
$\mathscr{H}_{F}$ and $\mathscr{H}_{K}$. We say that $K\ll F$ iff
there is a finite constant $A$ such that\index{RKHS}
\begin{equation}
\left\Vert K_{\varphi}\right\Vert _{\mathscr{H}_{K}}^{2}\leq A\left\Vert F_{\varphi}\right\Vert _{\mathscr{H}_{F}}^{2}\label{eq:FK-app-1}
\end{equation}
for all $\varphi\in C_{c}\left(0,a\right)$. Now this is a slight
adaptation of our Definition \ref{def:FK-1} above, but this modification
will be needed; for example in computing the indices of two p.d. functions
$F_{2}$ and $F_{3}$ from Table \ref{tab:F1-F6}; see also Section
\ref{sec:F2F3} below. In fact, a simple direct checking shows that
\begin{equation}
F_{2}\ll F_{3}\quad\left(\mbox{see Table }\ref{tab:F1-F6}\right),\label{eq:FK-app-2}
\end{equation}
and we now take $a=\frac{1}{2}$, $b=1$. 

Here, $F_{2}\left(x\right)=1-\left|x\right|$ in $\left|x\right|<\frac{1}{2}$;
and $F_{3}\left(x\right)=e^{-\left|x\right|}$ in $\left|x\right|<1$;
see Figure \ref{fig:FK-1}.

\begin{figure}[H]
\begin{minipage}[t]{1\columnwidth}%
\subfloat[$F_{2}\left(x\right)=1-\left|x\right|$, $\left|x\right|<\frac{1}{2}$]{\protect\includegraphics[width=0.45\textwidth]{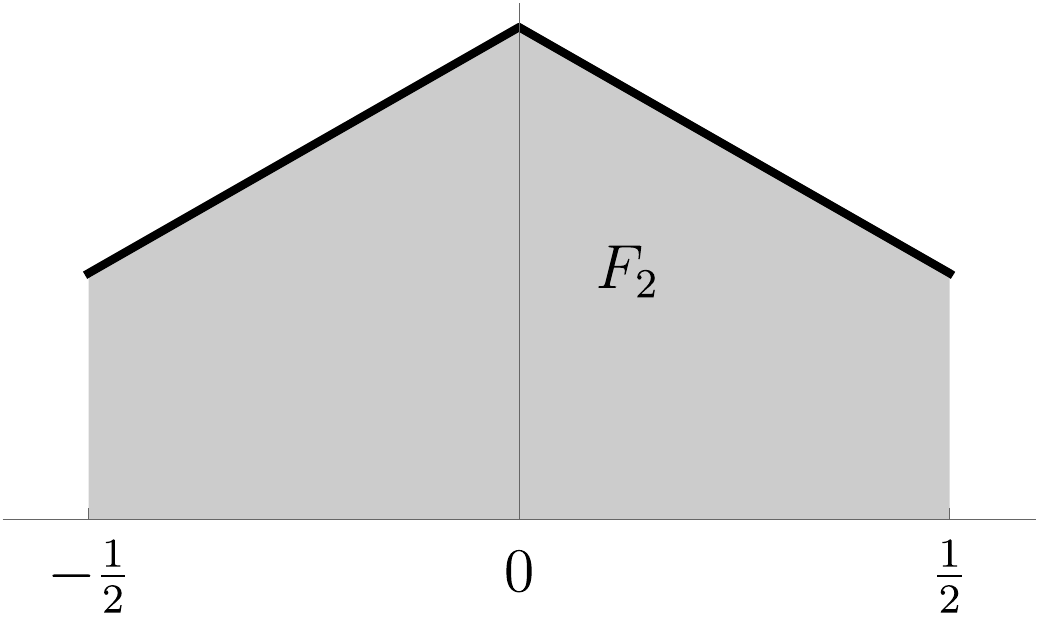}

}\hfill{}\subfloat[$F_{3}\left(x\right)=e^{-\left|x\right|}$, $\left|x\right|<1$]{\protect\includegraphics[width=0.45\textwidth]{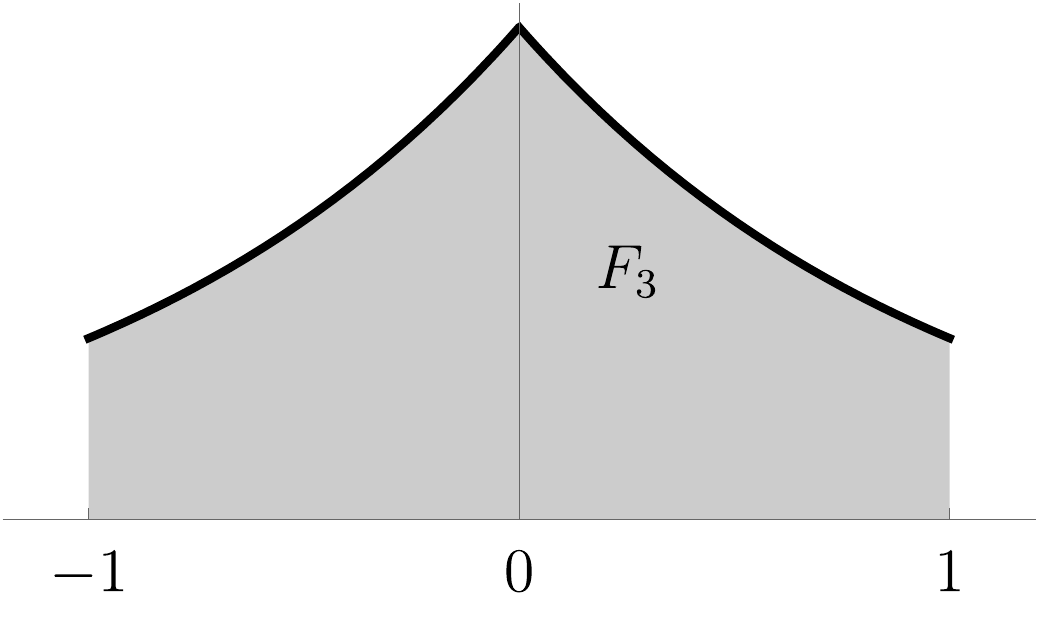}

}%
\end{minipage}

\protect\caption{\label{fig:FK-1}The examples of $F_{2}$ and $F_{3}$.}
\end{figure}

We wish to compare the respective skew-Hermitian operators, $D^{\left(F\right)}$
in $\mathscr{H}_{F}$; and $D^{\left(F\right)}$ in $\mathscr{H}_{K}$
(Section \ref{sub:euclid}), i.e., \index{operator!skew-Hermitian}\index{skew-Hermitian operator; also called skew-symmetric}
\begin{eqnarray}
D^{\left(F\right)}\left(F_{\varphi}\right) & = & F_{\varphi'},\;\forall\varphi\in C_{c}^{\infty}\left(0,b\right);\;\mbox{and}\label{eq:FK-app-3}\\
D^{\left(K\right)}\left(K_{\varphi}\right) & = & K_{\varphi'},\;\forall\varphi\in C_{c}^{\infty}\left(0,a\right).\label{eq:FK-app-4}
\end{eqnarray}

Let $z\in\mathbb{C}$; and we set 
\begin{eqnarray}
DEF_{F}\left(z\right) & = & \left\{ \xi\in dom\bigl(\bigl(D^{\left(F\right)}\bigr)^{*}\bigr)\:\Big|\:\bigl(D^{\left(F\right)}\bigr)^{*}\xi=z\xi\right\} ,\;\mbox{and}\label{eq:FK-app-5}\\
DEF_{K}\left(z\right) & = & \left\{ \xi\in dom\bigl(\bigl(D^{\left(K\right)}\bigr)^{*}\bigr)\:\Big|\:\bigl(D^{\left(K\right)}\bigr)^{*}\xi=z\xi\right\} .\label{eq:FK-app-6}
\end{eqnarray}

\begin{theorem}
\label{thm:FK-app-1}Let two continuous p.d. functions $F$ and $K$
be specified as above, and suppose 
\begin{equation}
K\ll F;\label{eq:FK-app-7}
\end{equation}
then 
\begin{equation}
DEF_{K}\left(z\right)=DEF_{F}\left(z\right)\Big|_{\left(0,a\right)}\label{eq:FK-app-8}
\end{equation}
i.e., restriction to the smaller interval.\end{theorem}
\begin{svmultproof2}
Since (\ref{eq:FK-app-7}) is assumed, it follows from Theorem \ref{thm:FK-1},
that $\mathscr{H}_{K}$ is a subspace of $\mathscr{H}_{F}$. 

If $\varphi\in C_{c}^{\infty}\left(0,b\right)$, and $\xi\in dom\bigl(\bigl(D^{\left(F\right)}\bigr)^{*}\bigr)$,
then 
\[
\left\langle \bigl(D^{\left(F\right)}\bigr)^{*}\xi,F_{\varphi}\right\rangle _{\mathscr{H}_{F}}=\left\langle \xi,F_{\varphi'}\right\rangle _{\mathscr{H}_{F}}=\int_{0}^{b}\overline{\xi\left(x\right)}\varphi'\left(x\right)dx;
\]
and it follows that functions $\xi$ in $DEF_{F}\left(z\right)$ must
be multiples of 
\begin{equation}
\left(0,b\right)\ni x\longmapsto e_{z}\left(x\right)=e^{-zx}.\label{eq:FK-app-9}
\end{equation}
Hence, by Theorem \ref{thm:FK-1}, we get 
\[
DEF_{K}\left(z\right)\subseteq DEF_{F}\left(z\right),
\]
and by (\ref{eq:FK-app-9}), we see that(\ref{eq:FK-app-8}) must
hold. 

Conversely, if $DEF_{F}\left(z\right)\neq0$, then $l\left(DEF_{F}\left(z\right)\right)\neq0$,
and its restriction to $\left(0,a\right)$ is contained in $DEF_{K}\left(z\right)$.
The conclusion in the theorem follows.\end{svmultproof2}

\begin{remark}
The spaces $DEF_{\left(F\right)}\left(z\right)$, $z\in\mathbb{C}$,
are also discussed in Theorem \ref{thm:Eigenspaces-for-the-adjoint}.\end{remark}
\begin{example}[Application]
 Consider the two functions $F_{2}$ and $F_{3}$ in Table \ref{tab:F1-F6}.
Both of the operators $D^{\left(F_{i}\right)}$, $i=2,3$, have deficiency
indices\index{deficiency indices} $\left(1,1\right)$. \end{example}
\begin{svmultproof2}
One easily checks that $F_{2}\ll F_{3}$. And it is also easy to check
directly that $D^{\left(F_{2}\right)}$ has indices $\left(1,1\right)$.
Hence, by (\ref{eq:FK-app-8}) in the theorem, it follows that $D^{\left(F_{3}\right)}$
also must have indices $\left(1,1\right)$. (The latter conclusion
is not as easy to verify by direct means!)
\end{svmultproof2}

\section{Radially Symmetric Positive Definite Functions}

Among other subclasses of positive definite\index{positive definite}
functions we have radially symmetric p.d. functions. If a given p.d.
function happens to be radially symmetric, then there are a number
of simplifications available, and the analysis in higher dimension
often simplifies. This is to a large extend due to theorems of I.
J. Schöenberg\index{Schöenberg, Isac} and D. V. Widder. Below we
sketch two highpoints, but we omit details and application to interpolation
and to geometry. These themes are in the literature, see e.g. \cite{Sch38,ScWh53,Sch64,Wid41,WeWi75}.
\begin{remark}
In some cases, the analysis in one dimension yields insight into the
possibilities in $\mathbb{R}^{n}$, $n>1$. This leads for example
for functions $F$ on $\mathbb{R}^{n}$ which are radial, i.e., of
the form $F\left(x\right)=\Phi\bigl(\left\Vert x\right\Vert ^{2}\bigr)$,
where $\left\Vert x\right\Vert ^{2}=\sum_{i=1}^{n}x_{i}^{2}$. 
\end{remark}
A function $q$ on $\mathbb{R}_{+}$, $q:\mathbb{R}_{+}\rightarrow\mathbb{R}$,
is said to be \emph{completely monotone} iff $q\in C\left(\left[0,\infty\right)\right)\cap C^{\infty}\left(\left(0,\infty\right)\right)$
and 
\begin{equation}
\left(-1\right)^{n}q^{\left(n\right)}\left(r\right)\geq0,\;r\in\mathbb{R}_{+},n\in\mathbb{N}_{0}.\label{eq:rad-1}
\end{equation}

\begin{example}
~
\begin{eqnarray*}
q\left(r\right) & = & e^{-\alpha r},\qquad\alpha\geq0;\\
q\left(r\right) & = & \frac{\alpha}{r^{1-\alpha}},\qquad\alpha\leq1;\\
q\left(r\right) & = & \frac{1}{\left(r+\alpha^{2}\right)^{\beta}},\;\alpha>0,\beta\geq0.
\end{eqnarray*}
\end{example}
\begin{theorem}[Schöenberg (1938) ]
 A function $q:\mathbb{R}_{+}\rightarrow\mathbb{R}$ is completely
monotone iff the corresponding function $F_{q}\left(x\right)=q\bigl(\left\Vert x\right\Vert ^{2}\bigr)$
is positive definite and radial on $\mathbb{R}^{n}$ for all $k\in\mathbb{N}$. \end{theorem}
\begin{svmultproof2}
See, e.g., \cite{BCR84,Sch38a,Sch38b}. We omit details, but the proof
uses:\end{svmultproof2}

\begin{theorem}[Bernstein-Widder ]
\label{thm:bernstein} A function $q:\mathbb{R}_{+}\rightarrow\mathbb{R}$
is completely monotone iff there is a finite positive Borel measure
\index{measure!Borel}on $\mathbb{R}_{+}$ s.t. 
\[
q\left(r\right)=\int_{0}^{\infty}e^{-rt}d\mu\left(t\right),\;r\in\mathbb{R}_{+},
\]
i.e., $q$ is the Laplace transform of a finite positive measure \index{measure!positive}$\mu$
on $\mathbb{R}_{+}$.\end{theorem}
\begin{svmultproof2}
See, e.g., \cite{BCR84,Wid41,Ber46,Ber29,Wid64,BW40}.
\end{svmultproof2}

\index{transform!Laplace-}
\begin{remark}
The condition that the function $q$ in (\ref{eq:rad-1}) be in $C^{\infty}\left(\mathbb{R}_{+}\right)$
may be relaxed; and then (\ref{eq:rad-1}) takes the following alternative
form:
\begin{equation}
\sum_{k=1}^{n}\left(-1\right)^{k}\binom{n}{k}q\left(r+k\delta\right)\geq0
\end{equation}
for all $n\in\mathbb{N}$, all $\delta>0$, and $x\in\bigl[0,\infty\bigr)$;
i.e., 
\begin{gather*}
q\left(r\right)-q\left(r+\delta\right)\geq0\\
q\left(r\right)-2q\left(r+\delta\right)+q\left(r+2\delta\right)\geq0\;\;\mbox{e.t.c. }
\end{gather*}

It is immediate that every completely monotone function $q$ on $\bigl[0,\infty\bigr)$
is convex.\index{convex}
\end{remark}

\begin{remark}
A related class of functions $q:\mathbb{R}_{+}\rightarrow\mathbb{R}$
is called $\Phi_{n}$; the functions $q$ such that $\mathbb{R}^{n}\ni x\longmapsto q\left(\left\Vert x\right\Vert \right)$
is positive definite on $\mathbb{R}^{n}$. Schönberg \cite{Sch38a}
showed that $q\in\Phi_{n}$ if and only if there is a finite positive
Borel measure $\mu$ on $\mathbb{R}_{+}$ such that 
\begin{equation}
q\left(r\right)=\int_{0}^{\infty}\Omega_{n}\left(rt\right)d\mu\left(t\right)\label{eq:pk1}
\end{equation}
where 
\[
\Omega_{n}\left(s\right)=\Gamma\left(\frac{n}{2}\right)\left(\frac{2}{s}\right)^{\frac{n-2}{2}}J_{\frac{n-2}{2}}\left(s\right)=\sum_{j=0}^{\infty}\left(-\frac{s^{2}}{4}\right)^{j}\frac{\Gamma\left(\frac{n}{2}\right)}{j!\Gamma\left(\frac{n}{2}+j\right)}.
\]
Indeed, if $\nu_{n}$ denotes the normalized uniform measure on $S:=\left\{ x\in\mathbb{R}^{n}\:|\:\left\Vert x\right\Vert =1\right\} $,
then 
\[
\Omega_{n}\left(\left\Vert x\right\Vert \right)=\int_{S}e^{ix\cdot y}d\nu_{n}\left(y\right),\quad\forall x\in\mathbb{R}^{n}.
\]

\end{remark}

\section{\label{sec:FFbar}Connecting $F$ and $\overline{F}$ When $F$ is
a Positive Definite Function}

Let $F:\left(-1,1\right)\rightarrow\mathbb{C}$ be continuous and
positive definite\index{positive definite}, and let $\overline{F}$
be the complex conjugate, i.e., $\overline{F}$$\left(x\right)=F\left(-x\right)$,
$\forall x\in\left(-1,1\right)$. Below, we construct a contractive-linking
operator $\mathscr{H}_{F}\rightarrow\mathscr{H}_{\overline{F}}$ between
the two RKHSs. \index{RKHS}
\begin{lemma}
\label{lem:conj-1-1}Let $\mu$ and $\mu^{\left(s\right)}$ be as
before, $\mu^{\left(s\right)}=\mu\circ s$, $s\left(x\right)=-x$;
and set 
\begin{equation}
g=\sqrt{\frac{d\mu^{\left(s\right)}}{d\mu}};\label{eq:conj-1-1}
\end{equation}
(the square root of the Radon-Nikodym\index{Radon-Nikodym} derivative)
then the $g$-multiplication operator is isometric between the respective
Hilbert spaces; $L^{2}\left(\mu^{\left(s\right)}\right)$ and $L^{2}\left(\mu\right)$
as follows: \index{operator!multiplication-}\index{derivative!Radon-Nikodym-}
\begin{equation}
\xymatrix{L^{2}\left(\mathbb{R},\mu^{\left(s\right)}\right)\ar@{|->}[rr]\sp(0.6){M_{g}}\sb(0.6){h\mapsto gh} &  & L^{2}\left(\mathbb{R},\mu\right)}
.\label{eq:conj-1-2}
\end{equation}
\end{lemma}
\begin{svmultproof2}
Let $h\in L^{2}\left(\mu^{\left(s\right)}\right)$, then 
\begin{eqnarray*}
\int_{\mathbb{R}}\left|gh\right|^{2}d\mu & = & \int_{\mathbb{R}}\left|h\right|^{2}\frac{d\mu^{\left(s\right)}}{d\mu}d\mu\\
 & = & \int_{\mathbb{R}}\left|h\right|^{2}d\mu^{\left(s\right)}=\left\Vert h\right\Vert _{L^{2}\left(\mu^{\left(s\right)}\right)}^{2}.
\end{eqnarray*}
\end{svmultproof2}

\begin{lemma}
If $F:\left(-1,1\right)\rightarrow\mathbb{C}$ is a given continuous
p.d. function, and if $\mu\in Ext\left(F\right)$, then 
\begin{equation}
\xymatrix{\mathscr{H}_{F}\ni F_{\varphi}\ar@{|->}[rr]\sp(0.5){V^{\left(F\right)}} &  & \widehat{\varphi}\in L^{2}\left(\mathbb{R},\mu\right)}
\label{eq:conj-1-3}
\end{equation}
extends by closure to an isometry\index{isometry}. \end{lemma}
\begin{svmultproof2}
For $\varphi\in C_{c}\left(0,1\right)$, we have: 
\begin{eqnarray*}
\left\Vert F_{\varphi}\right\Vert _{\mathscr{H}_{F}}^{2} & = & \int_{0}^{1}\int_{0}^{1}\overline{\varphi\left(x\right)}\varphi\left(y\right)F\left(x-y\right)dxdy\\
 & \underset{\text{\ensuremath{\mu\in}Ext(F)}}{=} & \int_{0}^{1}\int_{0}^{1}\overline{\varphi\left(x\right)}\varphi\left(y\right)\left(\int_{\mathbb{R}}e^{i\left(x-y\right)\lambda}d\mu\left(\lambda\right)\right)dxdy\\
 & \underset{\left(\text{Fubini}\right)}{=} & \int_{\mathbb{R}}\left|\widehat{\varphi}\left(\lambda\right)\right|^{2}d\mu\left(\lambda\right)\\
 & \underset{\text{(\ref{eq:conj-1-3})}}{=} & \left\Vert V^{\left(F\right)}\left(F_{\varphi}\right)\right\Vert _{L^{2}\left(\mathbb{R},\mu\right)}^{2}.
\end{eqnarray*}
\end{svmultproof2}

\begin{definition}
Set 
\begin{eqnarray}
\left(\varphi\ast g^{\vee}\right)\left(x\right) & := & \int_{0}^{1}\varphi\left(x\right)g^{\vee}\left(x-y\right)dy\label{eq:conj-1-4}\\
 & = & \left(g^{\vee}\ast\varphi\right)\left(x\right),\;x\in\left(0,1\right),\varphi\in C_{c}\left(0,1\right).\nonumber 
\end{eqnarray}
\end{definition}
\begin{theorem}
\label{thm:FKc}We have 
\begin{equation}
\left(V^{\left(F\right)*}M_{g}V^{\left(\overline{F}\right)}\right)\left(\overline{F}_{\varphi}\right)=T_{F}\left(g^{\vee}\ast\varphi\right),\;\forall\varphi\in C_{c}\left(0,1\right).\label{eq:conj-1-5}
\end{equation}
\end{theorem}
\begin{svmultproof2}
Let $\varphi\in C_{c}\left(0,1\right)$, we will then compute the
two sides in (\ref{eq:conj-1-5}), where $g^{\vee}:=$ inverse Fourier
transform: \index{transform!Fourier-}
\[
l.h.s.\left(\ref{eq:conj-1-5}\right)=V^{\left(F\right)*}\left(g\widehat{\varphi}\right);
\]
($\widehat{\varphi}\in L^{2}(\mu^{\left(s\right)})$, and using that
$g\widehat{\varphi}\in L^{2}\left(\mu\right)$ by \lemref{conj-1-1}
) we get: 
\begin{eqnarray*}
l.h.s.\left(\ref{eq:conj-1-5}\right) & = & \left(V^{\left(F\right)}\right)^{*}\underset{\in L^{2}\left(\mu\right)}{\left(\underbrace{\widehat{g^{\vee}\ast\varphi}}\right)}\\
 & \underset{\left(\ref{eq:conj-1-5}\right)}{=} & T_{F}\left(g^{\vee}\ast\varphi\right)
\end{eqnarray*}
where $T_{F}$ is the Mercer operator $T_{F}:L^{2}\left(\Omega\right)\rightarrow\mathscr{H}_{F}$
defined using 
\begin{eqnarray*}
T_{F}\left(\varphi\right)\left(x\right) & = & \int_{0}^{1}\varphi\left(x\right)F\left(x-y\right)dy\\
 & = & \chi_{\left[0,1\right]}\left(x\right)\left(\widehat{\varphi}d\mu\right)^{\vee}\left(x\right).
\end{eqnarray*}
\end{svmultproof2}

\begin{corollary}
Let $\mu$ and $\mu^{\left(s\right)}$ be as above, with $\mu\in Ext\left(F\right)$,
and $\mu^{\left(s\right)}\ll\mu$. Setting $g=\sqrt{\frac{d\mu^{\left(s\right)}}{d\mu}}$,
we get 
\begin{equation}
\Bigl\Vert\bigl(V^{\left(F\right)}\bigr)^{*}M_{g}V^{\left(\overline{F}\right)}\Bigr\Vert_{\mathscr{H}_{F}\rightarrow\mathscr{H}_{F}}\leq1.\label{eq:conj-2-1}
\end{equation}
\end{corollary}
\begin{svmultproof2}
For the three factors in the composite operator $\left(V^{\left(F\right)}\right)^{*}M_{b}V^{\left(\overline{F}\right)}$
in (\ref{eq:conj-2-1}), we have two isometries as follows:
\begin{eqnarray*}
V^{\left(\overline{F}\right)}:\mathscr{H}_{\overline{F}} & \rightarrow & L^{2}\bigl(\mu^{\left(s\right)}\bigr),\mbox{ and}\\
M_{g}:L^{2}\bigl(\mu^{\left(s\right)}\bigr) & \rightarrow & L^{2}\left(\mu\right),
\end{eqnarray*}
and both isometries; while
\[
\bigl(V^{\left(F\right)}\bigr)^{*}:L^{2}\left(\mu\right)\rightarrow\mathscr{H}_{F}
\]
is co-isometric, and therefore contractive, i.e., 
\begin{equation}
\bigl\Vert\bigl(V^{\left(F\right)}\bigr)^{*}\bigr\Vert_{L^{2}\left(\mu\right)\rightarrow\mathscr{H}_{F}}\leq1.\label{eq:conj-2-2}
\end{equation}
But then: 
\begin{eqnarray*}
\Bigl\Vert\bigl(V^{\left(F\right)}\bigr)^{*}M_{g}V^{\left(\overline{F}\right)}\Bigr\Vert_{\mathscr{H}_{F}\rightarrow\mathscr{H}_{F}} & \leq & \Bigl\Vert\bigl(V^{\left(F\right)}\bigr)^{*}\Bigr\Vert\Bigl\Vert M_{g}\Bigr\Vert\Bigl\Vert V^{\left(\overline{F}\right)}\Bigr\Vert\\
 & = & \Bigl\Vert\bigl(V^{\left(F\right)}\bigr)^{*}\Bigr\Vert\leq1,\:\mbox{by \ensuremath{\left(\ref{eq:conj-2-2}\right)}}.
\end{eqnarray*}

\end{svmultproof2}

\section{\label{sec:imgF}The Imaginary Part of a Positive Definite Function}
\begin{lemma}
\label{lem:Im-Lemma-1}Let $F:(-1,1)\to\mathbb{C}$ be a continuous
p.d. function. For $\phi$ in $C_{c}^{\infty}(0,1)$ let 
\[
\left(t(\phi)\right)(x)=\phi(1-x),\quad\text{for all }x\in(0,1).
\]
The operator $F_{\phi}\to F_{t(\phi)}$ is bounded in $\mathscr{H}_{F}$
iff 
\begin{equation}
\overline{F}\ll F\label{eq:im-1}
\end{equation}
where $\overline{F}$ is the complex conjugate of $F,$ and $\ll$
is the order on p.d. functions, i.e., there is an $A<\infty$ such
that 
\begin{equation}
\sum\sum\overline{c_{j}}c_{k}\overline{F}(x_{j}-x_{k})\leq A\sum\sum\overline{c_{j}}c_{k}F(x_{j}-x_{k}),\label{eq:im-2}
\end{equation}
for all finite systems $\left\{ c_{j}\right\} $ and $\left\{ x_{j}\right\} $,
where $c_{j}\in\mathbb{C}$, $x_{j}\in\left(0,1\right)$.\end{lemma}
\begin{svmultproof2}
It follows from (\ref{eq:im-2}) that $\overline{F}\ll F$ iff there
is an $A<\infty$ such that 
\begin{equation}
\left\Vert \overline{F_{\phi}}\right\Vert _{\mathscr{H}_{\overline{F}}}\leq\sqrt{A}\left\Vert F_{\phi}\right\Vert _{\mathscr{H}_{F}},\label{eq:im-3}
\end{equation}
for all $\phi$ in $C_{c}^{\infty}(0,1).$ Since 
\begin{align}
\left\Vert F_{t(\phi)}\right\Vert _{\mathscr{H}_{F}}^{2} & =\int_{0}^{1}\int_{0}^{1}\overline{\phi(1-x)}\phi(1-y)F(x-y)dxdy\nonumber \\
 & =\int_{0}^{1}\int_{0}^{1}\overline{\phi(x)}\phi(y)F(y-x)dxdy\label{eq:im-4}\\
 & =\left\Vert \overline{F_{\phi}}\right\Vert _{\mathscr{H}_{\overline{F}}}^{2}\nonumber 
\end{align}
we have established the claim. 
\end{svmultproof2}

Let $M=\left(M_{jk}\right)$ be an $N\times N$ matrix over $\mathbb{C}.$
Set 
\[
\Re\left\{ M\right\} =\left(\Re\left\{ M_{jk}\right\} \right),\;\Im\left\{ M\right\} =\left(\Im\left\{ M_{jk}\right\} \right).
\]
Assume $M^{*}=M,$ where $M^{*}$ is the conjugate transpose of $M,$
and $M\geq0$. Recall, $M\geq0$ iff all eigenvalues of $M$ are $\geq0$
iff all sub-determinants $\mathrm{det}M_{n}\geq0,$ $n=1,\ldots,N$,
where $M_{n}=\left(M_{jk}\right)_{j,k\leq n}.$ 
\begin{definition}
Let $s(x)=-x.$ For a measure $\mu$ on $\mathbb{R}$, let $\mu^{s}=\mu\circ s.$ \end{definition}
\begin{lemma}
\label{lem:Im-mu-s}If $F=\widehat{d\mu}$ then $\overline{F}=\widehat{d\mu^{s}}.$ \end{lemma}
\begin{svmultproof2}
Suppose $F=\widehat{d\mu}$, then the calculation 
\begin{align*}
\overline{F(x)} & =F(-x)=\int_{\mathbb{R}}e_{\lambda}(-x)d\mu(\lambda)\\
 & =\int_{\mathbb{R}}e_{-\lambda}(x)d\mu(\lambda)\\
 & =\int_{\mathbb{R}}e_{\lambda}(x)d\mu^{s}(\lambda)=\widehat{d\mu^{s}}(x)
\end{align*}
establishes the claim.\end{svmultproof2}

\begin{corollary}
\label{cor:Im-eq-3}If $F=\widehat{d\mu},$ then (\ref{eq:im-3})
takes the form 
\[
\int_{\mathbb{R}}\left|\widehat{\phi}(\lambda)\right|^{2}d\mu^{s}(\lambda)\leq A\int_{\mathbb{R}}\left|\widehat{\phi}(\lambda)\right|^{2}d\mu(\lambda),
\]
for all $\phi$ in $C_{c}^{\infty}(0,1).$ \end{corollary}
\begin{svmultproof2}
A calculation show that 
\[
\left\Vert F_{\phi}\right\Vert _{\mathscr{H}_{F}}^{2}=\int_{\mathbb{R}}\left|\widehat{\phi}(\lambda)\right|^{2}d\mu(\lambda)
\]
and similarly $\left\Vert \overline{F_{\phi}}\right\Vert _{\mathscr{H}_{\overline{F}}}^{2}=\int_{\mathbb{R}}\bigl|\widehat{\phi}(\lambda)\bigr|^{2}d\mu^{s}(\lambda),$
where we used Lemma \ref{lem:Im-mu-s}.\end{svmultproof2}

\begin{example}
\label{Example:Im-example-5}If $\mu=\tfrac{1}{2}\left(\delta_{-1}+\delta_{2}\right),$
then $\mu^{s}=\tfrac{1}{2}\left(\delta_{-1}+\delta_{2}\right).$ Set
\begin{align*}
F(x) & =\widehat{\mu}(x)=\tfrac{1}{2}\left(e^{-ix}+e^{i2x}\right),\text{ then}\\
\overline{F(x)} & =\widehat{\mu^{s}}(x)=\tfrac{1}{2}\left(e^{ix}+e^{-i2x}\right).
\end{align*}
It follows from Corollary \ref{cor:Im-eq-3} and Lemma \ref{lem:Im-Lemma-1}
that $\overline{F}\not\ll F$ and $F\not\ll\overline{F}$. In fact,
\begin{align*}
\left\Vert F_{\phi}\right\Vert _{\mathscr{H}_{F}}^{2} & =\tfrac{1}{2}\left(\left|\widehat{\phi}(-1)\right|^{2}+\left|\widehat{\phi}(2)\right|^{2}\right)\\
\left\Vert \overline{F_{\phi}}\right\Vert _{\mathscr{H}_{\overline{F}}}^{2} & =\tfrac{1}{2}\left(\left|\widehat{\phi}(1)\right|^{2}+\left|\widehat{\phi}(-2)\right|^{2}\right).
\end{align*}
Fix $f\in C_{c}^{\infty},$ such that $f(0)=1,$ $f\geq0,$ and $\int f=1.$
Considering $\phi_{n}(x)=\tfrac{1}{2}\left(e^{-ix}+e^{i2x}\right)f\left(\tfrac{x}{n}\right),$
and $\psi_{n}(x)=\tfrac{1}{2}\left(e^{-ix}+e^{i2x}\right)f\left(\tfrac{x}{n}\right),$
completes the verification, since $\widehat{\phi_{n}}\to\mu$ and
$\widehat{\psi_{n}}\to\mu^{s}.$ 

And similarly, $\overline{F}\not\ll F$ and $F\not\ll\overline{F}$,
where $F$ is as in Example \ref{Example:im-14}. \end{example}
\begin{remark}
In fact, $\overline{F}\ll F$ iff $\mu^{s}\ll\mu$ with Radon-Nikodym
derivative $\frac{d\mu^{s}}{d\mu}\in L^{\infty}(\mu).$ See, Section
\ref{sec:FFbar}. \index{Radon-Nikodym}\end{remark}
\begin{corollary}
\label{cor:im-mu-mu-s}If $F=\widehat{d\mu},$ then 
\begin{align*}
\Re\left\{ F\right\}  & =\tfrac{1}{2}\widehat{\left(\mu+\mu^{s}\right)},\text{ and }\\
\Im\left\{ F\right\}  & =\tfrac{1}{2i}\widehat{\left(\mu-\mu^{s}\right)}.
\end{align*}

\end{corollary}
We can rewrite the corollary in the form: If $F=\widehat{d\mu},$
then
\begin{align}
\Re\left\{ F\right\} (x) & =\int_{\mathbb{R}}\cos\left(\lambda x\right)\,d\mu(\lambda),\text{ and }\label{eq:im-a-1}\\
\Im\left\{ F\right\} (x) & =\int_{\mathbb{R}}\sin\left(\lambda x\right)\,d\mu(\lambda).\label{eq:im-a-2}
\end{align}

\begin{remark}
(\ref{eq:im-a-1}) simply states that if $F$ is positive definite,
so is its real part $\Re\left\{ F\right\} $. But (\ref{eq:im-a-2})
is deeper: If the function $\lambda$ is in $L^{1}(\mu),$ then 
\[
\frac{d}{dx}\Im\left\{ F\right\} (x)=\int_{\mathbb{R}}\cos(\lambda x)\lambda\,d\mu(\lambda)
\]
is the cosine transform of $\lambda d\mu(\lambda).$
\end{remark}
Suppose $F$ is p.d. on $(-a,a)$ and $\mu\in\mathrm{Ext}(F)$, i.e.,
$\mu$ is a finite positive measure satisfying 
\[
F(x)=\int_{\mathbb{R}}e_{\lambda}(x)d\mu(x).
\]
For a finite set $\{x_{j}\}\in(-a,a)$ let 
\[
M:=\left(F\left(x_{j}-x_{k}\right)\right).
\]
For $c_{j}$ in $\mathbb{C}$ consider 
\[
\overline{c^{T}}Mc=\sum\sum\overline{c_{j}}c_{k}M_{jk}.
\]
The for $\Re\left\{ F\right\} $, we have
\begin{align*}
 & \sum_{j}\sum_{k}\overline{c_{j}}c_{k}\Re\left\{ F\right\} (x_{j}-x_{k})\\
= & \sum_{j}\sum_{k}\overline{c_{j}}c_{k}\int_{\mathbb{R}}\left(\cos\left(\lambda x_{j}\right)\cos\left(\lambda x_{k}\right)+\sin\left(\lambda x_{j}\right)\sin\left(\lambda x_{k}\right)\right)d\mu(\lambda)\\
= & \int_{\mathbb{R}}\left(\left|C(\lambda)\right|^{2}+\left|S(\lambda)\right|^{2}\right)d\mu(\lambda)\geq0,
\end{align*}
where 
\begin{align*}
C(\lambda) & =C\left(\lambda,\left(x_{j}\right)\right)=\sum_{j}c_{j}\cos\left(\lambda x_{j}\right)\\
S(\lambda) & =S\left(\lambda,\left(x_{j}\right)\right)=\sum_{j}c_{j}\sin\left(\lambda x_{j}\right)
\end{align*}
for all $\lambda\in\mathbb{R}.$ Similarly, for $\Im\left\{ F\right\} $,
we have 
\begin{align*}
 & \sum_{j}\sum_{k}\overline{c_{j}}c_{k}\Im\left\{ F\right\} (x_{j}-x_{k})\\
= & \sum_{j}\sum_{k}\overline{c_{j}}c_{k}\int_{\mathbb{R}}\left(\sin\left(\lambda x_{j}\right)\cos\left(\lambda x_{k}\right)-\cos\left(\lambda x_{j}\right)\sin\left(\lambda x_{k}\right)\right)d\mu(\lambda)\\
= & \int_{\mathbb{R}}\left(\overline{S(\lambda)}C(\lambda)-\overline{C(\lambda)}S(\lambda)\right)d\mu(\lambda)\\
= & 2i\int_{\mathbb{R}}\Im\left\{ \overline{S\left(\lambda\right)}C\left(\lambda\right)\right\} d\mu(\lambda).
\end{align*}
If $\left\{ c_{j}\right\} \subset\mathbb{R},$ then $S(\lambda),C(\lambda)$
are real-valued and 
\[
\sum_{j}\sum_{k}\overline{c_{j}}c_{k}\Im\left\{ F\right\} (x_{j}-x_{k})=0.
\]

\subsection{Connections to, and applications of, Bochner\textquoteright s Theorem}

\index{Bochner, S.}\index{Theorem!Bochner's-}\index{Bochner's Theorem}

In this section, we study complex valued positive definite functions
$F$, locally defined, so on a fixed finite interval $(-a,a)$. The
purpose of the present section is to show how the real and the imaginary
parts of $F$ are related, when studied as individual functions on
$(-a,a)$.
\begin{lemma}
Let $F$ be a continuous positive definite function on some open interval
$(-a,a)$. Let $K$ be the real part $\Re\left\{ F\right\} $ of $F$
and let $L$ be the imaginary part $\Im\left\{ F\right\} $ of $F,$
hence $K$ and $L$ are real-valued, $K$ is a continuous positive
definite real-valued function, in particular $K$ is an even function,
and $L$ is an odd function.\end{lemma}
\begin{svmultproof2}
The even/odd claims follow from $F(-x)=\overline{F(x)}$ for $x\in(-a,a).$
For a finite set of points $\left\{ x_{j}\right\} _{j=1}^{N}$ in
$(-a,a)$ form the matrices 
\[
M_{F}=\left(F(x_{j}-x_{k})\right)_{j,k=1}^{N},M_{K}=\left(K(x_{j}-x_{k})\right)_{j,k=1}^{N}M_{L}=\left(L(x_{j}-x_{k})\right)_{j,k=1}^{N}.
\]
Let $c=(c_{j})$ be a vector in $\mathbb{R}^{N}.$ Since $L$ is an
odd function it follows that $c^{T}M_{L}c=0,$ consequently, 
\begin{equation}
c^{T}M_{K}c=c^{T}M_{F}c\geq0.\label{eq:im-MMM}
\end{equation}
It follows that $K$ is positive definite over the real numbers and
therefore also over the complex numbers \cite{Aro50}.\end{svmultproof2}

\begin{definition}
We say a signed measure $\mu$ is \emph{even,} if $\mu(B)=\mu(-B)$
for all $\mu-$measurable sets $B,$ where $-B=\left\{ -x:x\in B\right\} .$
Similarly, we say $\mu$ is \emph{odd,} if $\mu(B)=-\mu(-B)$ for
all $\mu$--measurable sets $B.$ \end{definition}
\begin{remark}
\label{remark:im-1}Let 
\begin{align*}
\mu_{K}(B) & :=\frac{\mu(B)+\mu(-B)}{2}\;\mbox{ and}\\
\mu_{L}(B) & :=\frac{\mu(B)-\mu(-B)}{2}
\end{align*}
for all $\mu-$measurable sets $B.$ Then $\mu_{K}$ is an even probability
measure and $\mu_{L}$ is an odd real-valued measure. If $F=\widehat{d\mu},$
$K=\widehat{d\mu_{K}},$ and $iL=\widehat{d\mu_{L}},$ then $K$ and
$L$ are real-valued continuous functions, $F$ and $K$ are continuous
positive definite functions, $L$ is a real-valued continuous odd
function and $F=K+iL.$\end{remark}
\begin{lemma}
\label{lem:im-bo}Suppose $K$ as the Fourier transform of some even
probability measure $\mu_{K}$ and $iL$ as the Fourier transform
of some odd measure $\mu_{L},$ then $F:=K+iL$ is positive definite
iff $\mu:=\mu_{K}+\mu_{L}$ is a probability measure, i.e., iff $\mu(B)\geq0$
for all Borel set $B.$ \index{transform!Fourier-}\end{lemma}
\begin{svmultproof2}
This is a direct consequence of Bochner's theorem. See e.g., \cite{BC48,BC49,Boc46,Boc47}.\index{Theorem!Bochner's-}\index{Bochner's Theorem}\end{svmultproof2}

\begin{corollary}
If $F$ is positive definite, and $\Im\left\{ F\right\} \neq0,$ then 

(i) $F_{m}:=\Re\left\{ F\right\} +im\Im\left\{ F\right\} $, is positive
definite for all $-1\leq m\leq1$ and 

(ii) $F_{m}$ is not positive definite for sufficiently large $m.$ \end{corollary}
\begin{svmultproof2}
(\emph{i}) We will use the notation from Remark \ref{remark:im-1}.
If $0<m$ and $\mu_{L}(B)<0,$ then
\[
\mu_{m}(B):=\mu_{K}(B)+m\,\mu_{L}(B)\geq\mu(B)\geq0.
\]
The cases where $m<0$ are handled by using that $\overline{F}$ is
positive definite. 

(\emph{ii}) Is established using a similar argument. \end{svmultproof2}

\begin{corollary}
Let $K$ and $L$ be real-valued continuous functions on $\mathbb{R}$.
Suppose $K$ positive definite and $L$ odd, and let $\mu_{K}$ and
$\mu_{L}$ be the correspond even and odd measures. If $K+im\,L$
is positive definite for some real $m\neq0,$ then the support of
$\mu_{L}$ is a subset of the support of $\mu_{K}.$ \end{corollary}
\begin{svmultproof2}
Fix $m\neq0.$ If the support of $\mu_{L}$ is not contained in the
support of $\mu_{K},$ then $\mu_{K}(B)+m\,\mu_{L}(B)<0$ for some
$B.$\end{svmultproof2}

\begin{remark}
The converse fails, support containment does not imply $\mu_{k}+m\mu_{L}$
is positive for some $m>0$ since $\mu_{K}$ can ``decrease'' much
faster than $\mu_{L}.$\end{remark}
\begin{example}
\label{Example:im-14}Let $d\mu(\lambda):=\delta_{-1}+\chi_{\mathbb{R}^{+}}(\lambda)e^{-\lambda}d\lambda$
and set $F:=\restr{\widehat{d\mu}}{(-1,1)}.$ Then
\begin{align*}
\Re\left\{ F\right\} (x) & =\cos(x)+\frac{1}{1+x^{2}}\\
\Im\left\{ F\right\} (x) & =-\sin(x)+\frac{x}{1+x^{2}};
\end{align*}
and $D^{(F)}$ in $\mathscr{H}_{F}$ has deficiency indices $(0,0).$
\index{deficiency indices}\end{example}
\begin{svmultproof2}
By construction:
\begin{align}
F(x) & =\int_{\mathbb{R}}e^{i\lambda x}d\mu(\lambda)=e^{-ix}+\int_{0}^{\infty}e^{i\lambda x}-\lambda d\lambda\nonumber \\
 & =e^{-ix}+\frac{1}{1-ix};\label{eq:im-10}
\end{align}
establishing the first claim. 

Consider $u=T_{F}\phi,$ for some $\phi\in C_{c}^{\infty}(0,1).$
By (\ref{eq:im-10})
\begin{align*}
u(x) & =\int_{0}^{1}\phi(y)F(x-y)dy\\
 & =\widehat{\phi}(-1)e^{-ix}+\int_{0}^{1}\phi(y)\frac{1}{1-i(x-y)}dy.
\end{align*}
Taking two derivatives we get 
\[
u''(x)=-\widehat{\phi}(-1)e^{-ix}+\int_{0}^{1}\phi(y)\frac{-2}{\left(1-i(x-y)\right)^{3}}dy.
\]
It follows that $u''+u\to0$ as $x\to\pm\infty,$ i.e., 
\begin{equation}
\lim_{|x|\to\infty}\left|u''(x)+u(x)\right|=0.\label{eq:im-9}
\end{equation}
A standard approximation argument shows that (\ref{eq:im-9}) holds
for all $u\in\mathscr{H}_{F}.$ 

Equation (\ref{eq:im-9}) rules out that either of $e^{\pm x}$ is
on $\mathscr{H}_{F}$, hence the deficiency indices are $(0,0)$ as
claimed. 
\end{svmultproof2}

If $\mu_{K}$ is an even probability measure and $f\left(x\right)$
is an odd function, s.t. $-1\leq f(x)\leq1$, then $d\mu(x):=\left(1+f(x)\right)d\mu_{K}(x)$
is a probability measure. Conversely, we have
\begin{lemma}
\label{lem:RN+Hahn}Let $\mu$ be a probability measure on the Borel
sets. There is an even probability measure $\mu_{K}$ and an odd real-valued
$\mu-$measurable function $f$ with $\left|f\right|\leq1,$ such
that $d\mu(\lambda)=\left(1+f(\lambda)\right)d\mu_{K}(\lambda).$ \end{lemma}
\begin{svmultproof2}
Let $\mu_{K}:=\tfrac{1}{2}\left(\mu+\mu^{s}\right)$ and $\mu_{L}:=\tfrac{1}{2}\left(\mu-\mu^{s}\right).$
Clearly, $\mu_{K}$ is an even probability measure and $\mu_{L}$
is an odd real-valued measure. Since $\mu_{K}(B)+\mu_{L}(B)=\mu(B)\geq0,$
it follows that $\mu_{K}(B)\geq\mu_{L}(B)$ for all Borel sets $B.$ 

Applying the Hahn decomposition theorem to $\mu_{L}$ we get sets
$P$ and $N$ such that $P\cap N=\emptyset,$ $\mu_{L}(B\cap P)\geq0$
and $\mu_{L}(B\cap N)\leq0$ for all $B.$ Let \index{Theorem!Hahn decomposition-}
\begin{align*}
P' & :=\left\{ x\in P:-x\in N\right\} \\
N' & :=\left\{ x\in N:-x\in P\right\} \\
O' & :=\left(P\setminus P'\right)\cup\left(N\setminus N'\right),
\end{align*}
then $N'=-P'$ and $\mu_{L}(B\cap O')=0$ for all $B.$ Write 
\[
\mu_{L}\left(B\right)=\mu_{L}\left(B\cap P'\right)+\mu_{L}\left(B\cap N'\right).
\]
Then $\mu_{K}\left(B\right)\geq\mu_{K}\left(B\cap P'\right)\geq\mu_{L}\left(B\cap P'\right)$
and
\begin{align*}
0\leq-\mu_{L}\left(B\cap N'\right) & =\mu_{L}\left(-\left(B\cap N'\right)\right)\\
 & =\mu_{L}\left(-B\cap P'\right)\\
 & \leq\mu_{K}\left(-B\cap P'\right)\leq\mu_{K}(B).
\end{align*}
Hence, $\mu_{L}$ is absolutely continuous with respect to $\mu_{K}.$
Setting $f:=\frac{d\mu_{L}}{d\mu_{K}}$, the Radon-Nikodym derivative
of $\mu_{L}$ with respect to $\mu_{K},$ completes the proof. \index{absolutely continuous}\index{derivative!Radon-Nikodym-}\end{svmultproof2}

\begin{corollary}
Let $F=\widehat{\mu}$ be a positive definite function with $F(0)=1.$
Let $\mu_{K}:=\tfrac{1}{2}\left(\mu+\mu^{s}\right)$ then $\Re\left\{ F\right\} (x)=\widehat{\mu_{K}}(x)$
and there is an odd function 
\[
-1\leq f(\lambda)\leq1,
\]
 such that $\Im\left\{ F\right\} (x)=\widehat{f\mu_{K}}(x)$.
\begin{corollary}
\label{cor:RI}Let $F$ be a continuous p.d. function on $(-a,a).$
Let $\Re\left\{ F\right\} $ be the real part of $F.$ Then $\mathscr{H}_{F}$
is a subset of $\mathscr{H}_{\Re\left\{ F\right\} }.$ In particular,
if $D^{\left(\Re\left\{ F\right\} \right)}$ has deficiency indices
$(1,1)$ so does $D^{\left(F\right)}.$ \index{deficiency indices}
\end{corollary}
\end{corollary}
\begin{svmultproof2}
Recall, a continuous function $\xi$ is in $\mathscr{H}_{F}$ iff
\[
\left|\int_{0}^{1}\psi(y)\xi(y)dy\right|^{2}\leq A\int_{0}^{1}\int_{0}^{1}\overline{\phi(x)}\phi(y)F(x-y)dxdy.
\]
Since, 
\begin{align*}
\int_{0}^{1}\int_{0}^{1}\overline{\phi(x)}\phi(y)F(x-y)dxdy & =\int_{0}^{1}\int_{0}^{1}\int_{\mathbb{R}}\overline{\phi(x)}\phi(y)e^{-i\lambda(x-y)}d\mu(\lambda)dxdy\\
 & =\int_{\mathbb{R}}\left|\phi(\lambda)\right|^{2}d\mu(\lambda)\\
 & \leq2\int_{\mathbb{R}}\left|\phi(\lambda)\right|^{2}d\mu_{K}(\lambda)\\
 & =2\int_{0}^{1}\int_{0}^{1}\overline{\phi(x)}\phi(y)K(x-y)dxdy
\end{align*}
it follows that $\mathscr{H}_{F}$ is contained in $\mathscr{H}_{\Re\left\{ F\right\} }$. 
\end{svmultproof2}

\chapter{\label{chap:conv}Convolution Products}

A source of interesting measures in probability are constructed as
product measures or convolutions; and this includes infinite operations;
see for example \cite{IM65,Jor07,KS02,Par09}.

Below we study these operations in the contest of our positive definite
functions, defined on subsets of groups. For example, most realizations
of fractal measures arise as infinite convolutions, see e.g., \cite{DJ10,JP12,JKS12,DJ12,JKS11,JKS08}.
Motivated by these applications, we show below that, given a system
of continuous positive definite functions $F_{1},F_{2},\ldots$, defined
on an open subset of a group, we can form well defined products, including
infinite products, which are again continuous positive definite. We
further show that if finite positive measures $\mu_{i}$, $i=1,2,\ldots$,
are given, $\mu_{i}\in Ext\left(F_{i}\right)$ then the convolution
of the measures $\mu_{i}$ is in $Ext\left(F\right)$ where $F$ is
the product of the p.d. functions $F_{i}$. This will be applied later
in the note.

\index{measure!product}

\index{measure!convolution}

\index{measure!fractal}

\index{group!locally compact Abelian}

\index{group!circle}

\index{convolution product}
\begin{definition}
Let $F$ be a continuous positive definite function defined on a subset
in $G$ (a locally compact Abelian group). Set 
\begin{equation}
Ext\left(F\right)=\left\{ \mu\in\mathscr{M}(\widehat{G})\:\big|\:\widehat{d\mu}\mbox{ is an extension of \ensuremath{F}}\right\} .\label{eq:conv-1}
\end{equation}

\end{definition}
In order to study the set $Ext\left(F\right)$ from above, it helps
to develop tools. One such tool is convolution, which we outline below.
It is also helpful in connection with the study of symmetric spaces,
such as the case $G=\mathbb{T}=\mathbb{R}/\mathbb{Z}$ (the circle
group), versus extensions to the group $\mathbb{R}$. 

Let $G$ be a locally compact group, and let $\Omega$ be a non-empty,
connected and open subset in $G$. Now consider systems of p.d. and
continuous functions on the set $\Omega^{-1}\Omega$. Specifically,
let $F_{i}$ be two or more p.d. continuous functions on $\Omega^{-1}\Omega$;
possibly an infinite family, so $F_{1},F_{2},\ldots$, all defined
on $\Omega^{-1}\Omega$. As usual, we normalize our p.d. functions
$F_{i}\left(e\right)=1$, where $e$ is the unit element in $G$.
\begin{lemma}
Form the point-wise product $F$ of any system of p.d. functions $F_{i}$
on $\Omega^{-1}\Omega$; then $F$ is again p.d. and continuous on
the set $\Omega^{-1}\Omega$.\end{lemma}
\begin{svmultproof2}
This is an application of a standard lemma about p.d. kernels, see
e.g., \cite{BCR84}. From this, we conclude that $F$ is again a continuous
and positive definite function on $\Omega^{-1}\Omega$.
\end{svmultproof2}

If we further assume that $G$ is also Abelian, and so $G$ is locally
compact Abelian, then the spectral theory takes a more explicit form.
\begin{lemma}
Assume $Ext\left(F_{i}\right)$ for $i=1,2,\ldots$ are non-empty.
For any system of individual measures $\mu_{i}\in Ext\left(F_{i}\right)$
we get that the resulting convolution-product measure $\mu$ formed
from the factors $\mu_{i}$ by the convolution in $G$, is in $Ext\left(F\right)$. 
\end{lemma}
\index{convolution product}
\begin{svmultproof2}
This is an application of our results in Sections \ref{sub:lcg}-\ref{sub:euclid}. \end{svmultproof2}

\begin{remark}
In some applications the convolution $\mu_{1}\ast\mu_{2}$ makes sense
even if only one of the measures is finite.
\end{remark}
\begin{flushleft}
\textbf{Application. }The case $G=\mathbb{R}$. Let $\mu_{1}$ be
the Dirac-comb (\cite{Ch03,Cor89}) 
\[
d\mu_{1}:=\sum_{n\in\mathbb{Z}}\delta\left(\lambda-n\right),\;\lambda\in\mathbb{R};
\]
let $\Phi\geq0$, $\Phi\in L^{1}\left(\mathbb{R}\right)$, and assume
$\int_{\mathbb{R}}\Phi\left(\lambda\right)d\lambda=1$. Set $d\mu_{2}=\Phi\left(\lambda\right)d\lambda$,
where $d\lambda=$ Lebesgue measure on $\mathbb{R}$; then $\mu_{1}\ast\mu_{2}$
yields the following probability measure on $\mathbb{T=\mathbb{R}}/\mathbb{Z}$:
Set 
\[
\Phi_{per}\left(\lambda\right)=\sum_{n\in\mathbb{Z}}\Phi\left(\lambda-n\right);
\]
then $\Phi_{per}\left(\lambda\right)\in L^{1}\left(\mathbb{T},dt\right)$,
where $dt=$ Lebesgue measure on $\mathbb{T}$, i.e., if $f\left(\lambda+n\right)=f\left(n\right)$,
$\forall n\in\mathbb{Z}$, $\forall\lambda\in\mathbb{R}$, then $f$
defines a function on $\mathbb{T}$, and $\int_{\mathbb{T}}f\,dt=\int_{0}^{1}f\left(t\right)\,dt$.
We get 
\[
d\left(\mu_{1}\ast\mu_{2}\right)=\Phi_{per}\left(\cdot\right)dt\;\mbox{on }\mathbb{T}.
\]

\par\end{flushleft}
\begin{svmultproof2}
We have
\begin{align*}
1 & =\int_{-\infty}^{\infty}\Phi\left(\lambda\right)d\lambda=\sum_{n\in\mathbb{Z}}\int_{n}^{n+1}\Phi\left(\lambda\right)d\lambda\\
 & =\int_{0}^{1}\sum_{n\in\mathbb{Z}}\Phi\left(\lambda-n\right)d\lambda\\
 & =\int_{\mathbb{T}}\Phi_{per}\left(t\right)dt.
\end{align*}

\end{svmultproof2}

We now proceed to study the relations between the other items in our
analysis, the RKHSs $\mathscr{H}_{F_{i}}$, for $i=1,2,\ldots$; and
computing $\mathscr{H}_{F}$ from the RKHSs $\mathscr{H}_{F_{i}}$.

We further study the associated unitary representations of $G$ when
$Ext\left(F_{i}\right)$, $i=1,2,\ldots$ are non-empty?

As an application, we get infinite convolutions, and they are fascinating;
include many fractal measures of course.

In the case of $\mathbb{G}=\mathbb{R}$ we will study the connection
between deficiency index values in $\mathscr{H}_{F}$ as compared
to those of the factor RKHSs $F_{i}$.

\chapter{\label{chap:spbd}Models for, and Spectral Representations of, Operator
Extensions}

A special case of our extension question for continuous positive definite
(p.d.) functions on a fixed finite interval $\left|x\right|<a$ in
$\mathbb{R}$ is the following: It offers a spectral model representation
for ALL Hermitian operators with dense domain in Hilbert space and
with deficiency indices $\left(1,1\right)$. (See e.g., \cite{vN32a,Kre46,DS88b,AG93,Ne69}.)

Specifically, on $\mathbb{R}$, all the partially defined continuous
p.d. functions extend, and we can make a translation of our p.d. problem
into the problem of finding all $\left(1,1\right)$ restrictions selfadjoint
operators.

By the Spectral theorem, every selfadjoint operator with simple spectrum
has a representation as a multiplication operator $M_{\lambda}$ in
some $L^{2}\left(\mathbb{R},\mu\right)$ for some probability measure
$\mu$ on $\mathbb{R}$. So this accounts for all Hermitian restrictions
operators with deficiency indices $\left(1,1\right)$.\index{operator!multiplication-}

So the problem we have been studying for just the case of $G=\mathbb{R}$
is the case of finding spectral representations for ALL Hermitian
operators with dense domain in Hilbert space having deficiency indices
$\left(1,1\right)$.

\index{deficiency indices}

\index{spectrum}

\index{measure!probability}

\index{representation!spectral-}

\index{Theorem!Spectral-}\index{Hilbert space}\index{Hermitian}

\section{\label{sec:R^1}Model for Restrictions of Continuous p.d. Functions
on $\mathbb{R}$}

Let $\mathscr{H}$ be a Hilbert space, $A$ a skew-adjoint operator,
$A^{*}=-A$, which is \uline{unbounded}; let $v_{0}\in\mathscr{H}$
satisfying $\left\Vert v_{0}\right\Vert _{\mathscr{H}}=1$. Then we
get an associated p.d. continuous function $F_{A}$ defined on $\mathbb{R}$
as follows:
\begin{equation}
F_{A}\left(t\right):=\left\langle v_{0},e^{tA}v_{0}\right\rangle =\left\langle v_{0},U_{A}\left(t\right)v_{0}\right\rangle ,\:t\in\mathbb{R},\label{eq:tmp11}
\end{equation}
where $U_{A}\left(t\right)=e^{tA}$ is a unitary representation of
$\mathbb{R}$. Note that we define $U\left(t\right)=U_{A}\left(t\right)=e^{tA}$
by the Spectral Theorem. Note (\ref{eq:tmp11}) holds for all $t\in\mathbb{R}$.

\index{unitary representation}\index{representation!unitary-}

\index{measure!PVM}

\index{measure!probability}

\index{projection-valued measure (PVM)}

Let $P_{U}\left(\cdot\right)$ be the projection-valued measure (PVM)
of $A$, then 
\begin{equation}
U\left(t\right)=\int_{-\infty}^{\infty}e^{i\lambda t}P_{U}\left(d\lambda\right),\:\forall t\in\mathbb{R}.\label{eq:tmp12}
\end{equation}

\begin{lemma}
Setting
\begin{equation}
d\mu=\left\Vert P_{U}\left(d\lambda\right)v_{0}\right\Vert ^{2}\label{eq:tmp18}
\end{equation}
we then get 
\begin{equation}
F_{A}\left(t\right)=\widehat{d\mu}\left(t\right),\:\forall t\in\mathbb{R}\label{eq:tmp14}
\end{equation}
Moreover, every probability measure $\mu$ on $\mathbb{R}$ arises
this way.\end{lemma}
\begin{svmultproof2}
By (\ref{eq:tmp11}), 
\begin{eqnarray*}
F_{A}\left(t\right) & = & \int e^{it\lambda}\left\langle v_{0},P_{U}\left(d\lambda\right)v_{0}\right\rangle \\
 & = & \int e^{it\lambda}\left\Vert P_{U}\left(d\lambda\right)v_{0}\right\Vert ^{2}\\
 & = & \int e^{it\lambda}d\mu\left(\lambda\right)
\end{eqnarray*}
\end{svmultproof2}

\begin{lemma}
\label{lem:iso}For Borel functions $f$ on $\mathbb{R}$, let 
\begin{equation}
f\left(A\right)=\int_{\mathbb{R}}f\left(\lambda\right)P_{U}\left(d\lambda\right)\label{eq:tmp15}
\end{equation}
be given by functional calculus. We note that 
\begin{equation}
v_{0}\in dom\left(f\left(A\right)\right)\Longleftrightarrow f\in L^{2}\left(\mu\right)\label{eq:tmp16}
\end{equation}
where $\mu$ is the measure in (\ref{eq:tmp18}). Then
\begin{equation}
\left\Vert f\left(A\right)v_{0}\right\Vert ^{2}=\int_{\mathbb{R}}\left|f\right|^{2}d\mu.\label{eq:tmp17}
\end{equation}
\end{lemma}
\begin{svmultproof2}
The lemma follows from 
\begin{eqnarray*}
\left\Vert f\left(A\right)v_{0}\right\Vert ^{2} & = & \left\Vert \int_{\mathbb{R}}f\left(\lambda\right)P_{U}\left(d\lambda\right)v_{0}\right\Vert ^{2}\,\,\,\,\,\,\,\,\,\,\,\,\,\,\,(\mbox{by }(\ref{eq:tmp15}))\\
 & = & \int\left|f\left(\lambda\right)\right|^{2}\left\Vert P_{U}\left(d\lambda\right)v_{0}\right\Vert ^{2}\\
 & = & \int\left|f\left(\lambda\right)\right|^{2}d\mu\left(\lambda\right).\,\,\,\,\,\,\,\,\,\,\,\,\,\,\,\,\,\,\,\,\,\,\,\,\,\,\,\,\,(\mbox{by }(\ref{eq:tmp18}))
\end{eqnarray*}

\end{svmultproof2}

Now we consider restriction of $F_{A}$ to, say $\left(-1,1\right)$,
i.e., 
\begin{equation}
F\left(\cdot\right)=F_{A}\Big|_{\left(-1,1\right)}\left(\cdot\right)\label{eq:tmp19}
\end{equation}

\begin{lemma}
Let $\mathscr{H}_{F}$ be the RKHS computed for $F$ in (\ref{eq:tmp14});
and for $\varphi\in C_{c}\left(0,1\right)$, set $F_{\varphi}=$ the
generating vectors in $\mathscr{H}_{F}$, as usual. Set 
\begin{equation}
U\left(\varphi\right):=\int_{0}^{1}\varphi\left(y\right)U\left(-y\right)dy\label{eq:tmp20}
\end{equation}
where $dy=$ Lebesgue measure on $\left(0,1\right)$; then
\begin{equation}
F_{\varphi}\left(x\right)=\left\langle v_{0},U\left(x\right)U\left(\varphi\right)v_{0}\right\rangle ,\:\forall x\in\left(0,1\right).
\end{equation}
\end{lemma}
\begin{svmultproof2}
We have
\begin{eqnarray*}
F_{\varphi}\left(x\right) & = & \int_{0}^{1}\varphi\left(y\right)F\left(x-y\right)dy\\
 & = & \int_{0}^{1}\varphi\left(y\right)\left\langle v_{0},U_{A}\left(x-y\right)v_{0}\right\rangle dy\,\,\,\,\,\,\,\,\,\,(\mbox{by }(\ref{eq:tmp11}))\\
 & = & \left\langle v_{0},U_{A}\left(x\right)\int_{0}^{1}\varphi\left(y\right)U_{A}\left(-y\right)v_{0}dy\right\rangle \\
 & = & \left\langle v_{0},U_{A}\left(x\right)U\left(\varphi\right)v_{0}\right\rangle \,\,\,\,\,\,\,\,\,\,(\mbox{by }(\ref{eq:tmp20}))\\
 & = & \left\langle v_{0},U\left(\varphi\right)U_{A}\left(x\right)v_{0}\right\rangle 
\end{eqnarray*}
for all $x\in\left(0,1\right)$, and all $\varphi\in C_{c}\left(0,1\right)$.\end{svmultproof2}

\begin{corollary}
Let $A$, $U\left(t\right)=e^{tA}$, $v_{0}\in\mathscr{H}$, $\varphi\in C_{c}\left(0,1\right)$,
and $F$ p.d. on $\left(0,1\right)$ be as above; let $\mathscr{H}_{F}$
be the RKHS of $F$; then, for the inner product in $\mathscr{H}_{F}$,
we have 
\begin{equation}
\left\langle F_{\varphi},F_{\psi}\right\rangle _{\mathscr{H}_{F}}=\left\langle U\left(\varphi\right)v_{0},U\left(\psi\right)v_{0}\right\rangle _{\mathscr{H}},\:\forall\varphi,\psi\in C_{c}\left(0,1\right).\label{eq:tmp-01}
\end{equation}
\end{corollary}
\begin{svmultproof2}
Note that 
\begin{eqnarray*}
\left\langle F_{\varphi},F_{\psi}\right\rangle _{\mathscr{H}_{F}} & = & \int_{0}^{1}\int_{0}^{1}\overline{\varphi\left(x\right)}\psi\left(y\right)F\left(x-y\right)dxdy\\
 & = & \int_{0}^{1}\int_{0}^{1}\overline{\varphi\left(x\right)}\psi\left(y\right)\left\langle v_{0},U_{A}\left(x-y\right)v_{0}\right\rangle _{\mathscr{H}}dxdy\,\,\,\,\,\,\,\,\,\,(\mbox{by }(\ref{eq:tmp19}))\\
 & = & \int_{0}^{1}\int_{0}^{1}\left\langle \varphi\left(x\right)U_{A}\left(-x\right)v_{0},\psi\left(y\right)U_{A}\left(-y\right)v_{0}\right\rangle _{\mathscr{H}}dxdy\\
 & = & \left\langle U\left(\varphi\right)v_{0},U\left(\psi\right)v_{0}\right\rangle _{\mathscr{H}}\,\,\,\,\,\,\,\,\,\,(\mbox{by }(\ref{eq:tmp20}))
\end{eqnarray*}
\end{svmultproof2}

\begin{corollary}
Set $\varphi^{\#}\left(x\right)=\overline{\varphi\left(-x\right)}$,
$x\in\mathbb{R}$, $\varphi\in C_{c}\left(\mathbb{R}\right)$, or
in this case, $\varphi\in C_{c}\left(0,1\right)$; then we have:
\begin{equation}
\left\langle F_{\varphi},F_{\psi}\right\rangle _{\mathscr{H}_{F}}=\left\langle v_{0},U\left(\varphi^{\#}\ast\psi\right)v_{0}\right\rangle _{\mathscr{H}},\;\forall\varphi,\psi\in C_{c}\left(0,1\right).\label{eq:tmp-02}
\end{equation}
\end{corollary}
\begin{svmultproof2}
Immediate from (\ref{eq:tmp-01}) and Fubini. \end{svmultproof2}

\begin{corollary}
Let $F$ and $\varphi\in C_{c}\left(0,1\right)$ be as above; then
in the RKHS $\mathscr{H}_{F}$ we have:
\begin{equation}
\left\Vert F_{\varphi}\right\Vert _{\mathscr{H}_{F}}^{2}=\left\Vert U\left(\varphi\right)v_{0}\right\Vert _{\mathscr{H}}^{2}=\int\left|\widehat{\varphi}\right|^{2}d\mu\label{eq:tmp-03}
\end{equation}
where $\mu$ is the measure in (\ref{eq:tmp18}). $\widehat{\varphi}=$
Fourier transform: $\widehat{\varphi}\left(\lambda\right)=\int_{0}^{1}e^{-i\lambda x}\varphi\left(x\right)dx$,
$\lambda\in\mathbb{R}$.\end{corollary}
\begin{svmultproof2}
Immediate from (\ref{eq:tmp-02}); indeed:
\begin{eqnarray*}
\left\Vert F_{\varphi}\right\Vert _{\mathscr{H}_{F}}^{2} & = & \int_{0}^{1}\int_{0}^{1}\overline{\varphi\left(x\right)}\varphi\left(y\right)\int_{\mathbb{R}}e_{\lambda}\left(x-y\right)d\mu\left(\lambda\right)\\
 & = & \int_{\mathbb{R}}\left|\widehat{\varphi}\left(\lambda\right)\right|^{2}d\mu\left(\lambda\right),\;\forall\varphi\in C_{c}\left(0,1\right).
\end{eqnarray*}
\end{svmultproof2}

\begin{corollary}
Every Borel probability measure $\mu$ on $\mathbb{R}$ arises this
way. \index{measure!probability}\end{corollary}
\begin{svmultproof2}
We shall need to following:
\begin{lemma}
Let $A$, $\mathscr{H}$, $\left\{ U_{A}\left(t\right)\right\} _{t\in\mathbb{R}}$,
$v_{0}\in\mathscr{H}$ be as above; and set
\begin{equation}
d\mu=d\mu_{A}\left(\cdot\right)=\left\Vert P_{U}\left(\cdot\right)v_{0}\right\Vert ^{2}\label{eq:tmp-04}
\end{equation}
as in (\ref{eq:tmp18}). Assume $v_{0}$ is cyclic; then $W_{\mu}f\left(A\right)v_{0}=f$
defines a unitary isomorphism $W_{\mu}:\mathscr{H}\rightarrow L^{2}\left(\mu\right)$;
and 
\begin{equation}
W_{\mu}U_{A}\left(t\right)=e^{it\cdot}W_{\mu}\label{eq:tmp-05}
\end{equation}
where $e^{it\cdot}$ is seen as a multiplication operator in $L^{2}\left(\mu\right)$.
More precisely:\index{operator!multiplication-}
\begin{equation}
\left(W_{\mu}U\left(t\right)\xi\right)\left(\lambda\right)=e^{it\lambda}\left(W_{\mu}\xi\right)\left(\lambda\right),\;\forall t,\lambda\in\mathbb{R},\forall\xi\in\mathscr{H}.\label{eq:tmp-06}
\end{equation}
(We say that the isometry $W_{\mu}$ \uline{intertwines} the two
unitary one-parameter groups.)

\index{isometry}

\index{operator!unitary one-parameter group}

\index{unitary one-parameter group}\index{representation!cyclic-}\end{lemma}
\begin{svmultproof2}
Since $v_{0}$ is cyclic, it is enough to consider $\xi\in\mathscr{H}$
of the following form: $\xi=f\left(A\right)v_{0}$, with $f\in L^{2}\left(\mu\right)$,
see (\ref{eq:tmp16}) in Lemma \ref{lem:iso}. Then
\begin{equation}
\left\Vert \xi\right\Vert _{\mathscr{H}}^{2}=\int_{\mathbb{R}}\left|f\left(\lambda\right)\right|^{2}d\mu\left(\lambda\right),\:\mbox{so}\label{eq:tmp-07}
\end{equation}
\[
\left\Vert W_{\mu}\xi\right\Vert _{L^{2}\left(\mu\right)}=\left\Vert \xi\right\Vert _{\mathscr{H}}\:(\Longleftrightarrow(\ref{eq:tmp-07}))
\]

For the adjoint operator $W_{\mu}^{*}:L^{2}\left(\mathbb{R},\mu\right)\rightarrow\mathscr{H}$,
we have
\[
W_{\mu}^{*}f=f\left(A\right)v_{0},
\]
see (\ref{eq:tmp15})-(\ref{eq:tmp17}). Note that $f\left(A\right)v_{0}\in\mathscr{H}$
is well-defined for all $f\in L^{2}\left(\mu\right)$. Also $W_{\mu}^{*}W_{\mu}=I_{\mathscr{H}}$,
$W_{\mu}W_{\mu}^{*}=I_{L^{2}\left(\mu\right)}$. 

Proof of (\ref{eq:tmp-06}). Take $\xi=f\left(A\right)v_{0}$, $f\in L^{2}\left(\mu\right)$,
and apply the previous lemma, we have 
\[
W_{\mu}U\left(t\right)\xi=W_{\mu}U\left(t\right)f\left(A\right)_{0}=W_{\mu}\left(e^{it\cdot}f\left(\cdot\right)\right)\left(A\right)v_{0}=e^{it\cdot}f\left(\cdot\right)=e^{it\cdot}W_{\mu}\xi;
\]
or written differently:
\[
W_{\mu}U\left(t\right)=M_{e^{it\cdot}}W_{\mu},\;\forall t\in\mathbb{R}
\]
where $M_{e^{it\cdot}}$ is the multiplication operator by $e^{it\cdot}$.
\end{svmultproof2}

\end{svmultproof2}

\begin{remark}
Deficiency indices $\left(1,1\right)$ occur for probability measures
$\mu$ on $\mathbb{R}$ such that\index{operator!adjoint of an-}
\begin{equation}
\int_{\mathbb{R}}\left|\lambda\right|^{2}d\mu\left(\lambda\right)=\infty.\label{eq:tmp-08}
\end{equation}
See examples below.
\end{remark}
\renewcommand{\arraystretch}{3}

\begin{table}[H]
\begin{tabular}{|>{\centering}p{0.3\textwidth}|>{\centering}p{0.1\textwidth}|>{\centering}p{0.4\textwidth}|>{\centering}p{0.15\textwidth}|}
\hline 
p.d. function $F$ & measure & condition (\ref{eq:tmp-08}) & deficiency indices\tabularnewline
\hline 
$F_{1}\left(x\right)=\frac{1}{1+x^{2}}$, $\left|x\right|<1$ & $\mu_{1}$ & $\int_{\mathbb{R}}\left|\lambda\right|^{2}\frac{1}{2}e^{-\left|\lambda\right|}d\lambda<\infty$ & $\left(0,0\right)$\tabularnewline
\hline 
$F_{2}\left(x\right)=1-\left|x\right|$, $\left|x\right|<\frac{1}{2}$ & $\mu_{2}$ & $\int_{\mathbb{R}}\left|\lambda\right|^{2}\frac{1}{2\pi}\left(\frac{\sin\left(\lambda/2\right)}{\lambda/2}\right)^{2}d\lambda=\infty$ & $\left(1,1\right)$\tabularnewline
\hline 
$F_{3}\left(x\right)=e^{-\left|x\right|}$, $\left|x\right|<1$ & $\mu_{3}$ & $\int_{\mathbb{R}}\left|\lambda\right|^{2}\frac{d\lambda}{\pi\left(1+\lambda^{2}\right)}=\infty$ & $\left(1,1\right)$\tabularnewline
\hline 
$F_{4}\left(x\right)=\left(\frac{\sin\left(x/2\right)}{x/2}\right)^{2}$,
$\left|x\right|<\frac{1}{2}$ & $\mu_{4}$ & $\int_{\mathbb{R}}\left|\lambda\right|^{2}\chi_{\left(-1,1\right)}\left(\lambda\right)\left(1-\left|\lambda\right|\right)d\lambda<\infty$ & $\left(0,0\right)$\tabularnewline
\hline 
$F_{5}\left(x\right)=e^{-x^{2}/2}$, $\left|x\right|<1$ & $\mu_{5}$ & $\int_{\mathbb{R}}\left|\lambda\right|^{2}\frac{1}{\sqrt{2\pi}}e^{-\lambda^{2}/2}d\lambda=1<\infty$ & $\left(0,0\right)$\tabularnewline
\hline 
$F_{6}\left(x\right)=\cos x$, $\left|x\right|<\frac{\pi}{4}$ & $\mu_{6}$ & $\int_{\mathbb{R}}\left|\lambda\right|^{2}\frac{1}{2}\left(\delta_{1}+\delta_{-1}\right)d\lambda=1<\infty$ & $\left(0,0\right)$\tabularnewline
\hline 
$F_{7}\left(x\right)=\left(1-ix\right)^{-p}$, $\left|x\right|<1$ & $\mu_{7}$ & $\int_{0}^{\infty}\lambda^{2}\frac{\lambda^{p}}{\Gamma\left(p\right)}e^{-\lambda}d\lambda<\infty$ & $\left(0,0\right)$\tabularnewline
\hline 
\end{tabular}

\protect\caption{\label{tab:meas}Application of Theorem \ref{thm:defmeas} to the
functions from Tables \ref{tab:F1-F6} and \ref{tab:Table-2}. From
the results in Chapter \ref{chap:types}, we conclude that $Ext_{1}\left(F\right)=\left\{ \mbox{a singleton}\right\} $
in the cases $i=1,4,5,6,7$. }
\end{table}

\begin{figure}[H]
\begin{tabular}{cc}
\includegraphics[width=0.45\textwidth]{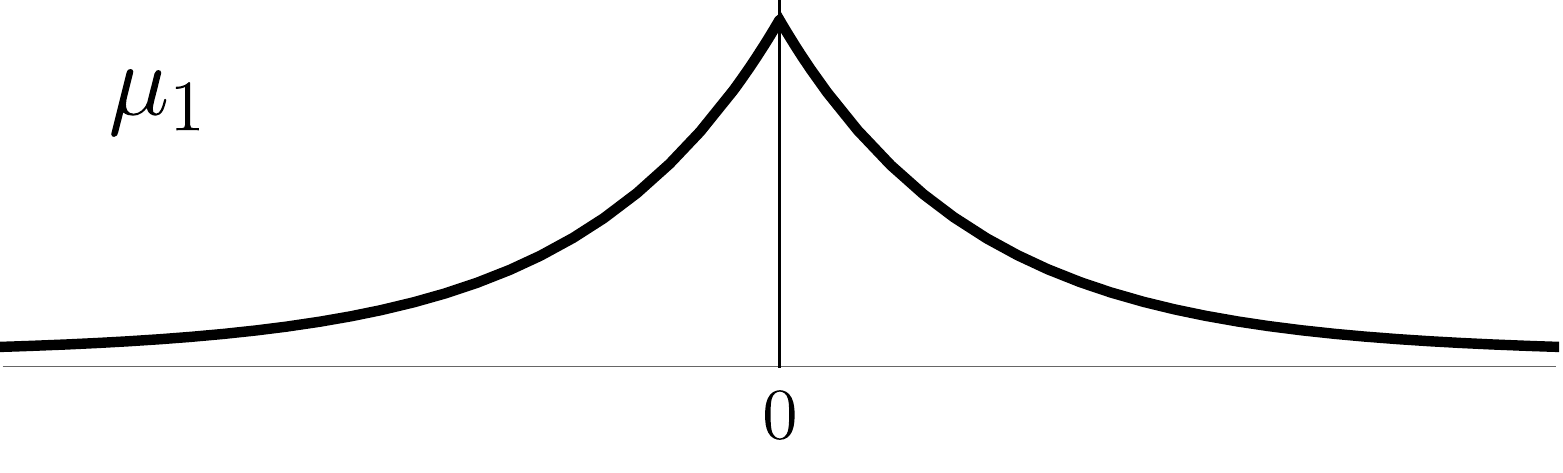} & \includegraphics[width=0.45\textwidth]{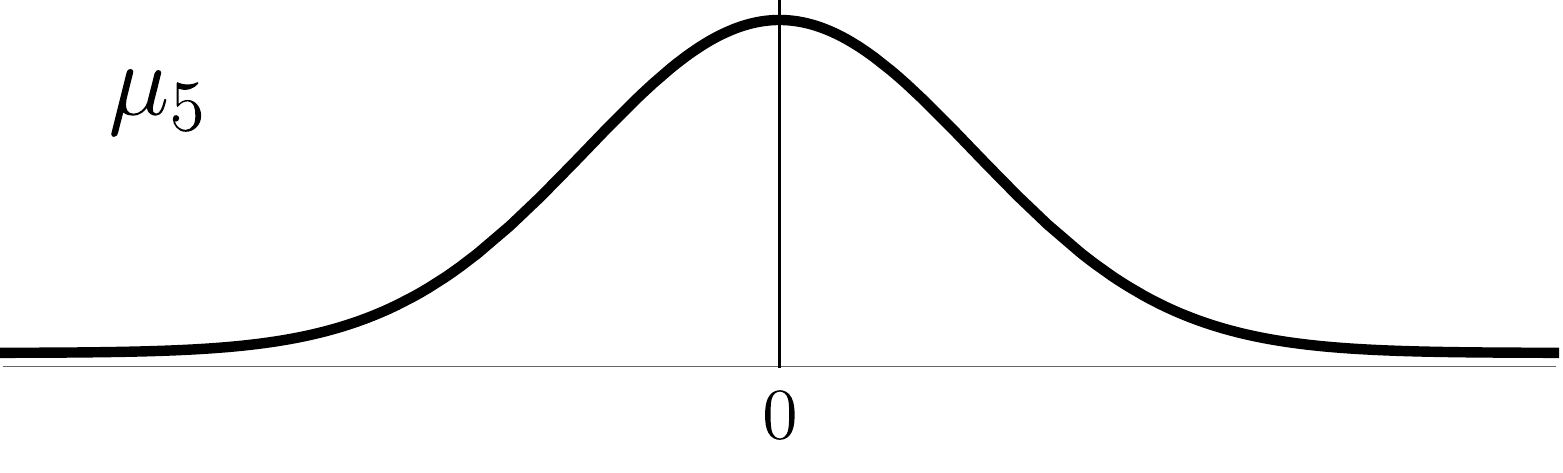}\tabularnewline
\includegraphics[width=0.45\textwidth]{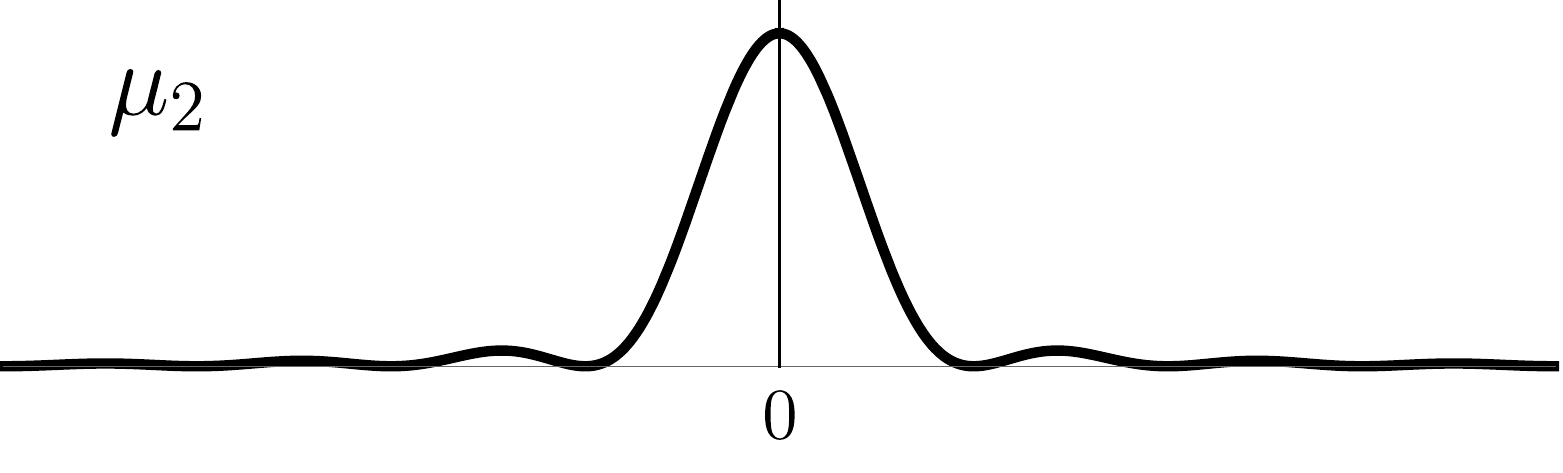} & \includegraphics[width=0.45\textwidth]{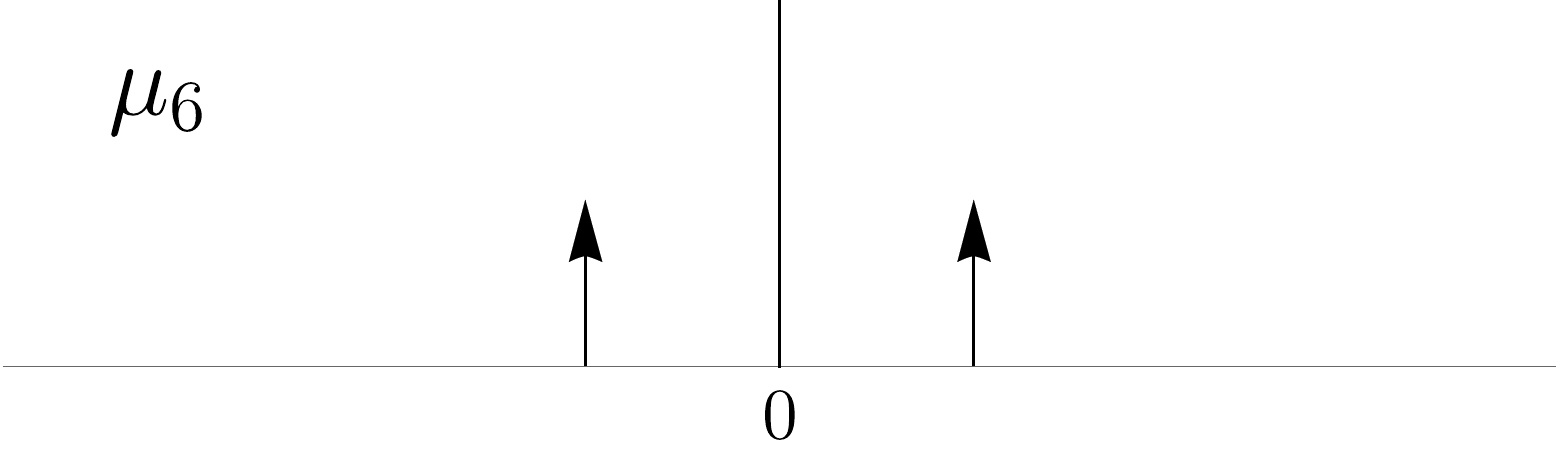}\tabularnewline
\includegraphics[width=0.45\textwidth]{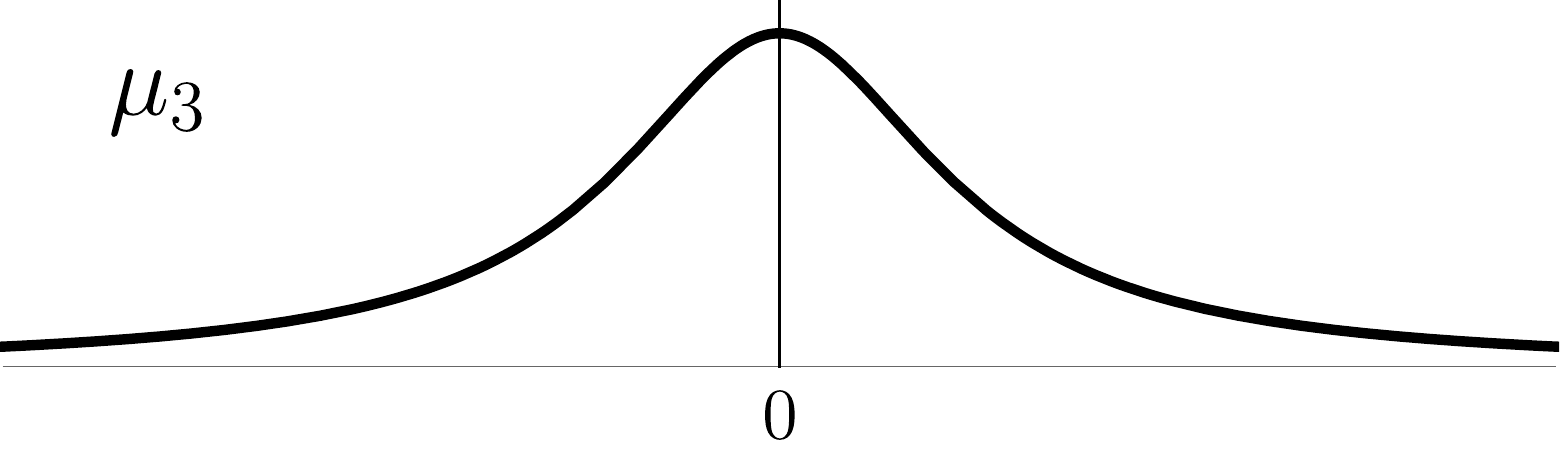} & \includegraphics[width=0.45\textwidth]{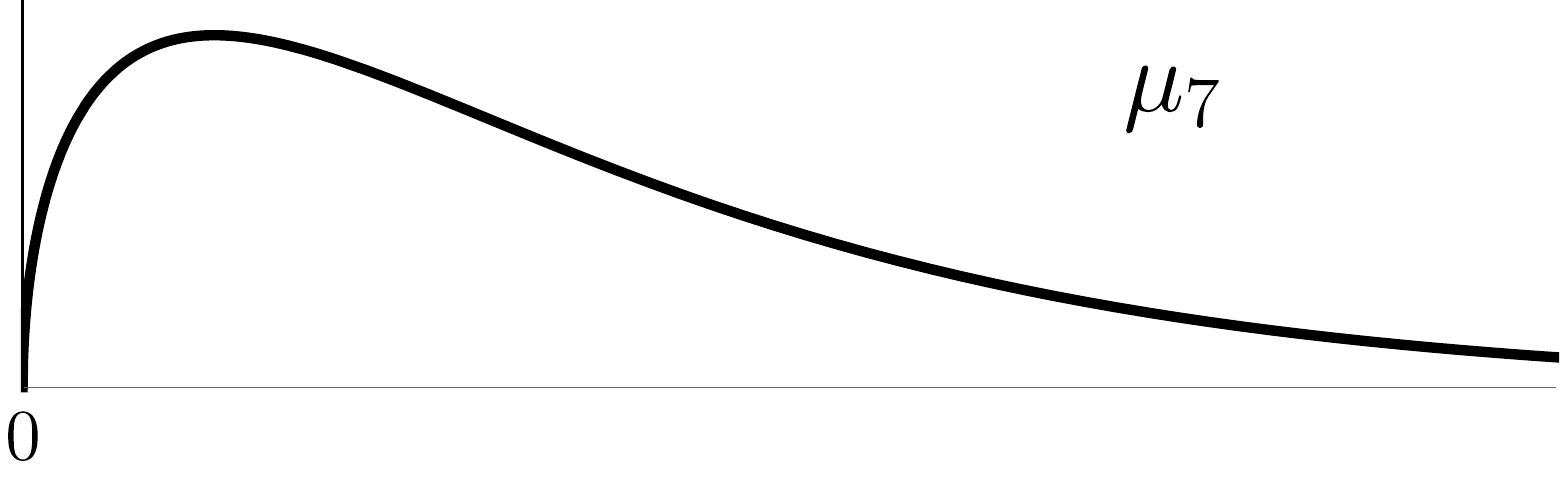}\tabularnewline
\includegraphics[width=0.45\textwidth]{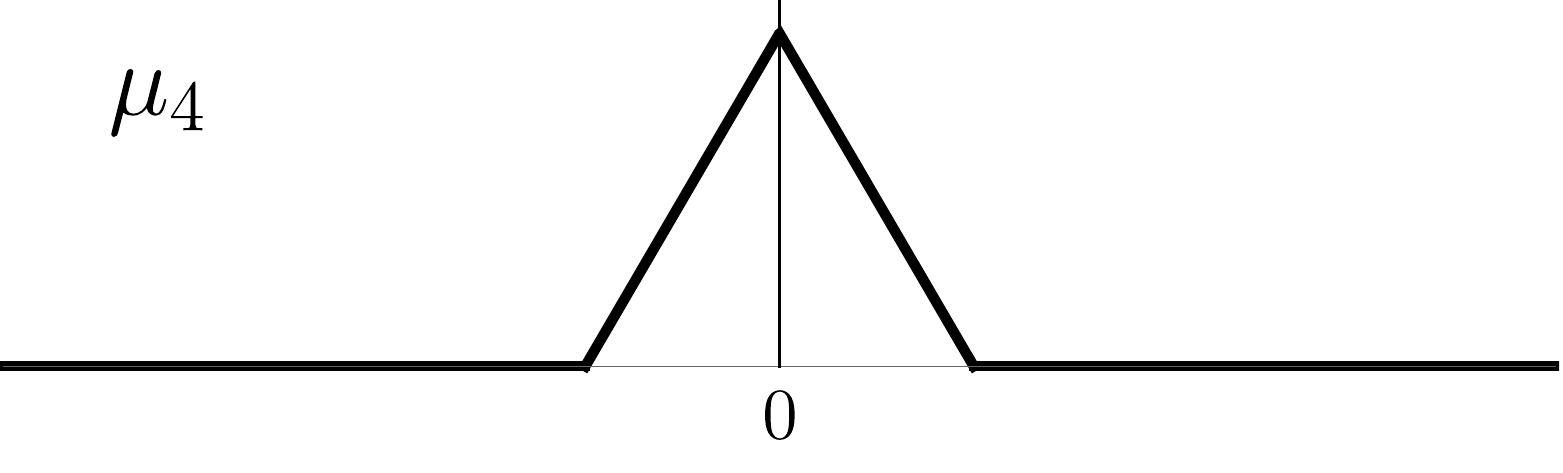} & \tabularnewline
\end{tabular}

\protect\caption{\label{fig:meas}The measures $\mu_{i}\in Ext\left(F_{i}\right)$
extending p.d. functions $F_{i}$ in Table \ref{tab:F1-F6}, $i=1,2,\ldots7$. }
\end{figure}

\renewcommand{\arraystretch}{1}
\begin{remark}
For the two p.d. functions $F_{2}$ and $F_{3}$ in Table \ref{fig:meas};
i.e., the cases of deficiency indices $\left(1,1\right)$, one shows
that, in each case, $Ext_{1}\left(F\right)$ is a one-parameter family
$\{F^{\left(\theta\right)}\}$ indexed by $\theta\in\mathbb{T}=\mathbb{R}/\mathbb{Z}$,
$F^{\left(\theta\right)}=\widehat{d\mu^{\left(\theta\right)}}$, where
the associated probability measures $\{\mu^{\left(\theta\right)}\}$
constitute a system of mutually singular measures, i.e., for $\theta_{1}\neq\theta_{2}$,
the two measures are mutually singular. 

The proof of these conclusions is based on the von Neumann formula
for the elements in $Ext_{1}\left(F\right)$, together with our results
for the index $\left(1,1\right)$ operator $D^{\left(F\right)}$;
see Chapters \ref{chap:ext}-\ref{chap:types}.
\end{remark}
Summary: restrictions with deficiency indices $\left(1,1\right)$.

\index{measure!probability}
\begin{theorem}
\label{thm:defmeas}If $\mu$ is a fixed probability measure on $\mathbb{R}$,
then the following two conditions are equivalent:
\begin{enumerate}
\item $\int_{\mathbb{R}}\lambda^{2}d\mu\left(\lambda\right)=\infty$;
\item The set 
\[
dom\left(S\right)=\left\{ f\in L^{2}\left(\mu\right)\:\Big|\:\lambda f\in L^{2}\left(\mu\right)\mbox{ and }\int_{\mathbb{R}}\left(\lambda+i\right)f\left(\lambda\right)d\mu\left(\lambda\right)=0\right\} 
\]
is the \uline{dense} domain of a restriction operator $S\subset M_{\lambda}$
with deficiency indices $\left(1,1\right)$, and the deficiency space
$DEF_{+}=\mathbb{C}1$, \emph{(}$1=$ the constant function $1$ in
$L^{2}\left(\mu\right)$.\emph{)}\index{operator!domain of-}
\end{enumerate}
\end{theorem}
\begin{remark}
By Table \ref{tab:meas}, the theorem applies to $\mu_{2}$ and $\mu_{3}$.
\end{remark}

\section{\label{sec:index 11}A Model of ALL Deficiency Index-$\left(1,1\right)$
Operators}
\begin{lemma}
\label{lem:def1}Let $\mu$ be a Borel probability measure on $\mathbb{R}$,
and denote $L^{2}\left(\mathbb{R},d\mu\right)$ by $L^{2}\left(\mu\right)$.
then we have TFAE:
\begin{enumerate}
\item 
\begin{equation}
\int_{\mathbb{R}}\left|\lambda\right|^{2}d\mu\left(\lambda\right)=\infty\label{eq:tmp-09}
\end{equation}

\item the following two subspaces in $L^{2}\left(\mu\right)$are dense (in
the $L^{2}\left(\mu\right)$-norm):
\begin{equation}
\left\{ f\in L^{2}\left(\mu\right)\Big|\left[\left(\lambda\pm i\right)f\left(\lambda\right)\right]\in L^{2}\left(\mu\right)\mbox{ and }\int\left(\lambda\pm i\right)f\left(\lambda\right)d\mu\left(\lambda\right)=0\right\} \label{eq:tmp-10}
\end{equation}
where $i=\sqrt{-1}$.
\end{enumerate}
\end{lemma}
\begin{svmultproof2}
See \cite{Jor81}.\end{svmultproof2}

\begin{remark}
If (\ref{eq:tmp-09}) holds, then the two dense subspaces $\mathscr{D}_{\pm}\subset L^{2}\left(\mu\right)$
in (\ref{eq:tmp-10}) form the dense domain of a restriction $S$
of $M_{\lambda}$ in $L^{2}\left(\mu\right)$; and this restriction
has deficiency indices $\left(1,1\right)$. Moreover, all Hermitian
operators having deficiency indices $\left(1,1\right)$ arise this
way. \index{deficiency indices}\index{operator!domain of-}

Assume (\ref{eq:tmp-09}) holds; then the subspace 
\[
\mathscr{D}=\left\{ f\in L^{2}\left(\mu\right)\Big|\left(\lambda+i\right)f\in L^{2}\left(\mu\right)\mbox{ and }\int\left(\lambda+i\right)f\left(\lambda\right)d\mu\left(\lambda\right)=0\right\} 
\]
is a dense domain of a restricted operator of $M_{\lambda}$, so $S\subset M_{\lambda}$,
and $S$ is Hermitian. \end{remark}
\begin{lemma}
\label{lem:index(1,1)}With $i=\sqrt{-1}$, set 
\begin{equation}
dom\left(S\right)=\left\{ f\in L^{2}\left(\mu\right)\Big|\lambda f\in L^{2}\left(\mu\right)\mbox{ and }\int\left(\lambda+i\right)f\left(\lambda\right)d\mu\left(\lambda\right)=0\right\} \label{eq:tmp-12}
\end{equation}
then $S\subset M_{\lambda}\subset S^{*}$; and the deficiency subspaces
$DEF_{\pm}$ are as follow:
\begin{eqnarray}
DEF_{+} & = & \mbox{the constant function in }L^{2}\left(\mu\right)=\mathbb{C}1\label{eq:tmp-11}\\
DEF_{-} & = & span\left\{ \frac{\lambda-i}{\lambda+i}\right\} _{\lambda\in\mathbb{R}}\subseteq L^{2}\left(\mu\right)\label{eq:tmp-14}
\end{eqnarray}
where $DEF_{-}$ is also a 1-dimensional subspace in $L^{2}\left(\mu\right)$.\end{lemma}
\begin{svmultproof2}
Let $f\in dom\left(S\right)$, then, by definition, 
\[
\int_{\mathbb{R}}\left(\lambda+i\right)f\left(\lambda\right)d\mu\left(\lambda\right)=0\;\mbox{and so}
\]
\begin{equation}
\left\langle 1,\left(S+iI\right)f\right\rangle _{L^{2}\left(\mu\right)}=\int_{\mathbb{R}}\left(\lambda+i\right)f\left(\lambda\right)d\mu\left(\lambda\right)=0\label{eq:tmp-13}
\end{equation}
hence (\ref{eq:tmp-11}) follows.

Note we have formula (\ref{eq:tmp-12}) for $dom\left(S\right)$.
Moreover $dom\left(S\right)$ is dense in $L^{2}\left(\mu\right)$
because of (\ref{eq:tmp-10}) in Lemma \ref{lem:def1}.

Now to (\ref{eq:tmp-14}): Let $f\in dom\left(S\right)$; then
\begin{eqnarray*}
\left\langle \frac{\lambda-i}{\lambda+i},\left(S-iI\right)f\right\rangle _{L^{2}\left(\mu\right)} & = & \int_{\mathbb{R}}\left(\frac{\lambda+i}{\lambda-i}\right)\left(\lambda-i\right)f\left(\lambda\right)d\mu\left(\lambda\right)\\
 & = & \int_{\mathbb{R}}\left(\lambda+i\right)f\left(\lambda\right)d\mu\left(\lambda\right)=0
\end{eqnarray*}
again using the definition of $dom\left(S\right)$ in (\ref{eq:tmp-12}).
\end{svmultproof2}

We have established a representation for \emph{all} Hermitian operators
with dense domain in a Hilbert space, and having deficiency indices
$\left(1,1\right)$. In particular, we have justified the answers
in Table \ref{tab:F1-F6} for $F_{i}$, $i=1,\ldots,5$. 

To further emphasize to the result we need about deficiency indices
$\left(1,1\right)$, we have the following:\index{representation!spectral-}
\begin{theorem}
Let $\mathscr{H}$ be a separable Hilbert space, and let $S$ be a
Hermitian operator with dense domain in $\mathscr{H}$. Suppose the
deficiency indices of $S$ are $\left(d,d\right)$; and suppose one
of the selfadjoint extensions of $S$ has simple spectrum. \index{selfadjoint extension}

Then the following two conditions are equivalent:
\begin{enumerate}
\item \label{enu:(11)-1}$d=1$;
\item \label{enu:(11)-2}for each of the selfadjoint extensions $T$ of
$S$, we have a unitary equivalence between $\left(S,\mathscr{H}\right)$
on the one hand, and a system $\left(S_{\mu},L^{2}\left(\mathbb{R},\mu\right)\right)$
on the other, where $\mu$ is a Borel probability measure on $\mathbb{R}$.
Moreover, 
\begin{equation}
dom\left(S_{\mu}\right)=\left\{ f\in L^{2}\left(\mu\right)\Big|\lambda f\left(\cdot\right)\in L^{2}\left(\mu\right),\mbox{ and }\int_{\mathbb{R}}\left(\lambda+i\right)f\left(\lambda\right)d\mu\left(\lambda\right)=0\right\} ,\label{eq:tmp-15}
\end{equation}
and
\begin{equation}
\left(S_{\mu}f\right)\left(\lambda\right)=\lambda f\left(\lambda\right),\;\forall f\in dom\left(S_{\mu}\right),\forall\lambda\in\mathbb{R}.\label{eq:tmp-16}
\end{equation}

\end{enumerate}

In case $\mu$ satisfies condition (\ref{eq:tmp-16}), then the constant
function $\mathbf{1}$ (in $L^{2}\left(\mathbb{R},\mu\right)$) is
in the domain of $S_{\mu}^{*}$, and 
\begin{equation}
S_{\mu}^{*}\mathbf{1}=i\mathbf{1}\label{eq:tmp-17}
\end{equation}
i.e., $\left(S_{\mu}^{*}\mathbf{1}\right)\left(\lambda\right)=i$,
a.a. $\lambda$ w.r.t. $d\mu$.

\end{theorem}
\begin{svmultproof2}
For the implication (\ref{enu:(11)-2})$\Rightarrow$(\ref{enu:(11)-1}),
see Lemma \ref{lem:index(1,1)}.

(\ref{enu:(11)-1})$\Rightarrow$(\ref{enu:(11)-2}). Assume that
the operator $S$, acting in $\mathscr{H}$ is Hermitian with deficiency
indices $\left(1,1\right)$. This means that each of the two subspaces
$DEF_{\pm}\subset\mathscr{H}$ is one-dimensional, where 
\begin{equation}
DEF_{\pm}=\left\{ h_{\pm}\in dom\left(S^{*}\right)\Big|S^{*}h_{\pm}=\pm ih_{\pm}\right\} .\label{eq:tmp-20}
\end{equation}
 Now pick a selfadjoint extension, say $T$, extending $S$. We have
\begin{equation}
S\subseteq T=T^{*}\subseteq S^{*}\label{eq:tmp-18}
\end{equation}
where ``$\subseteq$'' in (\ref{eq:tmp-18}) means ``containment
of the respective graphs.''\index{operator!graph of-}

Now set $U\left(t\right)=e^{itT}$, $t\in\mathbb{R}$, and let $P_{U}\left(\cdot\right)$
be the corresponding projection-valued measure\index{measure!PVM},
i.e., we have:

\begin{equation}
U\left(t\right)=\int_{\mathbb{R}}e^{it\lambda}P_{U}\left(d\lambda\right),\;\forall t\in\mathbb{R}.\label{eq:tmp-19}
\end{equation}

Using the assumption (\ref{enu:(11)-1}), and (\ref{eq:tmp-20}),
it follows that there is a vector $h_{+}\in\mathscr{H}$ such that
$\left\Vert h_{+}\right\Vert _{\mathscr{H}}=1$, $h_{+}\in dom\left(S^{*}\right)$,
and $S^{*}h_{+}=ih_{+}$. Now set 
\begin{equation}
d\mu\left(\lambda\right):=\left\Vert P_{U}\left(d\lambda\right)h_{+}\right\Vert _{\mathscr{H}}^{2}.\label{eq:tmp-21}
\end{equation}
Using (\ref{eq:tmp-19}), we then verify that there is a unitary (and
isometric) isomorphism of $L^{2}\left(\mu\right)\overset{W}{\longrightarrow}\mathscr{H}$
given by\index{projection-valued measure (PVM)}
\begin{equation}
Wf=f\left(T\right)h_{+},\;\forall f\in L^{2}\left(\mu\right);\label{eq:tmp-22}
\end{equation}
where $f\left(T\right)=\int_{\mathbb{R}}f\left(T\right)P_{U}\left(d\lambda\right)$
is the functional calculus applied to the selfadjoint operator $T$.
Hence 
\begin{eqnarray*}
\left\Vert Wf\right\Vert _{\mathscr{H}}^{2} & = & \left\Vert f\left(T\right)h_{+}\right\Vert _{\mathscr{H}}^{2}\\
 & = & \int_{\mathbb{R}}\left|f\left(\lambda\right)\right|^{2}\left\Vert P_{U}\left(d\lambda\right)h_{+}\right\Vert ^{2}\\
 & = & \int_{\mathbb{R}}\left|f\left(\lambda\right)\right|^{2}d\mu\left(\lambda\right)\,\,\,\,\,\,\,\,\,\,(\mbox{by }\ref{eq:tmp-21})\\
 & = & \left\Vert f\right\Vert _{L^{2}\left(\mu\right)}^{2}.
\end{eqnarray*}
To see that $W$ in (\ref{eq:tmp-22}) is an isometric isomorphism
of $L^{2}\left(\mu\right)$ onto $\mathscr{H}$, we use the assumption
that $T$ has simple spectrum\index{spectrum}. 

Now set 
\begin{eqnarray}
S_{\mu} & := & W^{*}SW\label{eq:tmp-23}\\
T_{\mu} & := & W^{*}TW.\label{eq:tmp-24}
\end{eqnarray}
We note that $T_{\mu}$ is then the multiplication operator $M$ in
$L^{2}\left(\mathbb{R},\mu\right)$, given by\index{operator!multiplication-}
\begin{equation}
\left(Mf\right)\left(\lambda\right)=\lambda f\left(\lambda\right),\;\forall f\in L^{2}\left(\mu\right)\label{eq:tmp-25}
\end{equation}
such that $\lambda f\in L^{2}\left(\mu\right)$. This assertion is
immediate from (\ref{eq:tmp-22}) and (\ref{eq:tmp-21}).

To finish the proof, we compute the integral in (\ref{eq:tmp-15})
in the theorem, and we use the intertwining\index{intertwining} \index{operator!intertwining-}properties
of the isomorphism $W$ from (\ref{eq:tmp-22}). Indeed, we have 
\begin{eqnarray}
\int_{\mathbb{R}}\left(\lambda+i\right)f\left(\lambda\right)d\mu\left(\lambda\right) & = & \left\langle \mathbf{1},\left(M+iI\right)f\right\rangle _{L^{2}\left(\mu\right)}\nonumber \\
 & = & \left\langle W\mathbf{1},W\left(M+iI\right)f\right\rangle _{\mathscr{H}}\nonumber \\
 & \overset{(\ref{eq:tmp-21})}{=} & \left\langle h_{+},\left(T+iI\right)Wf\right\rangle _{\mathscr{H}}.\label{eq:tmp-26}
\end{eqnarray}
Hence $Wf\in dom\left(S\right)$ $\Longleftrightarrow$ $f\in dom\left(S_{\mu}\right)$,
by (\ref{eq:tmp-23}); and, so for $Wf\in dom\left(S\right)$, the
r.h.s. in (\ref{eq:tmp-26}) yields $\left\langle \left(S^{*}-iI\right)h_{+},Wf\right\rangle _{\mathscr{H}}=0$;
and the assertion (\ref{enu:(11)-2}) in the theorem follows.
\end{svmultproof2}

\subsection{\label{sub:momentum}Momentum Operators in $L^{2}\left(0,1\right)$}

What about our work on momentum operators in $L^{2}\left(0,1\right)$?
They have deficiency indices $\left(1,1\right)$; and there is the
family of measures $\mu$ on $\mathbb{R}$:\index{operator!momentum-}

Fix $0\leq\theta<1$, fix 
\begin{equation}
w_{n}>0,\;\sum_{n\in\mathbb{Z}}w_{n}=1;\label{eq:L01-1}
\end{equation}
and let $\mu=\mu_{\left(\theta,w\right)}$ as follows:
\begin{equation}
d\mu=\sum_{n\in\mathbb{Z}}w_{n}\delta_{\theta+n}\label{eq:L01-2}
\end{equation}
such that 
\begin{equation}
\sum_{n\in\mathbb{Z}}n^{2}w_{n}=\infty.\label{eq:L01-3}
\end{equation}

In this case, there is the bijective $L^{2}\left(\mathbb{R},d\mu\right)\longleftrightarrow\left\{ \xi_{n}\right\} _{n\in\mathbb{Z}}$,
$\xi\in\mathbb{C}$ s.t.
\[
\sum_{n\in\mathbb{Z}}\left|\xi_{n}\right|^{2}w_{n}<\infty.
\]
We further assume that 
\[
\sum n^{2}w_{n}=\infty.
\]
The trick is to pick a representation in which $v_{0}$ in Lemma \ref{lem:iso},
i.e., (\ref{eq:tmp15})-(\ref{eq:tmp17}): $\mathscr{H}$, cyclic
vector $v_{0}\in\mathscr{H}$, $\left\Vert v_{0}\right\Vert _{\mathscr{H}}=1$,
$\left\{ U_{A}\left(t\right)\right\} _{t\in\mathbb{R}}$.\index{representation!cyclic-}
\begin{lemma}
\label{lem:(01)}The $\left(1,1\right)$ index condition is the case
of the momentum operator in $L^{2}\left(0,1\right)$. To see this
we set $v_{0}=e^{x}$, or $v_{0}=e^{2\pi x}$. \end{lemma}
\begin{svmultproof2}
Defect vector in $\mathscr{H}=L^{2}\left(0,1\right)$, say $e^{2\pi x}$:
\[
v_{0}\left(x\right)=\sum_{n\in\mathbb{Z}}c_{n}e_{n}^{\theta}\left(x\right)
\]
with 
\[
w_{n}\sim\frac{\left(\cos\left(2\pi\left(\theta+n\right)\right)-e^{-2\pi}\right)^{2}+\sin^{2}\left(2\pi\left(\theta+n\right)\right)}{1+\left(\theta+n\right)^{2}}
\]
and so the $\left(1,1\right)$ condition holds, i.e., 
\[
\sum_{n\in\mathbb{Z}}n^{2}w_{n}=\infty.
\]

\end{svmultproof2}

We specialize to $\mathscr{H}=L^{2}\left(0,1\right)$, and the usual
momentum operator in $\left(0,1\right)$ with the selfadjoint extensions.
Fix $\theta$, $0\leq\theta<1$, we have an ONB in $\mathscr{H}=L^{2}\left(0,1\right)$,
$e_{k}^{\theta}\left(x\right):=e^{i2\pi\left(\theta+k\right)x}$,
and\index{operator!momentum-} 
\begin{equation}
v_{0}\left(x\right)=\sum_{k\in\mathbb{Z}}c_{k}e_{k}^{\theta}\left(x\right),\;\sum_{k\in\mathbb{Z}}\left|c_{k}\right|^{2}=1.\label{eq:L01-4}
\end{equation}

\textbf{Fact.} $v_{0}$ is cyclic for $\left\{ U_{A_{\theta}}\left(t\right)\right\} _{t\in\mathbb{R}}$$\Longleftrightarrow$
$c_{k}\neq0$, $\forall k\in\mathbb{Z}$; then set $w_{k}=\left|c_{k}\right|^{2}$
, and the conditions (\ref{eq:L01-1})-(\ref{eq:L01-3}) holds. Reason:
the measure $\mu=\mu_{v_{0},\theta}$ depends on the choice of $v_{0}$.
The non-trivial examples here $v_{0}\longleftrightarrow\left\{ c_{k}\right\} _{k\in\mathbb{Z}}$,
$c_{k}\neq0$, $\forall k$, is cyclic, i.e., 
\[
cl\:span\left\{ U_{A_{\theta}}\left(t\right)v_{0}\Big|t\in\mathbb{R}\right\} \:\left(=\mathscr{H}=L^{2}\left(0,1\right)\right)
\]

\begin{lemma}
The measure $\mu$ in (\ref{eq:L01-2}) is determined as:
\begin{eqnarray*}
F_{\theta}\left(t\right) & = & \left\langle v_{0},U_{A_{\theta}}\left(t\right)v_{0}\right\rangle \\
 & = & \sum_{k\in\mathbb{Z}}e^{it2\pi\left(\theta+k\right)}w_{k}\\
 & = & \int_{\mathbb{R}}e^{it\lambda}d\mu\left(\lambda\right)
\end{eqnarray*}
where $\mu$ is as in (\ref{eq:L01-2}); so purely atomic.\index{atom}\index{purely atomic}

Note that our boundary conditions for the selfadjoint extensions of
the minimal momentum operator are implied by (\ref{eq:L01-4}), i.e.,
\begin{equation}
f_{0}\left(x+1\right)=e^{i2\pi\theta}f_{0}\left(x\right),\;\forall x\in\mathbb{R}.\label{eq:L01-05}
\end{equation}
It is implied by choices of $\theta$ s.t.
\[
f_{0}=\sum_{n\in\mathbb{Z}}B_{n}e_{n}^{\theta}\left(x\right),\;\sum_{n\in\mathbb{Z}}\left|B_{n}\right|^{2}<\infty.
\]

\end{lemma}

\subsection{Restriction Operators}

In this representation, the restriction operators are represented
by sequence $\left\{ f_{k}\right\} _{k\in\mathbb{Z}}$ s.t. $\sum\left|f_{k}\right|^{2}w_{k}<\infty$,
so use restriction the selfadjoint operator corresponds to the $\theta$-boundary
conditions (\ref{eq:L01-05}), $dom\left(S\right)$ where $S$ is
the Hermitian restriction operator, i.e., $S\subset s.a.\subset S^{*}$.
It has its dense domain $dom\left(S\right)$ as follows: $\left(f_{k}\right)\in dom\left(S\right)$
$\Longleftrightarrow$ $\left(kf_{k}\right)\in l^{2}\left(\mathbb{Z},w\right)$,
and 
\[
\sum_{k\in\mathbb{Z}}\left(\theta+k+i\right)f_{k}w_{k}=0;\;i=\sqrt{-1}.
\]

Comparison with (\ref{eq:tmp-12}) in the general case. What is special
about this case $\mathscr{H}=L^{2}\left(0,1\right)$ and the usual
boundary conditions at the endpoints is that the family of measures
$\mu$ on $\mathbb{R}$ are purely atomic\index{measure!atomic};
see (\ref{eq:L01-2})\index{purely atomic} \index{operator!restriction-}\index{Hermitian}
\[
d\mu=\sum_{k\in\mathbb{Z}}w_{k}\delta_{k+\theta}.
\]

\section{\label{sec:index (d,d)}The Case of Indices $\left(d,d\right)$ where
$d>1$ }

Let $\mu$ be a Borel probability measure on $\mathbb{R}$, and let
\begin{equation}
L^{2}\left(\mu\right):=L^{2}\left(\mathbb{R},\mathscr{B},\mu\right).\label{eq:d-1}
\end{equation}
The notation $\mbox{Prob}\left(\mathbb{R}\right)$ will be used for
these measures. 

We saw that the restriction/extension problem \index{extension problem}for
continuous positive definite\index{positive definite} (p.d.) functions
$F$ on $\mathbb{R}$ may be translated into a spectral theoretic
model in some $L^{2}\left(\mu\right)$ for suitable $\mu\in\mbox{Prob}\left(\mathbb{R}\right)$.
We saw that extension from a finite open ($\neq\emptyset$) interval
leads to spectral representation in $L^{2}\left(\mu\right)$, and
restrictions of\index{representation!spectral-} 
\begin{equation}
\left(M_{\mu}f\right)\left(\lambda\right)=\lambda f\left(\lambda\right),\;f\in L^{2}\left(\mu\right)\label{eq:d-2}
\end{equation}
having deficiency-indices $\left(1,1\right)$; hence the case $d=1$. 
\begin{theorem}
Fix $\mu\in\mbox{Prob}\left(\mathbb{R}\right)$. There is a 1-1 bijective
correspondence between the following:
\begin{enumerate}
\item \label{enu:d-1}certain closed subspaces $\mathscr{L}\subset L^{2}\left(\mu\right)$
\item \label{enu:d-2}Hermitian restrictions $S_{\mathscr{L}}$ of $M_{\mu}$
(see (\ref{eq:d-2})) such that 
\begin{equation}
DEF_{+}\left(S_{\mathscr{L}}\right)=\mathscr{L}.\label{eq:d-3}
\end{equation}

\end{enumerate}

The closed subspaces in (\ref{enu:d-1}) are specified as follows:
\begin{enumerate}[label=(\roman{enumi}),ref=\roman{enumi}]
\item \label{enu:d-3}$\dim\left(\mathscr{L}\right)=d<\infty$
\item \label{enu:d-4}the following implication holds:
\begin{equation}
g\neq0,\:\mbox{and }g\in\mathscr{L}\Longrightarrow\left(\left[\lambda\mapsto\lambda g\left(\lambda\right)\right]\notin L^{2}\left(\mu\right)\right)\label{eq:d-4}
\end{equation}

\end{enumerate}

Then set 
\begin{equation}
dom\left(S_{\mathscr{L}}\right):=\left\{ f\in dom\left(M_{\mu}\right)\Big|\int\overline{g\left(\lambda\right)}\left(\lambda+i\right)f\left(\lambda\right)d\mu\left(\lambda\right),\forall g\in\mathscr{L}\right\} \label{eq:d-5}
\end{equation}
and set 
\begin{equation}
S_{\mathscr{L}}:=M_{\mu}\Big|_{dom\left(S_{\mathscr{L}}\right)}\label{eq:d-6}
\end{equation}
where $dom\left(S_{\mathscr{L}}\right)$ is specified as in (\ref{eq:d-5}).

\end{theorem}
\begin{svmultproof2}
Note that the case $d=1$ is contained in the previous theorem. 

Proof of (\ref{enu:d-1}) $\Rightarrow$ (\ref{enu:d-2}). We will
be using an idea from \cite{Jor81}. With assumptions (\ref{enu:d-3})-(\ref{enu:d-4}),
in particular (\ref{eq:d-4}), one checks that $dom\left(S_{\mathscr{L}}\right)$as
specified in (\ref{eq:d-5}) is dense in $L^{2}\left(\mu\right)$.
In fact, the converse implication is also true. 

Now setting $S_{\mathscr{L}}$ to be the restriction in (\ref{eq:d-6}),
we conclude that
\begin{equation}
S_{\mathscr{L}}\subseteq M_{\mu}\subseteq S_{\mathscr{L}}^{*}\label{eq:d-7}
\end{equation}
where 
\begin{eqnarray*}
dom\left(S_{\mathscr{L}}^{*}\right) & = & \Biggl\{ h\in L^{2}\left(\mu\right)\:\Big|\:s.t.\:\exists C<\infty\mbox{ and}\\
 &  & \left|\int_{\mathbb{R}}\overline{h\left(\lambda\right)}\lambda f\left(\lambda\right)d\mu\left(\lambda\right)\right|^{2}\leq C\int_{\mathbb{R}}\left|f\left(\lambda\right)\right|^{2}d\mu\left(\lambda\right)\\
 &  & \mbox{holds }\forall f\in dom\left(S_{\mathscr{L}}\right)\Biggr\}
\end{eqnarray*}
 The assertions in (\ref{enu:d-2}) now follow from this.

Proof of (\ref{enu:d-2}) $\Rightarrow$ (\ref{enu:d-1}). Assume
that $S$ is a densely defined restriction of $M_{\mu}$, and let
$DEF_{+}\left(S\right)=$ the (+) deficiency space, i.e.,\index{operator!restriction-}
\begin{equation}
DEF_{+}\left(S\right)=\left\{ g\in dom\left(S^{*}\right)\:\Big|\:S^{*}g=ig\right\} \label{eq:d-9}
\end{equation}
Assume $\dim\left(DEF_{+}\left(S\right)\right)=d$, and $1\leq d<\infty$.
Then set $\mathscr{L}:=DEF_{+}\left(S\right)$. Using \cite{Jor81},
one checks that (\ref{enu:d-1}) then holds for this closed subspace
in $L^{2}\left(\mu\right)$. 

The fact that (\ref{eq:d-4}) holds for this subspace $\mathscr{L}$
follows from the observation:
\[
DEF_{+}\left(S\right)\cap dom\left(M_{\mu}\right)=\left\{ 0\right\} 
\]
for every densely defined restriction $S$ of $M_{\mu}$.
\end{svmultproof2}

\section{\label{sec:index11}Spectral Representation of Index $\left(1,1\right)$
Hermitian Operators}

In this section we give an explicit answer to the following question:
How to go from any index $\left(1,1\right)$ Hermitian operator to
a $\left(\mathscr{H}_{F},D^{\left(F\right)}\right)$ model; i.e.,
from a given index $\left(1,1\right)$ Hermitian operator with dense
domain in a separable Hilbert space $\mathscr{H}$, we build a p.d.
continuous function $F$ on $\Omega-\Omega$, where $\Omega$ is a
finite interval $\left(0,a\right)$, $a>0$.

So far, we have been concentrating on building transforms going in
the other direction. But recall that, for a given continuous p.d.
function $F$ on $\Omega-\Omega$, it is often difficult to answer
the question of whether the corresponding operator $D^{\left(F\right)}$
in the RKHS $\mathscr{H}_{F}$ has deficiency indices $\left(1,1\right)$
or $\left(0,0\right)$. \index{representation!spectral-}

Now this question answers itself once we have an explicit transform
going in the opposite direction. Specifically, given any index $\left(1,1\right)$
Hermitian operator $S$ in a separable Hilbert space $\mathscr{H}$,
we then to find a pair $\left(F,\Omega\right)$, p.d. function and
interval, with the desired properties. There are two steps:

Step 1, writing down explicitly, a p.d. continuous function $F$ on
$\Omega-\Omega$, and the associated RKHS $\mathscr{H}_{F}$ with
operator $D^{\left(F\right)}$.

Step 2, constructing an intertwining\index{intertwining} \index{operator!intertwining-}isomorphism
$W:\mathscr{H}\rightarrow\mathscr{H}_{F}$, having the following properties.
$W$ will be an isometric isomorphism, intertwining the pair $\left(\mathscr{H},S\right)$
with $(\mathscr{H}_{F},D^{\left(F\right)})$, i.e., satisfying $WS=D^{\left(F\right)}W$;
and also intertwining the respective domains and deficiency spaces,
in $\mathscr{H}$ and $\mathscr{H}_{F}$.

Moreover, starting with any $\left(1,1\right)$ Hermitian operator,
we can even arrange a normalization for the p.d. function $F$ such
that $\Omega=$ the interval $\left(0,1\right)$ will do the job.\index{Hermitian}

We now turn to the details:

We will have three pairs $\left(\mathscr{H},S\right)$, $\left(L^{2}\left(\mathbb{R},\mu\right),\mbox{ restriction of }M_{\mu}\right)$,
and $(\mathscr{H}_{F},D^{\left(F\right)})$, where:
\begin{enumerate}[leftmargin=*,label=(\roman{enumi})]
\item $S$ is a fixed Hermitian operator with dense domain $dom\left(S\right)$
in a separable Hilbert space $\mathscr{H}$, and with deficiency indices
$\left(1,1\right)$. 
\item From the given information in (i), we will construct a finite Borel
measure\index{measure!Borel} $\mu$ on $\mathbb{R}$ such that an
index-$\left(1,1\right)$ restriction of $M_{\mu}:f\mapsto\lambda f\left(\lambda\right)$
in $L^{2}\left(\mathbb{R},\mu\right)$, is equivalent to $\left(\mathscr{H},S\right)$.
\index{operator!restriction-}
\item Here $F:\left(-1,1\right)\rightarrow\mathbb{C}$ will be a p.d. continuous
function, $\mathscr{H}_{F}$ the corresponding RKHS; and $D^{\left(F\right)}$
the usual operator with dense domain 
\[
\left\{ F_{\varphi}\:\Big|\:\varphi\in C_{c}^{\infty}\left(0,1\right)\right\} ,\mbox{ and}
\]
\begin{equation}
D^{\left(F\right)}\left(F_{\varphi}\right)=\frac{1}{i}F_{\varphi'},\;\varphi'=\frac{d\varphi}{dx}.\label{eq:d11-1}
\end{equation}

\end{enumerate}
We will accomplish the stated goal with the following system of intertwining
operators: See Figure \ref{fig:d11-1}.

But we stress that, at the outset, only (i) is given; the rest ($\mu$,
$F$ and $\mathscr{H}_{F}$) will be constructed. Further, the solutions
$\left(\mu,F\right)$ in Figure \ref{fig:d11-1} are not unique; rather
they depend on choice of selfadjoint extension in (i): Different selfadjoint
extensions of $S$ in (i) yield different solutions $\left(\mu,F\right)$.
But the selfadjoint extensions of $S$ in $\mathscr{H}$ are parameterized
by von Neumann's theory; see e.g., \cite{Rud73,DS88b}. 

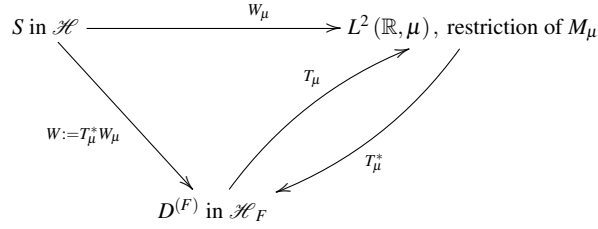
\begin{figure}[H]
$\xymatrix{S\mbox{ in }\mathscr{H}\ar[rr]^{W_{\mu}}\ar[ddr]_{W:=T_{\mu}^{*}W_{\mu}} &  & \ar@/^{1pc}/[ddl]^{T_{\mu}^{*}}L^{2}\left(\mathbb{R},\mu\right),\mbox{ restriction of }M_{\mu}\\
\\
 & D^{\left(F\right)}\mbox{ in }\mbox{\ensuremath{\mathscr{H}}}_{F}\ar@/^{1pc}/[uur]^{T_{\mu}}
}
$

\protect\caption{\label{fig:d11-1}A system of intertwining operators.}
\end{figure}

\begin{flushleft}
\index{von Neumann, John}
\par\end{flushleft}

\begin{flushleft}
\index{intertwining}
\par\end{flushleft}

\index{measure!probability}\index{operator!intertwining-}
\begin{remark}
In our analysis of (i)-(iii), we may without loss of generality assume
that the following normalizations hold:

$\left(z_{1}\right)$ $\mu\left(\mathbb{R}\right)=1$ , so $\mu$
is a probability measure;

$\left(z_{2}\right)$ $F\left(0\right)=1$, and the p.d. continuous
solution 

$\left(z_{3}\right)$ $F:\left(-1,1\right)\rightarrow\mathbb{C}$
is defined on $\left(-1,1\right)$; so $\Omega:=\left(0,1\right)$.

\begin{flushleft}
Further, we may assume that the operator $S$ in $\mathscr{H}$ from
(i) has simple spectrum.
\par\end{flushleft}\end{remark}
\begin{theorem}
\label{thm:d11-1}Starting with $\left(\mathscr{H},S\right)$ as in
(i), there are solutions $\left(\mu,F\right)$ to (ii)-(iii), and
intertwining operators $W_{\mu}$, $T_{\mu}$ as in Figure \ref{fig:d11-1},
such that\index{intertwining}\index{operator!intertwining-}
\begin{equation}
W:=T_{\mu}^{*}W_{\mu}\label{eq:d11-3}
\end{equation}
satisfies the intertwining properties for $\left(\mathscr{H},S\right)$
and $\left(\mathscr{H}_{F},D^{\left(F\right)}\right)$.\end{theorem}
\begin{svmultproof2}
Since $S$ has indices $\left(1,1\right)$, $\dim DEF_{\pm}\left(S\right)=1$,
and $S$ has selfadjoint extensions indexed by partial isometries
$DEF_{+}\overset{v}{\longrightarrow}DEF_{-}$; see \cite{Rud73,DS88b}.
We now pick $g_{+}\in DEF_{+}$, $\left\Vert g_{+}\right\Vert =1$,
and partial isometry $v$ with selfadjoint extension $S_{v}$, i.e.,
\index{selfadjoint extension}\index{operator!partial isometry}\index{operator!intertwining-}
\begin{equation}
S\subset S_{v}\subset S_{v}^{*}\subset S^{*}.\label{eq:d11-4}
\end{equation}

Hence $\left\{ U_{v}\left(t\right)\:\Big|\:t\in\mathbb{R}\right\} $
is a strongly continuous unitary representation of $\mathbb{R}$,
acting in $\mathscr{H}$, $U_{v}\left(t\right):=e^{itS_{v}}$, $t\in\mathbb{R}$.
Let $P_{S_{v}}\left(\cdot\right)$ be the corresponding projection
valued measure (PVM) on $\mathscr{B}\left(\mathbb{R}\right)$, i.e.,
we have \index{measure!PVM}\index{representation!unitary-}\index{strongly continuous}\index{Hermitian}
\begin{equation}
U_{v}\left(t\right)=\int_{\mathbb{R}}e^{it\lambda}P_{S_{v}}\left(d\lambda\right);\label{eq:d11-5}
\end{equation}
and set 
\begin{equation}
d\mu\left(\lambda\right):=d\mu_{v}\left(\lambda\right)=\left\Vert P_{S_{v}}\left(d\lambda\right)g_{+}\right\Vert _{\mathscr{H}}^{2}.\label{eq:d11-6}
\end{equation}
For $f\in L^{2}\left(\mathbb{R},\mu_{v}\right)$, set 
\begin{equation}
W_{\mu_{v}}\left(f\left(S_{v}\right)g_{+}\right)=f;\label{eq:d11-7}
\end{equation}
then $W_{\mu_{v}}:\mathscr{H}\rightarrow L^{2}\left(\mathbb{R},\mu_{v}\right)$
is isometric onto; and 
\begin{equation}
W_{\mu_{v}^{*}}\left(f\right)=f\left(S_{v}\right)g_{+},\label{eq:d11-8}
\end{equation}
where 
\begin{equation}
f\left(S_{v}\right)g_{+}=\int_{\mathbb{R}}f\left(\lambda\right)P_{S_{v}}\left(d\lambda\right)g_{+}.\label{eq:d11-9}
\end{equation}
For justification of these assertions, see e.g., \cite{Ne69}. Moreover,
$W_{\mu}$ has the intertwining properties sketched in Figure \ref{fig:d11-1}.

Returning to (\ref{eq:d11-5}) and (iii) in the theorem, we now set
$F=$ the restriction to $\left(-1,1\right)$ of 
\begin{eqnarray}
F_{\mu}\left(t\right) & := & \left\langle g_{+},U_{v}\left(t\right)g_{+}\right\rangle \label{eq:d11-10}\\
 & \underset{(\mbox{by }(\ref{eq:d11-5}))}{=} & \left\langle g_{+},\int_{\mathbb{R}}e^{it\lambda}P_{S_{v}}\left(d\lambda\right)g_{+}\right\rangle \nonumber \\
 & = & \int_{\mathbb{R}}e^{it\lambda}\left\Vert P_{S_{v}}\left(d\lambda\right)g_{+}\right\Vert ^{2}\nonumber \\
 & \underset{(\mbox{by }(\ref{eq:d11-6}))}{=} & \int_{\mathbb{R}}e^{it\lambda}d\mu_{v}\left(\lambda\right)\nonumber \\
 & = & \widehat{d\mu_{v}}\left(t\right),\;\forall t\in\mathbb{R}.\nonumber 
\end{eqnarray}

We now show that 
\begin{equation}
F:=F_{\mu}\Big|_{\left(-1,1\right)}\label{eq:d11-11}
\end{equation}
has the desired properties.

From Corollary \ref{cor:lcg-isom}, we have the isometry\index{isometry}
$T_{\mu}\left(F_{\varphi}\right)=\widehat{\varphi}$, $\varphi\in C_{c}\left(0,1\right)$,
with adjoint \index{operator!adjoint of an-} 
\begin{equation}
T_{\mu}^{*}\left(f\right)=\chi_{\overline{\Omega}}\left(fd\mu\right)^{\vee},\mbox{ and}\label{eq:d11-12}
\end{equation}
\[
\xymatrix{\mathscr{H}_{F}\ar@/^{1pc}/^{T_{\mu}}[rr] &  & L^{2}\left(\mathbb{R},\mu\right)\ar@/^{1pc}/^{T_{\mu}^{*}}[ll]}
;
\]
see also Figure \ref{fig:d11-1}. 

The following properties are easily checked:
\begin{equation}
W_{\mu}\left(g_{+}\right)=\mathbf{1}\in L^{2}\left(\mathbb{R},\mu\right),\mbox{ and}\label{eq:d11-13}
\end{equation}
\begin{equation}
T_{\mu}^{*}\left(\mathbf{1}\right)=F_{0}=F\left(\cdot-0\right)\in\mathscr{H}_{F},\label{eq:d11-14}
\end{equation}
as well as the intertwining properties stated in the theorem; see
Figure. \ref{fig:d11-1} for a summary.

\uline{Proof of (\mbox{\ref{eq:d11-13}})} We will show instead
that $W_{\mu}^{*}\left(\mathbf{1}\right)=g_{+}$. From (\ref{eq:d11-9})
we note that if $f\in L^{2}\left(\mathbb{R},\mu\right)$ satisfies
$f=\mathbf{1}$, then $f\left(S_{v}\right)=I=$ the identity operator
in $\mathscr{H}$. Hence 
\[
W_{\mu}^{*}\left(\mathbf{1}\right)\underset{(\mbox{by }(\ref{eq:d11-8}))}{=}\mathbf{1}\left(S_{v}\right)g_{+}=g_{+},
\]
which is (\ref{eq:d11-13}).

\uline{Proof of (\mbox{\ref{eq:d11-14}})} For $\varphi\in C_{c}\left(0,1\right)$
we have $\widehat{\varphi}\in L^{2}\left(\mathbb{R},\mu\right)$,
and 
\begin{eqnarray*}
T_{\mu}^{*}T_{\mu}\left(F_{\varphi}\right) & = & T_{\mu}^{*}\left(\widehat{\varphi}\right)\\
 & \underset{(\mbox{by }(\ref{eq:d11-12}))}{=} & \chi_{\overline{\Omega}}\left(\widehat{\varphi}d\mu\right)^{\vee}=F_{\varphi}.
\end{eqnarray*}
Taking now an approximation $\left(\varphi_{n}\right)\subset C_{c}\left(0,1\right)$
to the Dirac unit mass at $x=0$, we get (\ref{eq:d11-14}).\end{svmultproof2}

\begin{corollary}
The deficiency indices of $D^{\left(F\right)}$ in $\mathscr{H}_{F}$
for $F\left(x\right)=e^{-\left|x\right|}$, $\left|x\right|<1$, are
$\left(1,1\right)$. \end{corollary}
\begin{svmultproof2}
Take $\mathscr{H}=L^{2}\left(\mathbb{R}\right)=\left\{ f\mbox{ measurable on }\mathbb{R}\:\Big|\:\int_{\mathbb{R}}\left|f\left(x\right)\right|^{2}dx<\infty\right\} $,
where $dx=$ Lebesgue measure. 

Take $g_{+}:=\left(\frac{1}{\lambda+i}\right)^{\vee}\left(x\right)$,
$x\in\mathbb{R}$; then $g_{+}\in L^{2}\left(\mathbb{R}\right)=:\mathscr{H}$
since
\begin{eqnarray*}
\int_{\mathbb{R}}\left|g_{+}\left(x\right)\right|^{2}dx & \underset{(\mbox{Parseval})}{=} & \int_{\mathbb{R}}\left|\frac{1}{\lambda+i}\right|^{2}d\lambda\\
 & = & \int_{\mathbb{R}}\frac{1}{1+\lambda^{2}}d\lambda=\pi.
\end{eqnarray*}

Now for $S$ and $S_{v}$ in Theorem \ref{thm:d11-1}, we take 
\begin{equation}
S_{v}h=\frac{1}{i}\frac{d}{dx}h\mbox{ on }\left\{ h\in L^{2}\left(\mathbb{R}\right)\:\Big|\:h'\in L^{2}\left(\mathbb{R}\right)\right\} \mbox{ and}\label{eq:d11-16}
\end{equation}
\[
S=S_{v}\Big|_{\left\{ h\:\Big|\:h,h'\in L^{2}\left(\mathbb{R}\right),h\left(0\right)=0\right\} },
\]
then by \cite{Jor81}, we know that $S$ is an index $\left(1,1\right)$
operator, and that $g_{+}\in DEF_{+}\left(S\right)$. The corresponding
p.d. continuous function $F$ is the restriction to $\left|t\right|<1$
of the p.d. function:
\begin{eqnarray*}
\left\langle g_{+},U_{v}\left(t\right)g_{+}\right\rangle _{\mathscr{H}} & = & \int_{\mathbb{R}}\frac{1}{\lambda-i}\frac{e^{it\lambda}}{\lambda+i}d\lambda\\
 & = & \left(\frac{1}{1+\lambda^{2}}\right)^{\vee}\left(t\right)=\pi e^{-\left|t\right|}.
\end{eqnarray*}
\end{svmultproof2}

\begin{example}[Lévy-measures]
 Let $0<\alpha\leq2$, $-1<\beta<1$, $v>0$; then the Lévy-measures\index{measure!Lévy}
$\mu$ on $\mathbb{R}$ are indexed by $\left(\alpha,\beta,\nu\right)$,
so $\mu=\mu_{\left(\alpha,\beta,\nu\right)}$. They are absolutely
continuous\index{absolutely continuous} with respect to Lebesgue
measure $d\lambda$ on $\mathbb{R}$; and for $\alpha=1$, 
\begin{equation}
F_{\left(\alpha,\beta,\nu\right)}\left(x\right)=\widehat{\mu_{\left(\alpha,\beta,\nu\right)}}\left(x\right),\;x\in\mathbb{R},\label{eq:d11-17}
\end{equation}
satisfies 
\begin{equation}
F_{\left(\alpha,\beta,\nu\right)}\left(x\right)=\exp\left(-\nu\left|x\right|\cdot\left(1+\frac{2i\beta}{\pi}-\mbox{sgn}\left(x\right)\ln\left|x\right|\right)\right).\label{eq:d11-18}
\end{equation}
The case $\alpha=2$, $\beta=0$, reduces to the Gaussian distribution.
\index{Gaussian distribution}\index{distribution!Gaussian-}

The measures $\mu_{\left(1,\beta,\nu\right)}$ have infinite variance,
i.e., 
\[
\int_{\mathbb{R}}\lambda^{2}d\mu_{\left(1,\beta,\nu\right)}=\infty.
\]

As a Corollary of Theorem \ref{thm:d11-1}, we therefore conclude
that, for the restrictions, 
\[
F_{\left(1,\beta,\nu\right)}^{\left(res\right)}\left(x\right)=F_{\left(1,\beta,\nu\right)}\left(x\right),\;\mbox{in }x\in\left(-1,1\right),\;\mbox{(see }(\ref{eq:d11-17})-(\ref{eq:d11-18}))
\]
the associated Hermitian operator $D^{F^{\left(res\right)}}$ all
have deficiency indices $\left(1,1\right)$. 

In connection Levy measures\index{measure!Lévy}, see e.g., \cite{ST94}.
\end{example}

\chapter{\label{chap:question}Overview and Open Questions}

\section{From Restriction Operator to Restriction of p.d. Function}

\index{measure!tempered}

\index{spectrum}\index{operator!restriction-}

\index{Theorem!Bochner's-}\index{Bochner's Theorem}

The main difference between our measures on $\mathbb{R}$, and the
measures used in fractional Brownian motion and related processes
is that our measures are finite on $\mathbb{R}$, but the others aren\textquoteright t;
instead they are what is called tempered (see \cite{AL08}). If $\mu$
is a tempered positive measure, then the function $F=\widehat{d\mu}$
is still positive definite, but it is not continuous, unless $\mu\left(\mathbb{R}\right)<\infty$.

This means that they are unbounded but only with a polynomial degree.
For example, Lebesgue measure $dx$ on $\mathbb{R}$ is tempered.
Suppose $1/2<g<1$, then the positive measure $d\mu\left(x\right)=\left|x\right|^{g}dx$
is tempered, and this measure is what is used in accounting for the
p.d. functions discussed in the papers on Fractional Brownian motion.

But we could ask the same kind of questions for tempered measures
as we do for the finite positive measures. 

Following this analogy, one can say that the paper \cite{Jor81} (by
one of the co-authors) was about models of index-$\left(1,1\right)$
operators, admitting realizations in $L^{2}\left(\mathbb{R}\right)$,
i.e., $L^{2}$ of Lebesgue measure $dx$ on $\mathbb{R}$. This work
was followed up subsequently also for the index-$\left(m,m\right)$
case, $m>1$.

\section{The Splitting $\mathscr{H}_{F}=\mathscr{H}_{F}^{\left(atom\right)}\oplus\mathscr{H}_{F}^{\left(ac\right)}\oplus\mathscr{H}_{F}^{\left(sing\right)}$ }

Let $\Omega$ be as usual, connected and open in $\mathbb{R}^{n}$;
and let $F$ be a p.d. continuous function on $\Omega-\Omega$. Suppose
$Ext\left(F\right)$ is non-empty. We then get a unitary representation
$U$ of $G=\mathbb{R}^{n}$, with associated projection valued measure
(PVM) $P_{U}$, acting on the RKHS $\mathscr{H}_{F}$. This gives
rise to an orthogonal splitting of $\mathscr{H}_{F}$ into three parts,
atomic, absolutely continuous, and continuous singular, defined from
$P_{U}$.

\index{atom}\index{absolutely continuous}\index{orthogonal}\index{representation!unitary-}\index{orthogonal splitting}\index{purely atomic}
\begin{question}
What are some interesting examples illustrating the triple splitting
of $\mathscr{H}_{F}$, as in Theorem \ref{thm:R^n-spect}, eq (\ref{eq:rn6})? 
\end{question}
In Section \ref{sub:euclid}, we constructed an interesting example
for the splitting of the RKHS $\mathscr{H}_{F}$: All three subspaces
$\mathscr{H}_{F}^{\left(atom\right)}$, $\mathscr{H}_{F}^{\left(ac\right)}$,
and $\mathscr{H}_{F}^{\left(sing\right)}$ are non-zero; the first
one is one-dimensional, and the other two are infinite-dimensional.
See Example \ref{ex:splitting} for details.

For some recent developments, see, also, \cite{JPT11-3,JPT12-1,PeTi13}.

\section{The Case of $G=\mathbb{R}^{1}$}

We can move to the circle group $\mathbb{T}=\mathbb{R}/\mathbb{Z}$,
or stay with $\mathbb{R}$, or go to $\mathbb{R}^{n}$; or go to Lie
groups. For Lie groups, we must also look at the universal simply
connected covering groups. All very interesting questions. \index{universal covering group}\index{group!universal covering-}

A FEW POINTS, about a single interval and asking for p.d. extensions
to $\mathbb{R}$:
\begin{enumerate}[leftmargin=*]
\item The case of $\mathbb{R}$, we are asking to extend a (small) p.d.
continuous function $F$ from an interval to all of $\mathbb{R}$,
so that the extension to $\mathbb{R}$ is also p.d. continuous; --
in this case we always have existence. This initial interval may even
be \textquotedblleft very small.\textquotedblright{} If we have several
intervals, existence may fail, but the problem is interesting. We
have calculated a few examples, but nothing close to a classification!
\item Computable ONBs in $\mathscr{H}_{F}$ would be very interesting. We
can use what we know about selfadjoint extensions of operators with
deficiency indices $\left(1,1\right)$ to classify all the p.d extension
to $\mathbb{R}$. While both are \textquotedblleft extension\textquotedblright{}
problems, operator extensions and p.d. extensions is subtle, and non-intuitive.
For this problem, there are two sources of earlier papers; one is
M. Krein \cite{Kre46}\index{Krein, M.}, and the other L. deBranges.
Both are covered beautifully in a book by Dym and McKean \cite{DM76}.
\end{enumerate}

\section{The Extreme Points of $Ext\left(F\right)$ and $\Im\left\{ F\right\} $}

Given a locally defined p.d. function $F$, i.e., a p.d. function
$F$ defined on a proper subset in a group, then the real part $\Re\left\{ F\right\} $
is also positive definite. Can anything be said about the imaginary
part, $\Im\left\{ F\right\} $? See Section \ref{sec:imgF}.

Assuming that $F$ is also continuous, then what are the extreme points
in the compact convex set $Ext\left(F\right)$, i.e., what is the
subset $ext\left(Ext\left(F\right)\right)$? How do properties of
$Ext\left(F\right)$ depend on the imaginary part, i.e., on the function
$\Im\left\{ F\right\} $? How do properties of the skew-Hermitian\index{operator!skew-Hermitian}
operator $D^{\left(F\right)}$ (in the RKHS $\mathscr{H}_{F}$) depend
on the imaginary part, i.e., on the function $\Im\left\{ F\right\} $?\index{skew-Hermitian operator; also called skew-symmetric}\index{convex}

\listoffigures

\listoftables

\cleardoublepage
\renewcommand{\bibsection}{\chapter*{\bibname}}

\bibliographystyle{amsalpha}
\phantomsection\addcontentsline{toc}{chapter}{\bibname}\bibliography{number5}

\printindex{}
\end{document}